\documentclass[reqno]{amsart}
\usepackage[english]{babel}
\usepackage{graphicx,subfigure}
\usepackage{amsmath,amssymb}
\usepackage{color}
\definecolor{oneblue}{rgb}{0,0.0,0.75}

\usepackage{fancyhdr}

\makeatletter
\newcommand{\sech}{\mathop{\operator@font sech}}
\newcommand{\sign}{\mathop{\operator@font sign}}
\makeatother

\newtheorem{lemma}{Lemma}[section]
\newtheorem{theorem}{Theorem}[section]
\newtheorem{proposition}{Proposition}[section]

\newtheorem{remark}{Remark}[section]
\numberwithin{equation}{section}

\begin{document}

\title[Notes on numerical analysis...]{Notes on numerical analysis and solitary wave solutions of Boussinesq/Boussinesq systems for internal waves}

\author[V. A. Dougalis]{Vassilios A. Dougalis}
\address{Mathematics Department, University of Athens, 15784
Zographou, Greece \and Institute of Applied \& Computational
Mathematics, FO.R.T.H., 71110 Heraklion, Greece}
\email{doug@math.uoa.gr}

\author[A. Duran]{Angel Duran}
\address{ Applied Mathematics Department,  University of
Valladolid, 47011 Valladolid, Spain}
\email{angel@mac.uva.es}

\author[L. Saridaki]{Leetha Saridaki}
\address{Mathematics Department, University of Athens, 15784
Zographou, Greece \and Institute of Applied \& Computational
Mathematics, FO.R.T.H., 71110 Heraklion, Greece}
\email{leetha.saridaki@gmail.com}

%\author[D. Mitsotakis]{Dimitrios Mitsotakis}
%\address{School of Mathematics and Statistics, Victoria University of Wellington, PO Box 600
%Wellington 6140
%New Zealand}
%\email{dmitsot@gmail.com}
%\urladdr{http://dmitsot.googlepages.com/}

%\tableofcontents
%\author[A. Duran]{Angel Duran}
%\address{ Applied Mathematics Department, University of
%Valladolid, Valladolid, Spain}
%\email{angel@mac.uva.es}
%
%\author[D. Mitsotakis]{Dimitrios Mitsotakis}
%\address{Victoria University of Wellington, School of Mathematics, Statistics and Operations Research, PO Box 600, Wellington 6140, New Zealand}
%\email{dimitrios.mitsotakis@vuw.ac.nz}
%\urladdr{http://dmitsot.googlepages.com/}

\subjclass[2010]{76B15 (primary), 76B25, 65M70 (secondary)}
\keywords{Internal waves, Boussinesq/Boussinesq systems, solitary waves, spectral methods}

\begin{abstract}
In this paper a three-parameter family of Boussinesq systems is studied. The systems have been proposed as models of the propagation of long internal waves along the interface of a two-layer system of fluids with rigid-lid condition for the upper layer and under a Boussinesq regime for the flow in both layers. The contents of the paper are as follows. We first present some theoretical properties of well-posedness, conservation laws and Hamiltonian structure of the systems, using the results for analogous models for surface wave propagation. Then the corresponding periodic initial-value problem is discretized in space by the spectral Fourier Galerkin method and for each well posed system, error estimates for the semidiscrete approximation are proved. The rest of the paper is concerned with the study of existence and the numerical simulation of some issues of the dynamics of solitary-wave solutions. Standard theories are used to derive several results of existence of classical and generalized solitary waves, depending on the parameters of the models. A numerical procedure based on a Fourier collocation approximation for the ode system of the solitary wave profiles with periodic boundary conditions, and on the iterative solution of the resulting fixed-point systems with the Petviashvili scheme combined with vector extrapolation techniques, is used to generate numerically approximations of solitary waves. These are taken as initial conditions in a computational study of the dynamics of the solitary waves, both classical and generalized. To this end, the spectral semidiscretizations of the periodic initial-value problem for the systems are numerically integrated by a fourth-order Runge-Kutta-composition method based on the implicit midpoint rule. The fully discrete scheme is then used to approximate the evolution of small and large perturbations of computed solitary wave profiles, and to study computationally the collisions of solitary waves as well as the resolution of initial data into trains of solitary waves.
\end{abstract}

\maketitle

%\newpage
\tableofcontents

\section{Introduction}
The following three-parameter family of Boussinesq/Boussinesq (B/B) systems for internal waves was derived by Bona, Lannes and Saut, \cite{BonaLS2008}:
\begin{eqnarray}
(1-b\Delta)\zeta_{t}+\frac{1}{\gamma+\delta}\nabla\cdot{\bf v}_{\beta}+\left(\frac{\delta^{2}-\gamma}{(\delta+\gamma)^{2}}\right)\nabla\cdot \left(\zeta{\bf v}_{\beta}\right)+a\nabla\cdot\Delta {\bf v}_{\beta}=0,\nonumber&&\\
(1-d\Delta)({\bf v}_{\beta})_{t}+(1-\gamma)\nabla\zeta+\left(\frac{\delta^{2}-\gamma}{2(\delta+\gamma)^{2}}\right)\nabla |{\bf v}_{\beta}|^{2}+(1-\gamma)c\Delta\nabla\zeta=0.&&\label{BB1}
\end{eqnarray}
The system (\ref{BB1}) is a model (in nondimensional, unscaled form) for the propagation of internal waves along the interface of an inviscid, homogeneous two-layer system of fluids, the upper of which is labelled $1$ and the lower $2$. The layers have 
depths $d_{1}, d_{2}$ and densities $\rho_{1}, \rho_{2}$ with $\rho_{2}>\rho_{1}$. The upper layer is restricted by the rigid-lid assumption, at depth $z=0$, while the rigid, horizontal bottom lies at depth $z=-(d_{1}+d_{2})$. In (\ref{BB1}) 
 $\zeta=\zeta(x,y,t)$ represents the deviation of the interface from the rest position at $(x,y)$ at time $t$, while ${\bf v}_{\beta}=(I-\beta\Delta)^{-1}{\bf v}$, where $\beta\geq 0$ is a modelling parameter, $\Delta$ denotes the Laplace operator and ${\bf v}$ is a \lq velocity\rq\ variable defined in \cite{BonaLS2008} in terms of the horizontal components of the velocities of the two layers of fluids ${\bf v}^{(1)}$ and ${\bf v}^{(2)}$ as the difference ${\bf v}^{(2)}-\gamma {\bf v}^{(1)}$ evaluated at the interface. The constants
\begin{eqnarray*}
\gamma=\frac{\rho_{1}}{\rho_{2}}<1,\quad \delta=\frac{d_{1}}{d_{2}},
\end{eqnarray*}
denote the density and depth ratios, respectively. The parameters $a,b,c,d$ depend on the physical parameters $\delta,\gamma$ and the modelling parameters $\alpha_{1}\geq 0, \beta\geq 0$ and $\alpha_{2}\leq 1$, \cite{BonaLS2008}, and are given by
\begin{eqnarray}
a&=&\frac{(1-\alpha_{1})(1+\gamma\delta)-3\delta\beta (\delta+\gamma)}{3\delta (\gamma+\delta)^{2}},\quad b=\alpha_{1}\frac{1+\gamma\delta}{3\delta(\gamma+\delta)},\nonumber\\
c&=&\beta\alpha_{2},\quad d=\beta(1-\alpha_{2}).\label{BB1b}
\end{eqnarray}
These formulas lead to the relation
\begin{eqnarray}
(\delta+\gamma)a+b+c+d=S(\gamma,\delta),\quad S(\gamma,\delta):=\frac{1+\gamma\delta}{3\delta(\gamma+\delta)}.\label{BB1c}
\end{eqnarray}
The case $\gamma=0,\delta=1$ corresponds to the Boussinesq systems for surface water waves analyzed by Bona, Chen, and Saut in \cite{BonaChS2002,BonaChS2004}. In that case $\beta$ should be taken equal to $\frac{1}{2}(1-\theta^{2})$ in the notation of \cite{BonaChS2002,BonaChS2004}, where $0\leq \theta\leq 1$ defines a parametrization of the depth variable $z=-1+\theta$,
$\zeta$ is the displacement of the surface elevation of the wave over the rest position $z=0$, and the horizontal velocity at the free surface would be given by ${\bf v}$. The variable ${\bf v}_{\beta}$ represents now the horizontal velocity at depth $z=-1+\theta$.

In \cite{BonaLS2008} (see also \cite{S}), several asymptotic models for internal waves in different physical regimes are derived, and the consistency of the corresponding full Euler equations with them is established in a rigorous manner. The physical regimes are defined in terms of the scaling parameters
\begin{eqnarray}
\epsilon=\frac{a}{d_{1}},\quad \mu=\frac{d_{1}^{2}}{\lambda^{2}},\label{BB21b}
\end{eqnarray}
and $\delta$, where $a$ and $\lambda$ denote a typical amplitude and wavelength of the interface wave, respectively. The parameters (\ref{BB21b}) are defined with respect to the upper layer; similar ones, $\epsilon_{2}$ and $\mu_{2}$, can be defined with respect to the lower layer. Then the system (\ref{BB1}) is valid in the so-called Boussinesq/Boussinesq (B/B) regime; this means that the flow is in the Boussinesq regime in both fluid domains, i.~e. the physical parameters satisfy the conditions $\delta\sim 1$ and
\begin{eqnarray}
\mu\sim\mu_{2}\sim\epsilon\sim\epsilon_{2}<<1.\label{BB21bb}
\end{eqnarray}
For a review of several other issues concerning the modelling of internal waves in the B/B and the other asymptotic regimes defined in \cite{BonaLS2008} we refer the reader to the notes of Saut, \cite{S}. We would also like to mention that Nguyen and Dias derived in \cite{ND} a B/B system and extended it to the case of higher-power nonlinear terms. In \cite{Du} Duch\^{e}ne considered free-surface and rigid-lid B/B systems and studied their one-way KdV approximations.

The present study is focused on the one-dimensional version of (\ref{BB1}), which is written in unscaled, dimensionless variables, for $x\in\mathbb{R}, t\geq 0$, as
\begin{eqnarray}
(1-b\partial_{xx})\zeta_{t}+\frac{1}{\gamma+\delta}\partial_{x} v_{\beta}+\left(\frac{\delta^{2}-\gamma}{(\delta+\gamma)^{2}}\right)\partial_{x} \left(\zeta v_{\beta}\right)+a\partial_{xxx} v_{\beta}=0,\nonumber&&\\
(1-d\partial_{xx})( v_{\beta})_{t}+(1-\gamma)\partial_{x}\zeta+\left(\frac{\delta^{2}-\gamma}{2(\delta+\gamma)^{2}}\right)\partial_{x} v_{\beta}^{2}+(1-\gamma)c\partial_{xxx}\zeta=0.&&\label{BB2}
\end{eqnarray}
with ${v}_{\beta}=(1-\beta\partial_{xx})^{-1}{u}$.

In section \ref{sec2} of the paper at hand we present an alternative to that of \cite{BonaLS2008} derivation of the family of systems (\ref{BB2}). We study the linear and nonlinear well-posedness of the B/B systems for various values of the coefficients $a, b, c, d$, based on the analogous theory valid for the Boussinesq systems for surface waves presented in \cite{BonaChS2002,BonaChS2004}. We identify seven classes of $B/B$ systems that are linearly well posed with coefficients relevant to the internal wave problem and whose initial-value problems are nonlinearly, in general locally in time, well-posed in appropriate pairs of Sobolev spaces. In the rest of the paper we consider in detail the numerical approximation of these classes of systems and study the existence and numerical construction of their solitary-wave solutions. We also present a numerical study of various issues of stability and interactions of the solitary waves.

Specifically, in section \ref{sec3} we discretize in space the periodic initial-value problem (ivp) for these systems using the spectral Fourier-Galerkin method and prove $L^{2}$ error estimates for the ensuing semidiscrete approximations. These estimates remain of course valid for the analogous surface-wave Boussinesq systems (take $\gamma=0, \delta =1$ in (\ref{BB2})).

In recent years there have appeared several papers with rigorous error estimates for numerical methds for surface-wave Boussinesq systems. For example, in \cite{ADM1,ADM2,AD2,DougalisMS2007,DMS2,DMS3} one may find error analyses of Galerkin-finite element semidiscretizations for various initial-boundary-value problems (ibvp's) for several Boussinesq systems in one and two space dimensions. The papers \cite{ADM1} and \cite{AD2} also contain error estimates of temporal discretizations of the semidiscrete problems effected with high-order accurate, explicit Runge-Kutta (RK) time-stepping schemes. In \cite{X} Xavier {\it et al.} analyze spectral methods of collocation type, coupled with the explicit, \lq classical\rq, fourth-order accurate RK scheme for time-stepping for the surface-wave Boussinesq systems corresponding to the classes of the cases (i) and (v), see section \ref{sec2} in the sequel.

In section \ref{sec4} we consider the existence and numerical approximation of solitary-wave solutions of the systems (\ref{BB2}). We apply several techniques for proving existence of such waves, namely Normal Form Theory, \cite{I,HaragusI}, valid for solitary-wave speeds close to the limiting value $c_{\gamma,\delta}=\sqrt{(1-\gamma)/(\delta+\gamma)}$, and Toland's theory, \cite{Toland1986}, Concentration-Compactness theory, \cite{Lions}, and Positive Operator theory, \cite{BenjaminBB1990}, that predict existence of solitary waves also at larger speeds (relative to $c_{\gamma,\delta}$). We construct in several cases of interest classical and generalized solitary-wave profiles by solving numerically the second-order nonlinear ordinary differential equation (ode) systems satisfied by the solitary waves. The ode systems are discretized by a spectral method and the resulting nonlinear systems of algebraic equations are numerically solved by the iterative Petviashvili scheme, \cite{Petv1976,pelinovskys}, accelerated by the Minimal Polynomial extrapolation (MPE) algorithm, \cite{sidi,sidifs,smithfs,AlvarezD2015}.

In section \ref{sec5} we present a computational study of several issues associated with the stability and the dynamics of classical and generalized solitary waves of B/B systems with parameters belonging in the \lq generic\rq\ class (ii) of section \ref{sec2}. We note that a computational study of solitary waves has been carried out in \cite{ND} by a spectral-RK scheme, among others, for a B/B system resembling those in the class (vi); classical solitary waves are constructed and their overtaking and head-on collisions are simulated. In \cite{Du} the author used a Crank-Nicolson finite difference-relaxation scheme in order to compare the evolution of solutions of B/B systems with that of the solutions of their associated KdV one-way approximations. In section \ref{sec5} of the paper at hand, the ode ivp's resulting from the spectral semidiscretizations of the ivp's for some of the B/B systems (\ref{BB2}) are discretized in time by a three-stage, fourth-order accurate diagonally implicit RK method of composition type, \cite{Yoshida1990,FrutosS1992}. This scheme is effective for nonlinear dispersive wave problems and was analyzed recently in the case of the KdV equation and other related models in \cite{DD}. With this fully discrete method in hand we study computationally the temporal evolutions emanating from small and larger perturbations of initial classical and generalized solitary waves. We also investigate overtaking and head-on collisions of solitary waves and the resolution of initial profiles into sequences of solitary waves. We close the paper with a section of concluding remarks.

In the paper we denote the inner product, resp. norm, on $L^{2}=L^{2}(0,1)$ by $(\cdot,\cdot)$, $||\cdot ||$, resp. For real $\mu\geq 0$ we denote the $L^{2}-$based periodic Sobolev spaces on $[0,1]$ by $H^{\mu}$; for $g\in H^{\mu}$ its $H^{\mu}$ norm will be given by 
$$||g||_{\mu}=\left(\sum_{k\in\mathbb{Z}}(1+k^{2})^{\mu}|\widehat{g}(k)|^{2}\right)^{1/2},
$$ where $\widehat{g}(k)$ is the $k$th Fourier coefficient of $g$. We let $|\cdot |_{\infty}$ resp. $||\cdot ||_{j,\infty}$ be the norm on $L^{\infty}$, resp. $W^{j,\infty}$, on $(0,1)$, where for $1\leq p\leq\infty$ $W^{\mu,p}=W^{\mu,p}(0,1)$ is the Sobolev space of periodic functions on $(0,1)$ of order $\mu$, whose generalized derivatives are in $L^{p}$.

\section{Derivation and well-posedness}
\label{sec2}
In this section we present a derivation of the systems (\ref{BB2}) which is more classical than that of \cite{BonaLS2008} . We also review results of linear and nonlinear well-posedness of the systems (\ref{BB2}) based on the analogous theory of \cite{BonaChS2002,BonaChS2004} valid for surface-wave Boussinesq systems. Finally, we note some invariant functionals of the solutions of these systems.
\subsection{Alternative derivation}
\label{sec21}
The derivation of (\ref{BB1}) (and (\ref{BB2})) made in \cite{BonaLS2008} is based on the introduction of two nonlocal operators linking the velocity potentials associated with the two layers at the interface, and the reformulation of the Euler system for internal waves in terms of them, \cite{Z}. Then the combination of the assumptions of the Boussinesq regime in both fluids and suitable asymptotic expansions of these nonlocal operators lead to (\ref{BB1}) and its 1D version (\ref{BB2}).

An alternative derivation for (\ref{BB2}) with $a=b=c=0, d=\beta=S(\gamma,\delta)$ was made in \cite{Duran2019} using asymptotic expansions of the velocity potentials associated with the upper and lower fluid layers at the interface (cf. \cite{BonaChS2002} for the case of surface waves). Here, we summarize this procedure and derive the general $(a,b,c,d)$-system (\ref{BB2}) from it. If we take  into account the assumptions on the B/B regime, expressed in terms of the parameters (\ref{BB21b}) due to the conditions (\ref{BB21bb}), the alternative strategy leads, cf. \cite{Duran2019}, to the scaled system
\begin{eqnarray}
\zeta_{t}+\frac{1}{\delta+\gamma}\partial_{x}(1+\beta\epsilon\partial_{xx})u+\epsilon \kappa_{\gamma,\delta}\partial_{x}(\zeta u)&=&O(\epsilon^{2}),\nonumber\\
u_{t}+(1-\gamma)\zeta_{x}+\frac{\epsilon}{2}\kappa_{\gamma,\delta}\partial_{x}(u^{2})&=&O(\epsilon^{2}),\label{l213b}
\end{eqnarray}
 in which we took $\epsilon=\mu$,
where 
\begin{eqnarray}
\kappa_{\gamma,\delta}=\frac{\delta^{2}-\gamma}{(\delta+\gamma)^{2}},\label{kap}
\end{eqnarray} 
and ${v}_{\beta}=(1-\beta\epsilon\partial_{xx})^{-1}{u}$.. From (\ref{l213b}), the system (\ref{BB2}) with $a=b=c=0, d=\beta=S(\gamma,\delta)$ is obtained formally if we note that if ${v}_{\beta}=(I-\beta\epsilon \partial_{xx})^{-1}{u}$ then  $v_{\beta}=(1+\beta\epsilon\partial_{xx})u+O(\epsilon^{2})$ and drop the $O(\epsilon^{2})$ terms. The resulting system is of the form
\begin{eqnarray}
\zeta_{t}+\frac{1}{\gamma+\delta}\partial_{x}{v}_{\beta}
+\epsilon \kappa_{\gamma,\delta}\partial_{x} \left(\zeta{ v}_{\beta}\right)=O(\epsilon^{2}),&&\nonumber\\
(1-\beta\epsilon \partial_{xx})({v}_{\beta})_{t}+(1-\gamma)\partial_{x}\zeta+\frac{\epsilon}{2}\kappa_{\gamma,\delta}\partial_{x} {v}_{\beta}^{2}=O(\epsilon^{2}).&&\label{bbs1}
\end{eqnarray}
From (\ref{bbs1}) we may obtain the $(a,b,c,d)$-system (\ref{BB2}) by using the BBM trick, \cite{BBM1972, BonaLS2008}, i.~e. the observation that the first equation in  (\ref{bbs1}) implies that
\begin{eqnarray*}
\zeta_{t}=-\frac{1}{\gamma+\delta}u_{x}+O(\epsilon).
\end{eqnarray*}
From the second equation of (\ref{bbs1}) we also have
\begin{eqnarray*}
\partial_{t}v_{\beta}=-(1-\gamma)\zeta_{x}+O(\epsilon).
\end{eqnarray*}
Thus, introducing the modelling parameters $\alpha_{1}\geq 0, \alpha_{2}\leq 1$, substituting the expression
\begin{eqnarray*}
u_{x}=(1-\alpha_{1})u_{x}-\alpha_{1}(\delta+\gamma)\zeta_{t}+O(\epsilon),
\end{eqnarray*}
in the third-order derivative term in the first equation of (\ref{l213b}) and using
\begin{eqnarray*}
\partial_{t}v_{\beta}=(1-\alpha_{2})\partial_{t}v_{\beta}-\alpha_{2}(1-\gamma)\zeta_{x}+O(\epsilon),
\end{eqnarray*}
in the third-order derivative term of the second equation in (\ref{bbs1}), we obtain, using the definition of the parameters $a, b, c, d$ given in (\ref{BB1b}), that
\begin{eqnarray}
(1-b\epsilon\partial_{xx})\zeta_{t}+\frac{1}{\gamma+\delta}\partial_{x}{v}_{\beta}
+\epsilon \kappa_{\gamma,\delta}\partial_{x} \left(\zeta{ v}_{\beta}\right)+a\epsilon\partial_{xxx}v_{\beta}=O(\epsilon^{2}),&&\nonumber\\
(1-d\epsilon \partial_{xx})\partial_{t}{v}_{\beta}+(1-\gamma)\partial_{x}\zeta+\frac{\epsilon}{2}\kappa_{\gamma,\delta}\partial_{x} {v}_{\beta}^{2}+(1-\gamma)\epsilon c\partial_{xxx}\zeta=O(\epsilon^{2}).&&\label{bbs2}
\end{eqnarray}
Dropping the $O(\epsilon^{2})$ terms, we see that the system (\ref{bbs2}) is the scaled version of (\ref{BB2}).

Another form of the system (\ref{BB2}) is related to the property of the Galilean 
symmetry present in the Euler system for internal waves. The techniques introduced in \cite{DDM2} yields the system

\begin{eqnarray}
(1-b\partial_{xx})\zeta_{t}+\frac{1}{\gamma+\delta}\partial_{x} v_{\beta}+\left(\frac{\delta^{2}-\gamma}{(\delta+\gamma)^{2}}\right)\partial_{x} \left(\zeta v_{\beta}\right)+a\partial_{xxx} v_{\beta}&&\nonumber\\
-b\kappa_{\gamma,\delta}v_{\beta}\partial_{xxx}\zeta=0,\nonumber&&\\
(1-d\partial_{xx})( v_{\beta})_{t}+(1-\gamma)\partial_{x}\zeta+\left(\frac{\delta^{2}-\gamma}{2(\delta+\gamma)^{2}}\right)\partial_{x} v_{\beta}^{2}+(1-\gamma)c\partial_{xxx}\zeta&&\nonumber\\
-d\kappa_{\gamma,\delta}v_{\beta}\partial_{xxx}v_{\beta}=0,&&\label{BB31}
\end{eqnarray}
which can be shown to be invariant under the Galilean transformation
\begin{eqnarray*}
x\mapsto x-\kappa_{\gamma,\delta}Ct,\;t\mapsto t,\; \zeta\mapsto \zeta,\; v_{\beta}\mapsto v_{\beta}+C,
\end{eqnarray*}
for $C\in\mathbb{R}$. Compared to (\ref{BB2}), the new terms in (\ref{BB31}), given by $-b\kappa_{\gamma,\delta}v_{\beta}\partial_{xxx}\zeta$ and $-d\kappa_{\gamma,\delta}v_{\beta}\partial_{xxx}v_{\beta}$, are in scaled variables of $O(\epsilon^{2})$. In this sense, the Euler system for internal waves will lose the consistency with (\ref{BB31}) in the sense specified in Theorem 3 of \cite{BonaLS2008} for the system (\ref{BB2}). Note also that, since the new terms are nonlinear, (\ref{BB31}) and (\ref{BB2}) share the same linear well-posedness theory.

\subsection{Well-posedness theory} 
\label{sec22}
The associated to (\ref{BB2}) linearized system, written in terms of $\zeta$ and $u=(I-\beta\partial_{xx}){v}_{\beta}$ is
\begin{eqnarray}
(1-b\partial_{xx})\zeta_{t}+\frac{1}{\gamma+\delta}\partial_{x} (I-\beta\partial_{xx})^{-1}{u}+a(1-\beta\partial_{xx})^{-1}\partial_{xxx} u&=&0,\nonumber\\
(1-d\partial_{xx})(1-\beta\partial_{xx})^{-1}{u}_{t}+(1-\gamma)\partial_{x}\zeta+(1-\gamma)c\partial_{xxx}\zeta&=&0.\label{BB3}
\end{eqnarray}
The Fourier transform leads to the system
\begin{eqnarray}
\frac{d}{dt}\begin{pmatrix}\widehat{\zeta}(k,t)\\\widehat{u}(k,t)\end{pmatrix}+(ik)A(k)\begin{pmatrix}\widehat{\zeta}(k,t)\\\widehat{u}(k,t)\end{pmatrix}=0,\label{BB4}
\end{eqnarray}
where $k\in\mathbb{R}, t\geq 0$, a circumflex denotes the Fourier transform, and
\begin{eqnarray*}
A(k)=\begin{pmatrix}0&\omega_{1}(k)\\\omega_{2}(k)&0\end{pmatrix},\label{BB4b}
\end{eqnarray*}
where
\begin{eqnarray*}
\omega_{1}(k)=\frac{\left(\frac{1}{\delta+\gamma}-ak^{2}\right)}{(1+bk^{2})(1+\beta k^{2})},\quad \omega_{2}(k)=\frac{(1-\gamma)(1-ck^{2})(1+\beta k^{2})}{1+dk^{2}}.\nonumber
\end{eqnarray*}
The study of (\ref{BB3}) (or (\ref{BB4})) can be done in a similar way to that of \cite{BonaChS2002}. If
\begin{eqnarray*}
\omega_{1}(k)\omega_{2}(k)=\frac{\left(\frac{1}{\delta+\gamma}-ak^{2}\right)(1-\gamma)(1-ck^{2})}{(1+bk^{2})(1+dk^{2})}\geq 0,
\end{eqnarray*}
then the ivp for the linearized system (\ref{BB4}) is well posed if the matrix 
\begin{eqnarray*}
e^{-ikA(k)t}=\begin{pmatrix} \cos(k\sigma(k)t)&-i\frac{\omega_{1}(k)}{\sigma(k)}\sin(k\sigma(k)t)\\-i\frac{\omega_{2}(k)}{\sigma(k)}\sin(k\sigma(k)t)& \cos(k\sigma(k)t)\end{pmatrix},
\end{eqnarray*}
where $\sigma=\sqrt{\omega_{1}\omega_{2}}$, has elements which are bounded for bounded intervals of $k$. This holds when $\omega_{1}/\omega_{2}$ has neither poles nor zeros on the real axis. Since
\begin{eqnarray*}
\frac{\omega_{1}(k)}{\omega_{2}(k)}=\frac{\left(\frac{1}{\delta+\gamma}-ak^{2}\right)(1+dk^{2})}{(1+\beta k^{2})^{2}(1-\gamma)(1-ck^{2})(1+bk^{2})},\label{BB5}
\end{eqnarray*}
and $\beta\geq 0, 0<\gamma<1$ and $\delta>0$, this is equivalent to ask that the rational function
\begin{eqnarray*}
\frac{\left(\frac{1}{\delta+\gamma}-ak^{2}\right)(1+dk^{2})}{(1-ck^{2})(1+bk^{2})},
\end{eqnarray*}
have no poles nor zeros for $k\in\mathbb{R}$. This leads to the
three \lq admissible\rq\ cases,
\begin{itemize}
\item[(C1)] $a,c\leq 0, b,d\geq 0$.
\item[(C2)] $b,d\geq 0, c=a(\delta+\gamma)>0$.
\item[(C3)] $b=d<0, c=a(\delta+\gamma)>0$.
\end{itemize}
We note that
\begin{eqnarray*}
\alpha_{1}=\frac{3\delta(\delta+\gamma)}{1+\gamma\delta}b,\quad \beta=c+d,\quad \alpha_{2}=\frac{c}{c+d},
\end{eqnarray*}
and observe that (C3) does not satisfy the hypotheses $\alpha_{2}\leq 1, \alpha_{1}\geq 0$, while (C2) requires $0<\alpha_{1}<1, 0<\alpha_{2}\leq 1, \beta>0$. In the present paper only the case (C1) will be considered.
%
%
%. As for (C2), condition $c=a(\delta+\gamma)$ implies $\alpha_{1}\geq 1$; but then , if $\beta\geq 0$, we have $$a(\delta+\gamma)=\left(\frac{1-\alpha_{1}}{\alpha_{1}}\right)b-\beta\leq 0$$ and the second part of (C2) cannot be satisfied. In summary, under the hypotheses on $\alpha_{1},\alpha_{2}$ and $\beta$ given above, only (C1) is possible. (A different question is if it would be worth studying the cases (C2), (C3) by relaxing the hypotheses in (\ref{BB1b}), (\ref{BB1c}).)

If we recall that the order of $\sigma(k)$ is the integer $l$ such that
$$\sigma(k)\approx |k|^{l}, \quad |k|\rightarrow\infty,$$ then Theorem 3.2 of \cite{BonaChS2002} can be applied to prove that if $m_{1}=\max\{0,-l\}, m_{2}=\max\{0,l\}$ then the ivp for the linear system (\ref{BB3}) is well posed  for $(\zeta, u)$ in $H^{s+m_{1}}\times H^{s+m_{2}}$ for any $s\geq 0$. 
For example, 
%Contrary to the systems for surface waves (Remark 3.4 in \cite{BonaChS2002}), here having order $-2$ is possible due to the term $(1+\beta k^{2})^{-2}$. Indeed, the discussion for the case $\beta=0$ reduces to that of the Boussinesq systems for surface waves. 
when $\beta>0$ and $s=0$, in the case (C1) we have:
\begin{itemize}
\item $a=c=0, b,d>0$. In this case (\ref{BB3}) is well posed in $H^{2}\times L^{2}$.
\item $b,d>0, a=0, c<0$; then (\ref{BB3}) is well posed in $H^{1}\times L^{2}$. 
\item $b=d=0, a, c<0$; then (\ref{BB3}) is well posed in $L^{2}\times H^{2}$.
\end{itemize}
%Note that when $\beta>0$ and under (C1), the parameter $l$ can go from $l=-4$ to $l=0$.
%\begin{remark}
%Similar studies to those of \cite{BonaChS2004} for surface waves may give results on well-posedness in the nonlinear case.
%\end{remark}
As already mentioned,  the case $\beta=0$ leads to conditions for the linear well-posedness for Boussinesq systems for surface waves, \cite{BonaChS2002}. 
\begin{remark}
\label{remark22}
It is to be noted that not all cases described by the set $a, c\leq 0, b, d\geq 0$ are relevant for the internal wave problem, due to the restrictions on the physical parameters $\gamma,\delta$ and the modelling parameters , $\alpha_{1}, \alpha_{2}, \beta$ that determine $a, b, c$, and $d$, cf. (\ref{BB1b})-(\ref{BB1c}). Specifically, the case $a=0, c<0, b>0, d=0$ should be excluded since $d=0$ implies either $\beta=0$ or $\alpha_{2}=1$ and in either case $c<0$ cannot hold. Arguing similarly we may see that all cases with $b=d=0, a,c\leq 0$ are not valid for internal waves. In addition, note that several other cases hold, under easily checked conditions between the parameters. (For example, if two of the four parameters are zero, then (\ref{BB1c}) implies an affine relation between the other two.)
\end{remark}
As far as local in time well-posedness of the full nonlinear system is concerned, the analysis made in \cite{BonaChS2004} for the case of surface waves can also be used here. (This was confirmed in \cite{A}.) Let us consider the systems corresponding to those cases among the set of parameters $b,d\geq 0, a,c\leq 0$ that are relevant for internal waves. In each case of the following list, we mention the corresponding theorem of \cite{BonaChS2004} that applies. All the results concern existence, uniqueness, and regularity locally in $t$ of the corresponding solution in the  appropriate pairs of Sobolev spaces shown.
\begin{itemize}
\item Case (i): $b,d>0, a=c=0$ (systems of \lq BBM-BBM\rq\ type; Theorem 2.1, $H^{s}\times H^{s}, s\geq 0$).
\item Case (ii): $b,d>0, a,c<0$ (\lq generic\rq\ case; Theorem 2.5, $H^{s}\times H^{s}, s\geq 0$).
\item Case (iii): $b=0, d>0, a,c<0$ (Theorem 3.5, $H^{s}\times H^{s+1}, s\geq 1$). 
\item Case (iv): $b=0, d>0, a=c=0$ (\lq classical\rq\ Boussinesq system; Theorem 3.3, $H^{s}\times H^{s+1}, s\geq 1$, conditional global existence; see also \cite{Duran2019}), or
$b>0, d=0, a=c=0$ (analogous theory).
\item Case (v): $b,d>0, a=0, c<0$ (Bona-Smith system; Theorem 2.6, $H^{s+1}\times H^{s}, s\geq 0$, conditional global existence), or $b,d>0, a<0, c=0$ (analogous theory).
\item Case (vi): $b=0, d>0, a<0, c=0$ (Theorem 3.1, $H^{s}\times H^{s+2}, s\geq 1$).
\item Case (vii): $b>0, d=0, a<0, c=0$ (Theorem 3.9, $H^{s}\times H^{s}, s\geq 2$) or $b=0, d>0, a=0, c<0$ (analogous theory).
\end{itemize}
Note that slightly sharper regularity results were achieved in \cite{A} for some of these cases.
\subsection{Conserved quantities}
\label{sec23}
It is not hard to show that the linear functionals
\begin{eqnarray*}
M_{1}(\zeta)=\int_{-\infty}^{\infty} \zeta dx,\quad
M_{2}(v_{\beta})=\int_{-\infty}^{\infty}v_{\beta} dx=\int_{-\infty}^{\infty}(1-\beta\partial_{xx})^{-1}{u} dx.
\end{eqnarray*}
are invariant quantities during the evolution of solutions of (\ref{BB2}).
When $b=d$ (cf. the surface wave case, \cite{BonaChS2004}) we have the conserved functionals 
\begin{eqnarray}
I(\zeta,u)&=&\int_{-\infty}^{\infty}(\zeta v_{\beta}+b\partial_{x}\zeta \partial_{x}v_{\beta})dx,\label{mom}\\
H(\zeta,u)&=&\int_{-\infty}^{\infty}\left(\frac{(1-\gamma)}{2}\zeta^{2}+\frac{1}{2(\delta+\gamma)}v_{\beta}^{2}-a(\partial_{x}v_{\beta})^{2}-(1-\gamma)c(\partial_{x}\zeta)^{2}\right.\nonumber\\
&&\left. +\frac{\delta^{2}-\gamma}{2(\delta+\gamma)^{2}}\zeta v_{\beta}^{2}\right)dx,\label{energy}
\end{eqnarray}
(where ${v}_{\beta}=(1-\beta\partial_{xx})^{-1}{u}$) with the Hamiltonian structure for (\ref{BB2}) given by 
\begin{eqnarray}
&&\frac{d}{dt}\begin{pmatrix}\zeta\\u\end{pmatrix}=J\frac{\delta H}{\delta(\zeta,u)},\label{BB5b}\\
&&J=-\partial_{x}(1-\beta\partial_{xx})^{-1}\begin{pmatrix}(1-b\partial_{xx})&0\\0&(1-d\partial_{xx})\end{pmatrix}\begin{pmatrix} 0&-1\\1&0\end{pmatrix},\nonumber
\end{eqnarray}
where $\delta H / \delta (\zeta,u)$ stands for the variational derivative with respect to the variables $(\zeta,u)$. All these conservation laws as well as the Hamiltonian structure hold in suitable function spaces.
%%%%%%%%%%%%%%%%%%%%%%%%%%%%%%%
%%%%%%%%%%%%%%%%%%%%%%%%%%%%%%%
\section{Error estimates for a spectral semidiscretization of the periodic initial-value problem}
\label{sec3}
We consider the periodic initial-value problem (ivp) for the one-dimensional system (\ref{BB2}) on the spatial interval $[0,1]$. In order to simplify notation we denote $v_{\beta}=u, \lambda=\kappa_{\gamma,\delta}=\frac{\delta^{2}-\gamma}{(\delta+\gamma)^{2}}, \kappa_{1}=\frac{1}{\delta+\gamma}, \kappa_{2}=1-\gamma$. (Thus $\kappa_{1}$ and $\kappa_{2}$ are positive constants.) We also let $c'$ denote  the constant $(1-\gamma)c$ multiplying the term $\partial_{xxx}\zeta$ in the second pde of (\ref{BB2}); this does not change the sign of the original $c$. Thus, given $u_{0}(x), \zeta_{0}(x)$, $1$-periodic real functions, we seek for $0\leq t\leq T, \zeta(x,t), u(x,t)$, $1$-periodic in $x$, satisfying, for $0\leq x\leq 1, 0\leq t\leq T$
\begin{eqnarray}
\zeta_{t}+\kappa_{1}u_{x}+\lambda \left(\zeta u\right)_{x}+au_{xxx} -b\zeta_{xxt}&=&0,\nonumber\\
u_{t}+\kappa_{2}\zeta_{x}+\lambda uu_{x}+c'\zeta_{xxx} -du_{xxt}&=&0,
\label{dds31}
\end{eqnarray}
with
\begin{eqnarray}
\zeta(x,0)=\zeta_{0}(x),\; u(x,0)=u_{0}(x),\;0\leq x\leq 1.\label{dds32}
\end{eqnarray}
In the sequel we assume that the ivp (\ref{dds31})-(\ref{dds32}) has a unique solution which is smooth enough for the purposes of error estimation.

We will discretize the ivp (\ref{dds31})-(\ref{dds32}) in space by a spectral Fourier Galerkin method. To this end we let $N\geq 1$ be an integer and define the finite dimensional space $S_{N}$ as
\begin{eqnarray*}
S_{N}={\rm span} \{e^{ikx},\; k\in\mathbb{Z}, -N\leq k\leq N\}.
\end{eqnarray*}
Let $P=P_{N}$ denote the $L^{2}$-projection operator onto $S_{N}$ given explicitly for $v\in L^{2}$ by
\begin{eqnarray*}
Pv=\sum_{|k|\leq N}\widehat{v}(k)e^{ikx},
\end{eqnarray*}
where $\widehat{v}(k)$ is the $k$th Fourier coefficient of $v$. It is obvious that $P$ commutes with $\partial_{x}$. Moreover, given integers $0\leq j\leq \mu$, there exists a constant $C$ independent of $N$, such that for any $v\in H^{\mu}, \mu\geq 1$,
\begin{eqnarray}
||v-Pv||_{j}&\leq &CN^{j-\mu}||v||_{\mu},\label{dds33}\\
|v-Pv|_{\infty}&\leq &CN^{1/2-\mu}||v||_{\mu}.\label{dds34}
\end{eqnarray}
In addition, the following inverse inequalities hold in $S_{N}$: Given $0\leq j\leq \mu$, there exists a constant $C$ independent of $N$, such that for any $\psi\in S_{N}$
\begin{eqnarray}
||\psi||_{\mu}\leq CN^{\mu-j}||\psi||_{j},\; 
||\psi||_{\mu,\infty}\leq CN^{1/2+\mu-j}||\psi||_{j}.\label{dds35}
\end{eqnarray}
In what follows, as is customary, we will denote constants independent of $N$ by $C$.

The spectral Galerkin semidiscretization of the ivp (\ref{dds31})-(\ref{dds32}) is defined as follows. Let $T>0$. We seek real-valued $\zeta_{N}, u_{N}:[0,T]\rightarrow S_{N}$ satisfying for $0\leq t\leq T$ and $\forall \varphi,\chi\in S_{N}$
\begin{eqnarray}
(\zeta_{Nt},\varphi)+\kappa_{1}(u_{Nx},\varphi)+\lambda ((\zeta_{N}u_{N})_{x},\varphi)+a(u_{Nxxx},\varphi)-b(\zeta_{Nxxt},\varphi)=0,&&\label{dds36}\\
(u_{Nt},\chi)+\kappa_{2}(u_{Nx},\chi)+\lambda (u_{N}u_{Nx},\chi)+c'(\zeta_{Nxxx},\chi)-d(u_{Nxxt},\chi)=0,&&\label{dds37}
\end{eqnarray}
and for $t=0$
\begin{eqnarray}
\zeta_{N}(0)=P\zeta_{0},\; u_{N}(0)=Pu_{0}.\label{dds38}
\end{eqnarray}
The ode ivp (\ref{dds36})-(\ref{dds38}) has a unique solution locally in time and has the Fourier implementation
\begin{eqnarray}
(1+bk^{2})\widehat{\zeta}_{N,t}+ik \kappa_{1}\widehat{u}_{N}+ik\lambda\widehat{\zeta_{N}u_{N}}-ik^{3}a\widehat{u}_{N}&=&0,\nonumber\\
(1+dk^{2})\widehat{u}_{N,t}+ik \kappa_{2}\widehat{\zeta}_{N}+\frac{ik}{2}\lambda\widehat{u_{N}^{2}}-ik^{3}c'\widehat{\zeta}_{N}&=&0,\label{38b}
\end{eqnarray}
where $\widehat{\zeta}_{N}=\widehat{\zeta}_{N}(k,t), \widehat{u}_{N}=\widehat{u}_{N}(k,t), -N\leq k\leq N, t\geq 0$ are the Fourier coefficients of $\zeta_{N}, u_{N}$ with initial values $\widehat{\zeta}_{N}(k,0)=\widehat{\zeta}_{0}(k), \widehat{u}_{N}(k,0)=\widehat{u}_{0}(k)$.

In order to estimate the error of the semidiscretization let $\theta=\zeta_{N}-P\zeta, \rho=P\zeta-\zeta$, so that $\zeta_{N}-\zeta=\theta+\rho$, and $\xi=u_{N}-Pu, \sigma=Pu-u$, so that $u_{N}-u=\xi+\sigma$. Then, substracting the first pde in (\ref{dds31}) from (\ref{dds36}) we obtain, while the solution of (\ref{dds36})-(\ref{dds38}) exists and for all $\varphi\in S_{N}$
\begin{eqnarray*}
(\theta_{t},\varphi)+\kappa_{1}(\xi_{x},\varphi)+a(\xi_{xxx},\varphi)-b(\theta_{xxt},\varphi)=-\lambda ((\zeta_{N}u_{N})_{x},\varphi)+\lambda((\zeta u)_{x},\varphi).
\end{eqnarray*}
Therefore for $\varphi\in S_{N}$
\begin{eqnarray}
(\theta_{t},\varphi)+a(\xi_{xxx},\varphi)-b(\theta_{xxt},\varphi)=-\kappa_{1}(\xi_{x},\varphi)-(A_{x},\varphi),\label{dds39}
\end{eqnarray}
where
\begin{eqnarray*}
A=\lambda (\zeta_{N}u_{N}-\zeta u)=\lambda \left((\zeta+\theta+\rho)(u+\xi+\sigma)-\zeta u)\right),
\end{eqnarray*}
i.~e.
\begin{eqnarray}
A=\lambda\left(u\rho+\zeta \sigma+u\theta+\zeta\xi+\sigma\theta+\rho\xi+\rho\sigma+\theta\xi\right).\label{dds310}
\end{eqnarray}
Substracting the second pde in (\ref{dds31}) from (\ref{dds37}) we get for $\chi\in S_{N}$
\begin{eqnarray*}
(\xi_{t},\chi)+\kappa_{2}(\theta_{x},\chi)+c'(\theta_{xxx},\chi)-d(\xi_{xxt},\chi)=-\lambda (u_{N}u_{Nx},\chi)+\lambda(u u_{x},\chi).
\end{eqnarray*}
Therefore, for $\chi\in S_{N}$
\begin{eqnarray}
(\xi_{t},\chi)+c'(\theta_{xxx},\chi)-d(\xi_{xxt},\chi)=-\kappa_{2}(\theta_{x},\chi)-(B_{x},\chi),\label{dds311}
\end{eqnarray}
where
\begin{eqnarray}
B=\lambda\left(u\sigma+u\xi+\sigma\xi+\frac{1}{2}(\sigma^{2}+\xi^{2})\right).\label{dds312}
\end{eqnarray}
Using the error equations (\ref{dds39})-(\ref{dds312}) we proceed now to derive error estimates for the semidiscrete schemes (\ref{dds36})-(\ref{dds38}).

For the purpose of the error analysis, we consider the same seven cases of nonlinearly well posed systems identified in section \ref{sec22}. For simplicity, in the  cases (iv), (v) and (vii) the following systems will be analyzed; the others are similar:
\begin{itemize}
\item Case (iv): \lq Classical Boussinesq\rq\ case: $b=0, d>0, a=c=0$. 
\item Case (v): \lq Bona-Smith\rq\ systems: 
$b, d>0, a=0, c<0$. 
\item Case (vii): $b>0, d=0, a<0, c=0$.
\end{itemize}
As mentioned in the Introduction, the systems of cases (i) (BBM-BBM) and (v) (\lq Bona-Smith\rq) have been discretized by a collocation spectral method in space an analyzed by Xavier et al., \cite{X}, in the case of surface waves. The error estimates obtained in \cite{X} are similar to those that we obtain below for the spectral Galerkin method in these cases but we include the proofs as our techniques are somewhat different.

In all propositions below we assume for simplicity that $\zeta, u\in C^{1}(0,T,H^{\mu}), \mu\geq 1$, and specify in each case the least integer $\mu$ needed for the validity of the error estimates. In all cases it is clear that $\zeta_{N}, u_{N}$ satisfy (\ref{dds36})-(\ref{dds38}) at least locally in $t$; part of the proof is checking that they exist uniquely and satisfy (\ref{dds36})-(\ref{dds37}) up to $t=T$.
\begin{proposition}
\label{propo31}
Let $a,b,c,d$ as in case (i). If $\mu\geq 1$ then
\begin{eqnarray}
\max_{0\leq t\leq T}\left(||\zeta_{N}-\zeta||+||u_{N}-u||\right)\leq C N^{-\mu}.\label{dds313}
\end{eqnarray}
\end{proposition}
\begin{proof}
While the semidiscrete solution $\zeta_{N}, u_{N}$ exists, putting $\varphi=\theta$ in (\ref{dds39}), $\chi=\xi$ in (\ref{dds311}), using integration by parts and adding the resulting equations give
\begin{eqnarray}
\frac{1}{2}\frac{d}{dt}\left(||\theta||^{2}+||\xi||^{2}+b||\theta_{x}||^{2}+d||\xi_{x}||^{2}\right)&=&\kappa_{1}(\xi,\theta_{x})+\kappa_{2}(\theta,\xi_{x})\nonumber\\
&&+(A,\theta_{x})+(B,\xi_{x}).\label{dds314}
\end{eqnarray}
We estimate as follows the various terms in the right-hand side of (\ref{dds314}). First
\begin{eqnarray}
|(\xi,\theta_{x})|\leq ||\xi||||\theta_{x}||\leq \frac{1}{2}\left(||\xi||^{2}+||\theta_{x}||^{2}\right).\label{dds315a}
\end{eqnarray}
For the various terms of $(A,\theta_{x})$ we first see, using (\ref{dds33}), 
\begin{eqnarray}
|(u\rho,\theta_{x})|\leq |u|_{\infty}||\rho||||\theta_{x}||\leq C\left(N^{-2\mu}+||\theta_{x}||^{2}\right).\label{dds315b}
\end{eqnarray}
Similarly,
\begin{eqnarray}
|(\zeta\sigma,\theta_{x})|\leq |\zeta|_{\infty}||\sigma||||\theta_{x}||\leq C\left(N^{-2\mu}+||\theta_{x}||^{2}\right).\label{dds315c}
\end{eqnarray}
Now
\begin{eqnarray}
|(u\theta,\theta_{x})|&\leq&|u|_{\infty}||\theta|| ||\theta_{x}||\leq C||\theta||_{1}^{2},\label{dds15d}\\
|(\zeta\xi,\theta_{x})|&\leq &|\zeta|_{\infty}||\xi||||\theta_{x}||\leq C(||\xi||^{2}+||\theta_{x}||^{2}).\label{dds15e}
\end{eqnarray}
Using (\ref{dds34}) we see that $|\sigma|_{\infty}\leq C$ and therefore
\begin{eqnarray}
|(\sigma\theta,\theta_{x})|\leq |\sigma|_{\infty}||\theta||||\theta_{x}||\leq C||\theta||_{1}^{2}.\label{dds315f}
\end{eqnarray}
Similarly,
\begin{eqnarray}
|(\rho\xi,\theta_{x})|&\leq &|\rho|_{\infty}||\xi||||\theta_{x}||\leq C(||\xi||^{2}+||\theta_{x}||^{2}).\label{dds315g}\\
|(\rho\sigma,\theta_{x})|&\leq &|\rho|_{\infty}||\sigma||||\theta_{x}||\leq C(N^{-2\mu}+||\theta_{x}||^{2}).\label{dds315h}
\end{eqnarray}
Since $\theta(0)=0$, using continuity, let $t_{N}, 0<t_{N}\leq T$, be the maximal time for which the solution of (\ref{dds36})-(\ref{dds38}) exists and satisfies
\begin{eqnarray}
|\theta|_{\infty}\leq 1,\; 0\leq t\leq t_{N}.\label{dds316}
\end{eqnarray}
Then for $0\leq t\leq t_{N}$
\begin{eqnarray}
|(\theta\xi,\theta_{x})|\leq  |\theta|_{\infty}||\xi||||\theta_{x}||\leq ||\xi||||\theta_{x}||\leq  \frac{1}{2}(||\xi||^{2}+||\theta_{x}||^{2}).\label{dds315i}
\end{eqnarray}
From (\ref{dds315b})-(\ref{dds315i}) we have therefore for $0\leq t\leq t_{N}$ that
\begin{eqnarray}
|(A,\theta_{x})|\leq C\left(N^{-2\mu}+||\xi||^{2}+||\theta||_{1}^{2}\right).\label{dds317}
\end{eqnarray}
For the rest of the terms on the right-hand side of (\ref{dds314}) we first note that
\begin{eqnarray}
|(\theta,\xi_{x})|\leq ||\theta||||\xi_{x}||\leq \frac{1}{2}\left(||\theta||^{2}+||\xi_{x}||^{2}\right).\label{dds318a}
\end{eqnarray}
For the $(B,\xi_{x})$ terms, in view of (\ref{dds312}) we have the following estimates. Note that by (\ref{dds33})
\begin{eqnarray}
|(u\sigma,\theta_{x})|\leq |u|_{\infty}||\sigma||||\theta_{x}||\leq C\left(N^{-2\mu}+||\theta_{x}||^{2}\right).\label{dds318b}
\end{eqnarray}
In addition,
\begin{eqnarray}
|(u\xi,\xi_{x})|\leq |u|_{\infty}||\xi||||\xi_{x}||\leq C||\xi||_{1}^{2}.\label{dds18c}
\end{eqnarray}
By (\ref{dds34})
\begin{eqnarray}
|(\sigma\xi,\xi_{x})|\leq |\sigma|_{\infty}||\xi||||\xi_{x}||\leq C||\xi||_{1}^{2}.\label{dds318d}
\end{eqnarray}
By (\ref{dds33}), (\ref{dds34})
\begin{eqnarray}
|\frac{1}{2}(\sigma^{2},\xi_{x})|\leq \frac{1}{2}|\sigma|_{\infty}||\sigma||||\xi_{x}||\leq C(N^{-2\mu}+||\xi_{x}||^{2}).\label{dds318e}
\end{eqnarray}
And finally, by periodicity,
\begin{eqnarray}
\frac{1}{2}(\xi^{2},\xi_{x})=0.\label{dds318f}
\end{eqnarray}
We conclude from (\ref{dds318b})-(\ref{dds318f}) that, as long as the semidiscrete solution exists,
\begin{eqnarray}
|(B,\xi_{x})|\leq C\left(N^{-2\mu}+||\theta||^{2}+||\xi||_{1}^{2}\right).\label{dds319}
\end{eqnarray}
Hence, since $b,d>0$ we get from (\ref{dds314}), (\ref{dds315a}), (\ref{dds317}), (\ref{dds318a}), (\ref{dds319}) that
\begin{eqnarray*}
\frac{d}{dt}\left(||\theta||_{1}^{2}+||\xi||_{1}^{2}\right)\leq C\left(N^{-2\mu}+||\theta||_{1}^{2}+||\xi||_{1}^{2}\right),\; 0\leq t\leq t_{N}.\label{dds320}
\end{eqnarray*}
Hence, by Gronwall's lemma and (\ref{dds38}) we conclude that for $0\leq t\leq t_{N}$ and for some constant $C=C(T)$ there holds
\begin{eqnarray}
||\theta||_{1}+||\xi||_{1}\leq CN^{-\mu}.\label{dds321}
\end{eqnarray}
Therefore, since $|\theta|_{\infty}\leq C||\theta||_{1}$ by Sobolev's theorem, we conclude by (\ref{dds321}) and our hypothesis on $\mu$ that (for $N$ sufficiently large) $t_{N}$ was not maximal in (\ref{dds316}). Arguing in the customary way we see that $t_{N}$ may be taken equal to $T$, and (\ref{dds321}) holds for $0\leq t\leq T$. By (\ref{dds33}) we conclude that (\ref{dds313}) holds, so that $\zeta_{N},u_{N}$ satisfy optimal-order error estimates in $L^{2}$, where by \lq optimal-order\rq\ in the context of spectral methods we mean that the semidiscrete approximations satisfy estimates like (\ref{dds33}) if $\zeta,u\in H^{\mu}$.
\end{proof}
\begin{proposition}
\label{propo32}
Let $a,b,c,d$ as in case (ii). If $\mu\geq 1$ then
\begin{eqnarray*}
\max_{0\leq t\leq T}\left(||\zeta_{N}-\zeta||+||u_{N}-u||\right)\leq C N^{-\mu}.\label{dds322}
\end{eqnarray*}
\end{proposition}
\begin{proof}
While the semidiscrete solution $\zeta_{N}, u_{N}$ exists, putting $\varphi=\theta$ in (\ref{dds39}), $\chi=\xi$ in (\ref{dds311}) we obtain, using integration by parts, that 
\begin{eqnarray}
\frac{1}{2}\frac{d}{dt}\left(||\theta||^{2}+b||\theta_{x}||^{2}\right)-a(\xi_{xx},\theta_{x})&=&\kappa_{1}(\xi,\theta_{x})+(A,\theta_{x}).\label{dds323a}\\
\frac{1}{2}\frac{d}{dt}\left(||\xi||^{2}+d||\xi_{x}||^{2}\right)+c'(\theta_{x},\xi_{xx})&=&\kappa_{2}(\theta,\xi_{x})+(B,\xi_{x}).\label{dds323b}
\end{eqnarray}
Multiplying (\ref{dds323a}) by $-c'$ and (\ref{dds323b}) by $-a$ and adding the resulting equations we get
\begin{eqnarray*}
&&\frac{1}{2}\frac{d}{dt}\left(|c'|||\theta||^{2}+|a|||\xi||^{2}+b|c'|||\theta_{x}||^{2}+d|a|||\xi_{x}||^{2}\right)=-c'\kappa_{1}(\xi,\theta_{x})\nonumber\\
&&-c'(A,\theta_{x})-a\kappa_{2}(\theta,\xi_{x})-a(B,\xi_{x}).\label{dds324}
\end{eqnarray*}
The rest of the proof proceeds exactly along the lines of that of Proposition \ref{propo31}.
\end{proof}
\begin{proposition}
\label{propo33}
Let $a,b,c,d$ as in case (iii). If $\mu> 3/2$ then
\begin{eqnarray}
\max_{0\leq t\leq T}\left(||\zeta_{N}-\zeta||+||u_{N}-u||_{1}\right)\leq C N^{1-\mu}.\label{dds325}
\end{eqnarray}
\end{proposition}
\begin{proof}
While the semidiscrete solution $\zeta_{N}, u_{N}$ exists, putting $\varphi=\theta$ in (\ref{dds39}), $\chi=\xi$ in (\ref{dds311}) we obtain, using integration by parts, that 
\begin{eqnarray}
\frac{1}{2}\frac{d}{dt}||\theta||^{2}-a(\xi_{xx},\theta_{x})&=&-\kappa_{1}(\xi_{x},\theta)-(A_{x},\theta),\label{dds326a}\\
\frac{1}{2}\frac{d}{dt}\left(||\xi||^{2}+d||\xi_{x}||^{2}\right)+c'(\theta_{x},\xi_{xx})&=&\kappa_{2}(\theta,\xi_{x})+(B,\xi_{x}).\label{dds326b}
\end{eqnarray}
Multiplying (\ref{dds326a}) by $-c'$ and (\ref{dds326b}) by $-a$ and adding the resulting equations gives
\begin{eqnarray}
\frac{1}{2}\frac{d}{dt}\left(|c'|||\theta||^{2}+|a|||\xi||^{2}+d|a|||\xi_{x}||^{2}\right)&=&c'\kappa_{1}(\xi_{x},\theta)+c'(A_{x},\theta)\nonumber\\
&&-a\kappa_{2}(\theta,\xi_{x})-a(B,\xi_{x}).\label{dds327}
\end{eqnarray}
We estimate the terms of the right-hand side of the above. Obviously,
\begin{eqnarray}
|(\xi_{x},\theta)|\leq ||\xi_{x}||||\theta||\leq \frac{1}{2}\left(||\theta||^{2}+||\xi_{x}||^{2}\right).\label{dds328a}
\end{eqnarray}
For the terms of $(A_{x},\theta)$, using (\ref{dds310})
and (\ref{dds33}) and the fact that $H^{1}$ is an algebra, we have
\begin{eqnarray*}
|((u\rho)_{x},\theta)|\leq ||u\rho||_{1}||\theta||\leq C||u||_{1}||\rho||_{1}||\theta||\leq  C\left(N^{2(1-\mu)}+||\theta||^{2}\right).\label{dds328b}
\end{eqnarray*}
Similarly,
\begin{eqnarray}
|((\zeta\sigma)_{x},\theta)|\leq C\left(N^{2(1-\mu)}+||\theta||^{2}\right).\label{dds328c}
\end{eqnarray}
Using integration by parts we have
\begin{eqnarray}
|((u\theta)_{x},\theta)|=|\frac{1}{2}|(u_{x}\theta,\theta)|\leq C|u_{x}|_{\infty}||\theta||^{2}\leq C||\theta||^{2}.\label{dds328d}
\end{eqnarray}
Also
\begin{eqnarray}
|((\zeta\xi)_{x},\theta)|\leq C||\zeta||_{1}||\xi||_{1}||\theta||\leq C\left(||\xi||_{1}^{2}+||\theta||^{2}\right).\label{dds328e}
\end{eqnarray}
Using integration by parts and (\ref{dds34}) and our hypothesis on $\mu$
\begin{eqnarray}
|(\sigma \theta)_{x},\theta)|=\frac{1}{2}|(\sigma_{x}\theta,\theta)|\leq C|\sigma_{x}|_{\infty}||\theta||^{2}\leq  C||\theta||^{2}.\label{dds328f}
\end{eqnarray}
By (\ref{dds33})
\begin{eqnarray}
|(\rho\xi)_{x},\theta)|\leq C||\rho||_{1}||\xi||_{1}||\theta||\leq C\left(||\xi||_{1}^{2}+||\theta||^{2}\right).\label{dds328g}
\end{eqnarray}
By (\ref{dds33}), (\ref{dds34}), and our hypothesis on $\mu$
\begin{eqnarray}
|(\rho\sigma)_{x},\theta)|&\leq &|\rho|_{\infty}||\sigma_{x}||||\theta||+|\sigma|_{\infty}||\rho_{x}||||\theta||\nonumber\\
&\leq & CN^{\frac{3}{2}-2\mu}||\theta||\leq CN^{-\mu}||\theta||\nonumber\\
&\leq & C\left(N^{-2\mu}+||\theta||^{2}\right).\label{dds328h}
\end{eqnarray}
Now, since $\theta(0)=0$, using continuity, let $t_{N}$, $0<t_{N}\leq T$, be the maximal value of $t$ for which the solution of (\ref{dds36})-(\ref{dds38}) exists and satisfies
\begin{eqnarray}
|\theta|_{\infty}\leq 1,\; 0\leq t\leq t_{N}.\label{dds329}
\end{eqnarray}
By (\ref{dds329}) we have for $0\leq t\leq t_{N}$, using integration by parts
\begin{eqnarray}
|(\theta\xi)_{x},\theta)|&=&\frac{1}{2}|(\xi_{x}\theta,\theta)|\leq C|\theta|_{\infty}||\theta||||\xi_{x}||\nonumber\\
&\leq & C\left(||\theta||^{2}+||\xi_{x}||^{2}\right).\label{dds328i}
\end{eqnarray}
We conclude from (\ref{dds328c})-(\ref{dds328i}) that
\begin{eqnarray}
|(A_{x},\theta)|\leq C\left(N^{2(1-\mu)}+||\theta||^{2}+||\xi||_{1}^{2}\right),\; 0\leq t\leq t_{N}.\label{dds330}
\end{eqnarray}
We estimate now $(\theta,\xi_{x})$ and the terms of $(B,\xi_{x})$ exactly as in (\ref{dds318a})-(\ref{dds318f}), and conclude that as long as the semidiscrete approximation exists it holds that
\begin{eqnarray}
|(\theta,\xi_{x})|+|(B,\xi_{x})|\leq C\left(N^{-2\mu}+||\theta||^{2}+||\xi||_{1}^{2}\right).\label{dds331}
\end{eqnarray}
Therefore, by (\ref{dds327}), (\ref{dds328a}), (\ref{dds330}), (\ref{dds331}), since $c', a, d\neq 0$ we see that for $0\leq t\leq t_{N}$
\begin{eqnarray*}
\frac{1}{2}\frac{d}{dt}\left(||\theta||^{2}+||\xi||_{1}^{2}\right)\leq C\left(N^{2(1-\mu)}+||\theta||^{2}+||\xi||_{1}^{2}\right).\label{dds332}
\end{eqnarray*}
By Gronwall's lemma and (\ref{dds38}) we see that for $0\leq t\leq t_{N}$ and a constant $C(T)$ there holds
\begin{eqnarray}
||\theta||+||\xi||_{1}\leq C(T)N^{1-\mu}.\label{dds333}
\end{eqnarray}
Therefore, since $|\theta|_{\infty}\leq CN^{1/2}||\theta||$ by (\ref{dds35}), and using (\ref{dds333}) and the assumption that $\mu>3/2$, we see that $t_{N}$ was not maximal in (\ref{dds329}) if $N$ was sufficiently large. We conclude that (\ref{dds333}) holds up to $t=T$ which implies that
\begin{eqnarray*}
\max_{0\leq t\leq T}\left(||\zeta_{N}-\zeta||+||u_{N}-u||_{1}\right)\leq C N^{1-\mu},
\end{eqnarray*}
i.~e. that the conclusion of the proposition holds. Note that this implies that $u_{N}$ is optimally close to $u$ in $H^{1}$ but $\zeta_{N}$ suboptimally so to $\zeta$ in $L^{2}$.
\end{proof}
\begin{proposition}
\label{propo34}
Let $a,b,c,d$ as in case (iv), and with no loss of generality suppose that $b=0, d>0, a=c=0$. If $\mu> 3/2$ then
\begin{eqnarray*}
\max_{0\leq t\leq T}\left(||\zeta_{N}-\zeta||+||u_{N}-u||_{1}\right)\leq C N^{1-\mu}.\label{dds334}
\end{eqnarray*}
\end{proposition}
\begin{proof}
While the semidiscrete solution $\zeta_{N}, u_{N}$ exists, putting $\varphi=\theta$ in (\ref{dds39}), $\chi=\xi$ in (\ref{dds311}), and using integration by parts, then adding the resulting equations yields
\begin{eqnarray*}
\frac{1}{2}\frac{d}{dt}\left(||\theta||^{2}+||\xi||^{2}+d||\xi_{x}||^{2}\right)=-\kappa_{1}(\xi_{x},\theta)-(A_{x},\theta)+\kappa_{2}(\theta,\xi_{x})+(B,\xi_{x}).
\end{eqnarray*}
We estimate now the terms of $(A_{x},\theta)$ and $(B,\xi_{x})$ exactly as in Proposition \ref{propo33}. The conclusion follows.
\end{proof}
\begin{proposition}
\label{propo35}
Let $a,b,c,d$ as in case (v), and with no loss of generality suppose that $b>0, d>0, a=0, c<0$. If $\mu\geq 1$ then
\begin{eqnarray*}
\max_{0\leq t\leq T}\left(||\zeta_{N}-\zeta||+||u_{N}-u||\right)\leq C N^{-\mu}.\label{dds335}
\end{eqnarray*}
\end{proposition}
\begin{proof}
We write (\ref{dds39}) for $a=0$ as
\begin{eqnarray*}
\theta_{t}-b\theta_{xxt}=-\kappa_{1}\xi_{x}-PA_{x},
\end{eqnarray*}
i.~e. as
\begin{eqnarray}
(1-b\partial_{x}^{2})\theta_{t}=-\kappa_{1}\xi_{x}-PA_{x}.\label{360}
\end{eqnarray}
For a constant $\kappa>0$ let $\mathcal{T}_{\kappa}$ denote the operator $\mathcal{T}_{\kappa}=(I-\kappa\partial_{x}^{2})^{-1}$ which is well defined in $H^{s}$ for any $s\in\mathbb{R}$. Using its Fourier representation we see, for any $f\in H^{j-2}$, that $\mathcal{T}_{\kappa}f\in H^{j}$ and that
\begin{eqnarray}
||\mathcal{T}_{\kappa}f||_{j}\leq C_{k} ||f||_{j-2},\; j\in \mathbb{R},\label{361}
\end{eqnarray}
where $C_{k}$ is a constant depending on $\kappa$. (In the sequel we will only use the property that for $f\in L^{2}$, $||f||_{-j}\leq ||f||, j\geq 0$, for the negative norms.)

Using this notation we write (\ref{360}) as
\begin{eqnarray*}
\theta_{t}=-\kappa_{1}\mathcal{T}_{b}\xi_{x}-\mathcal{T}_{b}PA_{x}.
\end{eqnarray*}
Therefore, since $\partial_{x}$ commutes with $\mathcal{T}_{b}$ and $P$, (\ref{361}) gives
\begin{eqnarray}
||\theta_{t}||_{1}\leq |\kappa_{1}|||\mathcal{T}_{b}\xi_{x}||_{1}+||\mathcal{T}_{b}PA_{x}||_{1}\leq C\left(||\xi||+||A||\right).\label{362}
\end{eqnarray}
From the definition of $A$, cf. (\ref{dds310}), we see that
\begin{eqnarray}
||A||&\leq & C\left(|u|_{\infty}||\rho||+|\zeta|_{\infty}||\sigma||+|u|_{\infty}||\theta||+|\zeta|_{\infty}||\xi||\right.\nonumber\\
&&\left.+|\sigma|_{\infty}||\theta||+|\rho|_{\infty}||\xi||+|\rho|_{\infty}||\sigma||+||\theta|| |\xi|_{\infty}\right).\label{363}
\end{eqnarray}
Since $\xi(0)=0$, using continuity, let $t_{N}\in (0,T]$ denote the maximal time for which the solution of the semidiscrete ivp exists and satisfies
\begin{eqnarray}
|\xi|_{\infty}\leq 1,\; 0\leq t\leq t_{N}.\label{364}
\end{eqnarray}
Therefore, from (\ref{363}), using the fact that $\mu\geq 1$, Sobolev's inequality, (\ref{dds33}), and (\ref{dds34}), we obtain, in view of (\ref{364}), that
\begin{eqnarray}
||A||\leq C\left(N^{-\mu}+||\theta||+||\xi||\right), \; 0\leq t\leq t_{N}.\label{365}
\end{eqnarray}
Hence, from (\ref{362}) and (\ref{365}) we get
\begin{eqnarray*}
||\theta_{t}||_{1}\leq C\left(N^{-\mu}+||\theta||+||\xi||\right), \; 0\leq t\leq t_{N}.\label{366}
\end{eqnarray*}
We write now (\ref{dds311}) as
\begin{eqnarray*}
(1-d\partial_{x}^{2})\xi_{t}=-c'\theta_{xxx}-\kappa_{2}\theta_{x}-PB_{x},
\end{eqnarray*}
i.~e. as
\begin{eqnarray*}
\xi_{t}=-c'\mathcal{T}_{d}\theta_{xxx}-\kappa_{2}\mathcal{T}_{d}\theta_{x}-\mathcal{T}_{d}PB_{x},
\end{eqnarray*}
from which, using (\ref{361}), we get
\begin{eqnarray}
||\xi_{t}||&\leq & |c'|||\mathcal{T}_{d}\theta_{xxx}||+|\kappa_{2}|||\mathcal{T}_{d}\theta_{x}||+||\mathcal{T}_{d}PB_{x}||\nonumber\\
&\leq & C\left(||\theta||_{1}+||\theta||+||B||\right)\leq C\left(||\theta||_{1}+||B||\right).\label{367}
\end{eqnarray}
From (\ref{dds312}) we see that
\begin{eqnarray*}
||B||\leq C\left(|u|_{\infty}||\sigma||+|u|_{\infty}||\xi||+|\sigma|_{\infty}||\xi||+|\sigma|_{\infty}||\sigma||+|\xi|_{\infty}||\xi||\right).
\end{eqnarray*}
Therefore, since $\mu\geq 1$, from (\ref{dds33}) and (\ref{dds34}), we have, in view of (\ref{364}), that
\begin{eqnarray}
||B||\leq C\left(N^{-\mu}+||\xi||\right),\; 0\leq t\leq t_{N}.\label{368}
\end{eqnarray}
Therefore, by (\ref{367}) and (\ref{368}) it follows that
\begin{eqnarray}
||\theta_{t}||_{1}+||\xi_{t}||\leq C\left(N^{-\mu}+||\theta||_{1}+||\xi||\right), \; 0\leq t\leq t_{N}.\label{369}
\end{eqnarray}
where $C$ does not depend on $t_{N}$. From (\ref{dds38}) we infer that $\theta=\int_{0}^{t}\theta_{\tau}d\tau$, $\xi=\int_{0}^{t}\xi_{\tau}d\tau$. Hence, by (\ref{369}) we have, for $0\leq t\leq t_{N}$
\begin{eqnarray*}
||\theta||_{1}+||\xi||&\leq & \int_{0}^{t}\left(||\theta_{\tau}||_{1}+||\xi_{\tau}||\right)d\tau \nonumber\\
&\leq & C\int_{0}^{t}\left(N^{-\mu}+||\theta||_{1}+||\xi||\right)d\tau\nonumber\\
&\leq & C\int_{0}^{t}\left(||\theta||_{1}+||\xi||\right)d\tau+CTN^{-\mu}.
\end{eqnarray*}
Using Gronwall's lemma in integral form gives that for some $C=C(T)$
\begin{eqnarray}
||\theta||_{1}+||\xi||\leq CN^{-\mu},\; 0\leq t\leq t_{N}.\label{dds350}
\end{eqnarray}
From this we observe, since $\mu\geq 1$, that for $0\leq t\leq t_{N}$ $|\xi|_{\infty}\leq 1$ if $N$ is sufficiently large, and therefore that $t_{N}$ was not maximal in (\ref{364}). We may then take $t_{N}=T$ and conclude from (\ref{dds350}) that
\begin{eqnarray*}
||\theta||_{1}+||\xi||\leq CN^{-\mu},\; 0\leq t\leq T.
\end{eqnarray*}
Therefore, the conclusion of the proposition holds. It implies that $\zeta_{N}$ and $u_{N}$ satisfy optimal-order $L^{2}$-error estimates.
\end{proof}
\begin{proposition}
\label{propo36}
Let $a,b,c,d$ as in case (vi) and  $\mu>3/2$. Then
\begin{eqnarray*}
\max_{0\leq t\leq T}\left(||\zeta_{N}-\zeta||+||u_{N}-u||_{1}\right)\leq C N^{1-\mu}.\label{dds351}
\end{eqnarray*}
\end{proposition}
\begin{proof}
Motivated by an {\em a priori} estimate for this system in Theorem 3.1 of \cite{BonaChS2004}, and putting $\varphi=\theta$ in (\ref{dds39}) and using integration by parts we get, while the semidiscrete approximation exists
\begin{eqnarray}
\frac{1}{2}\frac{d}{dt}||\theta||^{2}=-\kappa_{1}(\xi_{x},\theta)-(A_{x},\theta)+a(\xi_{xx},\theta_{x}).\label{dds352}
\end{eqnarray}
Now putting $\chi=\xi+a\xi_{xx}$ in (\ref{dds311}) and using integration by parts, we get, while the semidiscrete approximation exists,
\begin{eqnarray}
\frac{1}{2}\frac{d}{dt}\left(||\xi||^{2}+(|a|+d)||\xi_{x}||^{2}+|a|d||\xi_{xx}||^{2}\right)&=&-\kappa_{2}(\theta_{x},\xi)-a\kappa_{2}(\theta_{x},\xi_{xx})\nonumber\\
&&-(B_{x},\xi+a\xi_{xx}).\label{dds353}
\end{eqnarray}
%Noting from (\ref{dds312}) that the first term of $-B_{x}$ is equal to $-\kappa_{2}\theta_{x}$, whose inner product with $a\xi_{xx}$ is $-a\kappa_{2}(\theta_{x},\xi_{xx})$,

In order to eliminate the term $(\xi_{xx},\theta_{x})$ from (\ref{dds352}) and (\ref{dds353}) we multiply (\ref{dds352}) by $\kappa_{2}>0$ and add the resulting equation to (\ref{dds353}). In this way we get
\begin{eqnarray}
\frac{1}{2}\frac{d}{dt}\left(\kappa_{2}||\theta||^{2}+||\xi||^{2}+(|a|+d)||\xi_{x}||^{2}+|a|d||\xi_{xx}||^{2}\right)
&=&-\kappa_{1}\kappa_{2}(\xi_{x},\theta)-\kappa_{2}(A_{x},\theta)\nonumber\\
&&-(B_{x},\xi+a\xi_{xx})\nonumber\\
&&-\kappa_{2}(\theta_{x},\xi).\label{dds354}
\end{eqnarray}
We now estimate the right-hand side of (\ref{dds354}). The terms $(\xi_{x},\theta)$,  $(A_{x},\theta)$ are estimated as in the proof of Proposition \ref{propo33}. Assuming that $t_{N}$ is the maximal time in $(0,T]$ for which the semidiscrete approximation exists and satisfies, in view of (\ref{dds38}),
\begin{eqnarray}
|\theta|_{\infty}\leq 1,\;\; 0\leq t\leq t_{N},\label{dds355}
\end{eqnarray}
we see, as in (\ref{dds328a}) and (\ref{dds330}), that for $ 0\leq t\leq t_{N}$
\begin{eqnarray}
|(\xi_{x},\theta)|+|(A_{x},\theta)|\leq C\left(N^{2(1-\mu)}+||\theta||^{2}+||\xi||_{1}^{2}\right).\label{dds356}
\end{eqnarray}
Examining the rest of the terms in the right-hand side of (\ref{dds354}) we first note that
\begin{eqnarray}
|(\theta_{x},\xi)|=|(\theta,\xi_{x})|\leq \frac{1}{2}(||\theta||^{2}+||\xi_{x}||^{2}).\label{dds357}
\end{eqnarray}
In the last term of the right-hand side of (\ref{dds354}) the inner product with $\xi$ is easily estimated by integrating by parts and arguing as in (\ref{dds318b})-(\ref{dds318f}). This gives
\begin{eqnarray}
|(B_{x},\xi)|=|(B,\xi_{x})|
\leq C(N^{-2\mu}+||\xi||_{1}^{2}).\label{dds358}
\end{eqnarray}
We now estimate the terms in the inner product $(B_{x},\xi_{xx})$. We have, since $H^{1}$ is an algebra, using (\ref{dds33}) and our hypothesis on $\mu$, that

\begin{eqnarray}
|((u\sigma)_{x},\xi_{xx})|&\leq &C||u||_{1}||\sigma||_{1}||\xi_{xx}||\leq C(N^{2(1-\mu)}+||\xi_{xx}||^{2}).\label{dds359a}\\
|((u\xi)_{x},\xi_{xx})|&\leq &C||u||_{1}||\xi||_{1}||\xi_{xx}||\leq C||\xi||_{2}^{2}.\label{dds359b}\\
|((\sigma\xi)_{x},\xi_{xx})|&\leq &C||\sigma||_{1}||\xi||_{1}||\xi_{xx}||\leq C||\xi||_{2}^{2}.\label{dds359c}
\end{eqnarray}
Since $\mu>3/2$ we have by (\ref{dds33}), (\ref{dds34})
\begin{eqnarray}
|((\sigma^{2})_{x},\xi_{xx})|&\leq & C|\sigma|_{\infty}||\sigma||_{1}||\xi_{xx}||\leq CN^{\frac{3}{2}-2\mu}||\xi_{xx}||\nonumber\\
&\leq & CN^{-\mu}||\xi_{xx}||\leq C(N^{-2\mu}+||\xi_{xx}||^{2}).\label{dds359d}
\end{eqnarray}
Finally, assuming that $t_{N}$ in (\ref{dds355}) is small enough so that in addition to (\ref{dds355}) we have 
\begin{eqnarray}
|\xi|_{\infty}\leq 1,\; 0\leq t\leq t_{N},\label{dds360}
\end{eqnarray}
we obtain for $0\leq t\leq t_{N}$
\begin{eqnarray}
|(\xi^{2})_{x},\xi_{xx})|\leq 2|\xi|_{\infty}||\xi_{x}||||\xi_{xx}||\leq C||\xi||_{2}^{2}.\label{dds359e}
\end{eqnarray}
From (\ref{dds354}), (\ref{dds356})-(\ref{dds359d}) and (\ref{dds359e}), since $\kappa_{2}, d>0$ we obtain
\begin{eqnarray*}
\frac{d}{dt}\left(||\theta||^{2}+||\xi||_{2}^{2}\right)\leq C(N^{2(1-\mu)}+||\theta||^{2}+||\xi||_{2}^{2}),\; 0\leq t\leq t_{N},\label{dds361}
\end{eqnarray*}
where the constant $C$ does not depend on $t_{N}$. By Gronwall's lemma then, for $0\leq t\leq t_{N}$
\begin{eqnarray}
||\theta||+||\xi||_{2}\leq C_{T}N^{1-\mu}.\label{dds362}
\end{eqnarray}
Since by (\ref{dds35}) $|\theta|_{\infty}\leq CN^{1/2}||\theta||$, and since $|\xi|_{\infty}\leq C||\xi||_{1}$ by Sobolev's theorem, we see that (\ref{dds362}) implies, in view of our assumption on $\mu$,  that $t_{N}$ in (\ref{dds355}) and (\ref{dds360}) is not maximal if $N$ is sufficiently large, and, as usual, can be taken equal to $T$. We infer that (\ref{dds362}) holds up to $t=T$ and that the conclusion of the proposition follows, giving an optimal-order $H^{1}$ error estimate for $u_{N}$ and a suboptimal-order one for $\zeta_{N}$ in $L^{2}$.
\end{proof}
\begin{proposition}
\label{propo37}
Let $a,b,c,d$ as in case (vii) and with no loss of generality suppose that $a<0, b>0, d=0, c=0$. If $\mu>3/2$, then
\begin{eqnarray}
\max_{0\leq t\leq T}\left(||\zeta_{N}-\zeta||+||u_{N}-u||\right)\leq C N^{1-\mu}.\label{dds363}
\end{eqnarray}
\end{proposition}
\begin{proof}
These systems are of the form
\begin{eqnarray}
\zeta_{t}+\kappa_{1}u_{x}+\lambda (\zeta u)_{x}+au_{xxx}-b\zeta_{xxt}&=&0,\label{dds364}\\
u_{t}+\kappa_{2}\zeta_{x}+\lambda uu_{x}&=&0.\label{dds365}
\end{eqnarray}
Motivated by an analogous observation in \cite{BonaChS2004}, Section 3.3, we write the first pde (\ref{dds364}) above in the equivalent form
\begin{eqnarray}
\zeta_{t}-\frac{a}{b}u_{x}+(\kappa_{1}+\frac{a}{b})\mathcal{T}_{b}u_{x}+\lambda \mathcal{T}_{b}(\zeta u)_{x}=0,\label{dds366}
\end{eqnarray}
where $\mathcal{T}_{b}=(I-b\partial_{x}^{2})^{-1}$ was introduced in the proof of Proposition \ref{propo35}.
%
%is the solution operator of the two-point bvp
%\begin{eqnarray}
%&&-bv_{xx}+v=f\;{\rm in}\; (0,1),\label{dds367}\\
%&&v, f \;\;1{\rm -periodic},\nonumber
%\end{eqnarray}
%given by $u=\mathcal{T}f$.

Consequently, we will adopt the following semidiscretization of the system: For all $\varphi, \chi\in S_{N}$
\begin{eqnarray}
(\zeta_{Nt},\varphi)-\frac{a}{b}(u_{Nx},\varphi)+(\kappa_{1}+\frac{a}{b})(\mathcal{T}_{b}u_{Nx},\varphi)+\lambda (\mathcal{T}_{b}(\zeta_{N}u_{N})_{x},\varphi)&=&0,\label{dds368}\\
(u_{Nt},\chi)+\kappa_{2}(\zeta_{Nx},\chi)+\lambda (u_{N}u_{Nx},\chi)&=&0,\label{dds369}
\end{eqnarray}
with 
\begin{eqnarray}
\zeta_{N}(0)=P\zeta_{0}, u_{N}(0)=Pu_{0}.\label{dds370}
\end{eqnarray}
It is clear that the solution of this ivp exists at least locally in time.

%whose Fourier implementation is for $t\geq 0$
%\begin{eqnarray*}
%\widehat{\zeta}_{N,t}-ik\frac{a}{b}\widehat{u}_{N}+\frac{ik(\kappa_{1}-\frac{a}{b})}{1+bk^{2}}\widehat{u}_{N}+\frac{ik\lambda}{1+bk^{2}}\widehat{\zeta_{N}u_{N}}&=&0,\\
%\widehat{u}_{N,t}+ik \kappa_{2}\widehat{\zeta}_{N}+\frac{ik}{2}\lambda\widehat{u_{N}^{2}}&=&0,
%\end{eqnarray*}
%where $\widehat{\zeta}_{N}=\widehat{\zeta}_{N}(k,t), \widehat{u}_{N}=\widehat{u}_{N}(k,t), -N\leq k\leq N, t\geq 0$ are the Fourier coefficients of $\zeta_{N}, u_{N}$, respectively with initial values $\widehat{\zeta}_{N}(k,0)=\widehat{\zeta}_{0}(k), \widehat{u}_{N}(k,0)=\widehat{u}_{0}(k)$.

We proceed now to the proof of the error estimate (\ref{dds363}). 
While the semidiscrete solution exists, using our usual notation and subtracting the weak form of (\ref{dds366}) from (\ref{dds368}) we have for all $\varphi\in S_{N}$
\begin{eqnarray}
(\theta_{t},\varphi)-\frac{a}{b}(\xi_{x},\varphi)=-(\kappa_{1}+\frac{a}{b})(\mathcal{T}_{b}(\xi_{x}+\sigma_{x}),\varphi)-\lambda(\mathcal{T}_{b}{A}_{x},\varphi),\label{dds371}
\end{eqnarray}
%where, cf. (\ref{dds310}), 
%\begin{eqnarray}
%\widetilde{A}=\zeta_{N}u_{N}-\zeta u=u\rho+\zeta\sigma+u\theta+\zeta\xi+\sigma\theta+\rho\xi+\rho\sigma+\theta\xi.\label{dds372}
%\end{eqnarray}
From the weak form of (\ref{dds365}) and (\ref{dds369}) we obtain
\begin{eqnarray}
(\xi_{t},\chi)+\kappa_{2}(\theta_{x},\chi)=-\lambda ({B}_{x},\chi),\forall \chi\in S_{N},\label{dds373}
\end{eqnarray}
%where, cf. (\ref{dds312}),
%\begin{eqnarray}
%\widetilde{B}=u\sigma+u\xi+\sigma\xi+\frac{1}{2}\sigma^{2}+\frac{1}{2}\xi^{2}.\label{dds374}
%\end{eqnarray}
Putting $\varphi=\theta$ in (\ref{dds371}), $\chi=\xi$ in (\ref{dds373}) gives
\begin{eqnarray*}
\frac{1}{2}\frac{d}{dt}||\theta||^{2}-\frac{a}{b}(\xi_{x},\theta)&=&-(\kappa_{1}+\frac{a}{b})(\mathcal{T}_{b}(\xi_{x}+\sigma_{x}),\theta)-\lambda(\mathcal{T}_{b}{A}_{x},\theta),\\
\frac{1}{2}\frac{d}{dt}||\xi||^{2}-\kappa_{2}(\theta,\xi_{x})&=&-\lambda ({B}_{x},\xi).
\end{eqnarray*}
Multiplying the first equation above by $\kappa_{2}$ and the second by $-a/b$ and adding we get
\begin{eqnarray}
\frac{1}{2}\frac{d}{dt}\left(\kappa_{2}||\theta||^{2}+\frac{|a|}{b}||\xi||^{2}\right)&=&-\kappa_{2}(\kappa_{1}+\frac{a}{b})(\mathcal{T}_{b}(\xi+\sigma)_{x},\theta)\nonumber\\
&&-\kappa_{2}\lambda(\mathcal{T}_{b}{A}_{x},\theta)+\frac{a}{b}\lambda ({B}_{x},\xi).\label{dds375}
\end{eqnarray}
%Let now $-bg_{xx}+g=f_{x}$ where $f\in H^{1}, g\in H^{2}$. Then $\mathcal{T}f_{x}=g$. Taking inner products with $g$ in this ode and using integration by parts we have
%\begin{eqnarray*}
%||g||^{2}+b||g_{x}||^{2}=-(f,g_{x})\leq ||f||||g_{x}||\leq  \frac{1}{2b}||f||^{2}+\frac{b}{2}||g_{x}||^{2}.
%\end{eqnarray*}
%Therefore $||g||_{1}\leq C||f||$, i.~e.
%\begin{eqnarray}
%||\mathcal{T}f_{x}||_{1}\leq C||f||.\label{dds376}
%\end{eqnarray}
For the first two terms in the right-hand side of (\ref{dds375}) we have, in view of (\ref{361})
\begin{eqnarray}
|(\mathcal{T}_{b}(\xi+\sigma)_{x},\theta)|&\leq &||\mathcal{T}_{b}(\xi+\sigma)_{x}||||\theta||\leq C(||\xi||+||\sigma||)||\theta||,\label{dds377}\\
|(\mathcal{T}_{b}{A}_{x},\theta)|&\leq & ||\mathcal{T}_{b}{A}_{x}||||\theta||\leq C||{A}||||\theta||.\label{dds378}
\end{eqnarray}
We bound the right-hand side of (\ref{dds377}) as usual by
\begin{eqnarray}
(||\xi||+||\sigma||)||\theta||\leq C(N^{-2\mu}+||\xi||^{2}+||\theta||^{2}).\label{dds379}
\end{eqnarray}
In order to bound the right-hand side of (\ref{dds378}), we argue as in the proof of Proposition \ref{propo35}; in particular, consider (\ref{363}). The only difference is in the last term of $A$, which we now bound by $||\xi|| |\theta|_{\infty}$. Consequently, let
%\begin{eqnarray}
%||u\rho||&\leq & ||u||_{\infty}||\rho||\leq CN^{-\mu},\label{dds380a}\\
%||\zeta\sigma||&\leq & ||\zeta||_{\infty}||\sigma||\leq CN^{-\mu},\label{dds380b}\\
%||u\theta||&\leq & ||u||_{\infty}||\theta||\leq C||\theta||,\label{dds380c}\\
%||\zeta\xi||&\leq & ||\zeta||_{\infty}||\xi||\leq C||\xi||,\label{dds380d}\\
%||\sigma\theta||&\leq & ||\sigma||_{\infty}||\theta||\leq C||\theta||,\label{dds380e}\\
%||\rho\xi||&\leq & ||\rho||_{\infty}||\xi||\leq C||\xi||,\label{dds380f}\\
%\end{eqnarray}
%and since $\mu>1/2$
%\begin{eqnarray}
%||\rho\sigma||\leq ||\rho||_{\infty}||\sigma||\leq CN^{\frac{1}{2}-\mu}N^{-\mu}\leq CN^{-\mu}.\label{dds380g}
%\end{eqnarray}
$t_{N}$, $0<t_{N}\leq T$ be the maximal time for which the semidiscrete approximation exists and is such that
\begin{eqnarray}
|\theta|_{\infty}\leq 1,\; 0\leq t\leq t_{N}.\label{dds381}
\end{eqnarray}

%Therefore, for $0\leq t\leq t_{N}$,
%\begin{eqnarray}
%||\theta\xi||\leq ||\theta||_{\infty}||\xi||\leq ||\xi||.\label{dds380h}
%\end{eqnarray}
%from (\ref{dds372}), (\ref{dds380a})-(\ref{dds380h}) we have then, 
We therefore obtain, for $0\leq t\leq t_{N}$, that
\begin{eqnarray}
||{A}||||\theta||\leq C(N^{-2\mu}+||\theta||^{2}+||\xi||^{2}).\label{dds382}
\end{eqnarray}
We now estimate the last term of the right-hand side of (\ref{dds375}). Using the definition (\ref{dds312}) of $B$, we have, since $H^{1}$ is an algebra and in view of (\ref{dds33}), that

%In view of (\ref{dds374}) we have
\begin{eqnarray}
|((u\sigma)_{x},\xi)|\leq  ||u||_{1}||\sigma||_{1}||\xi||\leq C||\sigma||_{1}||\xi||\leq  C(N^{2(1-\mu)}+||\xi||^{2})\label{dds383a}
\end{eqnarray}
By integrating by parts we see that
\begin{eqnarray}
|((u\xi)_{x},\xi)|=\frac{1}{2}|(u_{x}\xi,\xi)||\leq \frac{1}{2}|u_{x}|_{\infty}||\xi||^{2}\leq  C||\xi||^{2},\label{dds383b}
\end{eqnarray}
using our hypothesis on $\mu$. Similarly, by (\ref{dds34})
\begin{eqnarray}
|((\sigma\xi)_{x},\xi)|=\frac{1}{2}|(\sigma_{x}\xi,\xi)||\leq \frac{1}{2}|\sigma_{x}|_{\infty}||\xi||^{2}\leq  C||\xi||^{2}.\label{dds383c}
\end{eqnarray}
By our hypothesis on $\mu$ and (\ref{dds33}), (\ref{dds34})
\begin{eqnarray}
|(\sigma\sigma_{x},\xi)|&\leq &|\sigma|_{\infty}||\sigma_{x}||||\xi||\leq CN^{\frac{1}{2}-\mu}N^{1-\mu}||\xi||\nonumber\\
&\leq & CN^{-\mu}||\xi||\leq C(N^{-2\mu}+||\xi||^{2}).\label{dds383d}
\end{eqnarray}
And finally, by periodicity
\begin{eqnarray}
(\xi\xi_{x},\xi)=0.\label{dds383e}
\end{eqnarray}
By  (\ref{dds383a})-(\ref{dds383e}) we conclude, as long as the solution of (\ref{dds368})-(\ref{dds370}) exists, that
\begin{eqnarray}
|({B}_{x},\xi)|\leq C(N^{2(1-\mu)}+||\xi||^{2}).\label{dds384}
\end{eqnarray}
Therefore, by (\ref{dds375}), since $\kappa_{2}, b>0$, and using (\ref{dds377}), (\ref{dds379}), (\ref{dds382}), (\ref{dds384}) we have for $0\leq t\leq t_{N}$
\begin{eqnarray*}
\frac{d}{dt}\left(||\theta||^{2}+||\xi||^{2}\right)\leq C(N^{2(1-\mu)}+||\theta||^{2}+||\xi||^{2}).
\end{eqnarray*}
Hence, by Gronwall's lemma, in view of (\ref{dds370}) we obtain for $0\leq t\leq t_{N}$
\begin{eqnarray}
||\theta||+||\xi||\leq C_{T}N^{1-\mu}.\label{dds383}
\end{eqnarray}
Since $|\theta|_{\infty}\leq CN^{1/2}||\theta||\leq CN^{3/2-\mu}$ by the above, in view of our hypothesis on $\mu$, and provided we take $N$ sufficiently large, we infer that $t_{N}$ in (\ref{dds381}) was not maximal; as usual, we may take $t_{N}=T$. Thus (\ref{dds383}) holds up to $t=T$, giving (\ref{dds363}). This inequality implies that $\zeta_{N}$ and $u_{N}$ satisfy in this case suboptimal $L^{2}$-error estimates.
\end{proof}
%%%%%%%%%%%%%%%%%%%%%%%%%%%%%%%%%
%%%%%%%%%%%%%%%%%%%%%%%%%%%%%%%%%
\section{Solitary waves}
\label{sec4}
This section is focused on the existence and numerical generation of solitary-wave solutions of the B/B systems (\ref{BB2}). We first review the application of several theories of existence and then we illustrate the corresponding results with some examples of numerical generation of the profiles.

\subsection{Some existence results}
\label{sec41}
The solitary waves are solutions of (\ref{BB2}) of the form $\zeta=\zeta(x-c_{s}t), u=u(x-c_{s}t)$ that satisfy
\begin{eqnarray*}
\partial_{x}\begin{pmatrix}c_{s}(1-b\partial_{xx})&-\frac{1}{\delta+\gamma}-a\partial_{xx}\\-(1-\gamma)(1+c\partial_{xx})&c_{s}(1-d\partial_{xx})\end{pmatrix}\begin{pmatrix}\zeta\\v_{\beta}\end{pmatrix}=\kappa_{\gamma,\delta}\partial_{x}\begin{pmatrix}\zeta v_{\beta}\\ \frac{v_{\beta}^{2}}{2}\end{pmatrix},
\end{eqnarray*}
where $\kappa_{\gamma,\delta}=\frac{\delta^{2}-\gamma}{(\delta+\gamma)^{2}}$, cf. section \ref{sec21}.
Classical solitary waves (CSW), for which $\zeta,u\rightarrow 0$ as $|x-c_{s}t|\rightarrow \infty$, will be solutions of
\begin{eqnarray}
\begin{pmatrix}c_{s}(1-b\partial_{xx})&-\frac{1}{\delta+\gamma}-a\partial_{xx}\\-(1-\gamma)(1+c\partial_{xx})&c_{s}(1-d\partial_{xx})\end{pmatrix}\begin{pmatrix}\zeta\\v_{\beta}\end{pmatrix}=\kappa_{\gamma,\delta}\begin{pmatrix}\zeta v_{\beta}\\ \frac{v_{\beta}^{2}}{2}\end{pmatrix}.\label{BB6}
\end{eqnarray}
In addition, there exist Generalized solitary waves (GSW), which are not CSW and satisfy (\ref{BB6}). 
\subsubsection{Existence via linearization}
(The main references in this section are \cite{Champ,AmickT,IK2,IK,ChampS,IP}.)
In the case of small deviations of $c_{s}$ from the speed of sound the existence of solitary waves can be studied in a similar way to analogous studies in \cite{BonaDM2007,DougalisM2008} by using Bifurcation theory, \cite{I,HaragusI}. 
In order to apply this theory to (\ref{BB6}), we define $c_{\gamma,\delta}$ to be the speed of sound corresponding to (\ref{BB2}), i.~e. as
\begin{eqnarray}
c_{\gamma,\delta}=\sqrt{\frac{1-\gamma}{\delta+\gamma}},\label{BB7b}
\end{eqnarray}
Note that if $(c_{s},\zeta,v_{\beta})$ is a solution of (\ref{BB6}), then $(-c_{s},\zeta,-v_{\beta})$ is also a solution (with the same profile but opposite speed). Thus we can assume $c_{s}>0$ and define the parameter
\begin{eqnarray}
\widetilde{\mu}:=\frac{c_{s}}{c_{\gamma,\delta}}-1,\label{NFT1}
\end{eqnarray}
so that $c_{s}=\widetilde{\mu}c_{\gamma,\delta}+1$. The parameter $\widetilde{\mu}$ given by (\ref{NFT1}) is analogous to the parameter $\mu$ used in \cite{IK} to discuss, via the Normal Form Theory (NFT), the existence of solitary waves for free surface wave propagation on an inviscid fluid layer under gravity and surface tension effects.
In terms of the variables $\zeta, v_{\beta}$, (\ref{BB6}) can be rewritten as a first-order system, depending on $\widetilde{\mu}$, for $U=(U_{1},U_{2},U_{3},U_{4})^{T}:=(\zeta,\zeta^{\prime},v_{\beta},v_{\beta}^{\prime})^{T}$, namely as
\begin{eqnarray}
U^{\prime}&=&V(U,\widetilde{\mu}):=L(\widetilde{\mu})U+R(U,\widetilde{\mu}),\label{NFS}\\
L(\widetilde{\mu})&=&\begin{pmatrix}0&1&0&0\\\frac{dc_{s}^{2}+a(1-\gamma)}{D}&0&-\frac{c_{s}}{D} \left(a+\frac{d}{\delta+\gamma}\right)&0\\
0&0&0&1\\-\frac{(1-\gamma)c_{s}(b+c)}{D}&0&\frac{1}{D}\left(bc_{s}^{2}+\frac{c(1-\gamma)}{\delta+\gamma}\right)&0\end{pmatrix},\nonumber\\
R(U,\widetilde{\mu})&=&\begin{pmatrix}0\\\frac{1}{D}\frac{\delta^{2}-\gamma}{(\delta+\gamma)^{2}}\left(-dc_{s}U_{1}U_{3}+\frac{a}{2}U_{3}^{2}\right)\\
0\\\frac{1}{D}\frac{\delta^{2}-\gamma}{(\delta+\gamma)^{2}}\left(-c(1-\gamma)U_{1}U_{3}-\frac{bc_{s}}{2}U_{3}^{2}\right)\end{pmatrix},\nonumber\\
D&=&bdc_{s}^{2}-(1-\gamma)ac=bdc_{s}^{2}-\frac{ac}{\kappa_{1}}c_{\gamma,\delta}^{2}.\label{NFTD}
\end{eqnarray}
We assume that the condition (C1) (cf. section \ref{sec22}) holds and $D\neq 0$. The system (\ref{NFS}), (\ref{NFTD}) admits $U=0$ as fixed point, that is, $V(0,\widetilde{\mu})=0$. Additionally, the vector field $V$ is reversible, meaning that
$$SV(U,\widetilde{\mu})=-V(SU,\widetilde{\mu}),$$ where $S={\rm diag}(1,-1,1,-1)$. In view of these observations, we see that the study of homoclinic solutions of (\ref{BB6}) via NFT requires first to analyze the linearization of (\ref{NFS}) at the origin $U=0$. The characteristic equation is
\begin{eqnarray}
\lambda^{4}-B\lambda^{2}+A=0,\label{4A1}
\end{eqnarray}
where
\begin{eqnarray}
A=\frac{c_{s}^{2}-c_{\gamma,\delta}^{2}}{D},\;
B=\frac{(b+d)c_{s}^{2}+(c+a/\kappa_{1})c_{\gamma,\delta}^{2}}{D}.\label{4A2}
\end{eqnarray}
(No confusion need arise between these $A,B$ and those of section \ref{sec3}.)The structure of the spectrum of $L(\widetilde{\mu})$ in (\ref{NFS}), (\ref{NFTD}) can be studied following the survey \cite{Champ}. The distribution of the roots of (\ref{4A1}) in the $(B,A)$-plane is given in Figure \ref{fig_A1}, which reproduces the bifurcation diagram, along with the location and the type of the four eigenvalues, shown in Figure 1 of \cite{Champ}. 
Thus, the behaviour of the linear dynamics is determined by the four regions separated by the four bifurcation curves
\begin{eqnarray*}
C_{0}&=&\{(B,A) / A=0, B>0\},\\
C_{1}&=&\{(B,A) / A=0, B<0\},\\
C_{2}&=&\{(B,A) / A>0, B=-\sqrt{A}\},\\
C_{3}&=&\{(B,A) / A>0, B=\sqrt{A}\}.
\end{eqnarray*}
The Center Manifold Theorems and the theory of reversible bifurcations, \cite{I,HaragusI}, can be applied to study the existence of homoclinic orbits in each bifurcation. The reduced Normal Form systems reveal the existence of homoclinic to zero orbits and homoclinic to periodic orbits. The corresponding solutions are CSW's and GSW's, respectively. In addition, periodic and quasi-periodic orbits are identified, \cite{IK}.

\begin{figure}[htbp]
\centering
{\includegraphics[width=1\textwidth]{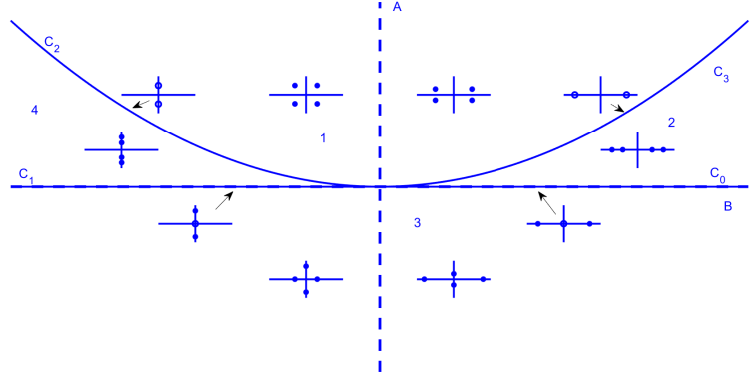}}
\caption{Linearization at the origin of (\ref{NFS}). As in Figure 1 of \cite{Champ}.}
\label{fig_A1}
\end{figure}

More specifically, we may adapt the discussion of \cite{Champ} near the bifurcation curves $C_{0}$ to $C_{3}$ to the case of (\ref{NFS}), (\ref{NFTD}), using (\ref{NFT1}) as bifurcation parameter. Note first that
%
%
%
%
%and displayed in Figure \ref{fig_A1}, along with the distribution of the eigenvalues in each region and curve, and which reproduces that of \cite{Champ}.
%\begin{figure}[htbp]
%\centering
%{\includegraphics[width=1\textwidth]{champ_fig.pdf}}
%\caption{Linearization at the origin of (\ref{NFS}).}
%\label{fig_A1}
%\end{figure}
%Following the theory developed in \cite{IK} and the discussion in \cite{Champ} we can identify CSW's (as homoclinic to zero orbits)near the curve $C_{0}$ in region $2$ and GSW's (as solutions homoclinic to periodic orbits) near the curve $C_{1}$ in region $3$. Note first that, 
by (\ref{4A2}), $A=0$ iff $c_{s}=c_{\gamma,\delta}$; then $B=S(\gamma,\delta)/(bd-ac/\kappa_{1})$, with $S(\gamma,\delta):=\frac{1+\gamma\delta}{3\delta (\delta+\gamma)}$. Therefore, under the hypothesis (C1), the curve $C_{0}$ is characterized by the conditions 
\begin{eqnarray}
a,c\leq 0, b,d\geq 0, bd-\frac{ac}{\kappa_{1}}>0,\; \widetilde{\mu}=0,\label{4A3}
\end{eqnarray}
while the conditions for the curve $C_{1}$ are
\begin{eqnarray}
a,c\leq 0, b,d\geq 0, bd-\frac{ac}{\kappa_{1}}<0.\; \widetilde{\mu}=0,\label{4A4}
\end{eqnarray}
We now study the information furnished by the Normal Form Theory (NFT) close to each curve $C_{j}, 0\leq j\leq 3$.
In the case of $C_{0}$, the linearization matrix $L(0)$ has two simple eigenvalues equal to
\begin{eqnarray}
\pm\sqrt{\frac{S(\gamma,\delta)}{bd-ac/\kappa_{1}}},\label{NFT3}
\end{eqnarray}
and the zero eigenvalue with geometric multiplicity one and algebraic multiplicity two. As in \cite{IK,Champ}, the main role in describing the dynamics close to $C_{0}$ by NFT is played by this two-dimensional center manifold, on which (\ref{NFS}) is reduced to a nonlinear oscillator system which depends on $\widetilde{\mu}$. When $\widetilde{\mu}$, $A$ and $B$ are positive, and near $C_{0}$ the linear dynamics is given by the spectrum of $L(\widetilde{\mu})$ which consists of four real eigenvalues (region 2 in Figure \ref{fig_A1}). In this case, the normal form system has a unique solution, homoclinic to zero at infinity, symmetric and unique up to spatial translations, (\cite{IK}, Proposition 3.1), that corresponds to a CSW solution of (\ref{BB6}).

\begin{remark}
If $\widetilde{\mu}$ is negative, and with similar arguments to those of \cite{IK}, NFT establishes the existence of a family of periodic solutions of the reduced system (close to $C_{0}$) for each $\widetilde{\mu}$, unique up to spatial translations. For the case at hand, and due to (\ref{4A2}), if $\widetilde{\mu}<0$ and $bd-ac/\kappa_{1}>0$, then $D>0$. Therefore, $A<0$ and we fall into region 3 of Figure \ref{fig_A1}. The numerical generation of periodic traveling wave solutions of (\ref{BB6}) under these conditions will be discussed in section \ref{sec42}.
\end{remark} 
In the case of $C_{1}$, the spectrum of $L(0)$ consists of zero (with algebraic multiplicity two) and the two simple imaginary eigenvalues given by (\ref{NFT3}) (recall that $bd-ac/\kappa_{1}<0$). The arguments used in \cite{IK}, Proposition 3.2, apply
here and NFT reduces (\ref{BB6}), on the center manifold, for $\widetilde{\mu}\neq 0$ small enough, to a normal form system which admits homoclinic solutions to periodic orbits, that is GSW solutions. Information about the structure of the periodic orbits can also be obtained, cf. \cite{Lombardi1,Lombardi2,Lombardi3}.
\begin{remark}
\label{remark42}
In addition to GSW's, the normal form derived in \cite{IK}, section 3.2, also reveals the existence of other solutions: periodic, quasi-periodic, and homoclinic to zero solutions (that is, CSW's). This normal form is used in \cite{Champ} to generalize the result of existence of a homoclinic solution of square hyperbolic secant form (in lowest order of $\widetilde{\mu}$) or a pair of hyperbolic secant solutions close to $C_{1}$ and corresponding to region 3 of Figure \ref{fig_A1}. As mentioned in \cite{Champ}, persistence of these solutions under small reversible perturbations of the normal form is not expected. 
\end{remark}
The normal form derived in \cite{IP} is used in \cite{Champ} to study the dynamics close to $C_{2}$. Its application to our case reveals the existence of homoclinic to zero solutions (CSW's) with nonmonotone decay for $\widetilde{\mu}<0$ and corresponding to region 1 of Figure \ref{fig_A1} (they have a truncated form like $E_{1}{\rm sech}(E_{2}x)e^{i\omega \theta}$, for constants $E_{1}, E_{2}$ and $\omega$ related to the coefficients of the normal form, and $\theta\in\mathbb{R}$), and homoclinic solutions to periodic orbits (GSW's) for $\widetilde{\mu}>0$ corresponding to region 4 of Figure \ref{fig_A1}.
\begin{remark}
These CSW's with nonmonotone decay require then speeds $c_{s}$ smaller than the speed of sound $c_{\gamma,\delta}$ (since $\widetilde{\mu}<0$) and are different from the CSW's obtained close to $C_{0}$, which are strictly positive or negative. The generation and stability of these waves will be discussed numerically in sections \ref{sec42} and \ref{sec55}, respectively.
\end{remark}
Finally, near $C_{3}$, and according to \cite{Champ}, crossing from the region 2 to region 1 in Figure \ref{fig_A1} forms a bifurcation causing the generation of infinity multiplicity of homoclinic orbits. This bifurcation was analyzed for specific problems, \cite{ChampT,Belyakov,Deveney}. In our case, the numerical computations suggest a similar situation to that of described in \cite{ChampT} (see also section 5.1 of \cite{Champ} and \cite{AmickT,ChampS}). In particular, CSW's of nonmonotone decay are numerically generated in section \ref{sec42}.

% Furthermore, (\cite{IK}, Proposition 3.2; see also \cite{Champ}), homoclinic to zero solutions (that is, CSW's) exist for $\mu<0$. In the case of (\ref{NFT1}), (\ref{NFT2}), since $bd-ac/\kappa_{1}$, if $\mu<0$ then $A>0$ so this combinaion corresponds to region 1. As mentioned in \cite{Champ}, persistence of these solutions under small reversible perturbations is not expected. For the case of (\ref{BB6}), the generation and persistence of these waves will be discussed numerically in sections \ref{sec42} and \ref{sec55}, respectively.

In summary, NFT establishes the existence of CSW solutions of (\ref{BB6}) for small and positive $c_{s}^{2}-c_{\gamma,\delta}^{2}$ when (\ref{4A3}) holds (close to $C_{0}$), and GSW solutions for small $c_{s}^{2}-c_{\gamma,\delta}^{2}$ under the conditions given in (\ref{4A4}) and close to $C_{1}$. Note that, in the latter case, in view of (\ref{BB1b}), it is not possible to have $a,c<0, b>0, d=0$. Hence, from (\ref{4A3}) and (\ref{4A4}), we may distinguish the cases shown in Table \ref{tavle0}.

%Thus defining $c_{\gamma,\delta}$ to be the speed of sound corresponding to (\ref{BB2}), i.~e. as
%\begin{eqnarray}
%c_{\gamma,\delta}=\sqrt{\frac{1-\gamma}{\delta+\gamma}},\label{BB7b}
%\end{eqnarray}
%then a similar discussion to that of \cite{DougalisM2008}, under hypothesis (C1), leads to existence of CSW or GSW for small and positive $c_{s}-c_{\gamma,\delta}$, and when $D\neq 0$, in the cases shown in Table \ref{tavle0}.
\begin{table}[ht]
\begin{center}
\begin{tabular}{|c|c|c|}
    \hline
Case&Admissible system&Type of solitary wave\\\hline\hline
(A1)&$a,c<0,b=0,d>0$&GSW\\\hline
(A2)&$a,c<0,b,d>0, \frac{bd}{\delta+\gamma}-ac<0$&GSW\\\hline
(A3)&$a,c<0,b,d>0, \frac{bd}{\delta+\gamma}-ac>0$&CSW\\\hline
(A4)&$a=0,c<0,b,d>0$&CSW\\\hline
(A5)&$a<0,c=0,b,d>0$&CSW\\\hline
(A6)&$a=c=0,b,d>0$&CSW\\
    \hline
\end{tabular}
\end{center}
\caption{Corresponding admissible system in (\ref{BB2}) and type of solitary waves and  for small and positive $c_{s}-c_{\gamma,\delta}$, when $D\neq 0$, according to the Normal Form Theory.}
\label{tavle0}
\end{table}

In addition, when $c_{s}<c_{\gamma,\delta}$, close to $C_{2}$, CSW's of nonmonotone decay are shown to exist. They have also been observed numerically close to $C_{3}$, cf. section \ref{sec42}.
\begin{remark}
In \cite{Champ} a review is made of existence results of homoclinic orbits using the linearization in each region of Figure \ref{fig_A1}. In our case, this is beyond the scope of the paper. However, in section \ref{sec42} we will indicate, for each computed solitary wave, the corresponding region, refering to \cite{Champ} and references therein for more information and possible similarities with other models.
\end{remark}

%\begin{itemize}
%
%\item[(A1)] $a,c<0,b=0,d>0$ (GSW).
%\item[(A2)] $a,c<0,b,d>0, \frac{bd}{\delta+\gamma}-ac<0$ (GSW).
%\item[(A3)] $a,c<0,b,d>0, \frac{bd}{\delta+\gamma}-ac>0$ (CSW).
%\item[(A4)] $a=0,c<0,b,d>0$ (CSW).
%\item[(A5)] $a<0,c=0,b,d>0$ (CSW).
%\item[(A6)] $a=c=0,b,d>0$ (CSW).
%%\item[(A7)] When $D=0$ numerical evidence of CSW will be shown below.
%\end{itemize}
Table \ref{tavle1} shows the relation between the classification (i)-(vii) of parameter cases for the study of well-posedness in section \ref{sec22} and  (A1)-(A6) 
\begin{table}[ht]
\begin{center}
\begin{tabular}{|c||c|}
    \hline
$(i)\leftrightarrow (A6)$&$\left.\begin{matrix}(iv)\\(vi)\\(vii)\end{matrix}\right\}\leftrightarrow D=0$\\
$(ii)\leftrightarrow\left\{\begin{matrix}(A2)\\(A3)\end{matrix}\right.$&\\
$(iii)\leftrightarrow (A1)$&$(v)\leftrightarrow\left\{\begin{matrix}(A4)\\(A5)\end{matrix}\right.$\\
    \hline
\end{tabular}
\end{center}
\caption{Relation between the cases (i)-(vii) of section \ref{sec22}, and (A1)-(A6) and $D=0$ for existence of solitary waves according to the Normal Form Theory.}
\label{tavle1}
\end{table}

For a particular case of $D=0$ (\lq classical Boussinesq\rq\ $b=0, d>0, a=c=0$) existence and numerical generation of CSW's are studied in \cite{Duran2019}, cf. also \cite{AD1} for the analogous case of surface waves. In other cases where $D=0$, numerical evidence of existence of CSW's will be presented below.

\subsubsection{Toland's Theory}
For larger deviations $c_{s}-c_{\gamma,\delta}$ or other conditions on the speed $c_{s}$, several general theories may be used to analyze the existence of classical solitary waves: Toland's Theory, \cite{Toland1986}, Concentration-Compactness Theory, \cite{Lions}, and Positive Operator Theory, \cite{BenjaminBB1990}. In the first case, and in order to apply the results of \cite{Toland1986} to study the existence of CSW's, note that system (\ref{BB6}) may be written in the form
\begin{eqnarray}
S_{1}{\bf v}^{\prime\prime}+S_{2}{\bf v}+\nabla g(v_{\beta},\zeta)=0,\label{BB8}
\end{eqnarray}
where ${\bf v}=(v_{\beta},\zeta)^{T}$ and
\begin{eqnarray}
&&S_{1}=-\begin{pmatrix}\frac{a}{c_{s}}&b\\d&\frac{(1-\gamma)c}{c_{s}}\end{pmatrix},\quad
S_{2}=-\begin{pmatrix}\frac{-1}{c_{s}(\delta+\gamma)}&1\\1&-\frac{(1-\gamma)}{c_{s}}\end{pmatrix},\label{BB9}\\
&&g(v_{\beta},\zeta)=-\frac{(\delta^{2}-\gamma)}{c_{s}(\delta+\gamma)^{2}}\frac{\zeta v_{\beta}^{2}}{2}=-\frac{\kappa_{\gamma,\delta}}{2c_{s}}\zeta v_{\beta}^{2},\nonumber
\end{eqnarray}
where we recall that $\kappa_{\gamma,\delta}$ is given by (\ref{kap}) and  $v_{\beta}=(1-\beta\partial_{xx})^{-1}u$. 
Note that the theory of \cite{Toland1986} can be applied here in the symmetric case $b=d$, and may be used to give existence results for classical solitary waves for specific systems in terms of $c_{s},\gamma$  and $\delta$ or, alternatively, in terms of $c_{s}, c_{\gamma.\delta}$, cf. (\ref{BB7b}). We recall the general form of the main result for existence of solitary waves derived in \cite{Toland1986}, in the form given in \cite{MChen1998}.
\begin{theorem}
\label{toland}
Let $S_{1},S_{2}$ be symmetric, $g\in C^{2}(\mathbb{R}^{2})$ such that $g,\nabla g, \nabla^{2}g$ are zero at $(0,0)$. Let $Q$ and $f$ be given for ${\bf u}:=(u_{1},u_{2})^{T}\in\mathbb{R}^{2}$ by 
\begin{eqnarray}
Q({\bf u})={\bf u}^{T}S_{1}{\bf u},\; 
f({\bf u})={\bf u}^{T}S_{2}{\bf u}+2g({\bf u}),\label{BB8b}
\end{eqnarray}
%$$f({\bf u})={\bf u}^{T}S_{2}{\bf u}+2g({\bf u})$$ 
and assume that:
\begin{itemize}
\item[(I)] $det(S_{1})<0$ and there are two linearly independent vectors ${\bf v}_{1}, {\bf v}_{2}$ such that $Q({\bf v}_{1})=Q({\bf v}_{2})=0$.
\item[(II)] There is a closed plane curve $\mathcal{F}$ with $(0,0)\in\mathcal{F}$ such that
\begin{itemize}
\item[(i)] $f=0$ on $\mathcal{F}$ and $\mathcal{F}\backslash \{(0,0)\}\subset\{{\bf u} | Q({\bf u})<0\}$.
\item[(ii)] $f(u_{1},u_{2})>0$ in the (nonempty) interior of $\mathcal{F}$.
\item[(iii)] $\mathcal{F}\backslash \{(0,0)\}$ is strictly convex.
\item[(iv)] $\nabla f(u_{1},u_{2})=0$ on $\mathcal{F}$ $\Leftrightarrow (u_{1},u_{2})=(0,0)$.
\end{itemize}
\end{itemize}
Then there is an orbit $\widetilde{\gamma}$ of (\ref{BB8}) in the $(v_{\beta}(0),\zeta)$ plane, which is homoclinic to the origin and
\begin{itemize}
\item[(a)] $(v_{\beta}(0),\zeta(0))\in\Gamma$ where $\Gamma$ is the segment of $\mathcal{F}$ including the origin between $P_{1}, P_{2}$ with $P_{j}$ satisfying
\begin{eqnarray*}
\label{BB9b}
\langle \nabla f(P_{j}),{\bf v}_{j}\rangle=0,\quad j=1,2.
\end{eqnarray*}
\item[(b)] $v_{\beta},\zeta$ are even functions on $\mathbb{R}$.
\item[(c)] $(v_{\beta}(\xi),\zeta(\xi))$ is in the interior of $\mathcal{F}$ for all $\xi\neq 0$.
\item[(d)] $\widetilde{\gamma}$ is monotone.
\end{itemize}
\end{theorem}
The application of this theory can be made in a similar way to the case of Boussinesq systems for surface waves, \cite{Chen2000}. In our context, from (\ref{BB8}), (\ref{BB9}) and assuming $a,c\leq 0, b=d>0$, we have
\begin{eqnarray}
f(v_{\beta},\zeta)&=&{\bf v}^{T}S_{2}{\bf v}+2g({\bf v})\nonumber\\
&=&\frac{1}{c_{s}}\left(-\frac{v_{\beta}^{2}}{\delta+\gamma)}\left(1+\frac{(\delta^{2}-\gamma}{(\delta+\gamma)}\zeta\right)-(1-\gamma)\zeta^{2}+2c_{s}\zeta v_{\beta}\right).\label{efe}
\end{eqnarray}
The quadratic form $Q$ defined in (\ref{BB8b})
%This may be used to give existence results on solitary waves for specific systems in terms of $c_{s},\gamma$ and $\delta$ or, alternatively, in terms of $c_{s}, c_{\gamma.\delta}$. Specifically, the application of this theory can be made in a similar way to the case of Boussinesq systems for surface waves, \cite{ref}. In our context, assuming $a,c\leq 0, b=d>0$ the quadratic form
%\begin{eqnarray}
%Q(u_{1},u_{2})=(u_{1},u_{2})S_{1}\begin{pmatrix}u_{1}\\u_{2}\end{pmatrix},\label{BB8b}
%\end{eqnarray}
is indefinite, and its diagonal form (used to identify two linearly independent directions in which $Q$ vanishes) is of the following types:
\begin{itemize}
\item[(1)] If $a<0$ then
$$Q(U_{1},U_{2})=-\frac{a}{c_{s}}U_{1}^{2}+\frac{D}{ac_{s}}U_{2}^{2},$$ where here $D=b^{2}c_{s}^{2}-(1-\gamma)ac\geq 0$ and
\begin{eqnarray*}
U_{1}=u_{1}+\frac{bc_{s}}{a}u_{2},\quad U_{2}=u_{2}.
\end{eqnarray*}
In this case the two linearly independent directions where $Q=0$ are given by
\begin{eqnarray*}
U_{2}=\pm\frac{a}{\sqrt{D}}U_{1}&\Rightarrow &u_{2}=\frac{a}{\sqrt{D}-bc_{s}}u_{1},\\
&&u_{2}=-\frac{a}{\sqrt{D}+bc_{s}}u_{1}.
\end{eqnarray*}
\item[(2)] When $a=0, c<0$ then
$$Q(U_{1},U_{2})=-(1-\gamma)\frac{c}{c_{s}}U_{1}^{2}+\frac{b c_{s}}{c(1-\gamma)}U_{2}^{2},$$ with 
\begin{eqnarray*}
U_{1}=u_{1}+\frac{bc_{s}}{c(1-\gamma)}u_{2},\quad U_{2}=u_{2}.
\end{eqnarray*}
Now $Q=0$ when
\begin{eqnarray*}
U_{2}=\pm\frac{c(1-\gamma)}{c_{s}\sqrt{b}}U_{1}&\Rightarrow &u_{2}=-\frac{c(1-\gamma)}{c_{s}\sqrt{b}(1+\sqrt{b})}u_{1}\\
&&u_{2}=\frac{c(1-\gamma)}{c_{s}\sqrt{b}(1-\sqrt{b})}u_{1}.
\end{eqnarray*}
\item[(3)] When $a=c=0$, then
$Q(u_{1},u_{2})=-2bu_{1}u_{2}$, and the two independent directions are clearly $u_{1}=0$ and $u_{2}=0$.
\end{itemize}
As we are investigating existence of classical solitary waves, we will apply Theorem \ref{toland} for the cases (A3)-(A6) in section \ref{sec41}. Note that the condition ${\rm det}(S_{1})<0$ in (I) of Theorem \ref{toland} implies that $D>0$, where $D$ is given by (\ref{NFTD}).
%(The case $D=0$ may be treated as in \cite{PegoW1997} for the case of surface waves.) 
By way of illustration, we consider the simplest case (3). 
%We define
%\begin{eqnarray}
%\label{BB8d}
%\kappa_{1}=\frac{1}{\delta+\gamma},\quad \kappa_{2}=(1-\gamma),\quad K_{\gamma,\delta}=\frac{\delta^{2}-\gamma}{(\delta+\gamma)^{2}}.
%\end{eqnarray}
The condition $Q<0$ (see (II)(i) of Theorem \ref{toland}) implies that $u_{1}, u_{2}$ are taken in the first or the third $(v_{\beta},\zeta)$ quadrant. The resulting solitary waves are, respectively, of elevation or of depression type, and this is determined by the sign of $\kappa_{\gamma,\delta}$. Since the linearly independent directions are given in this case by ${\bf v}_{1}=(1,0)^{T}, {\bf v}_{2}=(0,1)^{T}$, the points $P_{1}, P_{2}$ must satisfy
\begin{eqnarray*}
f(P_{1})=f(P_{2})=0,\quad \frac{\partial f}{\partial u_{1}}(P_{1})=\frac{\partial f}{\partial u_{2}}(P_{2})=0,
\end{eqnarray*}
which means, following \cite{Chen2000}
\begin{eqnarray}
-\kappa_{1}u_{1}^{2}-\kappa_{\gamma,\delta}u_{1}^{2}u_{2}-\kappa_{2}u_{2}^{2}+2c_{s}u_{1}u_{2}&=&0,\label{BB9c}\\
-\kappa_{1}u_{1}-\kappa_{\gamma,\delta}u_{1}u_{2}+c_{s}u_{2}&=&0,\label{BB9d}\\
-\kappa_{\gamma,\delta}u_{1}^{2}-2\kappa_{2}u_{2}+2c_{s}u_{1}&=&0,\label{BB9e}
\end{eqnarray}
where $\kappa_{1}=\frac{1}{\delta+\gamma}$ and $\kappa_{2}=1-\gamma$ as in section \ref{sec3}. The components of $P_{1}$ must satisfy (\ref{BB9c}), (\ref{BB9d}) and those of $P_{2}$  (\ref{BB9c}), (\ref{BB9e}). 
\begin{figure}[htbp]
\centering
\subfigure[]
{\includegraphics[width=6.27cm]{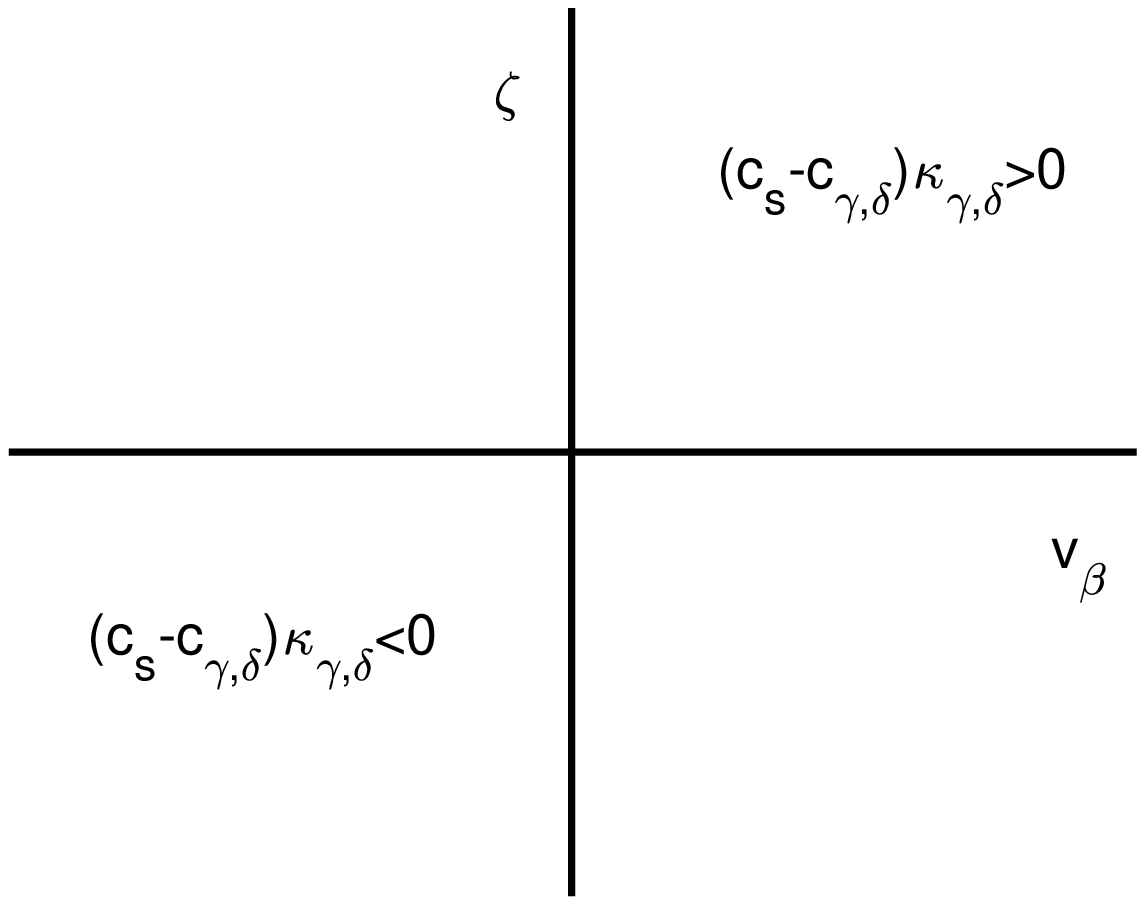}}
\subfigure[]
{\includegraphics[width=6.27cm]{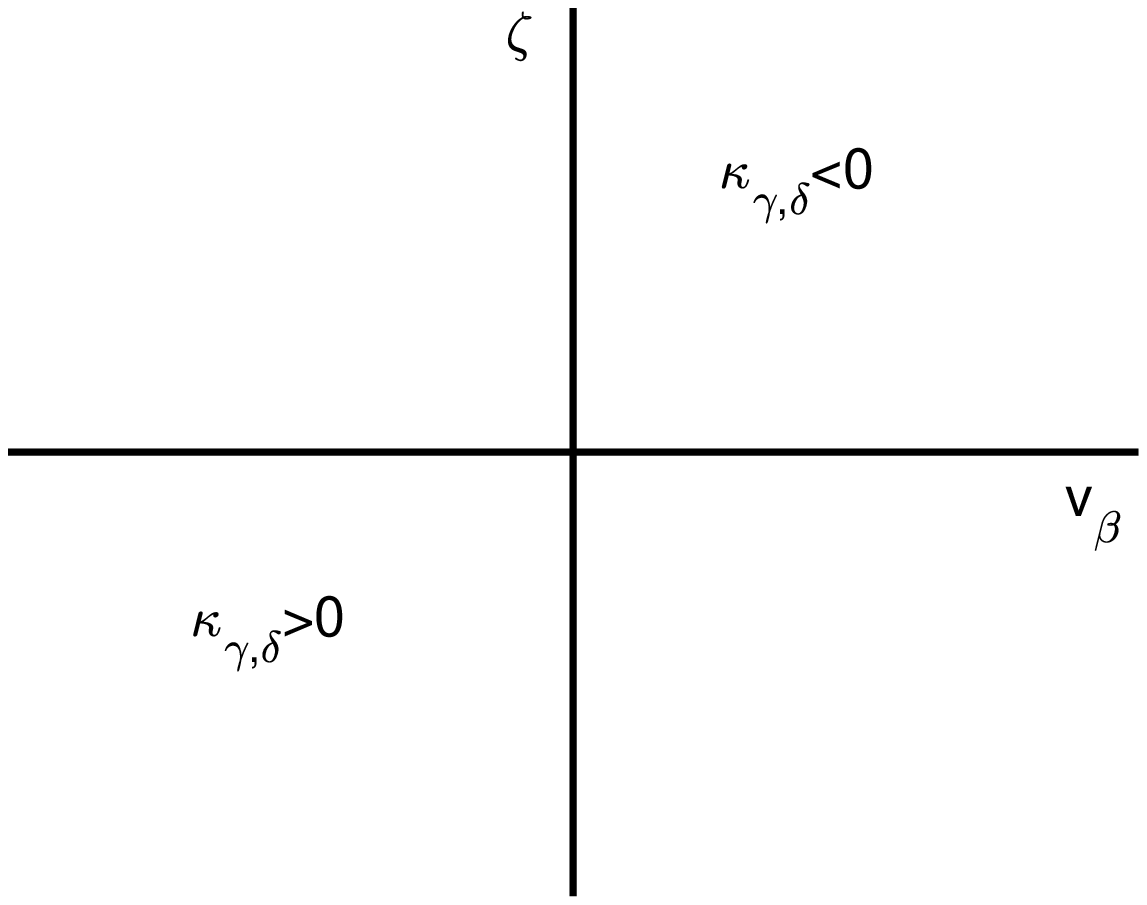}}
\caption{Conditions on $c_{s}-c_{\gamma,\delta}$ and $\kappa_{\gamma,\delta}$ to have $P_{1}$ and $P_{2}$ in the first or third $(v_{\beta},\zeta)$ quadrant. (a) $c_{s}>0$; (b) $c_{s}<0$.}
\label{fig_XXX}
\end{figure}
Solving (\ref{BB9c}), (\ref{BB9d}) leads to $u_{1}=u_{2}=0$ and
\begin{eqnarray}
u_{1}=\frac{c_{s}^{2}-c_{\gamma,\delta}^{2}}{c_{s}\kappa_{\gamma,\delta}},\; u_{2}=\frac{c_{s}}{\kappa_{2}}u_{1},\label{c1}
\end{eqnarray} 
while solving (\ref{BB9c}), (\ref{BB9e}) yields
\begin{eqnarray}
u_{1}^{\pm}=2\frac{c_{s}\mp c_{\gamma,\delta}}{\kappa_{\gamma,\delta}},\; u_{2}^{\pm}=\pm\sqrt{\frac{\kappa_{1}}{\kappa_{2}}}u_{1}^{\pm}.\label{c2}
\end{eqnarray} 
We now discuss the components of $P_{1}$ and $P_{2}$ by using the requirement that they must be in the first or third $(v_{\beta},\zeta)$ quadrant. Assume first that $c_{s}>0$. Then $u_{1}$ and $u_{2}$ in (\ref{c1}) have always the same sign, while in the case of (\ref{c2}) this holds only for $u_{1}^{+}$ and $u_{2}^{+}$. Therefore, $P_{1}$ has the components given by (\ref{c1}), and from (\ref{c2}) $P_{2}=(u_{1}^{+},u_{2}^{+})$. It is not hard to check that $P_{1}$ and $P_{2}$ are in the first $(v_{\beta},\zeta)$ quadrant when $(c_{s}-c_{\gamma,\delta})\kappa_{\gamma,\delta}>0$ and in the third quadrant when $(c_{s}-c_{\gamma,\delta})\kappa_{\gamma,\delta}<0$, see Figure \ref{fig_XXX}(a).
Then, 
according to the conclusions of Theorem \ref{toland}, for initial data on $\Gamma=P_{1}P_{2}$, the corresponding classical solitary waves are of elevation when  $(c_{s}-c_{\gamma,\delta})\kappa_{\gamma,\delta}>0$  and of depression when  $(c_{s}-c_{\gamma,\delta})\kappa_{\gamma,\delta}<0$. The orbit $\widetilde{\gamma}$ and the segment $\Gamma$ for several values of the speed $c_{s}$ are shown in Figures \ref{fig_locus1} and \ref{fig_locus2} respectively.
%Assume first that $\kappa_{\gamma,\delta}>0$ (that is $\delta^{2}-\gamma>0$). Then solving  (\ref{BB9c}), (\ref{BB9d}) for the components of $P_{1}$ leads to
%\begin{eqnarray*}
%u_{1}^{\pm}&=&\frac{1}{2}\left(-\frac{c_{s}}{2\kappa_{\gamma,\delta}}\pm\sqrt{R}\right),\quad R=\left(\frac{c_{s}}{2\kappa_{\gamma,\delta}}\right)^{2}-2\frac{c_{\delta,\gamma}^{2}-c_{s}^{2}}{\kappa_{\gamma,\delta}^{2}},\\
%u_{2}^{\pm}&=&\frac{\kappa_{1}u_{1}^{\pm}}{c_{s}-\kappa_{\gamma,\delta}u_{1}^{\pm}}.
%\end{eqnarray*}
%We assume $c_{s}^{2}-c_{\delta,\gamma}^{2}>0$. Thus, $R>(c_{s}/2\kappa_{\gamma,\delta})^{2}$ and $u_{1}^{+}>0, u_{1}^{-}<0$. On the other hand, it can be seen that $u_{1}^{+}<c_{s}/2\kappa_{\gamma,\delta}$ and therefore $u_{2}^{+}>0$. Hence $P_{1}=(u_{1}^{+},u_{2}^{+})^{T}$.
%
%In the case of (\ref{BB9c}), (\ref{BB9e}) for $P_{2}$ we have
%\begin{eqnarray*}
%u_{1}^{\pm}=\frac{2}{\kappa_{\gamma,\delta}}(c_{s}\pm c_{\delta,\gamma}),\quad u_{2}^{\pm}=\frac{2c_{s}-\kappa_{\gamma,\delta}u_{1}^{\pm}}{2\kappa_{2}}u_{1}^{\pm},
%\end{eqnarray*}
%and now, if $c_{s}-c_{\delta,\gamma}>0$, then $u_{1}^{-},u_{2}^{-}>0$ and $P_{2}=(u_{1}^{-},u_{2}^{-})^{T}$. In Figure \ref{fig_locus1} we show the orbit $\widetilde{\gamma}$ and the segment $\Gamma$ of (a) of the conclusion of Theorem \ref{toland} for several values of the speed $c_{s}$. For initial data on $\Gamma=P_{1}P_{2}$ the corresponding classical solitary waves are of elevation.
\begin{figure}[htbp]
\centering
{\includegraphics[width=0.9\textwidth]{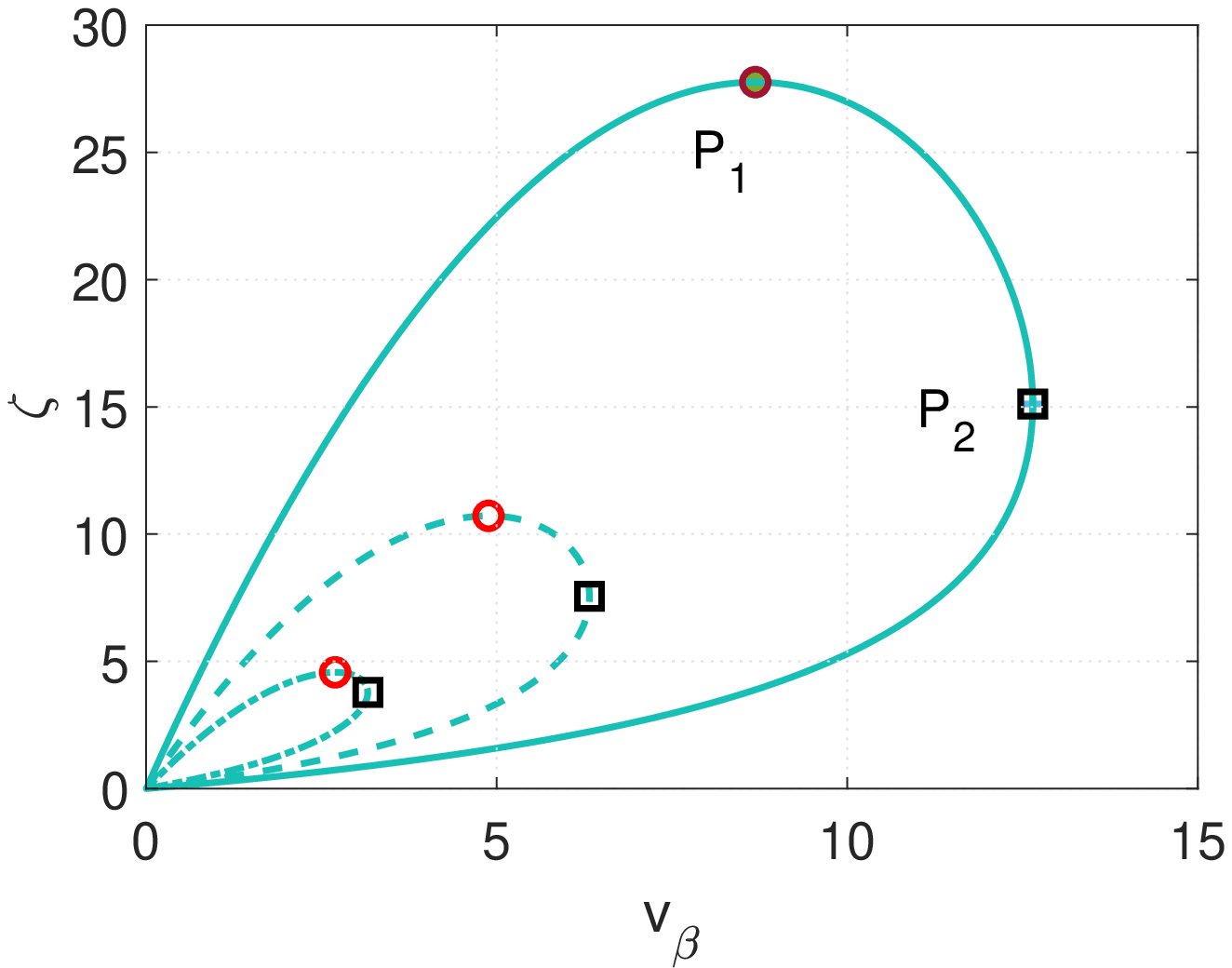}}
\caption{Generation of CSW's of elevation (case $(c_{s}-c_{\delta,\gamma})\kappa_{\gamma,\delta}>0$). Graphs of orbits for several positive values of $c_{s}-c_{\delta,\gamma}$. Dashed-dotted line: $c_{s}-c_{\delta,\gamma}=0.25$; dashed line: $c_{s}-c_{\delta,\gamma}=0.5$; solid line: $c_{s}-c_{\delta,\gamma}=1$.}
\label{fig_locus1}
\end{figure}

%If we now assume $\kappa_{\gamma,\delta}<0$ (that is $\delta^{2}-\gamma<0$), then the orbit $\widetilde{\gamma}$ is in the third quadrant ($u_{1},u_{2}<0$). Again if $c_{s}-c_{\delta,\gamma}>0$, then $P_{1}=(u_{1}^{-},u_{2}^{-})^{T}, P_{2}=(u_{1}^{+},u_{2}^{+})^{T}$. For the orbit of the corresponding solitary waves of depression, see Figure \ref{fig_locus2}.
\begin{figure}[htbp]
\centering
{\includegraphics[width=0.8\textwidth]{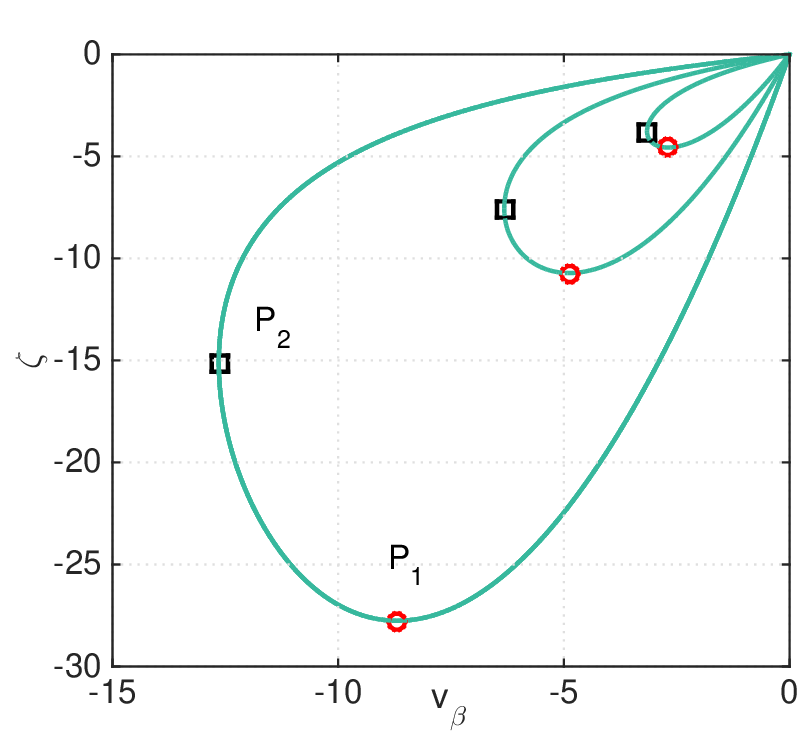}}
\caption{Generation of CSW's of depression  (case $(c_{s}-c_{\delta,\gamma})\kappa_{\gamma,\delta}<0$). Graphs of orbits for several positive values of $c_{s}-c_{\delta,\gamma}$. Dashed-dotted line: $c_{s}-c_{\delta,\gamma}=0.25$; dashed line: $c_{s}-c_{\delta,\gamma}=0.5$; solid line: $c_{s}-c_{\delta,\gamma}=1$.}
\label{fig_locus2}
\end{figure}
%\begin{remark}
%Since $u=(1-\beta\partial_{xx})v_{\beta}, \beta>0$, it remains to study which properties of $v_{\beta}$ would be inherited by $u$.
%\end{remark}
If we assume that $c_{s}<0$, then $u_{1}$ and $u_{2}$ from (\ref{c1}) have always opposite sign. Then necessarily $P_{1}=(0,0)$. The same argument as before leads again to obtain $P_{2}=(u_{1}^{+},u_{2}^{+})$, with $u_{1}^{+},u_{2}^{+}$ given by (\ref{c2}). Furthermore, in this case $c_{s}-c_{\gamma,\delta}$ is always negative and the condition for $P_{2}$ to be in the first or third $(v_{\beta},\zeta)$ quadrant depends only on the sign of $\kappa_{\gamma,\delta}$. For initial data on $\Gamma$, the corresponding classical solitary wave is of elevation if $\kappa_{\gamma,\delta}<0$ and of depression if $\kappa_{\gamma,\delta}>0$, see Figure \ref{fig_XXX}(b).

\begin{remark}
According to the previous arguments based on Toland's theory, the speed $c_{s}$ must satisfy, in view of the equation $f=0$, the speed-amplitude relation
\begin{eqnarray}
c_{s}&=&\frac{1}{2\mu}\left(\frac{1}{\delta+\gamma}\left(1+\frac{(\delta^{2}-\gamma}{(\delta+\gamma)}\zeta(0)\right)+(1-\gamma)\mu^{2}\right)\nonumber\\
&=&\frac{1}{2\mu}\left(\kappa_{1}+\kappa_{\gamma,\delta}\zeta(0)+\kappa_{2}\mu^{2}\right),\label{48b}
\end{eqnarray}
where $\mu=\zeta(0)/v_{\beta}(0)$. 
\end{remark}
\subsubsection{Concentration-Compactness Theory}
In the Hamiltonian case ($b=d>0$) and when $a, c< 0$, another result of existence of CSW's can be obtained from the application of the Concentration-Compactness (C-C) Theory. The method was developed by Lions in \cite{Lions} and has been used for proving the existence of solitary-wave solutions of a great number of nonlinear dispersive equations, see \cite{Weinstein,ABS,BouardS,Kichenassamy,AlbertAngulo}, among other. (A more exhaustive list of references can be found in \cite{Angulo}. For the case of Boussinesq systems for surface wave propagation, see \cite{Bao2015}.)
Here, its application is based on identifying solutions of (\ref{BB6}) with solutions of minimization problems of the form
\begin{eqnarray}
I_{r}=\inf\{E(\zeta,v): (\zeta,v)\in H^{1}\times H^{1}/ F(\zeta,v)=r\},\label{49a}
\end{eqnarray}
for $r>0$ where, in the case of the B/B systems at hand,
\begin{eqnarray}
E(\zeta,v)&=&\int_{-\infty}^{\infty}\left(\frac{\kappa_{2}}{2}\zeta J_{c}\zeta+\frac{\kappa_{1}}{2}v J_{a}v-c_{s}\zeta J_{b}v\right)dx,\label{49b}\\
F(\zeta,v)&=&-\frac{\kappa_{\gamma,\delta}}{2}\int_{-\infty}^{\infty}\zeta v^{2}dx,\label{49c}
\end{eqnarray}
and 
\begin{eqnarray}
J_{a}=1+\frac{a}{\kappa_{1}}\partial_{xx},\; J_{b}=1-b\partial_{xx},\; J_{c}=1+c\partial_{xx}.\label{49d}
\end{eqnarray}
Existence of CSW's by means of the C-C Theory has been recently proved in \cite{AnguloS2019} for another class of systems modelling internal wave propagation, specifically belonging to the Boussinesq/Full Dispersion (BFD) regime. The proof of \cite{AnguloS2019} can be adapted to the B/B system (\ref{BB6}) and we will emphasize here the main steps. The result is based on the following properties of the functional (\ref{49b}) (cf. Proposition 2.1 of \cite{AnguloS2019}).
\begin{proposition}
\label{propos41}
Assume that $b=d>0$ and $a, c< 0$.
The following properties concerning (\ref{49a})-(\ref{49d}) hold:
\begin{itemize}
\item[(i)] $E:H^{1}\times H^{1}\rightarrow \mathbb{R}$ is well defined and continuous: if $(\zeta,v)\in H^{1}\times H^{1}$ then
\begin{eqnarray}
E(\zeta,v)\leq C\left(||J_{c}^{1/2}\zeta||^{2}+||J_{a}^{1/2}v||^{2}+||J_{b}^{1/2}\zeta||^{2}+||J_{b}^{1/2}v||^{2}\right),\label{412a}
\end{eqnarray}
for some constant $C$.
\item[(ii)] If
\begin{eqnarray}
|c_{s}|\leq M:=\min\{\kappa_{2}\frac{|c|}{b},\frac{|a|}{2b}\},\label{49e}
\end{eqnarray}
then $E(\zeta,v)\geq 0$ and $E$ is coercive: 
\begin{eqnarray}
E(\zeta,v)\geq C\left(||\zeta||_{1}^{2}+||v||_{1}^{2}\right),\; (\zeta,v)\in H^{1}\times H^{1},\label{412b}
\end{eqnarray}
for some constant $C$ depending on $a, b, c, |c_{s}|, \kappa_{1}, \kappa_{2}$.
\item[(iii)] $I_{r}>0$, all minimizing sequences for $I_{r}$ are bounded in $H^{1}\times H^{1}$, and for $s\in (0,r)$, $I_{r}$ satisfies the sub-additivity property
\begin{eqnarray}
I_{r}<I_{s}+I_{r-s}.\label{413b}
\end{eqnarray}
\end{itemize}
\end{proposition}

\begin{proof}
The definition of the operators $J_{a}, J_{b}$ and $J_{c}$ in (\ref{49d}) implies tat $E$ is well defined in $H^{1}\times H^{1}$ while, from the Cauchy-Schwarz inequality, it is not hard to prove that $E$ satisfies (\ref{412a}) and therefore it is continuous.

Let $\langle\cdot,\cdot\rangle$ denote the $L^{2}(\mathbb{R})$-inner product. Since $J_{b}$ is a self-adjoint, positive operator, then from Cauchy-Schwarz inequality we have, for $(\zeta,v)\in H^{1}\times H^{1}$
\begin{eqnarray}
\left|c_{s}\int_{-\infty}^{\infty}\zeta J_{b}vdx\right|\leq |c_{s}|\langle J_{b}\zeta,\zeta\rangle^{1/2}\langle J_{b}v,v\rangle^{1/2}\leq \frac{|c_{s}|}{2}\left(\langle J_{b}\zeta,\zeta\rangle+\langle J_{b}v,v\rangle\right).\label{proof1b}
\end{eqnarray}
Note that, using Plancherel Theorem, we get
\begin{eqnarray}
\langle (\kappa_{2}J_{c}-|c_{s}|J_{b})\zeta,\zeta\rangle=\int_{-\infty}^{\infty}p(\xi)|\widehat{\zeta}(\xi)|^{2}d\xi,\label{proof1c}
\end{eqnarray}
where
\begin{eqnarray*}
p(\xi)=-(\kappa_{2}c+b|c_{s}|)\xi^{2}+\kappa_{2}-|c_{s}|.
\end{eqnarray*}
Since $a,c<0, b=d>0$, and using  (\ref{BB1b}), we have
\begin{eqnarray}
\frac{|c|}{b}=\frac{-\alpha_{2}}{1-\alpha_{2}}\leq 1,\label{proof1d}
\end{eqnarray}
Then, under condition (\ref{49e}), $p(\xi)\geq 0$ for all $\xi\in\mathbb{R}$, which by (\ref{proof1c}) implies
\begin{eqnarray}
\frac{|c_{s}|}{2}\langle J_{b}\zeta,\zeta\rangle\leq \frac{\kappa_{2}}{2}\langle J_{c}\zeta,\zeta\rangle.\label{proof2}
\end{eqnarray}
Similarly
\begin{eqnarray*}
\langle (\kappa_{1}J_{a}-|c_{s}|J_{b})v,v\rangle=\int_{-\infty}^{\infty}q(\xi)|\widehat{v}(\xi)|^{2}d\xi,\label{proof2a}
\end{eqnarray*}
with
\begin{eqnarray*}
q(\xi)=-(a+b|c_{s}|)\xi^{2}+\kappa_{1}-|c_{s}|.
\end{eqnarray*}
Since $b=d$, from (\ref{BB1b})
$$b=\alpha_{1}\frac{1+\gamma\delta}{3\delta(\gamma+\delta)}=\beta(1-\alpha_{2}).$$ Then
\begin{eqnarray*}
a&=&\frac{(1-\alpha_{1})(1+\gamma\delta)-3\delta\beta (\delta+\gamma)}{3\delta (\gamma+\delta)^{2}}=(1-\alpha_{1})\kappa_{1}\frac{b}{\alpha_{1}}-\beta \kappa_{1}\\
&=&\kappa_{1}b\left(\frac{1-\alpha_{1}}{\alpha_{1}}-\frac{1}{1-\alpha_{2}}\right).
\end{eqnarray*}
Therefore
\begin{eqnarray}
-\frac{a}{b}=k_{1}\left(\frac{\alpha_{1}-1}{\alpha_{1}}+\frac{1}{1-\alpha_{2}}\right),\label{proof2b}
\end{eqnarray}
where, again from (\ref{BB1b}), and since $b>0, c<0$, we have $\alpha_{1}>0, \alpha_{2}<0$. Then 
\begin{eqnarray*}
\frac{\alpha_{1}-1}{\alpha_{1}}+\frac{1}{1-\alpha_{2}}\leq 2,
\end{eqnarray*}
which, by the hypothesis $a<0$ and (\ref{proof2b}), leads to
\begin{eqnarray}
\kappa_{1}\geq \frac{|a|}{2b}.\label{proof2c}
\end{eqnarray}
Therefore, (\ref{proof2c}) and (\ref{49e}) imply that $q(\xi)\geq 0$ for all $\xi\in\mathbb{R}$, and then, from (\ref{proof1c}), that
\begin{eqnarray}
\frac{|c_{s}|}{2}\langle J_{b}v,v\rangle\leq \frac{\kappa_{1}}{2}\langle J_{a}v,v\rangle.\label{proof3}
\end{eqnarray}
Thus, using (\ref{proof1b}), (\ref{proof2}) and (\ref{proof3}), we see that it holds that $E(\zeta,v)\geq 0$ for $(\zeta,v)\in H^{1}\times H^{1}$.

We now prove coercitivity of $E$. From (\ref{proof1b}) we have
\begin{eqnarray}
E(\zeta,v)\geq\underbrace{\frac{1}{2}\langle (\kappa_{2}J_{c}-|c_{s}|J_{b})\zeta,\zeta\rangle}_{I}+\underbrace{\frac{1}{2}\langle (\kappa_{1}J_{a}-|c_{s}|J_{b})v,v\rangle}_{II}.\label{proof4a}
\end{eqnarray}
Note that for $(\zeta,v)\in H^{1}\times H^{1}$, using (\ref{49e}) and the fact that $|c|/b\leq 1$, we have
\begin{eqnarray}
2I&=&\int_{-\infty}^{\infty} \left(\kappa_{2}\zeta^{2}+\kappa_{2}|c|\zeta_{x}^{2}-|c_{s}|\zeta^{2}-b|c_{s}|\zeta_{x}^{2}\right)dx\nonumber\\
&=&\int_{-\infty}^{\infty}\left((\kappa_{2}-|c_{s}|)\zeta^{2}+(\kappa_{2}\frac{|c|}{b}-|c_{s}|)b\zeta_{x}^{2}\right)dx\nonumber\\
&\geq & (\kappa_{2}\frac{|c|}{b}-|c_{s}|)||J_{b}^{1/2}\zeta||^{2}\geq M ||J_{b}^{1/2}\zeta||^{2}.\label{proof4b}
\end{eqnarray}
Similarly, using (\ref{proof2c}) and (\ref{49e}), we have
\begin{eqnarray}
2II&=&\int_{-\infty}^{\infty} \left(\kappa_{1}v^{2}+|a|v_{x}^{2}-|c_{s}|v^{2}-b|c_{s}|v_{x}^{2}\right)dx\nonumber\\
&=&\int_{-\infty}^{\infty}\left((\kappa_{1}-|c_{s}|)v^{2}+(\frac{|a|}{b}-|c_{s}|)bv_{x}^{2}\right)dx\nonumber\\
&\geq & (\frac{|c|}{2b}-|c_{s}|)||J_{b}^{1/2}v||^{2}\geq M ||J_{b}^{1/2}v||^{2}.\label{proof4c}
\end{eqnarray}
Now, since $||J_{b}^{1/2}\cdot||$ is equivalent to $||\cdot||_{1}$, then (\ref{proof4a}), (\ref{proof4b}) and (\ref{proof4c}) imply the existence of a constant $C>0$ such that (\ref{412b}) holds.

Note now that
\begin{eqnarray*}
0<r=F(\zeta,v)\leq \frac{\kappa_{\gamma,\delta}}{2}|v|_{\infty}||v||||\zeta||\leq C||(\zeta,v)||_{1}^{3},
\end{eqnarray*}
for some constant $C>0$ and where $||(\zeta,v)||_{1}:=\left(||\zeta||_{1}^{2}+||v||_{1}^{2}\right)^{1/2}$. Then, using (\ref{412b}), we can find a constant $C$ such that
\begin{eqnarray*}
r\leq C E(\zeta,v)^{3/2},\; (\zeta,v)\in H^{1}\times H^{1},
\end{eqnarray*}
and therefore
\begin{eqnarray*}
E(\zeta,v)\geq \left(\frac{r}{C}\right)^{2/3}, (\zeta,v)\in H^{1}\times H^{1},
\end{eqnarray*}
which implies $I_{r}\geq (r/C)^{2/3}>0$. On the other hand, due to (\ref{412b}), it is clear that all minimizing sequences for $I_{r}$ are bounded for $ (\zeta,v)\in H^{1}\times H^{1}$. Finally, the sub-additivity property (\ref{413b}) is obtained from the fact that $I_{r}>0$ and the property
$$I_{\tau r}=\tau^{2/3}I_{r},$$ for all $\tau>0$, which is a consequence of the homogeneity of $E$ and $F$ (of orders two and three respectively), cf. \cite{AnguloS2019}.
\end{proof}

At this point we recall the C-C Principle, the main tool of the theory, see Lemma 1.1 of \cite{Lions}.

\begin{lemma}
\label{lemma_lions}
Let $\{\rho_{k}\}_{k\geq 1}$ be a sequence of non-negative functions in $L^{1}(\mathbb{R})$ such that $||\rho_{k}||_{L^{1}}$ converges to some $\sigma>0$. Then there is a subsequence (also denoted by $\{\rho_{k}\}_{k\geq 1}$) satisfying one of the following three conditions:
\begin{itemize}
\item[(a)] Compactness: there exist $y_{k}\in\mathbb{R}$ such that for any $\epsilon>0$ there is $R(\epsilon)>0$ such that for all $k$
\begin{eqnarray*}
\int_{|x-y_{k}|\leq R(\epsilon)}\rho_{k}dx\geq \int_{-\infty}^{\infty}\rho_{k}dx-\epsilon=L-\epsilon.
\end{eqnarray*}
\item[(b)] Vanishing: For every $R>0$
\begin{eqnarray*}
\lim_{k\rightarrow\infty}\sup_{y\in\mathbb{R}}\int_{|x-y|\leq R}\rho_{k}dx=0.
\end{eqnarray*}
\item[(c)] Dichotomy: there exists $l\in (0,L)$ such that for all $\epsilon>0$ there are $R,R_{k}\rightarrow\infty, y_{k}\in\mathbb{R}$ and $k_{0}$ satisfying
\begin{eqnarray*}
\left|\int_{|x-y_{k}|\leq R}\rho_{k}dx-l\right|<\epsilon,\quad 
\left|\int_{R<|x-y_{k}|\leq R_{k}}\rho_{k}dx\right|<\epsilon,
\end{eqnarray*}
for $k>k_{0}$.
\end{itemize}
\end{lemma}
The application of Lemma \ref{lemma_lions} to the existence of solitary waves can be summarized as follows, see e.~g. \cite{KenigPV1991,KenigPV1993,BonaCh2002} for different examples. One considers $\{\zeta_{n},v_{n}\}_{n\geq 1}$, a minimizing sequence for (\ref{49a}), and defines from this a new sequence $\{\rho_{n}\}_{n\geq 1}$ of non-negative functions in $L^{1}$ satisfying the conditions of the lemma; in our case (as in \cite{AnguloS2019}) this is
\begin{eqnarray}
\rho_{n}(x)=|\zeta_{n}(x)|^{2}+|\zeta^{\prime}_{n}(x)|^{2}+|v_{n}(x)|^{2}+|v^{\prime}_{n}(x)|^{2},\label{49f}
\end{eqnarray}
from which, thanks to (iii) of Proposition \ref{propos41}, one can find a subsequence $\{\rho_{n}\}_{n\geq 1}$ such that $||\rho_{n}||_{L^{1}}$ is convergent.
The next step is dismissing the possibility of vanishing and dichotomy for $\rho_{n}$; therefore Lemma \ref{lemma_lions} implies that the compactness property holds.

The final step is: From $\rho_{n}$ find a subsequence of the translated sequence $\{\zeta_{n}(x-y_{n}),v_{n}(x-y_{n})\}_{n}$ which converges weakly to some $H^{1}\times H^{1}$ function $(\zeta_{0},v_{0})$ and prove that $(\zeta_{0},v_{0})$ solves the variational problem (\ref{49a}). A solution $(\zeta,v)$ of (\ref{BB6}) is obtained from $(\zeta_{0},v_{0})$ after some scaling involving the Lagrange multiplier $r$. The application of all these steps to our system leads to the following result, which should be compared with that of the case of surface waves, \cite{Bao2015}. 
\begin{theorem}
\label{th_CC}
Let $\delta>0, \gamma<1$ and assume that $a,c< 0, b=d>0$. If (\ref{49e}) holds, then the system (\ref{BB2}) admits a classical solitary wave solution $\zeta=\zeta(x-c_{s}t), v=v(x-c_{s}t)$ with $(\zeta,v)\in H^{\infty}\times H^{\infty}$. Furthermore, $\zeta$ and $v$ decay exponentially as $x\rightarrow \pm\infty$ with
\begin{eqnarray}
\lim_{x\rightarrow \pm\infty}e^{\sigma_{a}|x|}v(x)&=&C_{1},\label{49g}\\
\lim_{x\rightarrow \pm\infty}e^{\sigma_{0}|x|}\zeta(x)&=&C_{2},\label{49h}
\end{eqnarray}
with $C_{1}, C_{2}$ constants, $\sigma_{a}=\sqrt{\kappa_{1}/|a|}$ and $\sigma_{0}\in (0,\sigma_{a}], \sigma_{0}<\sqrt{|c|}$.
\end{theorem}
\begin{proof}
Regularity of the profiles can be proved in a similar way to that of the corresponding result in \cite{AnguloS2019}. On the other hand, it may be worth to explain the details of the proof of the properties of asymptotic decay. In terms of the operators (\ref{49d}) the system (\ref{BB6}) has the form
\begin{eqnarray}
-c_{s}J_{b}\zeta+\kappa_{1}J_{a}v&=&-\kappa_{\gamma,\delta}\zeta v\nonumber,\\
\kappa_{2}J_{c}\zeta-c_{s}J_{b}v&=&-\frac{\kappa_{\gamma,\delta}}{2}v^{2}.\label{410a}
\end{eqnarray}
The second equation of (\ref{410a}) can be written as
\begin{eqnarray*}
\kappa_{2}\zeta=c_{s}J_{c}^{-1}J_{b}v-\frac{\kappa_{\gamma,\delta}}{2}v^{2}.
\end{eqnarray*}
Using that $c_{\gamma,\delta}^{2}=\kappa_{1}\kappa_{2}$, and substituting into the first equation of (\ref{410a}) we get
\begin{eqnarray*}
c_{\gamma,\delta}^{2}J_{a}v&=&G(v)=c_{s}^{2}J_{c}^{-1}J_{b}^{2}v-c_{s}\frac{\kappa_{\gamma,\delta}}{2}J_{b}J_{c}^{-1}(v^{2})\\
&&-\kappa_{\gamma,\delta}c_{s}vJ_{c}^{-1}J_{b}v+\frac{\kappa_{\gamma,\delta}^{2}}{2}vJ_{c}^{-1}(v^{2}).
\end{eqnarray*}
Therefore
\begin{eqnarray*}
v=\frac{1}{c_{\gamma,\delta}^{2}}J_{a}^{-1}G(v)=\frac{1}{c_{\gamma,\delta}^{2}}K_{a}\ast G(v),
\end{eqnarray*}
where $\widehat{K}_{a}(\xi)=\frac{1}{1+\frac{|a|}{\kappa_{1}}\xi^{2}}$, and therefore
\begin{eqnarray*}
K_{a}(x)=\frac{\pi}{2}\sqrt{\frac{\kappa_{1}}{|a|}}e^{-\sigma_{a}|x|},
\end{eqnarray*}
where $\sigma_{a}=\sqrt{\kappa_{1}/|a|}$ . According to \cite{BonaLi}, this implies (\ref{49g}). Similar arguments to those of \cite{AnguloS2019} can be used to obtain (\ref{49h}).
\end{proof}

\begin{remark}
\label{remark42b}
Note that it is not possible to have $b=d$ in the case (A2) of Table \ref{tavle0}. The reason is the following. Assume that $b=d$, that $a, b$ and $c$ satisfy (A2) and let $S({\gamma,\delta})$ be given in (\ref{BB1c}).  Then the condition $\frac{bd}{\delta+\gamma}-ac<0$ gives
\begin{eqnarray}
b^{2}<ac(\delta+\gamma)=|a||c|(\delta+\gamma).\label{BB_2}
\end{eqnarray}
From (\ref{BB1c}), we have
$$b=\frac{1}{2}(S({\gamma,\delta})-a(\delta+\gamma)-c)>0.$$ Substituting this into (\ref{BB_2}), and since $a, c<0$, after some computations we obtain
\begin{eqnarray*}
S({\gamma,\delta})<-(\sqrt{|c|}-\sqrt{|a|(\delta+\gamma)})^{2},
\end{eqnarray*}
which contradicts the fact that $S({\gamma,\delta})>0$. 
%This clarifies the range of speeds for existence of CSW given by Toland's Theory and C-C Theory. Recall that in the first case, the condition $|c_{s}|>c_{\gamma,\delta}$ is assumed. 
On the other hand, note
that, since $\kappa_{2}<1$ and from (\ref{proof1d})
%\begin{eqnarray*}
%\frac{|c|}{b}=\frac{-\alpha_{2}}{1-\alpha_{2}}<1,
%\end{eqnarray*}
it holds that $M<1$. Therefore the inequalities
\begin{eqnarray*}
c_{\gamma,\delta}=\sqrt{\kappa_{1}\kappa_{2}}<|c_{s}|\leq M,\label{410b}
\end{eqnarray*}
imply in particular that $c_{\gamma,\delta}<1$ and $b^{2}\kappa_{1}-|a||c|<0$, that is, $a, b$ and $c$ are in the case (A2) of Table \ref{tavle0}. Since we know that this is not possible,  Theorem \ref{th_CC} applies when $|c_{s}|<c_{\gamma,\delta}$. This should be compared with Toland's theory, cf. Theorem \ref{toland}, for which this condition was not required.
\end{remark}
%\begin{remark}
%Note that the result on the asymptotic decay in Theorem \ref{th_CC} assumes that $a, c<0$. If $c=0, a<0$, it s not hard to see that, due to regularity, the argument above to prove (\ref{49g}) is still valid. Since in this case $\kappa_{2}\zeta=c_{s}J_{b}v-\frac{}{2}v^{2}$, then  (\ref{49h}) now holds for $\sigma_{0}=\sigma_{a}$.
%
%As far as the other two cases ($a=0, c<0$ and $a=c=0$) are concerned, the asymptotic decay can be obtained from the approach based on the Positive Operator Theory, explained below.
%\end{remark}

\subsubsection{Positive Operator theory}
The Positive Operator theory can also be applied as in \cite{BonaCh2002} where some existence results of classical solitary wave solutuions of Boussinesq systems for surface waves were derived. In our case, if we take the Fourier transform in (\ref{BB6}) we obtain

%when the system (\ref{BB6}) is invertible, it can be written in a fixed point form. Using the Fourier transform, we have
\begin{eqnarray}
\begin{pmatrix}c_{s}(1+bk^{2})&-\kappa_{1}+ak^{2}\\-\kappa_{2}(1-ck^{2})&c_{s}(1+dk^{2})\end{pmatrix}\begin{pmatrix}\widehat{\zeta}(k)\\\widehat{v_{\beta}}(k)\end{pmatrix}=\kappa_{\gamma,\delta}\begin{pmatrix}\widehat{\zeta v_{\beta}}(k)\\\widehat{v_{\beta}^{2}/2}(k)\end{pmatrix}.\label{BB10}
\end{eqnarray}
System (\ref{BB10}) is invertible for all $k\in\mathbb{R}$ and for $b,d>0, a,c\leq 0$ if
\begin{eqnarray}
\Delta(k)=\Delta_{0}+\Delta_{1}k^{2}+\Delta_{2}k^{4}\neq 0,\label{445a}
\end{eqnarray}
where
\begin{eqnarray*}
&&\Delta_{0}=c_{s}^{2}-c_{\gamma,\delta}^{2},\quad \Delta_{1}=(c_{s}^{2}(b+d)-c_{\gamma,\delta}^{2}(-c-\frac{a}{\kappa_{1}}),\\
&& \Delta_{2}=c_{s}^{2}bd-c_{\gamma,\delta}^{2}\frac{ac}{\kappa_{1}}=c_{s}^{2}bd-(1-\gamma)ac.
\end{eqnarray*}
(Note that $A$ is the determinant $D$ given by (\ref{NFTD}).) We assume that $a,c\leq 0, b,d>0, bd-ac/\kappa_{1}>0$ (case (A3) of Table \ref{tavle0}), and $|c_{s}|>c_{\gamma,\delta}$. Then $\Delta_{j}>0, 0\leq j\leq 2$, and we may write (\ref{BB10}) in the form

%We observe that if $A>0$ then $B>0$ as well. So if we additionally assume $C>0$ we may write (\ref{BB10}) in the form
\begin{eqnarray*}
\widehat{\zeta}(k)&=&\frac{\kappa_{\gamma,\delta}}{\Delta(k)}\left(c_{s}(1+dk^{2})\widehat{\zeta v_{\beta}}(k)+(\kappa_{1}-ak^{2})\widehat{v_{\beta}^{2}/2}(k)\right)\\
\widehat{v_{\beta}}(k)&=&\frac{\kappa_{\gamma,\delta}}{\Delta(k)}\left(c_{s}(1+bk^{2})\widehat{v_{\beta}^{2}/2}(k)+\kappa_{2}(1-ck^{2})\widehat{\zeta v_{\beta}}(k)\right).
\end{eqnarray*}
As in \cite{BonaCh2002}, this leads to an integral form of (\ref{BB6})
\begin{eqnarray*}
\zeta&=&k_{12}\ast (\zeta v_{\beta})+\frac{1}{2}k_{22}\ast (v_{\beta}^{2}),\nonumber\\
v_{\beta}&=&m_{12}\ast (\zeta v_{\beta})+\frac{1}{2}m_{22}\ast (v_{\beta}^{2}),\label{BB11}
\end{eqnarray*}
where the integral kernels are
\begin{eqnarray*}
k_{12}(x)&=&\frac{c_{s}\kappa_{\gamma,\delta}}{2\Delta_{2}}\left(\frac{1-dr_{-}^{2}}{r_{-}(r_{+}^{2}-r_{-}^{2})}e^{-r_{-}|x|}-
\frac{1-dr_{+}^{2}}{r_{+}(r_{+}^{2}-r_{-}^{2})}e^{-r_{+}|x|}\right)\\
k_{22}(x)&=&\frac{\kappa_{\gamma,\delta}}{2\Delta_{2}}\left(\frac{\kappa_{1}+ar_{-}^{2}}{r_{-}(r_{+}^{2}-r_{-}^{2})}e^{-r_{-}|x|}-
\frac{\kappa_{1}+ar_{+}^{2}}{r_{+}(r_{+}^{2}-r_{-}^{2})}e^{-r_{+}|x|}\right)\\
m_{12}(x)&=&\frac{\kappa_{\gamma,\delta}}{2\Delta_{2}}\left(\frac{\kappa_{2}(1+cr_{-}^{2})}{r_{-}(r_{+}^{2}-r_{-}^{2})}e^{-r_{-}|x|}-
\frac{\kappa_{2}(1+cr_{+}^{2})}{r_{+}(r_{+}^{2}-r_{-}^{2})}e^{-r_{+}|x|}\right)\\
m_{22}(x)&=&\frac{c_{s}\kappa_{\gamma,\delta}}{2\Delta_{2}}\left(\frac{1-br_{-}^{2}}{r_{-}(r_{+}^{2}-r_{-}^{2})}e^{-r_{-}|x|}-
\frac{1-br_{+}^{2}}{r_{+}(r_{+}^{2}-r_{-}^{2})}e^{-r_{+}|x|}\right),
\end{eqnarray*}
with
\begin{eqnarray*}
r_{\pm}^{2}=\frac{1}{2\Delta_{2}}\left(\Delta_{1}\pm\sqrt{\Delta_{1}^{2}-4\Delta_{0}\Delta_{2}}\right).\label{445b}
\end{eqnarray*}
Then we have
\begin{theorem}
\label{pot}
Assume that $b,d>0,a,c\leq 0$, and $bd-ac/\kappa_{1}>0$. If $|c_{s}|>c_{\gamma,\delta}$, then the system (\ref{BB6}) admits classical solitary wave solutions of elevation if $\kappa_{\gamma,\delta}>0$ and of depression if $\kappa_{\gamma,\delta}<0$.
\end{theorem}
\begin{proof}
The conclusion follows from the application of the Positive Operator theory as in \cite{BonaCh2002}. When $\kappa_{\gamma,\delta}>0$, the theory is applied on the cone of continuous real-valued functions $(f,g)$ on $\mathbb{R}$ which are even, positive and non-increasing on $(0,\infty)$, while if $\kappa_{\gamma,\delta}<0$, the cone consists of continuous real-valued functions on $\mathbb{R}$ which are even, negative and non-decreasing on $(0,\infty)$.
\end{proof}

%assuming $c_{s}>c_{\gamma,\delta}$ and $A>0$, for $b,d>0, a,c\leq 0$ Positive Operator theory can be applied as in \cite{BonaCh2002} to obtain the existence of smooth classical solitary wave solutions of elevation if $\kappa_{\gamma,\delta}>0$ and of depression if $\kappa_{\gamma,\delta}<0$.
%\begin{remark}
%In a similar way, taking initial superposition of profiles, symmetric multi-pulses (of elevation and of depression) can also be numerically generated.
%\end{remark}
\begin{remark}
Note that when $b=d$, Theorem \ref{pot} reproduces part of the results of Toland's theory in Theorem \ref{toland}.
\end{remark}
\begin{remark}
The previous formulas can also be used to estimate the asymptotic decay at infinity of the classical solitary waves. In general,  both $\zeta, v_{\beta}$ should behave as $e^{-r_{-}|x|}$ as $|x|\rightarrow\infty$. In particular, for the cases $b=d>0, a=0, c\leq 0$ we may specify $r_{-}$ in (\ref{445b}). If $a=0, c<0$ then
\begin{eqnarray*}
\Delta_{2}=b^{2}c_{s}^{2},\; \Delta_{1}=2bc_{s}^{2}+cc_{\gamma,\delta}^{2},\;\Delta_{0}=c_{s}^{2}-c_{\gamma,\delta}^{2},
\end{eqnarray*}
and $\Delta_{1}^{2}-4\Delta_{0}\Delta_{2}=c^{2}c_{\gamma,\delta}^{4}+4b(b+c)c_{s}^{2}c_{\gamma,\delta}^{2}>0$ (since $b+c=\beta>0$). Therefore
\begin{eqnarray*}
\lim_{x\rightarrow \pm\infty}e^{-r_{-}|x|}v(x)=C,\; r_{-}=\left(\frac{1}{2\Delta_{2}}(\Delta_{1}-\sqrt{\Delta_{1}^{2}-4\Delta_{0}\Delta_{2}})\right)^{1/2},
\end{eqnarray*}
and $C$ a constant. The argument used in Theorem \ref{th_CC} can be applied here as well to obtain (\ref{49h}) for some $\sigma_{0}\in (0,r_{-}], \sigma_{0}<\sqrt{|c|}$.

On the other hand, if $a=c=0$, then
\begin{eqnarray*}
r_{-}=\frac{1}{\sqrt{b}}\sqrt{1-\frac{c_{\gamma,\delta}}{|c_{s}|}},
\end{eqnarray*}
and in this case
\begin{eqnarray*}
\lim_{x\rightarrow \pm\infty}e^{-r_{-}|x|}v(x)=C_{1},\;
 \lim_{x\rightarrow \pm\infty}e^{-r_{-}|x|}\zeta(x)=C_{2},
\end{eqnarray*}
for some constants $C_{1}, C_{2}$.
\end{remark}
In summary, in addition to the results obtained by the Normal Form Theory, valid for small and positive $c_{s}-c_{\gamma,\delta}$, and shown in Table \ref{tavle0}, the other standard theories contribute to the existence of classical solitary waves as follows:
\begin{itemize}
\item Toland's Theory ensures the existence of CSW's when $a,c\leq 0, b=d>0$ (Theorem \ref{toland}), a condition which intersects with those of the cases (A3) to (A6) of Table \ref{tavle0}. The speed and amplitude must satisfy the relation (\ref{48b}).
\item The C-C Theory establishes the existence of CSW's when $a,c< 0, b=d>0$, and $bd-ac/\kappa_{1}>0$ (cf. the  case (A3) of Table \ref{tavle0}), and for speeds satisfying (\ref{49e}) and $|c_{s}|<c_{\gamma,\delta}$, cf. Remark \ref{remark42}.
\item The Positive Operator Theory proves the existence of CSW's when $b,d>0, a,c\leq 0$, and $bd-ac/\kappa_{1}>0$, with speeds $c_{s}$ satisfying $|c_{s}|>c_{\gamma,\delta}$. The result can also justify the existence of CSW's predicted by NFT in the cases (A3) to (A6) of Table \ref{tavle0}.
\end{itemize}

We finally note that in \cite{Duran2019} existence of even, classical solitary waves of the \lq classical Boussinesq\rq\ system (case (vii) with $b=0, d>0, a=c=0$, for which $D=0$) is proved for speeds $c_{s}$ with $|c_{s}|>c_{\gamma,\delta}$ and for $\delta^{2}-\gamma\neq 0$. The waves are of elevation when $\delta^{2}-\gamma>0$ and of depression otherwise. The result is based on phase plane analysis of the system (\ref{BB6}), which is conservative in this case; see \cite{PegoW1997} for the case of surface waves.

\subsubsection{Exact solitary wave solutions in 1D}
For particular values of the speed, exact classical solitary waves can be derived by following the arguments used in \cite{MChen1998}. If we look for solutions $(\zeta,v_{\beta})$ with $v_{\beta}=B\zeta$ for some constant $B$, then, substituting into (\ref{BB6}) we have
\begin{eqnarray}
(c_{s}-\kappa_{1}B)\zeta -(bc_{s}+aB)\zeta^{\prime\prime}&=&B\kappa_{\gamma,\delta}\zeta^{2}\nonumber\\
(c_{s}B-\kappa_{2})\zeta-(dc_{s}B+\kappa_{2}c)\zeta^{\prime\prime}&=&\frac{B^{2}}{2}\kappa_{\gamma,\delta}\zeta^{2}.\label{esw1}
\end{eqnarray}
The existence of a solution $\zeta$ of (\ref{esw1}) requires then
\begin{eqnarray}
2c_{s}B-2\kappa_{2}&=&c_{s}B-\kappa_{1}B^{2}\nonumber\\
2dc_{s}B+2\kappa_{2}c&=&bc_{s}B+aB^{2}.\label{esw2}
\end{eqnarray}
If we consider (\ref{esw2}) as a linear system for the variables $B^{2}$ and $c_{s}B$, then we have two possibilities:

\begin{itemize}
\item If $b-2d-a(\delta+\gamma)\neq 0$, then (\ref{esw2}) has a unique solution from which
\begin{eqnarray}
B^{2}=\frac{2\kappa_{2}(b-2d-c)}{\kappa_{1}(b-2d)-a},\quad
c_{s}=\frac{1}{B}\frac{2\kappa_{2}(c\kappa_{1}-a)}{\kappa_{1}(b-2d)-a}.\label{esw3a}
\end{eqnarray}
%
%. Let
%\begin{eqnarray*}
%p=\frac{(1-\gamma)(b-2d-c)}{\kappa_{1}(b-2d)-a},\quad \kappa_{1}=\frac{1}{\delta+\gamma}.
%\end{eqnarray*}
%Then the system has classical solitary wave solutions if and only if $p>0$ and
%\begin{eqnarray*}
%(\kappa_{1}p-\frac{1-\gamma}{2})((\kappa_{1}b-a)p-b(1-\gamma))>0.
%\end{eqnarray*}
%Explicitly
%\begin{eqnarray}
%\zeta(x,t)&=&\zeta_{0}{\sech}^{2}(\lambda (x-c_{s}t-x_{0})),\label{BB12}\\
%u(x,t)&=&B(\zeta(x,t)-\beta\partial_{xx}\zeta(x,t)),\nonumber
%\end{eqnarray}
%where
%\begin{eqnarray}
%&&\zeta_{0}=\frac{3((1-\gamma)-2\kappa_{1}p)}{2\kappa_{\gamma,\delta}p},\quad \kappa_{\gamma,\delta}=\frac{\delta^{2}-\gamma}{(\delta+\gamma)^{2}},\nonumber\\
%&&B=\pm\sqrt{\frac{3(1-\gamma)}{\kappa_{\gamma,\delta}\zeta_{0}+3\kappa_{1}}},\quad c_{s}=\frac{2(1-\gamma)-\kappa_{1}B^{2}}{B},\label{BB13}\\
%&&\lambda=\frac{1}{2}\left(\frac{2\kappa_{\gamma,\delta}\zeta_{0}}{2b\kappa_{\gamma,\delta}(1-\gamma)\zeta_{0}+3(1-\gamma)(a+\kappa_{1}b)}\right)^{1/2}.\nonumber
%\end{eqnarray}
\item If $b-2d-a(\delta+\gamma)=0$, then (\ref{esw2}) has infinitely many solutions iff $c=a(\delta+\gamma)$; they satisfy
\begin{eqnarray}
c_{s}=\frac{2\kappa_{2}-\kappa_{1}B^{2}}{B},\label{esw3b}
\end{eqnarray} 
for $B\neq 0$ arbitrary.
%
%. Here there are solitary wave solutions of the form (\ref{BB12}), (\ref{BB13}) for any $B^{2}>0$ when
%\begin{eqnarray*}
%(B^{2}-(1-\gamma)(\delta+\gamma))(dB^{2}-b(\delta+\gamma)(1-\gamma))>0.
%\end{eqnarray*}
\end{itemize}
As far as the exact form of the solutions is concerned, differentiating one of the equations of (\ref{esw1}) and using (\ref{esw2}) we have
\begin{eqnarray}
\mu_{1}\zeta^{\prime}-\mu_{2}\zeta^{\prime\prime\prime}=\zeta\zeta^{\prime},\label{esw4}
\end{eqnarray}
where
\begin{eqnarray*}
\mu_{1}=\frac{\kappa_{2}-\kappa_{1}B^{2}}{\kappa_{\gamma,\delta}B^{2}},\;
\mu_{2}=\frac{(a-b\kappa_{1})B^{2}+2b\kappa_{2}}{2\kappa_{\gamma,\delta}B^{2}}.
\end{eqnarray*}
Thus, (\ref{esw4}) admits solutions of square hyperbolic secant form if $\mu_{1}\mu_{2}>0$, that is, if
\begin{eqnarray*}
(\kappa_{2}-\kappa_{1}B^{2})((a-b\kappa_{1})B^{2}+2b\kappa_{2})>0.\label{esw5}
\end{eqnarray*}
If this condition is satisfied, then
\begin{eqnarray}
\zeta(x,t)&=&3\mu_{1}{\rm sech}^{2}\left(\frac{1}{2}\sqrt{\frac{\mu_{1}}{\mu_{2}}}(x-c_{s}t-x_{0})\right),\nonumber\\
u(x,t)&=&B(\zeta(x,t)-\beta\zeta_{xx}(x,t)),\label{esw6}
\end{eqnarray}
with $B$ and $c_{s}$ given by (\ref{esw3a}) or (\ref{esw3b}), and $x_{0}\in\mathbb{R}$ is arbitrary.
\subsection{Numerical generation of solitary waves}
\label{sec42}
In this section some classical and generalized solitary wave profiles will be generated numerically. To this end and following a standard procedure, cf. e.~g. \cite{DougalisDM2012}, the system (\ref{BB6}) is discretized on a long enough interval $(-l,l)$ and with periodic boundary conditions by the Fourier collocation method based on $N$ collocation points given by $x_{j}=-l+jh, j=0,\ldots,N-1$ for an even integer $N\geq 1$ and where $h=2l/N$. If the vectors $\zeta_{h}=(\zeta_{h,0},\ldots,\zeta_{h,N-1})^{T}$ and $v_{h}=(v_{h,0},\ldots,v_{h,N-1})^{T}$ denote, respectively, the approximations to the values of $\zeta$ and $v_{\beta}$ at the $x_{j}$, then the discrete system satisfied by $\zeta_{h}$ and $v_{h}$ has the form
\begin{eqnarray}
S_{N}\begin{pmatrix} \zeta_{h}\\v_{h}\end{pmatrix}=\kappa_{\gamma,\delta}\begin{pmatrix} \zeta_{h}.v_{h}\\(v_{h}.^{2})/2\end{pmatrix},\label{421}
\end{eqnarray}
where $S_{N}$ is the $2N$-by-$2N$ matrix
\begin{eqnarray}
S_{N}:=\begin{pmatrix} c_{s}(I_{N}-b D_{N}^{2})&-(\kappa_{1}I_{N}+a D_{N}^{2})\\-\kappa_{2}(I_{N}+c D_{N}^{2})& c_{s}(I_{N}-d D_{N}^{2})\end{pmatrix},\label{421b}
\end{eqnarray}
with $I_{N}$ standing for the $N$-by-$N$ identity matrix and $D_{N}$ denoting the $N$-by-$N$ pseudospectral differentiation matrix. The products of the nonlinear terms on the right hand-side of (\ref{421}) are understood in the Hadamard (componentwise) sense. The system (\ref{421}), (\ref{421b}) is implemented in the Fourier space, that is, for the discrete Fourier components of $\zeta_{h}$ and $v_{h}$, leading to a $2$-by-$2$ system for each component of fixed-point form
\begin{eqnarray}
S(k)\begin{pmatrix} \widehat{\zeta_{h}}(k)\\\widehat{v_{h}}(k)\end{pmatrix}=\kappa_{\gamma,\delta}\begin{pmatrix} \widehat{\zeta_{h}.v_{h}}(k)\\\widehat{(v_{h}.^{2})/2}(k)\end{pmatrix},\label{422}
\end{eqnarray}
where
\begin{eqnarray}
S(k)=\begin{pmatrix} c_{s}(1+b\widetilde{k}^{2}&-(\kappa_{1}-a\widetilde{k}^{2})\\ -\kappa_{2}(1-c\widetilde{k}^{2})&c_{s}(1+d\widetilde{k}^{2}\end{pmatrix},\label{423}
\end{eqnarray}
with $\widetilde{k}=\pi k/l, -N/2\leq k\leq N/2-1$ and $\widehat{\zeta_{h}}(k), \widehat{v_{h}}(k) $ denoting, respectively, the $k$-th discrete Fourier component of $\zeta_{h}$ and $v_{h}$.
\begin{figure}[htbp]
\centering
\subfigure[]
{\includegraphics[width=0.8\textwidth]{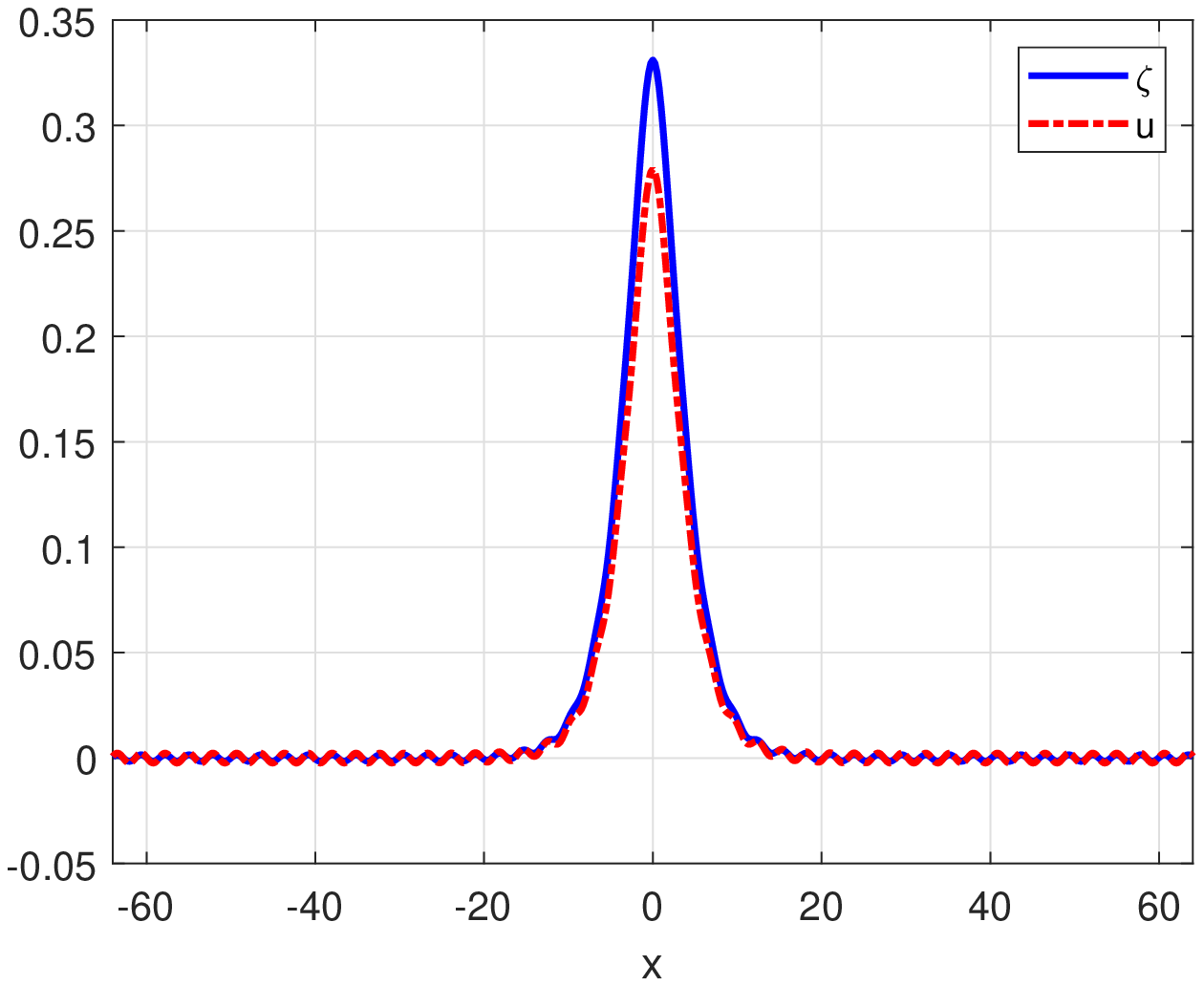}}
\subfigure[]
{\includegraphics[width=6.27cm]{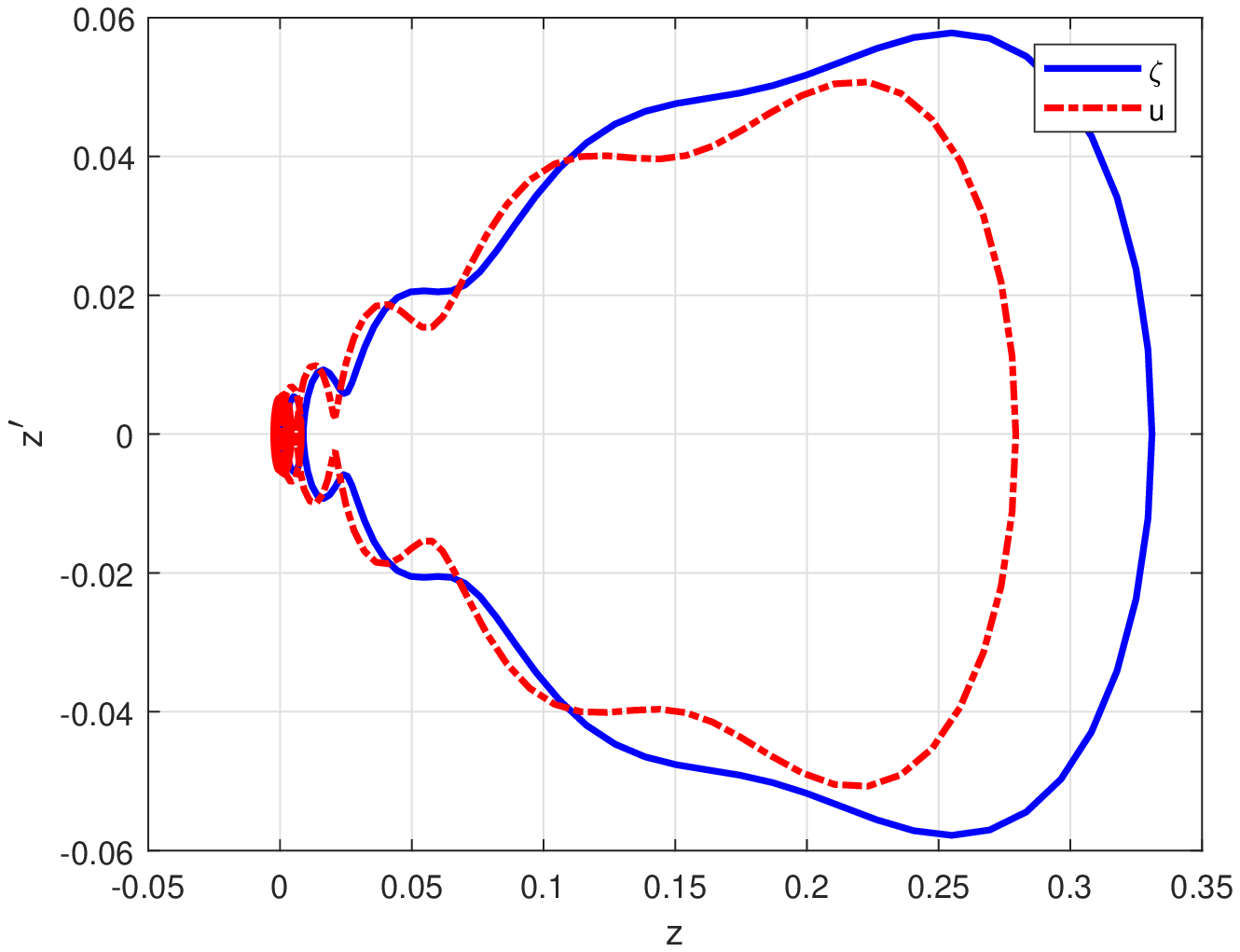}}
\subfigure[]
{\includegraphics[width=6.27cm]{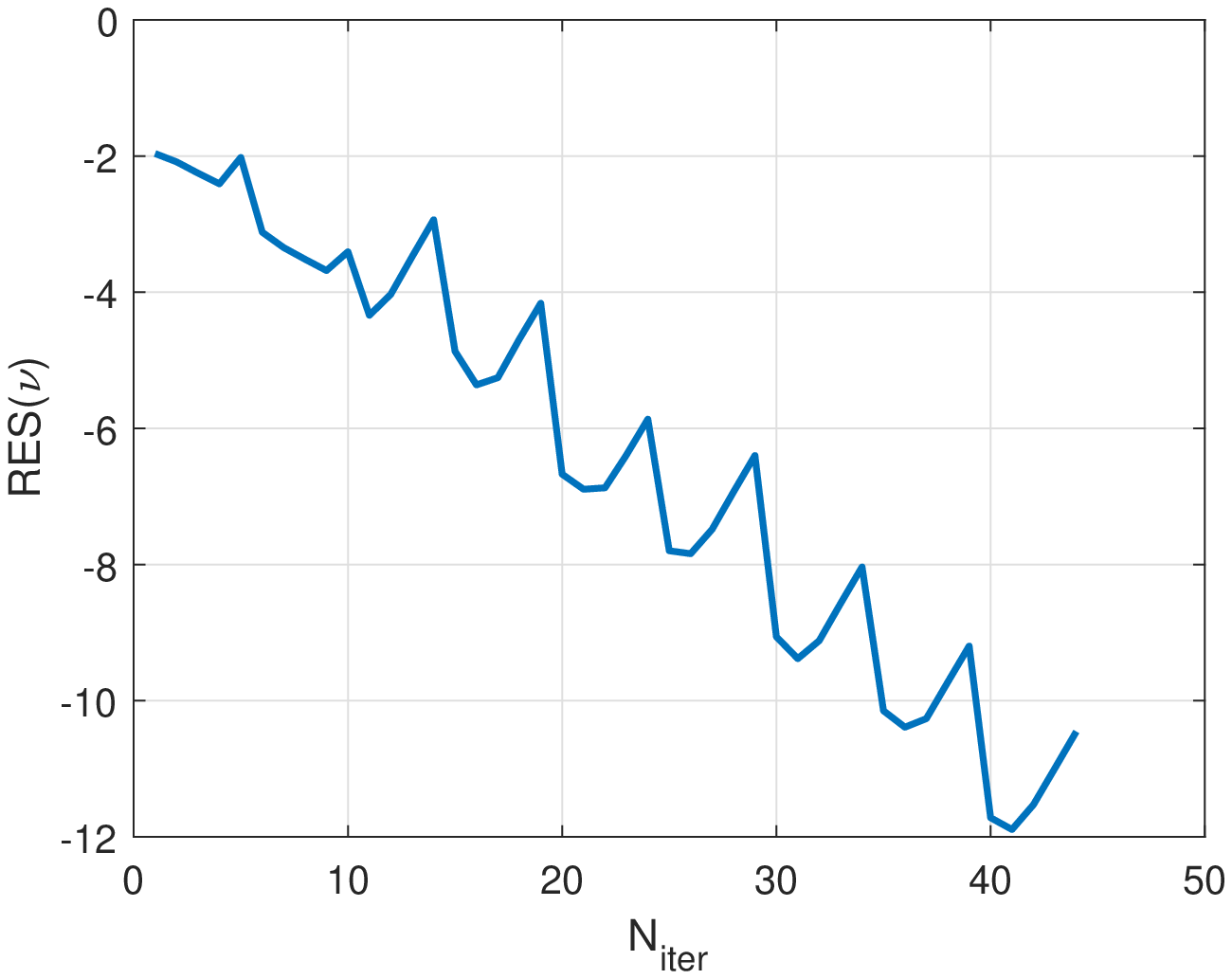}}
\caption{GSW generation. Case (A1) with $\delta=0.9, \gamma=0.5,  c=-1/3, b=0, d=1, a=\kappa_{1}(S(\gamma,\delta)-b-c-d)\approx -0.2022; c_{s}=c_{\gamma,\delta}+0.02\approx 0.6176$. (a) $\zeta$ and $u$ profiles; (b) $\zeta$ and $u$ phase portraits; (c) Residual error vs. number of iterations (semilog scale).}
\label{fig_BB4}
\end{figure}

%\begin{figure}[htbp]
%\centering
%\subfigure[]
%{\includegraphics[width=6.27cm]{case8_a.eps}}
%\subfigure[]
%{\includegraphics[width=6.27cm]{case8_b.eps}}
%\subfigure[]
%{\includegraphics[width=6.27cm]{case8_d.eps}}
%\subfigure[]
%{\includegraphics[width=6.27cm]{case8_e.eps}}
%\subfigure[]
%{\includegraphics[width=6.27cm]{case8_c.eps}}
%\caption{GSW generation. Case (A1) with $\delta=0.9, \gamma=0.5, a=-1/3, c=-1/3, b=0, d=-(\delta+\gamma)a-b-c+\frac{1+\gamma\delta}{3\delta(\gamma+\delta)}; c_{s}=c_{\gamma,\delta}+0.01.$ (a), (b) $\zeta$ and $u$ profiles; (c), (d) $\zeta$ and $u$ phase portraits; (e) Residual error vs. number of iterations.}
%\label{fig_BB1b}
%\end{figure}
Assuming that $S(k)$ given by (\ref{423}) is nonsingular  for all $k,  -N/2\leq k\leq N/2-1$, cf. (\ref{BB10}), then the system (\ref{422}) is solved iteratively with the 
Petviashvili scheme, \cite{Petv1976,pelinovskys},
\begin{eqnarray}
S(k)\begin{pmatrix} \widehat{\zeta_{h}^{[\nu+1]}}(k)\\\widehat{v_{h}^{[\nu+1]}}(k)\end{pmatrix}&=&m_{h}^{2}\kappa_{\gamma,\delta}\begin{pmatrix} \widehat{\zeta_{h}^{[\nu]}.v_{h}^{[\nu]}}(k)\\\widehat{((v_{h}^{[\nu]}).^{2})/2}(k)\end{pmatrix},\nonumber\\
&& \nu=0,1,\ldots,\; -N/2\leq k\leq N/2-1,\label{424}
\end{eqnarray}
where $m_{h}$ is the corresponding stabilizing factor
$$
m_{h}={\left(S_{N}\begin{pmatrix} \zeta_{h}^{[\nu]}\\ v_{h}^{[\nu]}\end{pmatrix}, \begin{pmatrix} \zeta_{h}^{[\nu]}\\ v_{h}^{[\nu]}\end{pmatrix}\right)_{N}}/{\left(\begin{pmatrix} \zeta_{h}^{[\nu]}.v_{h}^{[\nu]}\\((v_{h}^{[\nu]}).^{2}/2)(k)\end{pmatrix}, \begin{pmatrix} \zeta_{h}^{[\nu]}\\ v_{h}^{[\nu]}\end{pmatrix}\right)_{N}},
$$ 
with $(\cdot,\cdot)_{N}$ denoting the Euclidean inner product in $\mathbb{C}^{2N}$. The iterative procedure (\ref{424}) is in some cases accelerated by using vector extrapolation methods, \cite{sidi,sidifs,smithfs}. For the application of these techniques to the Petviashvili's method for traveling wave computations see \cite{AlvarezD2015}. The benefits of their use include a reduction in the number of iterations when the Petviahsvili's method is convergent, and the transformation of divergent cases into convergent. Once the iteration is completed and approximations $\zeta_{h}$ and $v_{h}$ are computed, an approximation of $u=(1-\beta\partial_{xx})v_{\beta}$ can be obtained as $u_{h}=(I_{N}-\beta D_{N}^{2})v_{h}$.

\begin{figure}[htbp]
\centering
\subfigure[]
{\includegraphics[width=0.8\textwidth]{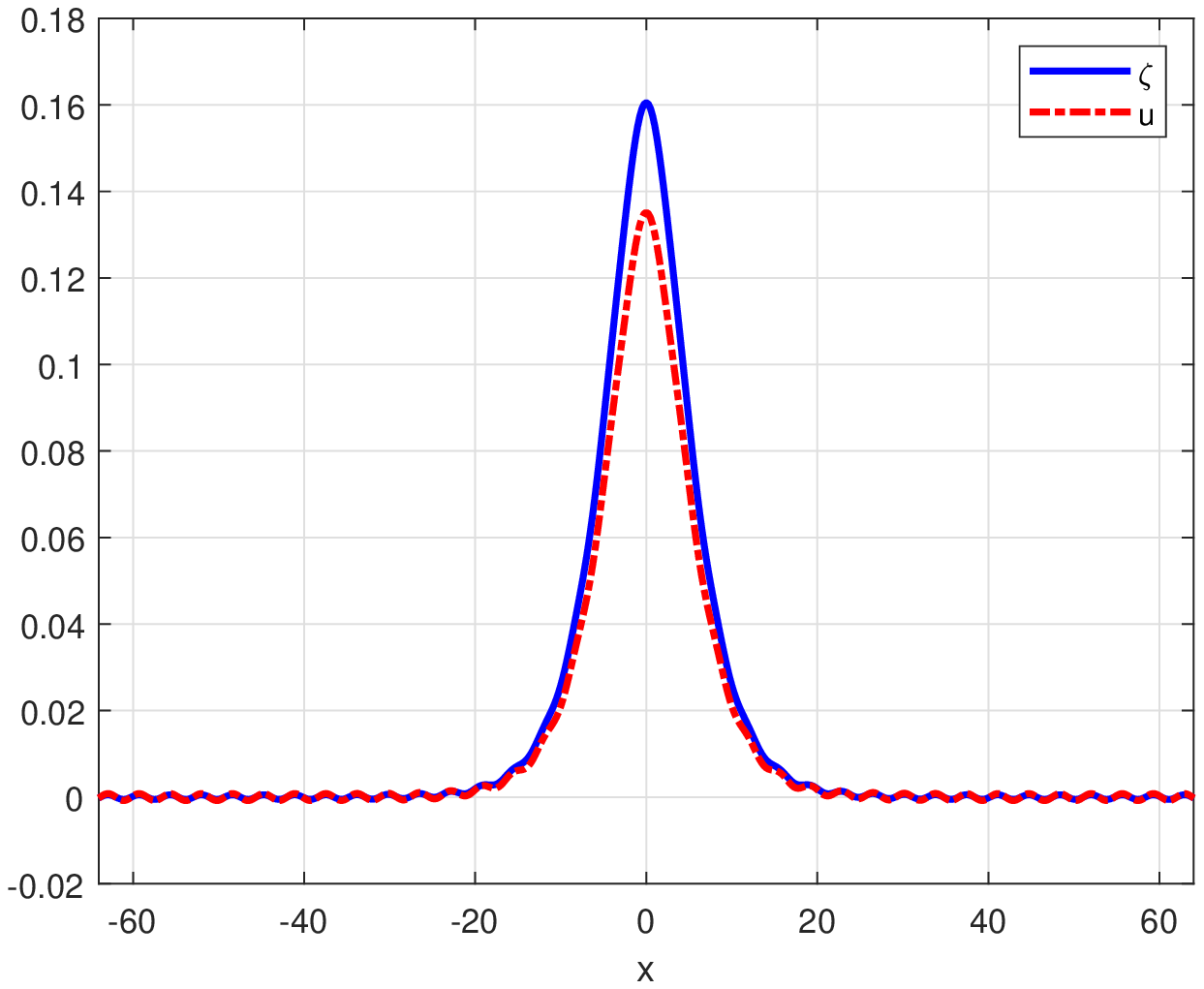}}
\subfigure[]
{\includegraphics[width=6.25cm,height=5.2cm]{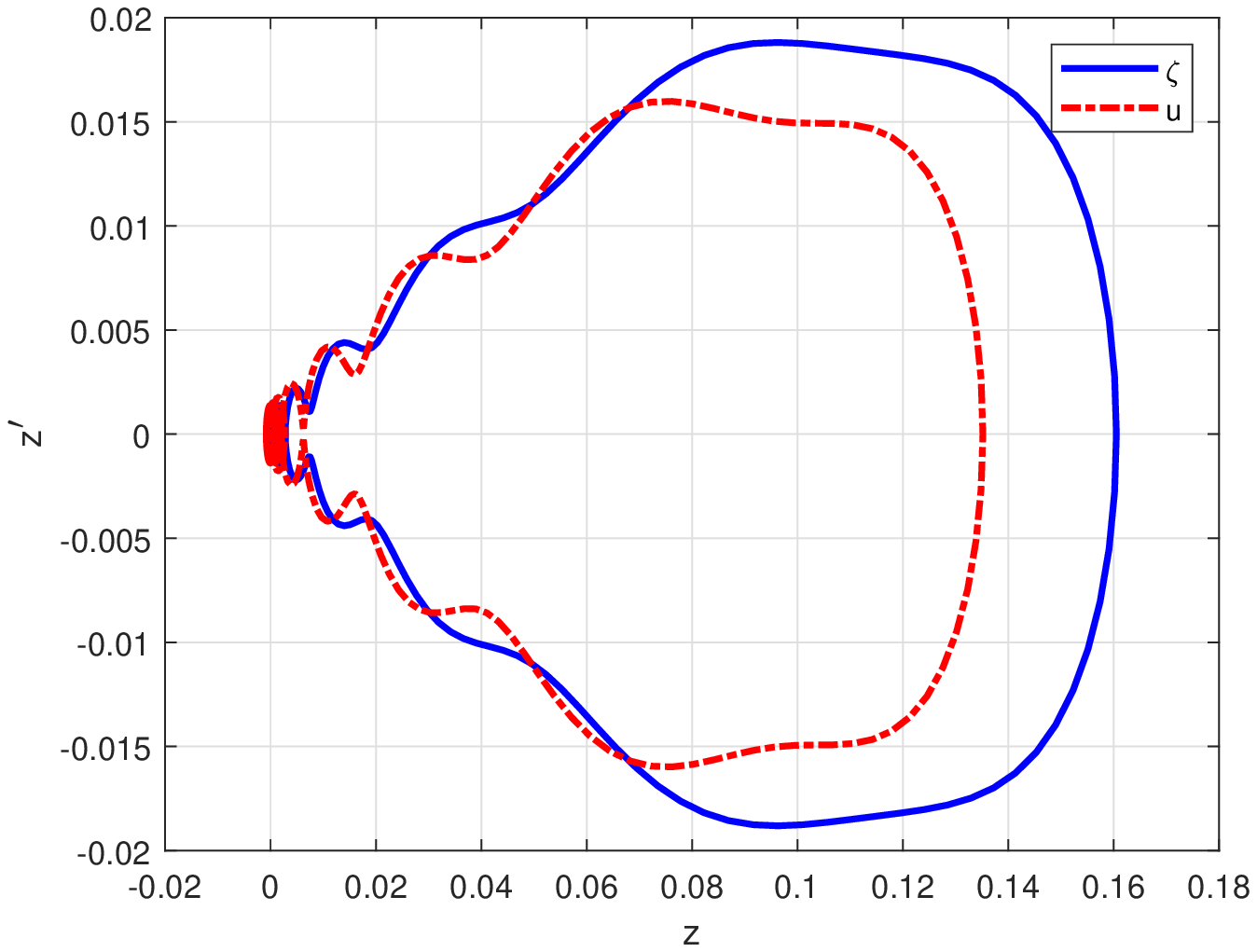}}
\subfigure[]
{\includegraphics[width=6.25cm,height=5.2cm]{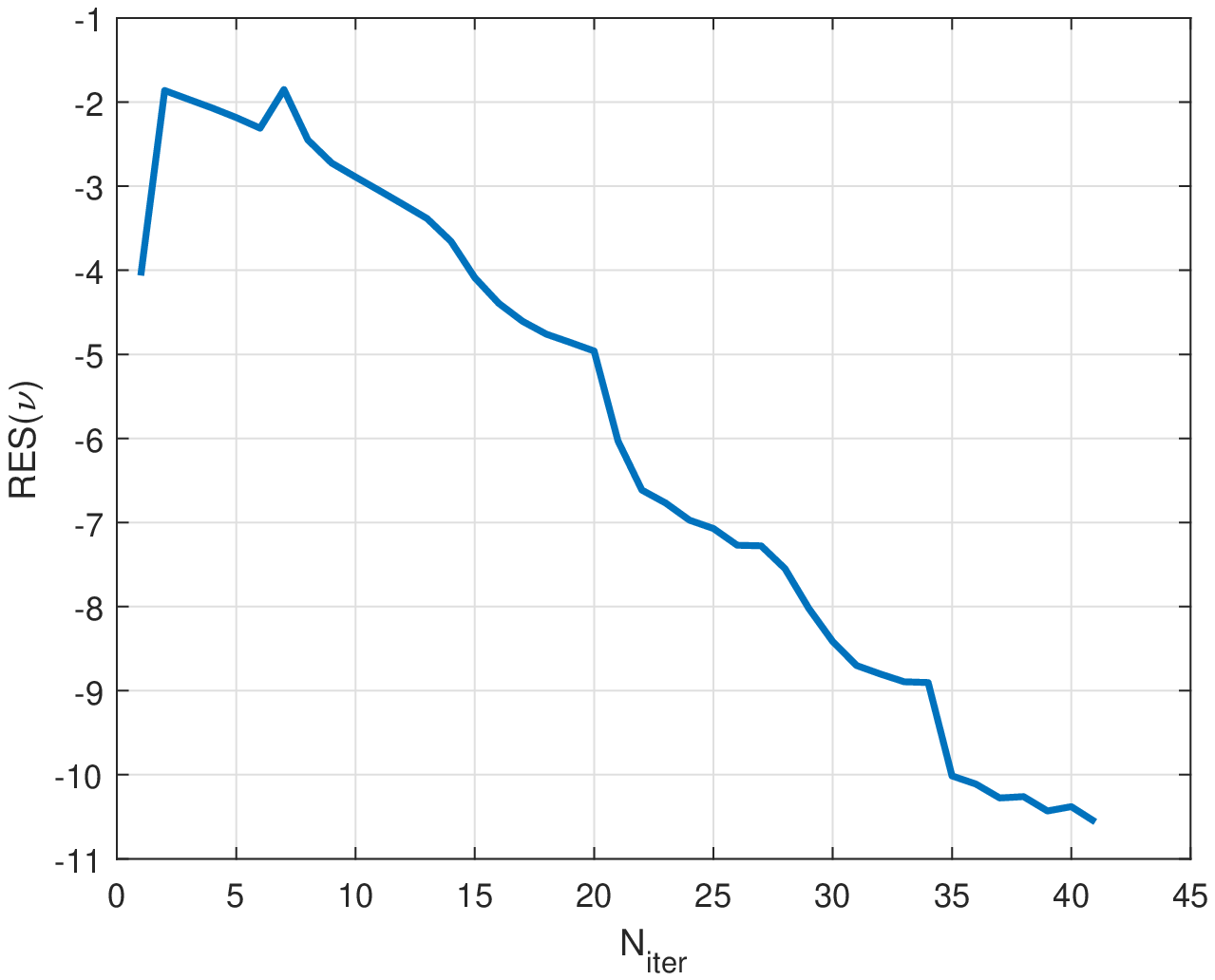}}
\caption{GSW generation. Case (A2) with $\delta=0.9, \gamma=0.5, a=-1/3, c=-2/3, b=1/9, d=S-a/\kappa_{1}-b-c\approx 1.4058; c_{s}=c_{\gamma,\delta}+0.01\approx 0.6076.$ (a) $\zeta$ and $u$ profiles; (b) $\zeta$ and $u$ phase portraits; (c) Residual error vs. number of iterations (semilog scale).}
\label{fig_BB5}
\end{figure}
%\begin{figure}[htbp]
%\centering
%\subfigure[]
%{\includegraphics[width=6.27cm]{case9_a.eps}}
%\subfigure[]
%{\includegraphics[width=6.27cm]{case9_b.eps}}
%\subfigure[]
%{\includegraphics[width=6.27cm]{case9_d.eps}}
%\subfigure[]
%{\includegraphics[width=6.27cm]{case9_e.eps}}
%\subfigure[]
%{\includegraphics[width=6.27cm]{case9_c.eps}}
%\caption{GSW generation. Case (A2) with $\delta=0.9, \gamma=0.5, a=-1/3, c=-1/3, b=1/9, d=d=-(\delta+\gamma)a-b-c+\frac{1+\gamma\delta}{3\delta(\gamma+\delta)}; c_{s}=c_{\gamma,\delta}+1E-02.$ (a), (b) $\zeta$ and $u$ profiles; (c), (d) $\zeta$ and $u$ phase portraits; (e) Residual error vs. number of iterations.}
%\label{fig_BB2b}
%\end{figure}

Some details on the implementation are now given. For the experiments below $h$ varies in a range between $6.25\times 10^{-2}$ and $2.5\times 10^{-1}$. In all the computations the approximate profiles for $\zeta$ and $u$ and corresponding phase portraits are displayed. The accuracy of the profiles is monitored in two ways. First, the behaviour of the residual error at each iteration
\begin{eqnarray}
RES(\nu)=\left|\left|S_{N}\begin{pmatrix} \zeta_{h}^{[\nu]}\\{v_{h}^{[\nu]}}\end{pmatrix}-\begin{pmatrix} {\zeta_{h}^{[\nu]}.v_{h}^{[\nu]}}\\{((v_{h}^{[\nu]}).^{2})/2}\end{pmatrix}\right|\right|,\label{RES}
\end{eqnarray}
where $||\cdot ||$ denotes the Euclidean norm, is checked and displayed. A second test of accuracy consists of integrating numerically the periodic ivp of (\ref{BB2}) by some fully discrete scheme, considering the computed profiles as initial condition and monitoring several error indicators during the evolution of the corresponding numerical solution. For the experiments to follow, we use the Fourier pseudospectral discretization in space (justified by the analysis of convergence made in section \ref{sec3}), and, as time integrator, we use a fourth-order Runge-Kutta (RK)-composition method based on the Implicit Midpoint Rule. The fully discrete method and the evolution experiments will be described 
in section \ref{sec5}.

In the rest of the present section we illustrate several cases of numerically generated solitary waves corresponding to the cases (A1)-(A6) of Table \ref{tavle0}, to the case $D=0$ and some exceptional cases predicted by NFT. In each one of Figures \ref{fig_BB4}- we present (a) the profiles of the $\zeta-$ and $u-$solitary waves, (b) their phase space diagrams, and (c) a graph of the residual error defined by (\ref{RES}) versus the number of iterations required to reach a certain level of residual error down to a minimum value of about $10^{-11}$.

%The cases (A1)-(A6) of Table \ref{tavle0} and the case $D=0$ have been illustrated

%As mentioned before, each case from (A1) to (A6) of Table \ref{tavle0} as well as the case $D=0$ are illustrated in the experiments by plotting the approximate profiles for $\zeta$ and $u$, the corresponding phase portraits and the behaviour (in semilog scale) of the residual error (\ref{RES}) as function of the number of iterations in order to check the convergence of the process. 

\begin{figure}[htbp]
\centering
\subfigure[]
{\includegraphics[width=0.8\textwidth]{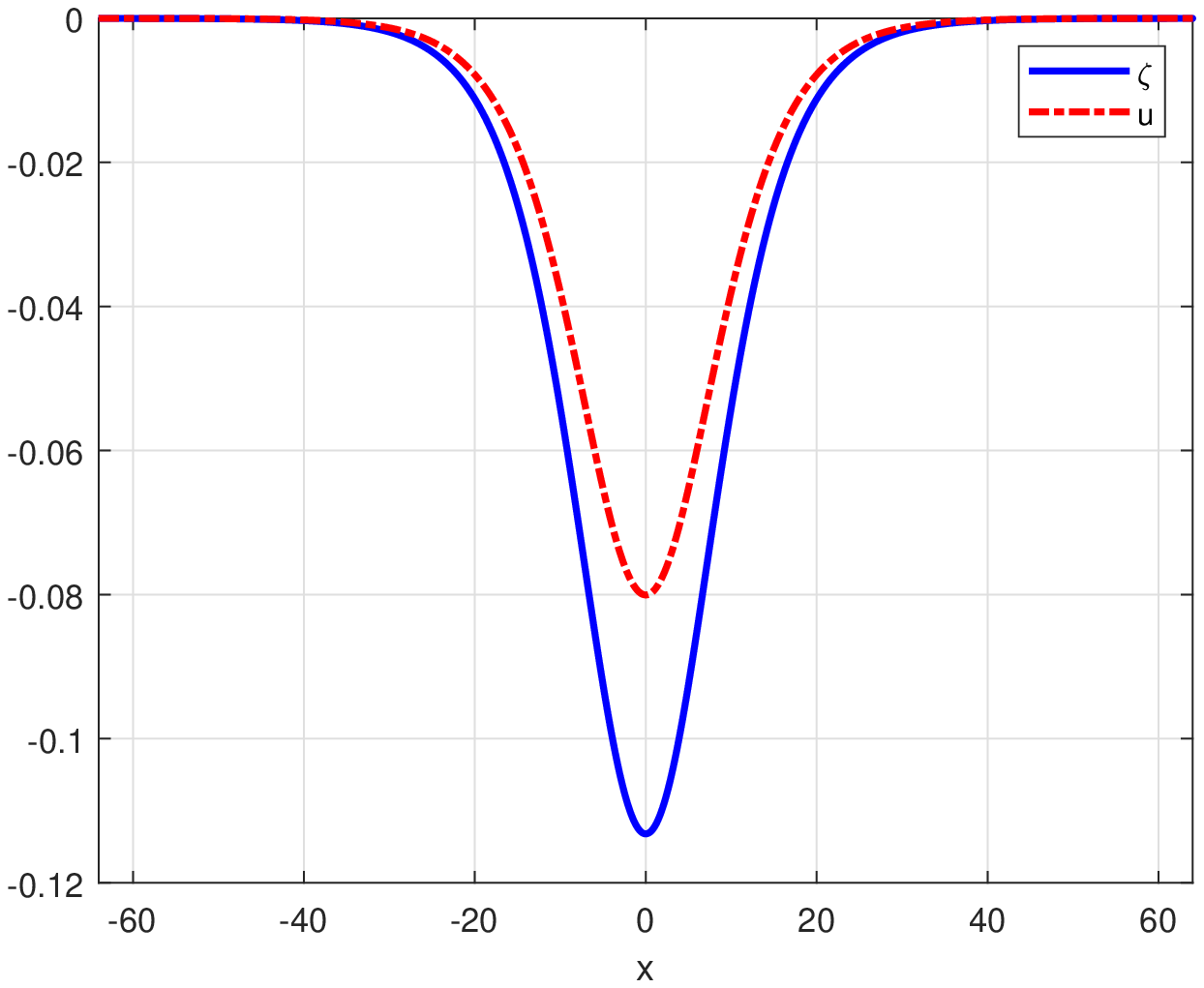}}
\subfigure[]
{\includegraphics[width=6.27cm]{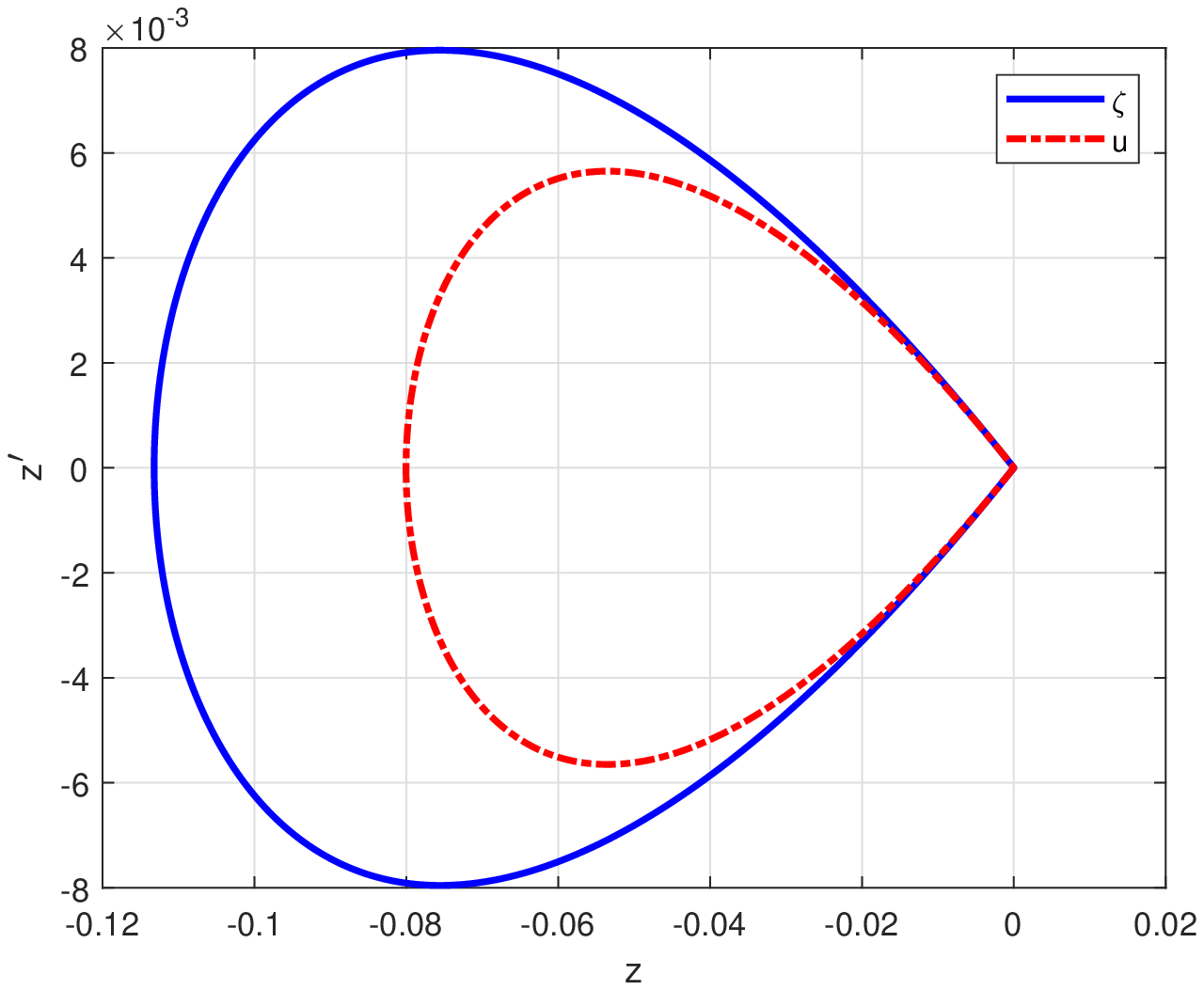}}
\subfigure[]
{\includegraphics[width=6.27cm]{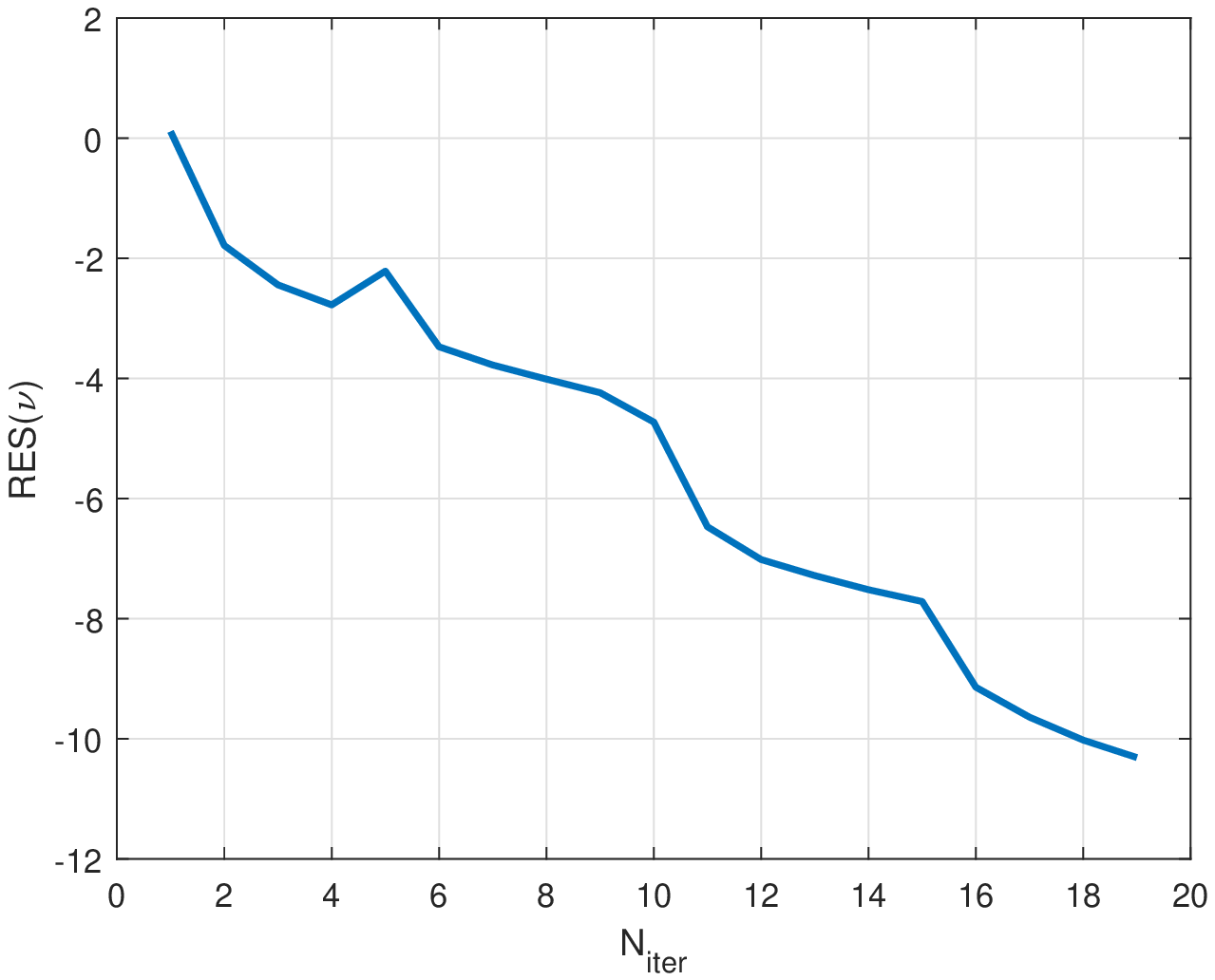}}
\caption{CSW generation. Case (A4) with $\delta=\gamma=0.5, a=0, c=-1/3, b=1, d=1/6; c_{s}=c_{\gamma,\delta}+1E-02.$ (a) $\zeta$ and $u$ profiles; (c) $\zeta$ and $u$ phase portraits; (c) Residual error vs. number of iterations (semilog scale).}
\label{fig_BB1}
\end{figure}
\begin{figure}[htbp]
\centering
\subfigure[]
{\includegraphics[width=0.8\textwidth]{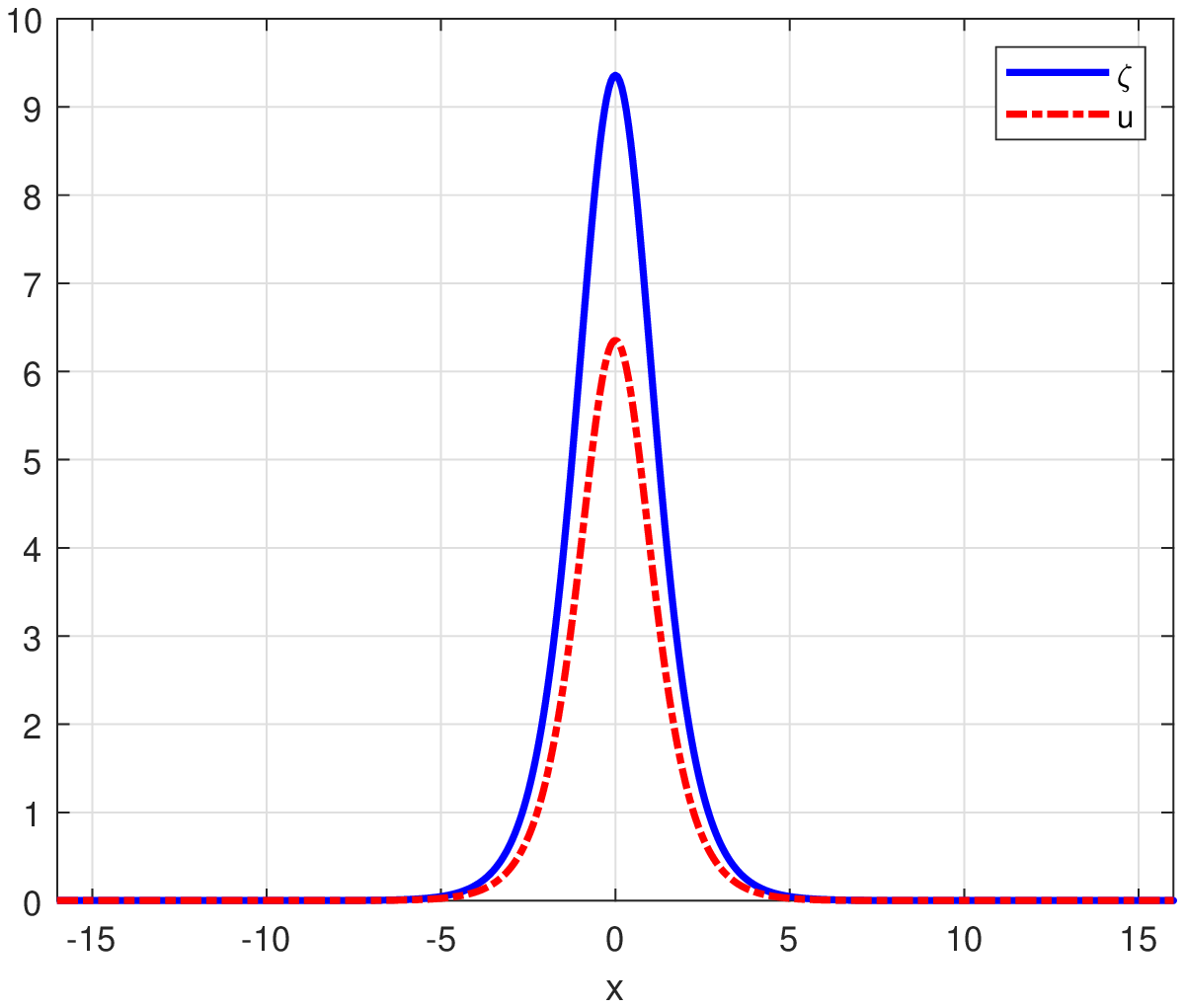}}
\subfigure[]
{\includegraphics[width=6.27cm]{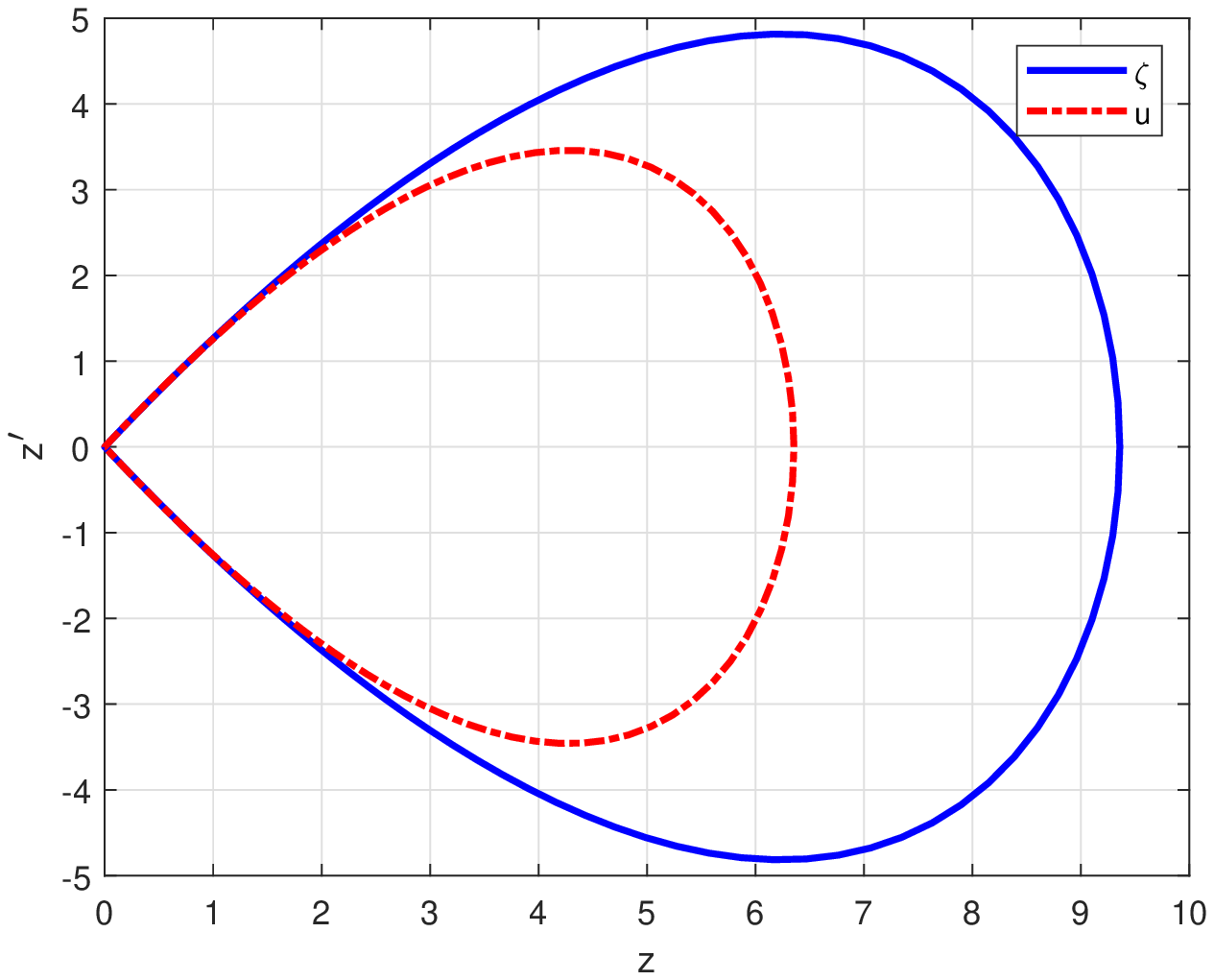}}
\subfigure[]
{\includegraphics[width=6.25cm]{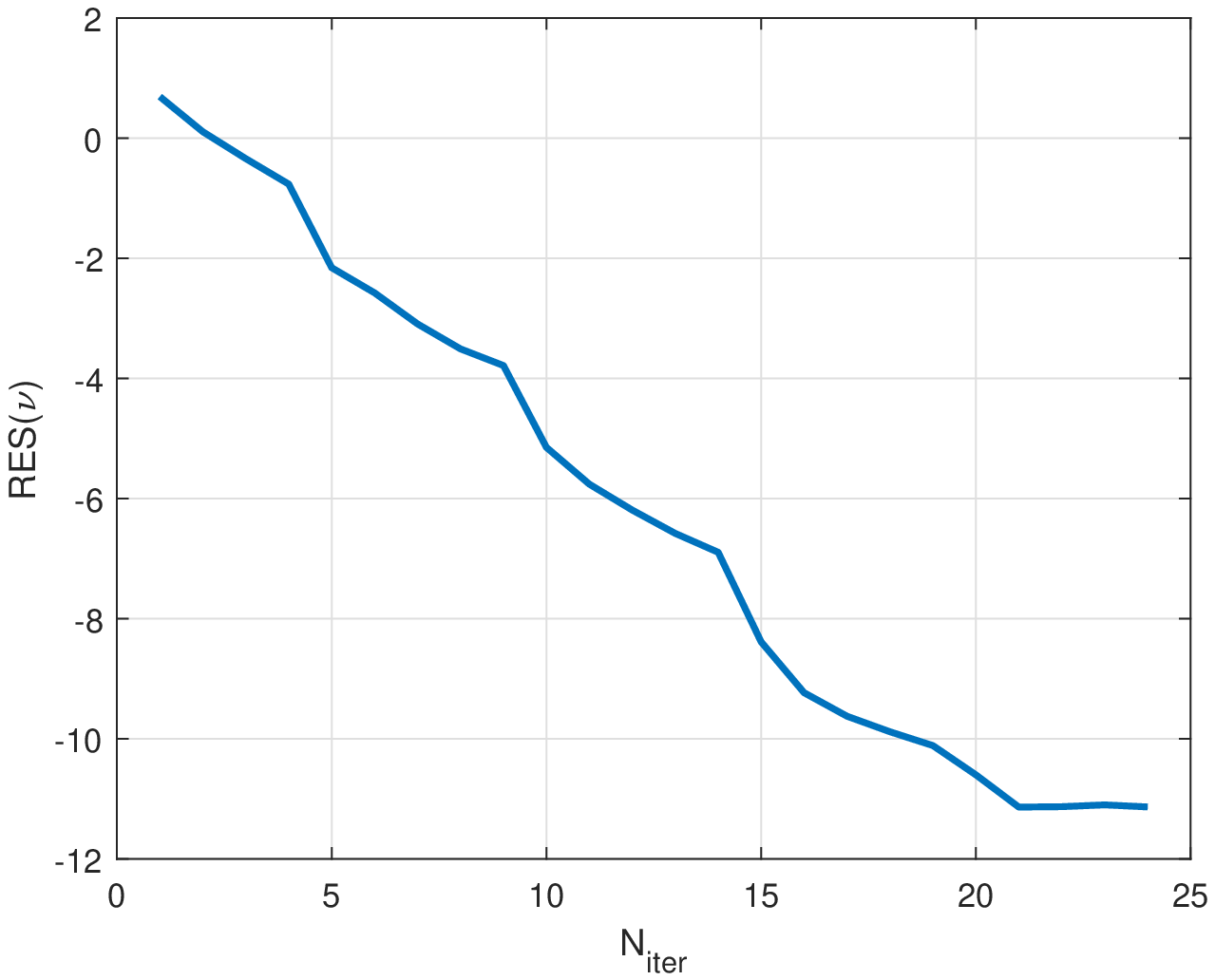}}
\caption{CSW generation. Case (A4) with $\delta=0.9, \gamma=0.5, a=0, c=-1/3, b=1/3, d=-(\delta+\gamma)a-b-c+\frac{1+\gamma\delta}{3\delta(\gamma+\delta)}; c_{s}=c_{\gamma,\delta}+0.5.$ (a) $\zeta$ and $u$ profiles; (c) $\zeta$ and $u$ phase portraits; (c) Residual error vs. number of iterations (semilog scale).}
\label{fig_BB4b}
\end{figure}
We start with two illustrations,  in Figures \ref{fig_BB4} and \ref{fig_BB5},
of GSW's, predicted by NFT (cases (A1) and (A2), respectively, of Table \ref{tavle0}). Figure \ref{fig_BB4} corresponds to the case (A1)  with 
 $$\delta=\gamma=0.5, c=-1/3, b=0, d=1, a=\kappa_{1}(S(\gamma,\delta)-b-c-d)\approx -0.2022.$$ The speed is $c_{s}=c_{\gamma,\delta}+0.02\approx 0.6176$.  Figure \ref{fig_BB5} illustrates the case (A2), with  
$$\delta=0.9, \gamma=0.5, a=-1/3, c=-2/3, b=1/9, d=S-a/\kappa_{1}-b-c\approx 1.4058,$$ and $c_{s}=c_{\gamma,\delta}+0.01\approx 0.6076$.
Both are waves of elevation and the corresponding phase portraits show how the profiles are homoclinic to a periodic orbit at infinity. We checked that these homoclinic orbits correspond to region 3 of Figure \ref{fig_A1}.

We turn now to the numerical generation of CSW's. First we illustrate
the cases (A3)-(A6) of Table \ref{tavle0} with examples that are covered by 
the other three theories considered in section \ref{sec41}, namely Toland's Theory, Concentration-Compactness Theory and Positive Operator Theory. The associated homoclinic orbits correspond to region 2 of the $(B,A)$-plane in Figure \ref{fig_A1}. 

%We complete the results with the computation of approximate CSW's when $D=0$ in (\ref{NFTD}) and the numerical generation of CSW's with nonmonotone decay, predicted by the NFT (cf. section \ref{sec41}) for speeds $c_{s}<c_{\gamma,\delta}$.

\begin{figure}[htbp]
\centering
\subfigure[]
{\includegraphics[width=0.8\textwidth]{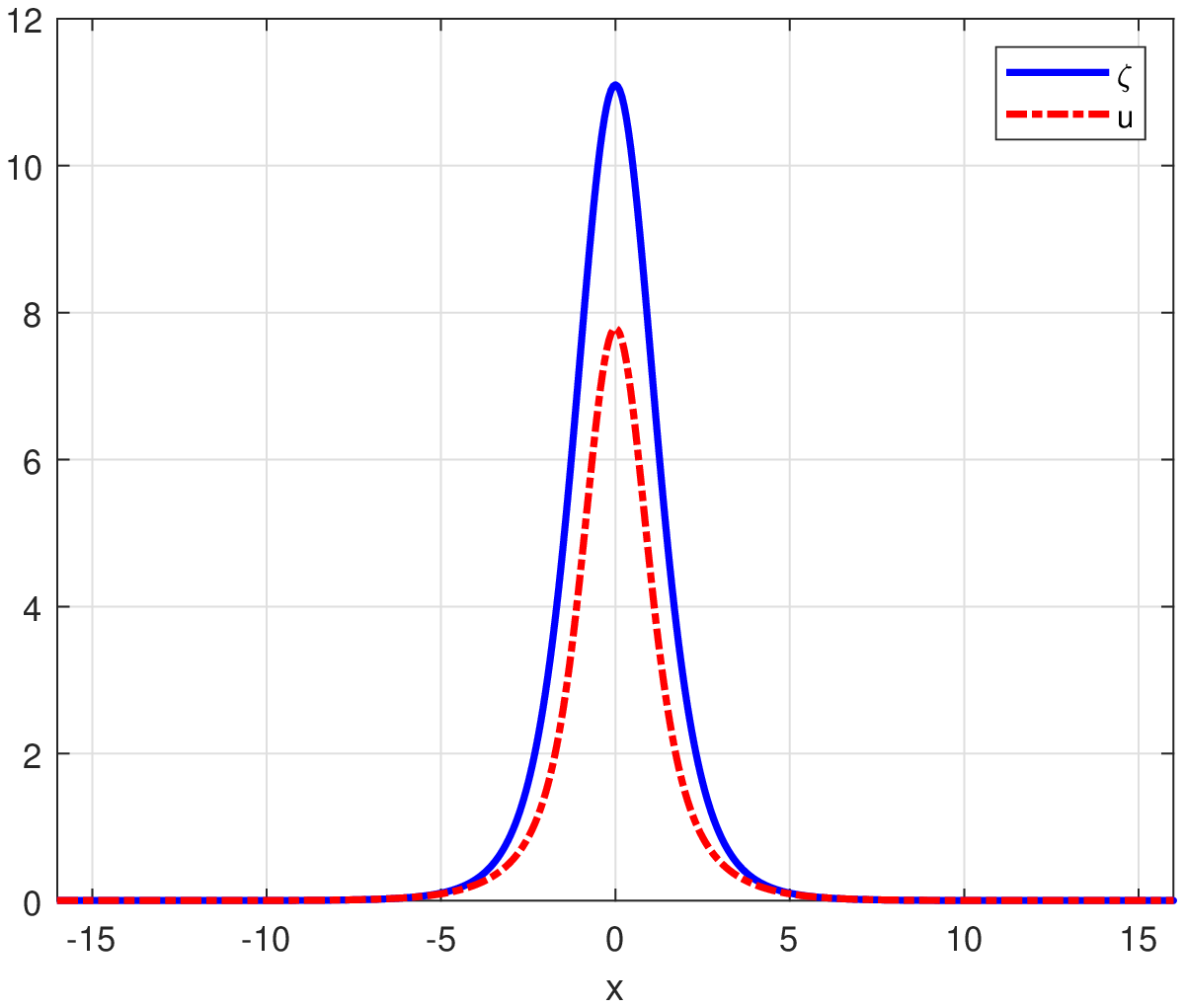}}
\subfigure[]
{\includegraphics[width=6.27cm]{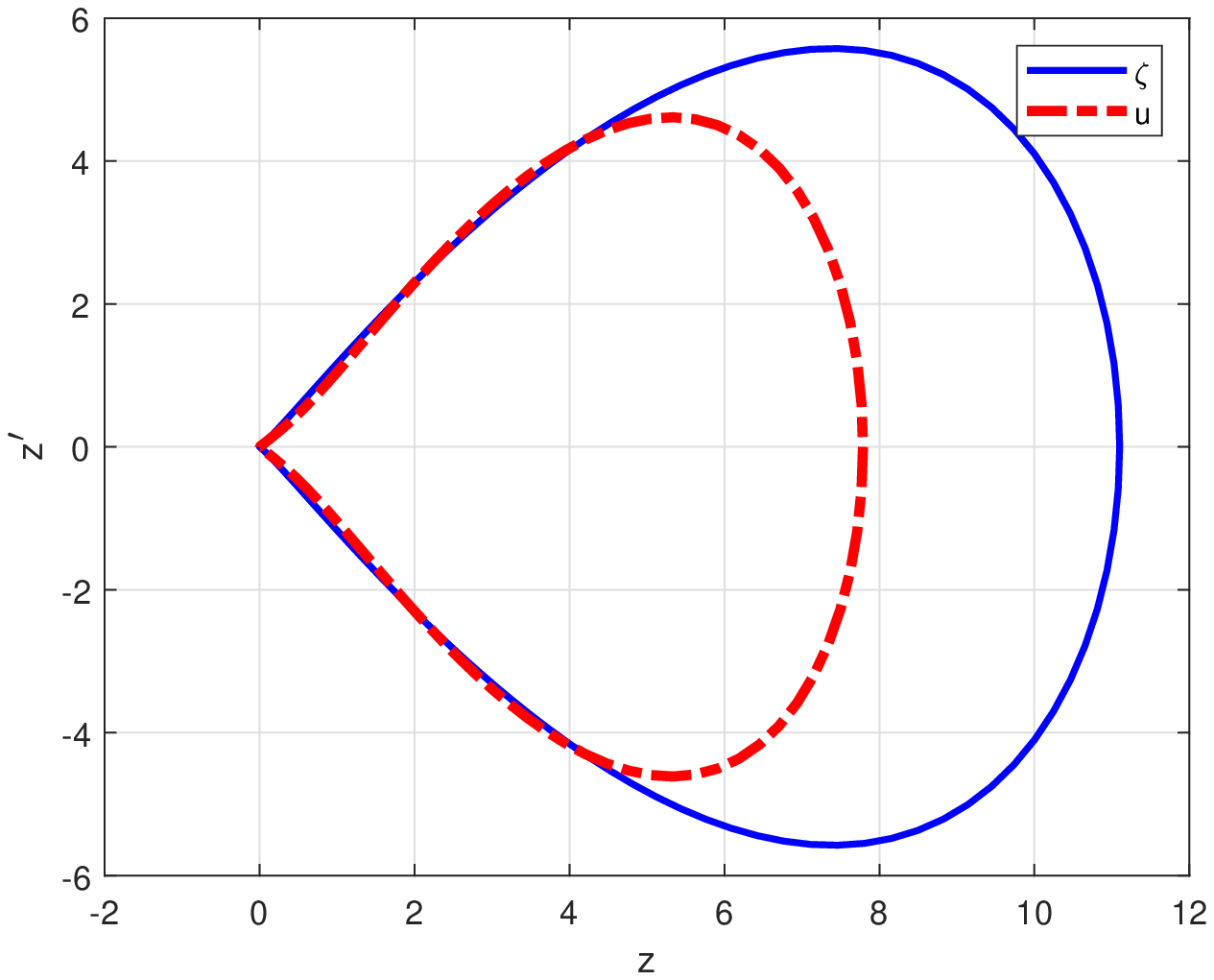}}
\subfigure[]
{\includegraphics[width=6.27cm]{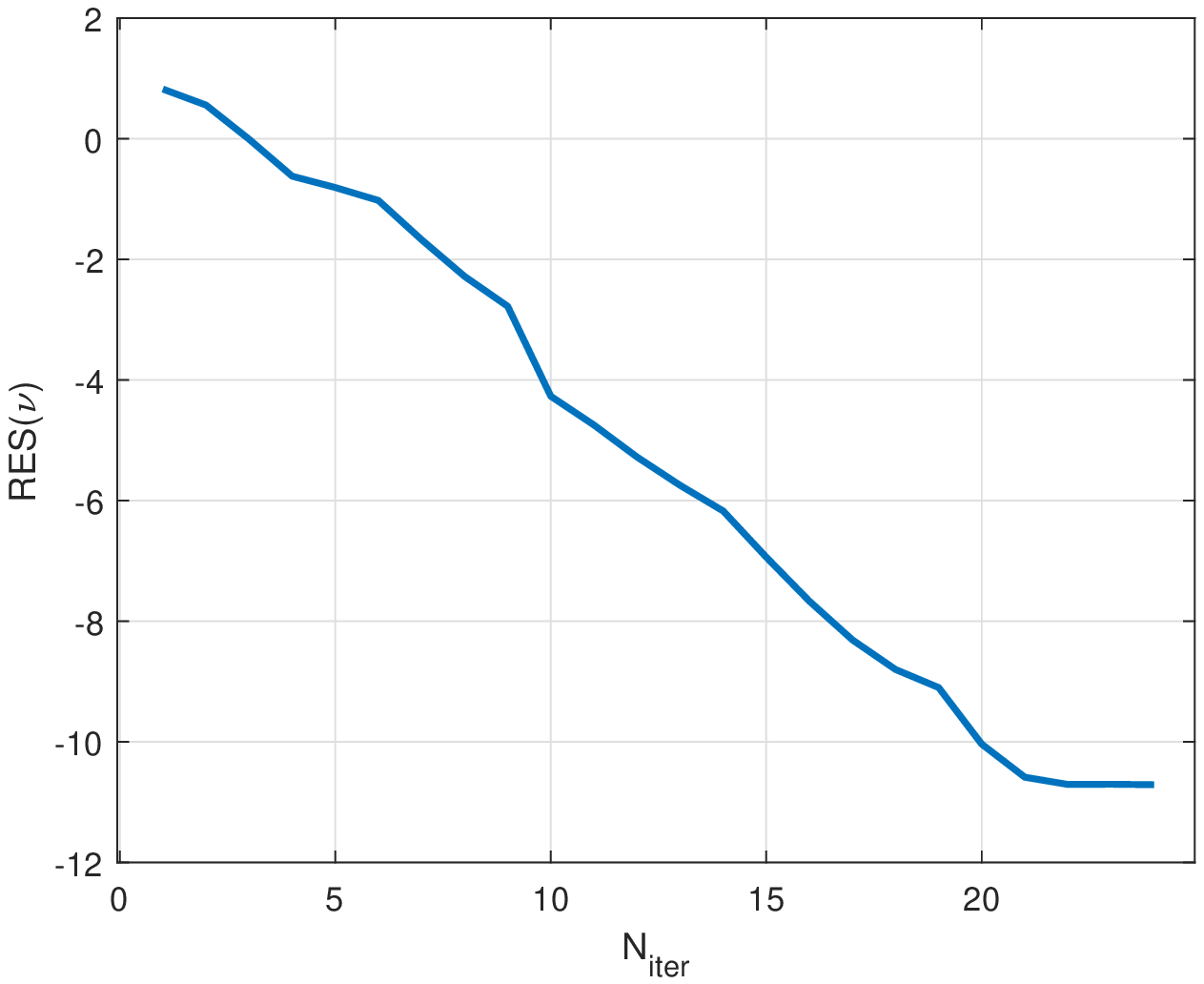}}
%\subfigure[]
%{\includegraphics[width=6.27cm]{case10_e.eps}}
%\subfigure[]
%{\includegraphics[width=6.27cm]{case10_c.eps}}
\caption{CSW generation. Case (A3) with $\delta=0.9, \gamma=0.5, a=-1/3, c=-2/3, b=1/3, d=-(\delta+\gamma)a-b-c+\frac{1+\gamma\delta}{3\delta(\gamma+\delta)}\approx 1.1836; c_{s}=c_{\gamma,\delta}+0.5\approx=1.0976.$ (a) $\zeta$ and $u$ profiles; (c) $\zeta$ and $u$ phase portraits; (c) Residual error vs. number of iterations (semilog scale).}
\label{fig_BB3b}
\end{figure}

\begin{figure}[htbp]
\centering
\subfigure[]
{\includegraphics[width=0.8\textwidth]{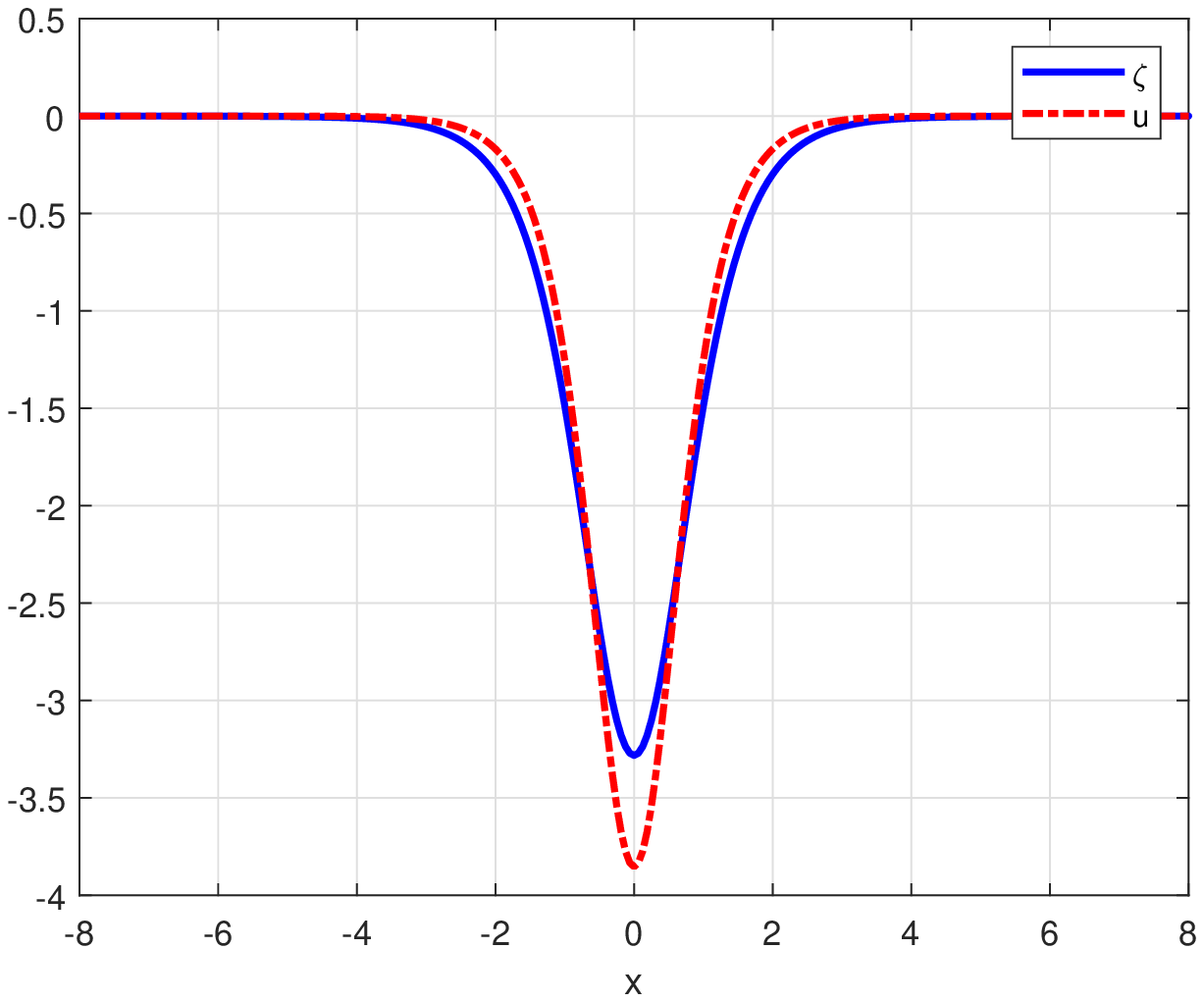}}
\subfigure[]
{\includegraphics[width=6.27cm]{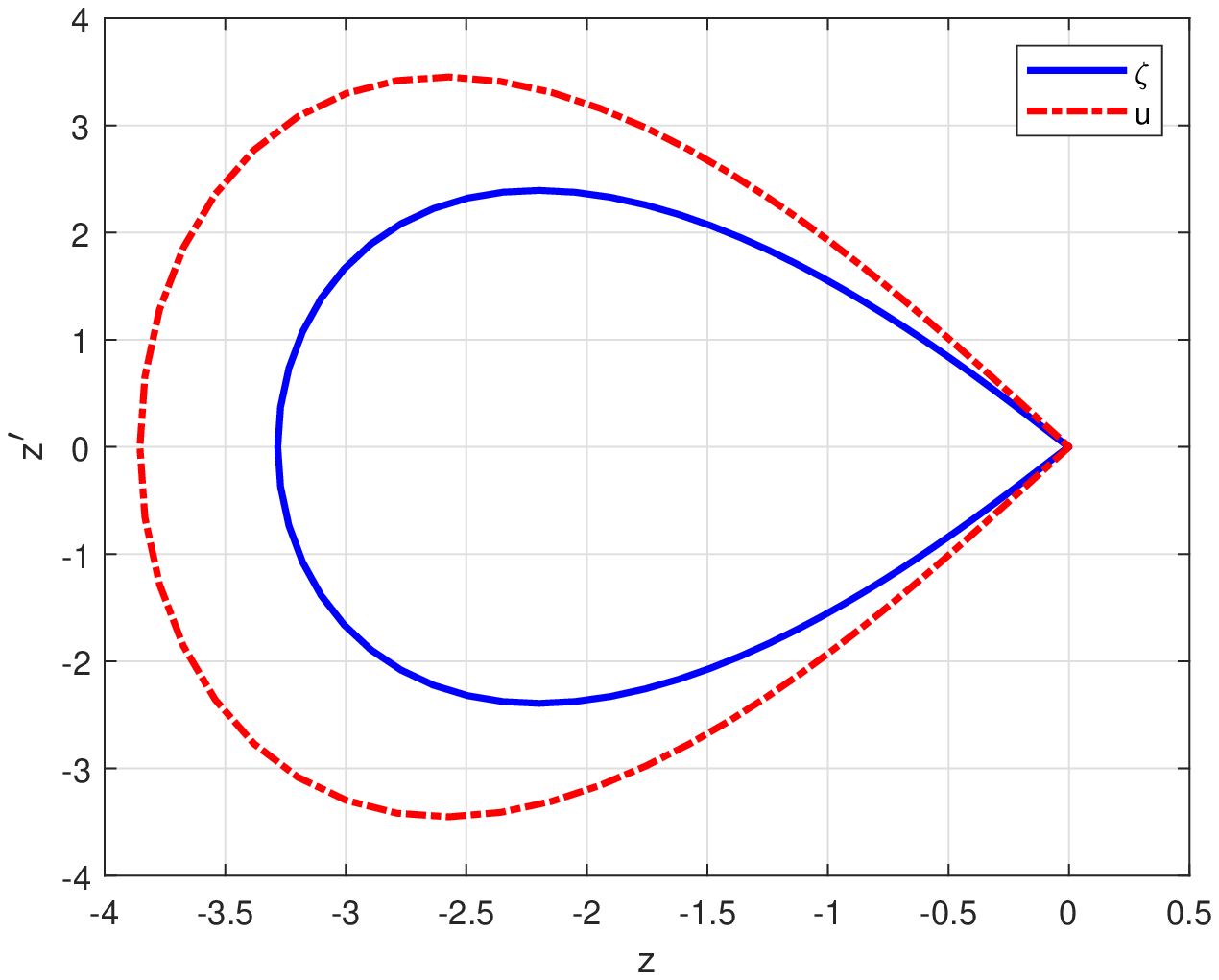}}
\subfigure[]
{\includegraphics[width=6.27cm]{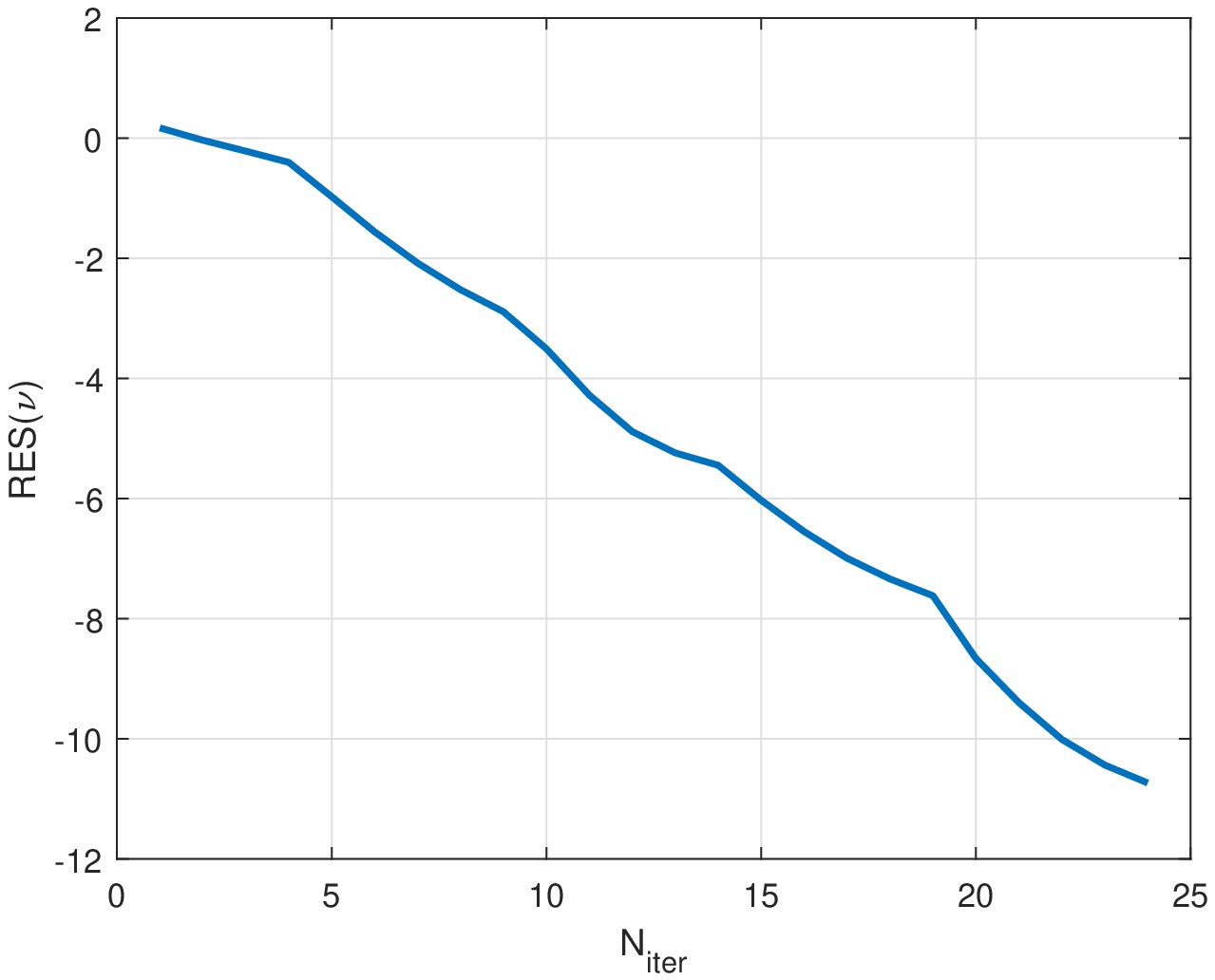}}
\caption{CSW generation. Case (A3) with $\delta=0.9, \gamma=0.5, a=-1/3, c=-2/3, b=d=\frac{1}{2}(-(\delta+\gamma)a-c+\frac{1+\gamma\delta}{3\delta(\gamma+\delta)})\approx 0.7585; c_{s}=c_{\gamma,\delta}-0.25\approx 0.3476.$  (a) $\zeta$ and $u$ profiles; (c) $\zeta$ and $u$ phase portraits; (c) Residual error vs. number of iterations (semilog scale).}
\label{fig_BB3bb}
\end{figure}
The case (A3) is exemplified in Figures \ref{fig_BB3b} and \ref{fig_BB3bb}. In Figure \ref{fig_BB3b}, we took $\delta=0.9, \gamma=0.5$, $a=-1/3, c=-2/3, b=1/3, d=-(\delta+\gamma)a-b-c+\frac{1+\gamma\delta}{3\delta(\gamma+\delta)}\approx 1.1836,$
 and $c_{s}=c_{\gamma,\delta}+0.5\approx=1.0976$. The resulting classical solitary wave is of elevation type. It was also checked that $\Delta_{j}>0, 0\leq j\leq 3$, where $\Delta_{j}$ are given by (\ref{445a}). Hence this example illustrates Theorem \ref{pot}, i.~e. is an application of Positive Operator Theory to (\ref{BB6}). On the other hand, an application of Theorem \ref{th_CC}, deduced from the Concentration-Compactness Theory, is illustrated in Figure \ref{fig_BB3bb}. In this case, with the same values of $\delta$ and $\gamma$, $a$, and $c$, but different $b$ and $d$, the computed CSW is of depression type and has speed  
$c_{s}=c_{\gamma,\delta}-0.25\approx 0.3476$.

\begin{figure}[htbp]
\centering
\subfigure[]
{\includegraphics[width=0.8\textwidth]{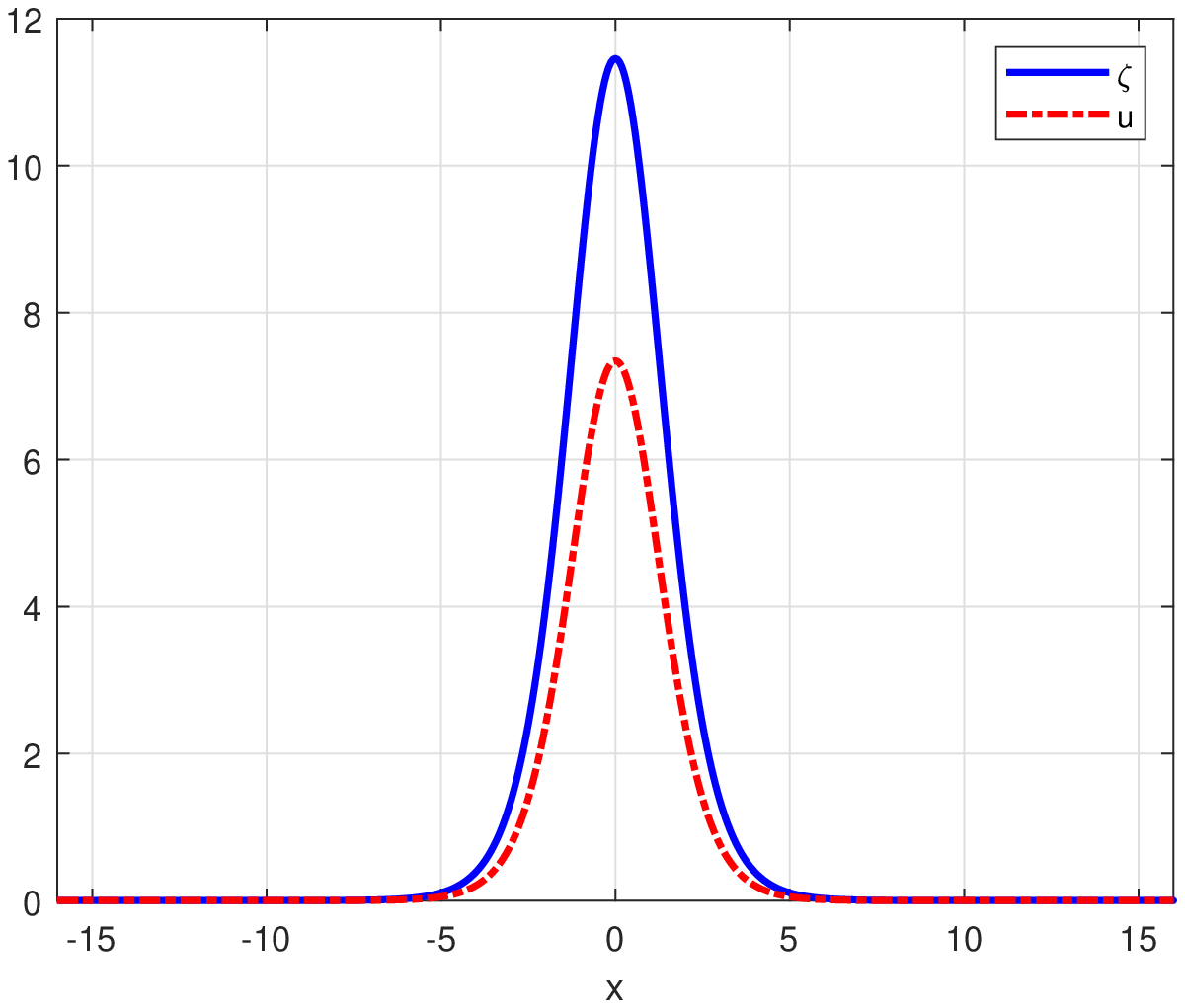}}
\subfigure[]
{\includegraphics[width=6.27cm]{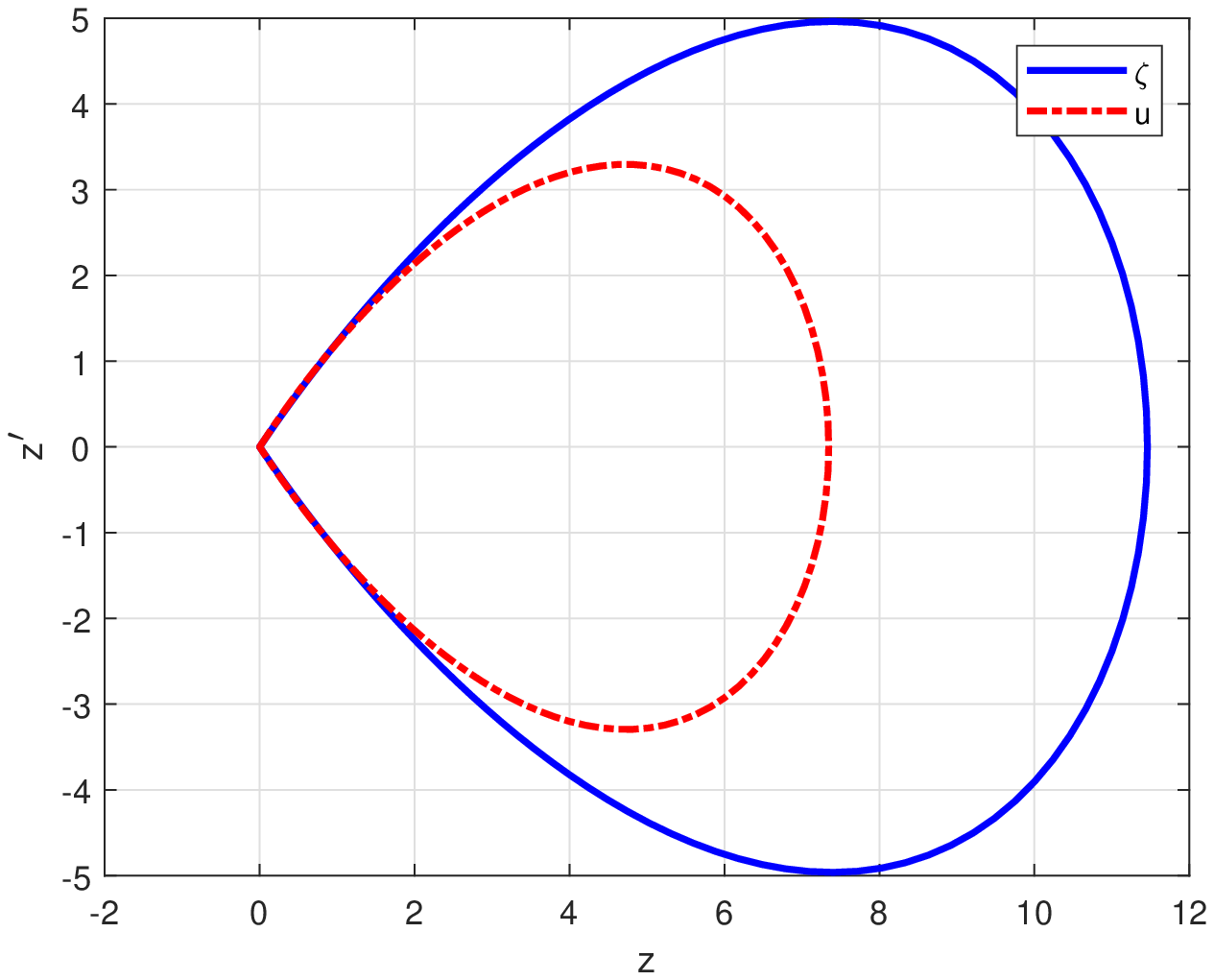}}
\subfigure[]
{\includegraphics[width=6.27cm]{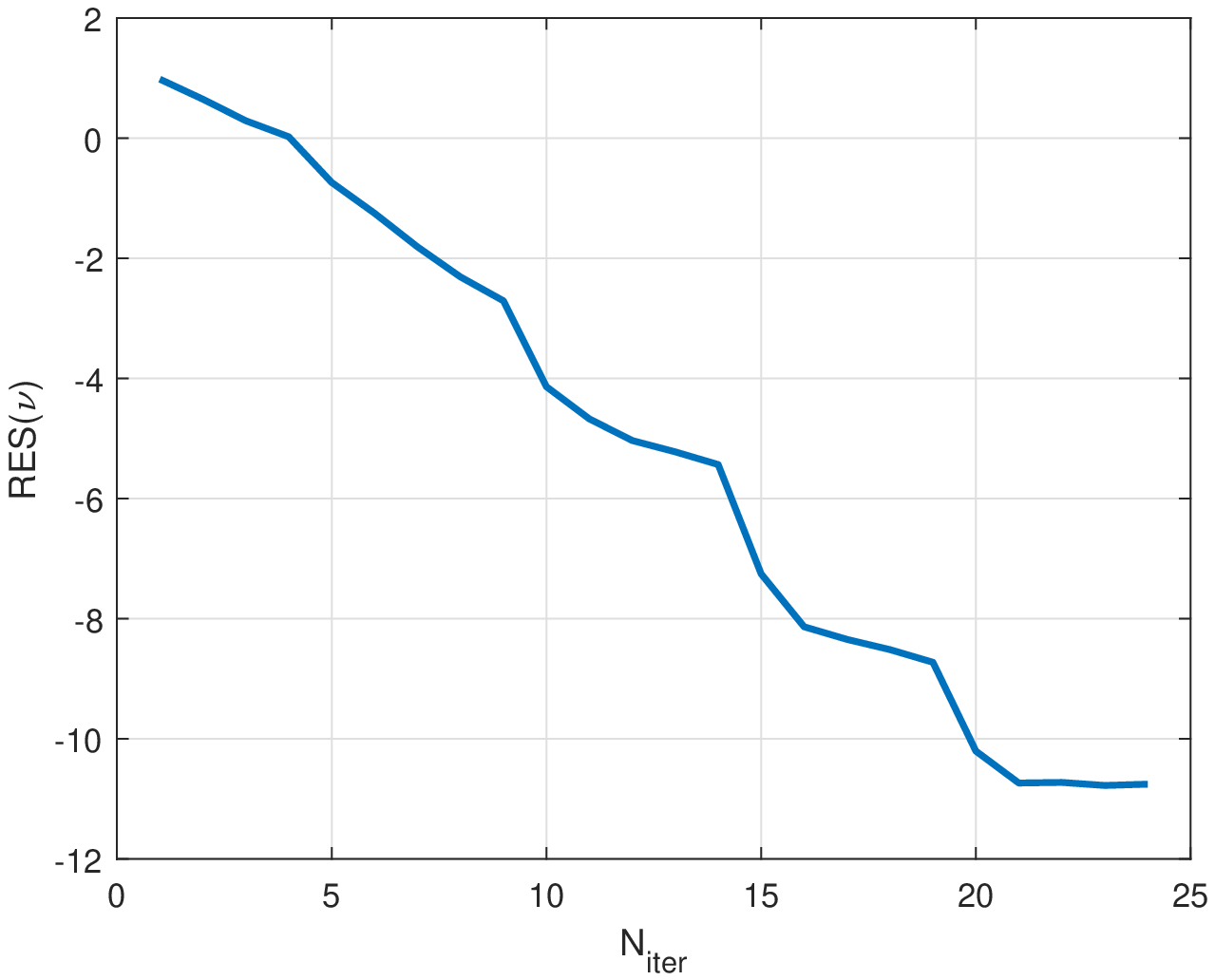}}
\caption{CSW generation. Case (A5) with $\delta=0.9,\gamma=0.5, a=-2/3, c=0, b=1/3, d=-(\delta+\gamma)a-b-c+\frac{1+\gamma\delta}{3\delta(\gamma+\delta)}\approx 0.9836; c_{s}=c_{\gamma,\delta}+0.5\approx 1.0976.$ (a) $\zeta$ and $u$ profiles; (c) $\zeta$ and $u$ phase portraits; (c) Residual error vs. number of iterations (semilog scale).}
\label{fig_BB5b}
\end{figure}

The case (A5) is illustrated in Figure \ref{fig_BB5b}. This example can also be justified by Theorem \ref{pot}.

\begin{figure}[htbp]
\centering
\subfigure[]
{\includegraphics[width=0.8\textwidth]{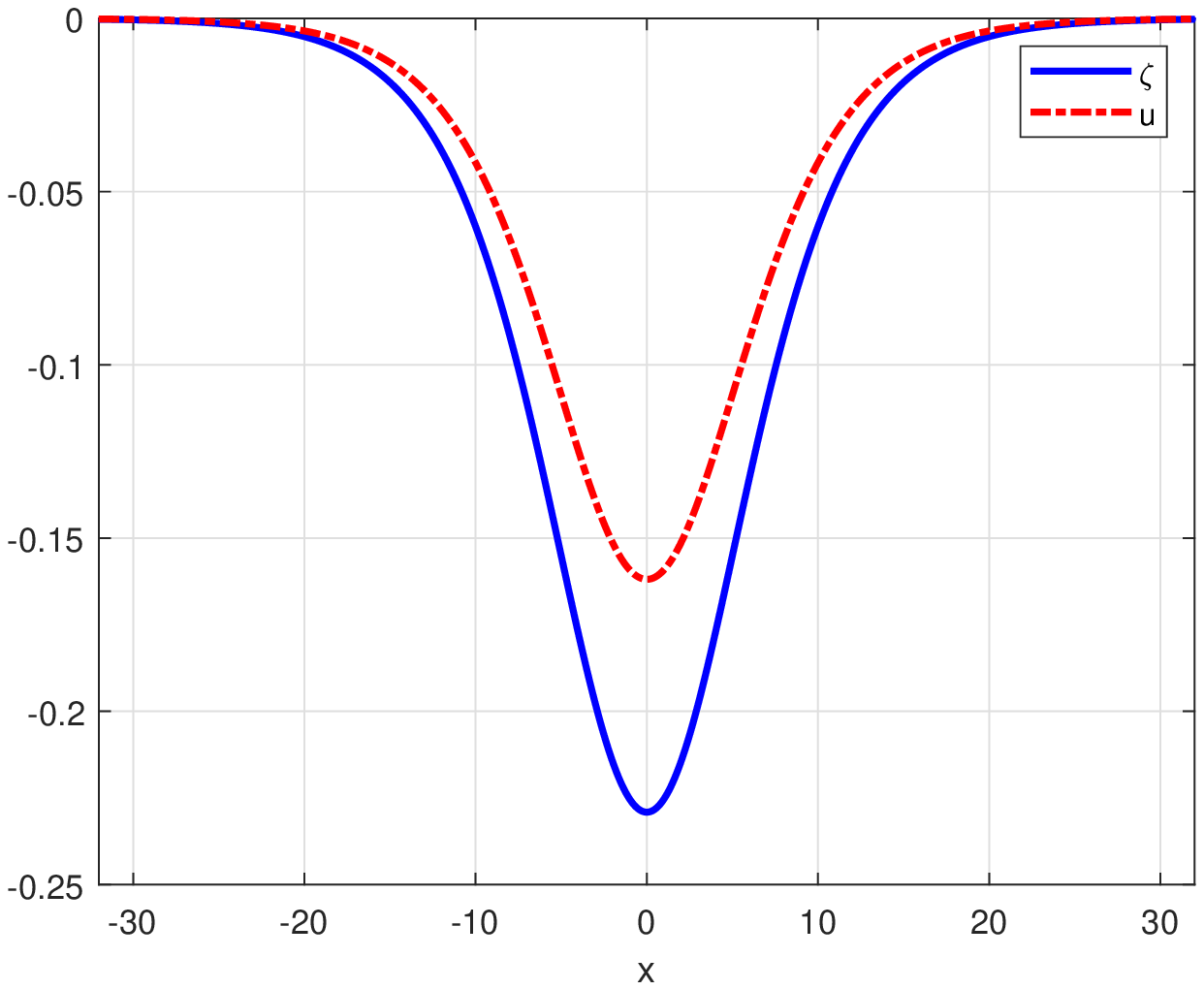}}
\subfigure[]
{\includegraphics[width=6.27cm]{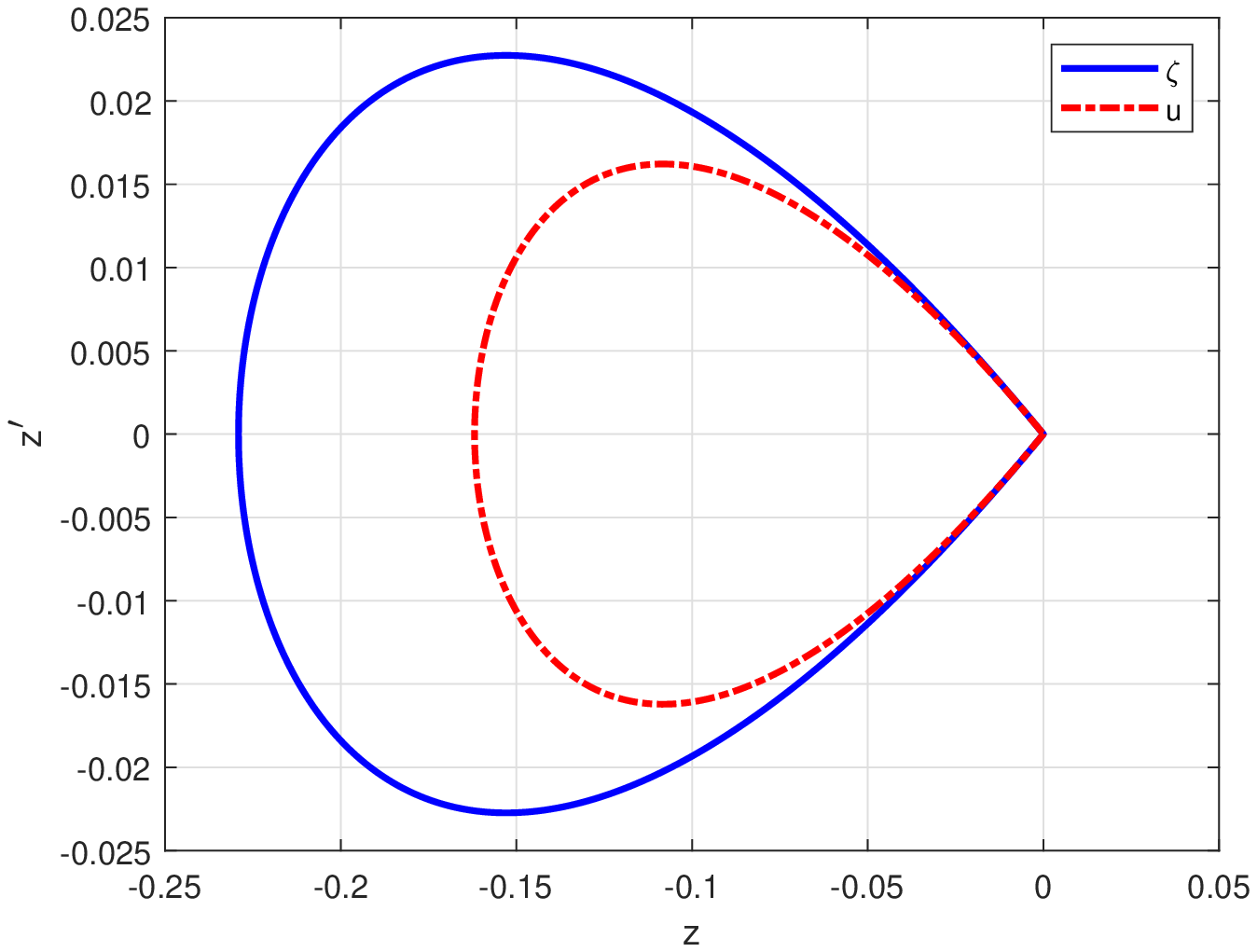}}
\subfigure[]
{\includegraphics[width=6.27cm]{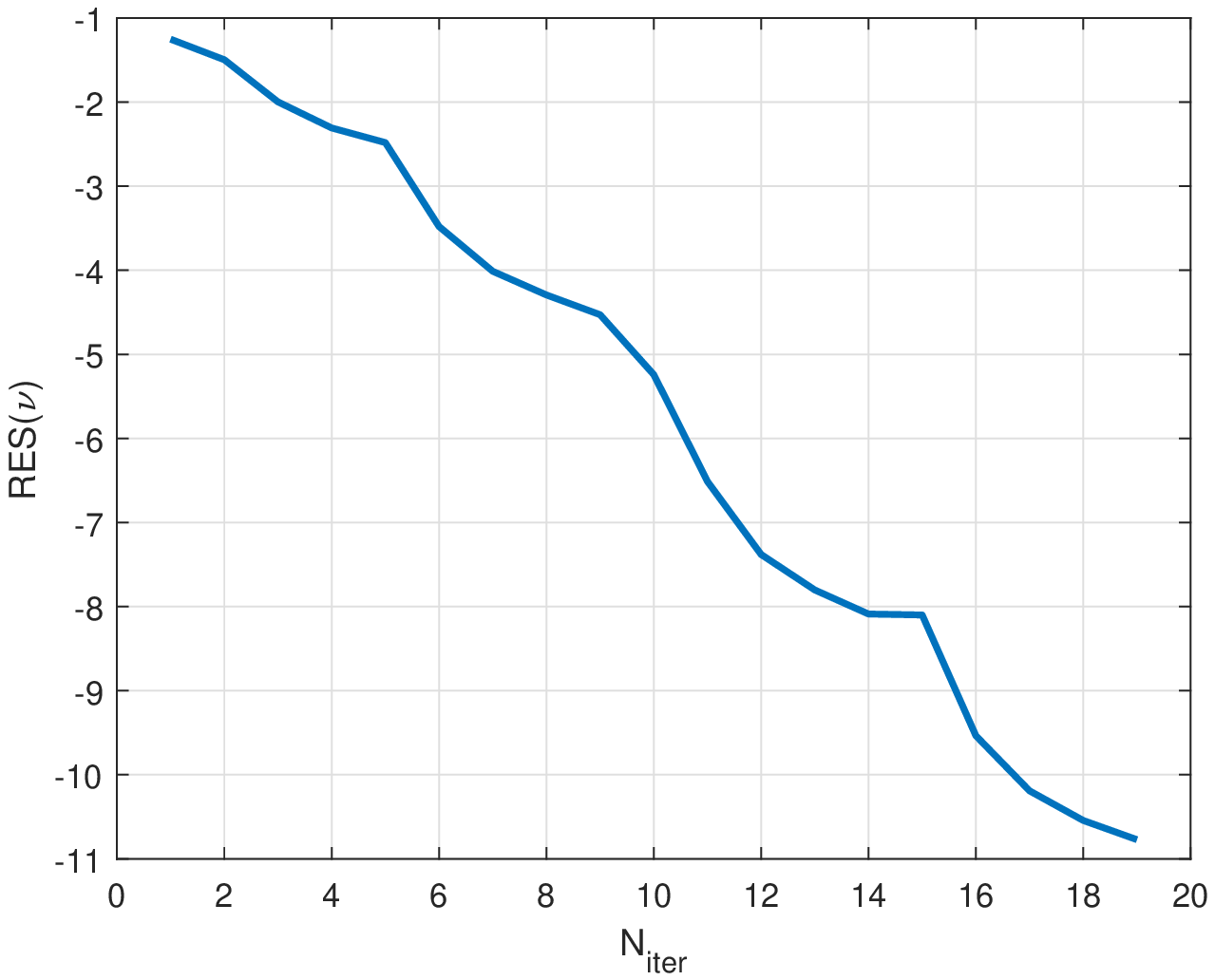}}
\caption{CSW generation. Case (A6) with $\delta=\gamma=0.5, a=0, c=0, b=1/3, d=1/2; c_{s}=c_{\gamma,\delta}+0.01\approx 0.7171.$ (a) $\zeta$ and $u$ profiles; (c) $\zeta$ and $u$ phase portraits; (c) Residual error vs. number of iterations (semilog scale).}
\label{fig_BB2}
\end{figure}

\begin{figure}[htbp]
\centering
\subfigure[]
{\includegraphics[width=0.8\textwidth]{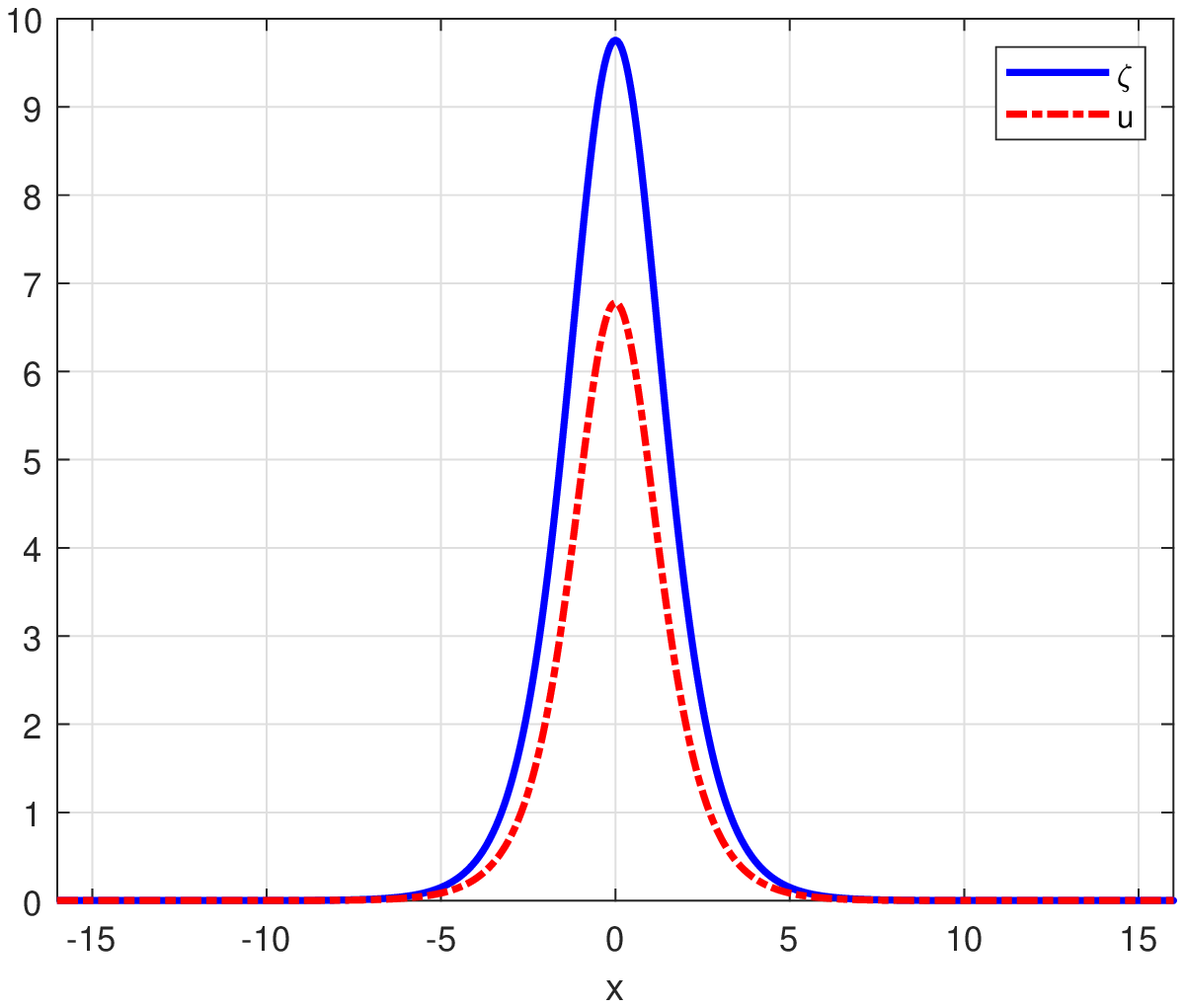}}
\subfigure[]
{\includegraphics[width=6.27cm]{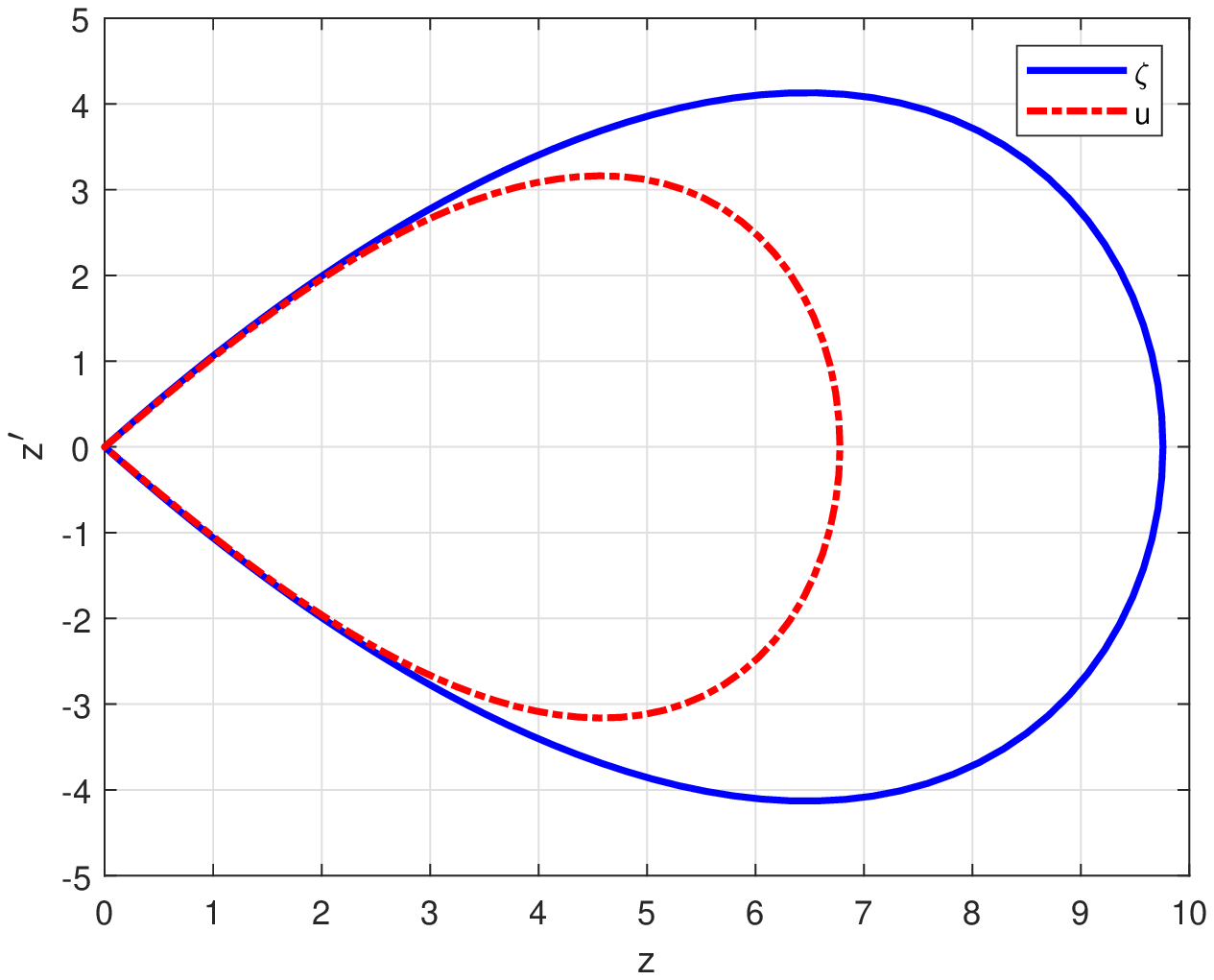}}
\subfigure[]
{\includegraphics[width=6.27cm]{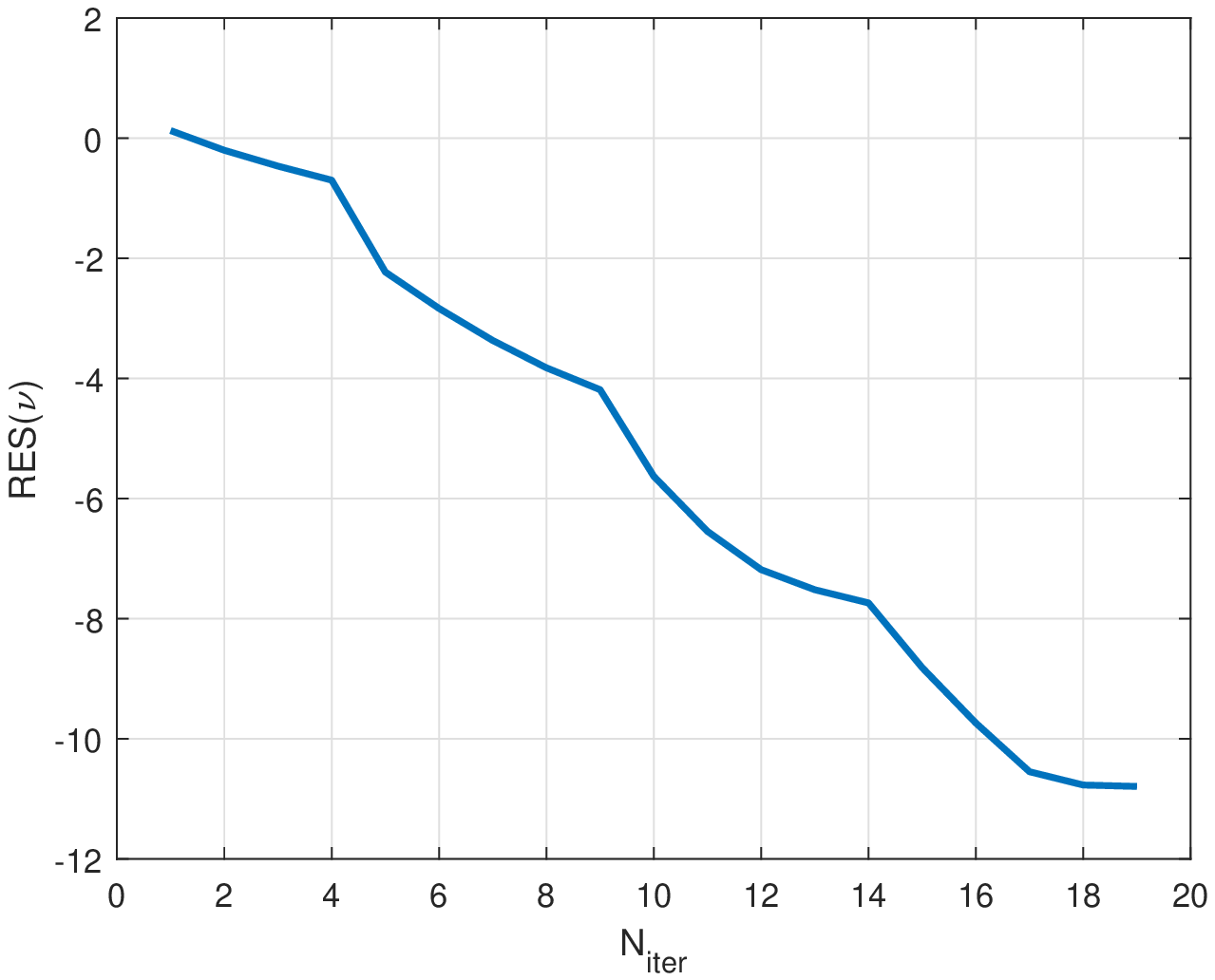}}
\caption{CSW generation. Case (A6) with $\delta=0.9, \gamma=0.5, a=0, c=0, b= d=\frac{1}{2}\frac{1+\gamma\delta}{3\delta(\gamma+\delta)}\approx 0.1918; c_{s}=c_{\gamma,\delta}+0.5\approx 1.0976$. (a) $\zeta$ and $u$ profiles; (c) $\zeta$ and $u$ phase portraits; (c) Residual error vs. number of iterations (semilog scale).}
\label{fig_BB6b}
\end{figure}

Case (A6) is illustrated by two examples, displayed in Figures \ref{fig_BB2} and \ref{fig_BB6b}. In the case of Figure \ref{fig_BB2}, the speed $c_{s}$ is quite close to $c_{\gamma,\delta}$ and the example is explained by NFT and Positive Operator Theory. On the other hand, Toland's theory may justify the example of Figure \ref{fig_BB6b}, corresponding to a Hamiltonian case of (A6). The numerical experiment is perfomed using a variant of the Petviashvili method (\ref{424}) as follows. For a fixed speed $c_{s}$, each iterate $(v_{h}^{[\nu]},\zeta_{h}^{[\nu]})$ of the method is forced to be in the manifold $\{f=0\}$, where $f$ is given by (\ref{efe}). This is accomplished by complementing the Petviashvili iteration with a projection method, implemented in the standard way, \cite{HairerLW2004}. For the particular example of Figure \ref{fig_BB6b}, with speed $c_{s}=c_{\gamma,\delta}+0.5$, the points $P_{1}$ and $P_{2}$ of the segment $\Gamma$, cf. Theorem \ref{toland}, have the components (approximately)
$$P_{1}=(4.8825,10.7182),\; P_{2}=(6.3226,7.5569),$$ and the computed solitary wave profile has amplitude $\zeta_{max}\approx 9.7566$ with $v_{max}\approx 5.9357$ and $u_{max}\approx 6.7788$.

Recall that the theories examined in this paper do not consider the case $D=0$, where $D$ is given by (\ref{NFTD}). However, some numerical experiments in that case suggest existence of classical solitary waves. (Existence of CSW for the particular case of the  \lq classical Boussinesq\rq\ system, with $b=0, d>0, a=c=0$, is a special case and is proved e.~g. in
\cite{Duran2019}.) By way of illustration Figures \ref{fig_BB32} and \ref{fig_BB33} show two CSW's generated for different combinations of the parameters $a, b, c$ and $d$ leading to $D=0$, and different speeds.

\begin{figure}[htbp]
\centering
\subfigure[]
{\includegraphics[width=0.8\textwidth]{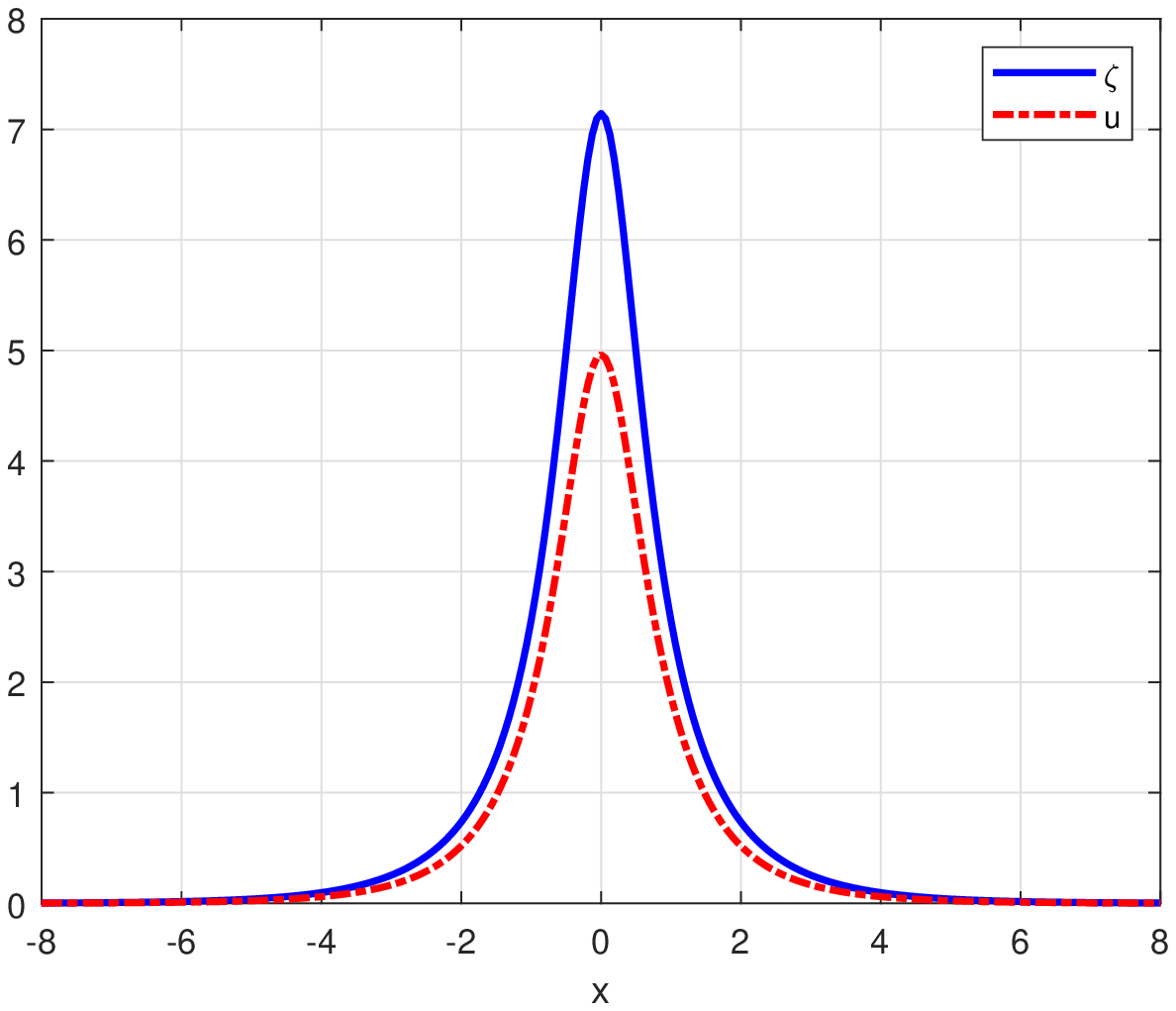}}
\subfigure[]
{\includegraphics[width=6.27cm]{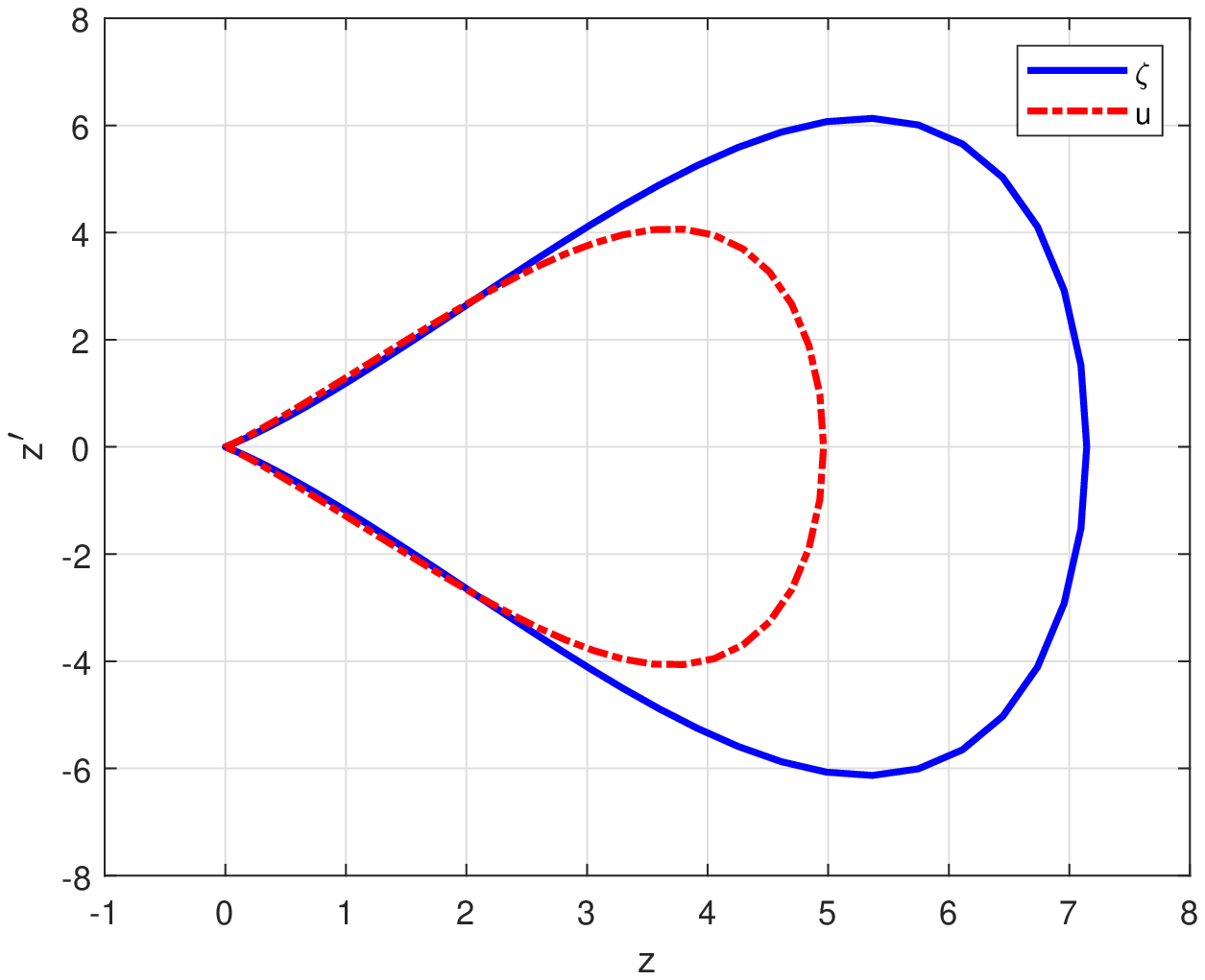}}
\subfigure[]
{\includegraphics[width=6.27cm]{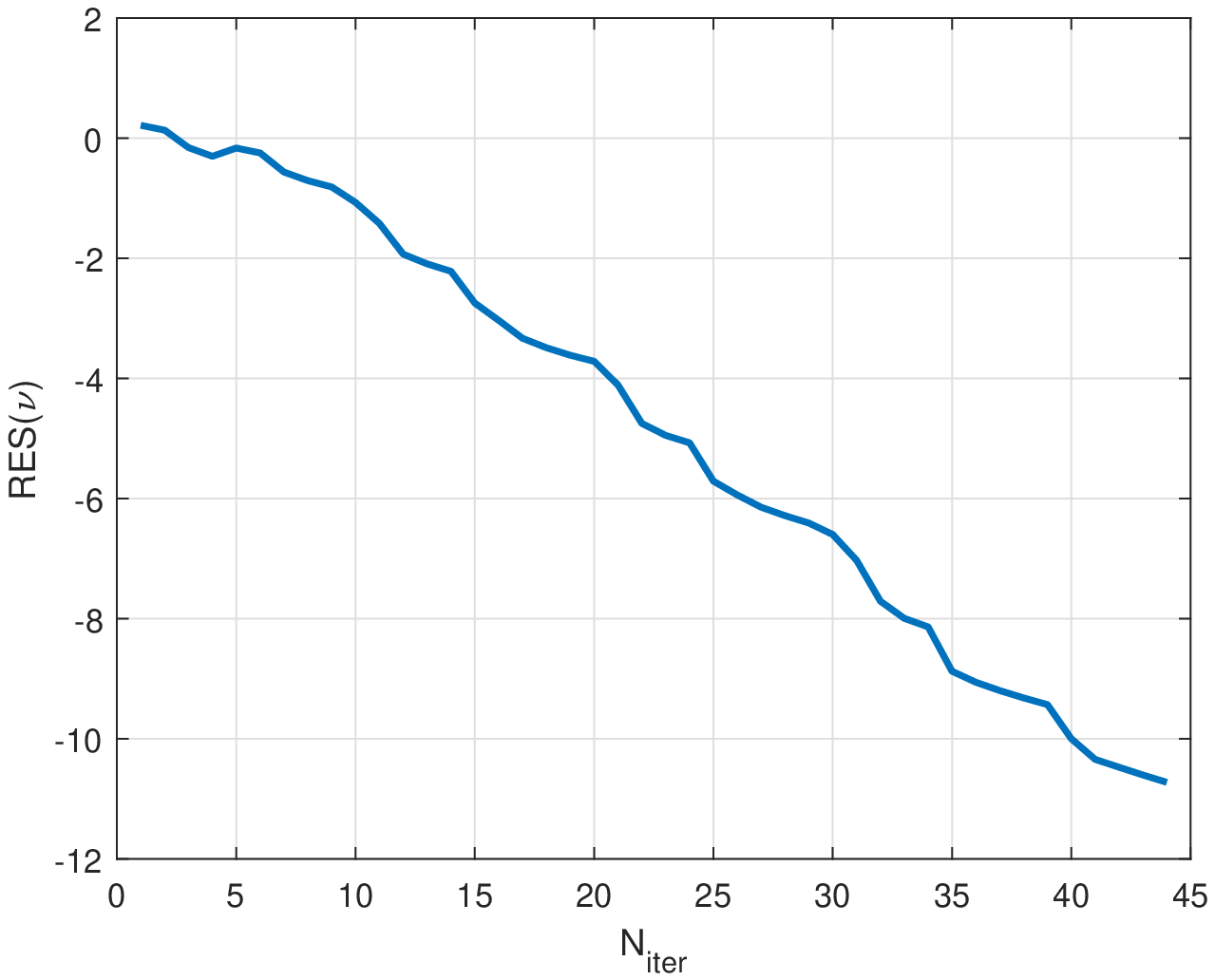}}
\caption{CSW generation. Case $D=0$ with $\delta=0.9, \gamma=0.5, a=-1/3, c=0, b=0, d=-(\delta+\gamma)a-b-c+\frac{1+\gamma\delta}{3\delta(\gamma+\delta)}\approx 0.8503; c_{s}=c_{\gamma,\delta}+0.25\approx 0.8476.$ (a) $\zeta$ and $u$ profiles; (c) $\zeta$ and $u$ phase portraits; (c) Residual error vs. number of iterations (semilog scale).}
\label{fig_BB32}
\end{figure}
\begin{figure}[htbp]
\centering
\subfigure[]
{\includegraphics[width=0.8\textwidth]{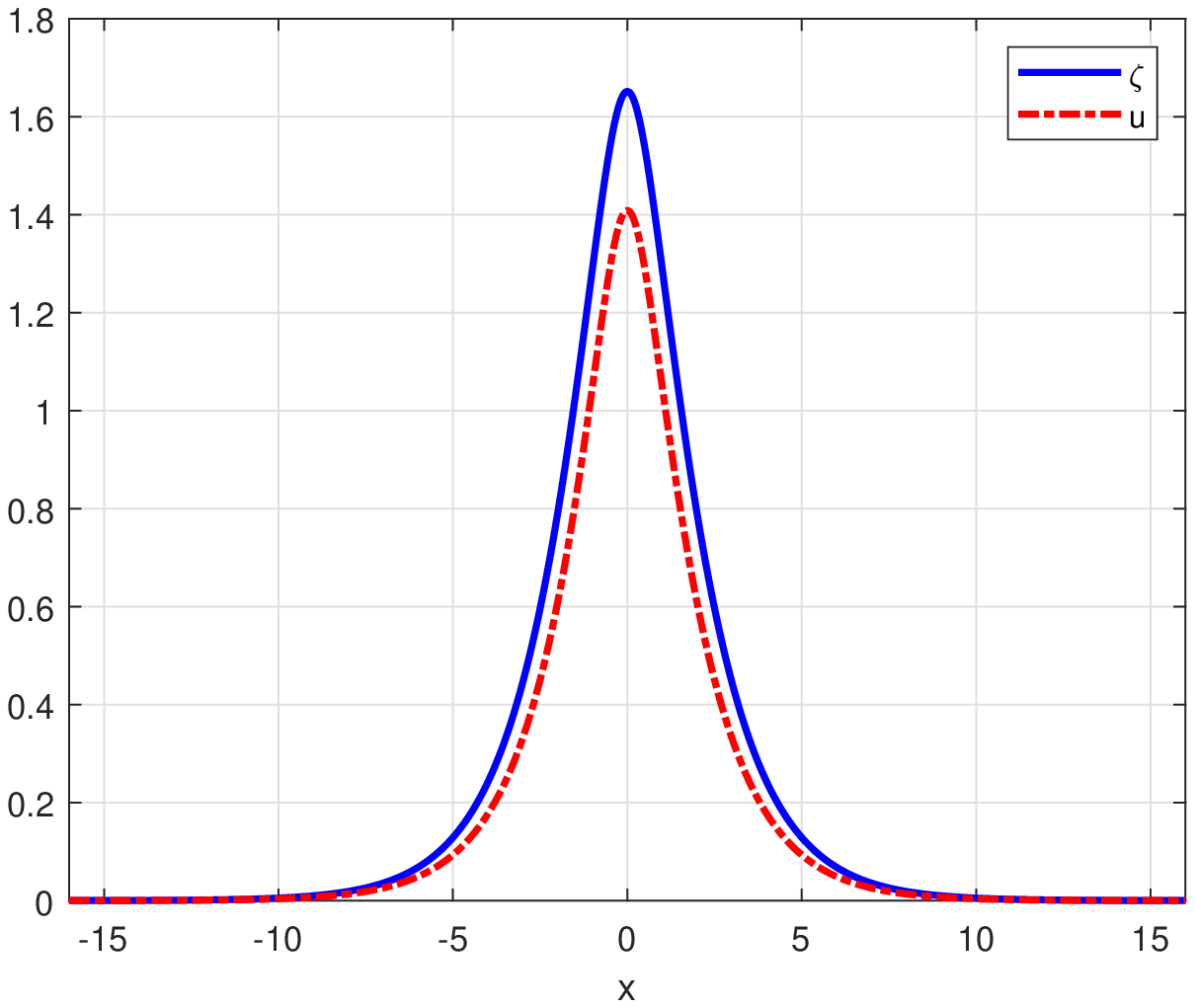}}
\subfigure[]
{\includegraphics[width=6.27cm]{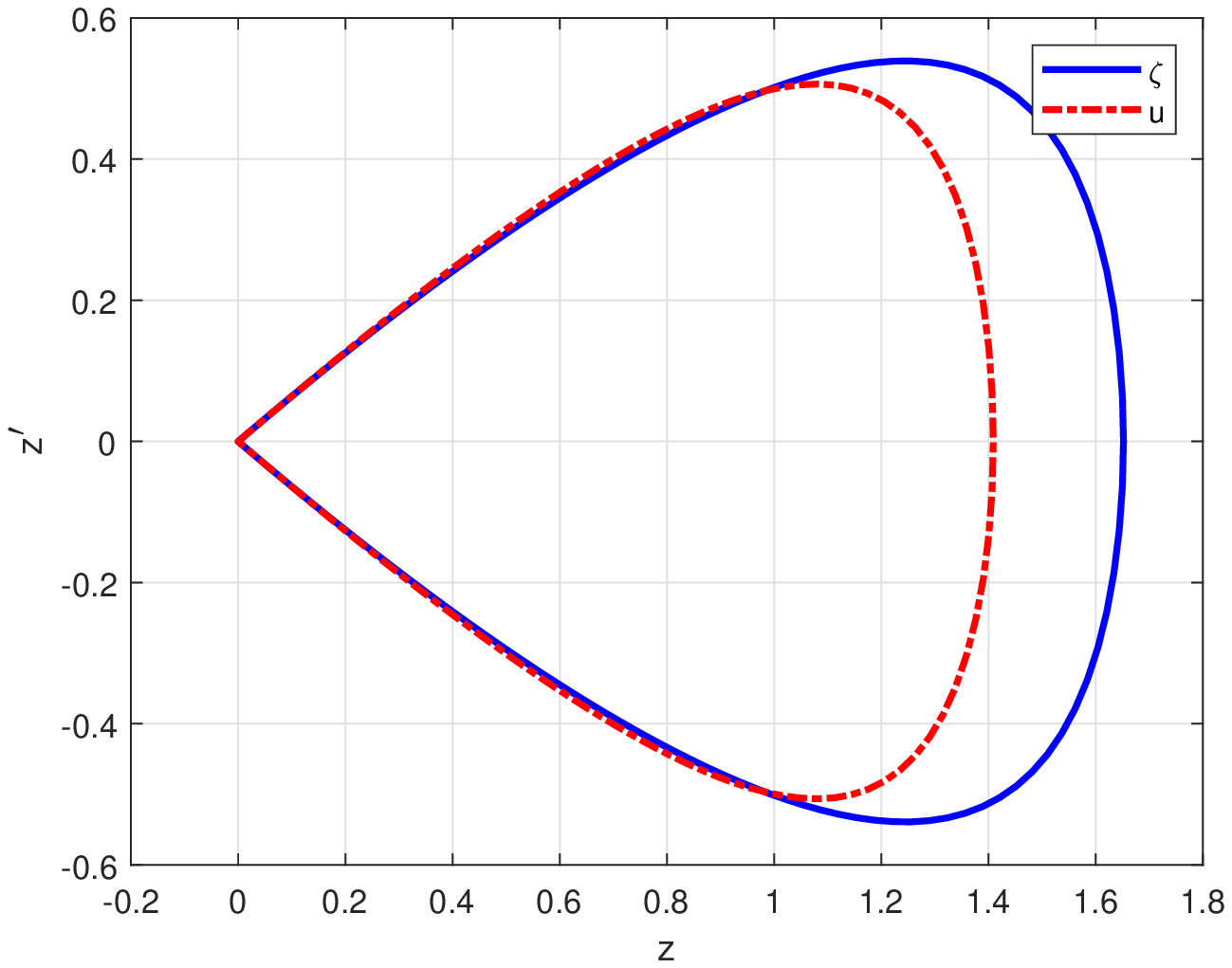}}
\subfigure[]
{\includegraphics[width=6.27cm]{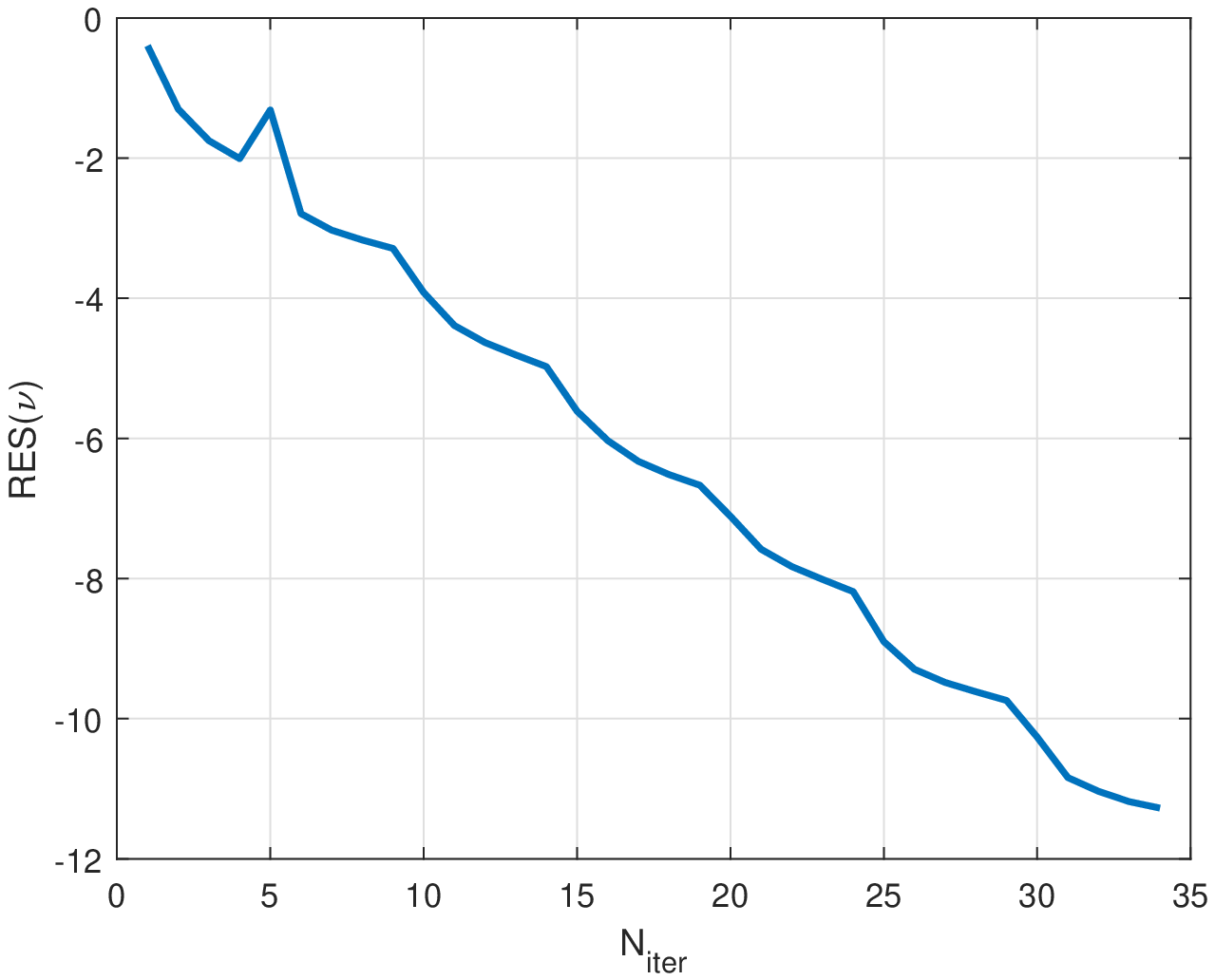}}
\caption{CSW generation. Case $D=0$ with $\delta=0.9, \gamma=0.5, a=-2/3, c=0, d=0, b=-(\delta+\gamma)a-d-c+\frac{1+\gamma\delta}{3\delta(\gamma+\delta)}\approx 1.3169; c_{s}=c_{\gamma,\delta}+0.1\approx 0.6976$ (a) $\zeta$ and $u$ profiles; (c) $\zeta$ and $u$ phase portraits; (c) Residual error vs. number of iterations (semilog scale).}
\label{fig_BB33}
\end{figure}
We complete this numerical study by illustrating the generation of CSW's with non-monotonic decay and of periodic traveling waves, both predicted by NFT. CSW's with non-monotonic decay appear in Figures \ref{fig_A2} and \ref{fig_A3}.

\begin{figure}[htbp]
\centering
\subfigure[]
{\includegraphics[width=0.8\textwidth]{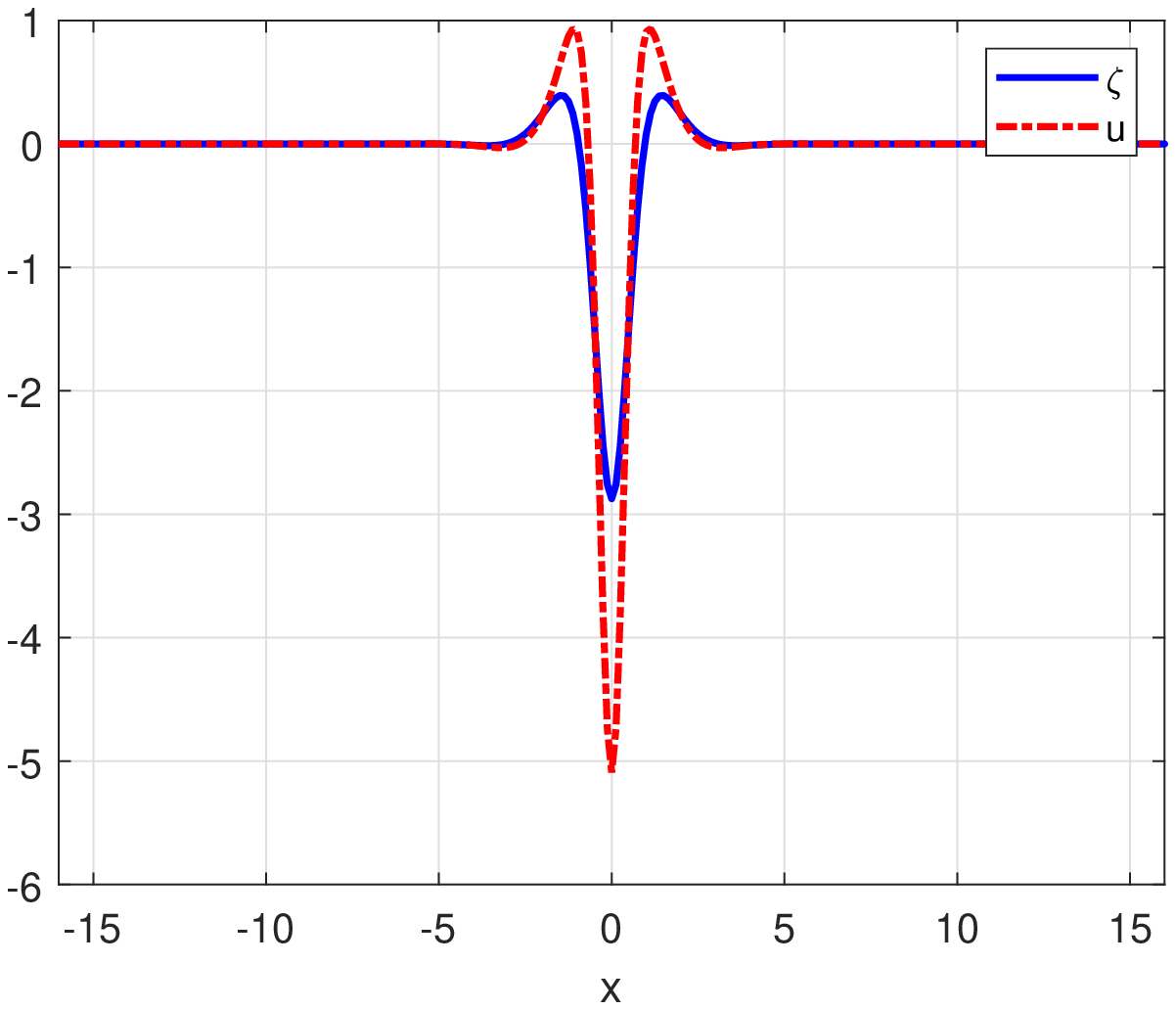}}
\subfigure[]
{\includegraphics[width=6.27cm]{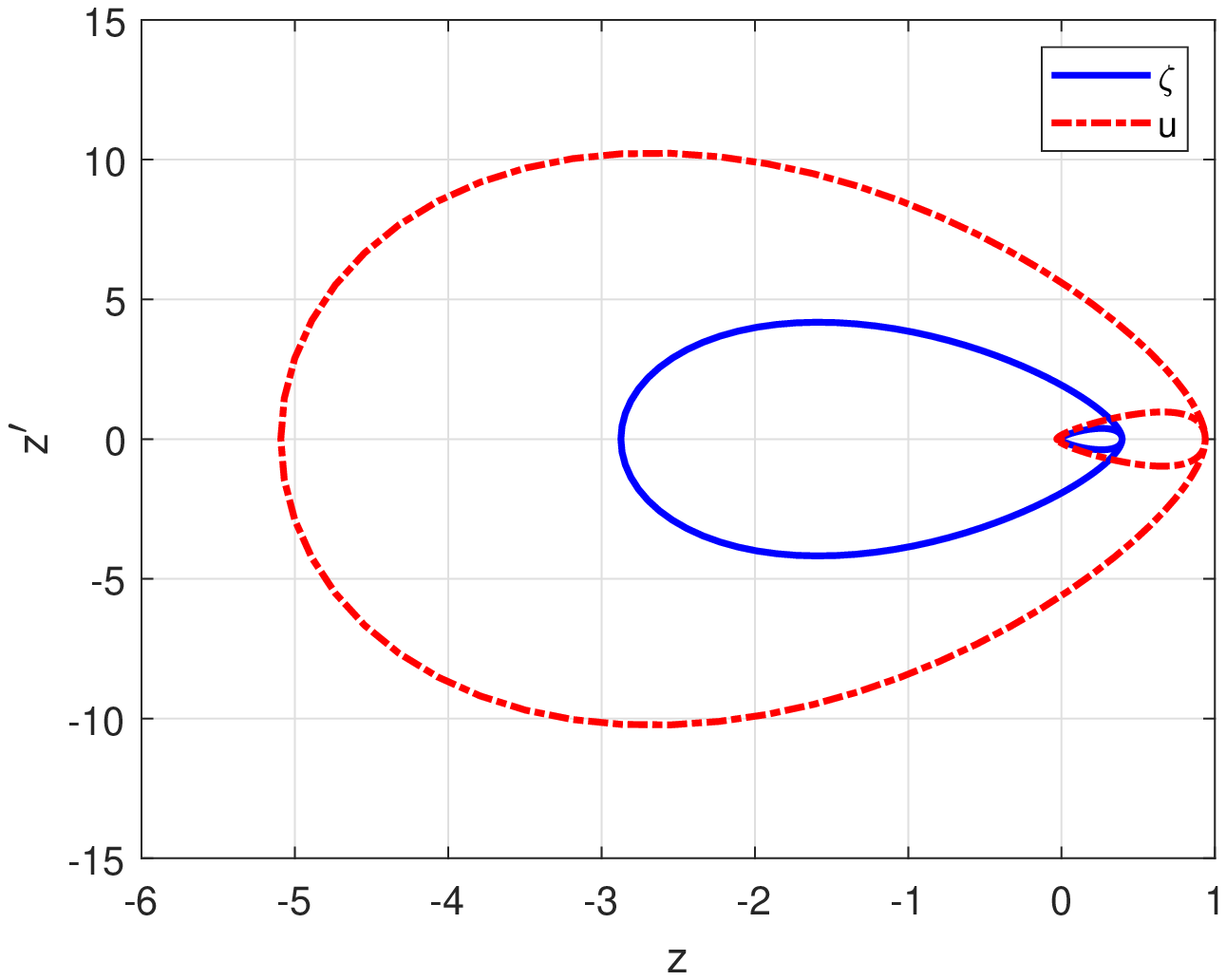}}
\subfigure[]
{\includegraphics[width=6.27cm]{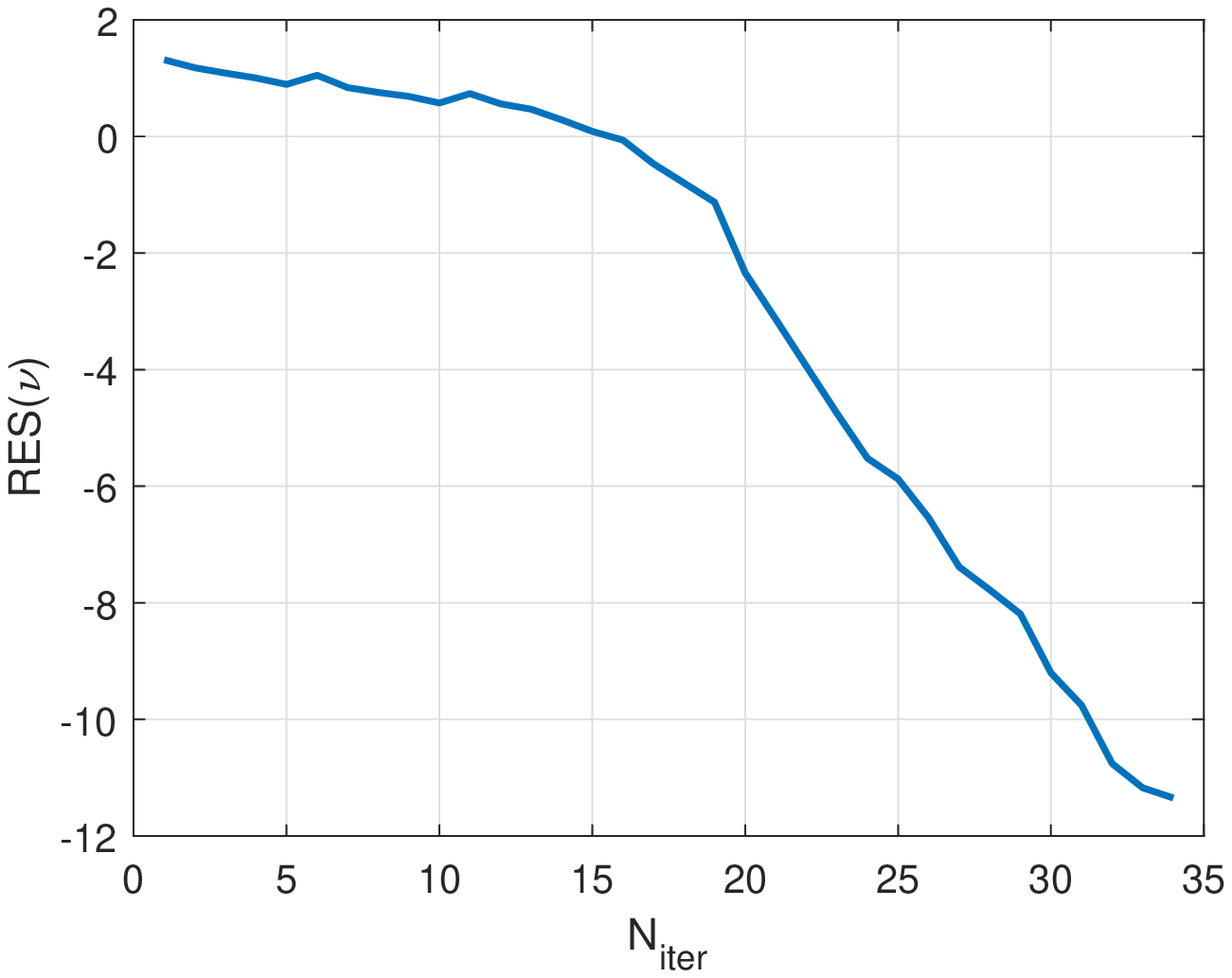}}
\caption{CSW generation with $a=-1/9,c=-1/6,b=0,d=-a/\kappa_{1}-c-b+S(\gamma,\delta)\approx 0.7058, \delta=0.9, \gamma=0.5,$ and speed $c_{s}=c_{\gamma,\delta}-0.2\approx 3.9761\times 10^{-1}$ (a) $\zeta$ and $u$ profiles; (c) $\zeta$ and $u$ phase portraits; (c) Residual error vs. number of iterations (semilog scale).}
\label{fig_A2}
\end{figure}

\begin{figure}[htbp]
\centering
\subfigure[]
{\includegraphics[width=0.8\textwidth]{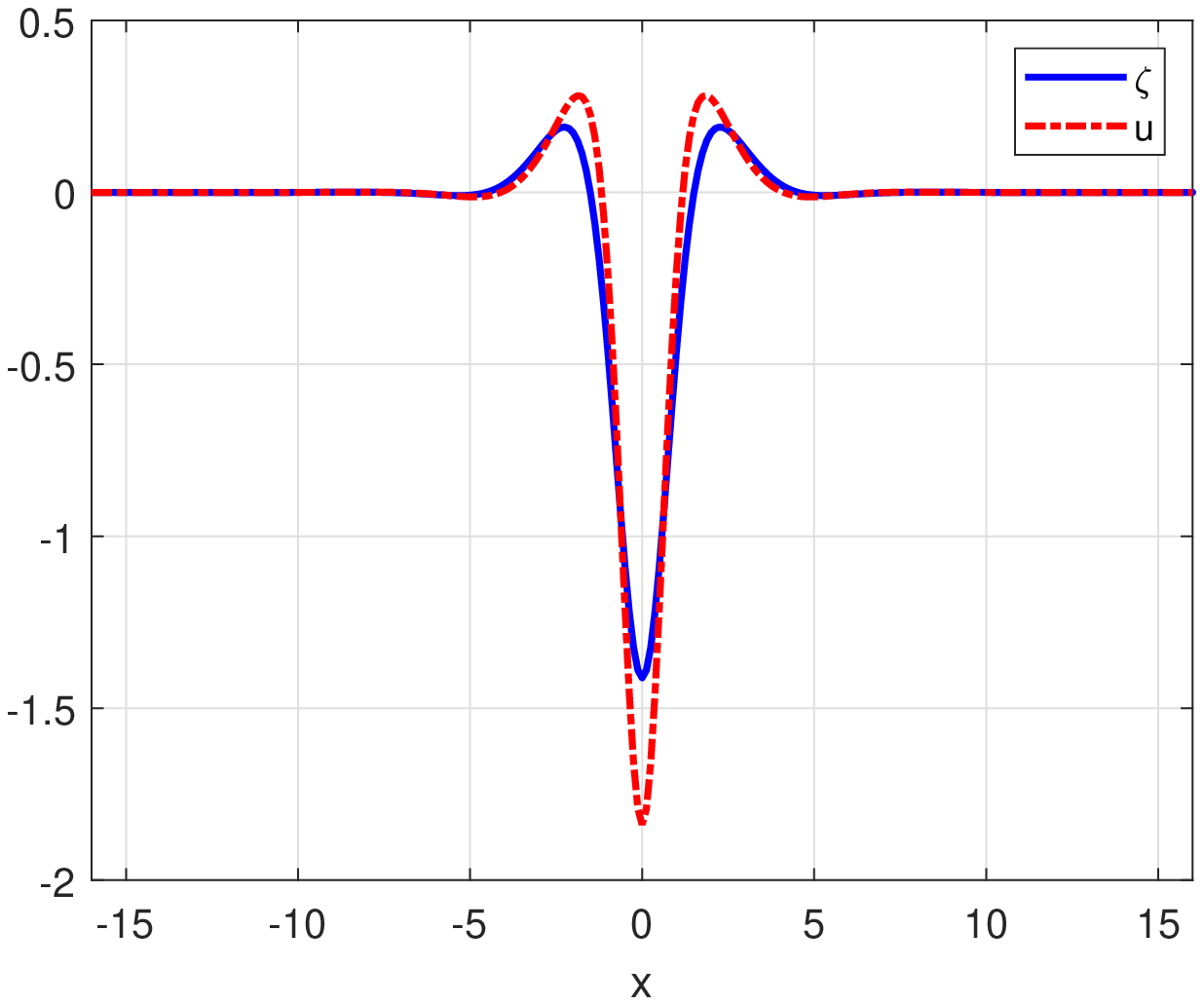}}
\subfigure[]
{\includegraphics[width=6.27cm]{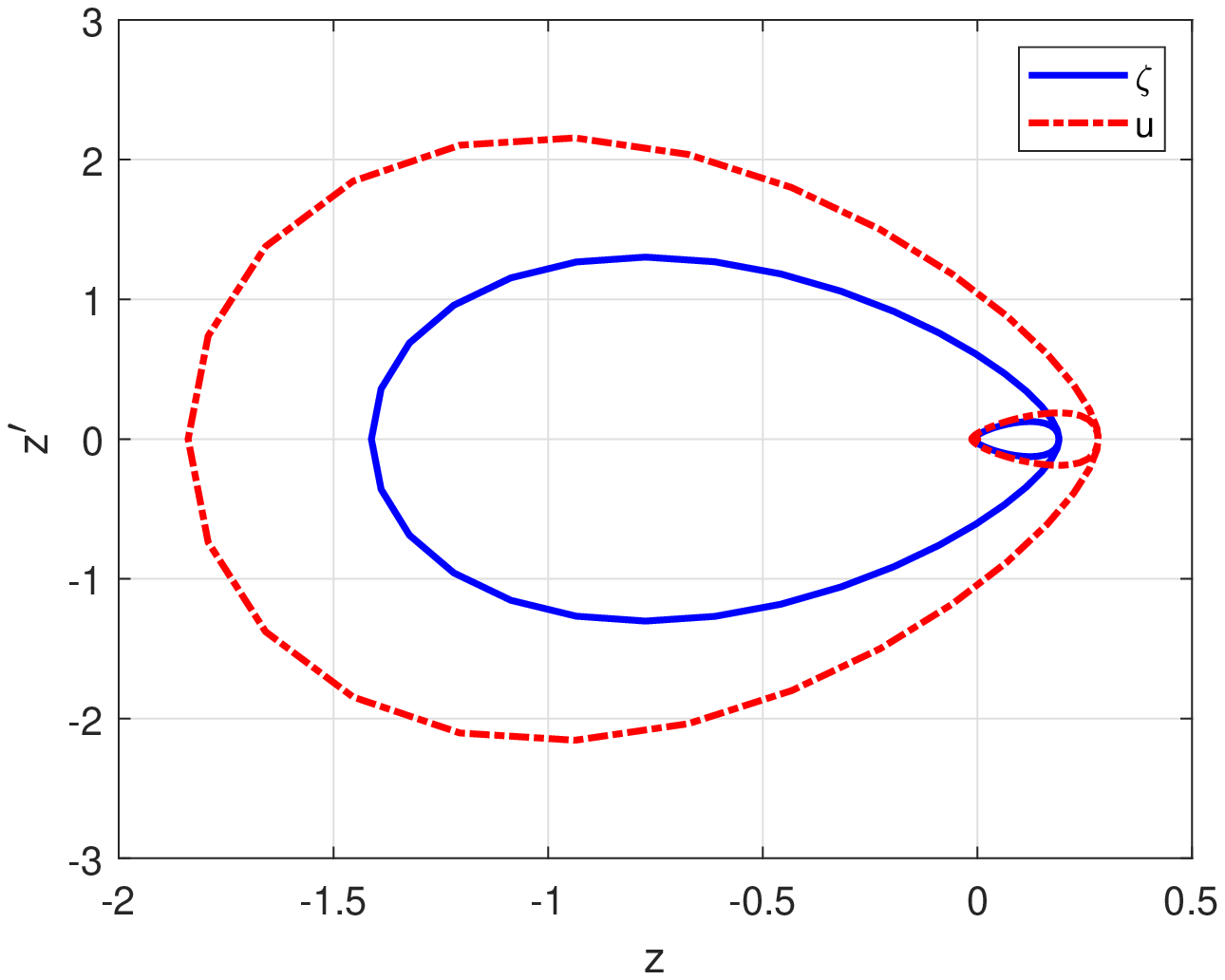}}
\subfigure[]
{\includegraphics[width=6.27cm]{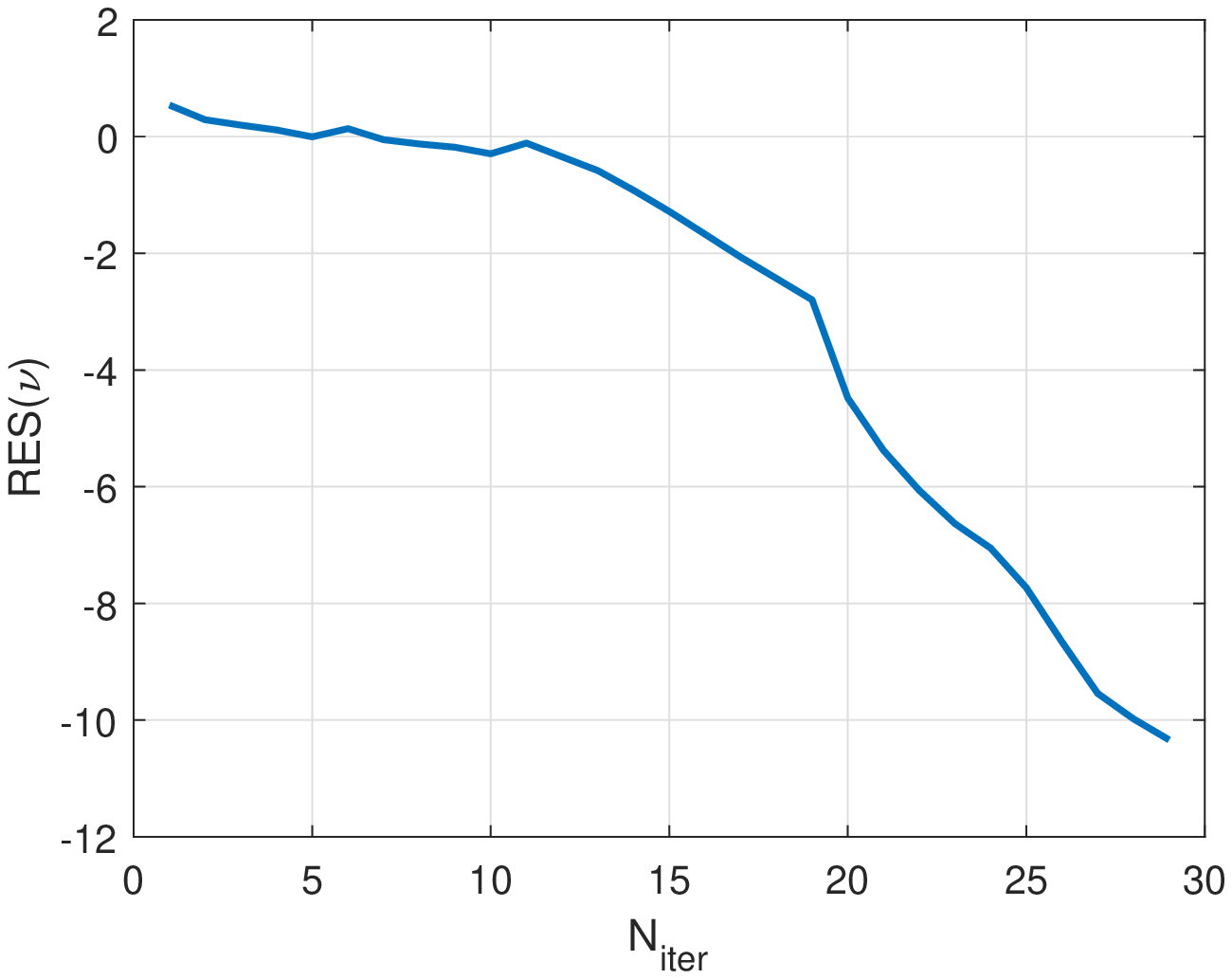}}
\caption{CSW generation with $a=-1/3,c=-1/3,b=1/9,d=-a/\kappa_{1}-c-b+S(\gamma,\delta)\approx 1.0725, \delta=0.9, \gamma=0.5,$ and speed $c_{s}=c_{\gamma,\delta}-0.2\approx 3.9761\times 10^{-1}$ (a) $\zeta$ and $u$ profiles; (c) $\zeta$ and $u$ phase portraits; (c) Residual error vs. number of iterations (semilog scale).}
\label{fig_A3}
\end{figure}

Figure \ref{fig_A2} corresponds to a homoclinic orbit belonging to the first-quadrant part of region 1 in Figure \ref{fig_A1} and corresponding to the values
$$a=-1/9,c=-1/6,b=0,d=-a/\kappa_{1}-c-b+S(\gamma,\delta)\approx 0.7058, \delta=0.9, \gamma=0.5,$$ and speed $c_{s}=c_{\gamma,\delta}-0.2\approx 3.9761\times 10^{-1}$, while the one in Figure \ref{fig_A3} belongs to the second quadrant of region 1 and is generated by the parameters
 $$a=-1/3,c=-1/3,b=1/9,d=-a/\kappa_{1}-c-b+S(\gamma,\delta)\approx 1.0725, \delta=0.9, \gamma=0.5,$$ and speed $c_{s}=c_{\gamma,\delta}-0.1\approx 4.9761\times 10^{-1}.$ (Recall that the generation of these CSW's requires $c_{s}<c_{\gamma,\delta}$, cf. section \ref{sec41}.)

\begin{figure}[htbp]
\centering
\subfigure[]
{\includegraphics[width=0.8\textwidth]{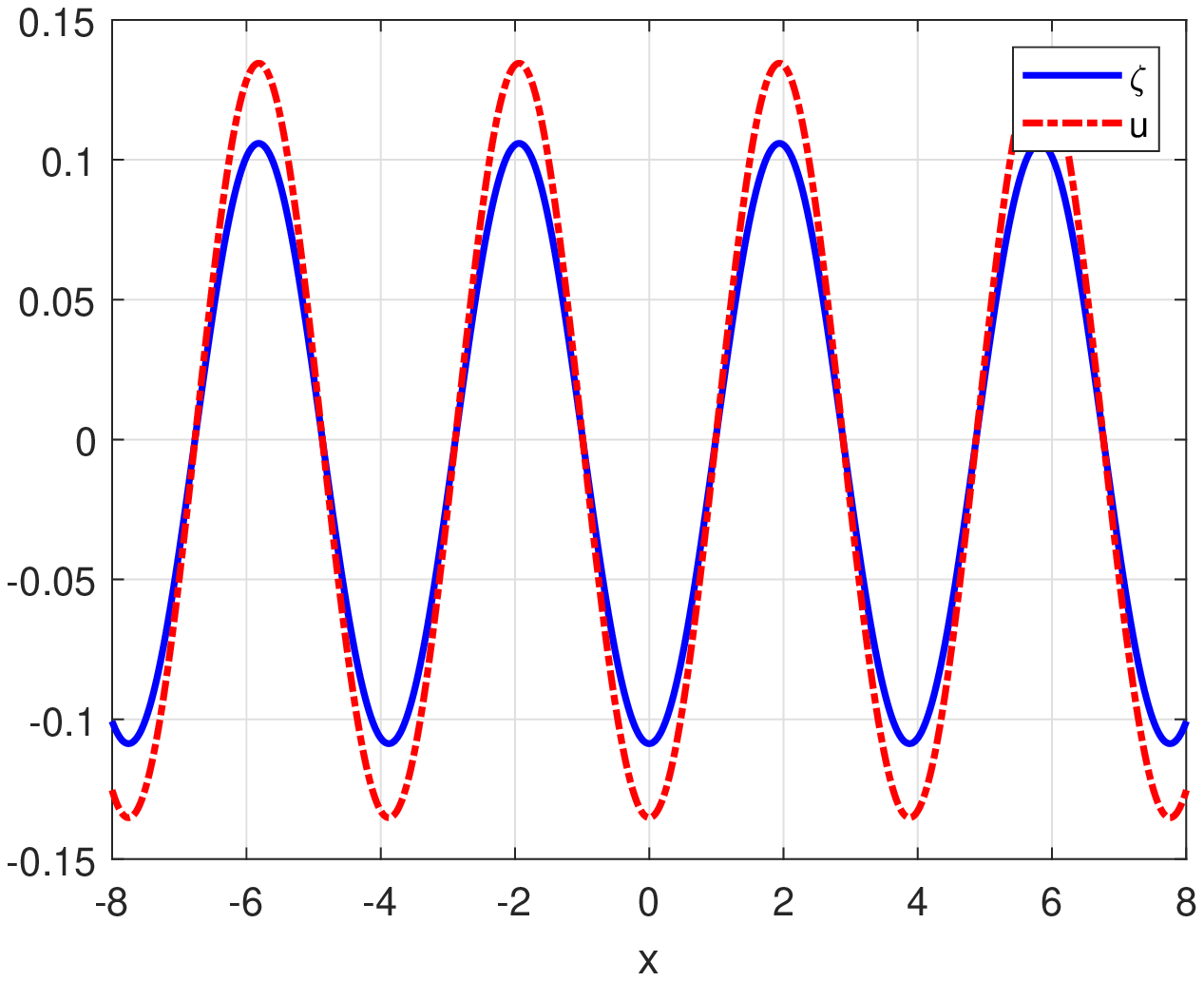}}
\subfigure[]
{\includegraphics[width=6.27cm]{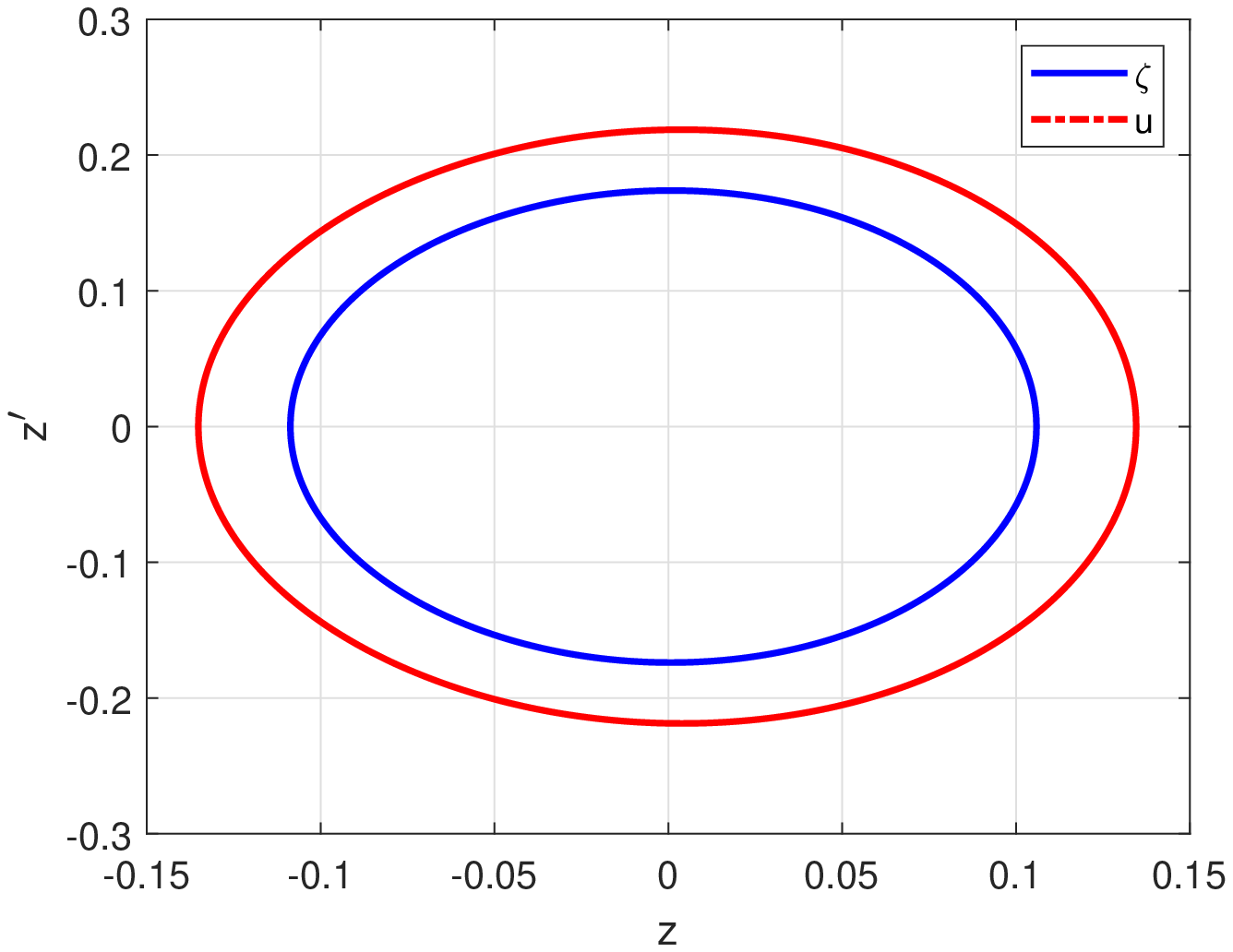}}
\subfigure[]
{\includegraphics[width=6.27cm]{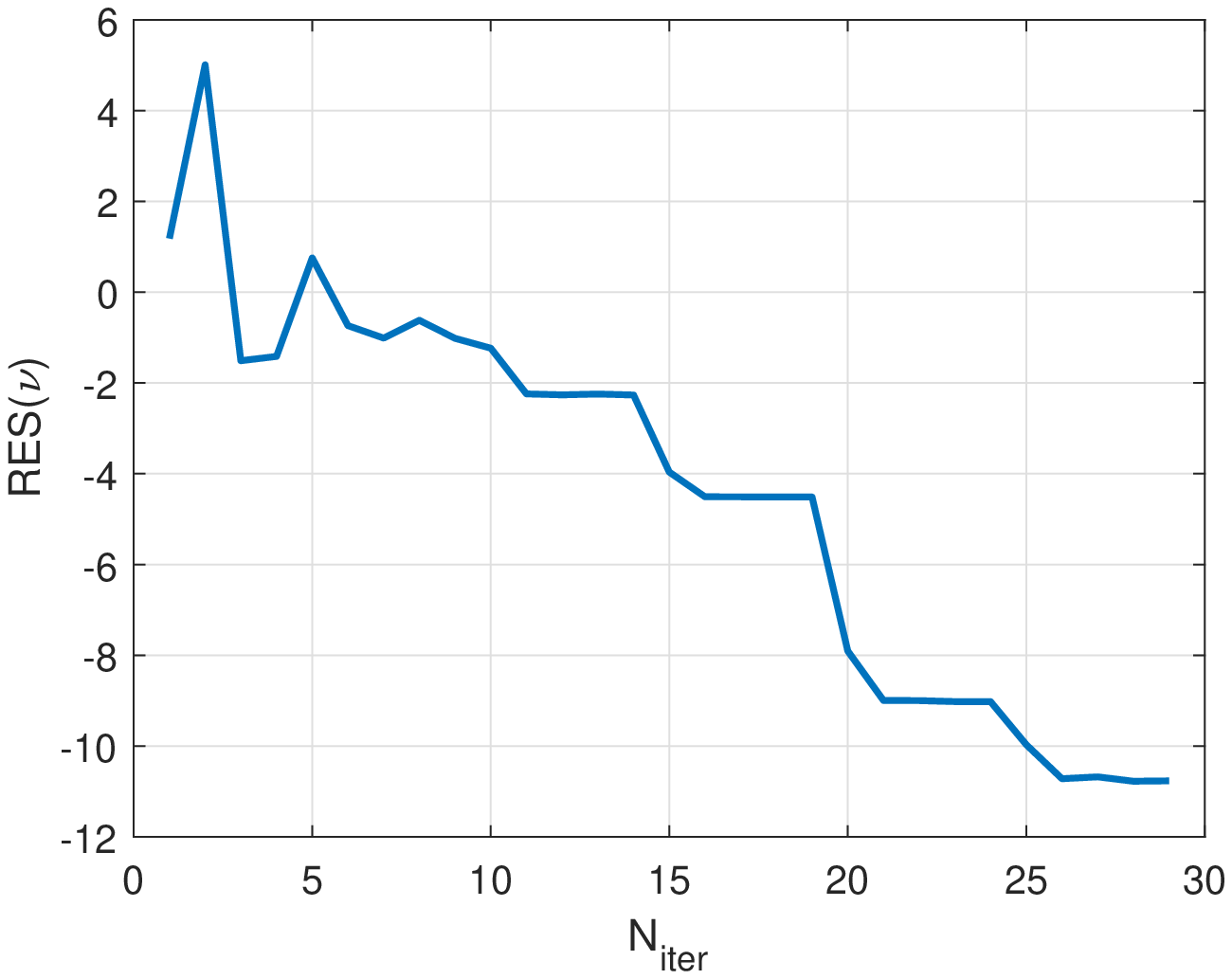}}
\caption{PTW generation with $a=c=0,b=1/6,d=-a/\kappa_{1}-c-b+S(\gamma,\delta)\approx 0.2169, \delta=0.9, \gamma=0.5,$ and speed $c_{s}=c_{\gamma,\delta}-0.2\approx 3.9761\times 10^{-1}$ (a) $\zeta$ and $u$ profiles; (c) $\zeta$ and $u$ phase portraits; (c) Residual error vs. number of iterations (semilog scale).}
\label{fig_A4}
\end{figure}

The numerical generation of periodic traveling wave solutions is illustrated in Figure \ref{fig_A4}, wherein the values
$$a=c=0,b=1/6,d=-a/\kappa_{1}-c-b+S(\gamma,\delta)\approx 0.7058, \delta=0.9, \gamma=0.5,$$ and speed $c_{s}=c_{\gamma,\delta}-0.1\approx 3.9761\times 10^{-1}$, are taken. In the spectral analysis of the linearization the corresponding point in the $(B,A)$-plane is in region 3 of Figure \ref{fig_A1}.

%discussing the generation of other solutions of (\ref{BB6}). As mentioned in section \ref{sec41}, Normal Form Theory predicts CSW solutions in region 1 (Figure \ref{fig_A1}) and periodic solutions in region 3. The first case is illustrated by Figures \ref{fig_A2} and \ref{fig_A3}. The first one corresponds to the first quadrant of region 1 and with the values
%$$a=-1/9,c=-1/6,b=0,d=-a/\kappa_{1}-c-b+S(\gamma,\delta)\approx 0.7058, \delta=0.9, \gamma=0.5,$$ and speed $c_{s}=c_{\gamma,\delta}-0.2\approx 3.9761\times 10^{-1}$, while the second one corresponds to the second quadrant of region 1 and is generated with
% $$a=-1/3,c=-1/3,b=1/9,d=-a/\kappa_{1}-c-b+S(\gamma,\delta)\approx 1.0725, \delta=0.9, \gamma=0.5,$$ and speed $c_{s}=c_{\gamma,\delta}-0.1\approx 4.9761\times 10^{-1}$.

The last experiment in this section concerns the speed-amplitude relation. Figure \ref{speedamp} displays, in linear and log-log scales and for fixed $\gamma, \delta$, the amplitude of the computed profiles for $\zeta, u$ and $v_{\beta}$ as functions of the difference $c_{s}-c_{\gamma,\delta}$. The results correspond to an experiment with $\delta=0.9, \gamma=0.5$ and parameters $a=c=0, b=d=\frac{1}{2}\frac{1+\gamma\delta}{3\delta (\gamma+\delta)}$ (case (A6) of Table \ref{tavle0} with Hamiltonian structure). Similar experiments were made for different parameters leading to other CSW's (including those of nonmonotone decay) as well as GSW's, and the results resemble qualitatively those of this Figure
The three maximum values (for $\zeta, u$ and $v_{\beta}$) are increasing functions of $c_{s}-c_{\gamma,\delta}$, with the amplitude of $\zeta$ increasing faster. Figure \ref{speedamp}(b), in log-log scale, includes a dotted line of slope $1$ for comparison purposes. The representation of the amplitudes as affine functions for small $c_{s}-c_{\gamma,\delta}$ seems to fit the results (as expected). For larger values of $c_{s}-c_{\gamma,\delta}$, the slope of the line for the $\zeta$-amplitude is increasing faster while the approximate linear fitting persists longer for the velocity variables.

\begin{figure}[htbp]
\centering
%\subfigure[]
%{\includegraphics[width=0.8\textwidth]{speed-amp1.eps}}
\subfigure[]
{\includegraphics[width=6.27cm]{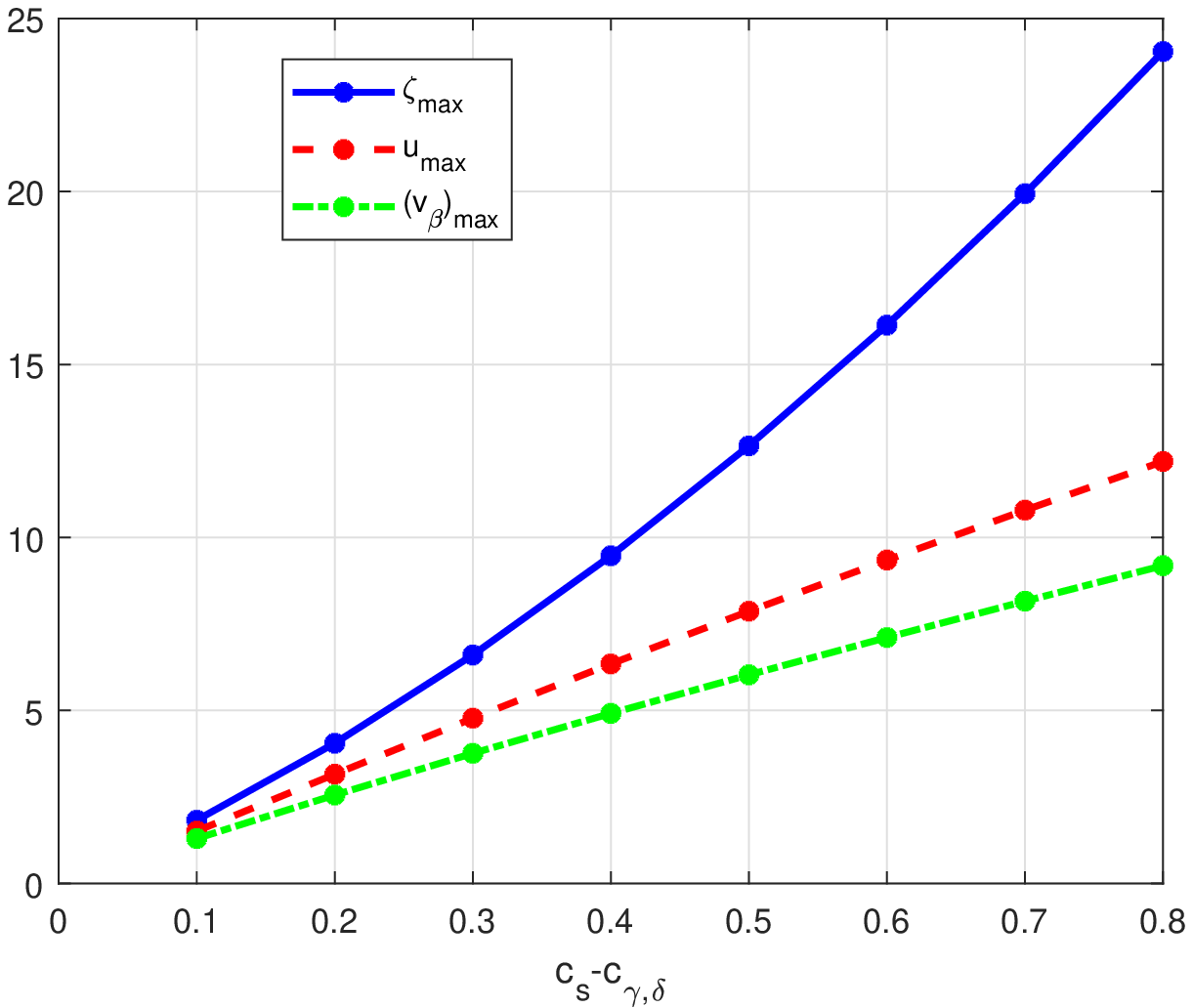}}
\subfigure[]
{\includegraphics[width=6.27cm]{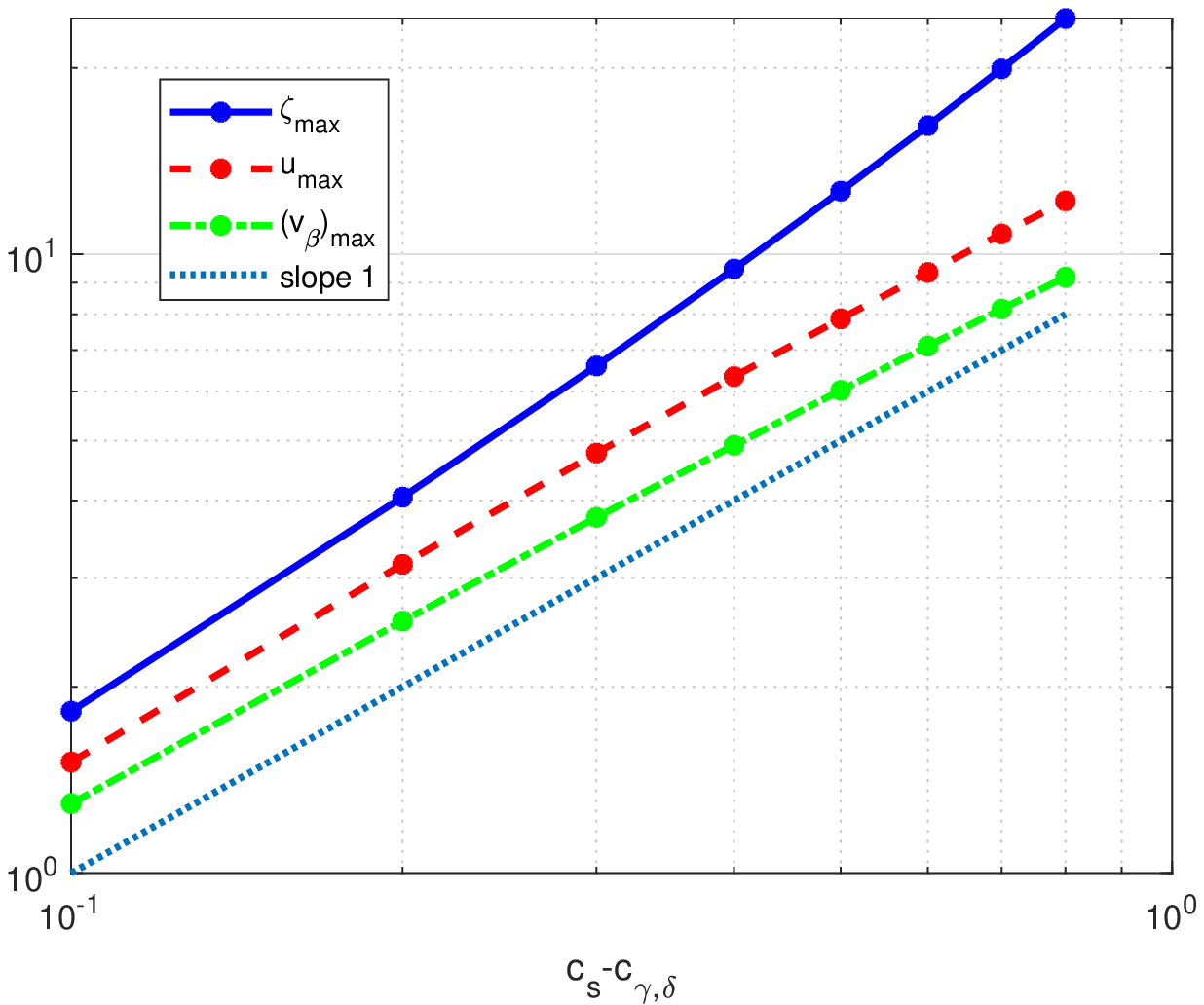}}
\caption{Speed-amplitude relation in (a) linear scale and (b) log-log scale. Case (A6) with $\delta=0.9, \gamma=0.5, a=0, c=0, b= d=\frac{1}{2}\frac{1+\gamma\delta}{3\delta(\gamma+\delta)}$.}
\label{speedamp}
\end{figure}

\begin{remark}
\label{remark44}
The study of solitary waves may be completed by studying numerically generations of bi-modal homoclinic orbits in the form of multi-pulses, typically bifurcating from the CSW's, \cite{ChampS,Champ,AmickT}. An example of these solutions (a positive, two-pulse wave) is shown in section \ref{sec52}.
\end{remark}

\section{Computational study of solitary wave dynamics}
\label{sec5}
In this section we present a computational study of the dynamics of some aspects of solitary waves, both classical and generalized. The {\it generic case} ($a, c<0, b, d>0$) is considered for the experiments. In this case we have the existence of classical solitary waves when $\kappa_{1}bd-ac>0$ (case (A3) of Table \ref{tavle0}) and generalized solitary waves when $\kappa_{1}bd-ac<0$ (case (A2) of Table \ref{tavle0}). 
\subsection{Numerical approximation of the periodic ivp for the B/B system}
\label{sec51}
First we briefly describe the numerical method used for the computational study to follow. We consider the periodic initial-value problem  (\ref{dds31})-(\ref{dds32})  on a long enough interval $(-L,L)$ and discretize it in space with the spectral Galerkin method introduced in Section \ref{sec3}. After the change from the interval $(0,1)$ (for which the convergence analysis was made) to $(-L,L)$ is performed, the corresponding system (\ref{dds36})-(\ref{dds38}) is solved in terms of the Fourier coefficients of the semidiscretization, in analogy to (\ref{38b}). This leads to an ode system of the form
\begin{eqnarray}
\frac{d}{dt}\begin{pmatrix}\widehat{\zeta_{N}}\\ \widehat{u_{N}}\end{pmatrix}=F\begin{pmatrix}\widehat{\zeta_{N}}\\ \widehat{u_{N}}\end{pmatrix},\label{51}
\end{eqnarray}
where for $-N\leq k\leq N, \widehat{k}=k\pi/L$
\begin{eqnarray}
F\begin{pmatrix}\widehat{\zeta_{N}}\\ \widehat{u_{N}}\end{pmatrix}(\widehat{k})=\begin{pmatrix}
\frac{1}{1+b\widehat{k}^{2}}\left(i\widehat{k}((-\kappa_{1}+a\widehat{k}^{2})\widehat{u_{N}}(\widehat{k},t)-\lambda\widehat{\zeta_{N}u_{N}}(\widehat{k},t))\right)\\
\frac{1}{1+d\widehat{k}^{2}}\left(i\widehat{k}(-\kappa_{2}(1-c\widehat{k}^{2})\widehat{\zeta_{N}}(\widehat{k},t)-\frac{\lambda}{2}\widehat{u_{N}^{2}}(\widehat{k},t))\right)
\end{pmatrix},\label{52}
\end{eqnarray}
with  $\widehat{\zeta}_{N}(\widehat{k},0)=\widehat{\zeta}_{0}(\widehat{k}), \widehat{u}_{N}(\widehat{k},0)=\widehat{u}_{0}(\widehat{k})$. The ode system (\ref{51}), (\ref{52})  is then discretized in time with the fourth-order, three-stage RK-composition method based on the implicit midpoint rule (IMR), \cite{Yoshida1990,HairerLW2004}. The scheme belongs to the family of RK methods with Butcher tableau
\begin{equation}\label{44}
\begin{tabular}{c | c}
& $a_{ij}$\\ \hline
 & $b_{i}$
\end{tabular}=
\begin{tabular}{c | ccccc}
& $b_{1}/2$ & & &&\\
&  $b_{1}$ & $b_{2}/2$ &&&\\
&$b_{1}$&$b_{2}$&$\ddots$&&\\
&$\vdots$&$\vdots$ &&$\ddots$&\\
&$b_{1}$&$b_{2}$&$\cdots$&$\cdots$&$b_{s}/2$\\ \hline
 & $b_{1}$&$b_{2}$&$\cdots$&$\cdots$&$b_{s}$
\end{tabular},
\end{equation}
in the particular case of $s=3$ and
\begin{eqnarray}
&&b_{1}=(2+2^{1/3}+2^{-1/3})/3=\frac{1}{2-2^{1/3}}\sim 1.351,\nonumber\\
&& b_{2}=1-2b_{1}\sim -1.702,\quad b_{3}=b_{1},\label{46}
\end{eqnarray}

The method can be written as a composition of three steps of the IMR with stepsizes $b_{1}\Delta t, b_{2}\Delta t$ and $b_{3}\Delta t$. For a general ode system $\dot{y}=f(y)$ this reads
\begin{eqnarray*}
y^{n,1}&=&y^{n}+b_{1}\Delta t f\left(\frac{y^{n}+y^{n,1}}{2}\right),\nonumber\\
y^{n,j}&=&y^{n}+b_{j}\Delta t f\left(\frac{y^{n,j-1}+y^{n,j}}{2}\right),\; j=2,3,\nonumber\\
y^{n+1}&=&y^{n,3}.\label{45}
\end{eqnarray*}

The scheme is fourth-order accurate, symplectic, symmetric and of easy implementation. The full discretization has been analyzed in \cite{DD} in the case of spectral semidiscretizations of the periodic ivp for the KdV equation and its efficiency has been checked in computations with other nonlinear dispersive equations, \cite{FrutosS1992,DDM}. It has been shown to be $L^{2}$-conservative and convergent under suitable CFL conditions in the case of the KdV. Here, in our experiments with the B/B system (\ref{dds31})-(\ref{dds32})  in the generic case $a, c<0, b, d>0$, we also observed that a Courant stability condition of the form $N\Delta t\leq \alpha$ for some $\alpha>0$ was sufficient to ensure stability and convergence of the fully discrete scheme.

For an integer $M\geq 1$ the corresponding fully discrete approximation at times $t_{m}=m\Delta t, m=0,\ldots,M$, where $T=M\Delta t$, is computed in the Fourier space with FFT techniques. The computation of initial solitary wave profiles is carried out with the method outlined in Section \ref{sec42}.
\subsection{Validation of the codes}
\label{sec52}
In this section we present some numerical evidence in order to validate the full discretization introduced in section \ref{sec51} and to complete the study of the accuracy of the numerical procedure to generate approximate solitary wave profiles, developed in section \ref{sec42}. The experiments will additionally serve to give confidence to the numerical simulation of stability and interactions of solitary waves to be carried out in sections \ref{sec53} and \ref{sec54}.

\begin{table}[htbp]
\begin{center}
\begin{tabular}{|c|c|c|c|c|}
\hline
\multicolumn{5}{|c|}{$N=2048$}
\\\hline
$\Delta t$&$\zeta$-error&Rate&$v_{\beta}$-error&Rate
\\\hline
$1/40$&$1.1028E-05$&&$7.3309E-06$&\\\hline
$1/80$&$6.8977E-07$&$3.9989$&$4.5852E-07$&$3.9989$\\\hline
$1/160$&$4.3076E-08$&$4.0012$&$2.8635E-08$&$4.0011$\\
    \hline\hline
\multicolumn{5}{|c|}{$N=4096$}
\\\hline
$\Delta t$&$\zeta$-error&Rate&$v_{\beta}$-error&Rate
\\\hline
$1/40$&$1.1028E-05$&&$7.3309E-06$&\\\hline
$1/80$&$6.8977E-07$&$3.9989$&$4.5852E-07$&$3.9989$\\\hline
$1/160$&$4.2945E-08$&$4.0056$&$2.8548E-08$&$4.0055$\\
    \hline
\end{tabular}
\end{center}
\caption{$L^{2}$-errors and temporal convergence rates at $T=100$ with $L=256$ and $N=2048$ ($h=0.25$) and $N=4096$ ($h=0.125$).}
\label{tavle2}
\end{table}

\begin{figure}[htbp]
\centering
\subfigure[]
{\includegraphics[width=0.8\textwidth]{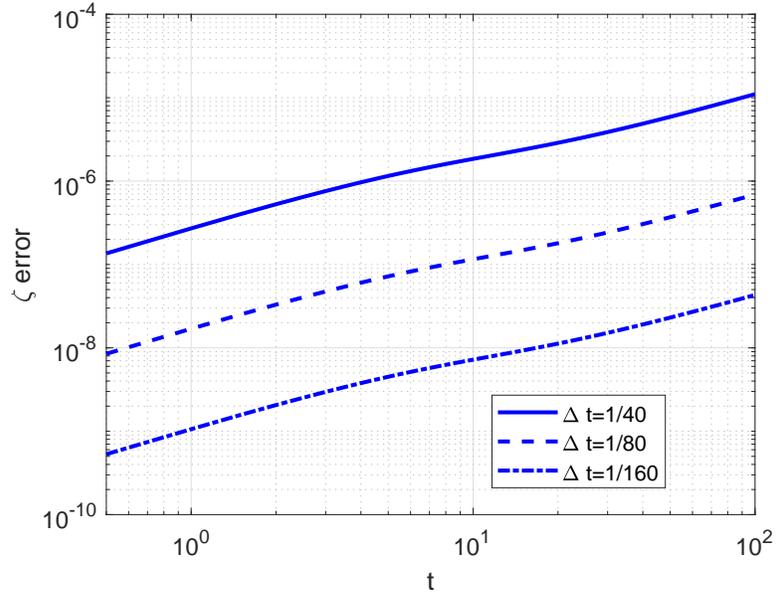}}
\subfigure[]
{\includegraphics[width=6.25cm]{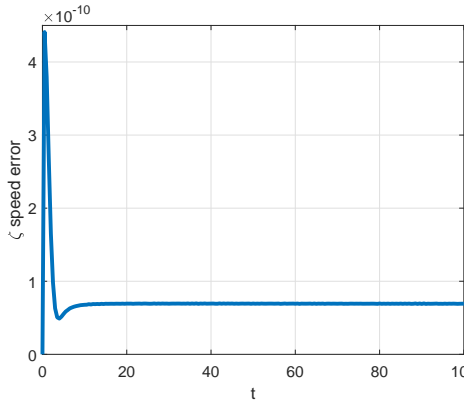}}
\subfigure[]
{\includegraphics[width=6.25cm,height=5.25cm]{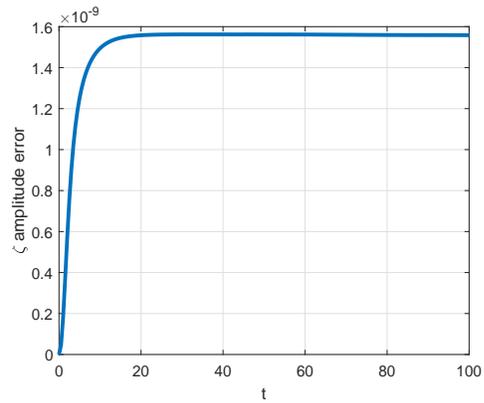}}
\subfigure[]
{\includegraphics[width=6.25cm,height=5.25cm]{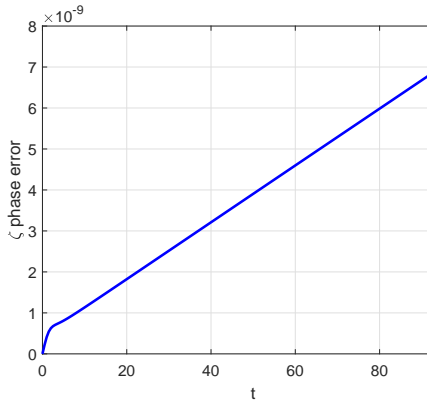}}
\subfigure[]
{\includegraphics[width=6.25cm,height=5.15cm]{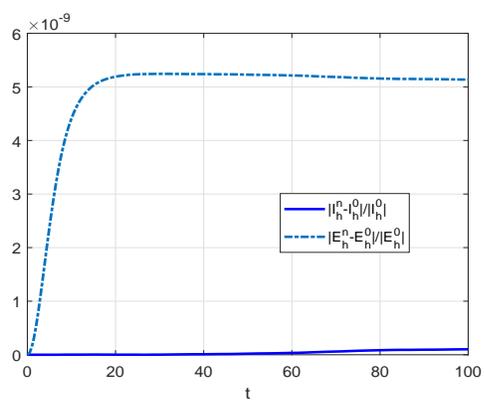}}
\caption{Numerical simulation of (\ref{BB2}) with (\ref{ap1}) up to $T=100$ with $N=4096$. (a) Normalized $L^{2}$-error for $\zeta$ vs. time in log-log scale; (b) $\zeta$ speed error vs. time; (c) $\zeta$ amplitude error vs. time; (d) $\zeta$ phase error vs. time; (e) Errors of (\ref{ap2}) and (\ref{ap3}) vs. time (semilog scale). In (b)-(e) the errors are normalized and $\Delta t=1/160$ is taken.}
\label{fig_BB521}
\end{figure}

In a first experiment we check the temporal order of convergence by simulating with the fully discrete scheme  the propagation of an exact solitary wave solution of the form (\ref{esw6}) corresponding to the parameter values
\begin{eqnarray}
&&\delta=0.9, \gamma=0.5,\nonumber\\
&&a=-1/3, c=-2/3, b=d=\frac{1}{2}(\frac{a}{\kappa_{1}}-c+\frac{1+\gamma\delta}{3\delta(\delta+\gamma)}\approx 0.2918\label{ap1}
\end{eqnarray}
up to a final time $T=100$ on an interval $(-L,L)$ with $L=256$ and periodic boundary conditions. According to (\ref{esw3a}), the speed is $c_{s}\approx 1.0328$ and the amplitude is $\zeta_{\max}\approx 7.9846$. We made two runs with $N=2048$ and $N=4096$ (that is with $h=2L/N=0.250$ and $0.125$ respectively) and several values of $\Delta t$. The normalized $L^{2}$-errors for $\zeta$ and $v_{\beta}$ at the final time and corresponding temporal convergence rates are displayed in Table \ref{tavle2}. The results show the fourth order of convergence of the time discretization and the associated errors in space.
%\begin{figure}[htbp]
%\centering
%\subfigure[]
%{\includegraphics[width=0.8\textwidth]{code_check1b.eps}}
%\subfigure[]
%{\includegraphics[width=6.25cm]{code_check1d.eps}}
%\subfigure[]
%{\includegraphics[width=6.25cm,height=5.25cm]{code_check1c.eps}}
%\subfigure[]
%{\includegraphics[width=6.25cm,height=5.25cm]{code_check1e.eps}}
%\subfigure[]
%{\includegraphics[width=6.25cm,height=5.25cm]{code_check1.eps}}
%\caption{Numerical simulation of (\ref{esw4}) with (\ref{ap1}) up to $T=100$ with $N=4096$. (a) Normalized $L^{2}$-error for $\zeta$ vs. time in log-log scale; (b) $\zeta$ speed error vs. time; (c) $\zeta$ amplitude error vs. time; (d) $\zeta$ phase error vs. time; (e) Errors in (\ref{ap2}) and (\ref{ap3}) vs. time. In (b)-(e) the errors are normalized and $\Delta=1/160$ is taken.}
%\label{fig_BB521}
%\end{figure}
Since for the values (\ref{ap1}) the system (\ref{BB2}) is Hamiltonian and its solutions are smooth, decaying to zero at infinity, and preserving the quantities (\ref{mom}) and (\ref{energy}), we may study the ability of the numerical solution to conserve the corresponding discrete versions of the invariants, given by
\begin{eqnarray}
I_{h}(U,V)&=&h\sum_{j=0}^{N-1}(U_{j}V_{j}+b(D_{N}U)_{j}(D_{N}V)_{j}),\label{ap2}\\
E_{h}(u,V)&=&h\sum_{j=0}^{N-1}\left(\frac{\kappa_{2}}{2}U_{j}^{2}+\frac{\kappa_{1}}{2}V_{j}^{2}-a(D_{N}V)_{j}^{2}\right.\nonumber\\
&&\left.-\kappa_{2}c(D_{N}U)_{j}^{2}+\frac{\kappa_{\gamma,\delta}}{2}U_{j}V_{j}^{2}\right),\label{ap3}
\end{eqnarray}
for $U=(U_{0},\ldots,U_{N-1}), V=(V_{0},\ldots,V_{N-1})$ and $D_{N}$ standing for the pseudospectral differentiation matrix of order $N$, cf. section \ref{sec42}. If $I_{h}^{n}, E_{h}^{n}$ denote, respectively, the values of (\ref{ap2}) and (\ref{ap3}) of the numerical solution at time $t_{n}=n\Delta t$, we obtain the evolution of the normalized errors $|(I_{h}^{n}-I_{h}^{0})/I_{h}^{0}|, |(E_{h}^{n}-E_{h}^{0})/E_{h}^{0}|$ shown in Figure \ref{fig_BB521}(e) for $\Delta t=1/160$. It shows the preservation of the two quantities up to the final time $T=100$ with a normalized error of $1.0198\times 10^{-10}$ and $5.1366\times 10^{-9}$ respectively. Due to the form of the bilinear invariant (\ref{mom}), we have almost exact in time conservation of the discrete $H^{1}$-norm, defined by the pseudospectral differentiation of the numerical solution.

In addition, Figure \ref{fig_BB521}(a) shows, in log-log scale, the evolution of the normalized $\zeta$ errors in $L^{2}$ norm as function of time for $N=4096$ and several values of $\Delta t$. (The experiment with $N=2048$ was also made, with similar results.) The slopes of the lines suggest approximately linear growth in time of the errors, as expected, \cite{FrutosS1997}. A source of this behaviour is due to the evolution of the error with respect to some parameters of the solitary waves. For the case at hand this is illustrated by Figures \ref{fig_BB521}(b)-(d) showing, for $N=4096, \Delta t=1/160$, the relative errors in speed, amplitude, and the error in phase, respectively, of the  numerical approximation of $\zeta$ as functions of time. These errors are computed in a standard way, as in e.~g. \cite{DougalisDLM2007}. The results show the preservation of speed and amplitude (up to an error at $t=100$ of about $6.9340\times 10^{-11}$ and $1.5584\times 10^{-9}$ respectively) and a linear growth of the phase error.
%\subsection{CSW dynamics. Numerical experiments}
%\label{sec53}
%In this section we 
%introduce and comment upon some numerical experiments concerning the dynamics of CSW's for the generic case ($a, c<0, b, d>0$) and $\kappa_{1}bd-ac>0$ (case (A3)). The experiments concern the evolution of small and large perturbations of CSW's, as well as head-on and overtaking collisions. For simplicity, only the results concerning the first component $\zeta^{N}$ will be shown.
\begin{figure}[htbp]
\centering
\subfigure[]
{\includegraphics[width=0.8\textwidth]{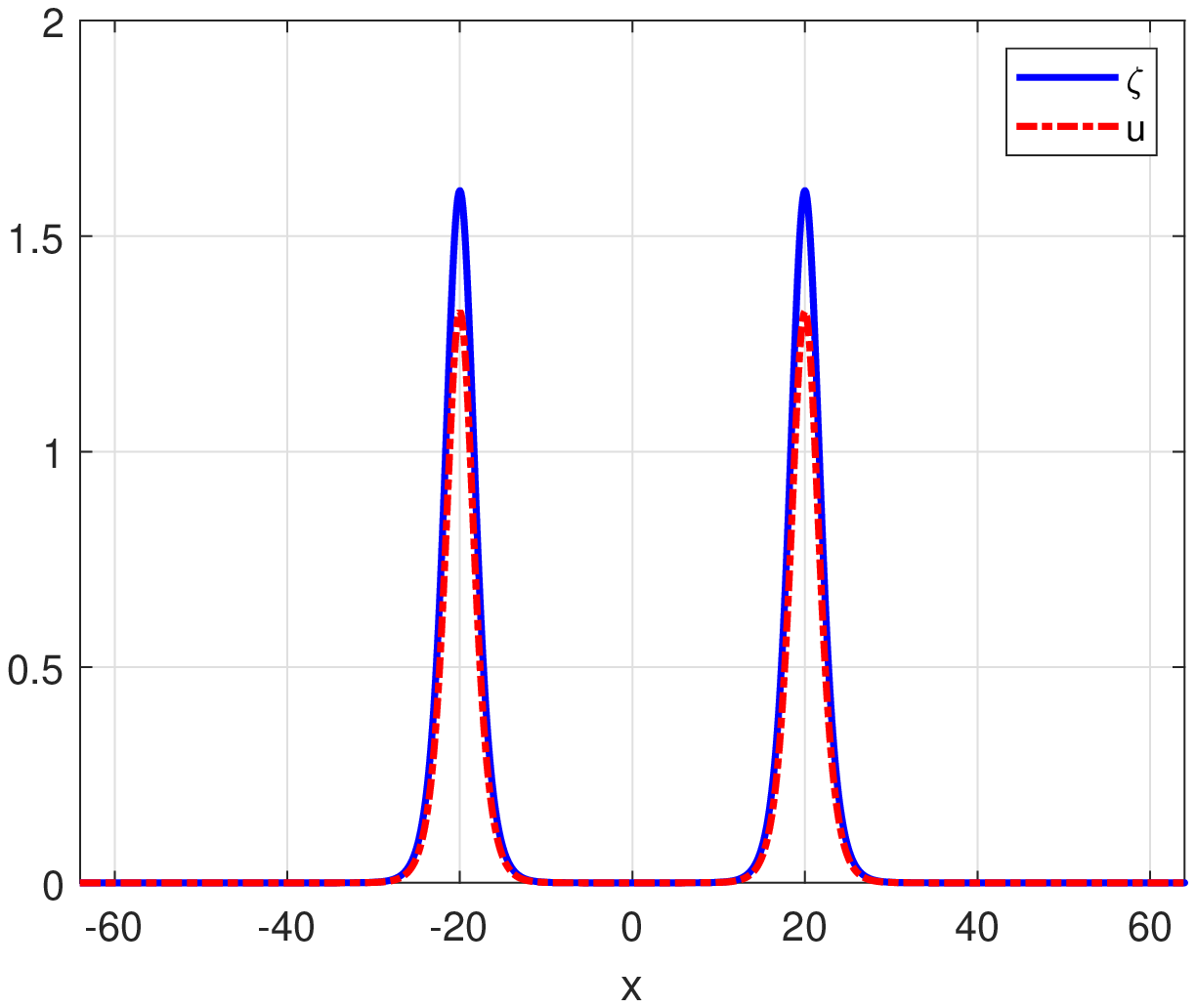}}
\subfigure[]
{\includegraphics[width=6.27cm]{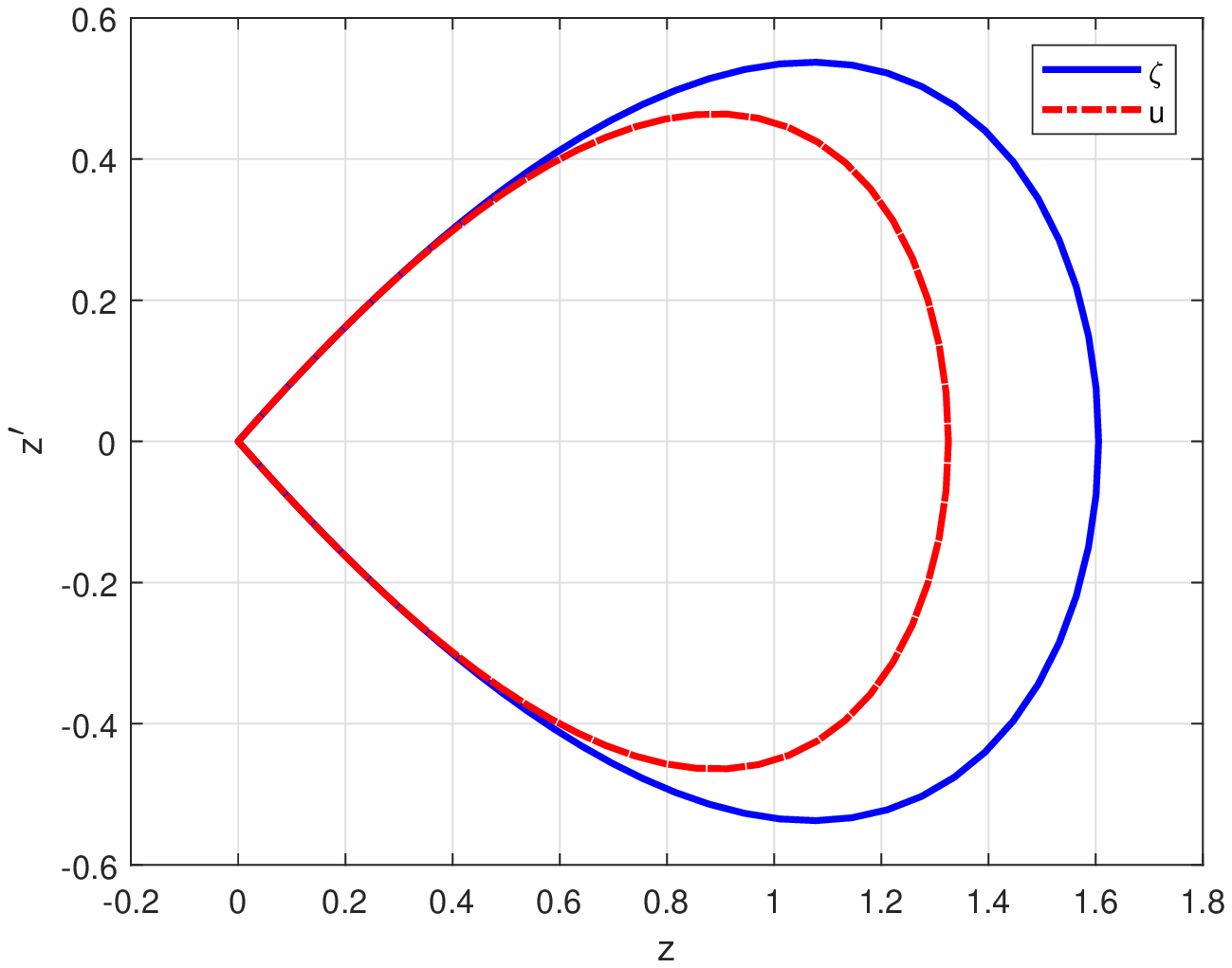}}
\subfigure[]
{\includegraphics[width=6.27cm]{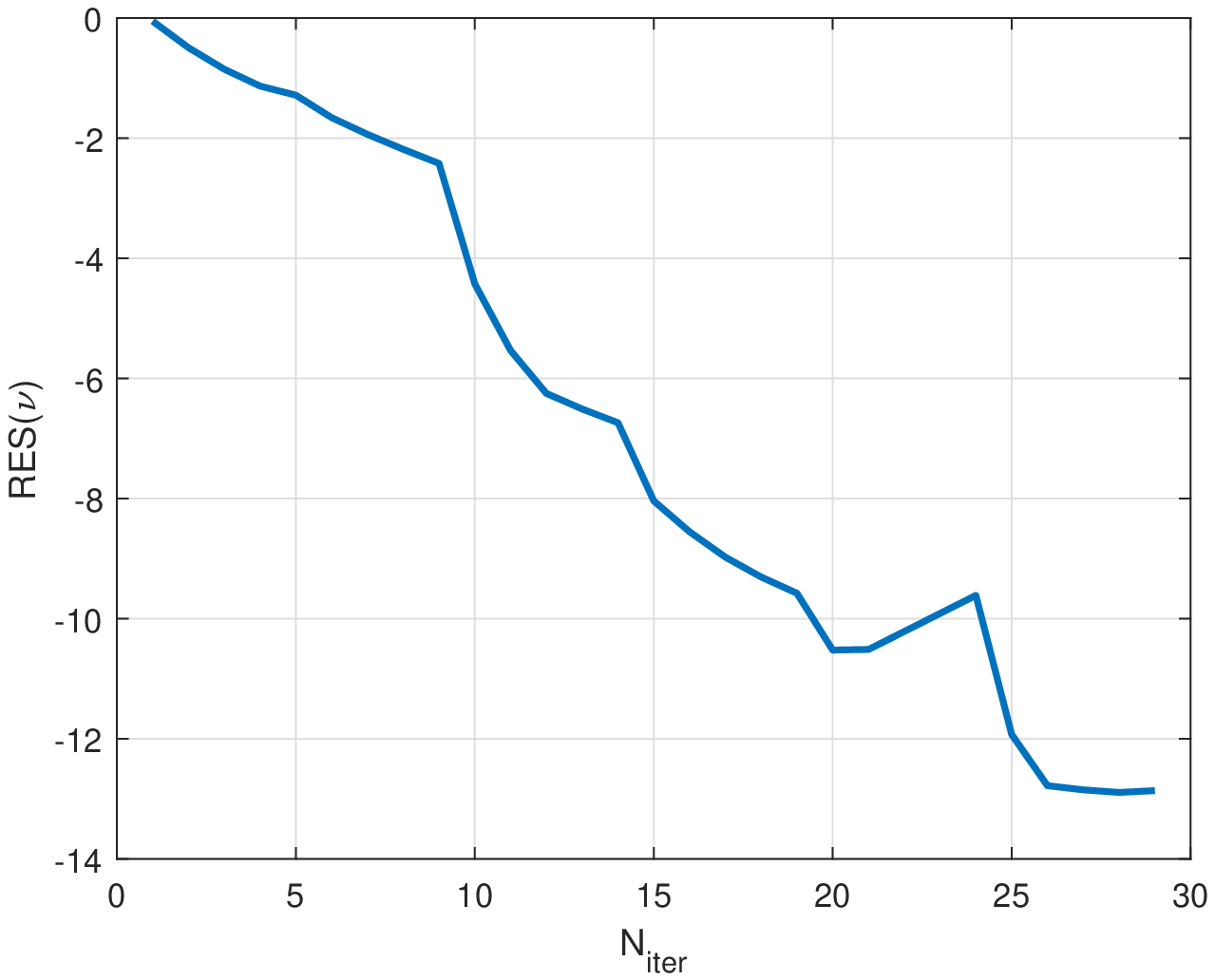}}
\caption{Two-pulse solitary wave generated numerically with parameters given by (\ref{ap4}). (a) $\zeta$ and $u$ profiles; (c) $\zeta$ and $u$ phase portraits; (c) Residual error vs. number of iterations (semilog scale).}
\label{fig_BB522}
\end{figure}

In a second experiment we took the values
\begin{eqnarray}
\delta=0.9, \gamma=0.5,\;
a=c=0, b=d=\frac{1}{2}(\frac{1+\gamma\delta}{3\delta(\delta+\gamma)}\approx 0.1918,\label{ap4}
\end{eqnarray}
and taking as initial profile a superposition of two hyperbolic secant square profiles of amplitudes equal to $1$ and centered at $x=\pm 20$, we ran the iteration (\ref{424}). The method converges to the profile shown in Figure \ref{fig_BB522}(a), which has the form of a symmetric two-pulse solitary wave (cf. Remark \ref{remark44}). The accuracy of the iteration is first suggested by the convergence to practically zero of the residual error shown in Figure \ref{fig_BB522}(c). Both pulses have a computed amplitude of about $1.6055$ and the two-pulse wave travels with a speed of $6.9761\times 10^{-1}$.

\begin{figure}[htbp]
\centering
\subfigure[]
{\includegraphics[width=6.25cm]{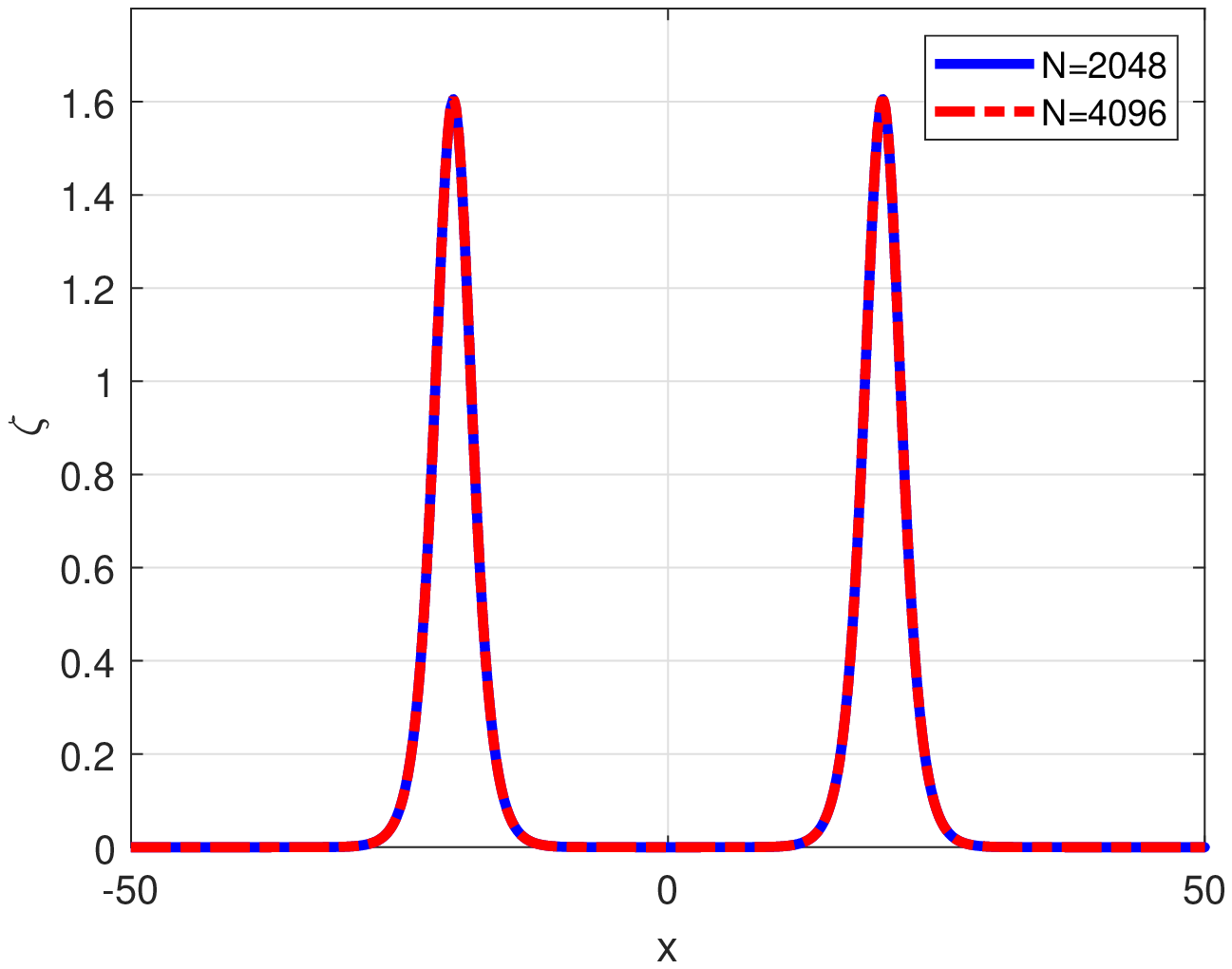}}
\subfigure[]
{\includegraphics[width=6.25cm]{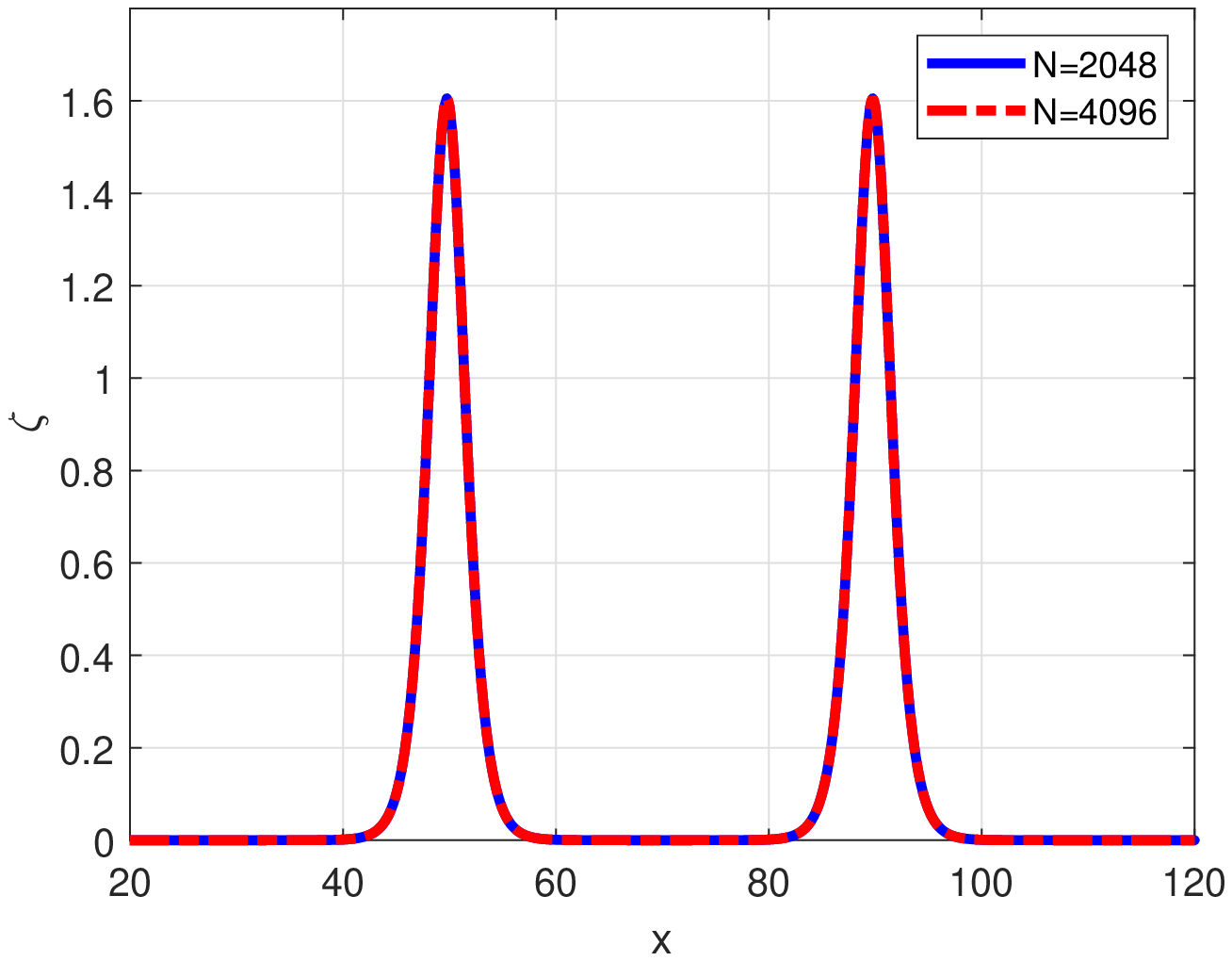}}
\subfigure[]
{\includegraphics[width=6.25cm,height=5.25cm]{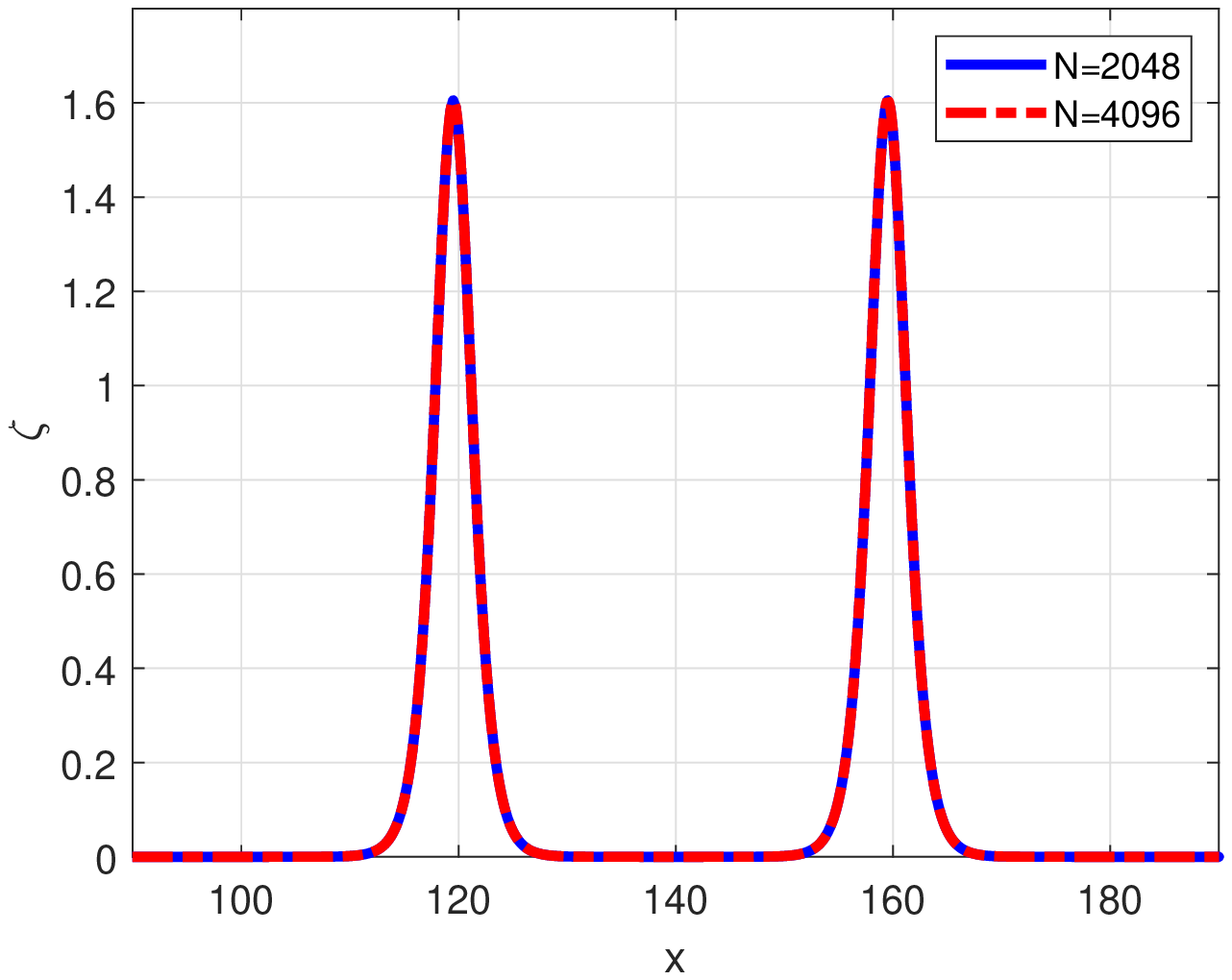}}
\subfigure[]
{\includegraphics[width=6.25cm,height=5.25cm]{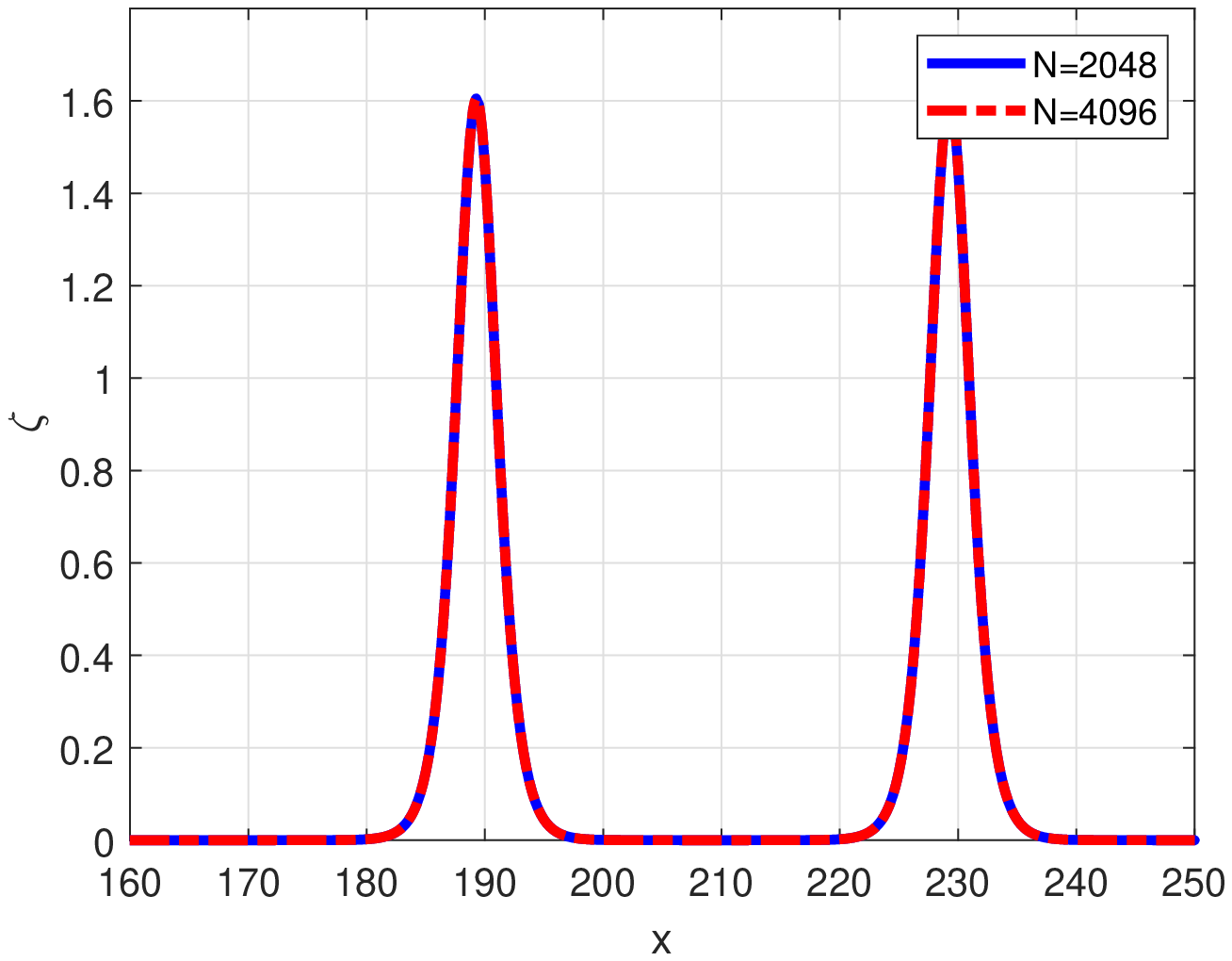}}
%\subfigure[]
%{\includegraphics[width=6.25cm,height=5.25cm]{code_check1.eps}}
\caption{Simulation of (\ref{BB2}) with (\ref{ap4}) up to $T=400$ with $\Delta t=1/160$. Approximation of $\zeta(x,t)$ at (a) $t=0$; (b) $t=100$; (c) $t=200$; (d) $t=300$.}
\label{fig_BB523}
\end{figure}

In addition, we took this computed two-pulse as initial condition of the fully discrete method and monitored the evolution of the numerical solution up to a final time $T=400$ with several values of the discretization parameters $h$ (or $N$) and $\Delta t$. The profiles at several times generated with $\Delta t=1/160$ and $N=2048, 4096$, are shown in Figure \ref{fig_BB523}; they coincide within graph thickness. 

%The good performance of the invariants (\ref{ap2}) and (\ref{ap3}) as well as the evolution of amplitude and speed errors is shown in Figures \ref{fig_BB523}(b)-(d).
\begin{figure}[htbp]
\centering
\subfigure[]
{\includegraphics[width=6.25cm]{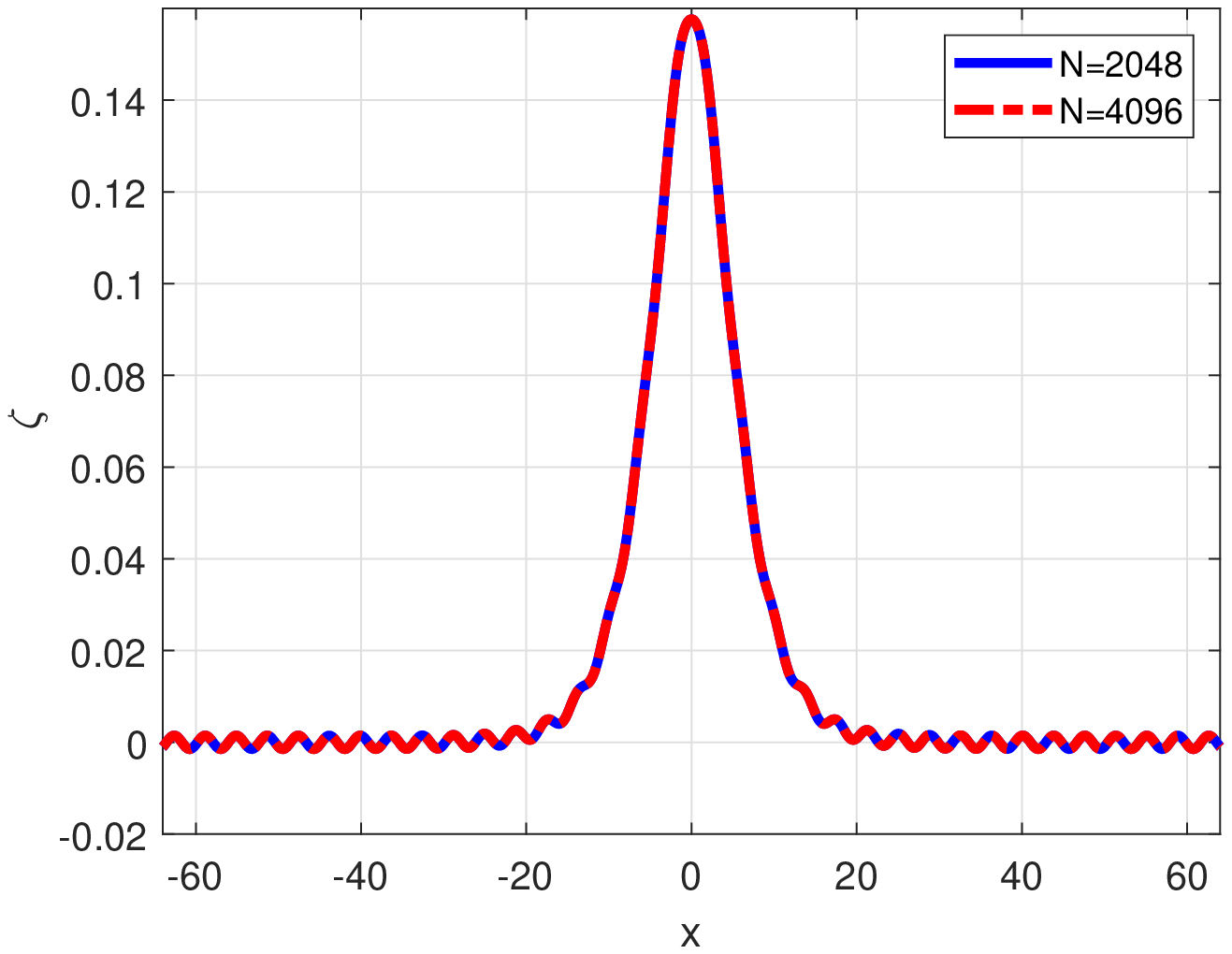}}
\subfigure[]
{\includegraphics[width=6.25cm]{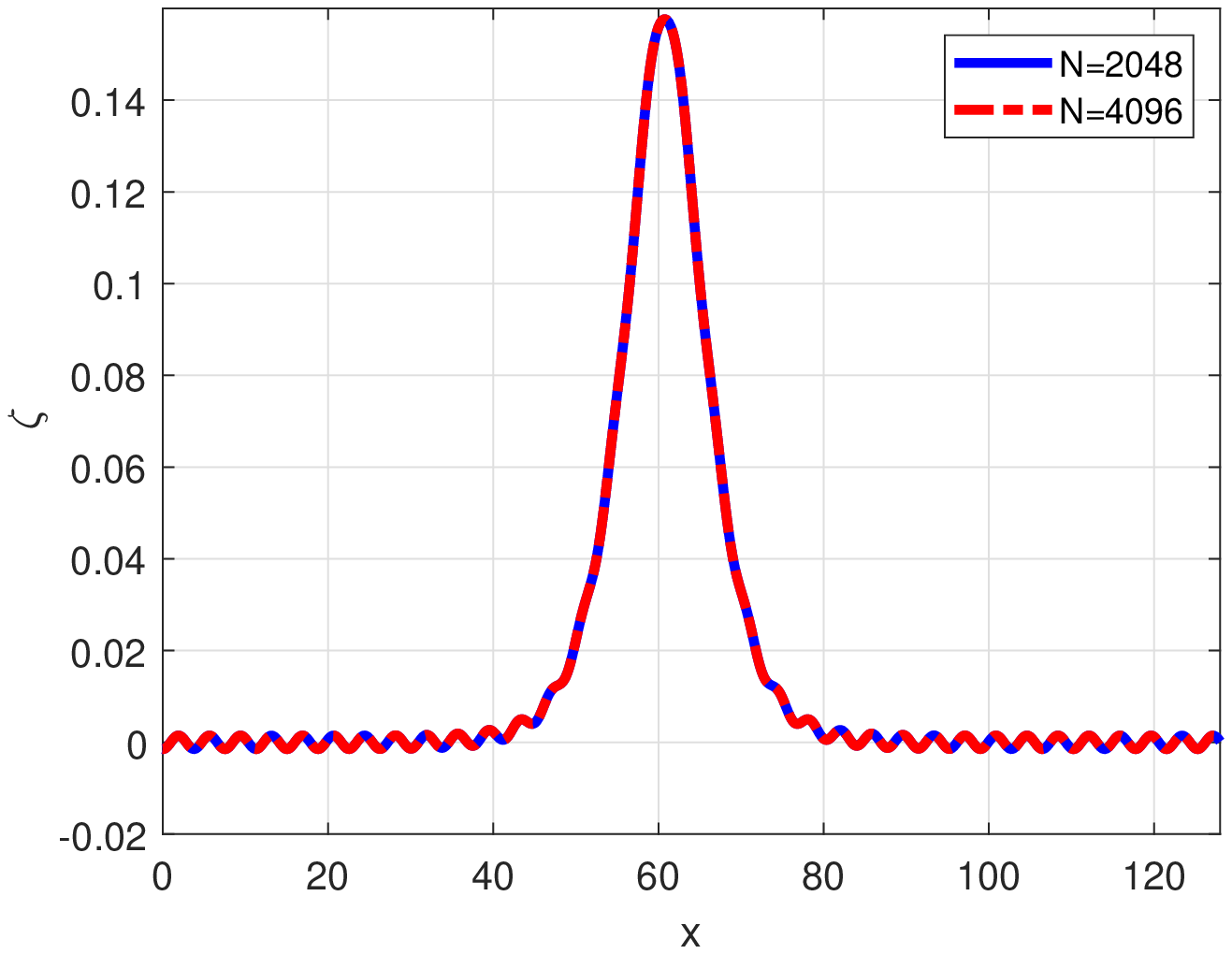}}
\subfigure[]
{\includegraphics[width=6.25cm,height=5.25cm]{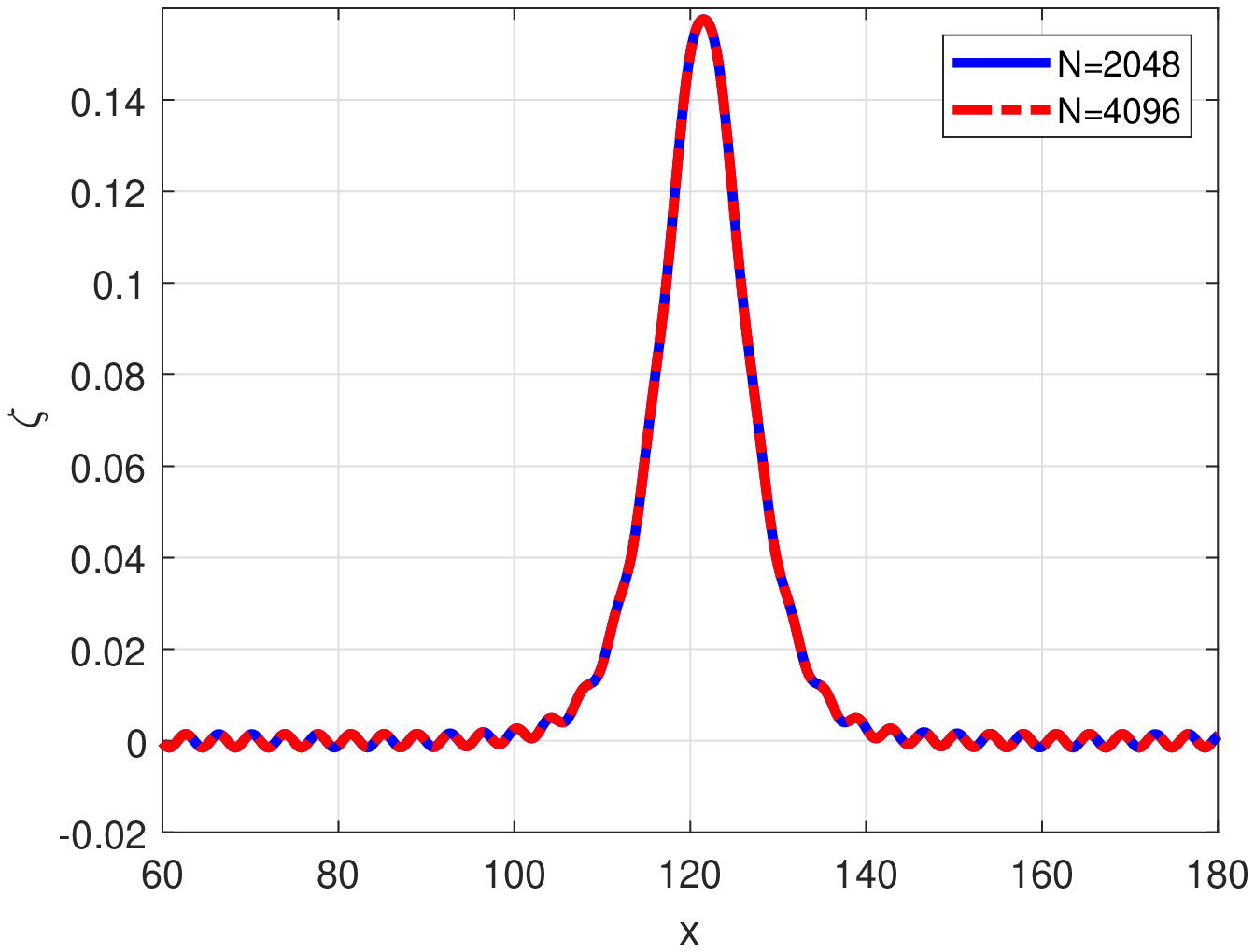}}
\subfigure[]
{\includegraphics[width=6.25cm,height=5.25cm]{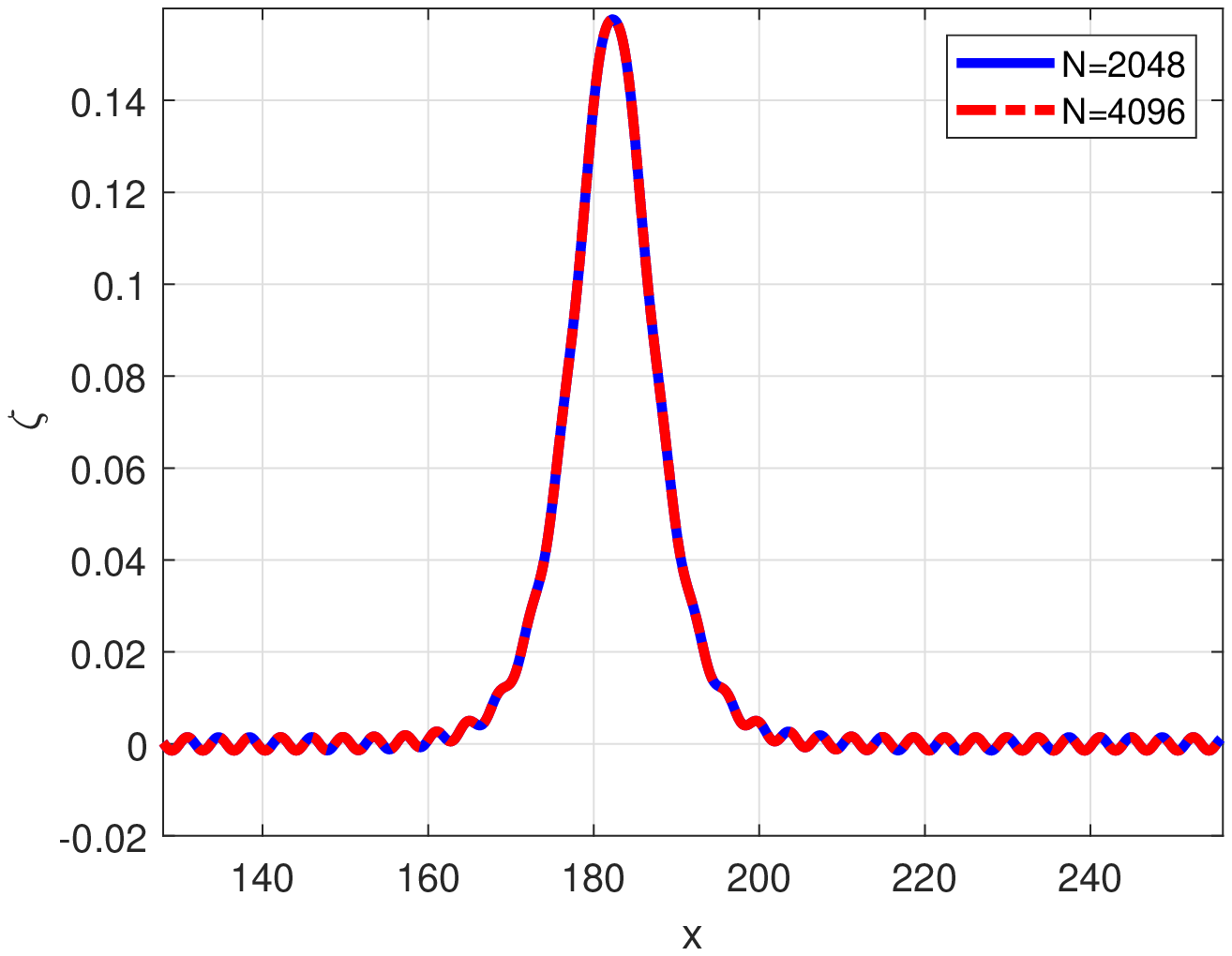}}
%\subfigure[]
%{\includegraphics[width=6.25cm,height=5.25cm]{code_check1.eps}}
\caption{Simulation of (\ref{BB2}) with (\ref{ap5}) up to $T=400$ with $\Delta t=1/160$. Approximation of $\zeta(x,t)$ at (a) $t=0$; (b) $t=100$; (c) $t=200$; (d) $t=300$.}
\label{fig_BB524}
\end{figure}

A third experiment for checking the accuracy of the codes concerns the simulation of a generalized solitary wave. Taking the values $L=256$, and
\begin{eqnarray}
\delta=0.9, \gamma=0.5,\;
a=c=-1/3, b=0, d=-\frac{a}{\kappa_{1}}-c+\frac{1+\gamma\delta}{3\delta(\delta+\gamma)}\approx 1.1836,\label{ap5}
\end{eqnarray}
and the GSW profile generated by the iteration (\ref{424}) as initial condition for the fully discrete scheme, the evolution of the resulting numerical approximation is observed in the following experiments. Figure \ref{fig_BB524} shows the numerical approximation of $\zeta$ at several time instances with $N=2048$ and $\Delta t=1/160$, confirming the preservation of the permanent form of the wave as it evolves. (The amplitude of computed initial GSW profile is $\zeta_{max}\approx 1.5764\times 10^{-1}$ and the profile was generated with speed $c_{s}=c_{\gamma,\delta}+0.01\approx 6.0761\times 10^{-1}$.) The profiles obtained with $N=4096$ were also computed and they coincide with the corresponding ones obtained for $N=2048$ within the graph thickness. Figure \ref{fig_BB525} is a magnification of Figure \ref{fig_BB524} and shows the structure of the ripples in more detail.
\begin{figure}[htbp]
\centering
\subfigure[]
{\includegraphics[width=6.25cm]{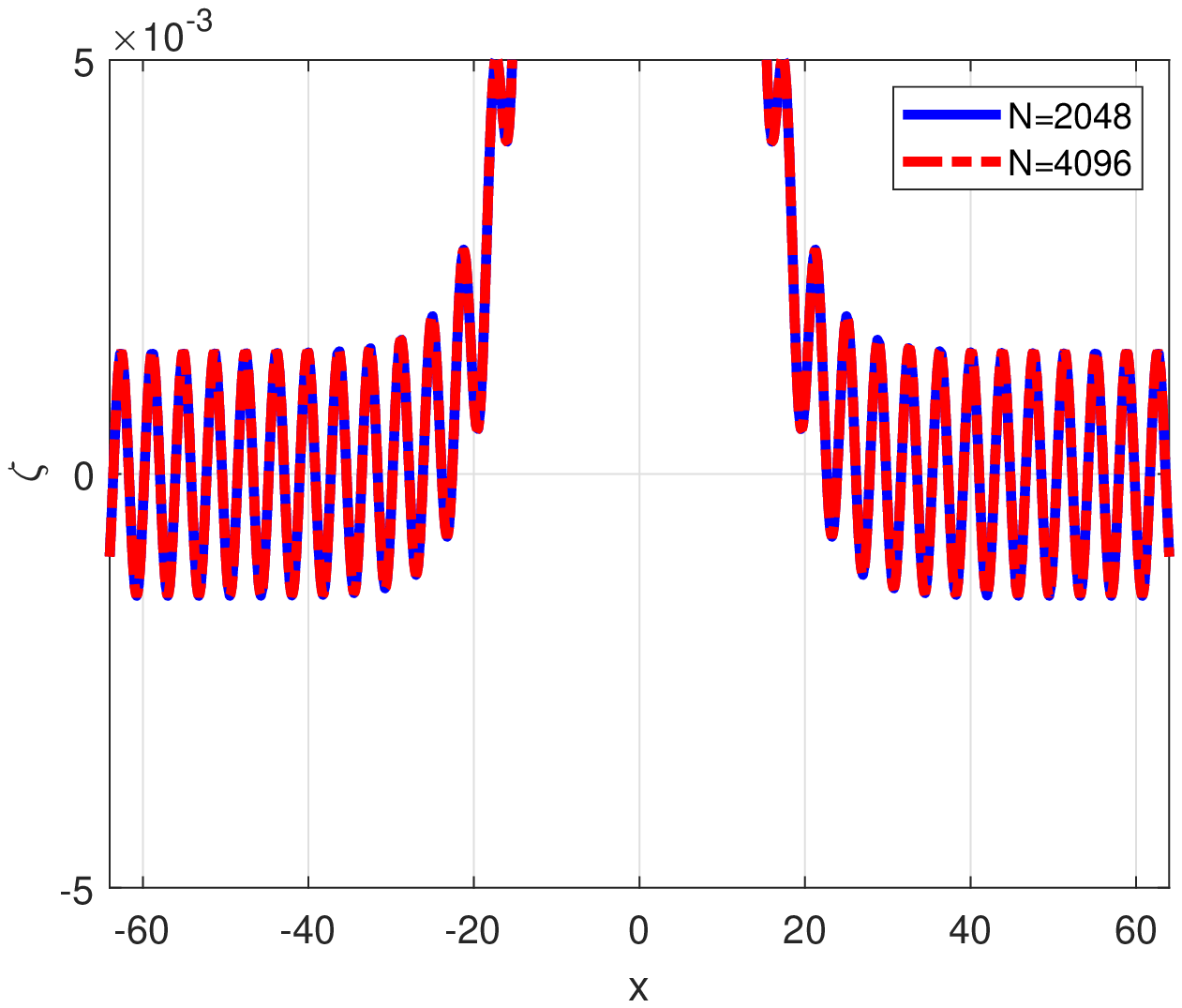}}
\subfigure[]
{\includegraphics[width=6.25cm]{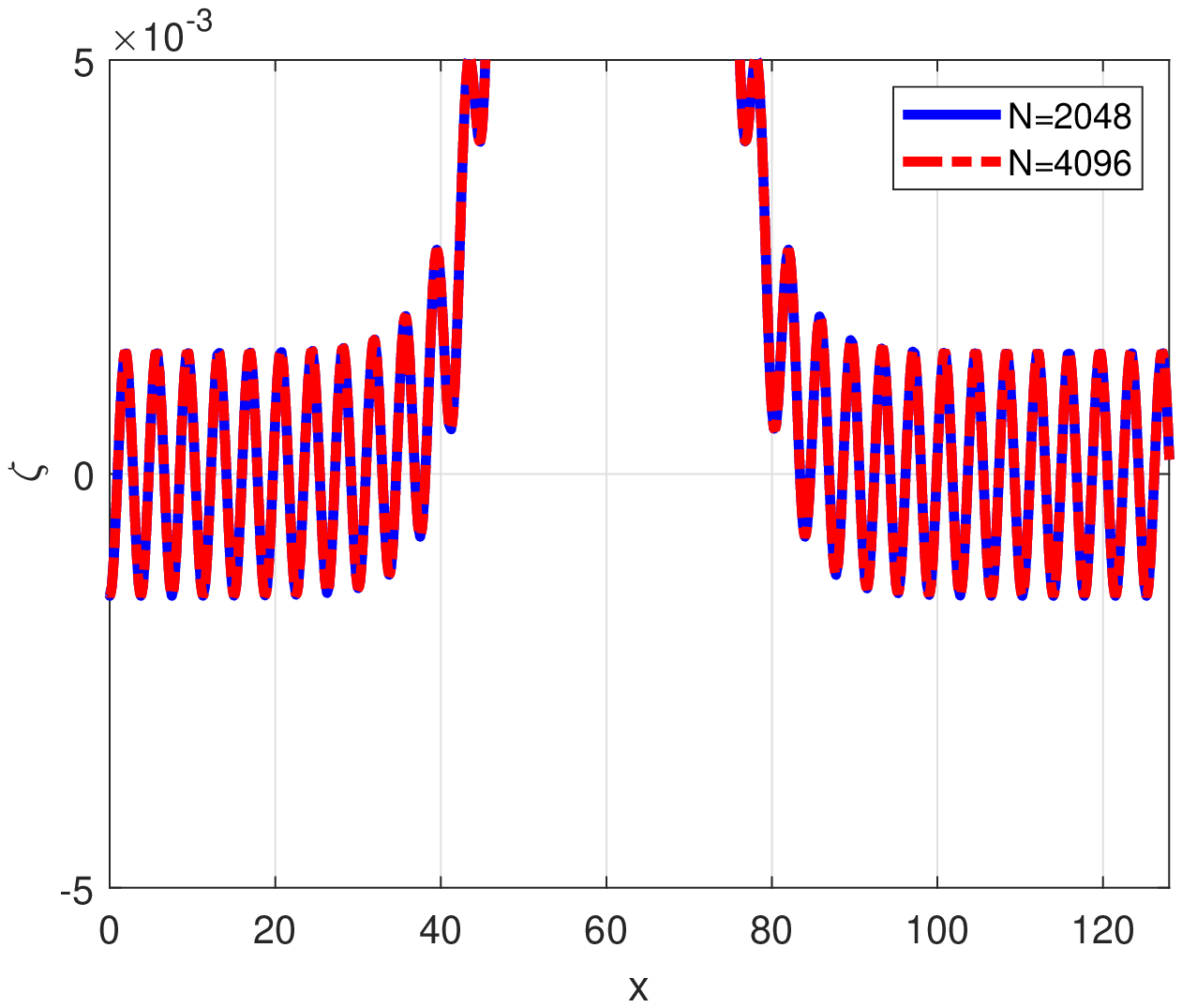}}
\subfigure[]
{\includegraphics[width=6.25cm,height=5.25cm]{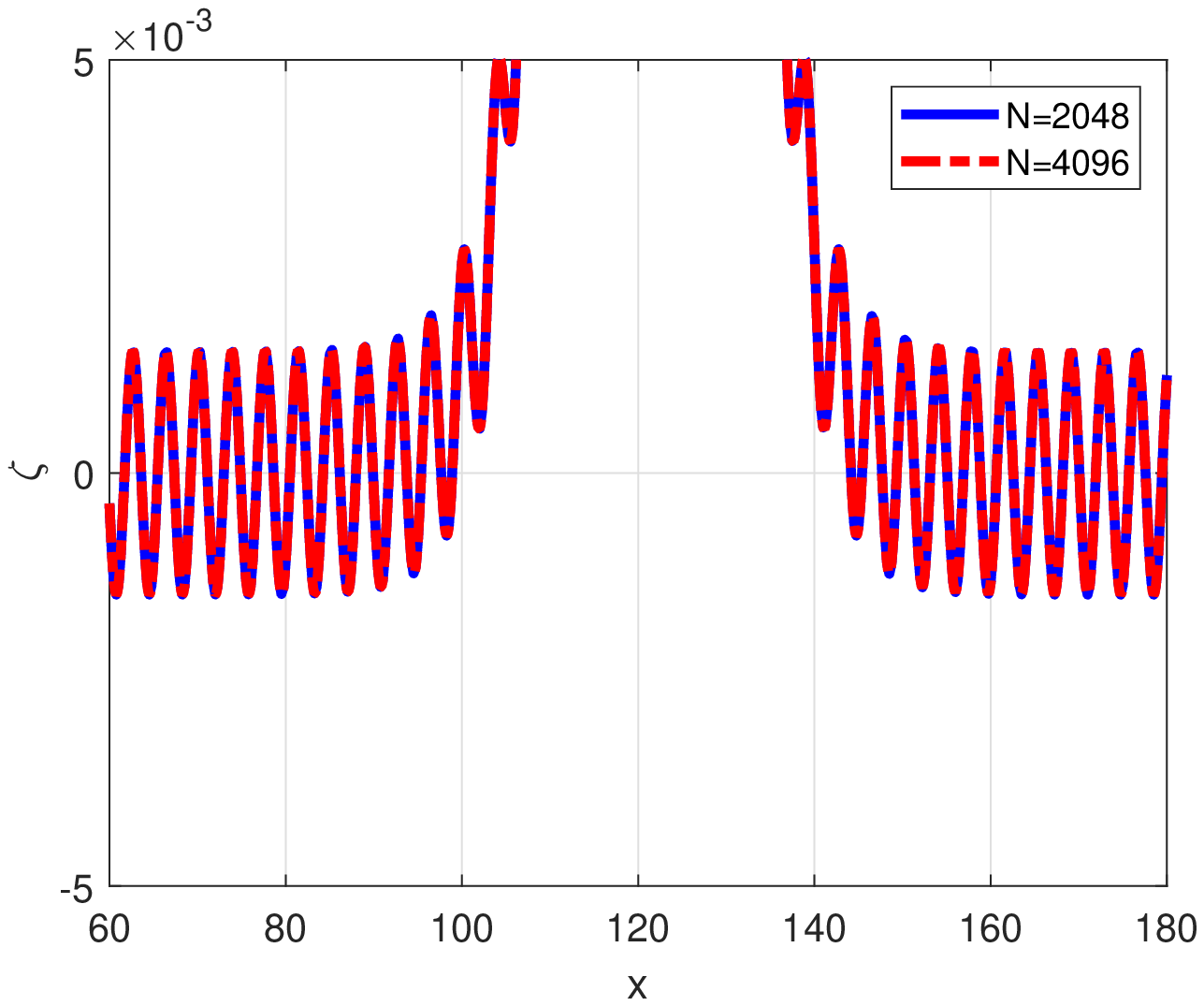}}
\subfigure[]
{\includegraphics[width=6.25cm,height=5.25cm]{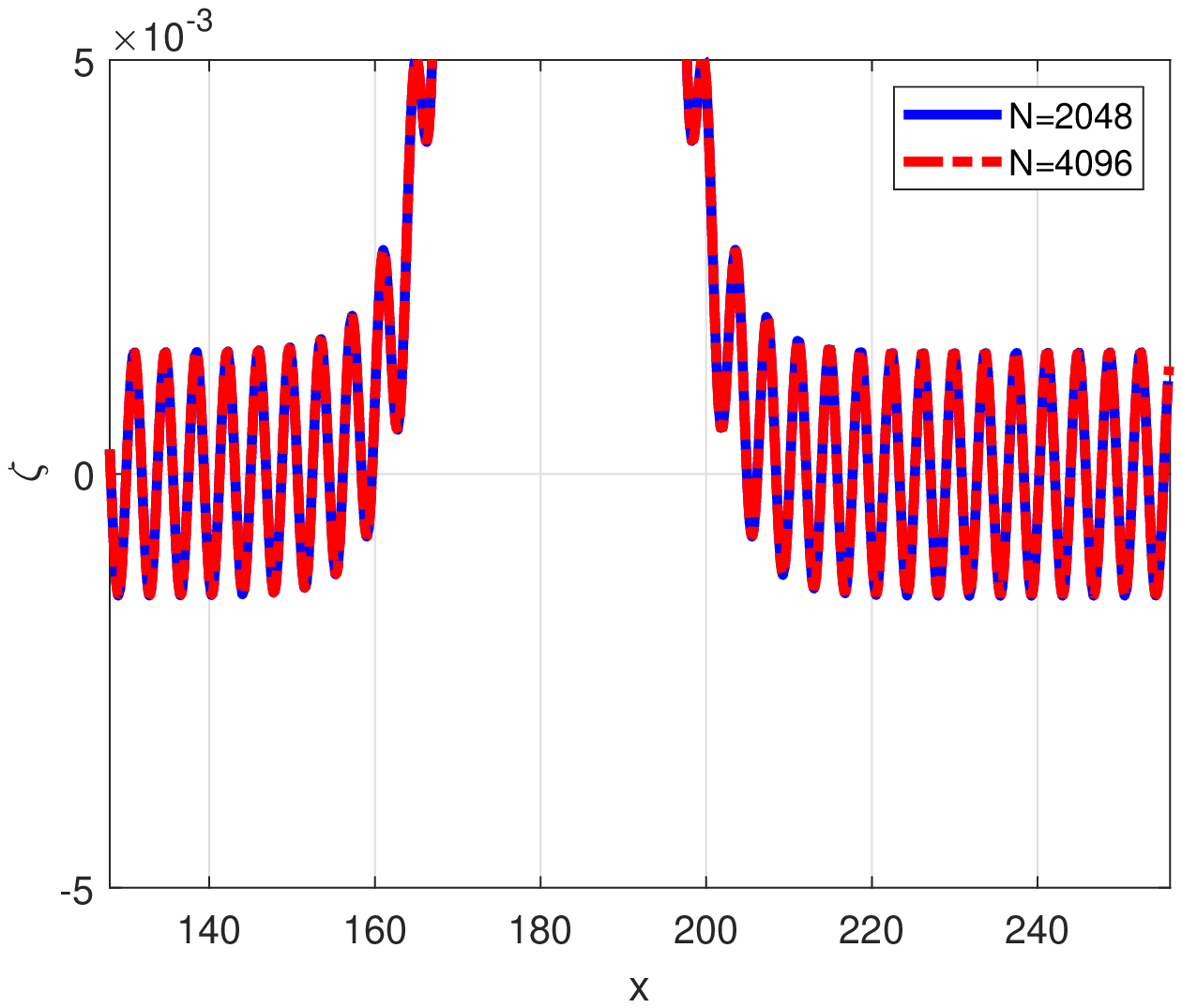}}
%\subfigure[]
%{\includegraphics[width=6.25cm,height=5.25cm]{code_check1.eps}}
\caption{Magnification of Figure \ref{fig_BB524} near the base of the main pulse. 
Approximation of $\zeta(x,t)$ at (a) $t=0$; (b) $t=100$; (c) $t=200$; (d) $t=300$.}
\label{fig_BB525}
\end{figure}

The accuracy of the computations is also confirmed by monitoring the evolution of amplitude and speed errors, shown in Figure \ref{fig_BB526}. Observe that this case is not Hamiltonian and the quantities (\ref{mom}) , (\ref{energy}) are not preserved by the solution. The $L^{2}$ norm of the numerical solution for $N=2048, 4096$, evaluated at several time instances is shown in Table \ref{tavle3}. This is preserved up to twelve significant digits, and is equal to $4.5789161770\times 10^{-1}$, this furnishes more evidence of the accuracy of the computations.

\begin{figure}[htbp]
\centering
\subfigure[]
{\includegraphics[width=6.25cm,height=5.25cm]{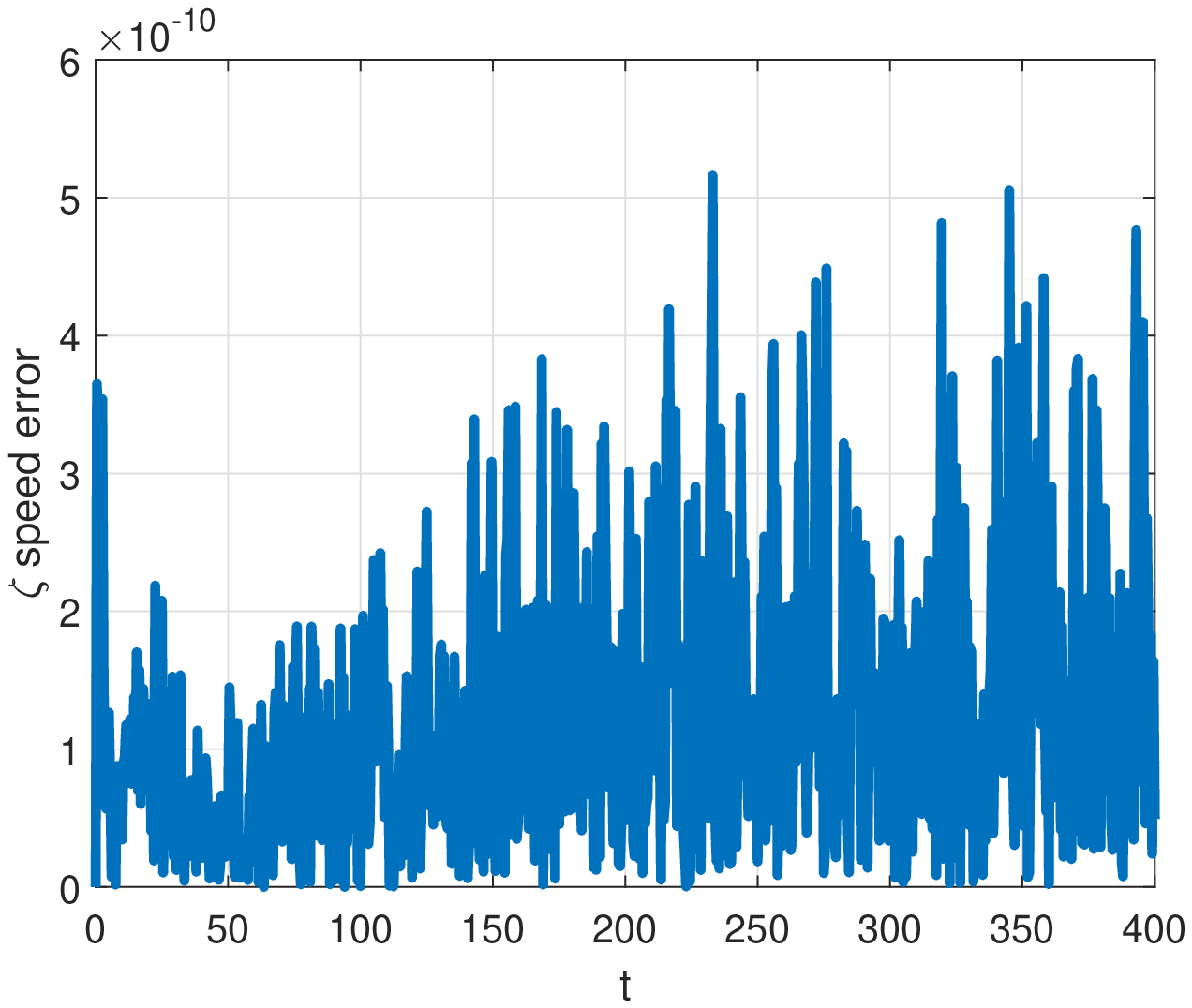}}
\subfigure[]
{\includegraphics[width=6.25cm,height=5.25cm]{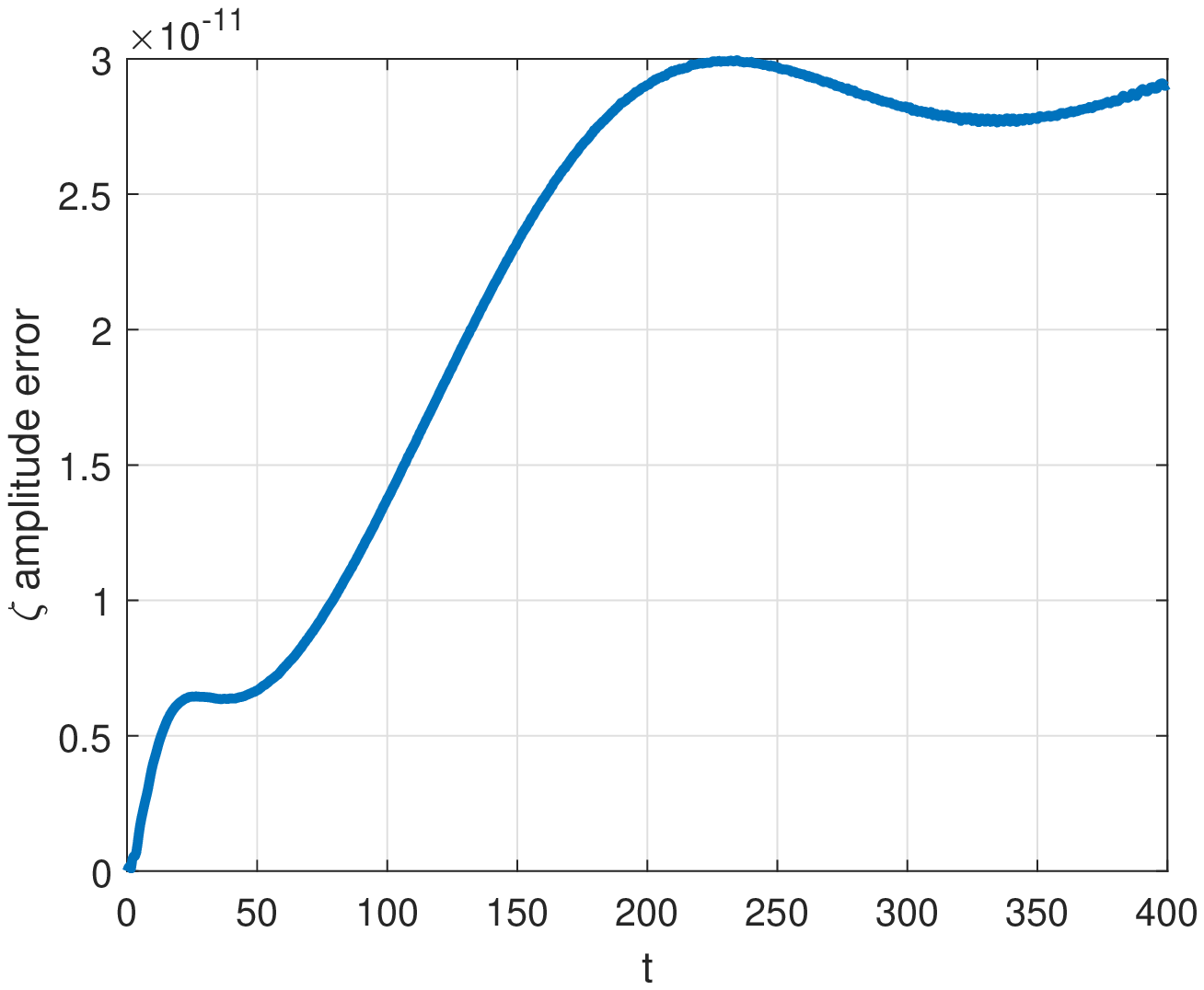}}
\caption{Simulation of (\ref{BB2}) with (\ref{ap5}) up to $T=400$ with $\Delta t=1/160$ and $N=4096$. (a) Normalized $\zeta$ speed error vs. time; (b) Normalized $\zeta$ amplitude error vs. time.}
\label{fig_BB526}
\end{figure}
\begin{table}[htbp]
\begin{center}
\begin{tabular}{|c|c|c|}
\hline
$t$&$N=2048$&$N=4096$
\\\hline
$t=0$&$4.578916177071256e-01$&$4.578916177063826e-01 $\\\hline
$t=100$&$4.578916177071556e-01$&$4.578916177064732e-01$\\\hline
$t=200$&$4.578916177072831e-01$&$4.578916177066413e-01$\\\hline
$t=300$&$4.578916177073125e-01$&$4.578916177067275e-01$\\\hline
$t=400$&$4.578916177075660e-01$&$4.578916177070251e-01$\\\hline
\end{tabular}
\end{center}
\caption{Simulation of (\ref{BB2}) from a GSW (Figure \ref{fig_BB524}) up to $T=400$ with $\Delta=1/160$. $L^{2}$ norm of $\zeta$ component.}
\label{tavle3}
\end{table}

\subsection{CSW dynamics. Numerical experiments}
\label{sec53}
In this section we present some numerical experiments on the dynamics of CSW's in the generic case ($a, c<0, b, d>0$) and $\kappa_{1}bd-ac>0$ (case (A3)). The experiments concern the evolution of small and larger perturbations of CSW's, as well as head-on and overtaking collisions of CSW's. For simplicity, only the results for the $\zeta$-component $\zeta^{N}$ will be shown.
\subsubsection{Small perturbations of a CSW}
In order to illustrate the evolution of a perturbed CSW in this generic case, we take the specific values
\begin{eqnarray}
&&\gamma=0.5, \delta=0.9, a=-1/3, c=-2/3, b=1/3, \nonumber\\
&&d=-\frac{a}{\kappa_{1}}-b-c+\frac{1+\gamma\delta}{3\delta(\gamma+\delta)}\approx 1.1836.\label{52a}
\end{eqnarray}
A typical experiment consists of generating a corresponding approximate CSW profile $(\zeta_{s}^{N}, u_{s}^{N})$ for these values (here with specific data $L=256, N=4096, c_{s}=c_{\gamma,\delta}+0.1\approx 0.6976$, starting from a ${\rm sech}^{2}$-initial profile), taking a perturbation
\begin{eqnarray}
\zeta^{N}(0)=A\zeta_{s}^{N},\; u^{N}(0)=Au_{s}^{N},\label{52b}
\end{eqnarray}
with $A$ constant as initial condition and monitoring the resulting numerical solution up to some final time $T$, which was taken in the experiments to be up to $T=800$. The time step was $\Delta t=6.25\times 10^{-3}$.
\begin{figure}[htbp]
\centering
\subfigure[]
{\includegraphics[width=\columnwidth]{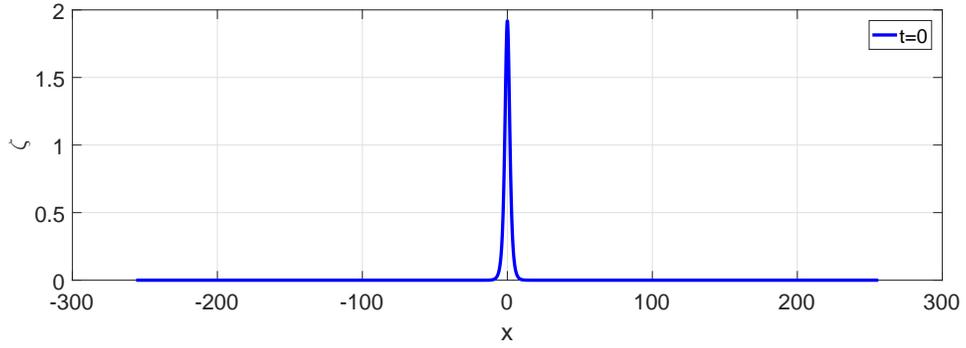}}
\subfigure[]
{\includegraphics[width=\columnwidth]{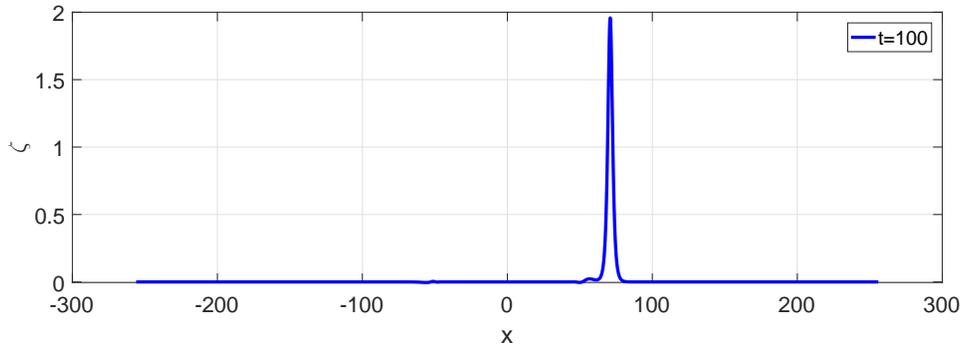}}
\subfigure[]
{\includegraphics[width=\columnwidth]{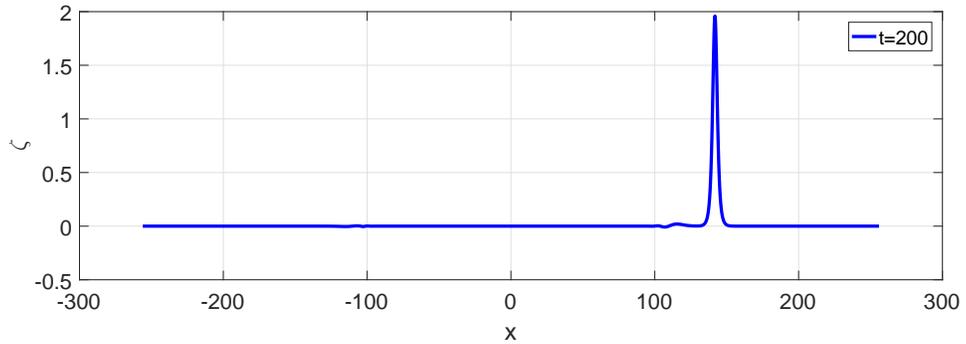}}
\subfigure[]
{\includegraphics[width=\columnwidth]{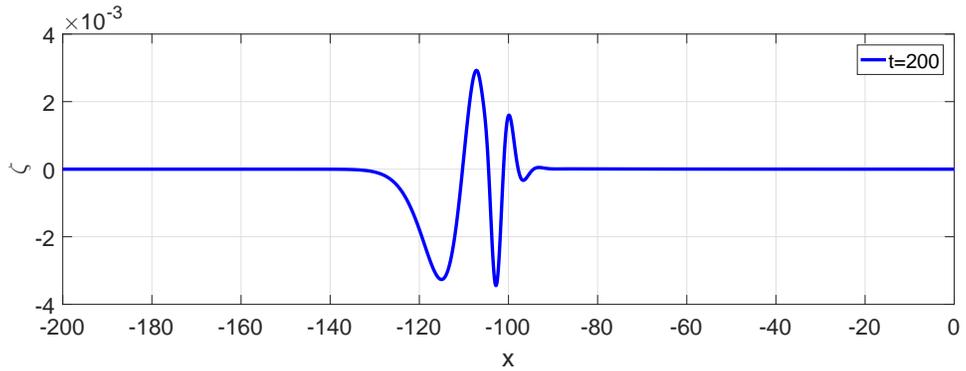}}
\caption{Evolution of a CSW perturbed by a small factor. Case (A3) with (\ref{52a}), (\ref{52b}) $A=1.1, c_{s}=c_{\gamma,\delta}+0.1.$ (a)-(c) $\zeta$ component of the numerical solution; (d) Magnification of (c) between $x=-200$ and $x=0$.}
\label{fdds5_1}
\end{figure}
For $A=1.1$ the results are given in Figures \ref{fdds5_1}-\ref{fdds5_3}. Figure \ref{fdds5_1} shows the evolution of the numerical approximation at several time instances. In this small perturbation case, a new CSW emerges, with some small-amplitude tail of dispersive nature following the main wave and shown in detail in Figure \ref{fdds5_1}(d).

The structure of these tails emerges when we analyze the behaviour of small- amplitude solutions of the linearized system associated to (\ref{BB2}) in a frame $y=x-c_{s}t$ moving with the speed $c_{s}$ of the CSW, given by
\begin{eqnarray}
(1-b\partial_{yy})(\partial_{t}-c_{s}\partial_{y})\zeta+\kappa_{1}\partial_{y} J_{a}v_{\beta}=0,\label{53}&&\\
(1-d\partial_{yy})(\partial_{t}-c_{s}\partial_{y}) v_{\beta}+\kappa_{2}\partial_{y}J_{c}\zeta=0,&&\label{54}
\end{eqnarray}
where $J_{a}, J_{c}$ are given by (\ref{49d}). 
Applying the operator $(1-d\partial_{yy})(\partial_{t}-c_{s}\partial_{y})$ to (\ref{53}) and using (\ref{54}) we obtain the high-order wave equation
\begin{eqnarray}
(1-d\partial_{yy})(1-b\partial_{yy})(\partial_{t}-c_{s}\partial_{y})^{2}\zeta-\kappa_{1}\kappa_{2}\partial_{y}^{2} J_{a}J_{c}\zeta=0.\label{55}
\end{eqnarray}
Plane wave solutions $\zeta(y,t)=e^{i(ky-\omega(k)t)}$ of (\ref{55}) satisfy the linear dispersion relation
\begin{eqnarray*}
\omega(k)=\omega_{\pm}(k)=-kc_{s}\pm c_{\gamma,\delta}k\phi(k^{2}),
\end{eqnarray*}
where $\phi:[0,\infty)\rightarrow\mathbb{R}$ is the function
\begin{eqnarray*}
\phi(x)=\sqrt{\frac{(1-\widetilde{a}x)(1-cx)}{(1+bx)(1+dx)}},\; \widetilde{a}=\frac{a}{\kappa_{1}}.
\end{eqnarray*}
This leads to a local phase speed (relative to the speed of the CSW)
\begin{eqnarray}
v_{\pm}(k)=-c_{s}\pm c_{\gamma,\delta}\phi(k^{2}).\label{514b}
\end{eqnarray}
Some properties of the function $\phi$ can explain the behaviour of (\ref{514b}). These are:
\begin{itemize}
\item[(1)] $\phi(x)\geq 0, x\geq 0, \; \phi(0)=1$.
\item[(2)] From the condition $\kappa_{1}bd-ac>0$, characterizing (A3), we have
$$\lim_{x\rightarrow\infty}\phi(x)=\phi_{*}:=\sqrt{\frac{ac}{\kappa_{1}bd}}<1.$$
\item[(3)] Let 
\begin{eqnarray}
p_{1}&=&\widetilde{a}c(b+d)+bd(\widetilde{a}+c),\nonumber\\
 p_{2}&=&2(\widetilde{a}c-bd)<0 \; ({\rm cf. \; (A3)}),\nonumber\\
p_{3}&=&\frac{1+\gamma\delta}{3\delta(\gamma+\delta)}>0.\label{514c}
\end{eqnarray}
Then
\begin{eqnarray}
\phi^{\prime}(x)=\frac{P(x)}{2\phi(x)(1+bx)^{2}(1+dx)^{2}},\; P(x)=p_{1}x^{2}+p_{2}x-p_{3}.\label{514d}
\end{eqnarray}
Therefore, the monotonicity of $\phi$ depends on the polynomial $P(x)$, in particular on the sign of $p_{1}$. Let $\Delta=p_{2}^{2}+4p_{1}p_{3}$ be the discriminant of the equation $P(x)=0$. Then we have the cases:
\begin{itemize}
\item[(i)] $p_{1}<0$. In this case $\Delta\leq 0$. Then $P(x)$ has the same sign for all $x\geq 0$. Since $P(0)=-p_{3}<0$, then $P(x)<0, x\geq 0$. Therefore, $\phi(x)$ is decreasing for $x\geq 0$. On the other hand, when $\Delta=0$ then $P(x)$ has a double root at $x=x_{*}=-p_{2}/2>0$, where $P(x)$ attains a minimum. Thus:
\begin{enumerate}
\item If $0\leq x\leq x_{*}$, then $\phi(x)$ is decreasing with $1>\phi(x)\geq \phi(x_{*})$.
\item If $x_{*}\leq x$, then $\phi(x)$ is increasing with $\phi(x_{*})\leq \phi(x)\leq \phi_{*}<1$.
\end{enumerate}
\item[(ii)] $p_{1}=0$. In this case $P(x)=p_{2}x-p_{3}$ has only the root $x=p_{3}/p_{2}<0$ and therefore $\phi(x)$ is decreasing for $x>0$ with $0\leq \phi(x)\leq \phi(0)=1$.
\item[(iii)] $p_{1}>0$. This implies $\Delta>0$ and the existence of two simple roots of $P(x)$, one positive $x_{*}=\frac{1}{2}(-p_{2}+\sqrt{\Delta})$ and one negative. The behaviour is then similar to the case $\Delta=0$ in (i), that is
\begin{enumerate}
\item If $0\leq x\leq x_{*}$, then $\phi(x)$ is decreasing with $1>\phi(x)\geq \phi(x_{*})$.
\item If $x_{*}\leq x$, then $\phi(x)$ is increasing with $\phi(x_{*})\leq \phi(x)\leq \phi_{*}<1$.
\end{enumerate}
\end{itemize}
The form of $\phi$ for the example at hand is given in Figure \ref{fdds5_2}.
\end{itemize}
\begin{figure}[htbp]
\centering
\subfigure[]
{\includegraphics[width=\columnwidth]{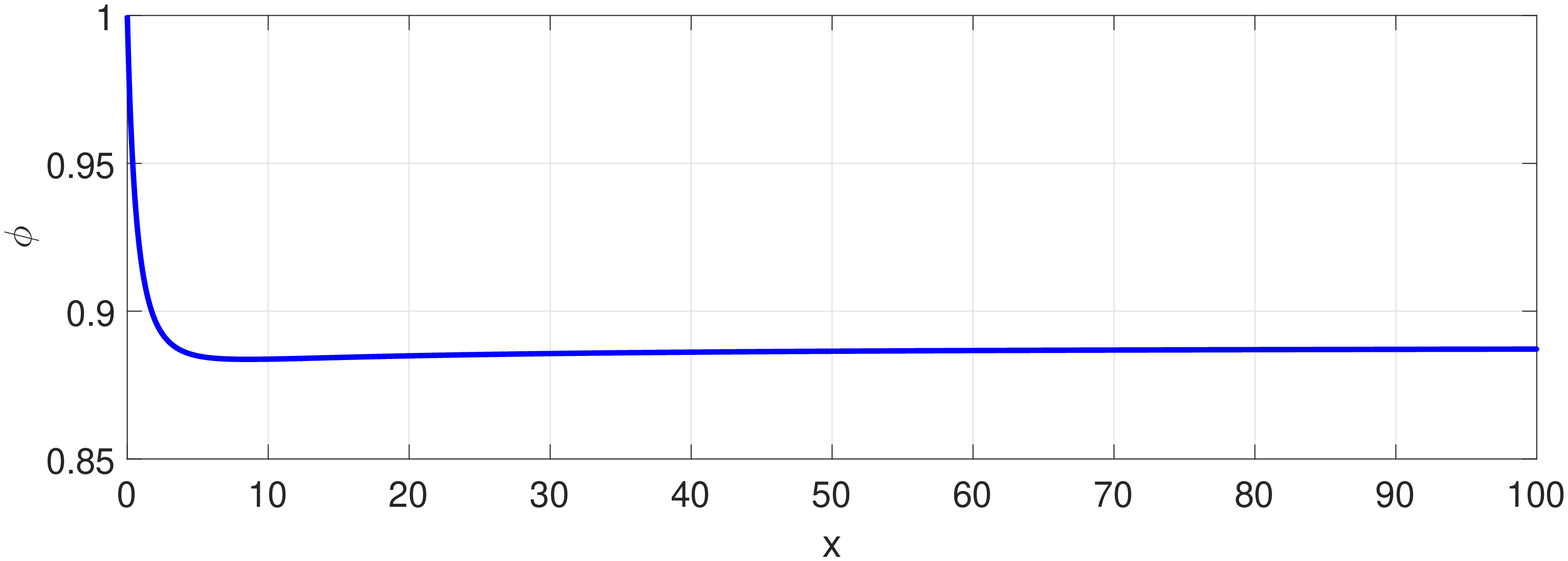}}
\subfigure[]
{\includegraphics[width=\columnwidth]{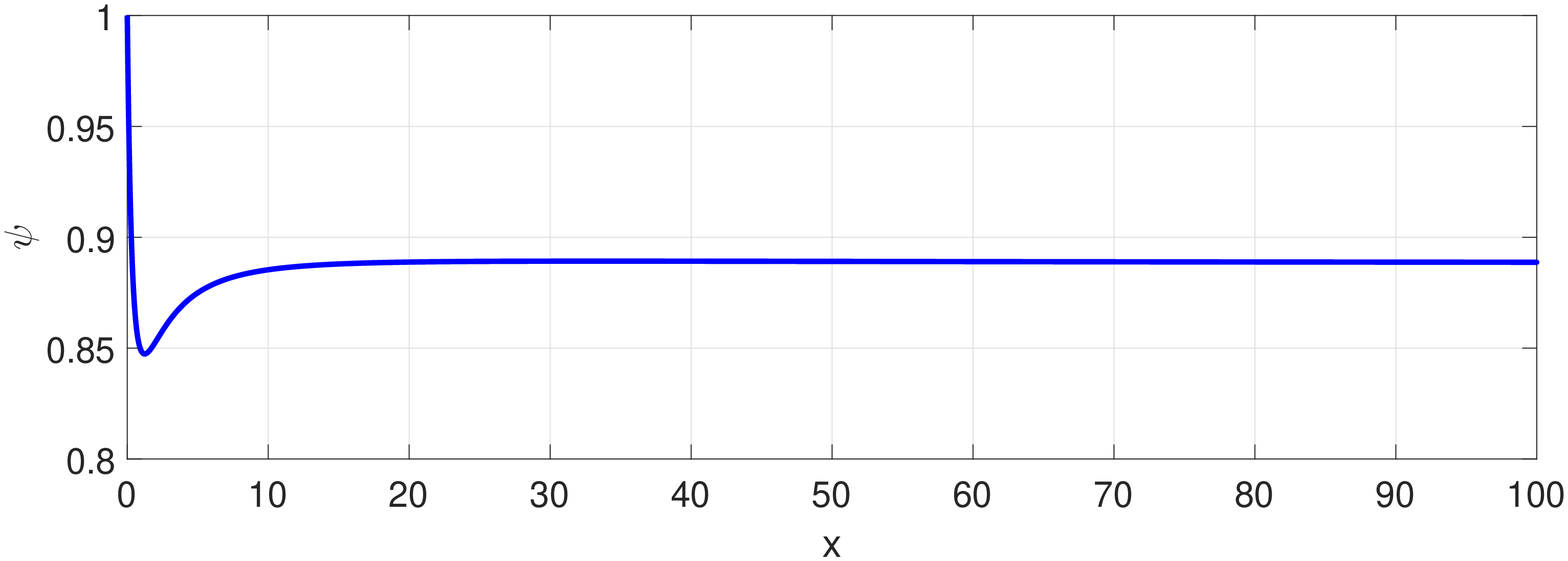}}
\caption{Form of the functions (a) $\phi(x)$ and (b) $\psi(x)$ for the values (\ref{52a}).}
\label{fdds5_2}
\end{figure}
With all these properties we have
\begin{eqnarray*}
-c_{s}-c_{\gamma,\delta}<v_{-}(k)<-c_{s}<v_{+}(k)<-c_{s}+c_{\gamma,\delta}.
\end{eqnarray*}
If we assume $c_{s}>c_{\gamma,\delta}$, then we conclude that plane wave components of the dispersive tail, traveling to the right or to the left, trail the solitary wave with an absolute phase speed satisfying $v_{+}(k)+c_{s}<c_{\gamma,\delta}$ and $|v_{-}(k)+c_{s}|<c_{\gamma,\delta}$. Furthermore, components with smaller $k$ (long wavelength) are faster than those with larger $k$ (short wavelength).

For the group velocities we have
\begin{eqnarray*}
\omega^{\prime}_{\pm}(k)=-c_{s}\pm c_{\gamma,\delta}\psi(k^{2}),
\end{eqnarray*}
where  $\psi:[0,\infty)\rightarrow\mathbb{R}$ is the function
\begin{eqnarray}
\psi(x)=2x\phi^{\prime}(x)+\phi(x),\label{514e}
\end{eqnarray}
Figure \ref{fdds5_2}(b) shows this function when the coefficients of the system are given by (\ref{52a}). The behaviour in this case seems to be similar to that of $\phi$. Different values of the parameters (always in the generic case and with the condition $\kappa_{1}bd-ac>0$) seem to suggest that this is a typical form, with the existence of a minimum, to the left of which $0\leq \psi(x)\leq 1$ and after which $\psi$ grows up to
$$\lim_{x\rightarrow\infty}\psi(x)=\phi_{*}:=\sqrt{\frac{ac}{\kappa_{1}bd}}<1.$$
Under these conditions, we would have 
\begin{eqnarray*}
-c_{s}-c_{\gamma,\delta}<\omega^{\prime}_{-}(k)<-c_{s}<\omega^{\prime}_{+}(k)<-c_{s}+c_{\gamma,\delta}.
\end{eqnarray*}
If we assume $c_{s}>c_{\gamma,\delta}$ again, this would imply the existence of two dispersive groups, one traveling to the left and one to the right following the solitary wave, with group velocity smaller than $c_{s}$ with $|\omega^{\prime}_{\pm}(k)+c_{s}|<c_{\gamma,\delta}$. In Figure  \ref{fdds5_1}(c) a first tail, close to the main wave is observed, while a second wave packet of smaller amplitude appears in the magnified 
Figure \ref{fdds5_1}(d). At $t=200$, this wavelet is traveling to the left with effective support in $[-130,-90]$ 
(cf.  \cite{DougalisDLM2007} for a similar wavelet in some cases of surface waves).
\begin{figure}[htbp]
\centering
\subfigure[]
{\includegraphics[width=\columnwidth]{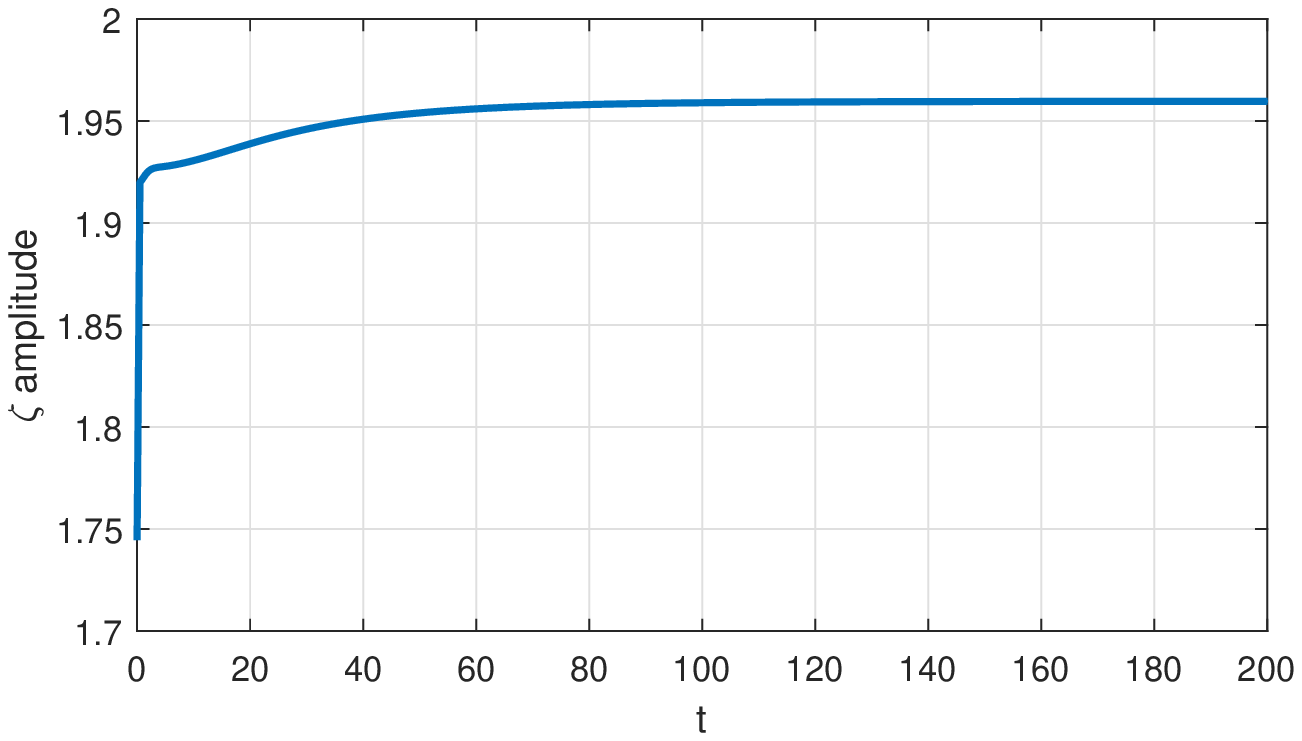}}
\subfigure[]
{\includegraphics[width=\columnwidth]{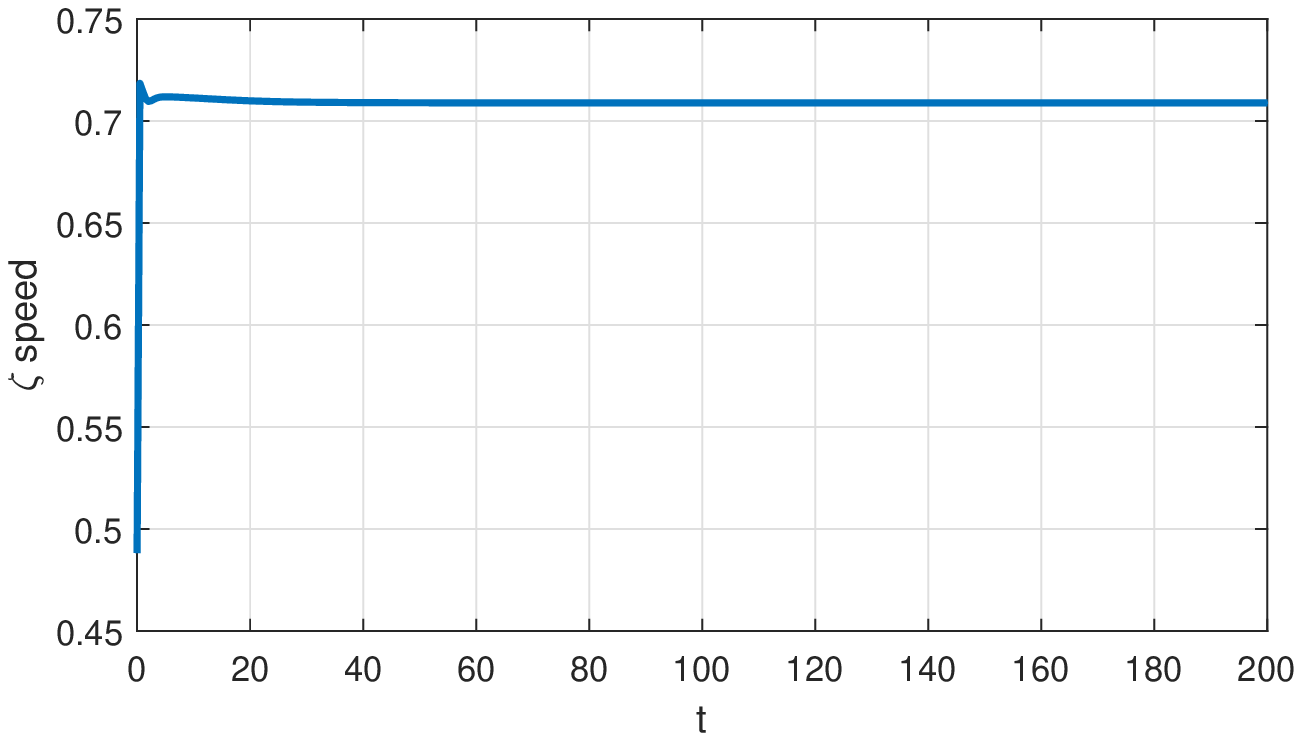}}
\caption{Evolution of a CSW perturbed by a small amount. Case (A3) with (\ref{52a}), (\ref{52b}) $A=1.1, c_{s}=c_{\gamma,\delta}+0.1$. Evolution of amplitude (a) and speed (b) of the $\zeta$ component of the main pulse.}
\label{fdds5_3}
\end{figure}

Finally, Figures \ref{fdds5_3}(a) and \ref{fdds5_3}(b) show, respectively, the evolution of the amplitude and speed of the $\zeta$ component of the numerical approximation of the main pulse for the system with parameters given by (\ref{52a}). In this example, the emerging solitary wave is faster and larger than the perturbed initial profile. 
\begin{remark}
A similar study can also be applied to the rest of the cases (A4)-(A6) of Table \ref{tavle0}. Thus, in the cases (A4) and (A5), we have $p_{1}<0$, while in the case (A6), $p_{1}=0$. Then, analogous results to those of the generic case (A3) of CSW's hold.
\end{remark}

\subsubsection{Larger perturbation of a CSW}
\label{sec522}
\begin{figure}[htbp]
\centering
\subfigure[]
{\includegraphics[width=\columnwidth]{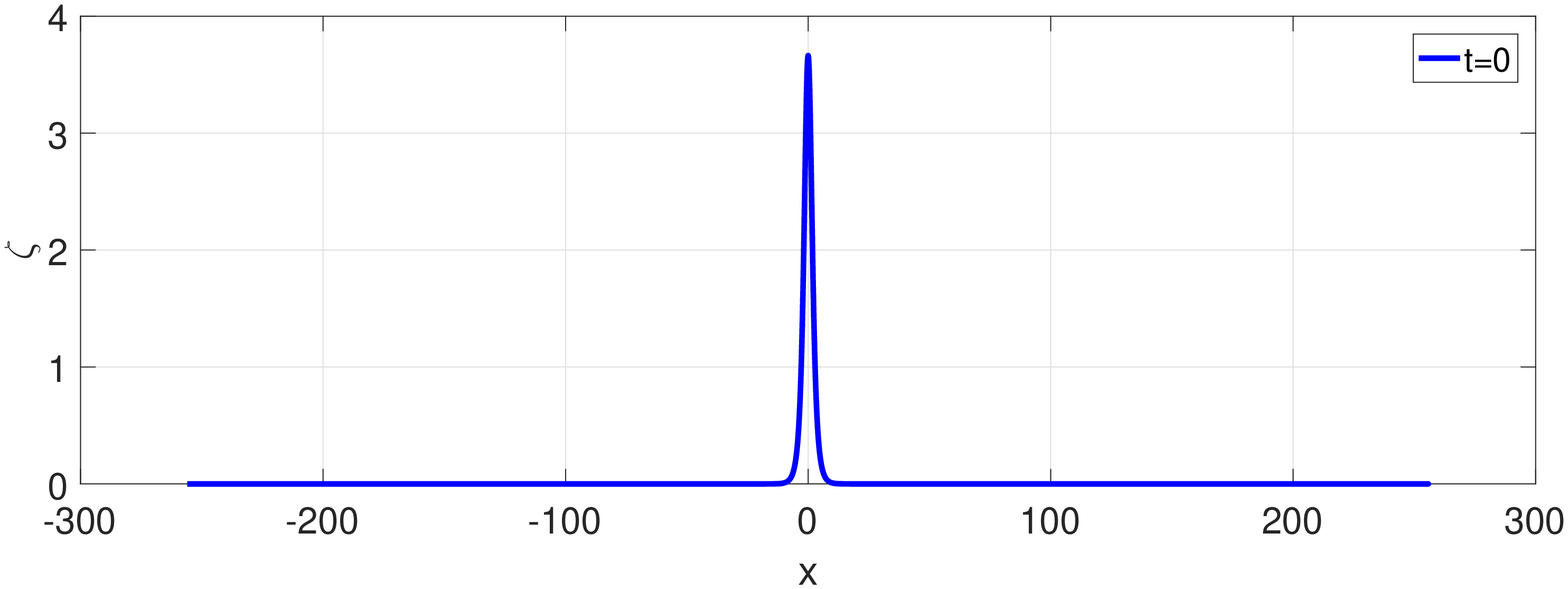}}
\subfigure[]
{\includegraphics[width=\columnwidth]{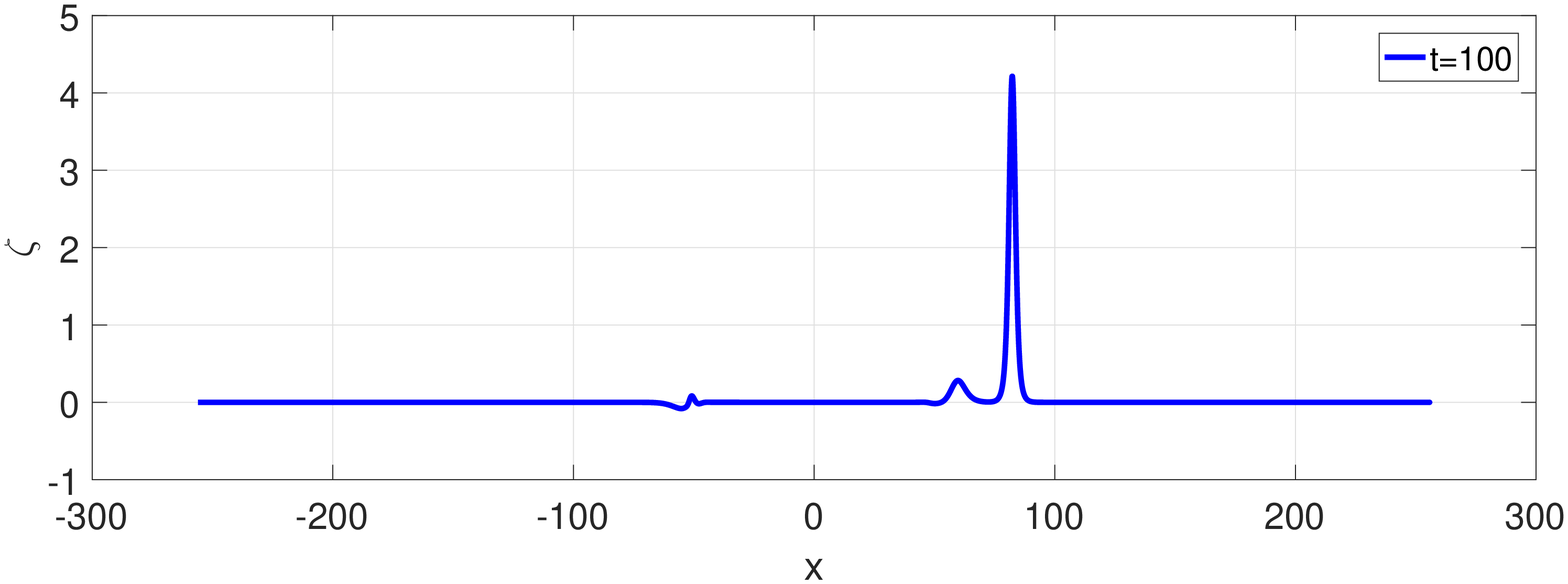}}
\subfigure[]
{\includegraphics[width=\columnwidth]{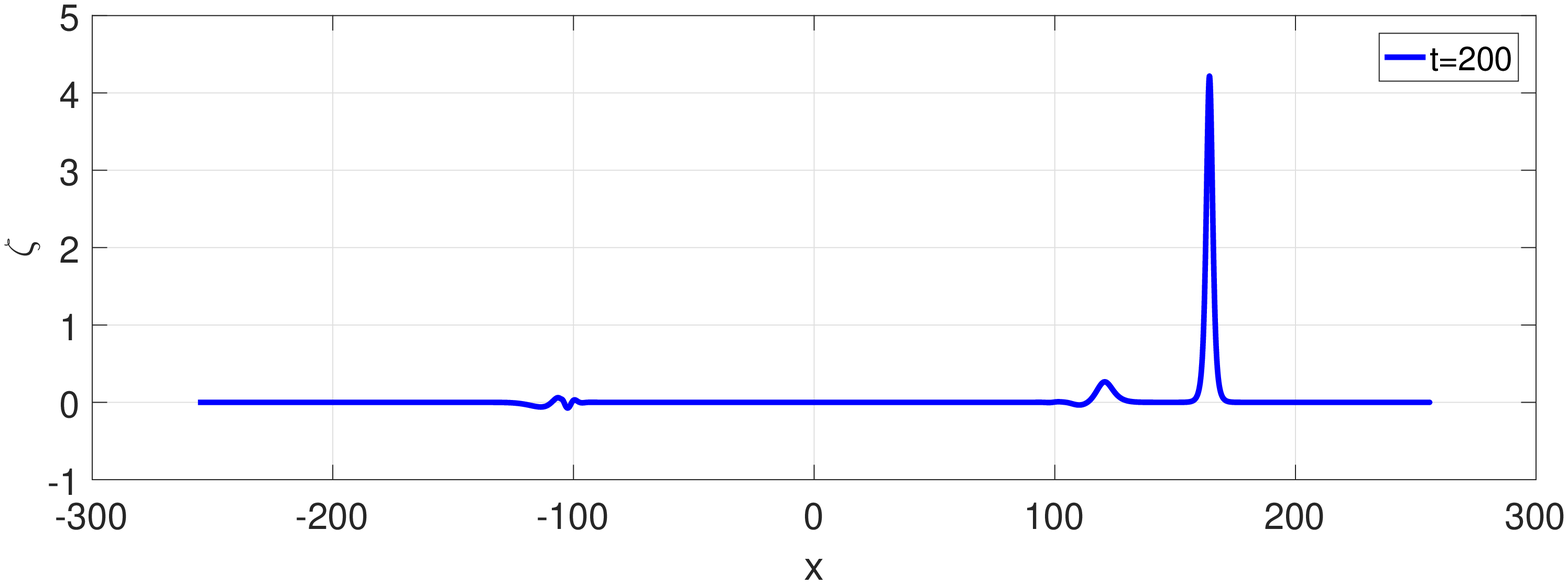}}
\caption{Evolution of a perturbed CSW. Case (A3) with (\ref{52a}), (\ref{52b}) $A=2.1, c_{s}=c_{\gamma,\delta}+0.1.$ (a)-(c) $\zeta$ component of the numerical solution.}
\label{fdds5_4}
\end{figure}
Increasing the value of the perturbation parameter $A$ leads, in our example, to the generation of another solitary wave; no iunstability was observed.
Figures \ref{fdds5_4}-\ref{fdds5_5} illustrate the experiment with $A=2.1$. The evolution of the numerical approximation (see Figure \ref{fdds5_4}) shows, in addition to the generation of a dominant, emerging solitary wave, (a  bit taller than that of the perturbed initial data, cf. Figure \ref{fdds5_5}), the formation of two other structures. 
\begin{figure}[htbp]
\centering
\subfigure[]
{\includegraphics[width=\columnwidth]{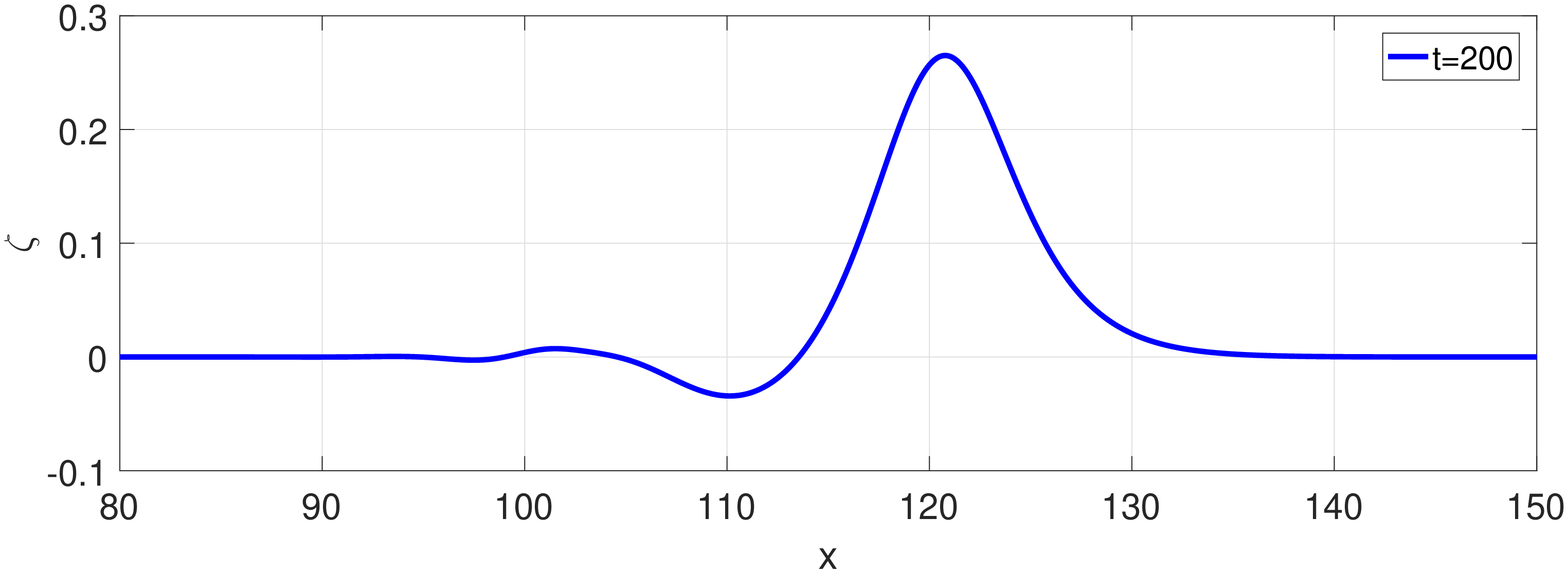}}
\subfigure[]
{\includegraphics[width=\columnwidth]{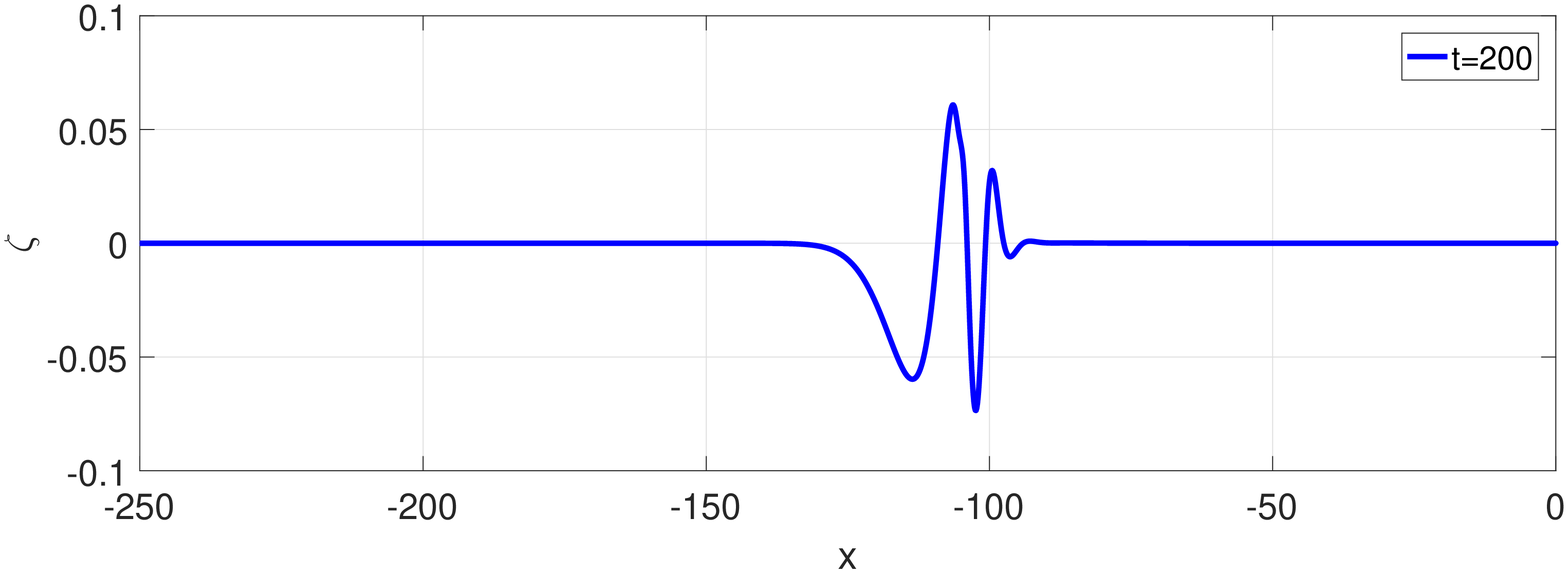}}
\caption{Evolution of a perturbed CSW. Case (A3) with (\ref{52a}), (\ref{52b}) $A=2.1, c_{s}=c_{\gamma,\delta}+0.1.$ (a) First structure of the $\zeta$ component of the numerical solution; (b) Second structure.}
\label{fdds5_6}
\end{figure}

Immediately behind the main wave, a nonlinear wave of solitary-wave type seems to be forming, along with a dispersive tail trailing it, as Figure \ref{fdds5_6}(a) reveals. The nature of the second structure on the left (see the magnification in Figure \ref{fdds5_6}(b)) is not so clear. Apparently, the wave should contain some dispersive component (as in the experiment with small perturbations in the previous section) but its persistence during the evolution may suggest the existence of something that does not disperse strongly. Similar conclusions hold from the experiments (not shown here) with $A=3.1$. This behaviour also suggests, as the structure seems to be moving to the left, the formation of a wavelet, or a superposition of them, such as observed in Boussinesq systems for surfaces waves, see e.~g. \cite{ADM1}. Nevertheless, the possibility of the formation of a CSW with nonmonotone decay is not to be discarded.
\begin{figure}[htbp]
\centering
\subfigure[]
{\includegraphics[width=\columnwidth]{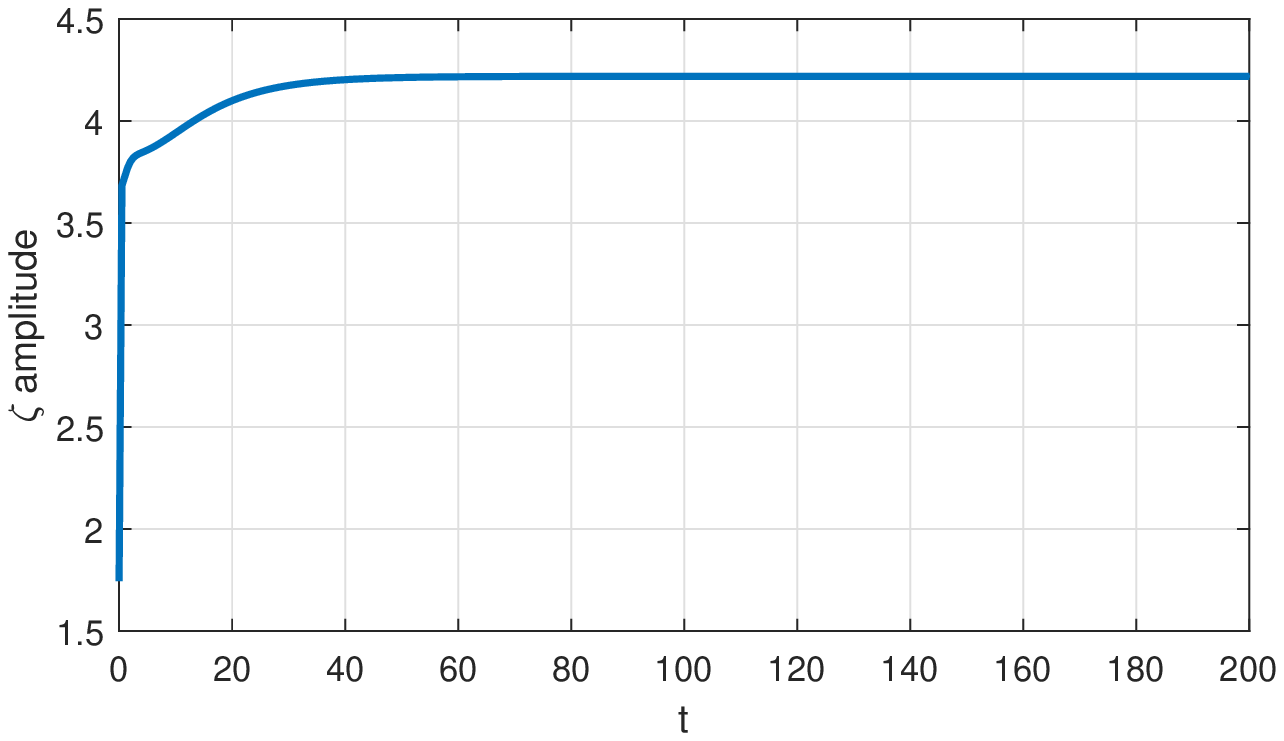}}
\subfigure[]
{\includegraphics[width=\columnwidth]{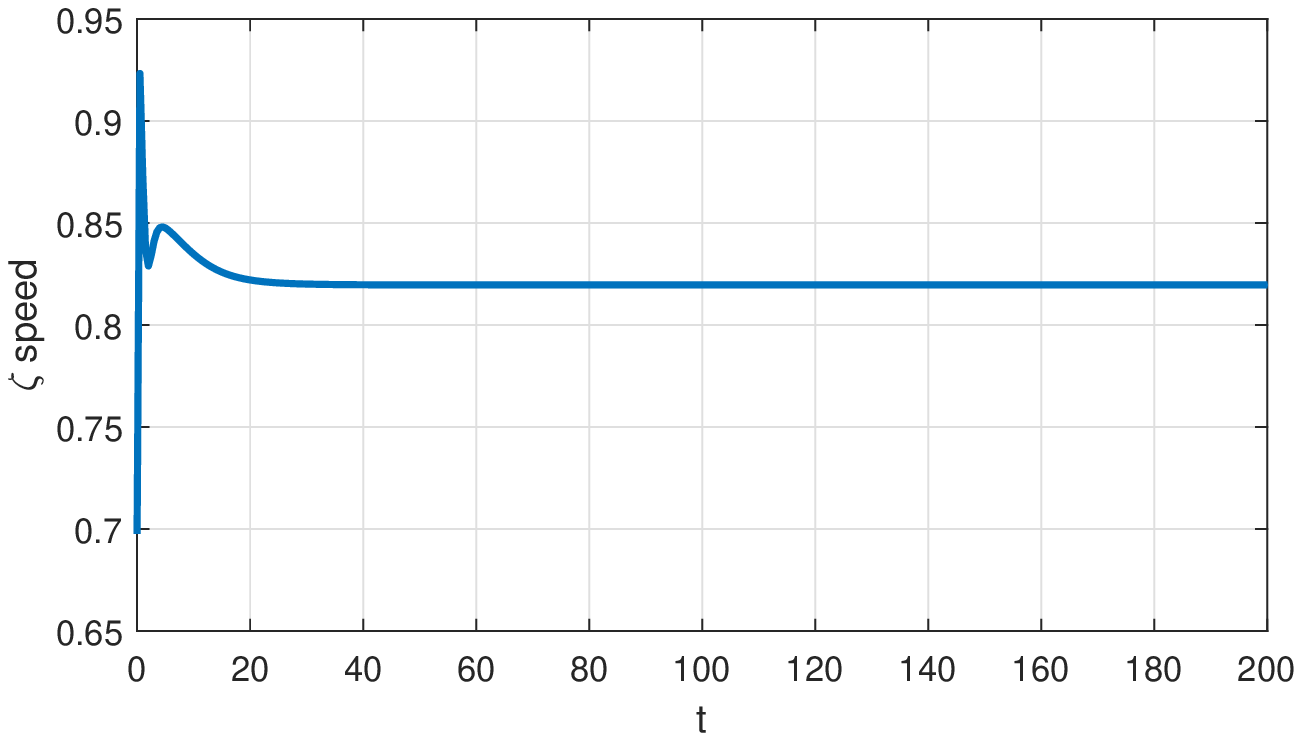}}
\caption{Evolution of a perturbed CSW. Case (A3) with (\ref{52a}), (\ref{52b}) $A=2.1, c_{s}=c_{\gamma,\delta}+0.1.$. Evolution of amplitude (a) and speed (b) of the $\zeta$ component of the numerical solution.}
\label{fdds5_5}
\end{figure}

These observations persist when the value of $A$ is taken to be larger. In order to illustrate this, we show the analogous experiment corresponding to $A=6.1$, in Figures \ref{fdds5_7}-\ref{fdds5_8}. For this larger perturbation, the initial perturbed wave
gives rise to a main solitary-wave pulse
and seems to exhibit a sort of stronger resolution property, with a solitary wave of elevation forming behind the main wave, see Figure \ref{fdds5_8}(a). As for the second structure, the possible formation of a CSW of depression with nonmonotone decay is more clearly observed (Figure \ref{fdds5_8}(b)) than in the previous experiment. 

\begin{figure}[htbp]
\centering
\subfigure[]
{\includegraphics[width=\columnwidth]{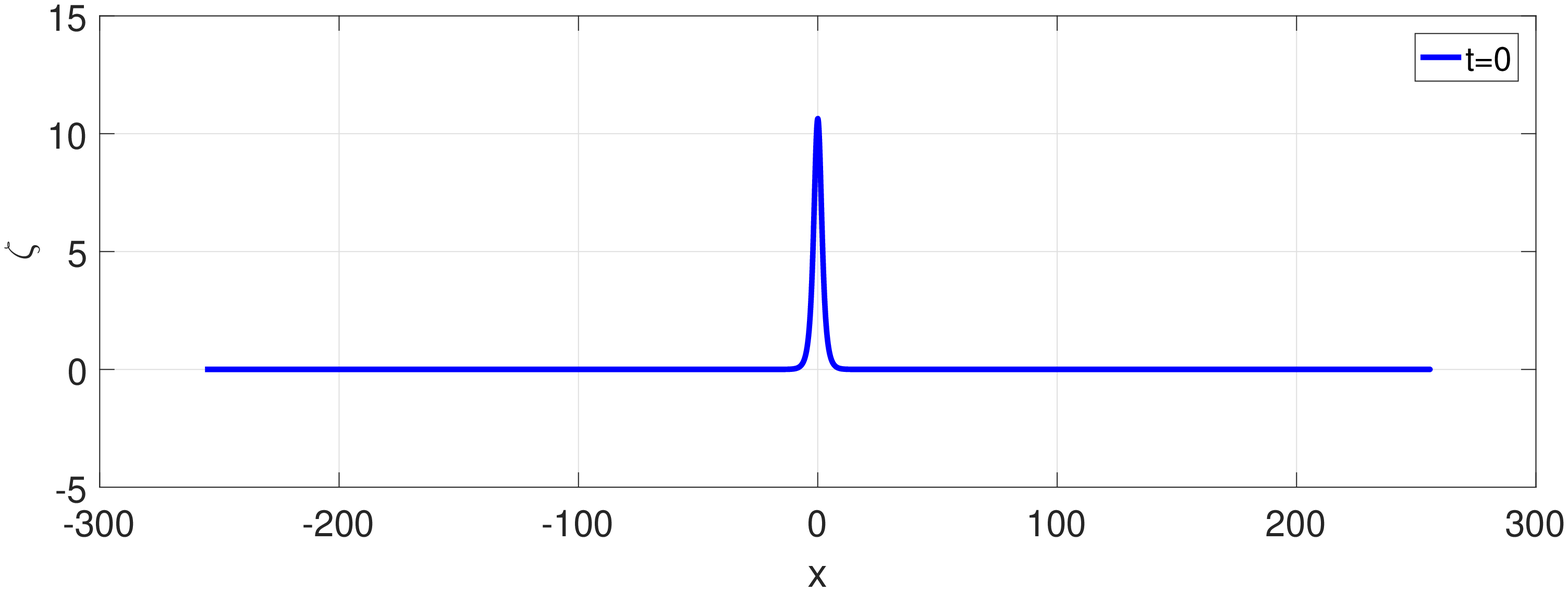}}
\subfigure[]
{\includegraphics[width=\columnwidth]{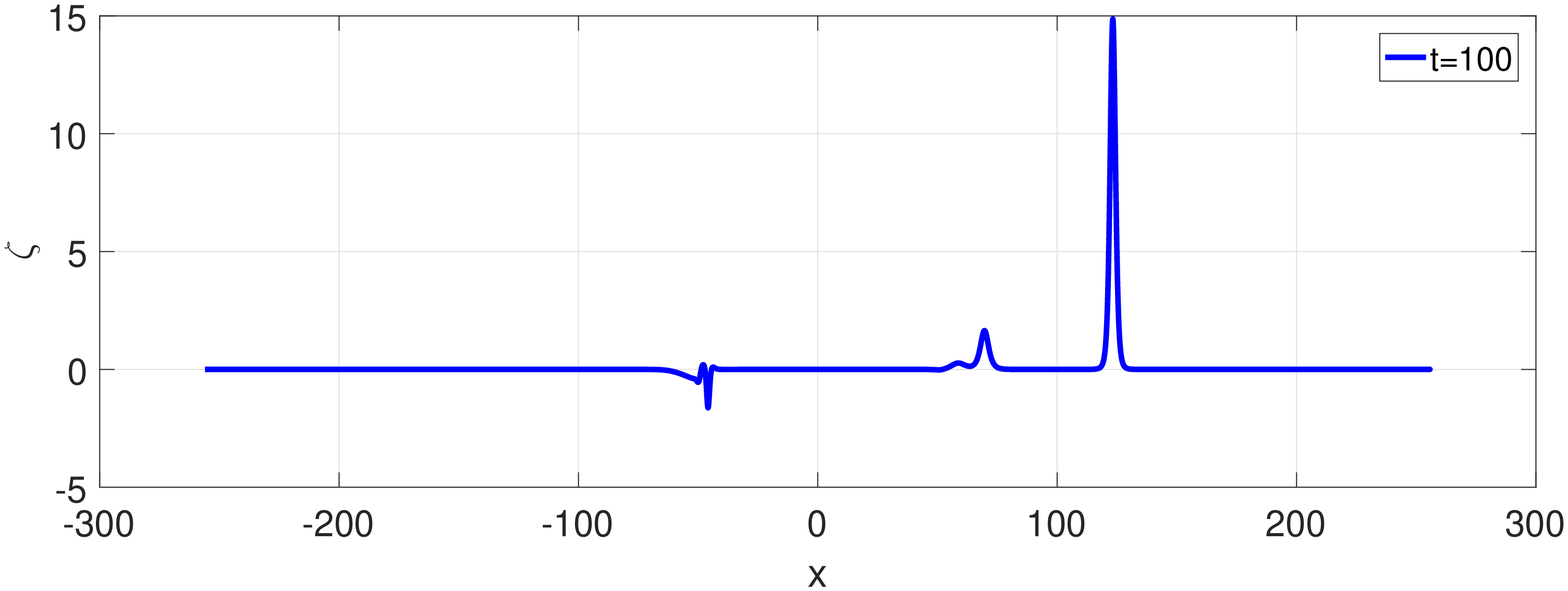}}
\subfigure[]
{\includegraphics[width=\columnwidth]{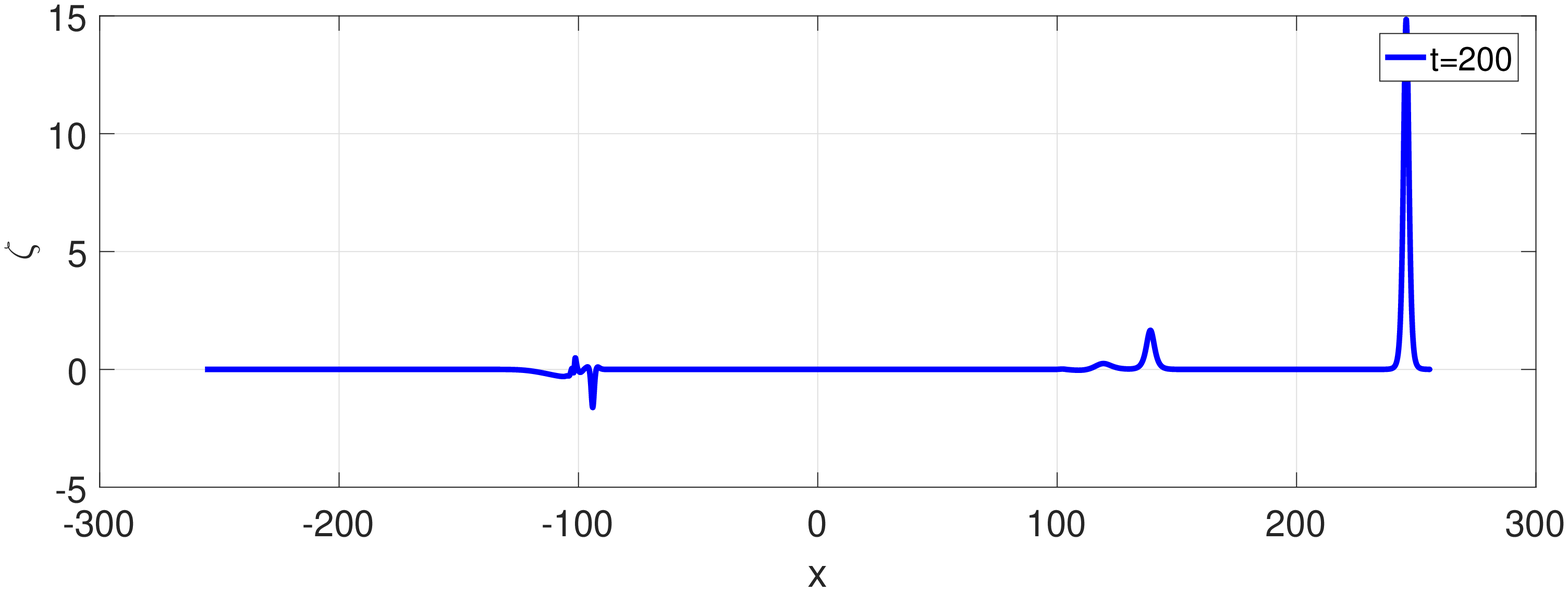}}
\caption{Evolution of a perturbed CSW. Case (A3) with (\ref{52a}), (\ref{52b}) $A=6.1, c_{s}=c_{\gamma,\delta}+0.1.$ (a)-(c) $\zeta$ component of the numerical solution.}
\label{fdds5_7}
\end{figure}

\begin{figure}[htbp]
\centering
\subfigure[]
{\includegraphics[width=\columnwidth]{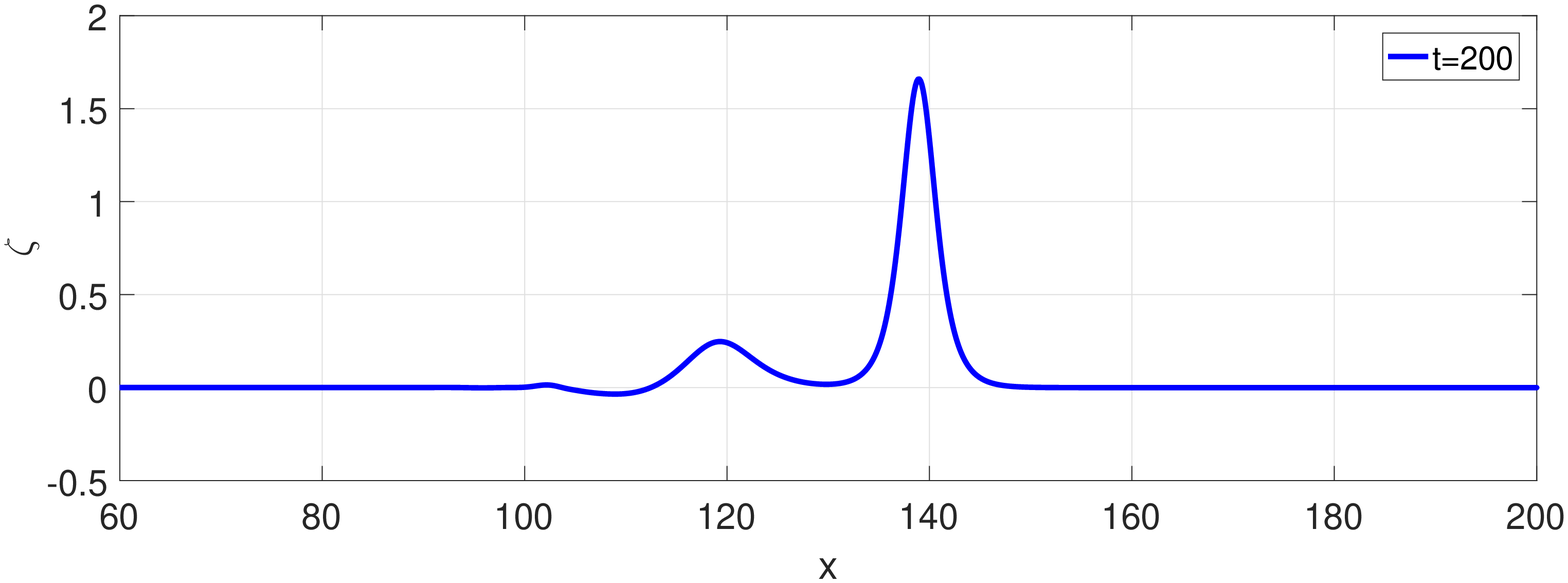}}
\subfigure[]
{\includegraphics[width=\columnwidth]{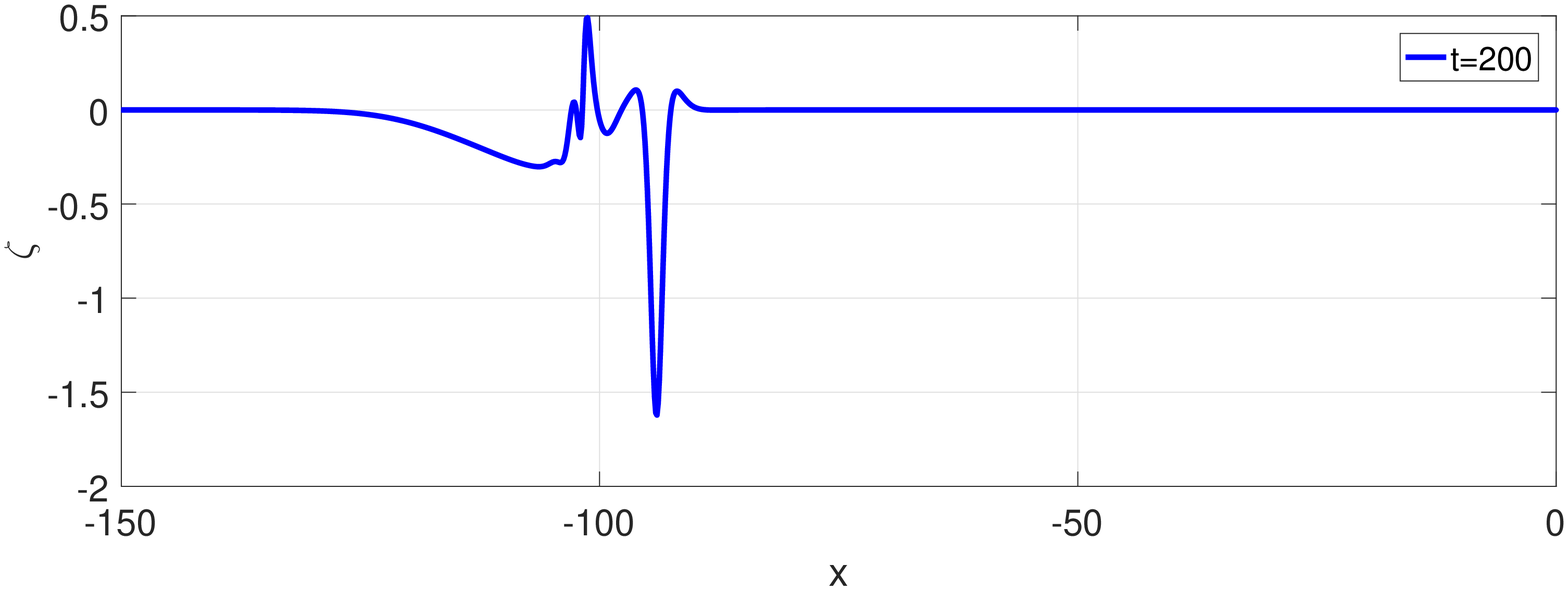}}
\caption{Evolution of a perturbed CSW. Case (A3) with (\ref{52a}), (\ref{52b}) $A=6.1, c_{s}=c_{\gamma,\delta}+0.1.$ (a) First structure of the $\zeta$ component of the numerical solution; (b) Second structure.}
\label{fdds5_8}
\end{figure}

\subsubsection{Overtaking collisions}
Overtaking collisions are illustrated in the following experiment. With the same values of the parameters given by (\ref{52a}), we generate two approximate CSW profiles with speed $c_{s}^{(1)}=c_{\gamma,\delta}+0.5\approx 1.0976$ centered at $x_{0}^{(1)}=0$, and speed $c_{s}^{(2)}=c_{\gamma,\delta}+0.2\approx 0.7976$ centered at $x_{0}^{(2)}=20$, respectively. The superposition of these profiles is taken as initial condition for the numerical method. The ensuing evolution is shown in Figure \ref{fdds5_9}, while Figure \ref{fdds5_9m} shows a magnified version. 

\begin{figure}[htbp]
\centering
\subfigure[]
{\includegraphics[width=\columnwidth]{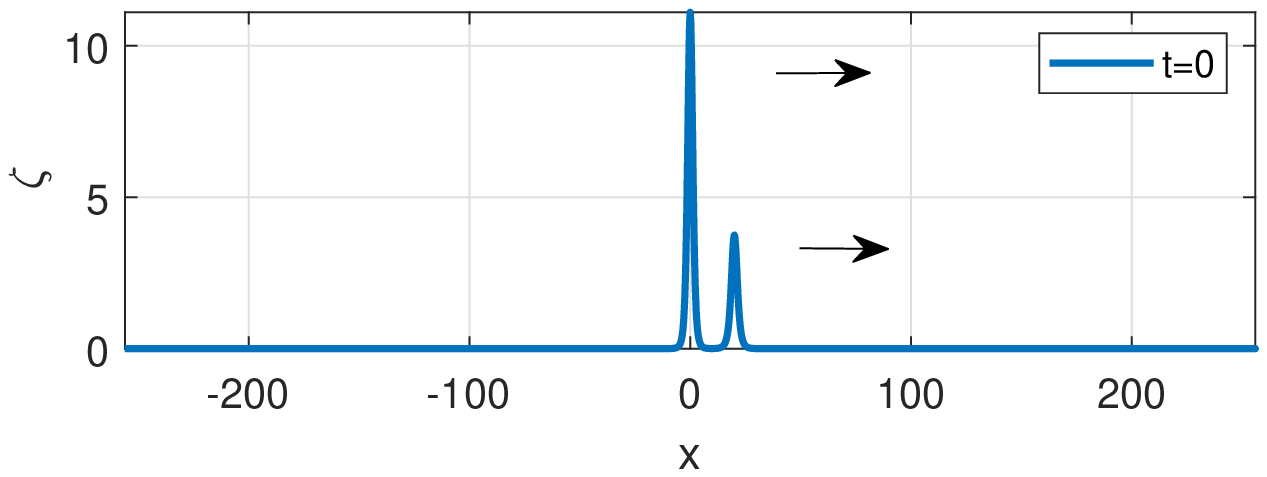}}
%\subfigure[]
%{\includegraphics[width=\columnwidth]{overcol_t200.eps}}
\subfigure[]
{\includegraphics[width=\columnwidth]{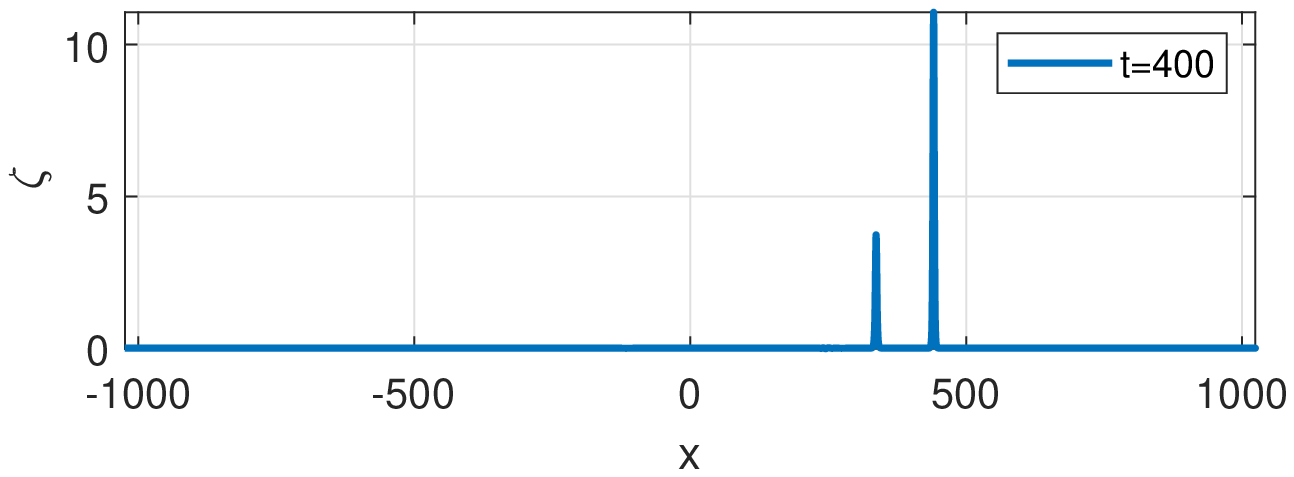}}
%\subfigure[]
%{\includegraphics[width=\columnwidth]{overcol_t600.eps}}
\subfigure[]
{\includegraphics[width=\columnwidth]{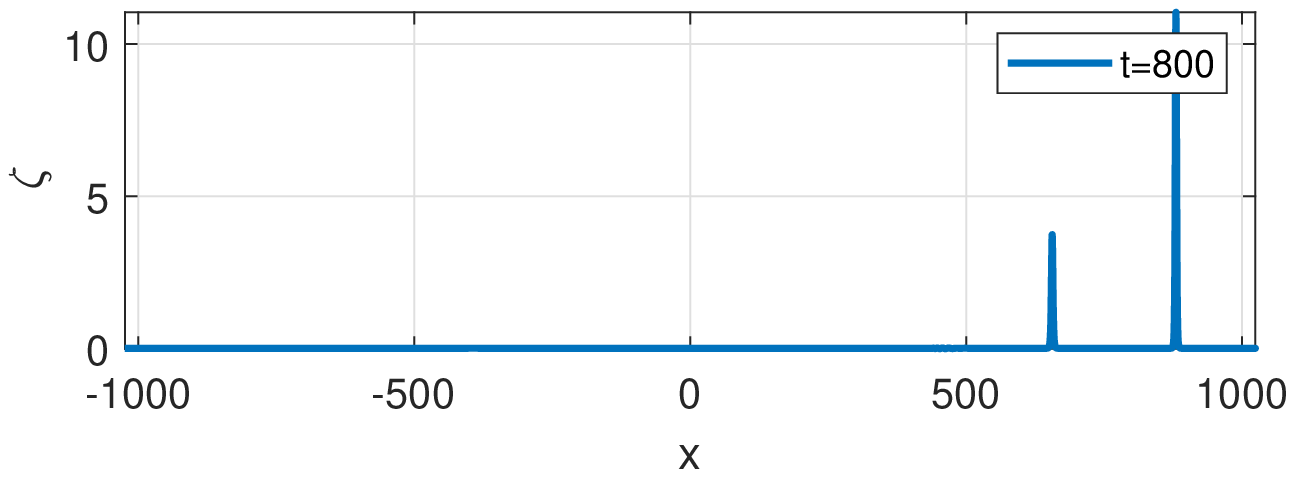}}
%\subfigure[]
%{\includegraphics[width=\columnwidth]{csw4_4.eps}}
\caption{Overtaking collision of CSW's.  (a)-(c) $\zeta$ component of the numerical approximation.}
\label{fdds5_9}
\end{figure}

The experiment shows that after the one-way collision, two solitary waves emerge. The amplitude of the larger one, compared to that of the corresponding wave before the collision, has decreased slightly; the relative difference is of $9.7\times 10^{-6}$. In the case of the second, smaller solitary wave, the comparison shows an increase of the amplitude after the collision, which in relative terms is about $2.0\times 10^{-6}$. The effect in the corresponding speeds is qualitatively similar to the experimental  
speed-amplitude relation developed in section \ref{sec4}, as the taller wave reduces slightly its speed after the collision, while that of the shorter one is increasing. The evolution of the errors of amplitude and speed of the tall solitary wave is displayed in Figure \ref{fdds5_10}.

\begin{figure}[htbp]
\centering
\subfigure[]
{\includegraphics[width=\columnwidth]{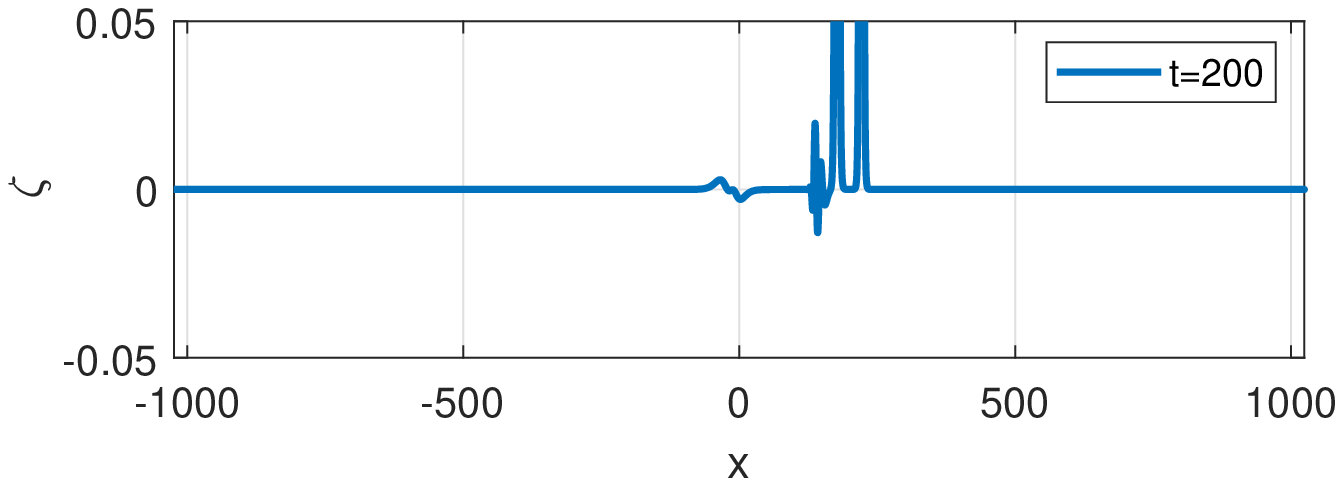}}
\subfigure[]
{\includegraphics[width=\columnwidth]{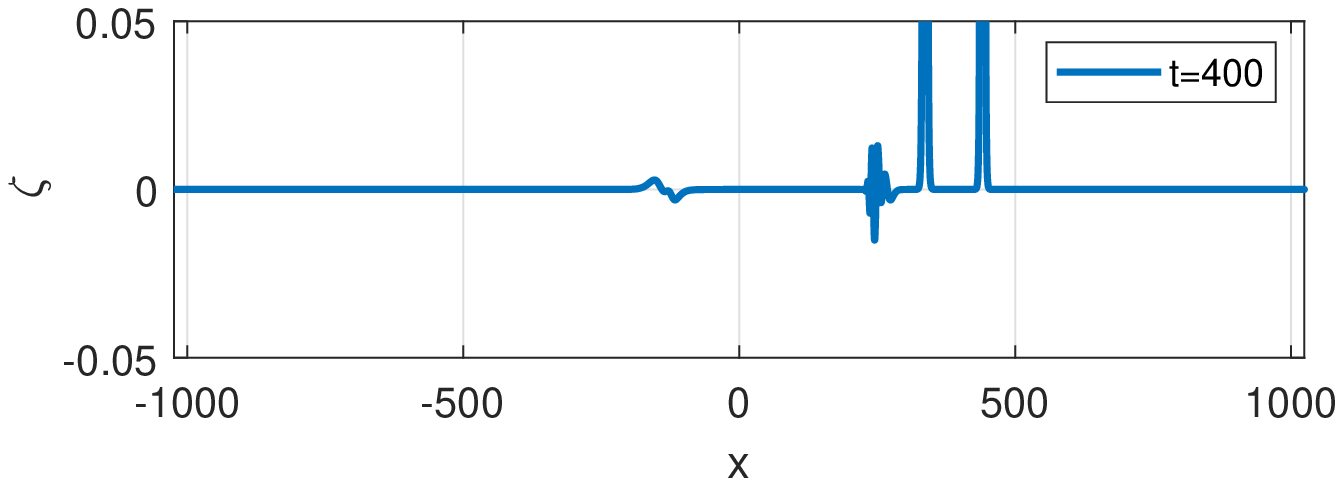}}
\subfigure[]
{\includegraphics[width=\columnwidth]{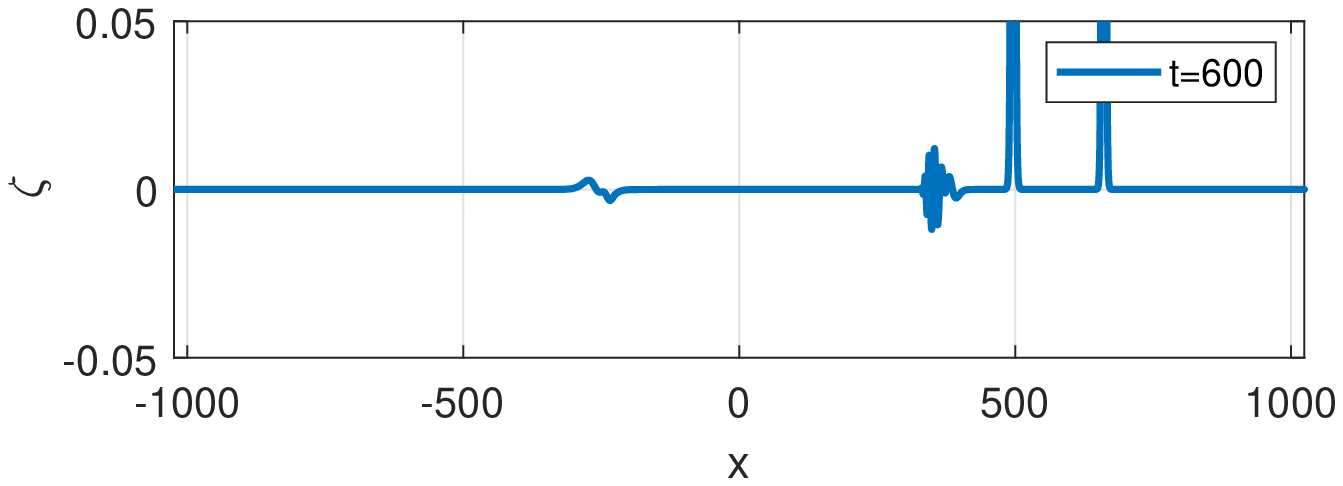}}
\subfigure[]
{\includegraphics[width=\columnwidth]{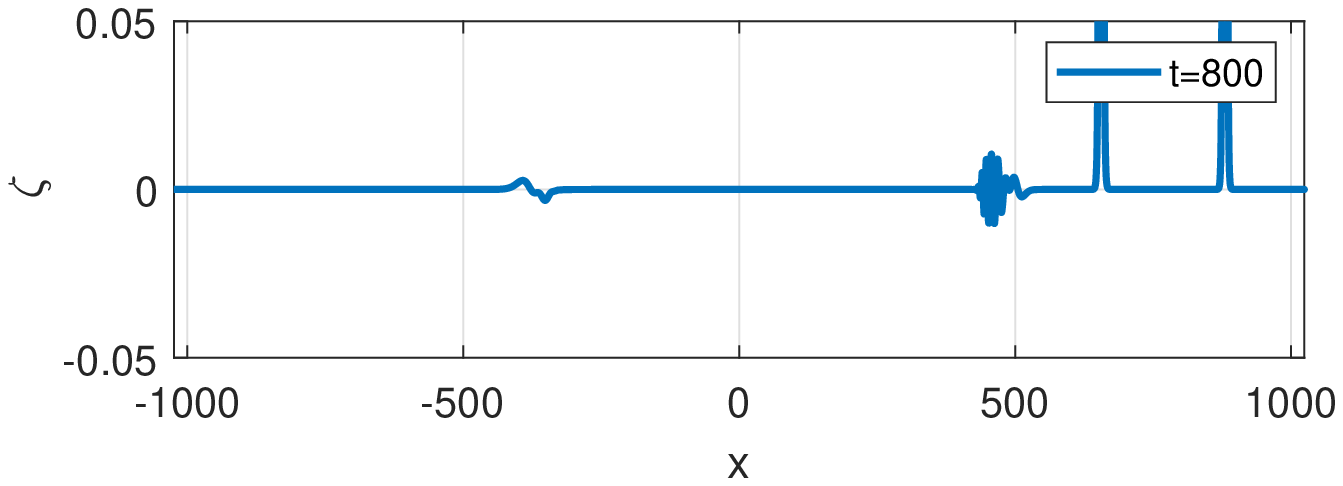}}
\caption{Overtaking collision of CSW's. Magnifications of the numerical solution of Figure \ref{fdds5_9}.}
\label{fdds5_9m}
\end{figure}

\begin{figure}[htbp]
\centering
\subfigure[]
{\includegraphics[width=\columnwidth]{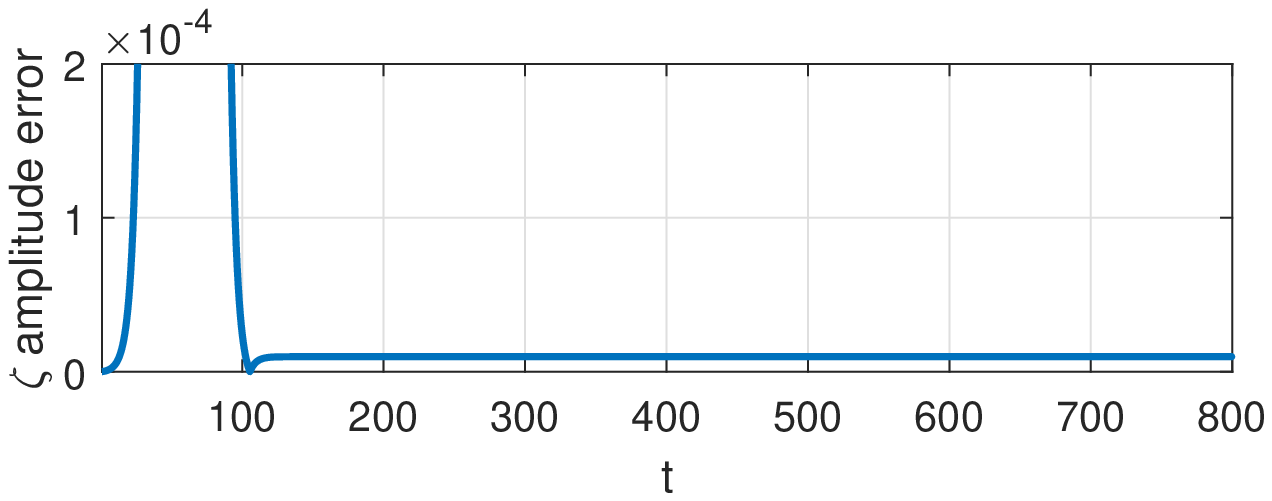}}
\subfigure[]
{\includegraphics[width=\columnwidth]{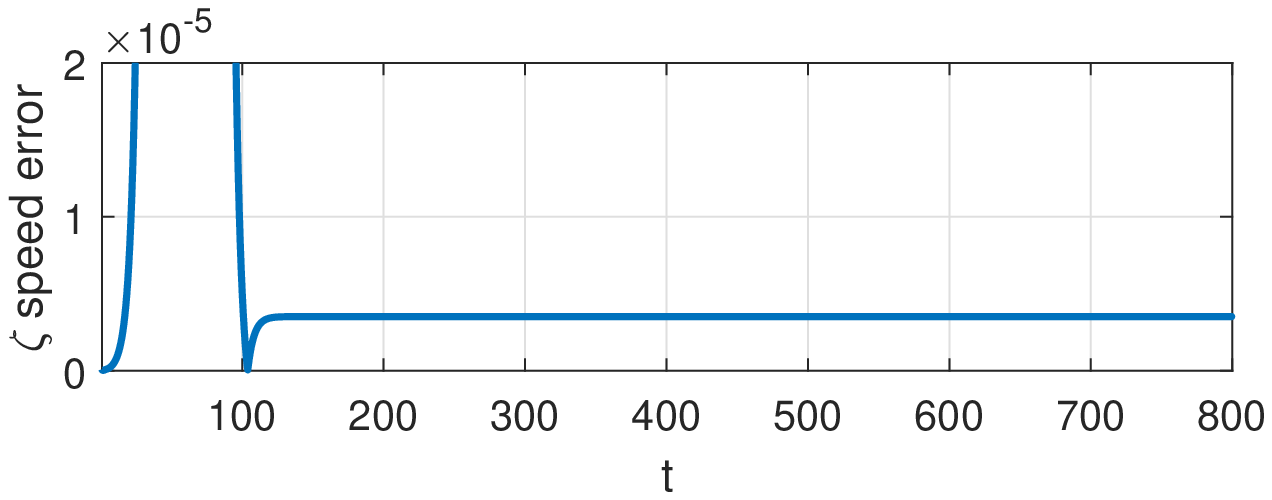}}
\caption{Overtaking collision of CSW's. Evolution of amplitude (a) and speed (b) errors of the tall emerging solitary wave ($\zeta$ component of the numerical solution).}
\label{fdds5_10}
\end{figure}

%\begin{figure}[htbp]
%\centering
%\subfigure[]
%{\includegraphics[width=\columnwidth]{overcol_t200m.eps}}
%\subfigure[]
%{\includegraphics[width=\columnwidth]{overcol_t400m.eps}}
%\subfigure[]
%{\includegraphics[width=\columnwidth]{overcol_t600m.eps}}
%\subfigure[]
%{\includegraphics[width=\columnwidth]{overcol_t800m.eps}}
%\caption{Overtaking collisions of CSW. Case (A3) with (\ref{52a}).  Magnifications of the numerical solution of Figure \ref{fdds5_9}.}
%\label{fdds5_9m}
%\end{figure}

Additional features of this inelastic interaction are shown in Figure \ref{fdds5_9m} and in magnification in Figure \ref{fdds5_9m2}. Behind the shorter emerging wave a small dispersive tail is generated (Figures \ref{fdds5_9m2}(b),(d),(f)), while a second, wavelet-type structure is observed to have formed  and to be traveling to the left (Figures \ref{fdds5_9m2}(a),(c),(e)).

\begin{figure}[htbp]
\centering
\subfigure[]
{\includegraphics[width=6.27cm,height=5.05cm]{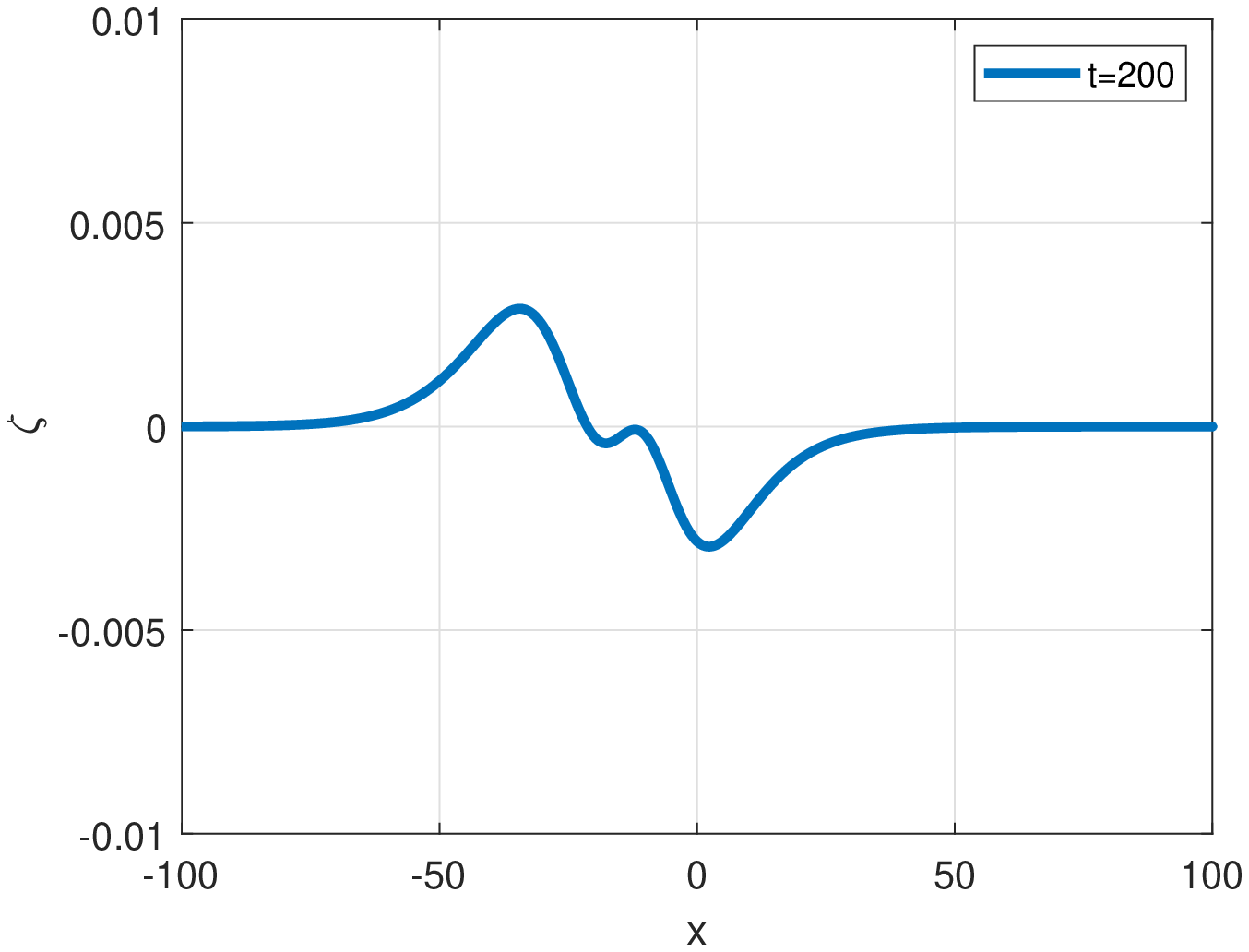}}
\subfigure[]
{\includegraphics[width=6.27cm,height=5.05cm]{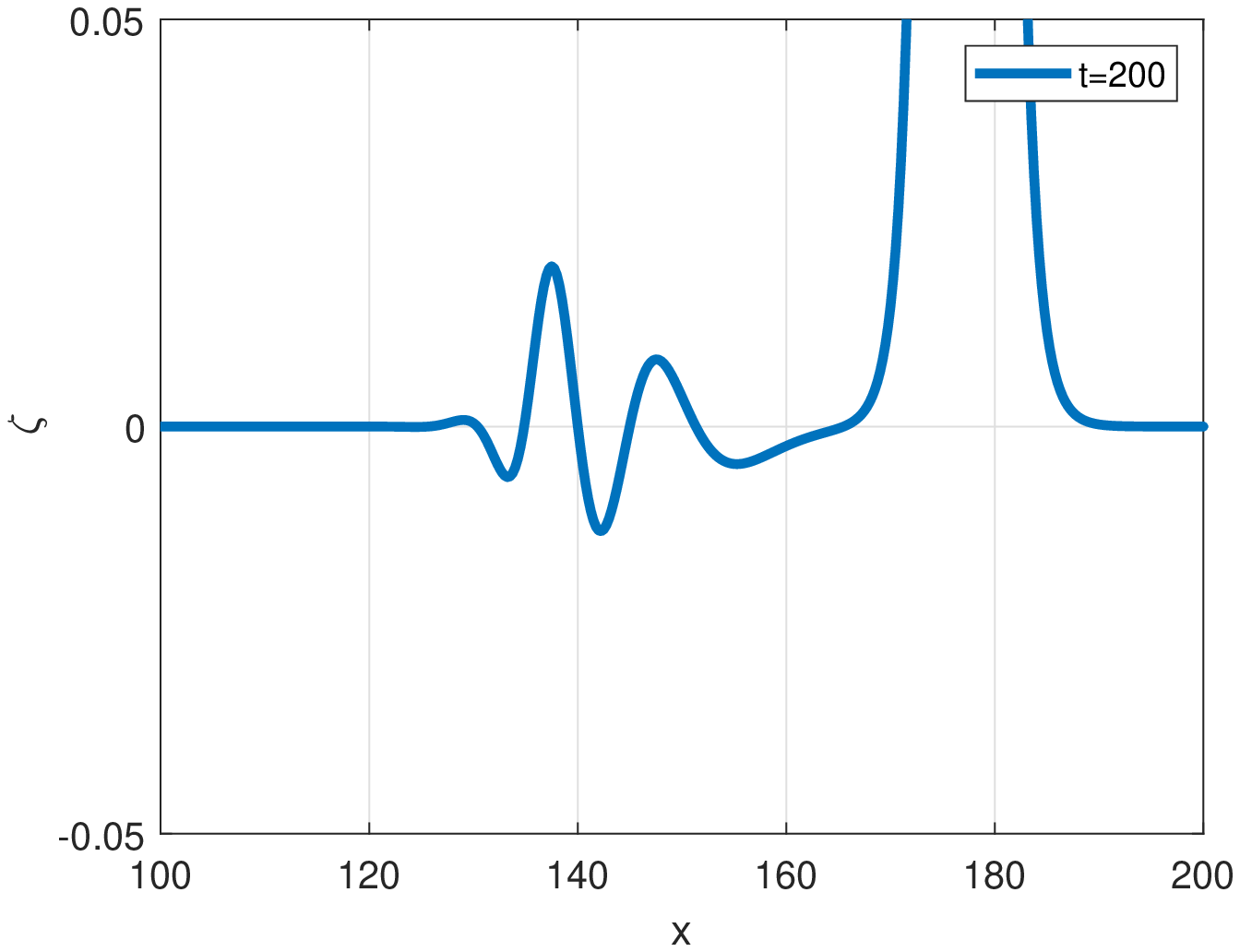}}
\subfigure[]
{\includegraphics[width=6.27cm,height=5.05cm]{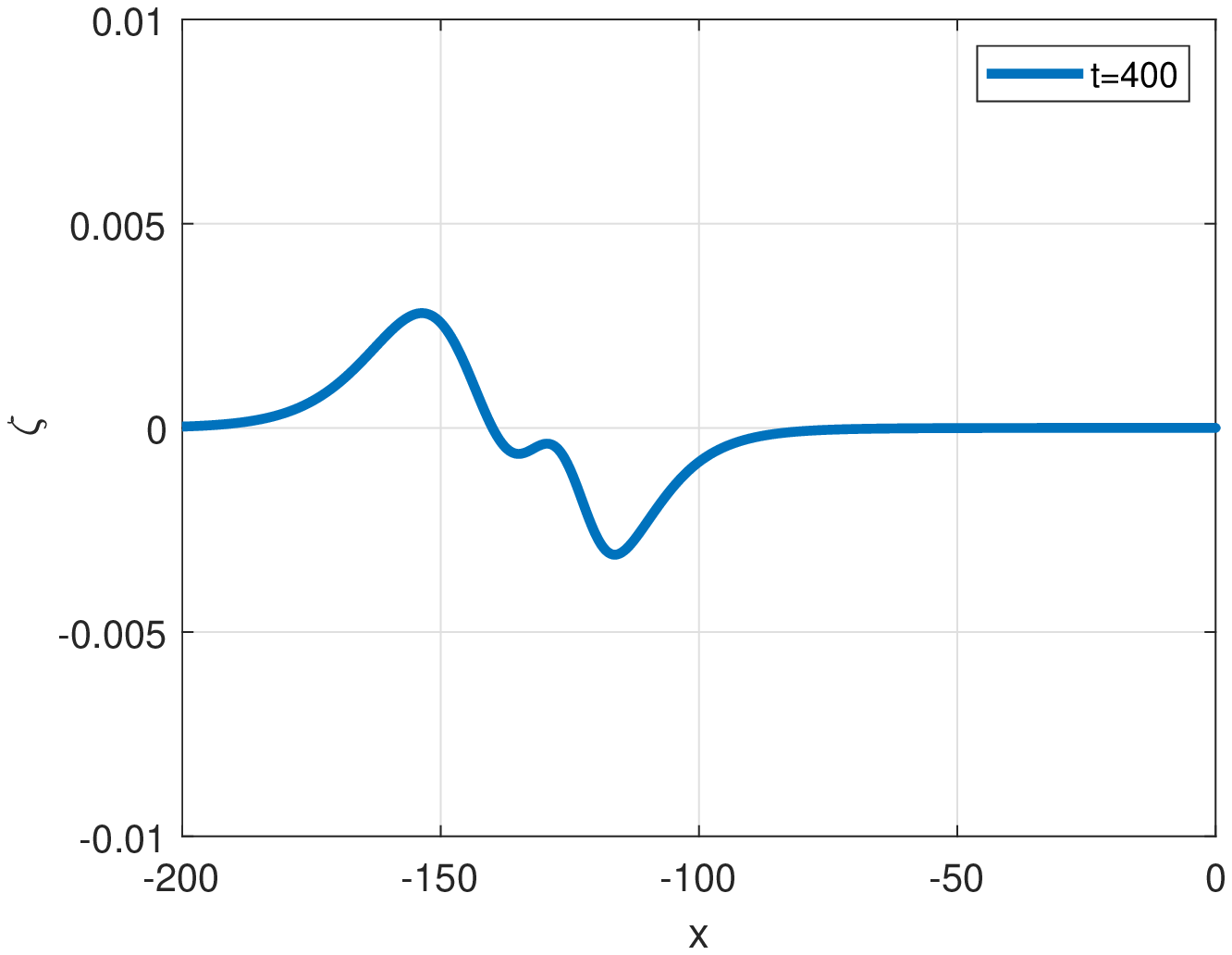}}
\subfigure[]
{\includegraphics[width=6.27cm,height=5.05cm]{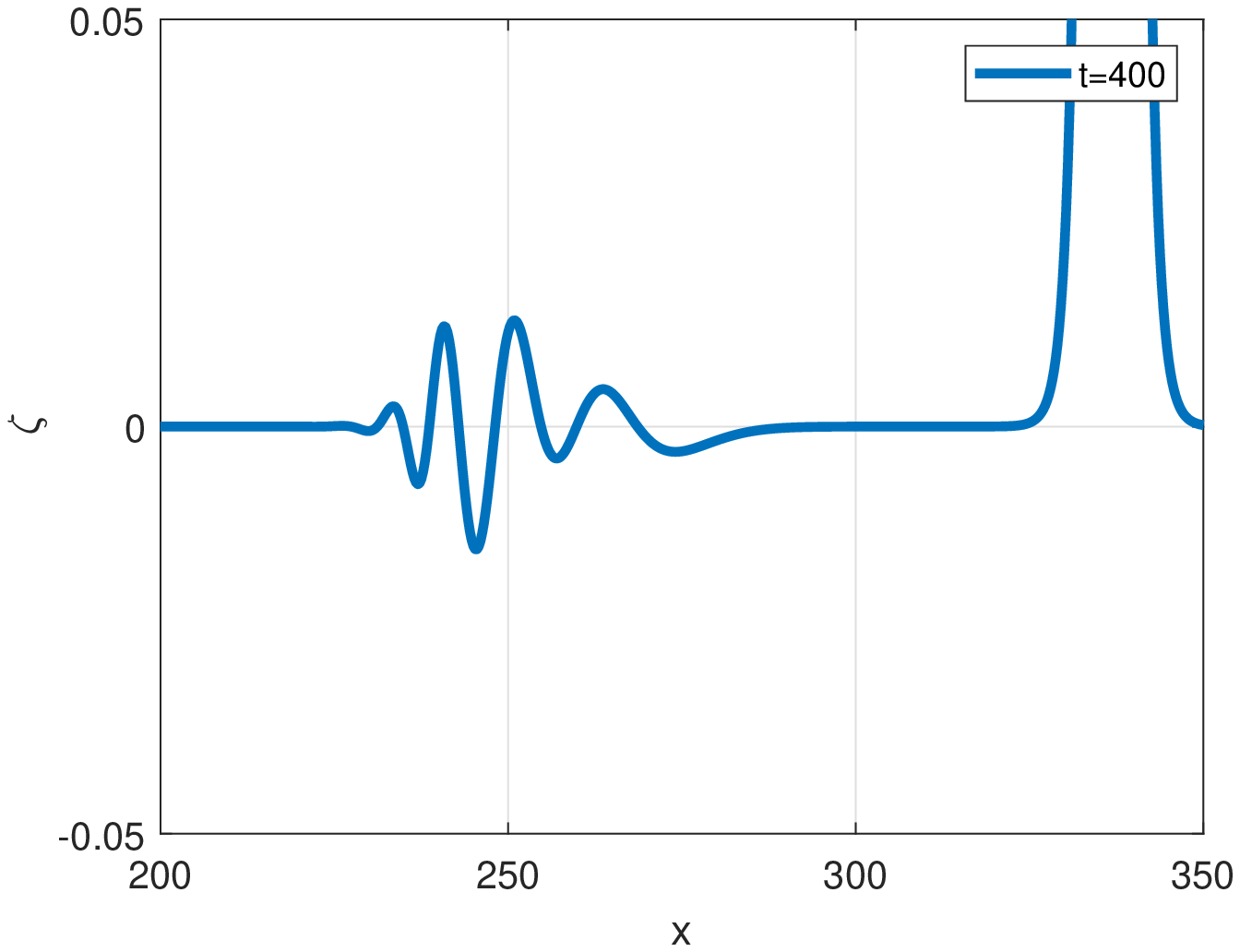}}
\subfigure[]
{\includegraphics[width=6.27cm,height=5.05cm]{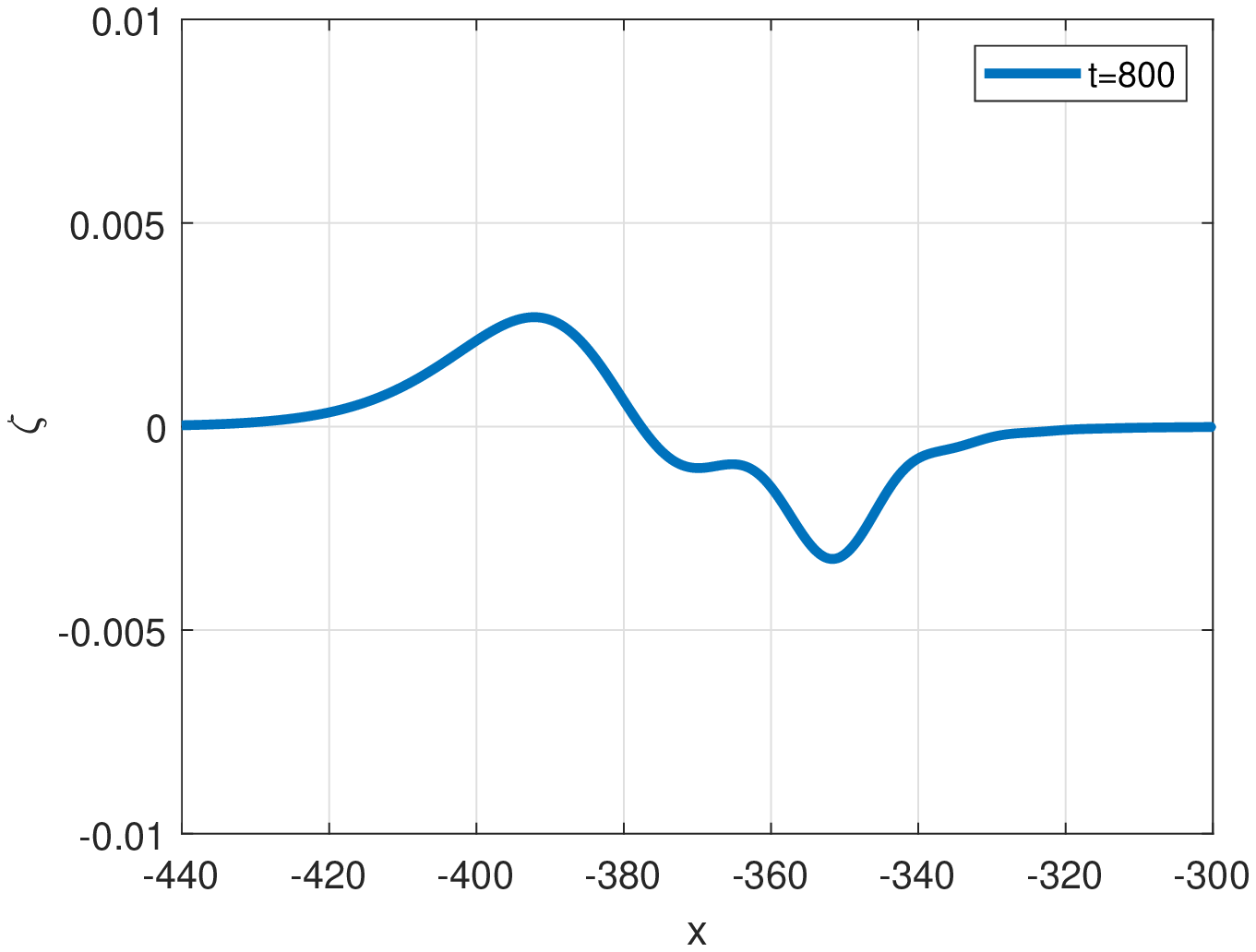}}
\subfigure[]
{\includegraphics[width=6.27cm,height=5.05cm]{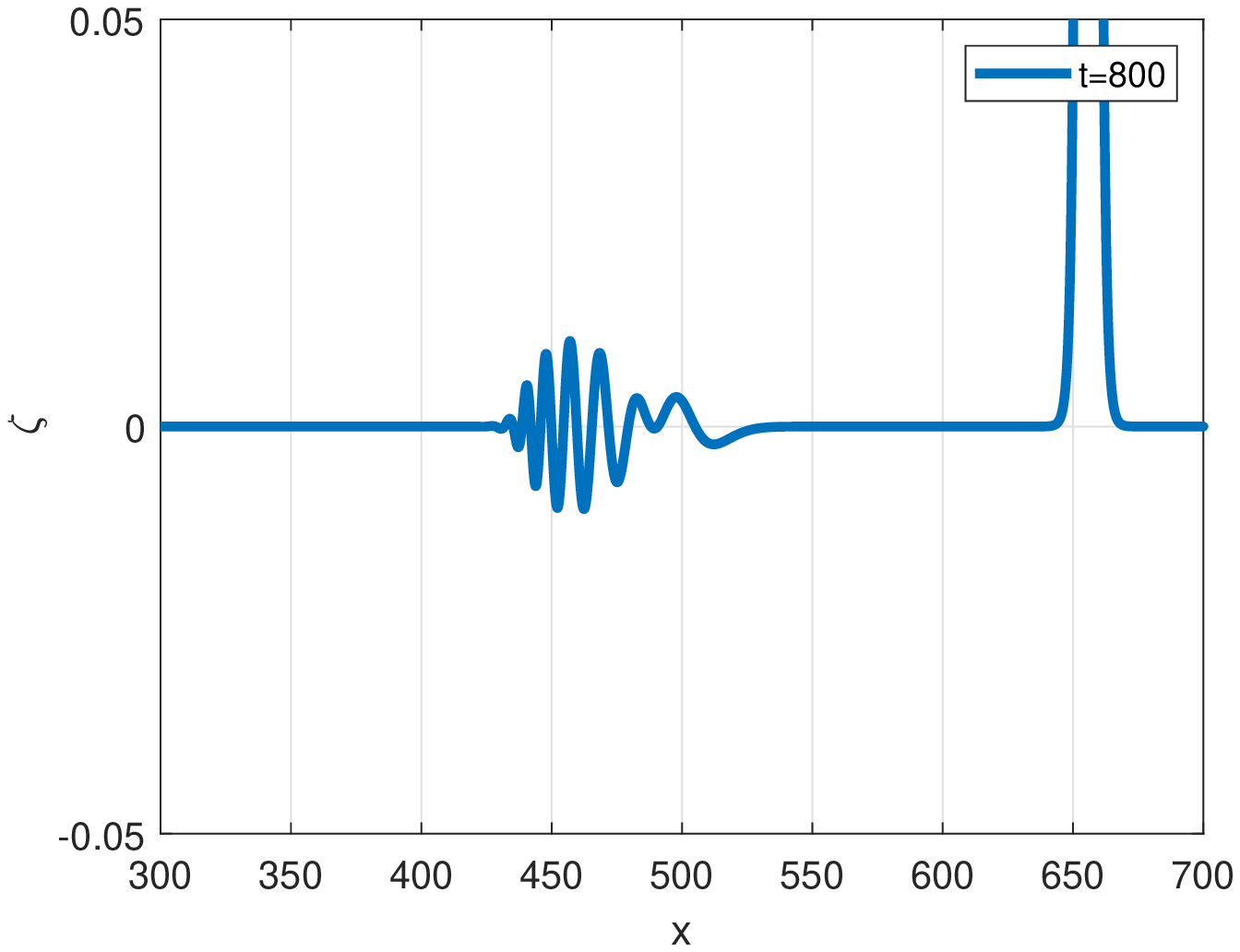}}
\caption{Overtaking collision of CSW's. Magnifications of the numerical solution of Figure \ref{fdds5_9m}.}
\label{fdds5_9m2}
\end{figure}

%It may be worth mentioning that the collision is very slightly inelastic; after the interaction a very small (compared to the amplitude of the emerging waves) dispersive tail behind the slower wave is generated. The evolution of the amplitude of the tallest solitary wave (Figure \ref{fdds5_10}) supports this idea; after the collision, the wave emerges with a relative increment of amplitude of about $9.73\times 10^{-6}$.
%[COMMENT ON THE TWO STRUCTURES]
%\begin{figure}[htbp]
%\centering
%\subfigure[]
%{\includegraphics[width=\columnwidth]{overcol_amp.eps}}
%\subfigure[]
%{\includegraphics[width=\columnwidth]{overcol_speed.eps}}
%\caption{Overtaking collisions of CSW. Case (A3) with (\ref{52a}). Evolution of amplitude (a) and speed (b) errors of the tallest emerging solitary wave ($\zeta$ component of the numerical solution).}
%\label{fdds5_10}
%\end{figure}

\subsubsection{Head-on collisions}
Two experiments on head-on collisions are reported here. With the same values of the parameters as in (\ref{52a}), in the first experiment we follow the evolution of the superposition of two approximate CSW-profiles of equal heights, initially centered at $x=\pm 20$,
with opposite speeds of absolute values equal to $c_{s}=c_{\gamma,\delta}+0.5\approx 1.0976$, the waves undergo a symmetric head-on collision, shown in Figures \ref{fdds5_11} and \ref{fdds5_11m}.

\begin{figure}[htbp]
\centering
\subfigure[]
{\includegraphics[width=\columnwidth]{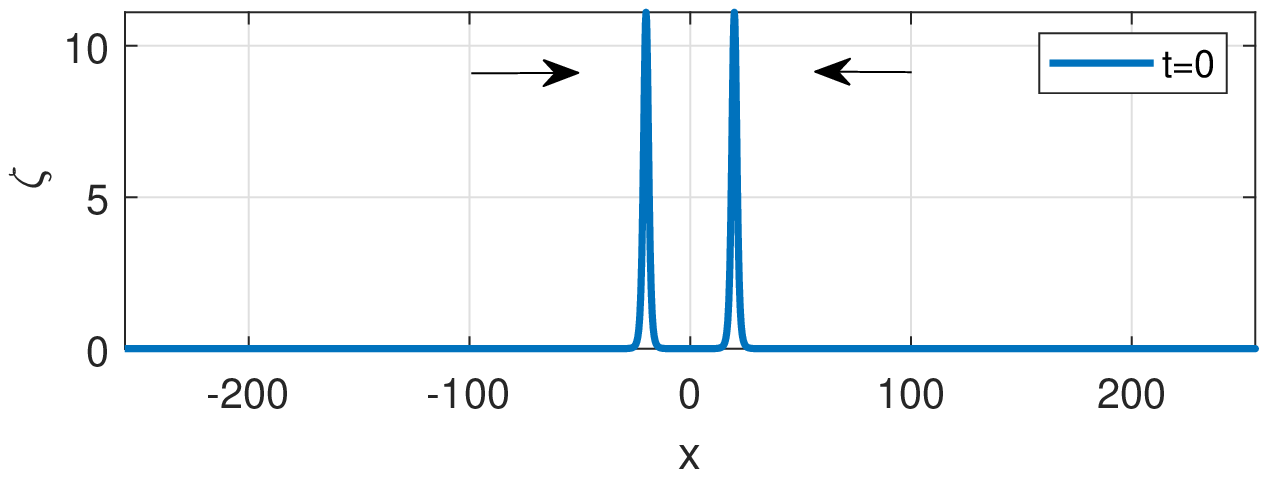}}
%\subfigure[]
%{\includegraphics[width=\columnwidth]{overcol_t200.eps}}
\subfigure[]
{\includegraphics[width=\columnwidth]{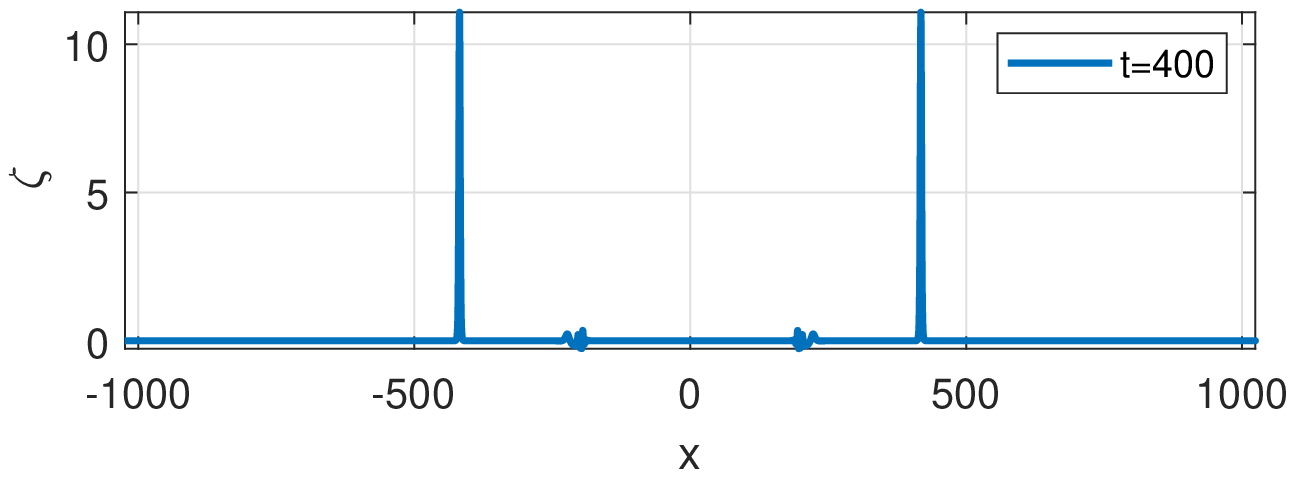}}
%\subfigure[]
%{\includegraphics[width=\columnwidth]{overcol_t600.eps}}
\subfigure[]
{\includegraphics[width=\columnwidth]{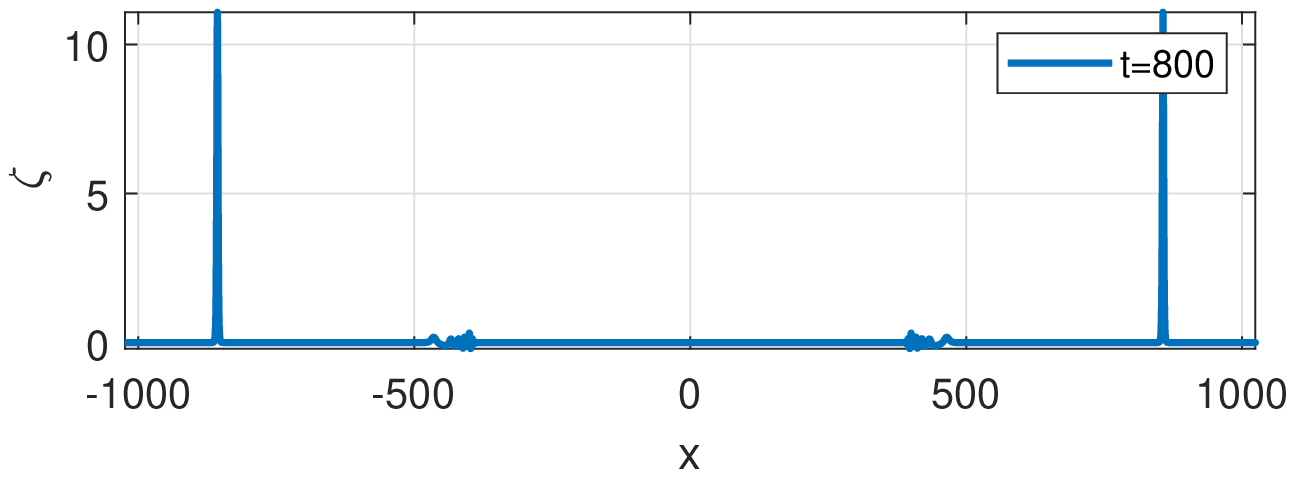}}
%\subfigure[]
%{\includegraphics[width=\columnwidth]{csw4_4.eps}}
\caption{Symmetric head-on collision of CSW's.  (a)-(c) $\zeta$ component of the numerical approximation.}
\label{fdds5_11}
\end{figure}

\begin{figure}[htbp]
\centering
\subfigure[]
{\includegraphics[width=\columnwidth]{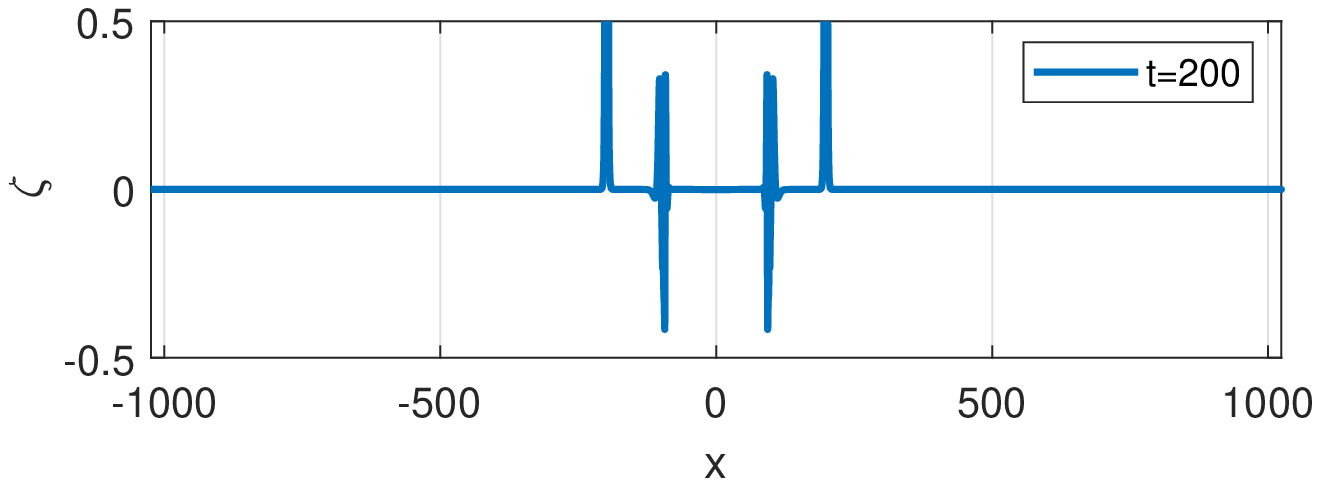}}
\subfigure[]
{\includegraphics[width=\columnwidth]{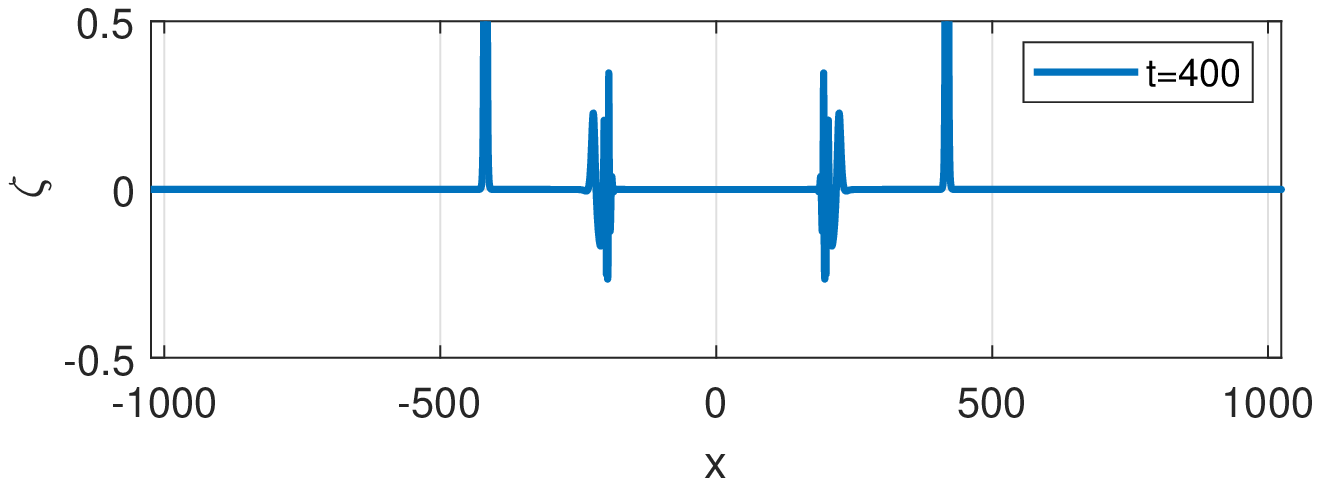}}
\subfigure[]
{\includegraphics[width=\columnwidth]{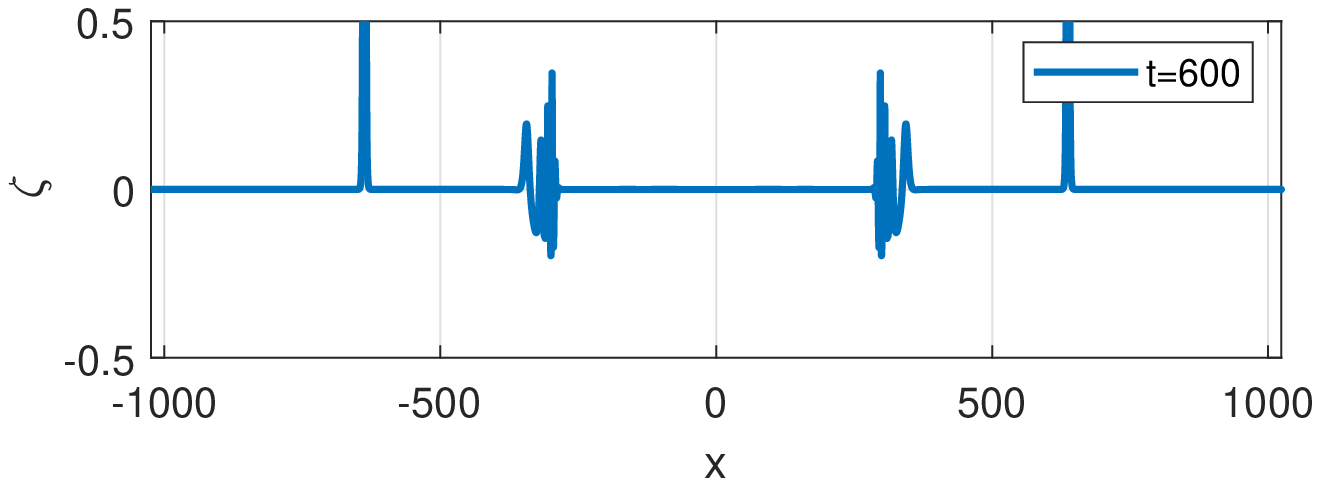}}
\subfigure[]
{\includegraphics[width=\columnwidth]{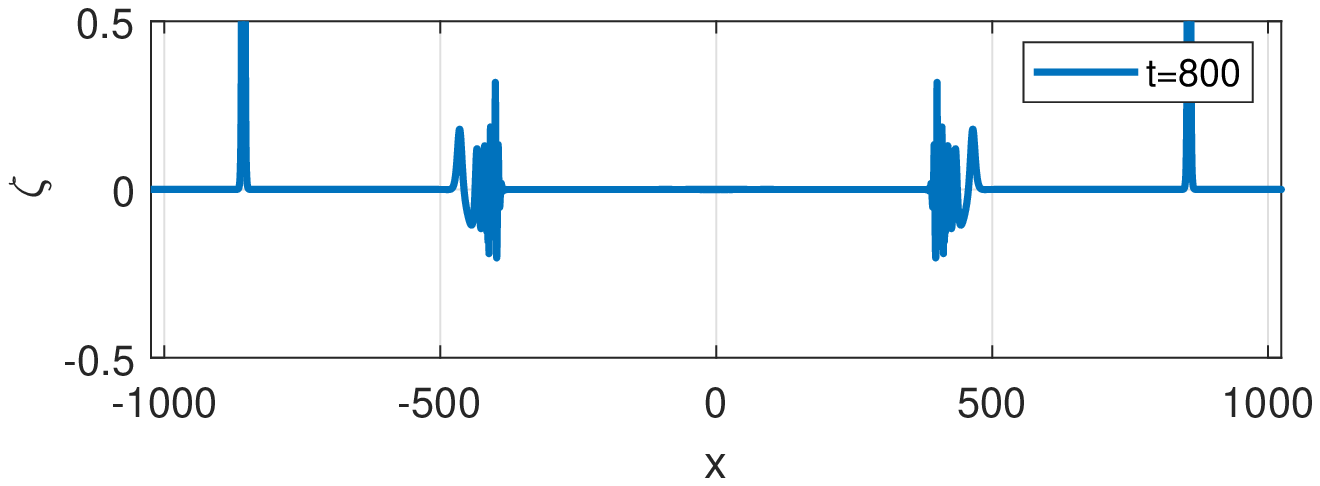}}
\caption{Symmetric head-on collision of CSW's. Magnifications of the numerical solution of Figure \ref{fdds5_11}.}
\label{fdds5_11m}
\end{figure}

%\begin{figure}[htbp]
%\centering
%\subfigure[]
%{\includegraphics[width=\columnwidth]{symho_amp.eps}}
%\subfigure[]
%{\includegraphics[width=\columnwidth]{symho_speed.eps}}
%\caption{Symmetric head-on of CSW. Case (A3) with (\ref{52a}). Evolution of amplitude (a) and speed (b) errors of the right-going emerging solitary wave ($\zeta$ component of the numerical solution); cf. Figure \ref{fdds5_11}.}
%\label{fdds5_11b}
%\end{figure}
\begin{figure}[htbp]
\centering
\subfigure[]
{\includegraphics[width=6.27cm,height=5.05cm]{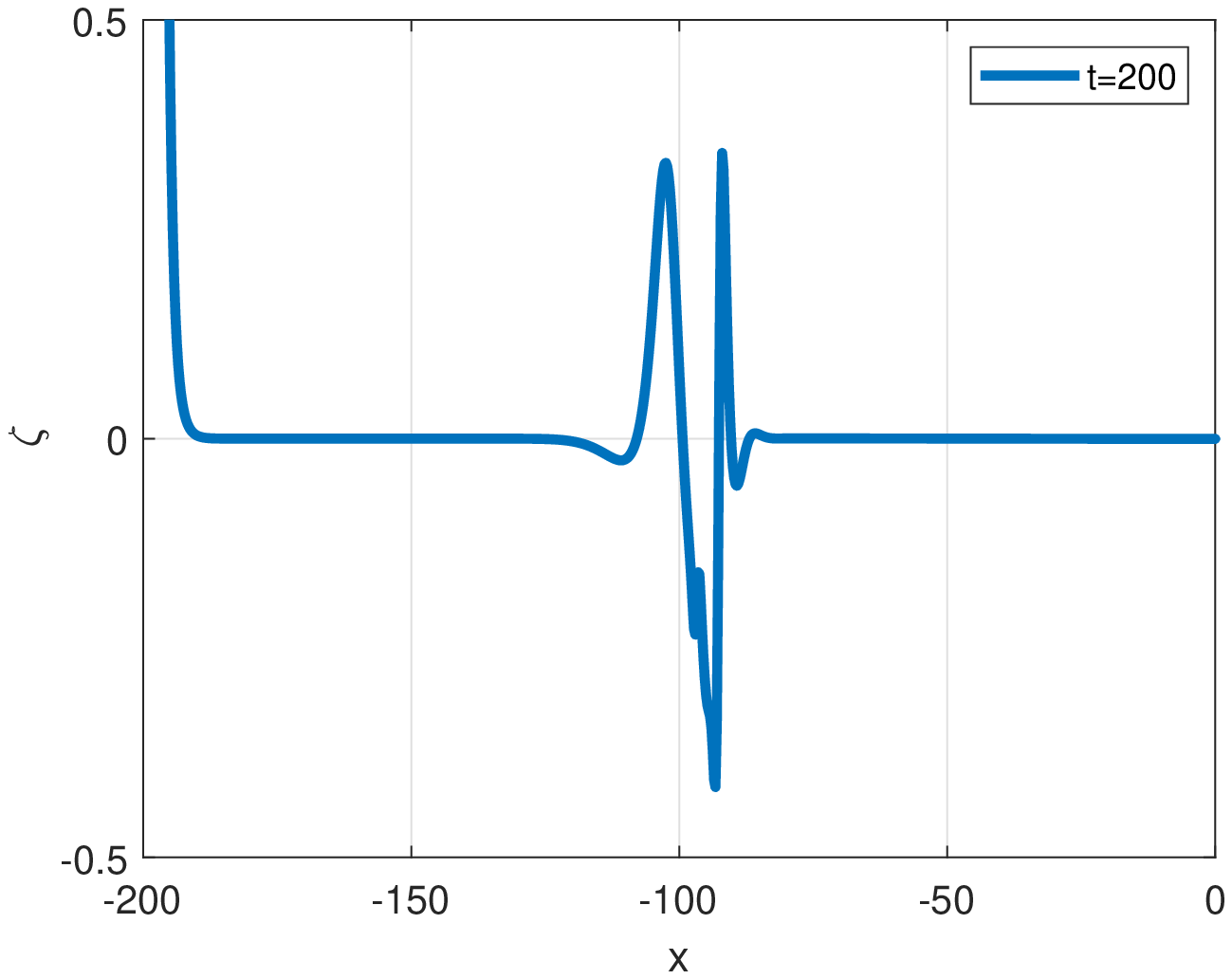}}
\subfigure[]
{\includegraphics[width=6.27cm,height=5.05cm]{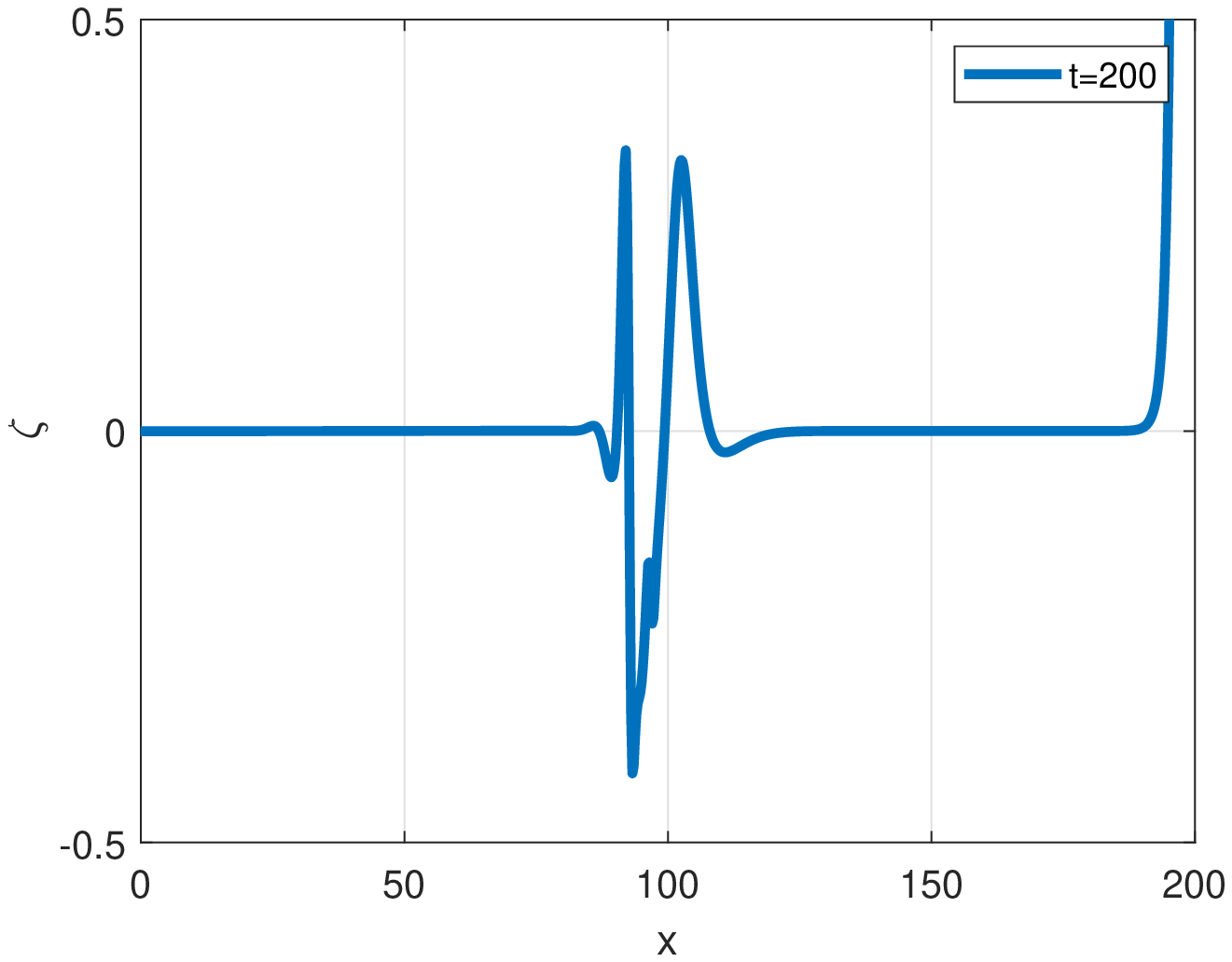}}
\subfigure[]
{\includegraphics[width=6.27cm,height=5.05cm]{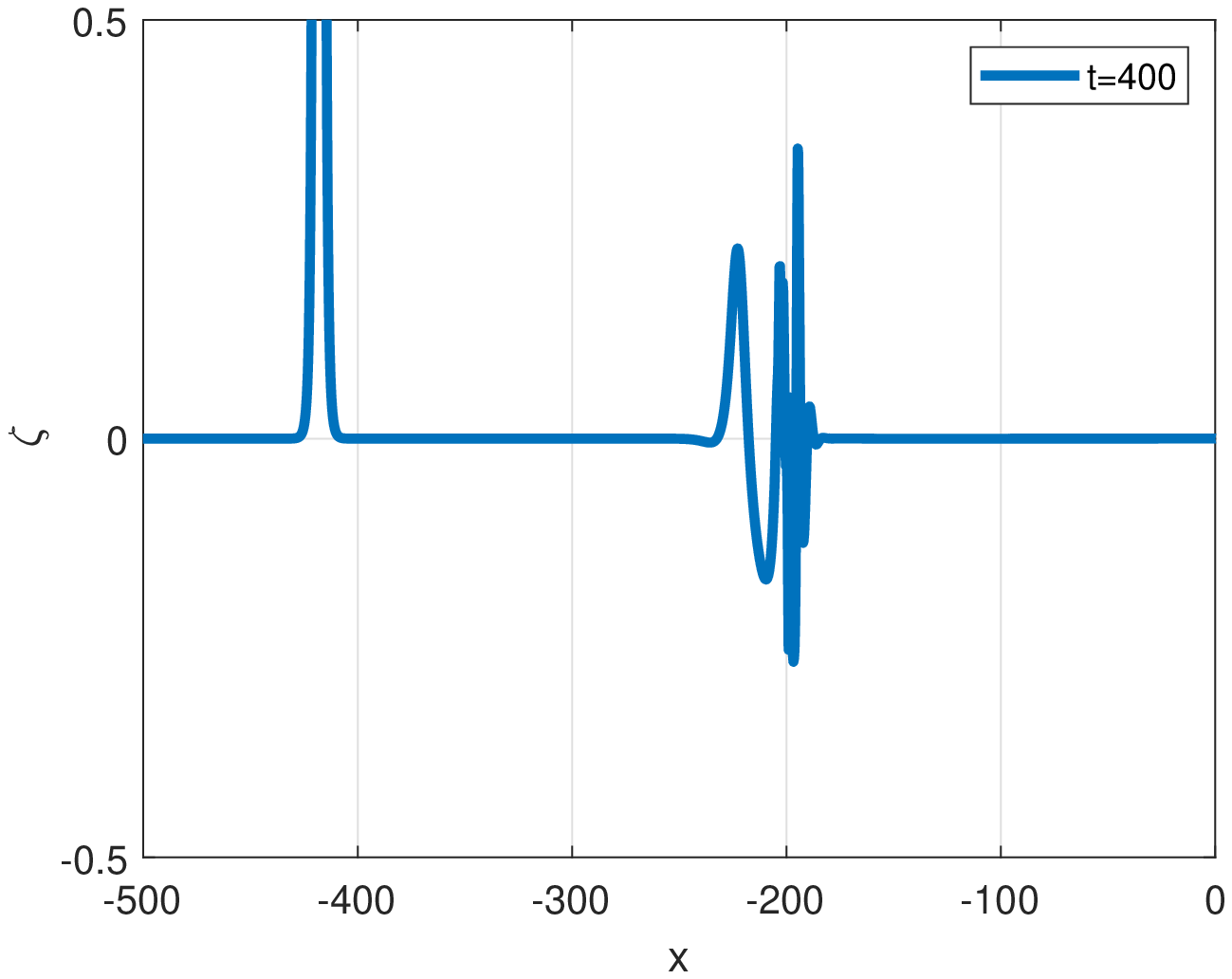}}
\subfigure[]
{\includegraphics[width=6.27cm,height=5.05cm]{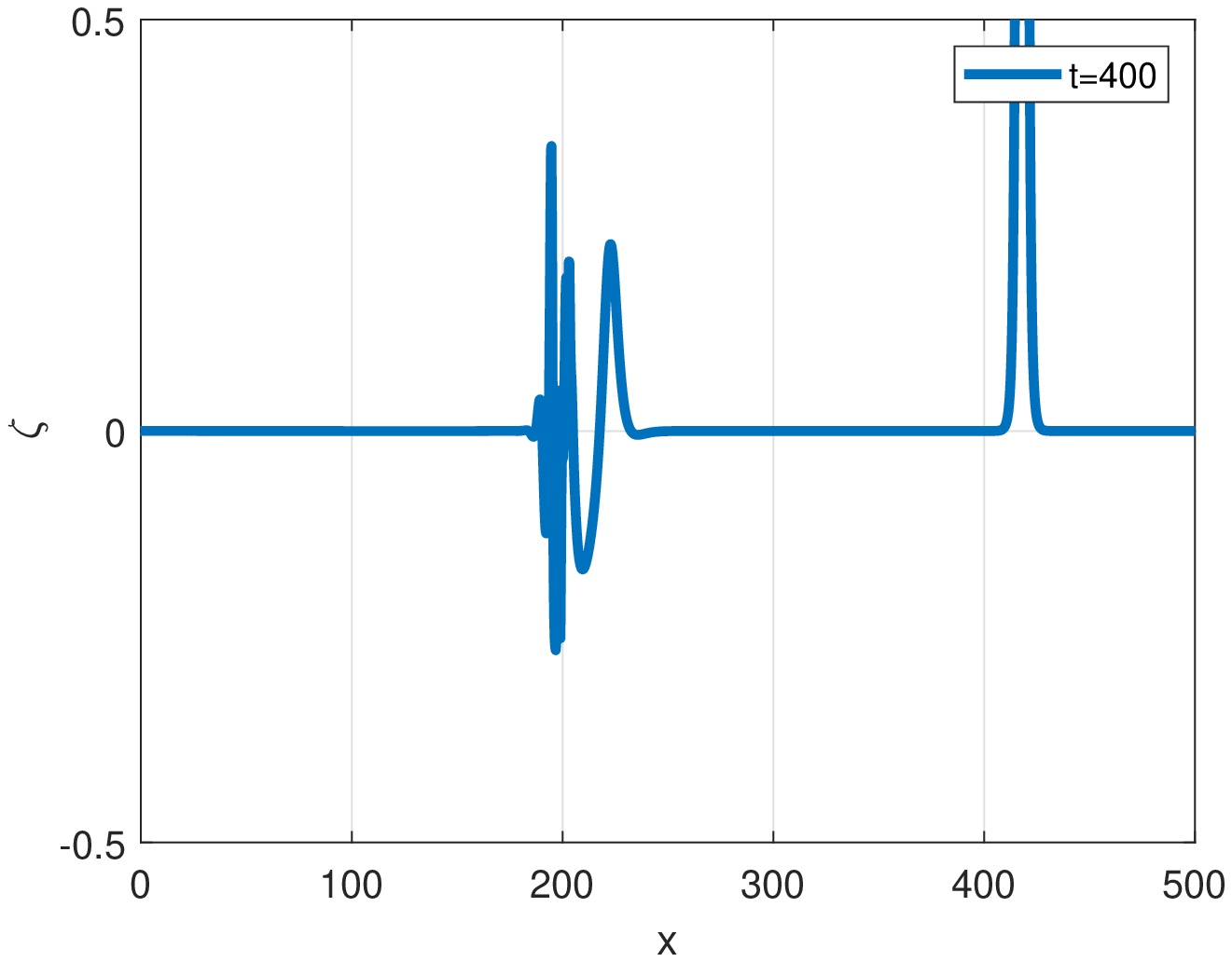}}
\subfigure[]
{\includegraphics[width=6.27cm,height=5.05cm]{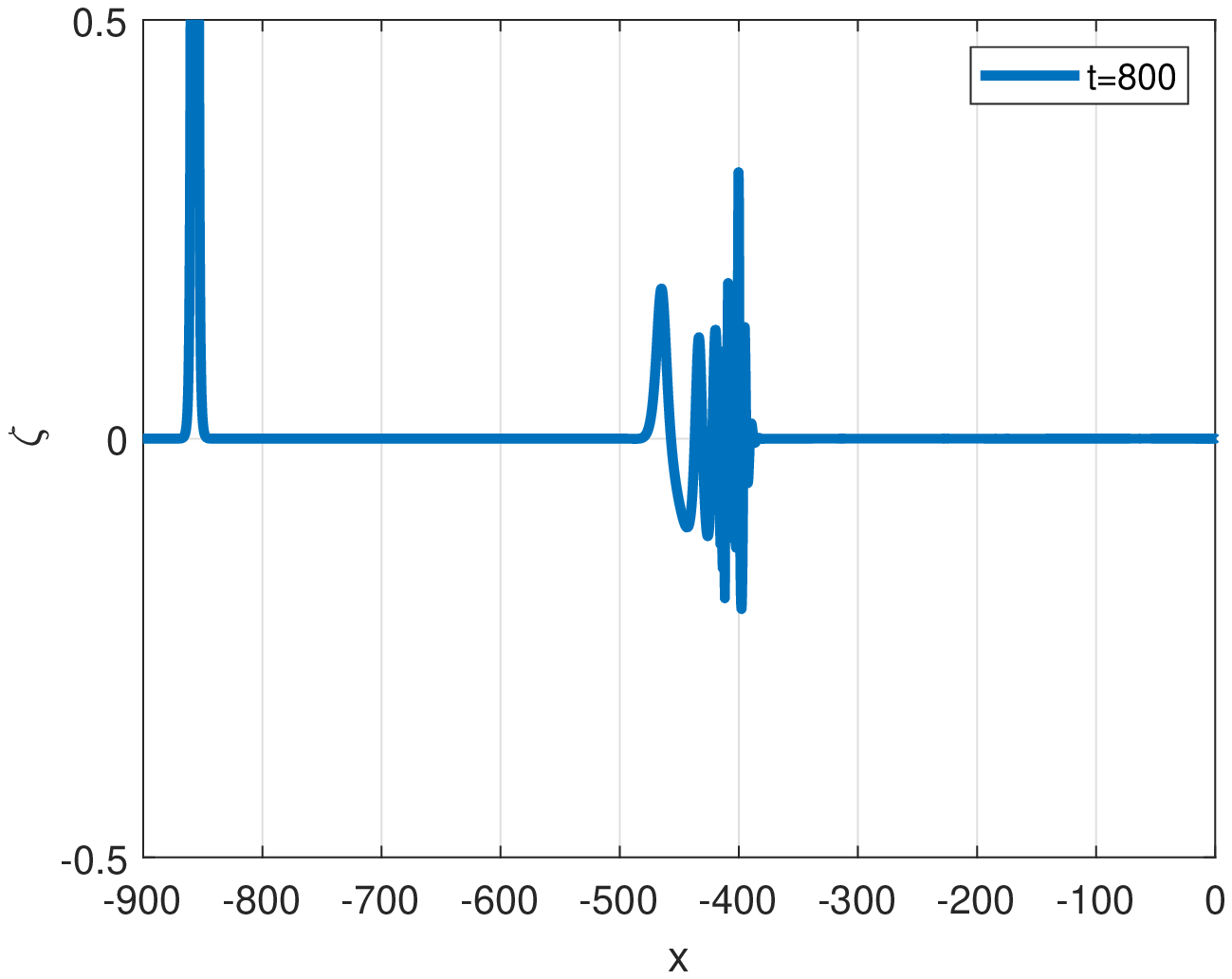}}
\subfigure[]
{\includegraphics[width=6.27cm,height=5.05cm]{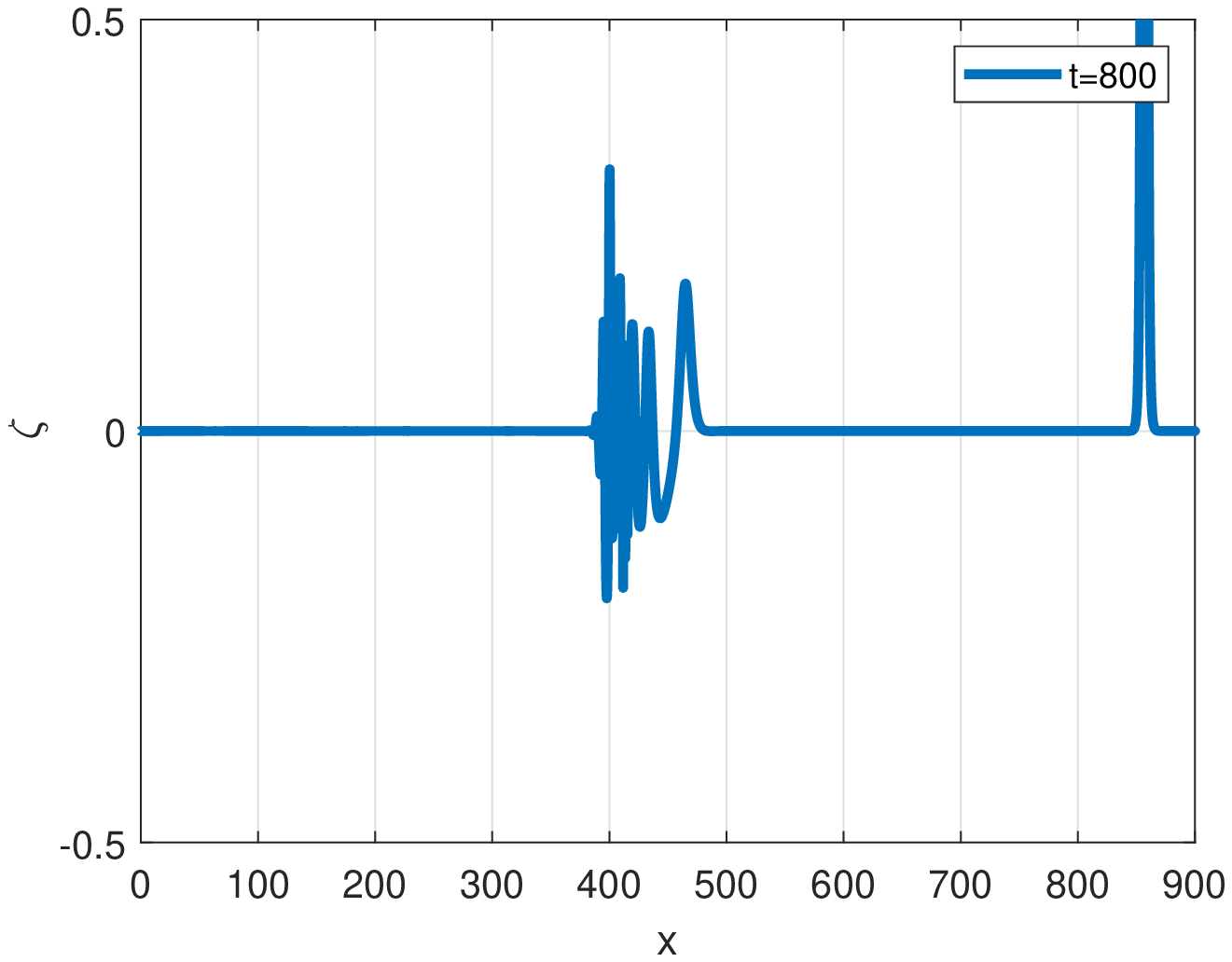}}
\caption{Symmetric head-on collision of CSW's. Magnifications of the numerical solution of Figure \ref{fdds5_11m}.}
\label{fdds5_11m1}
\end{figure}

\begin{figure}[htbp]
\centering
\subfigure[]
{\includegraphics[width=6.27cm,height=5.05cm]{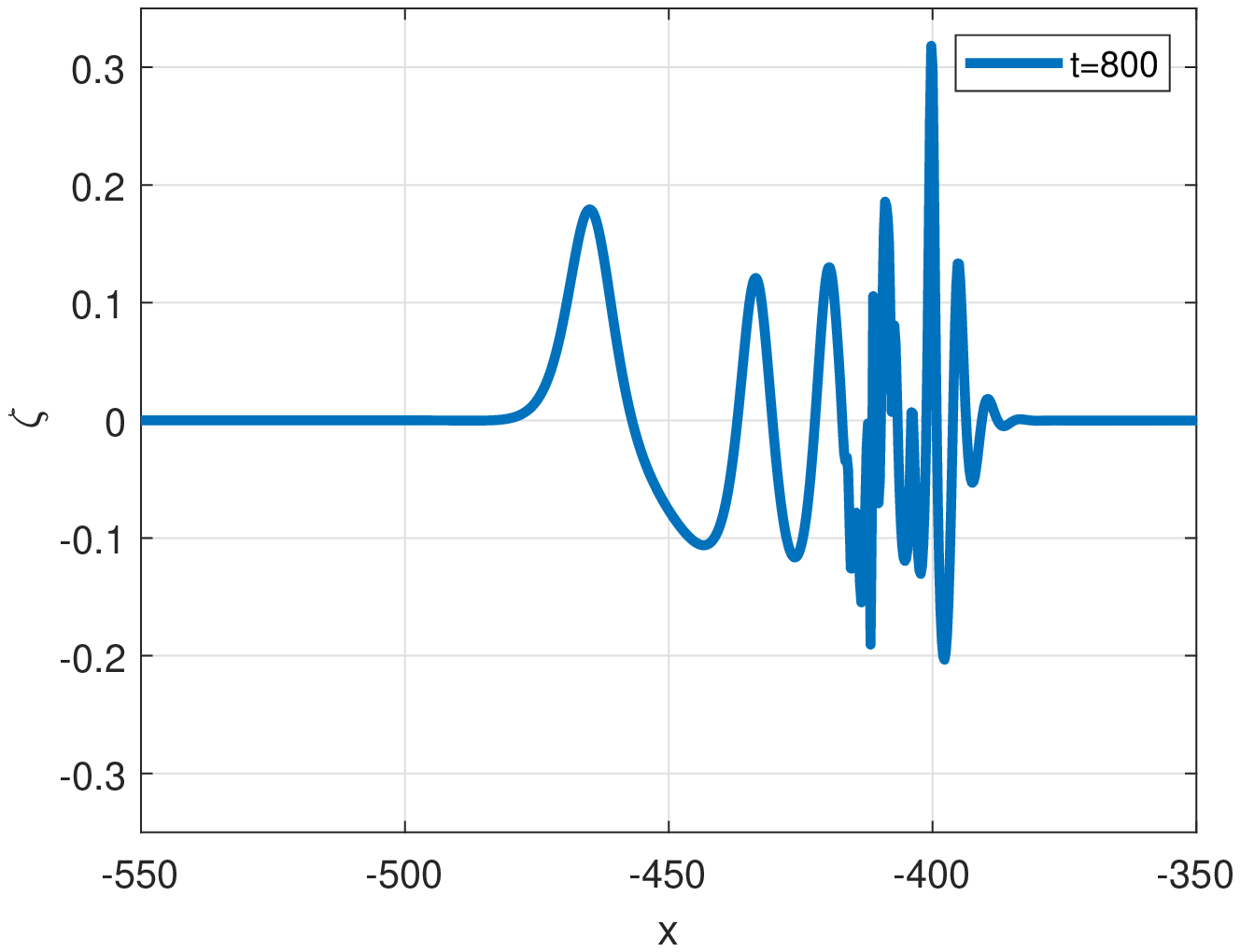}}
\subfigure[]
{\includegraphics[width=6.27cm,height=5.05cm]{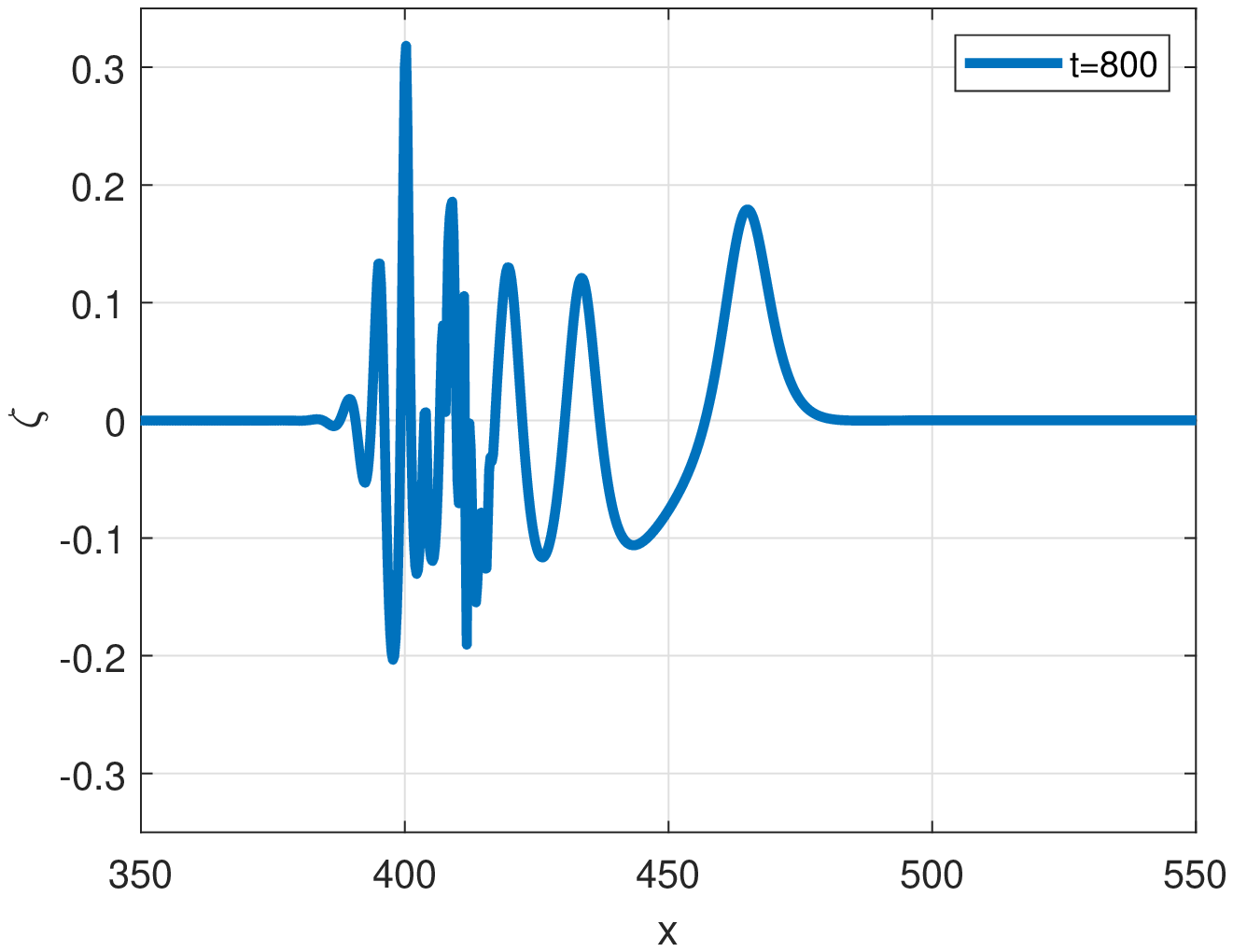}}
\caption{Symmetric head-on collision of CSW's. Magnifications of the numerical solution of Figure \ref{fdds5_11m1}(c).}
\label{fdds5_11m2}
\end{figure}
%\begin{figure}[htbp]
%\centering
%\subfigure[]
%{\includegraphics[width=10cm]{symho_t0.eps}}
%\subfigure[]
%{\includegraphics[width=10cm]{symho_t400.eps}}
%\subfigure[]
%{\includegraphics[width=10cm]{symho_t800.eps}}
%\subfigure[]
%{\includegraphics[width=10cm]{csw5_4.eps}}
%\caption{Symmetric head-on collisions of CSW. Case (A3) with (\ref{52a}).  (a)-(c) $\zeta$ component of the numerical solution; (d) Magnification of (c).}
%\label{fdds5_11}
%\end{figure}

The outcome of the collision is symmetric as well. After the interaction, there emerge two solitary waves and structures traveling behind them, similar to ones already observed in other experiments. In this case, the amplitude of the emerging right-traveling CSW is smaller than that of the initial one with a relative difference of about $1.8\times 10^{-3}$. The emerging waves are also slightly slower than the initial ones, with a relative decrease of $6.4\times 10^{-4}$ in their speeds.

The structures behind the solitary waves, shown in more detail in Figures \ref{fdds5_11m1} and \ref{fdds5_11m2}, are of bigger size than those of the previous experiment of overtaking collision, but
seem to be again of dispersive and nonlinear type. The form of the nonlinear structure, however, is not yet clear from Figure \ref{fdds5_11m2}.
%Trailing each main wave, at least one  wavelet is observed to be formed, fighting for separating from the dispersive tail following it.
%The behaviour is apparently similar to that of the experiments from large perturbations of CSW: we may distinguish the formation of CSW with the rest of the structure which may hide dispersion and possibly other nonlinear components (in the form of CSW of elevation or of depression). In this case, the amplitude of the emerging right-going CSW is smaller than that of the original one; the difference is about $2.1\times 10^{-2}$. 

\begin{figure}[htbp]
\centering
\subfigure[]
{\includegraphics[width=\columnwidth]{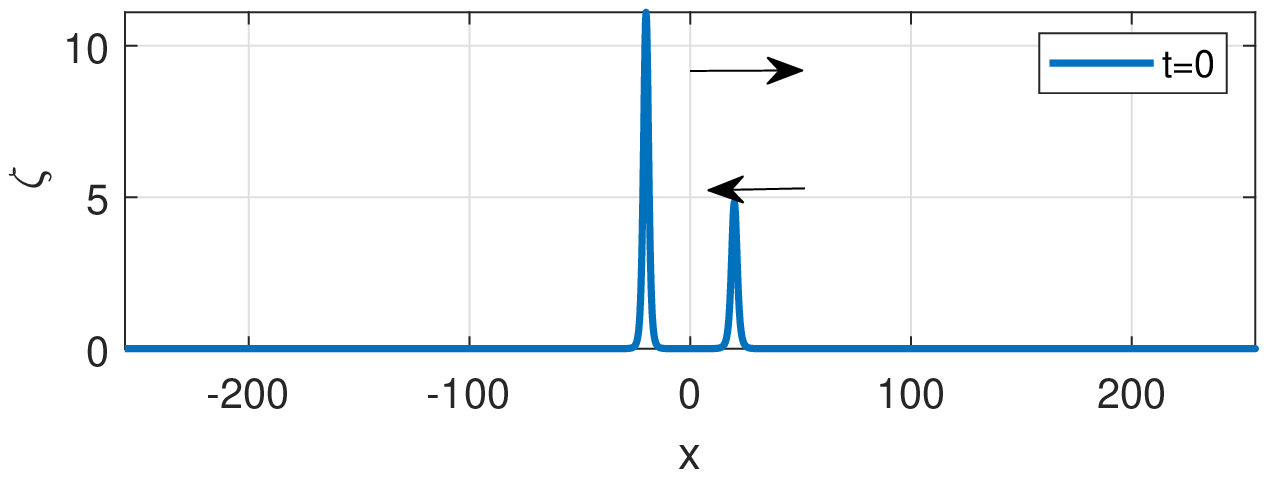}}
%\subfigure[]
%{\includegraphics[width=\columnwidth]{overcol_t200.eps}}
\subfigure[]
{\includegraphics[width=\columnwidth]{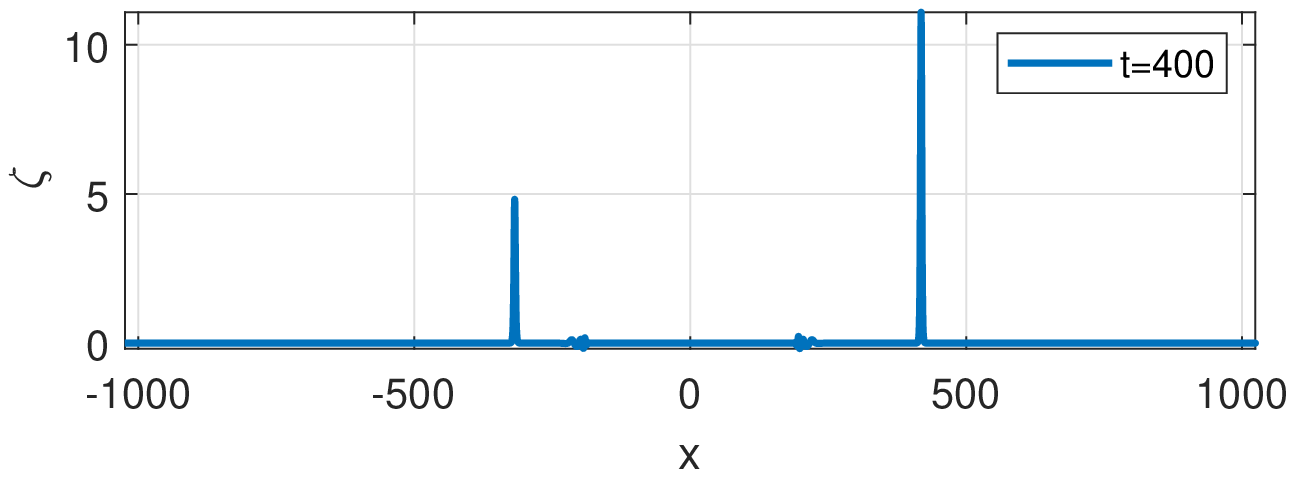}}
%\subfigure[]
%{\includegraphics[width=\columnwidth]{overcol_t600.eps}}
\subfigure[]
{\includegraphics[width=\columnwidth]{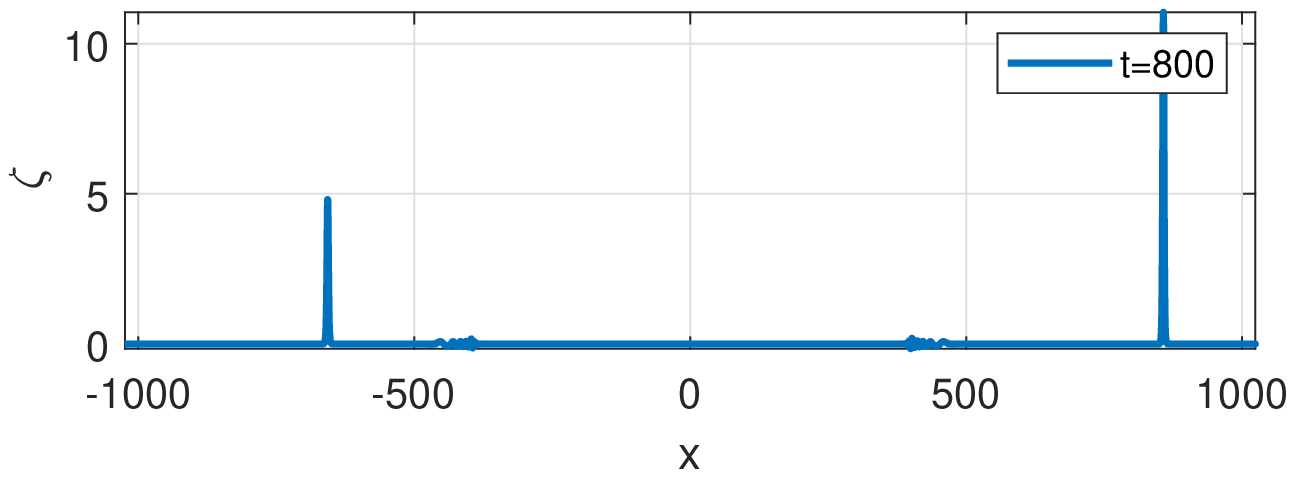}}
%\subfigure[]
%{\includegraphics[width=\columnwidth]{csw4_4.eps}}
\caption{Non-symmetric head-on collision of CSW's. (a)-(c) $\zeta$ component of the numerical approximation.}
\label{fdds5_12}
\end{figure}

%\begin{figure}[htbp]
%\centering
%\subfigure[]
%{\includegraphics[width=\columnwidth]{nsymho_amp.eps}}
%\subfigure[]
%{\includegraphics[width=\columnwidth]{nsymho_speed.eps}}
%\caption{Non-symmetric head-on collision of CSW's. Evolution of amplitude (a) and speed (b) errors of the right-going emerging solitary wave ($\zeta$ component of the numerical solution); cf. Figure \ref{fdds5_12}.}
%\label{fdds5_12b}
%\end{figure}

\begin{figure}[htbp]
\centering
\subfigure[]
{\includegraphics[width=\columnwidth]{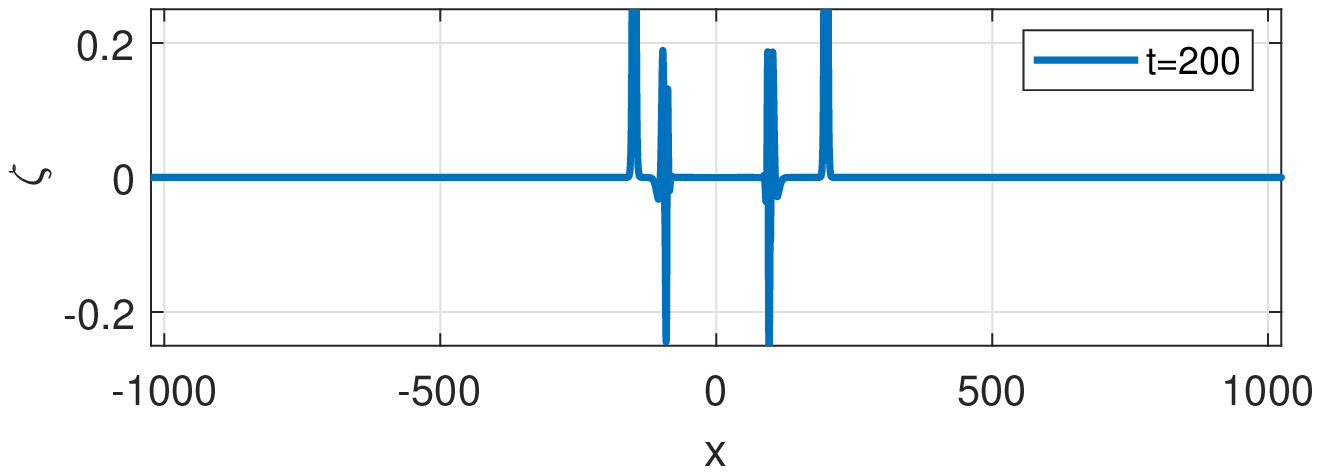}}
\subfigure[]
{\includegraphics[width=\columnwidth]{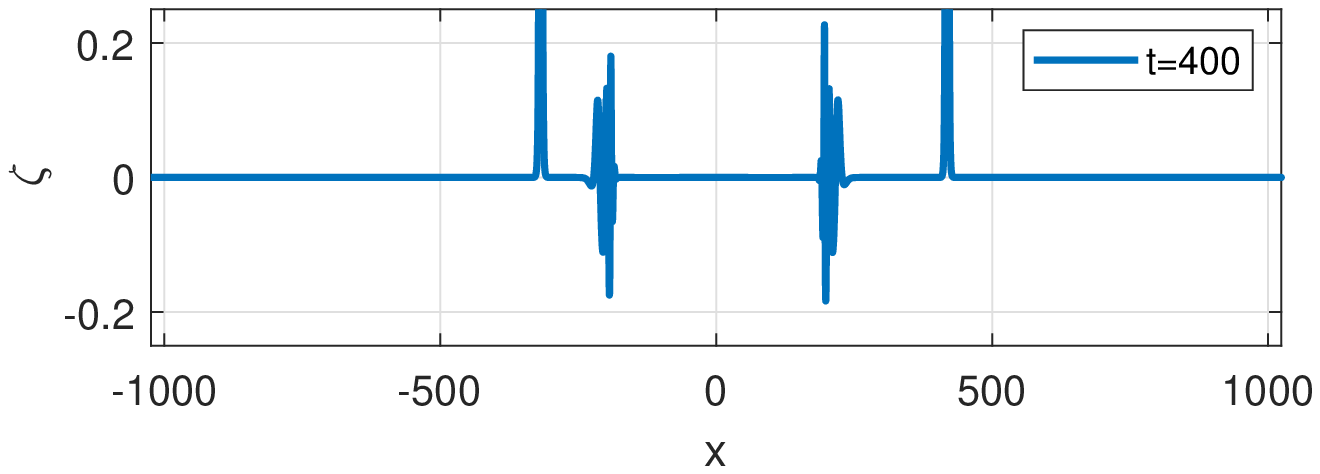}}
\subfigure[]
{\includegraphics[width=\columnwidth]{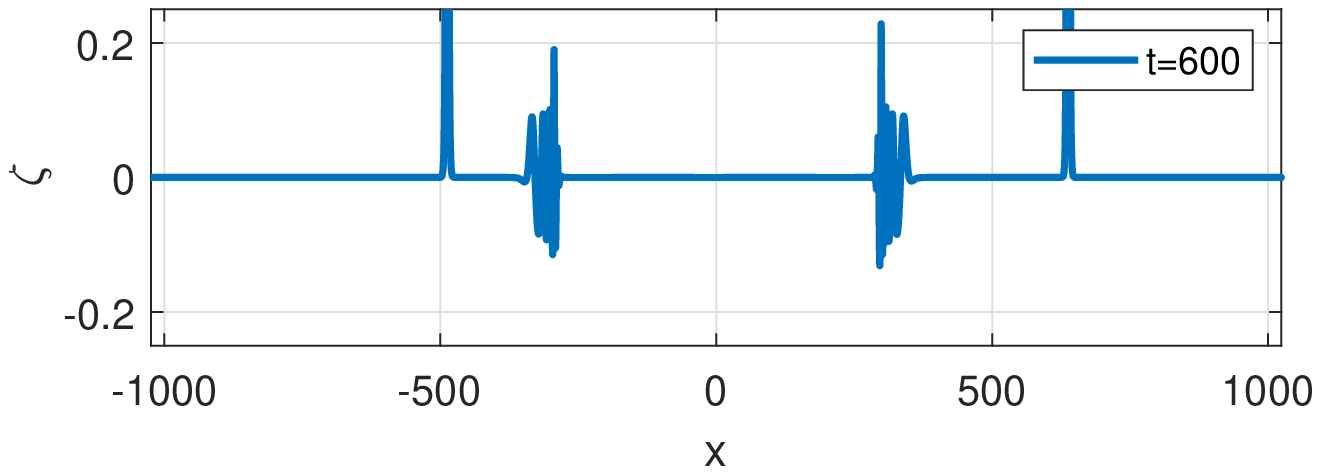}}
\subfigure[]
{\includegraphics[width=\columnwidth]{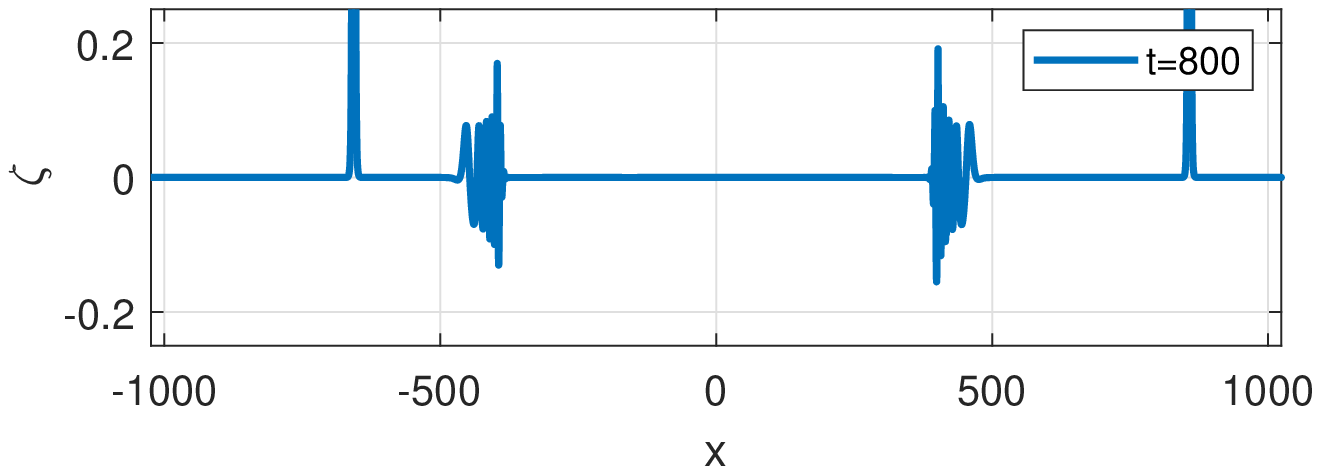}}
\caption{Non-symmetric head-on collision of CSW's. Magnifications of the numerical solution of Figure \ref{fdds5_12}.}
\label{fdds5_12m}
\end{figure}

\begin{figure}[htbp]
\centering
\subfigure[]
{\includegraphics[width=6.27cm,height=5.05cm]{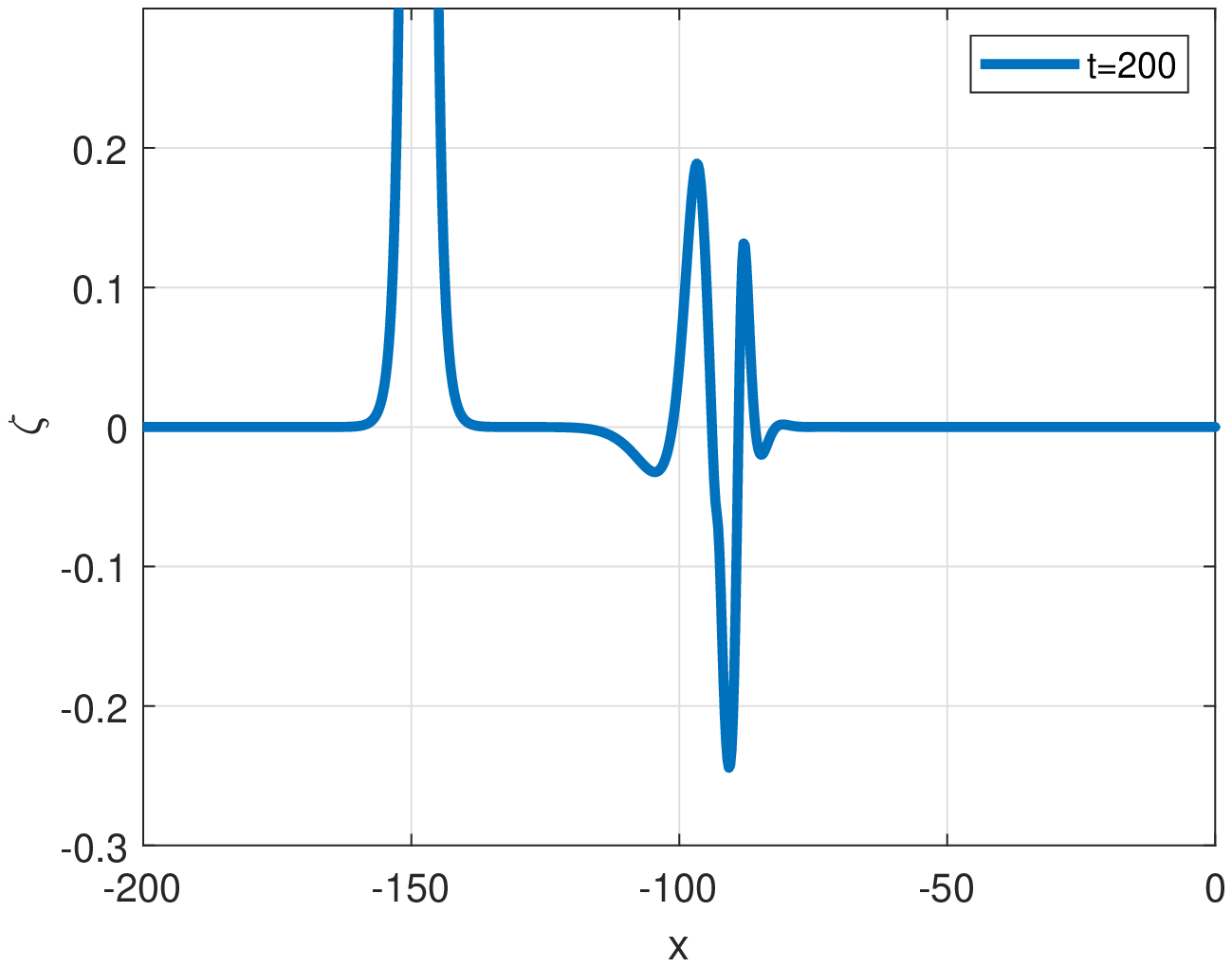}}
\subfigure[]
{\includegraphics[width=6.27cm,height=5.05cm]{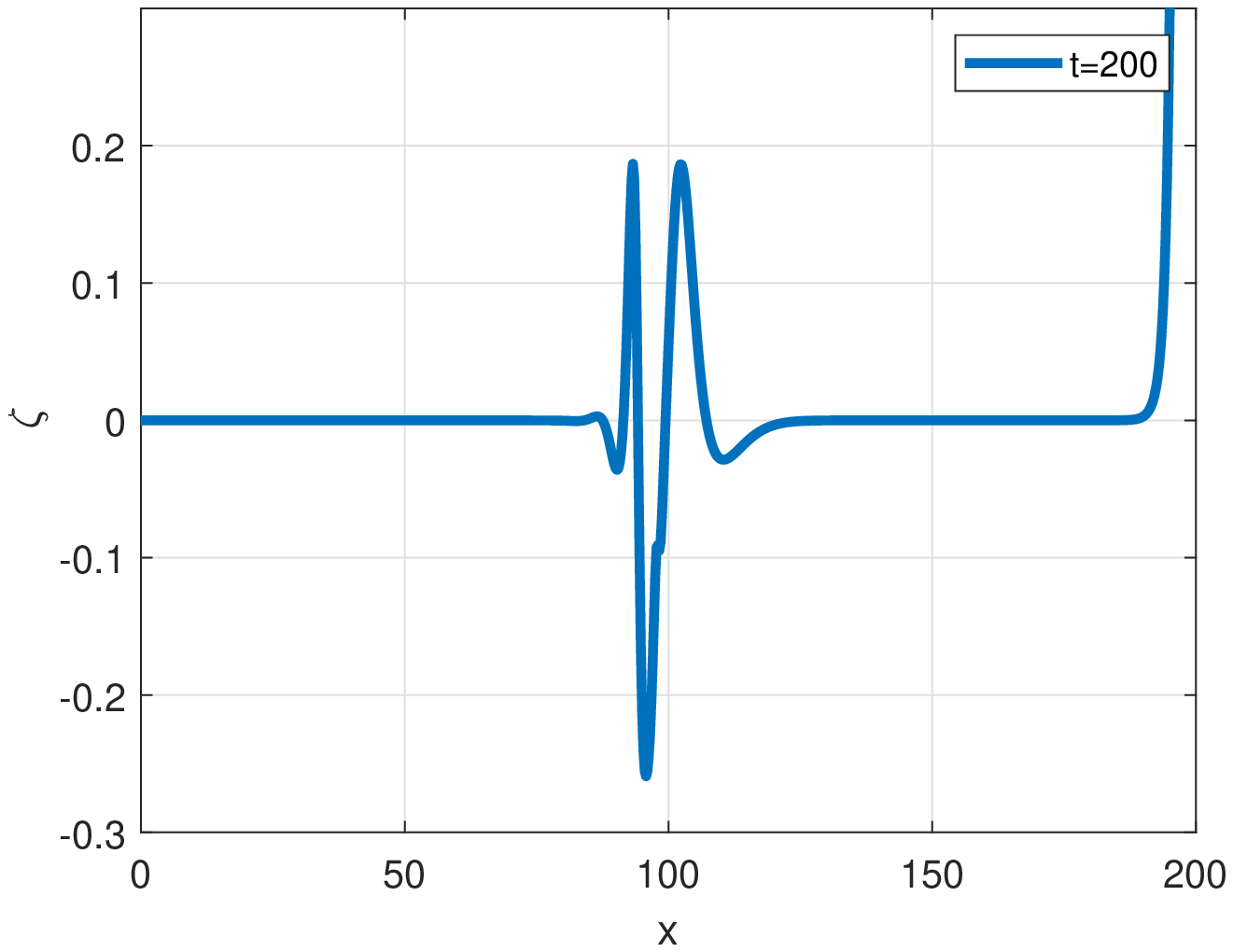}}
\subfigure[]
{\includegraphics[width=6.27cm,height=5.05cm]{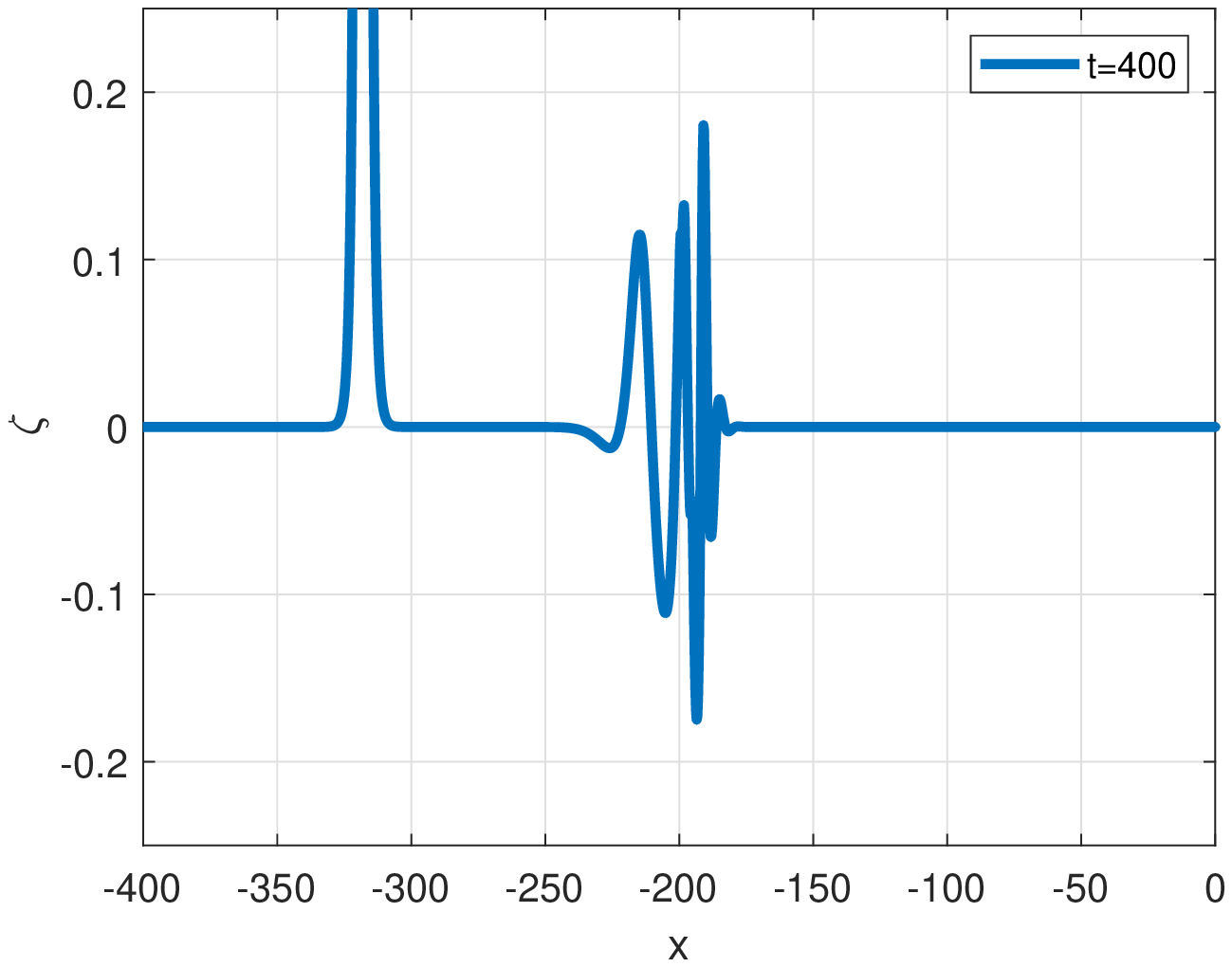}}
\subfigure[]
{\includegraphics[width=6.27cm,height=5.05cm]{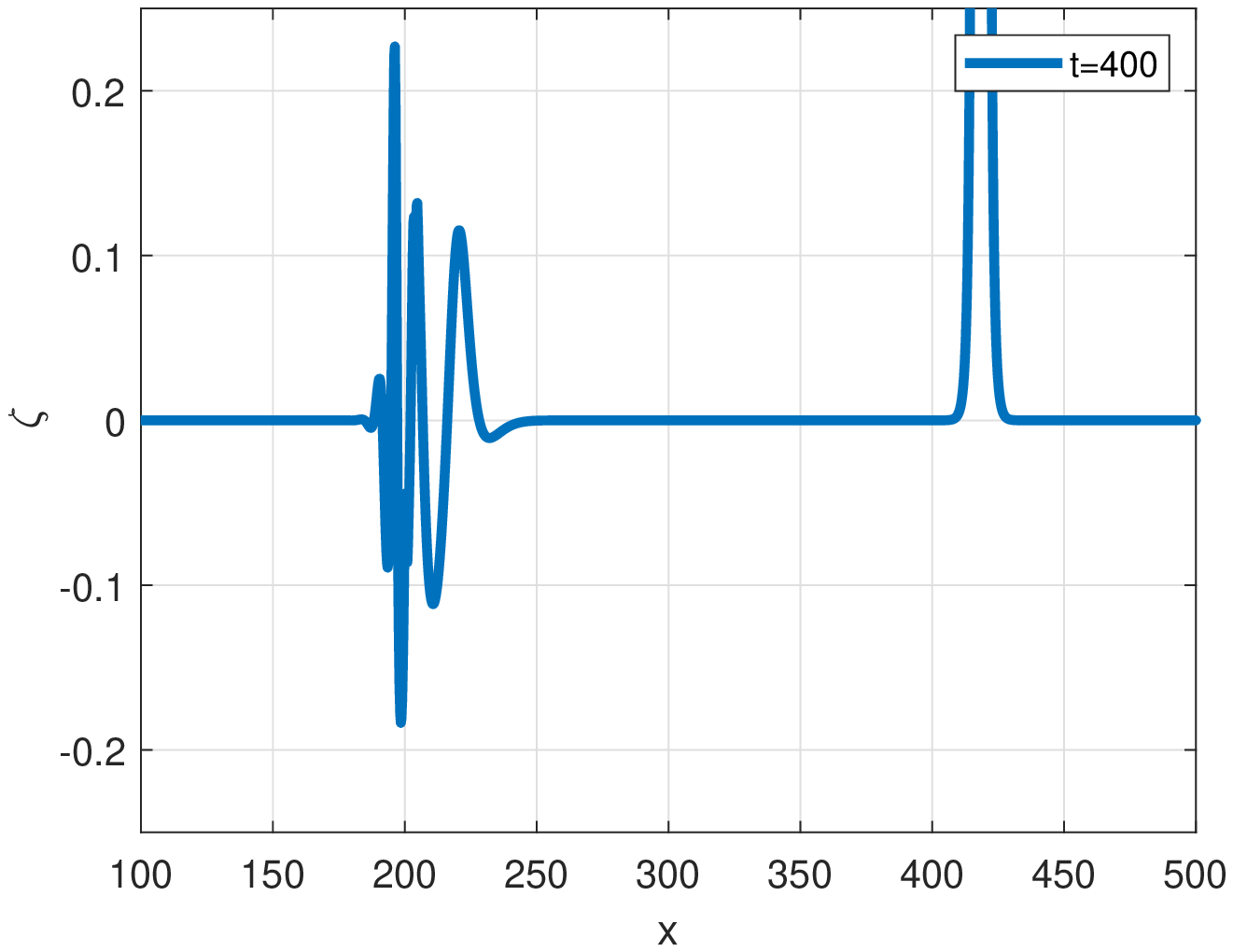}}
\subfigure[]
{\includegraphics[width=6.27cm,height=5.05cm]{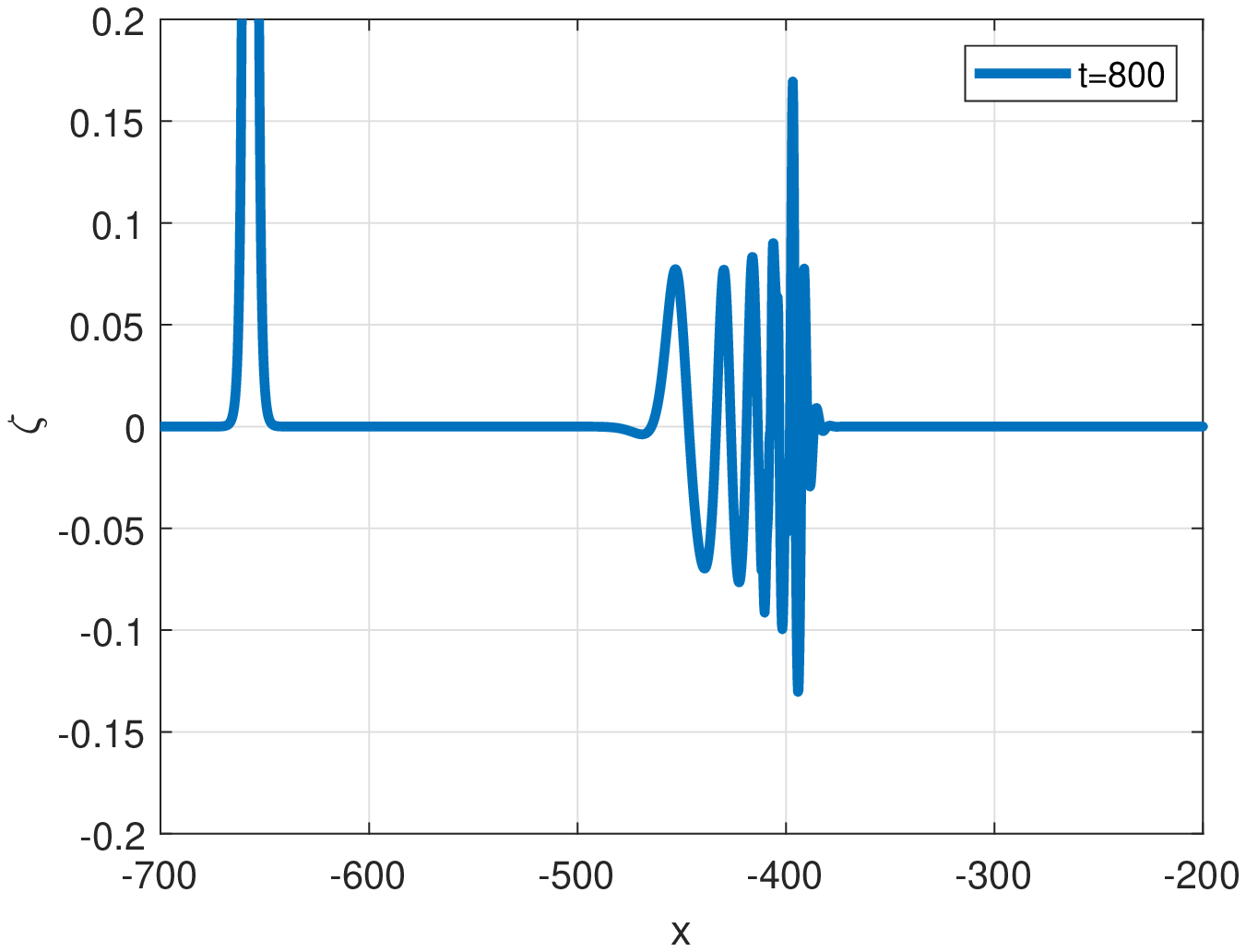}}
\subfigure[]
{\includegraphics[width=6.27cm,height=5.05cm]{symho_t800m2.eps}}
\caption{Non-symmetric head-on collision of CSW's. Magnifications of the numerical solution of Figure \ref{fdds5_12m}.}
\label{fdds5_12m1}
\end{figure}

\begin{figure}[htbp]
\centering
\subfigure[]
{\includegraphics[width=6.27cm,height=5.05cm]{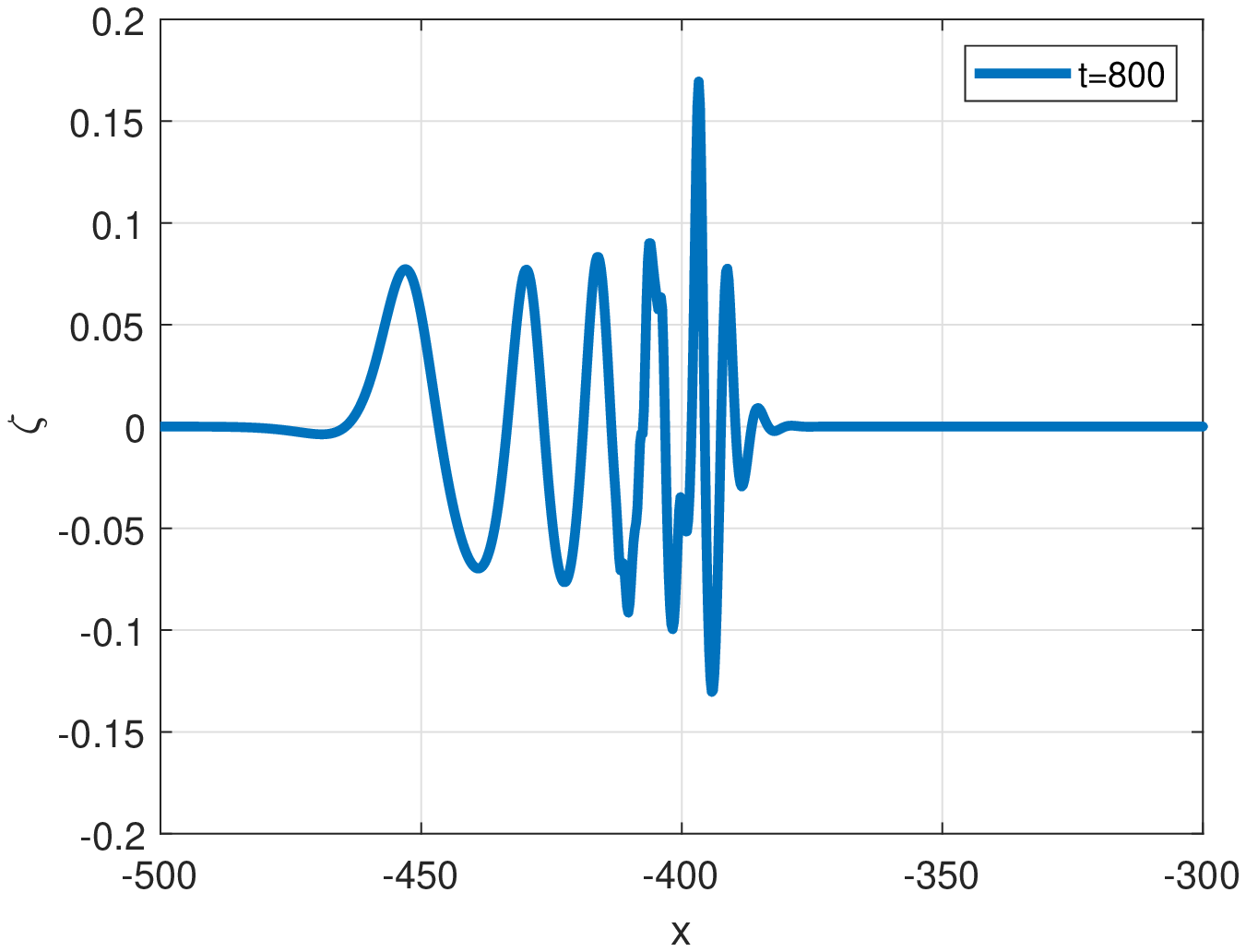}}
\subfigure[]
{\includegraphics[width=6.27cm,height=5.05cm]{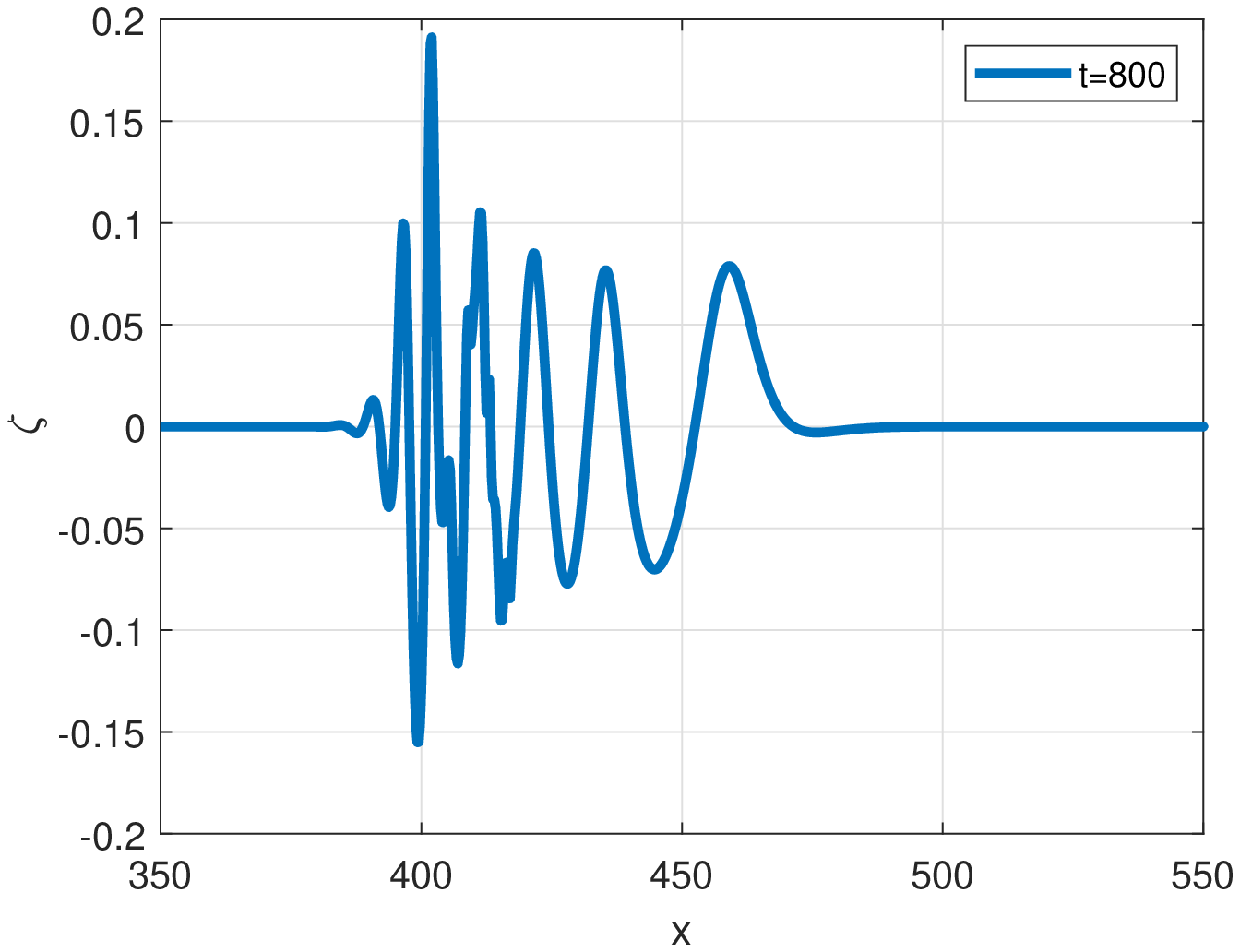}}
\caption{Non-symmetric head-on collision of CSW's. Magnifications of the numerical solution of Figure \ref{fdds5_12m1}(e),(f).}
\label{fdds5_12m2}
\end{figure}

\begin{figure}[htbp]
\centering
\subfigure[]
{\includegraphics[width=\columnwidth]{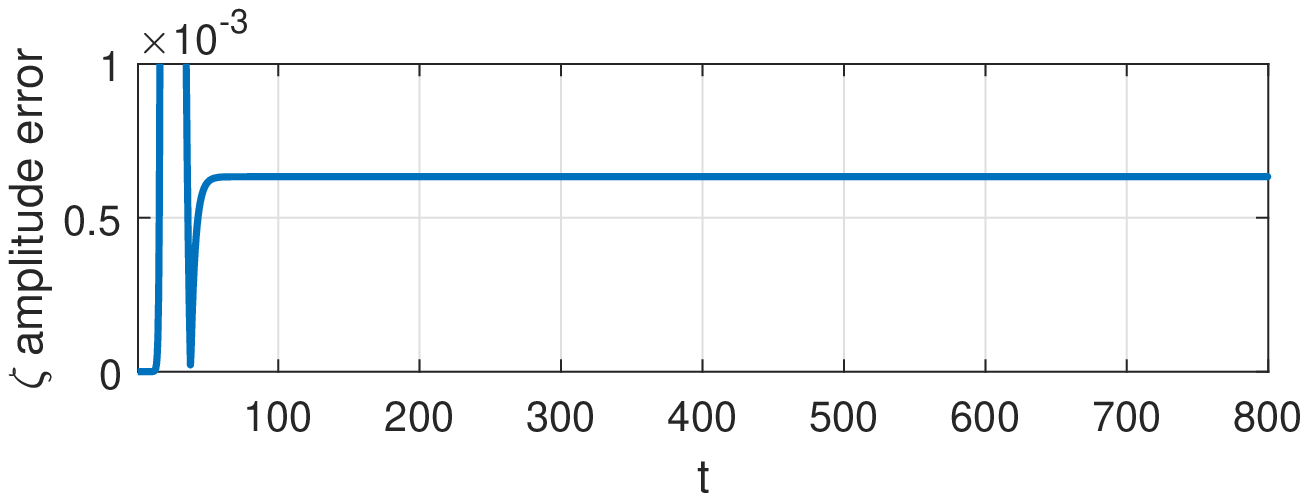}}
\subfigure[]
{\includegraphics[width=\columnwidth]{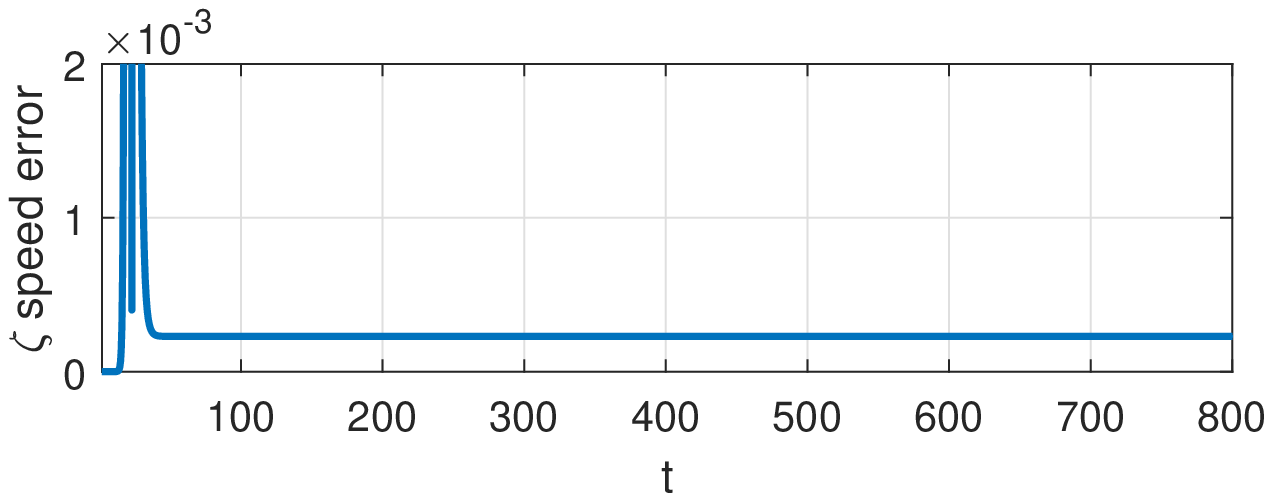}}
\caption{Non-symmetric head-on collision of CSW's. Evolution of amplitude (a), and speed (b) errors of the right-traveling emerging solitary wave ($\zeta$ component of the numerical approximation); cf. Figure \ref{fdds5_12}.}
\label{fdds5_12b}
\end{figure}

\begin{figure}[htbp]
\centering
\subfigure[]
{\includegraphics[width=\columnwidth]{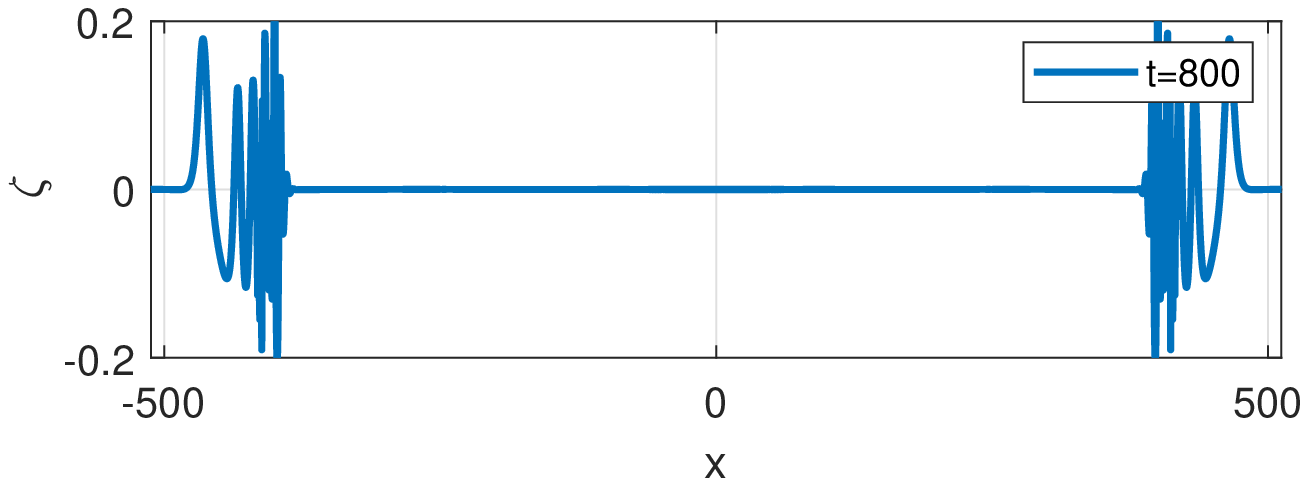}}
\subfigure[]
{\includegraphics[width=\columnwidth]{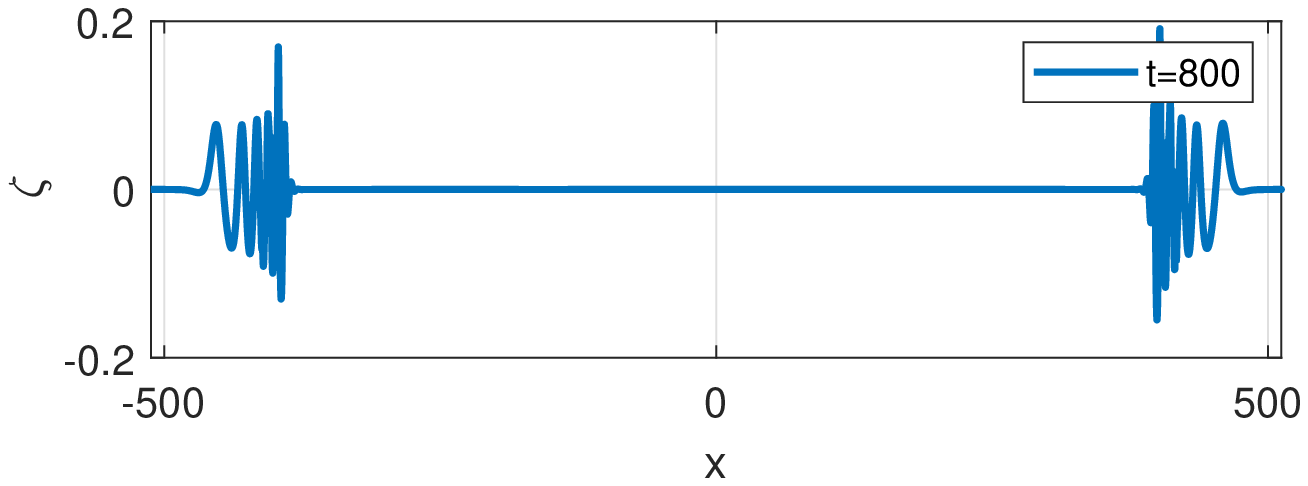}}
\caption{Symmetric vs. non-symmetric head-on collisions of CSW's. (a) Magnification of Figure \ref{fdds5_11m}(d); (b) Magnification of Figure \ref{fdds5_12m}(d).}
\label{fdds5_12mm}
\end{figure}

%\begin{figure}[htbp]
%\centering
%\subfigure[]
%{\includegraphics[width=\columnwidth]{nsymho_amp.eps}}
%\subfigure[]
%{\includegraphics[width=\columnwidth]{nsymho_speed.eps}}
%\caption{Non-symmetric head-on collisions of CSW. Case (A3) with (\ref{52a}). Evolution of amplitude (a) and speed (b) errors of the right-going emerging solitary wave ($\zeta$ component of the numerical solution); cf. Figure \ref{fdds5_12}.}
%\label{fdds5_12b}
%\end{figure}

%\begin{figure}[htbp]
%\centering
%\subfigure[]
%{\includegraphics[width=10cm]{csw6_1.eps}}
%\subfigure[]
%{\includegraphics[width=10cm]{csw6_2.eps}}
%\subfigure[]
%{\includegraphics[width=10cm]{csw6_3.eps}}
%\subfigure[]
%{\includegraphics[width=10cm]{csw6_4.eps}}
%\caption{Non-symmetric head-on collisions of CSW. Case (A3) with (\ref{52a}).  (a)-(c) $\zeta$ component of the numerical solution; (d) Magnification of (c).}
%\label{fdds5_12}
%\end{figure}

The second experiment concerns a non-symmetric head-on collision. Here the initial condition is a superposition of two approximate CSW profiles, one with amplitude $11.101$, speed $c_{s}^{(1)}=c_{\gamma,\delta}+0.5\approx 1.0976$ (traveling to the right) and centered at $x=-20$, and a second one with amplitude $4.8324$, absolute value of speed equal to $c_{s}^{(2)}=c_{\gamma,\delta}+0.25\approx 0.8476$ traveling to the left and centered at $x=20$. The evolution at  several time instances of the corresponding numerical approximation of $\zeta$ is shown in Figures \ref{fdds5_12} and \ref{fdds5_12m}. After the non-symmetric interaction, the taller emerging CSW has smaller amplitude than before the collision (with a relative difference of about $6.3\times 10^{-4}$, see Figure \ref{fdds5_12b}(a)) but the amplitude of the shorter emerging CSW has decreased more (about $5.2\times 10^{-3}$). (Consequently, the emerging CSW's are slower than their corresponding counterparts before the collision, cf. Figure \ref{fdds5_12b}(b)). The tails trailing the emerging waves are not symmetric either and they seem to have a similar structure to those of the symmetric head-on collision shown in Figure \ref{fdds5_11}: A wavelet-type form in front and a strong dispersive component behind. This is shown in Figures \ref{fdds5_12m1} and \ref{fdds5_12m2}. For a comparison with the symmetric case see Figure \ref{fdds5_12mm}.

\subsubsection{Resolution property}
%\begin{figure}[htbp]
%\centering
%\subfigure[]
%{\includegraphics[width=10cm]{csw7_1.eps}}
%\subfigure[]
%{\includegraphics[width=10cm]{csw7_2.eps}}
%\subfigure[]
%{\includegraphics[width=10cm]{csw7_3.eps}}
%\subfigure[]
%{\includegraphics[width=10cm]{csw7_4.eps}}
%\subfigure[]
%{\includegraphics[width=10cm]{csw7_5.eps}}
%\caption{Resolution property from a Gaussian pulse. Case (A3) with (\ref{52a}).  (a)-(d) $\zeta$ component of the numerical solution; (e) Magnification of (d).}
%\label{fdds5_13}
%\end{figure}
In addition to the evidence of generation of more than one CSW observed in the evolution of initial profiles with larger perturbations in section \ref{sec522}, the resolution into solitary waves also appears in the evolution of other types of initial conditions.
This is illustrated in Figures \ref{fdds5_13} and \ref{fdds5_13a}, which represent the temporal behaviour of the $\zeta$-component of the numerical solution emerging from an intial Gaussian pulse $\zeta(x,0)=Ae^{-\tau x^{2}}, u(x,0)=\zeta(x,0)$, with $A=2, \tau=0.01$.
In this example, a train of solitary waves of elevation is formed, followed by a left-traveling dispersive structure, see Figure \ref{fdds5_13c}.
The evolution of the maximum of the $\zeta$-component of the numerical solution is shown in Figure \ref{fdds5_13b}(a). It stabilizes to around $4.1790$, which is the amplitude of the leading solitary wave profile of the train. The speed, computed from the point where the maximum is attained, cf. \cite{DougalisDLM2007}, is shown in Figure \ref{fdds5_13b}(a), and is about $8.1779\times 10^{-1}$.

\begin{figure}[htbp]
\centering
\subfigure[]
{\includegraphics[width=\columnwidth]{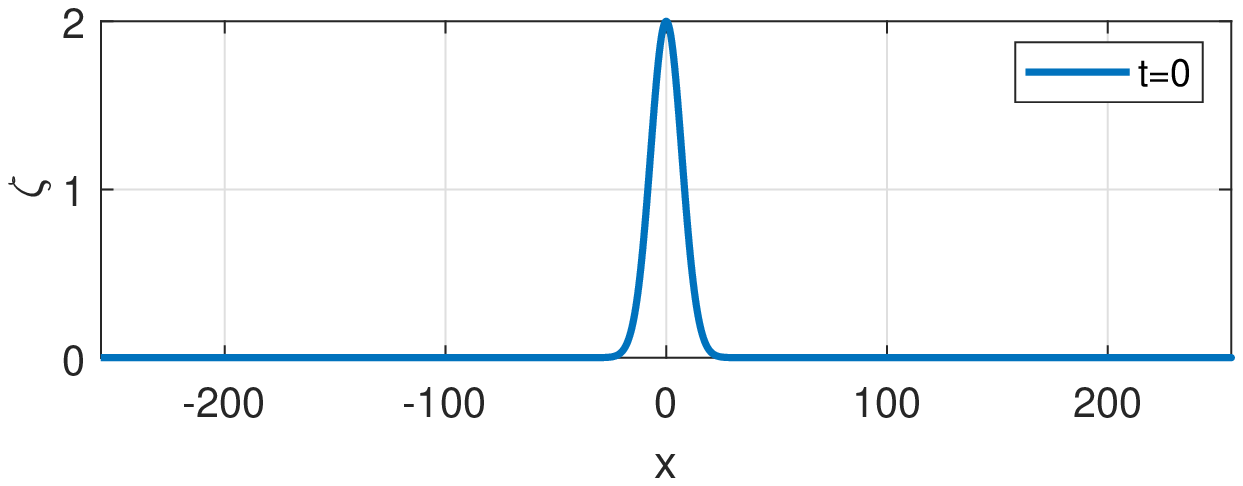}}
%\subfigure[]
%{\includegraphics[width=\columnwidth]{overcol_t200.eps}}
\subfigure[]
{\includegraphics[width=\columnwidth]{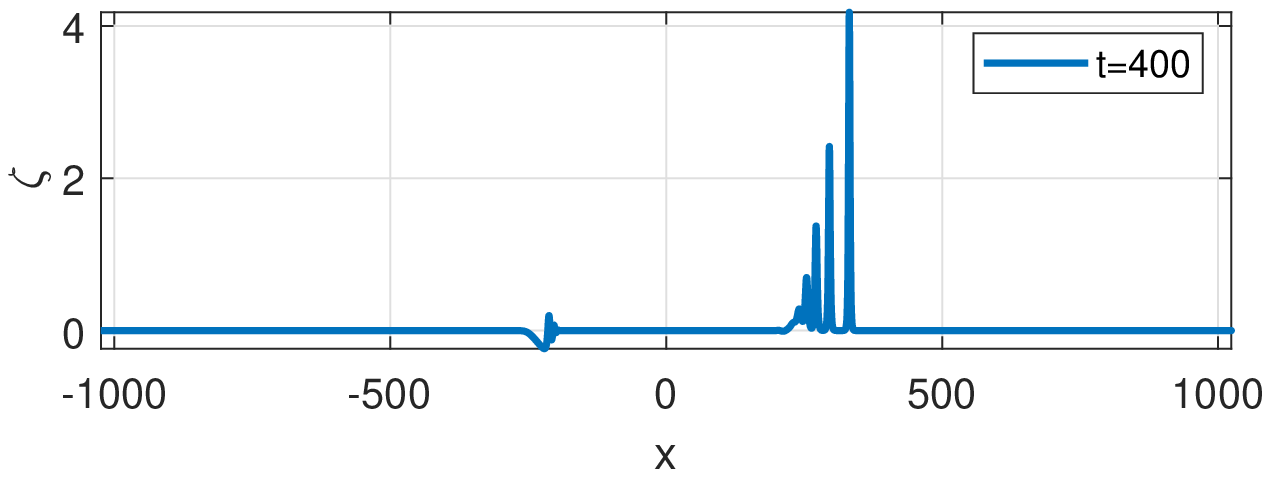}}
%\subfigure[]
%{\includegraphics[width=\columnwidth]{overcol_t600.eps}}
\subfigure[]
{\includegraphics[width=\columnwidth]{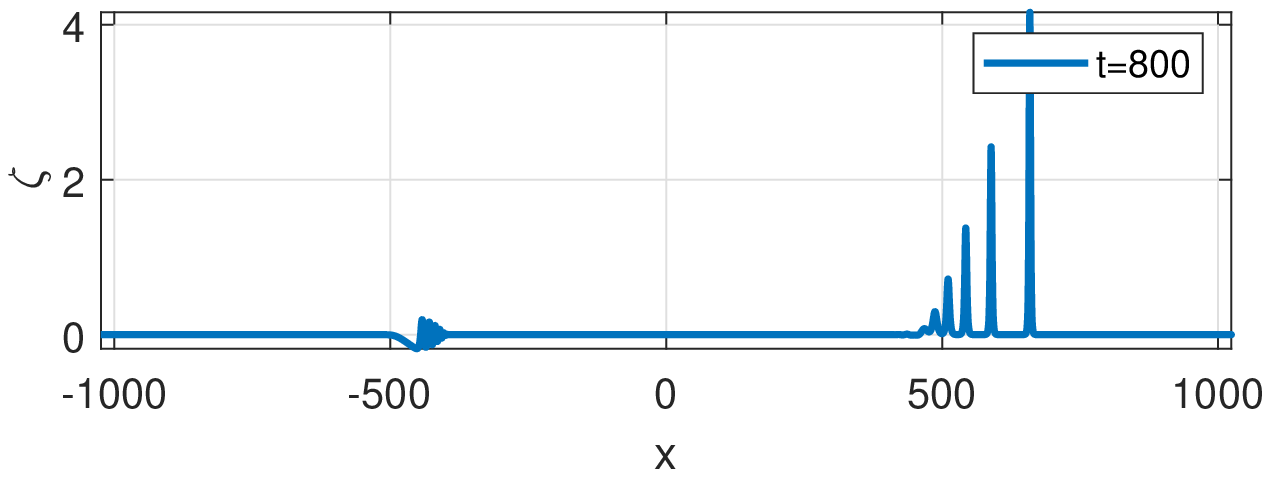}}
%\subfigure[]
%{\includegraphics[width=\columnwidth]{csw4_4.eps}}
\caption{Resolution property. Initial Gaussian pulse $\zeta(x,0)=Ae^{-\tau x^{2}}, u(x,0)=\zeta(x,0)$, with $A=2, \tau=0.01$. Case (A3).  (a)-(c) $\zeta$ component of the numerical solution.}
\label{fdds5_13}
\end{figure}

\begin{figure}[htbp]
\centering
\subfigure[]
{\includegraphics[width=\columnwidth]{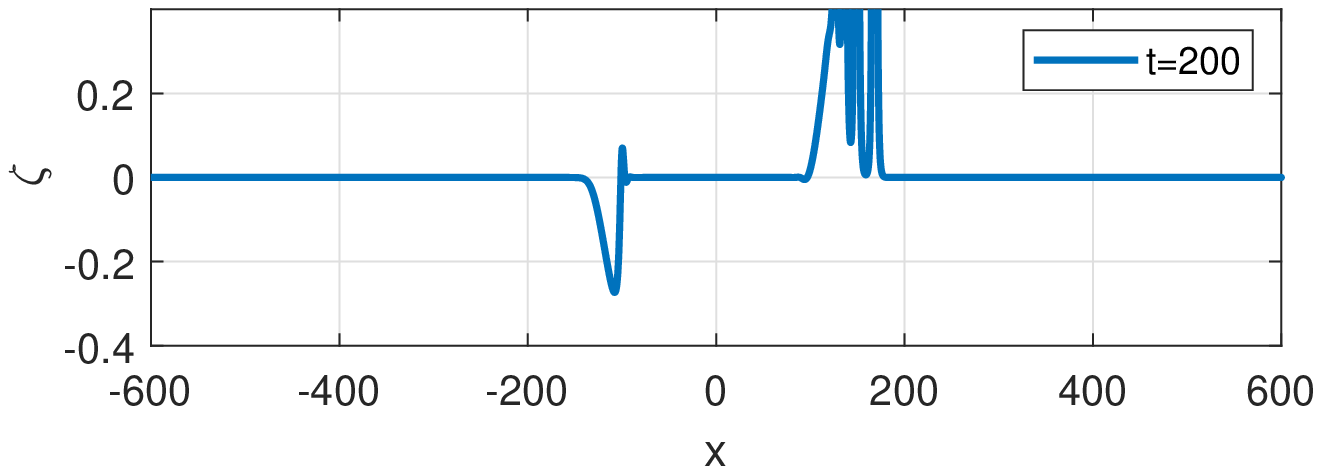}}
\subfigure[]
{\includegraphics[width=\columnwidth]{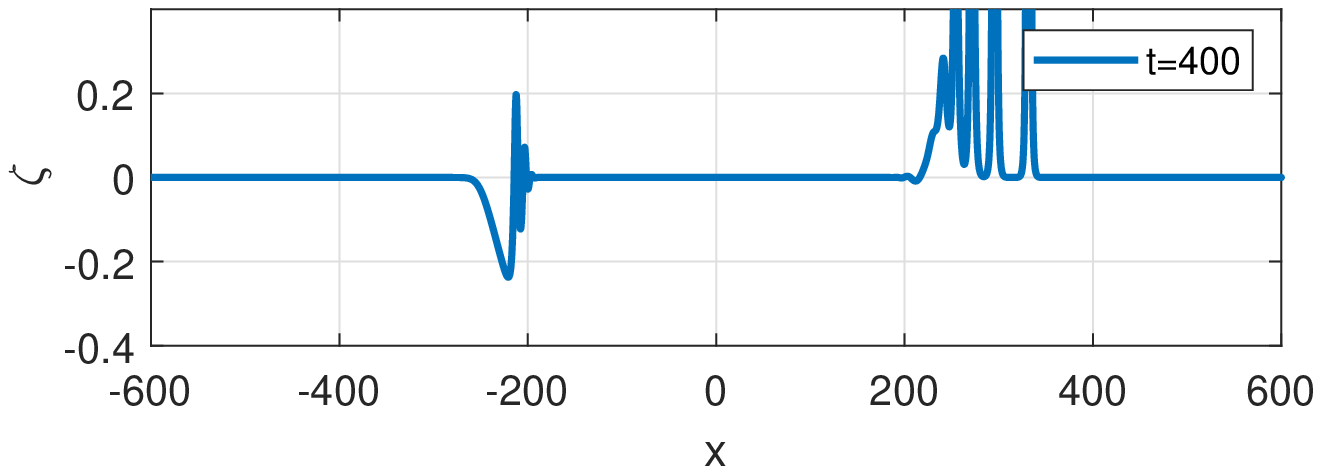}}
\subfigure[]
{\includegraphics[width=\columnwidth]{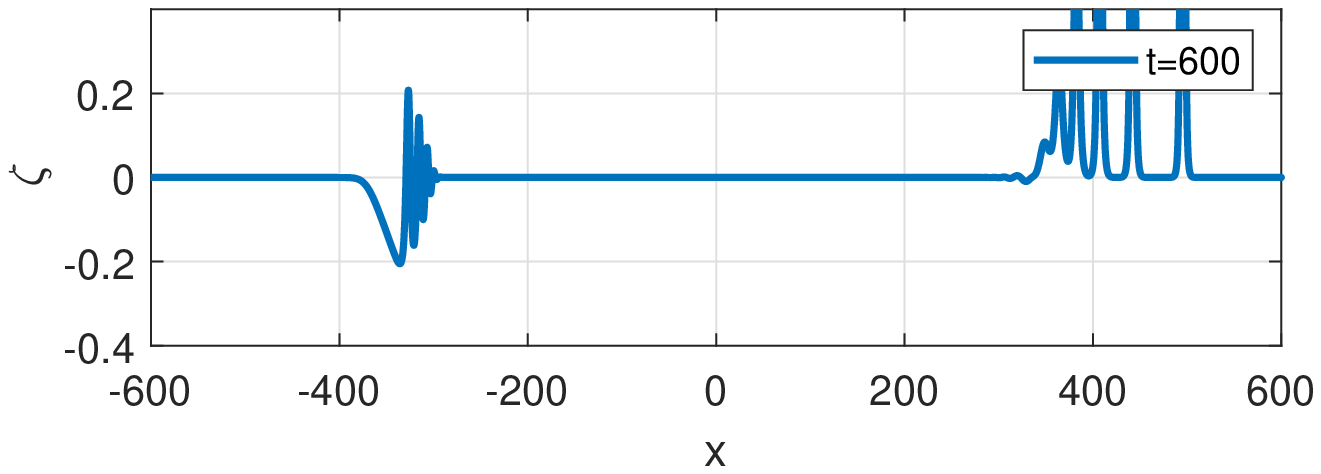}}
\subfigure[]
{\includegraphics[width=\columnwidth]{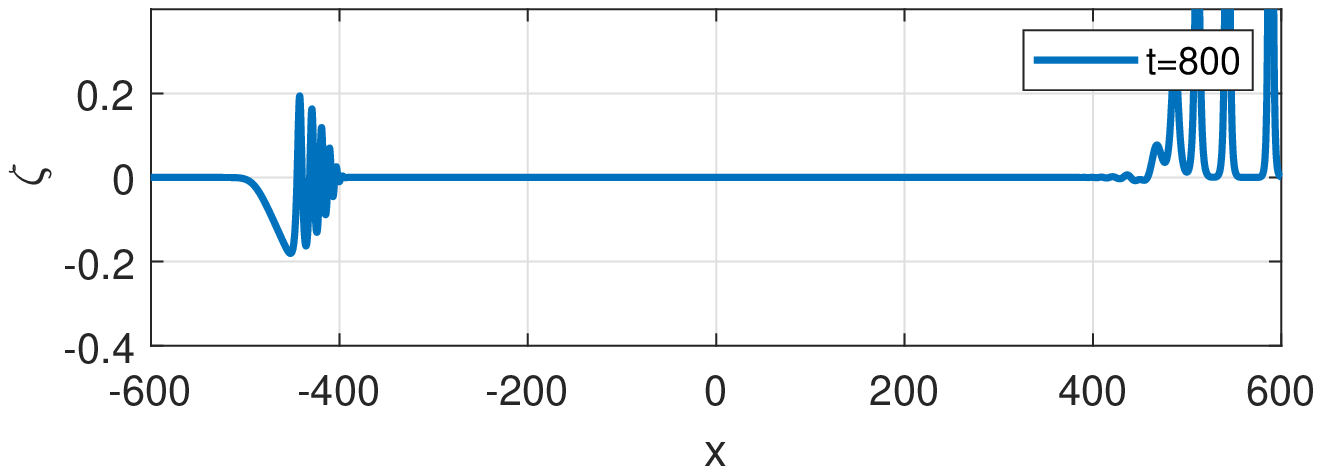}}
\caption{Resolution property. Initial Gaussian pulse $\zeta(x,0)=Ae^{-\tau x^{2}}, u(x,0)=\zeta(x,0)$, with $A=2, \tau=0.01$. Case (A3).  (a)-(c) $\zeta$ component of the numerical solution (magnified).}
\label{fdds5_13a}
\end{figure}

\begin{figure}[htbp]
\centering
{\includegraphics[width=\columnwidth]{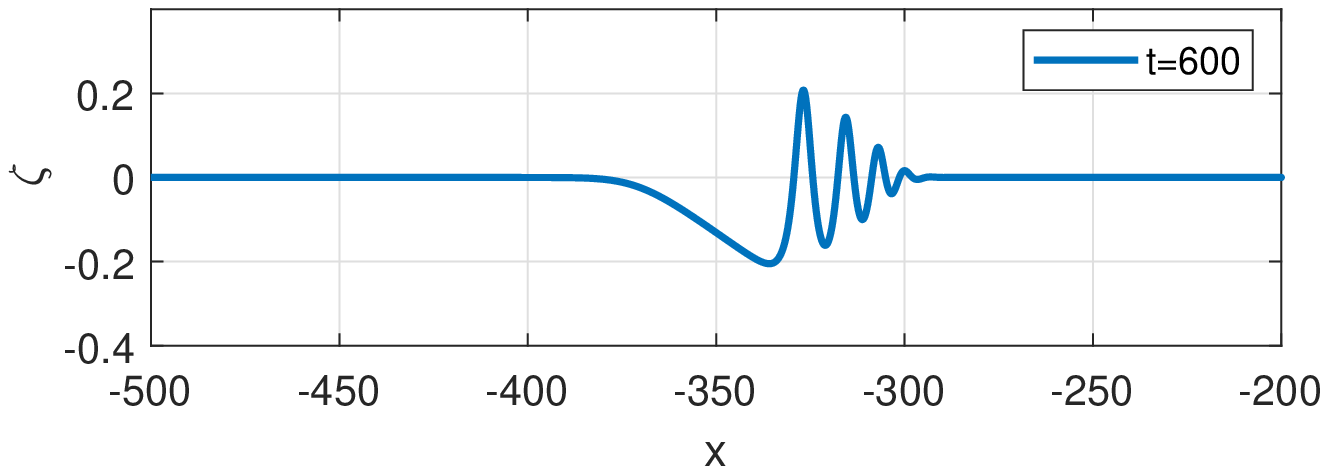}}
\subfigure[]
{\includegraphics[width=\columnwidth]{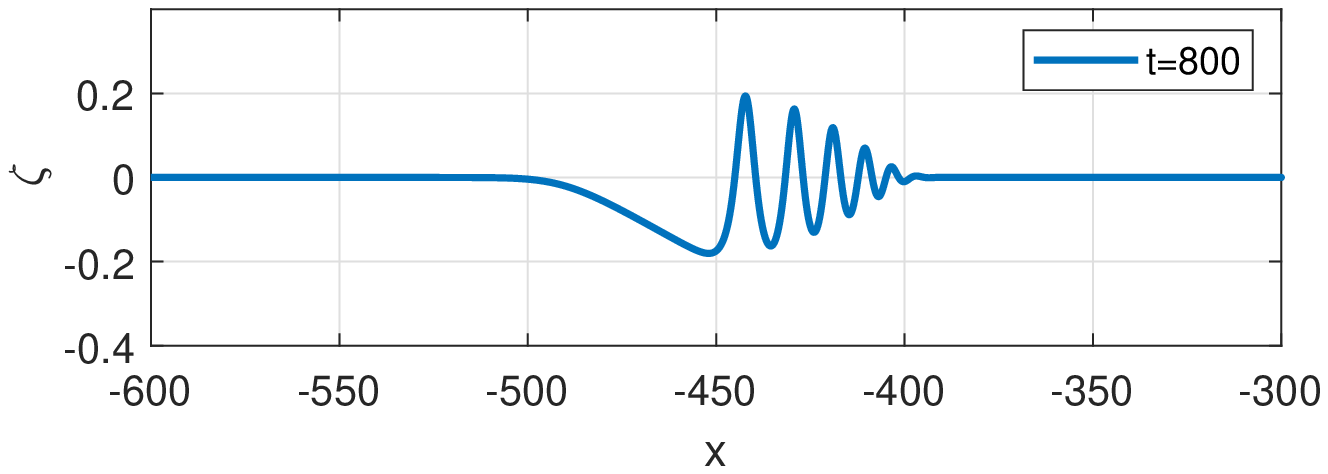}}
\caption{Resolution property. Initial Gaussian pulse $\zeta(x,0)=Ae^{-\tau x^{2}}, u(x,0)=\zeta(x,0)$, with $A=2, \tau=0.01$. Case (A3).  $\zeta$ component of the numerical solution at (a) $t=600$; (b) $t=800$.}
\label{fdds5_13c}
\end{figure}

\begin{figure}[htbp]
\centering
\subfigure[]
{\includegraphics[width=\columnwidth]{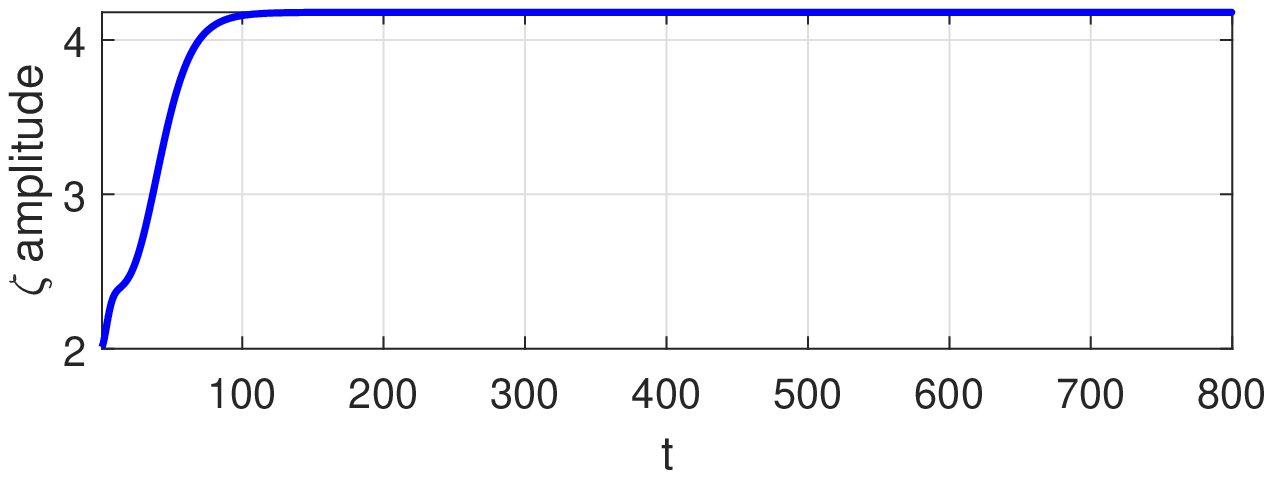}}
\subfigure[]
{\includegraphics[width=\columnwidth]{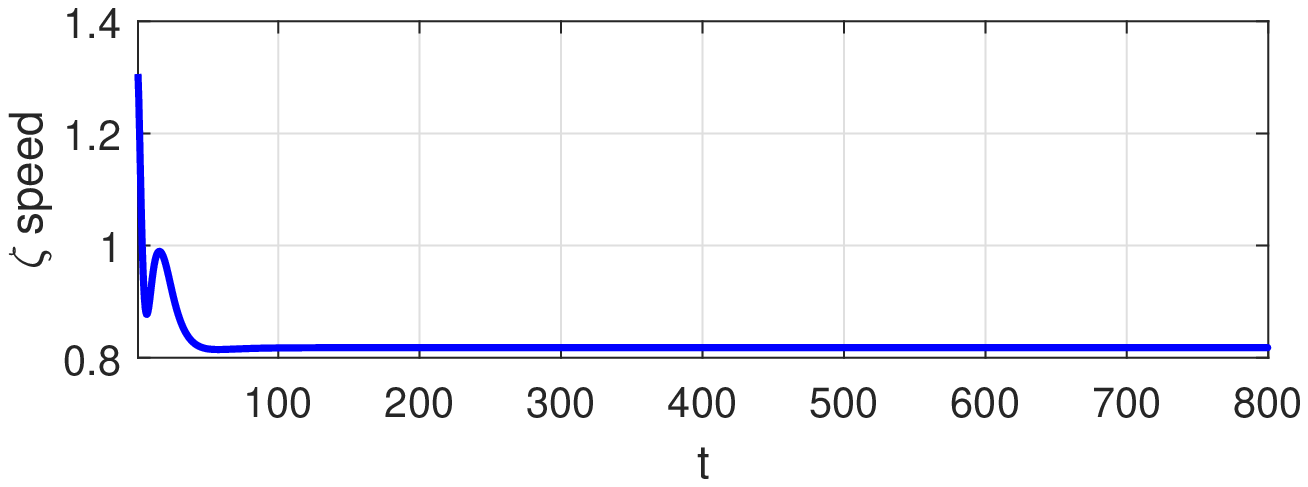}}
\caption{Resolution property. Initial Gaussian pulse $\zeta(x,0)=Ae^{-\tau x^{2}}, u(x,0)=\zeta(x,0)$, with $A=2, \tau=0.01$. Case (A3).  Evolution of amplitude (a) and speed (b) of the taller emerging solitary wave in the $\zeta$-component of the numerical solution; cf. Figure \ref{fdds5_13}.}
\label{fdds5_13b}
\end{figure}
%\begin{figure}[htbp]
%\centering
%\subfigure[]
%{\includegraphics[width=10cm]{csw7_1.eps}}
%\subfigure[]
%{\includegraphics[width=10cm]{csw7_2.eps}}
%\subfigure[]
%{\includegraphics[width=10cm]{csw7_3.eps}}
%\subfigure[]
%{\includegraphics[width=10cm]{csw7_4.eps}}
%\caption{Resolution property from a Gaussian pulse. Case (A3).  (a)-(d) $\zeta$ component of the numerical solution.}
%\label{fdds5_13}
%\end{figure}
%\begin{figure}[htbp]
%\centering
%\subfigure[]
%{\includegraphics[width=10cm]{csw7_5.eps}}
%\caption{Resolution property from a Gaussian pulse. Case (A3).  Magnification of (d) of Figure \ref{fdds5_13}.}
%\label{fdds5_13a}
%\end{figure}

\subsection{GSW dynamics. Numerical experiments}
\label{sec54}
In this section we illustrate some aspects of the dynamics of GSW's in the generic case (A2). As in the previous section, the experiments are concerned with perturbations and collisions of GSW's.
\subsubsection{Perturbations of GSW}
\begin{figure}[htbp]
\centering
\subfigure[]
{\includegraphics[width=\columnwidth]{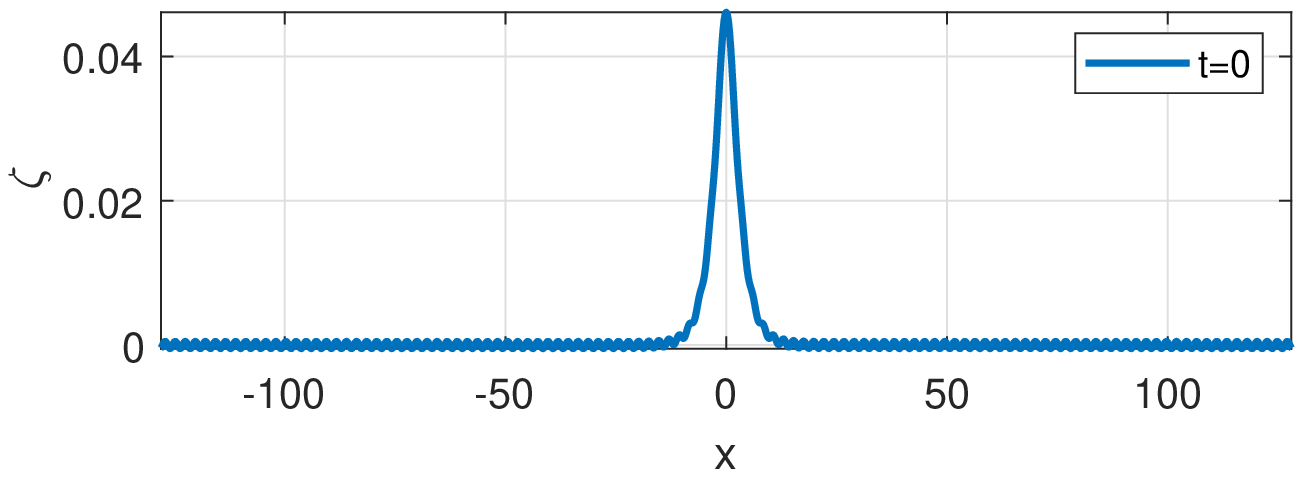}}
\subfigure[]
{\includegraphics[width=\columnwidth]{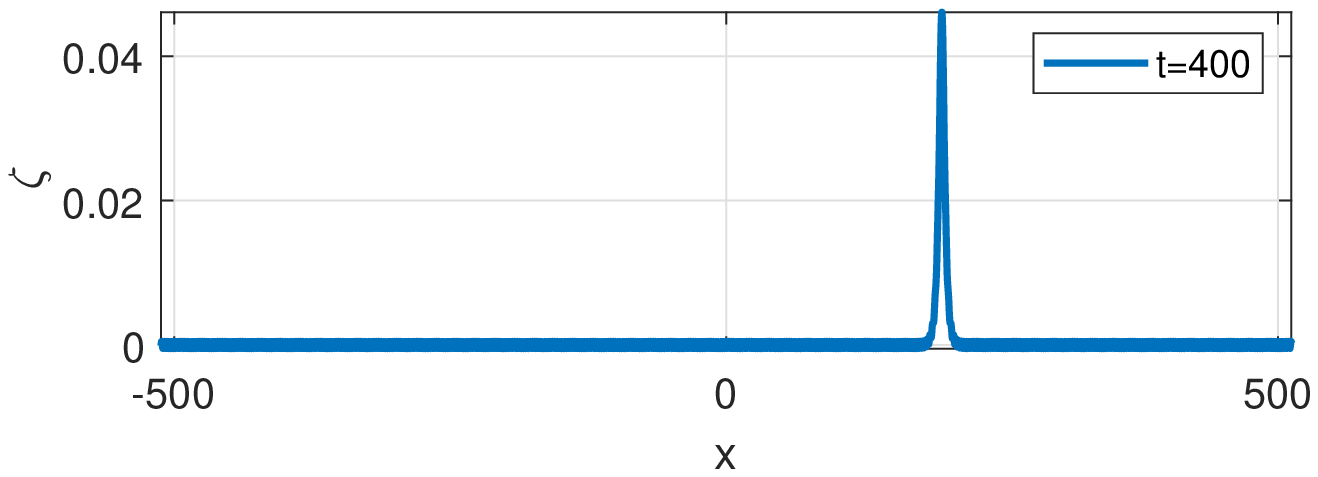}}
\subfigure[]
{\includegraphics[width=\columnwidth]{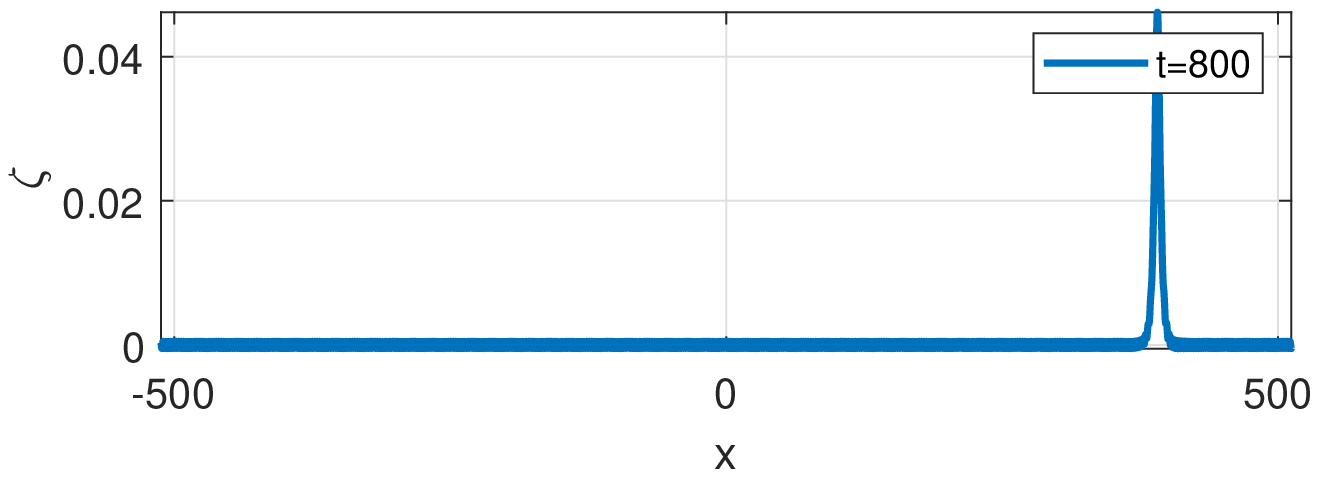}}
%\subfigure[]
%{\includegraphics[width=10cm]{gsw8_4.eps}}
%\subfigure[]
%{\includegraphics[width=10cm]{gsw8_5.eps}}
\caption{Small perturbation of a GSW. Case (A2) with (\ref{53a}), (\ref{53b}) with $A=1.01$.  (a) Perturbed GSW profile; (b), (c) $\zeta$ component of the numerical solution.}
\label{fdds5_14}
\end{figure}

%\subsubsection{Perturbations of GSW}
We consider the parameters
\begin{eqnarray}
&&\gamma=0.5, \delta=\frac{\gamma+\sqrt{\gamma^{2}+8}}{2}\approx 1.6861,\nonumber\\
&& c=-1/6, b=1/9, d=4/3, a={\kappa_{1}}(1/6-b-c-d)\approx -0.5083,\label{53a}
\end{eqnarray}
and generate the corresponding approximate GSW, with amplitude $4.5728\times 10^{-2}$ and speed $4.8824\times 10^{-1}$. (The values of $\gamma$ and $\delta$ in (\ref{53a}) are taken appropriately to ensure that the parameters $a, b, c$ and $d$ satisfy the conditions in (A2) in Table \ref{tavle0}.) The first experiment consists of perturbing the two components $(\zeta^{N}_{s},u_{s}^{N})$ of the GSW with the same quantity $A=1.01$, i.~e. considering 
\begin{eqnarray}
\zeta^{N}(0)=A\zeta_{s}^{N},\; u^{N}(0)=Au_{s}^{N},\label{53b}
\end{eqnarray}
as initial condition of the numerical method and monitoring the evolution of the corresponding numerical solution. This is shown in Figure \ref{fdds5_14}. This small perturbation of the GSW generates a new GSW.

\begin{figure}[htbp]
\centering
\subfigure[]
{\includegraphics[width=\columnwidth]{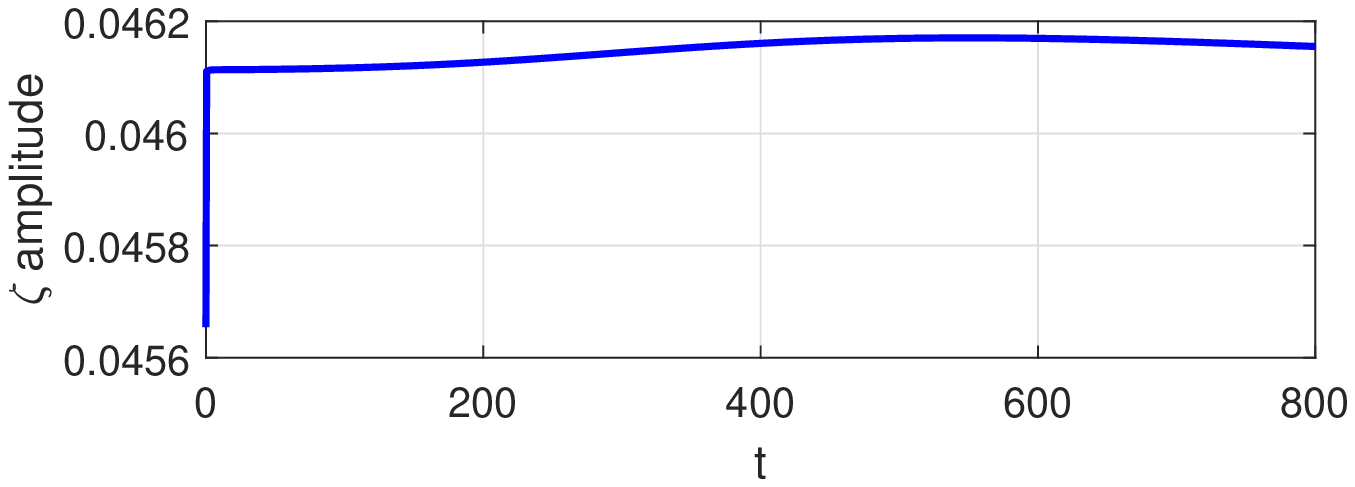}}
\subfigure[]
{\includegraphics[width=\columnwidth]{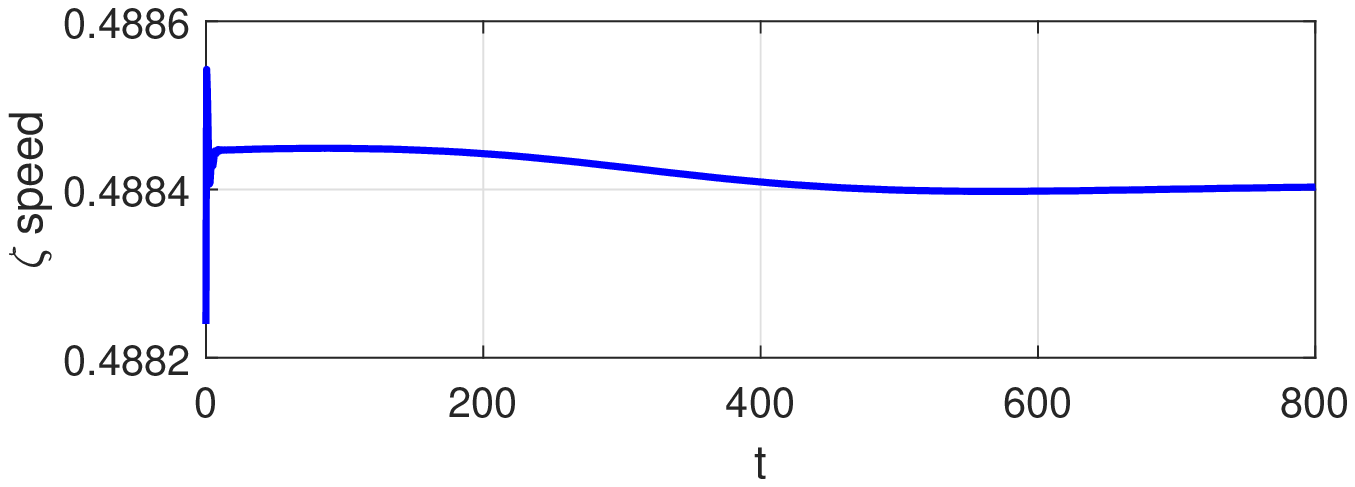}}
\subfigure[]
{\includegraphics[width=\columnwidth]{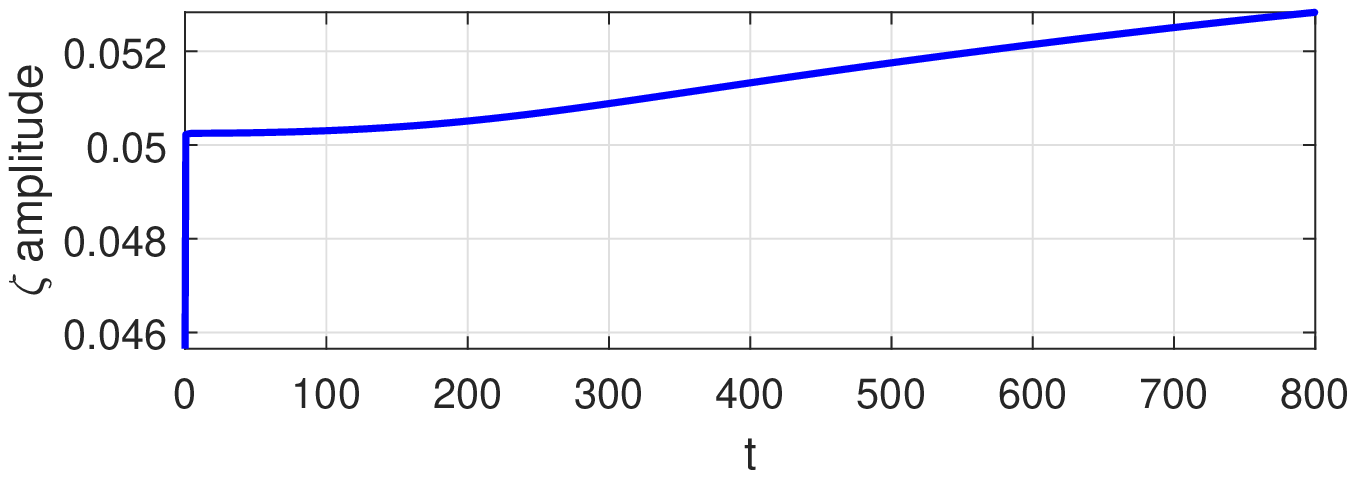}}
\subfigure[]
{\includegraphics[width=\columnwidth]{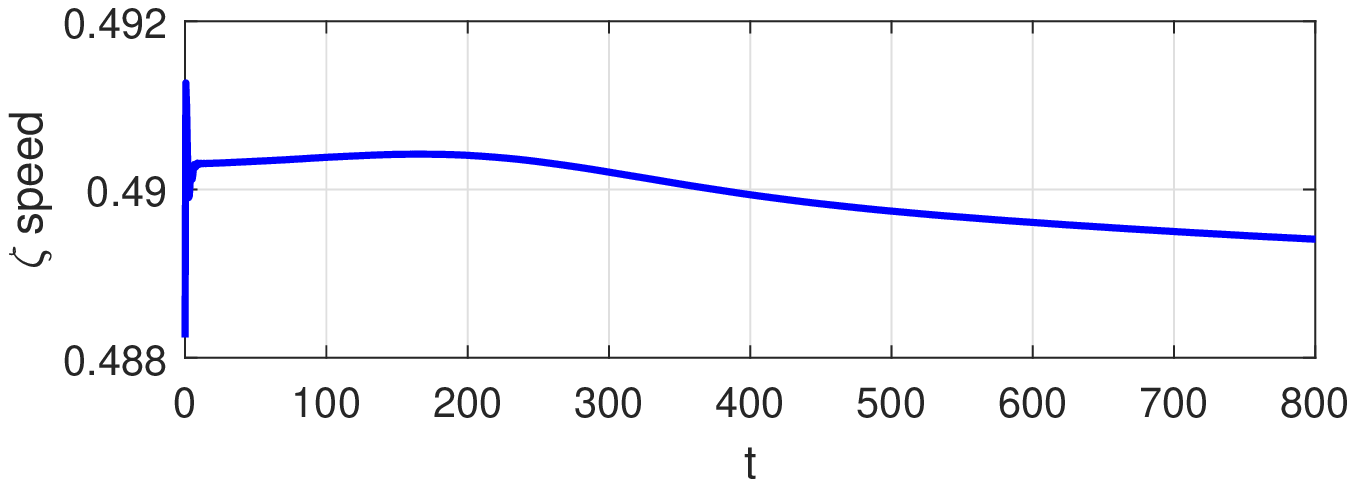}}
\caption{Small perturbation of a GSW. Case (A2) with (\ref{53a}), (\ref{53b}). Evolution of amplitude and speed of the tallest emerging solitary wave ($\zeta$ component of the numerical solution); (a), (b) $A=1.01$; (c), (d) $A=1.1$.}
\label{fdds5_15}
\end{figure}

The evolution of the amplitude and speed of the emerging wave, shown in Figures \ref{fdds5_15}(a),(b), suggests that its parameters stabilize at slightly larger values than those of the initial condition. Specifically, the amplitude of the perturbed initial GSW is of about $4.612\times 10^{-2}$ and that of emerging GSW is between $4.61\times 10^{-2}$ and $4.62\times 10^{-2}$. In the case of the speed, the relative difference is about $3.26\times 10^{-4}$.
The structure of the ripples appears to be the same. Since we expect that the small perturbation will generate some sort of dispersion, for this experiment these are apparently of very small size and are probably hidden in the ripples.

The study of the structure of dispersive tails via the linearized system (\ref{53}), (\ref{54}), made in section \ref{sec53} for the generic case (A3) of classical solitary waves can be adapted to the generic case (A2) for generalized solitary waves. Now we have $p_{2}>0$ in (\ref{514c}) and the cases (i)-(iii) are as follows:
\begin{itemize}
\item[(i)] $p_{1}<0$. Compared to the analogous case in section \ref{sec53}, the only change occurs when $\Delta=0$. Then $P(x)$ has a double root at $x=x_{*}=-p_{2}/2<0$; therefore,  $\phi(x)$ is decreasing for $x\geq 0$.
\item[(ii)] $p_{1}=0$. Here $P(x)$ has only one root $x=x_{*}=p_{3}/p_{2}>0$ and therefore:
\begin{enumerate}
\item If $0\leq x\leq x_{*}$, then $\phi(x)$ is decreasing with $1>\phi(x)\geq \phi(x_{*})$.
\item If $x_{*}\leq x$, then $\phi(x)$ is increasing with $\phi(x_{*})\leq \phi(x)\leq \phi_{*}<1$.
\end{enumerate}
\item[(iii)] $p_{1}>0$.  Then $\Delta>0$ and $P$ has two simple roots, one positive and one negative. The same behaviour as in section \ref{sec53} follows.
\end{itemize}
On the other hand, the case (A1) is different. Now we have
\begin{eqnarray*}
\phi(x)=\sqrt{\frac{(1-\widetilde{a}x)(1-cx)}{1+dx}},\; \widetilde{a}=\frac{a}{\kappa_{1}},\; x\geq 0,
\end{eqnarray*}
and $p_{1}=\widetilde{a}cd>0$ in (\ref{514c}). The case (iii) of section \ref{sec53} applies but note that now $\phi(x)\rightarrow +\infty$ as $x\rightarrow +\infty$ (see Figure \ref{Z1}(a)). Therefore $\phi(x)>1$ for large enough $x>0$. This means that for sufficiently large $|k|$ and from (\ref{514b}) we have
\begin{eqnarray*}
-c_{s}-c_{\gamma,\delta}<v_{-}(k)<-c_{s}<v_{+}(k),
\end{eqnarray*}
and $v_{+}(k)>-c_{s}+c_{\gamma,\delta}$. This implies the existence of plane wave components of small amplitude traveling to the right and in front of the GSW. Similarly, the formation of two dispersive groups, one traveling to the left behind the solitary wave and one to the right in front of it, is suggested by the form of the function $\psi$, given by (\ref{514e}), for the case (A1), and displayed in Figure \ref{Z1}(b). 

\begin{figure}[htbp]
\centering
\subfigure[]
{\includegraphics[width=\columnwidth]{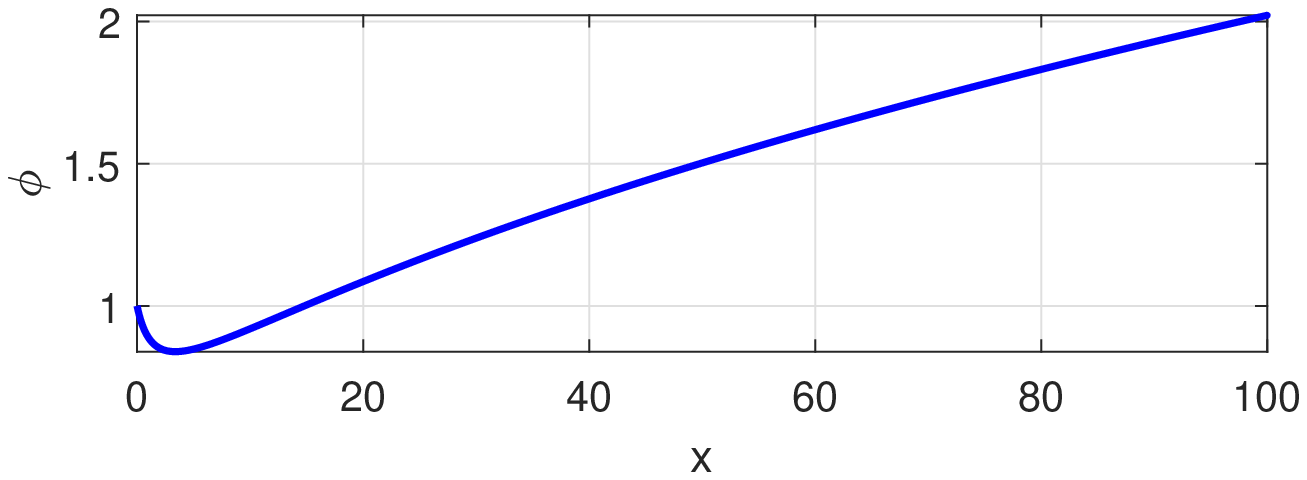}}
\subfigure[]
{\includegraphics[width=\columnwidth]{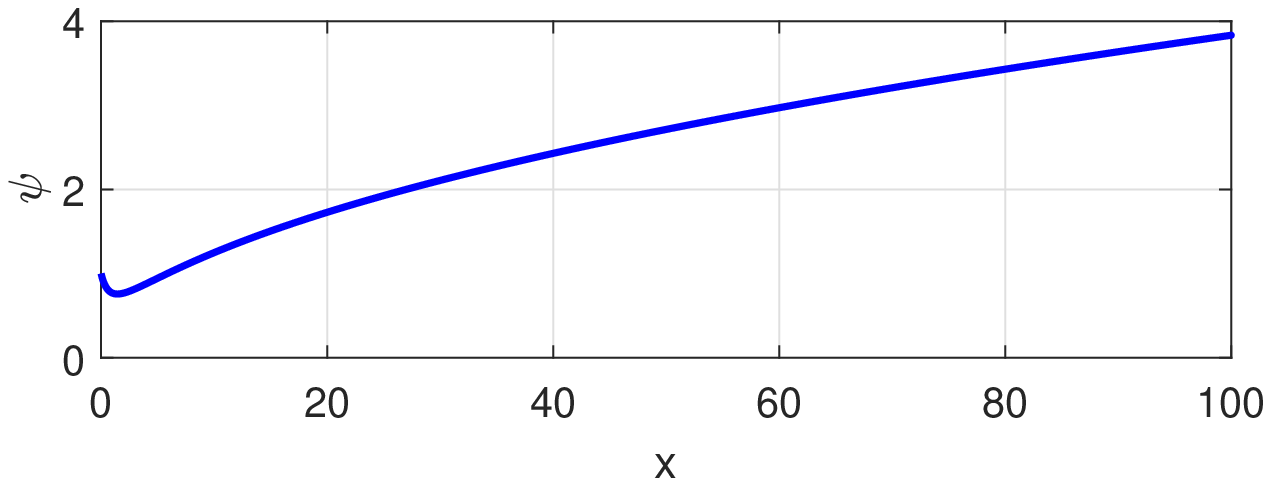}}
\caption{Form of the functions (a) $\phi(x)$ and (b) $\psi(x)$ for the case (A1).}
\label{Z1}
\end{figure}

The range of the size of the perturbations from which the GSW evolves in a stable way seem to be smaller than those for CSW. Figures \ref{fdds5_15}(c),(d) show the evolution of amplitude and speed, respectively, of the numerical solution from a perturbed GSW with (\ref{53a}), (\ref{53b}) and $A=1.1$. Note that by $t=800$ the parameters have not stabilized. 
\begin{figure}[htbp]
\centering
\subfigure[]
{\includegraphics[width=\columnwidth]{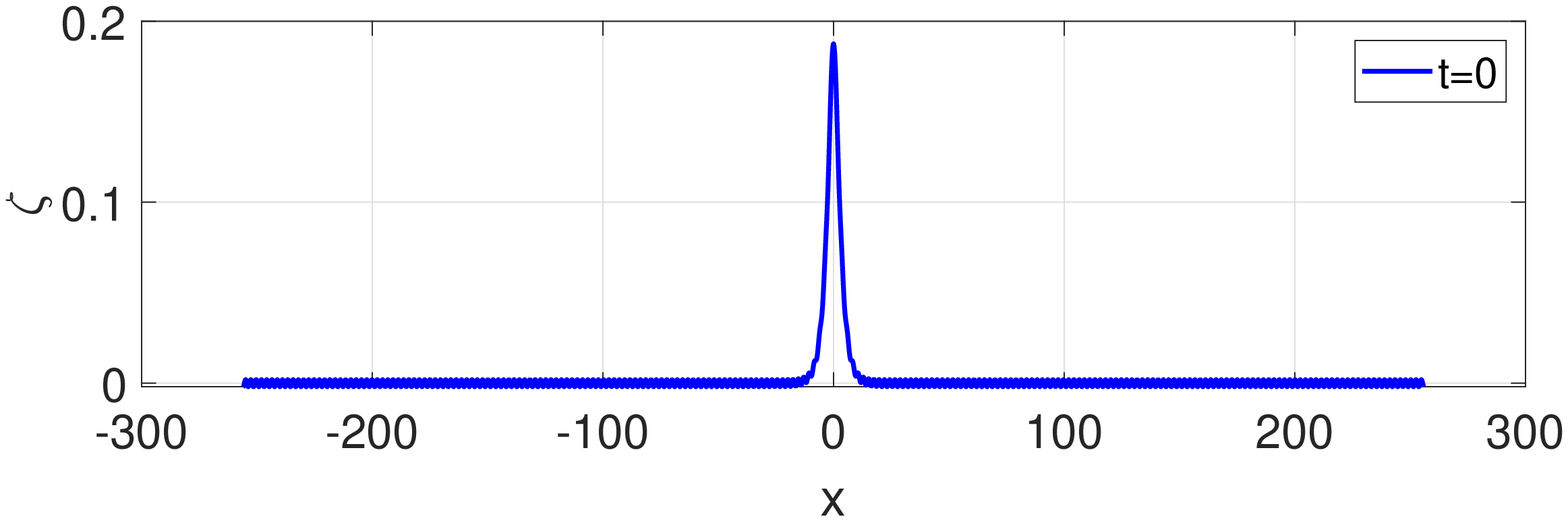}}
\subfigure[]
{\includegraphics[width=\columnwidth]{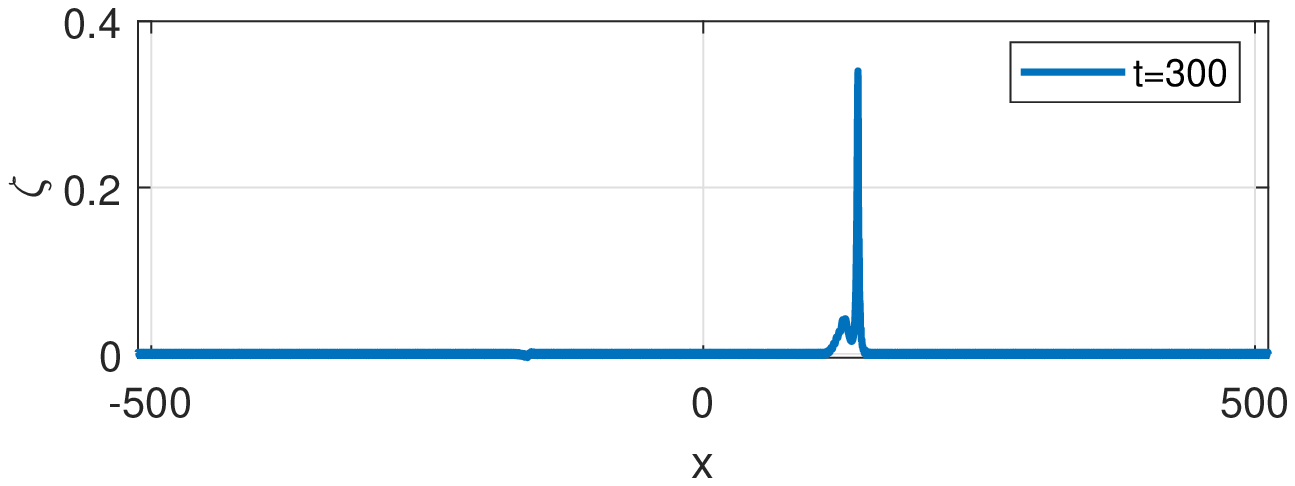}}
\subfigure[]
{\includegraphics[width=\columnwidth]{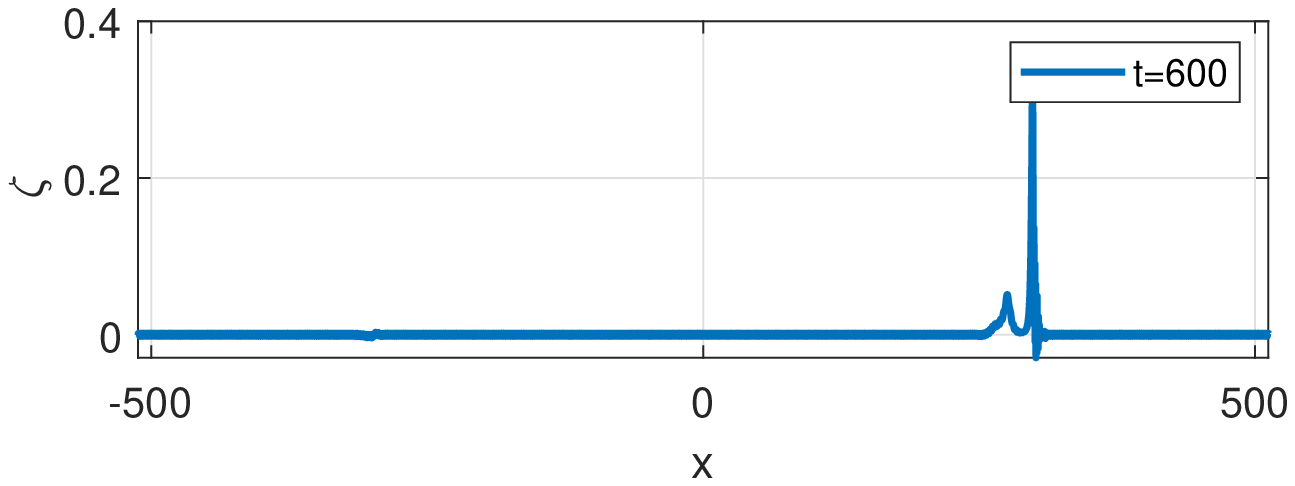}}
\subfigure[]
{\includegraphics[width=\columnwidth]{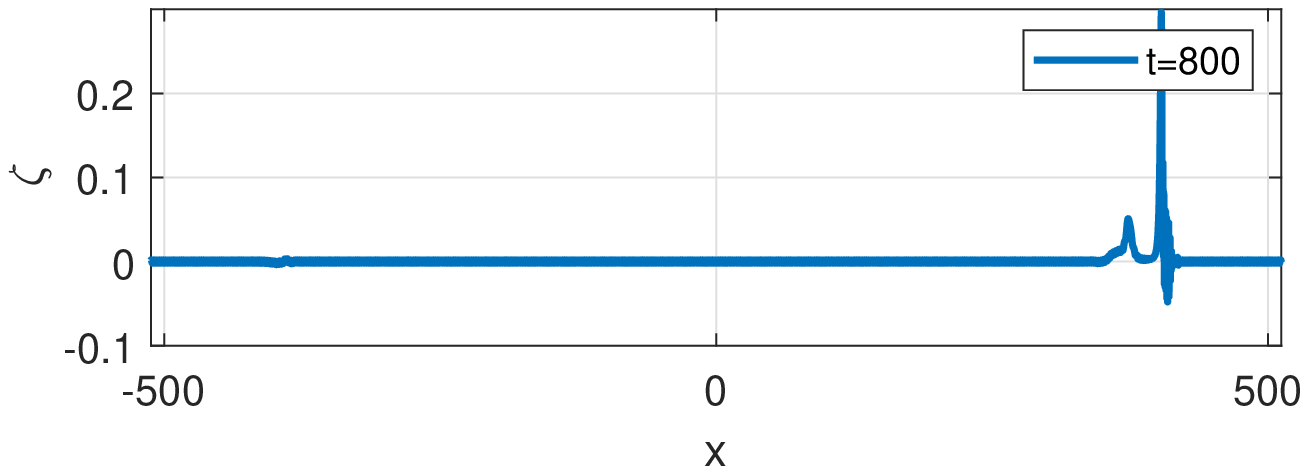}}
%\subfigure[]
%{\includegraphics[width=10cm]{gsw9_1.eps}}
%\subfigure[]
%{\includegraphics[width=10cm]{gsw9_3.eps}}
%\subfigure[]
%{\includegraphics[width=10cm]{gsw9_5.eps}}
%\subfigure[]
%{\includegraphics[width=10cm]{gsw9_5m.eps}}
\caption{Large perturbation of a GSW. Case (A2) with (\ref{53a}), (\ref{53b}) with $A=4.1$. (a) Perturbed GSW profile; (b), (c) $\zeta$ component of the numerical solution.}
\label{fdds5_16}
\end{figure}

As the perturbation parameter $A$ grows, new phenomena in the dynamics appear. Figures \ref{fdds5_16} and \ref{fdds5_16b} show the evolution of the numerical solution generated  from an initial GSW profile of the system with parameters given by (\ref{53a}) multiplied by a perturbation factor $A=4.1$ as in(\ref{53b}) . The experiment suggests (cf. Figure \ref{fdds5_16b}) that the perturbed initial GSW evolves into a new GSW, although by the final time of simulation ($t=800$), Figure \ref{fdds5_17} shows that the amplitude and speed do not seem to have completely stabilized. Behind the main wave, similar structures to those observed in the case of large perturbations of CSW's (cf. Figures \ref{fdds5_4} and \ref{fdds5_6}) seem to be generated, superimposed on the ripples. They are observed in Figures \ref{fdds5_16c}(a),(b), By $t=400$, some perturbation tails have formed in front of the main emerging wave, see Figure \ref{fdds5_16c}(c). This fact and the behaviour of the amplitude and speed of the main pulse observed in Figure \ref{fdds5_17} suggest a possible instability.

\begin{figure}[htbp]
\centering
\subfigure[]
{\includegraphics[width=\columnwidth]{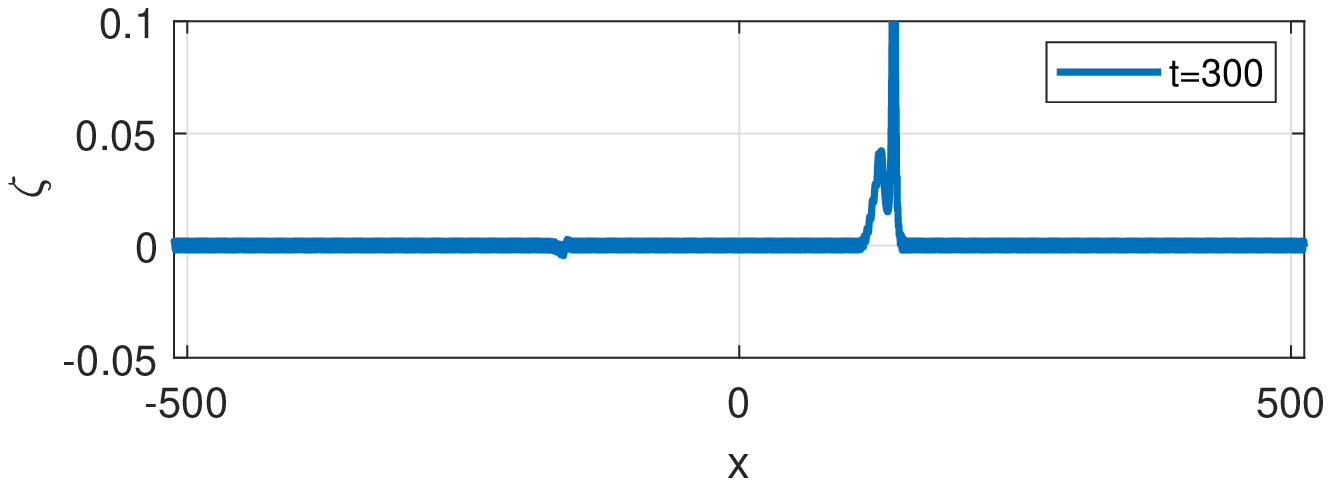}}
\subfigure[]
{\includegraphics[width=\columnwidth]{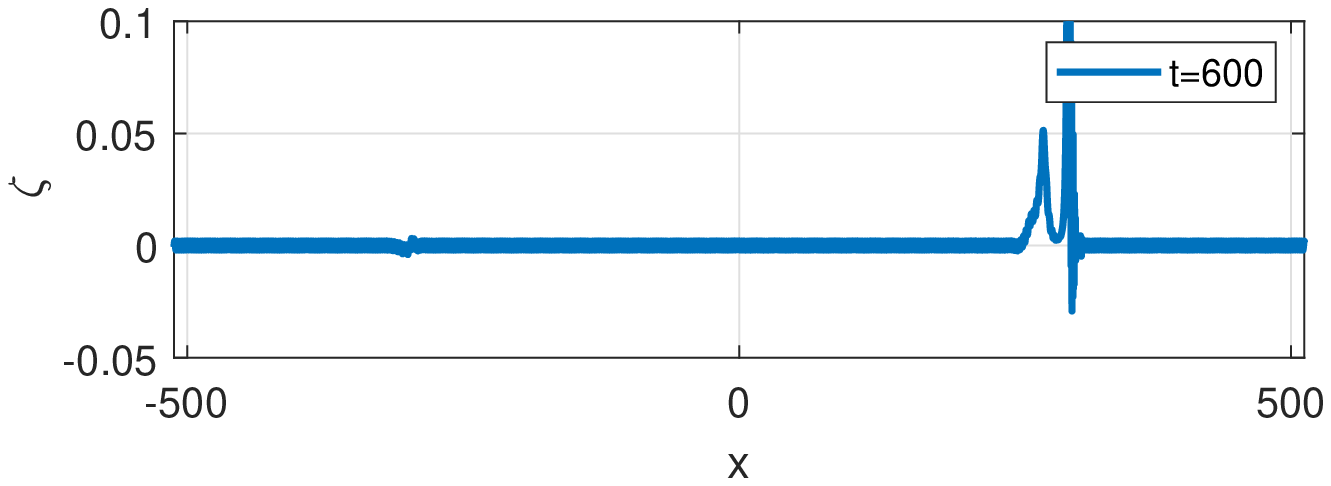}}
\subfigure[]
{\includegraphics[width=\columnwidth]{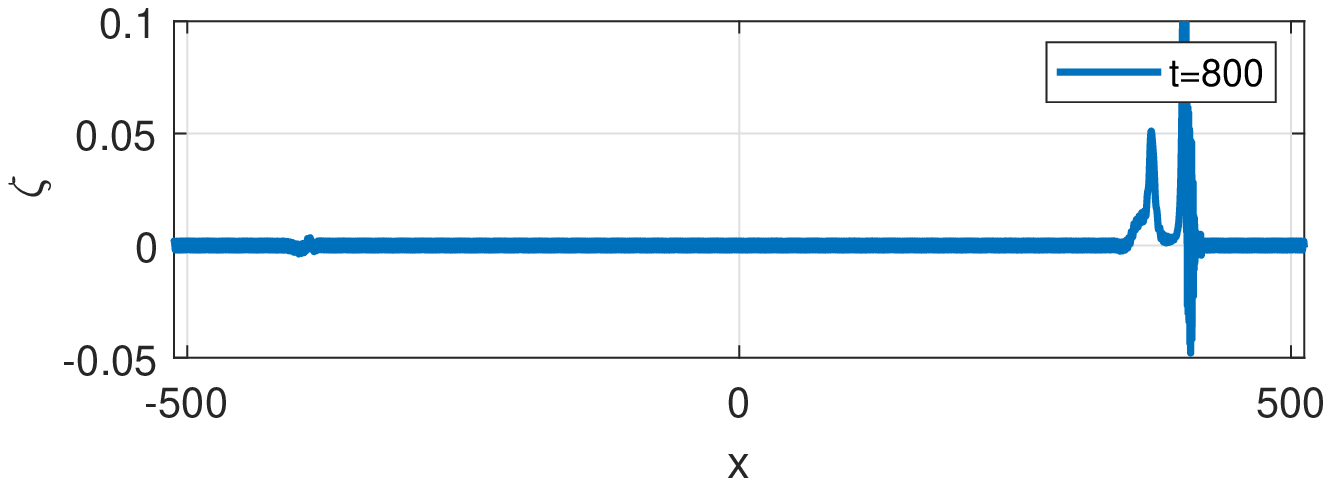}}
%\subfigure[]
%{\includegraphics[width=\columnwidth]{bb_gsw2_t800.eps}}
%\subfigure[]
%{\includegraphics[width=10cm]{gsw9_1.eps}}
%\subfigure[]
%{\includegraphics[width=10cm]{gsw9_3.eps}}
%\subfigure[]
%{\includegraphics[width=10cm]{gsw9_5.eps}}
%\subfigure[]
%{\includegraphics[width=10cm]{gsw9_5m.eps}}
\caption{Large perturbation of a GSW. Case (A2) with (\ref{53a}), (\ref{53b}) with $A=4.1$. Magnification of Figures \ref{fdds5_16}(b)-(d).}
\label{fdds5_16b}
\end{figure}

\begin{figure}[htbp]
\centering
\subfigure[]
{\includegraphics[width=\columnwidth]{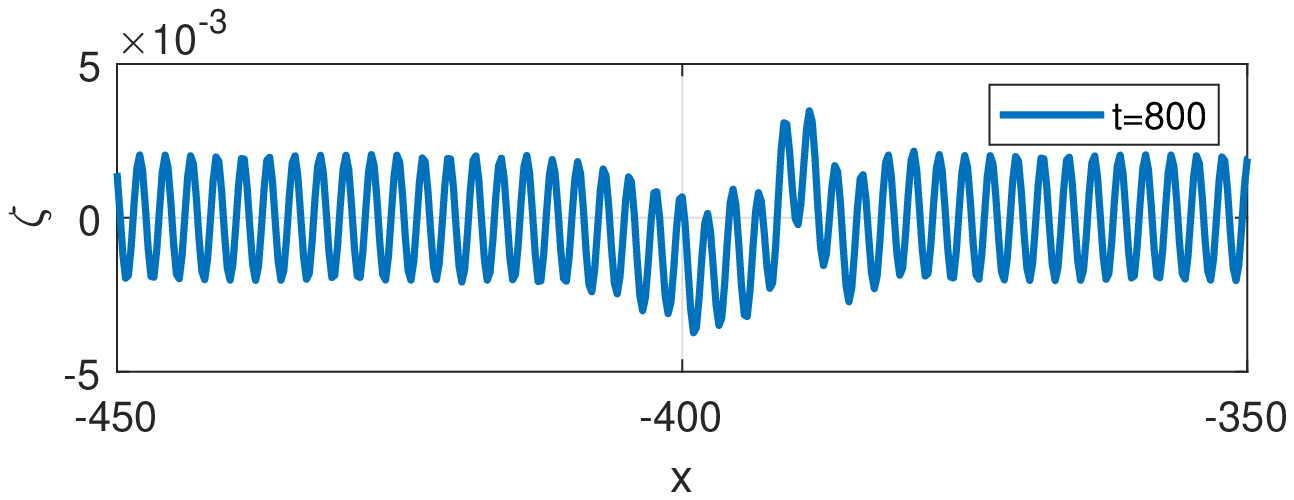}}
\subfigure[]
{\includegraphics[width=\columnwidth]{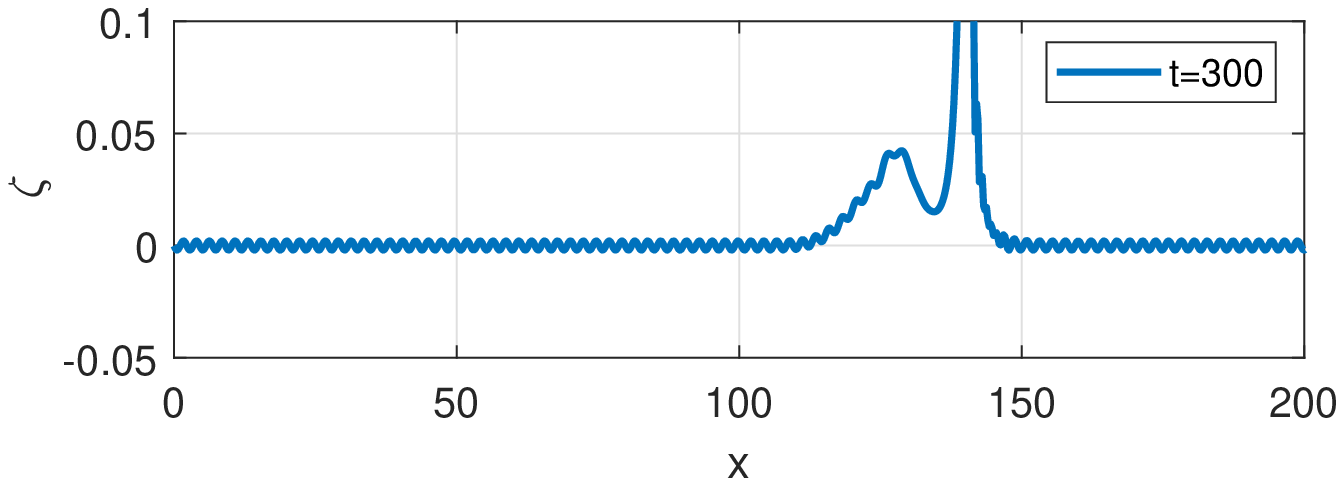}}
\subfigure[]
{\includegraphics[width=\columnwidth]{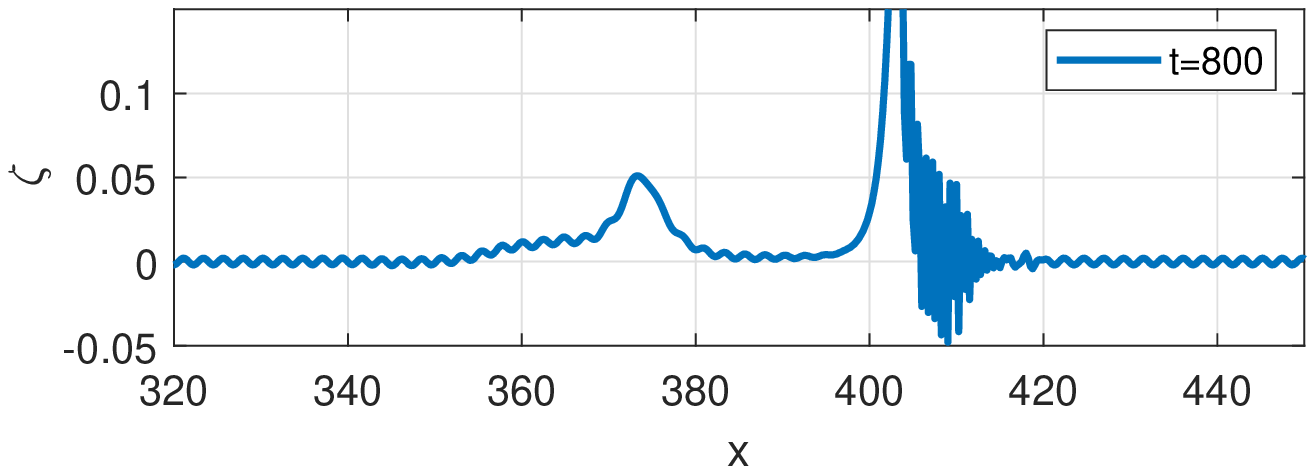}}
\caption{Large perturbation of a GSW. Case (A2) with (\ref{53a}), (\ref{53b}) with $A=4.1$. (a), (c) Magnifications of Figure \ref{fdds5_16}(d); (b) Magnifications of Figure \ref{fdds5_16}(b).}
\label{fdds5_16c}
\end{figure}
%in the formation of the modified GSW, a larger perturbation generates a more relevant dispersive tail traveling along the chain of ripples behind the main profile. Furthermore, as time goes by, an additional nonlinear structure separates very slowly from this main wave, in a process that reminds the corresponding dynamics from large perturbations of CSW. The structure seems to have the form of a GSW, in a sort of resolution property (cf. \cite{BonaDM2008}). 
%Figure \ref{fdds5_17} shows the evolution of the amplitude and speed of the main emerging GSW.

\begin{figure}[htbp]
\centering
%\subfigure[]
%{\includegraphics[width=0.6\textwidth]{gsw9_6.eps}}
%\subfigure[]
%{\includegraphics[width=0.6\textwidth]{gsw9_7.eps}}
\subfigure[]
{\includegraphics[width=\columnwidth]{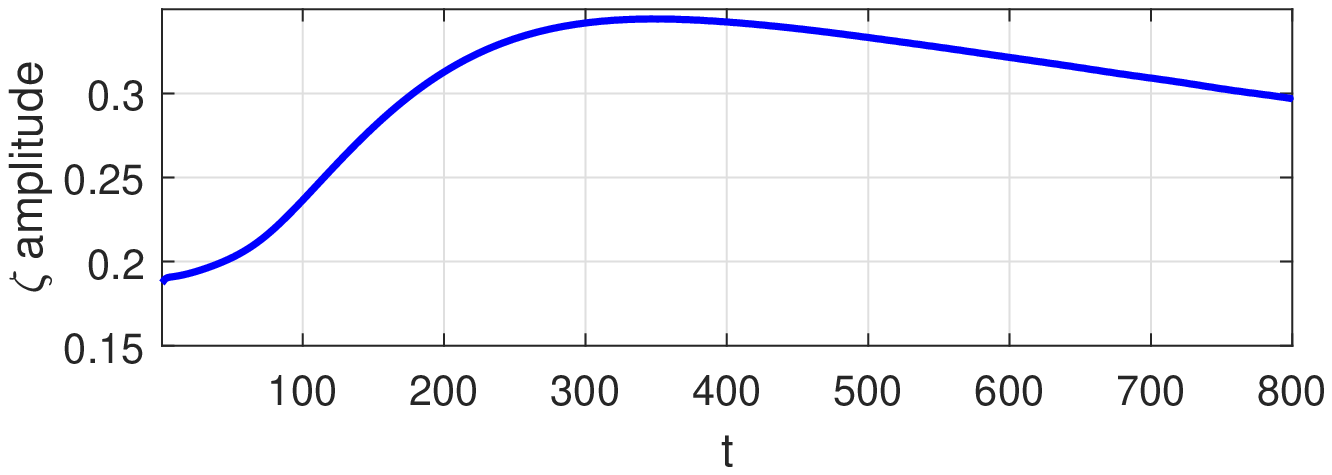}}
\subfigure[]
{\includegraphics[width=\columnwidth]{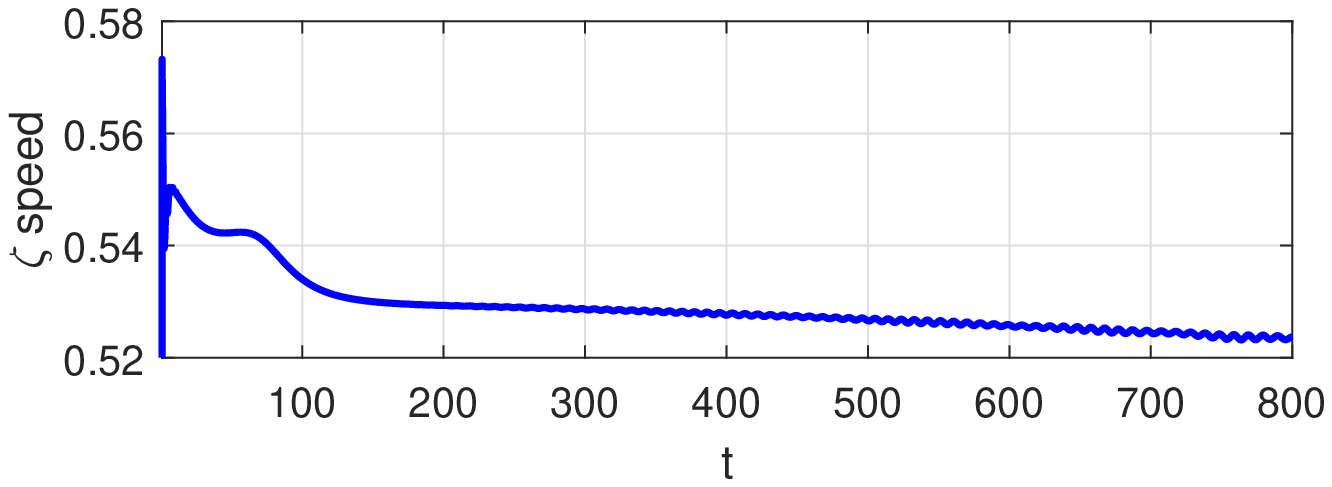}}
\caption{Large perturbation of a GSW. Case (A2) with (\ref{53a}), (\ref{53b}) with $A=4.1$. Evolution of amplitude (a) and speed (b) of the taller emerging  wave ($\zeta$ component of the numerical solution).}
\label{fdds5_17}
\end{figure}

%It may be worthwhile to mention the radiation generated in front of he main pulse; it looks like a way of adaptation of the ripples to the modified GSW. It may be interesting to run the experiment on a longer interval and during a longer time of simulation. Here the structures seem to be in an initial stage of their formation.

\subsubsection{Resolution}
\label{sec532}
In order to study the resolution property in systems with generalized solitary wave solutions, we consider again the values of the parameters given by (\ref{53a}) and use, as initial condition, the same Gaussian pulse as that considered in section \ref{sec53} for CSW's, of the form $\zeta(x,0)=Ae^{-\tau x^{2}}$, with $A=2, \tau=0.01$, and $u(x,0)=\zeta(x,0)$. The evolution of the numerical approximation is illustrated in Figure \ref{fdds5_18}.

\begin{figure}[htbp]
\centering
\subfigure[]
{\includegraphics[width=\columnwidth]{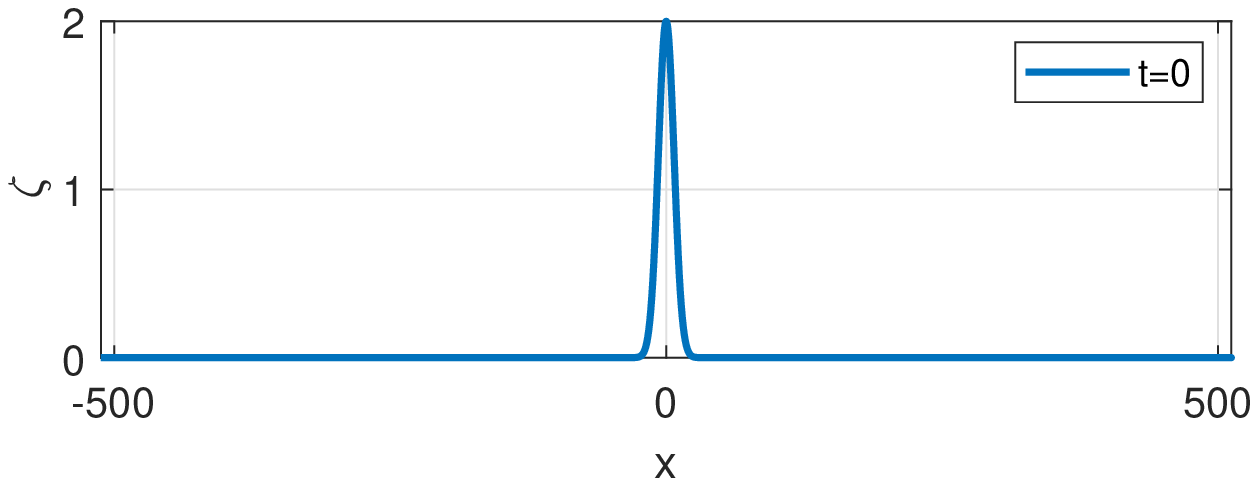}}
\subfigure[]
{\includegraphics[width=\columnwidth]{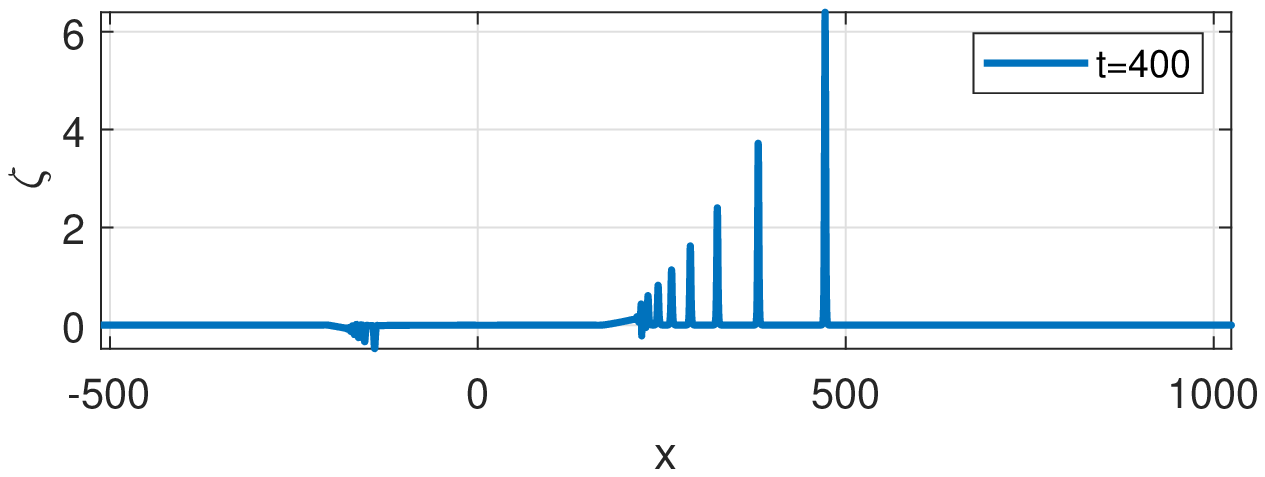}}
\subfigure[]
{\includegraphics[width=\columnwidth]{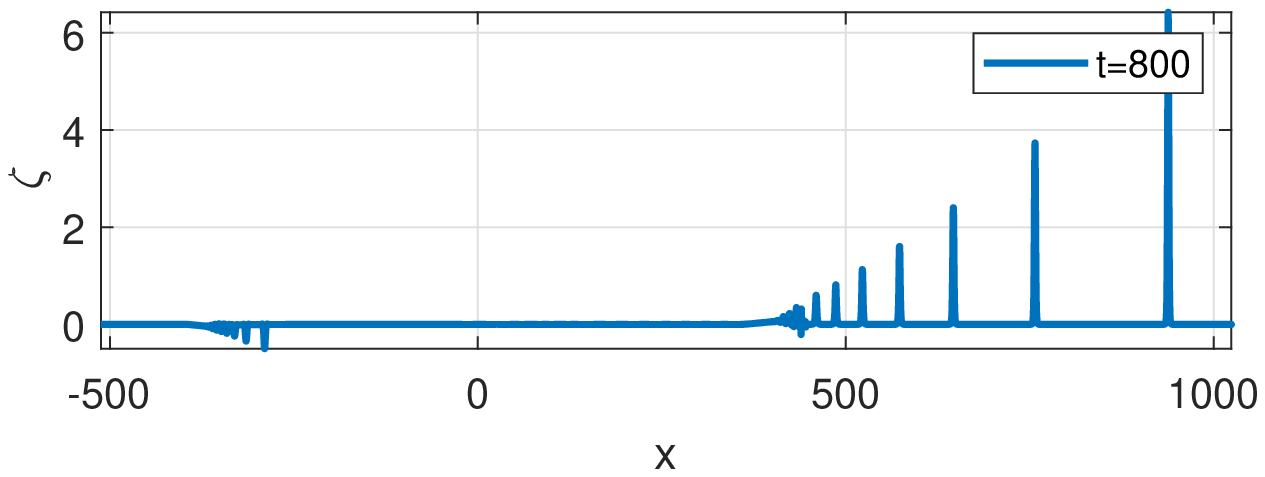}}
%\subfigure[]
%{\includegraphics[width=10cm]{gsw10_3m.eps}}
\caption{Resolution property. Case (A2) with (\ref{53a}), and a Gaussian pulse $\zeta(x,0)=u(x,0)=Ae^{-\tau x^{2}}$, with $A=2, \tau=0.01$, as initial condition.  $\zeta$ component of the numerical solution.}
\label{fdds5_18}
\end{figure}

\begin{figure}[htbp]
\centering
\subfigure[]
{\includegraphics[width=\columnwidth]{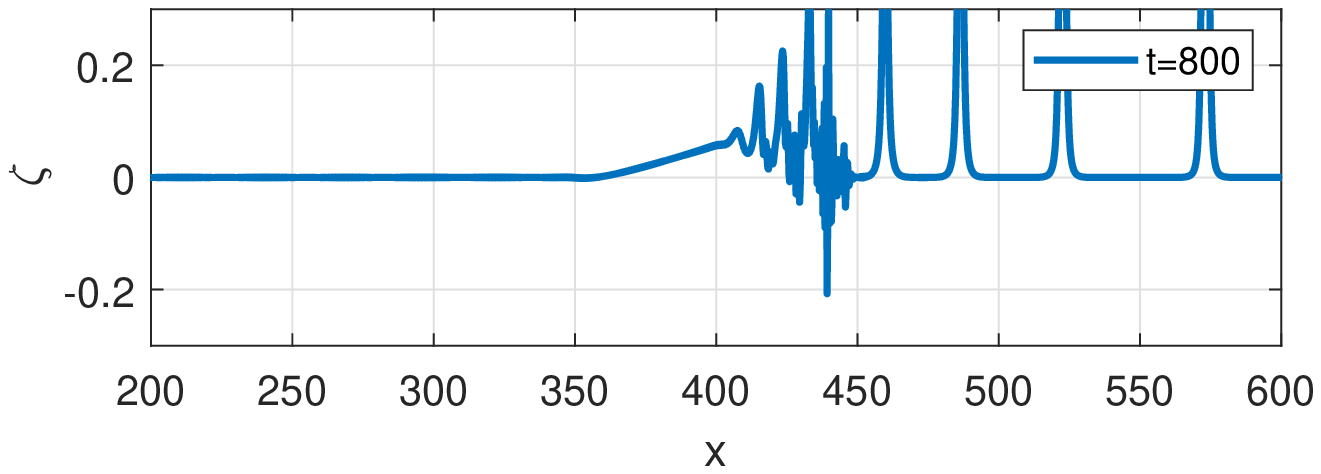}}
\subfigure[]
{\includegraphics[width=\columnwidth]{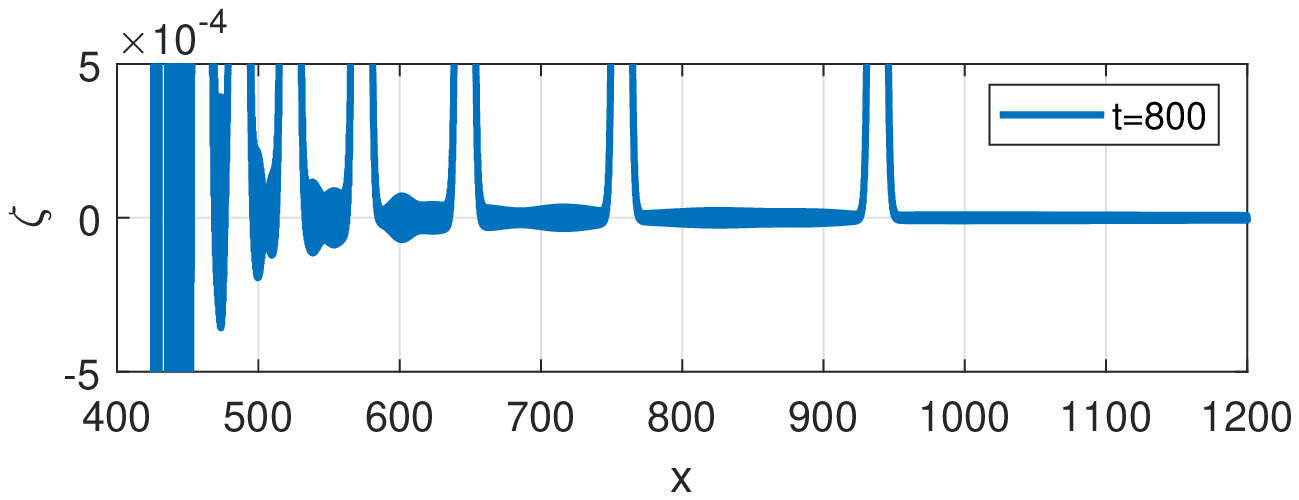}}
\subfigure[]
{\includegraphics[width=\columnwidth]{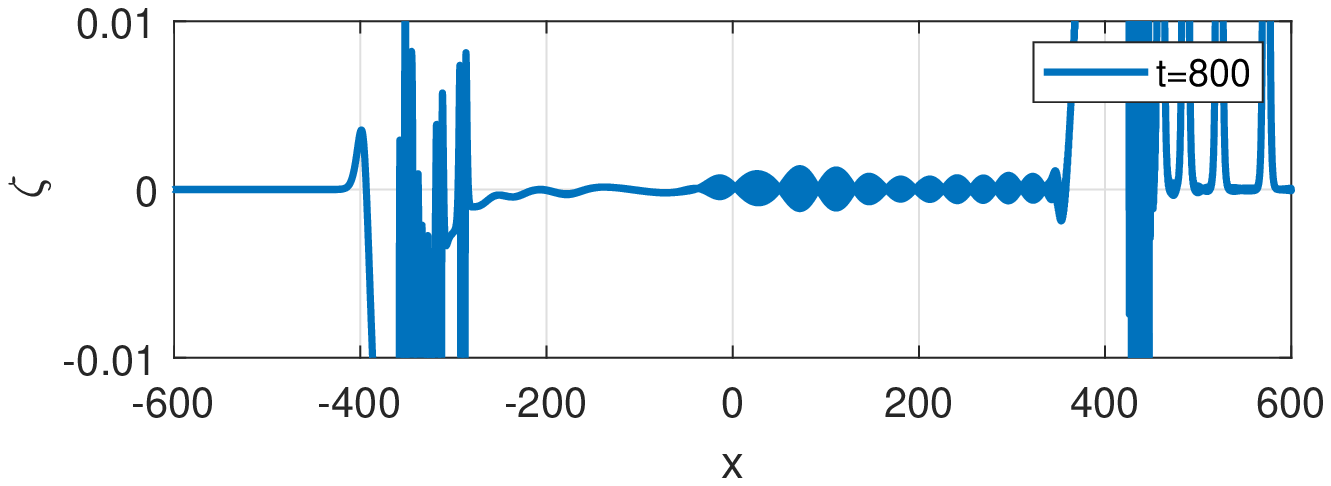}}
\subfigure[]
{\includegraphics[width=\columnwidth]{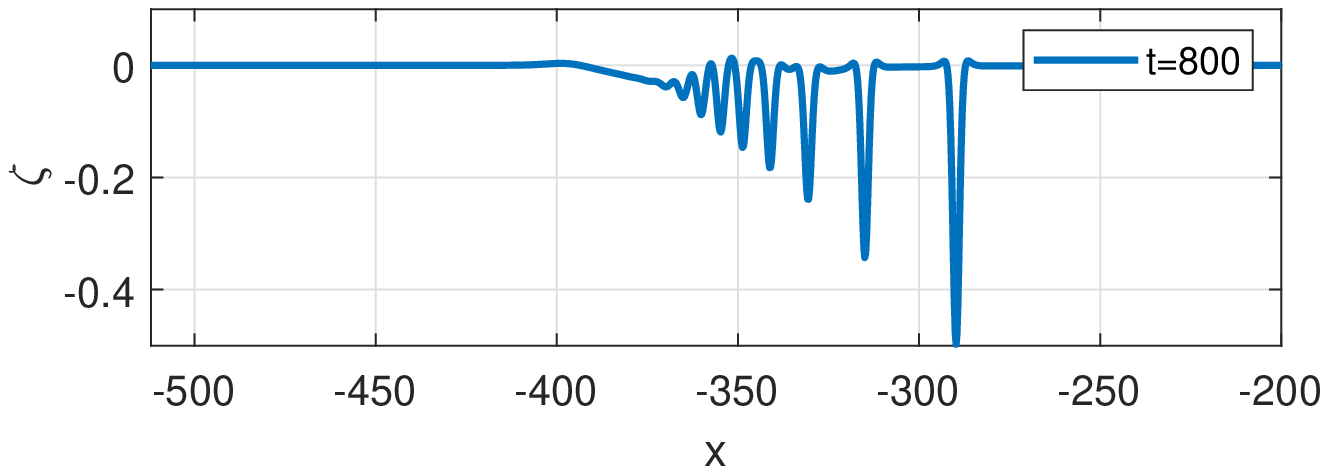}}
%\subfigure[]
%{\includegraphics[width=10cm]{gsw10_3m.eps}}
\caption{Resolution property. Case (A2) with (\ref{53a}), and a Gaussian pulse $\zeta(x,0)=u(x,0)=Ae^{-\tau x^{2}}$, with $A=2, \tau=0.01$, as initial condition.  Several structures of the $\zeta$ component of the numerical solution at $t=800$.}
\label{fdds5_18b}
\end{figure}

The behaviour of the approximation of $\zeta$ may be compared with that of the CSW case (Figures \ref{fdds5_13}-\ref{fdds5_13c}). Now, the Gaussian profile seems to evolve into a train of solitary waves of elevation, followed by some structures of different form. They are displayed in more detail in Figure \ref{fdds5_18b}. Figure \ref{fdds5_18b}(a) is a magnification of the tail formed just behind the solitary wave train
and Figure \ref{fdds5_18b}(b) is a detail of the solution between two solitary waves. Some dispersive pulses are radiated in front of each profile. These are also observed behind the wave train of solitary waves, see Figure \ref{fdds5_18b}(c). A third structure, magnified in Figure \ref{fdds5_18b}(d), seems to consist of a train of classical solitary waves with non-monotonic decay and a dispersive tail.

%develops a first structure, traveling to the left. In the magnification shown in Figure \ref{fdds5_18b}(a) we observe that the structure seems to consist of 
%
%
%(which seems to evolve to a train of CSW of depression) and then a train of CSW of elevaion. This seems to be formed before the corresponding one in the experiment of Figure \ref{fdds5_13}. The evolution of these two trains is not symmetric: the waves of elevation are more and much higher. This experiment would also suggest the existence of CSW in this example of case (A2), where the Normal Form Theory only predicts GSW.
\subsubsection{Head-on collisions}
The interactions of GSW's are illustrated with experiments of head-on collisions.
The first experiment shown in the sequel is concerned with a symmetric head-on collision of GSW's. For the system with parameters given by (\ref{53a}), a superposition of approximate GSW profiles with absolute values of the speed $c_{s}=c_{\gamma,\delta}+10^{-2}\approx 4.8824\times 10^{-1}
$, traveling to the right and to the left and centered at $x_{0}=-20$ and $x_{0}=20$ respectively, is taken as initial condition, and the simulation of the evolution is represented in Figures \ref{fdds5_19} and \ref{fdds5_19c}. The initial amplitudes are about $4.5654\times 10^{-2}$.

\begin{figure}[htbp]
\centering
\subfigure[]
{\includegraphics[width=\columnwidth]{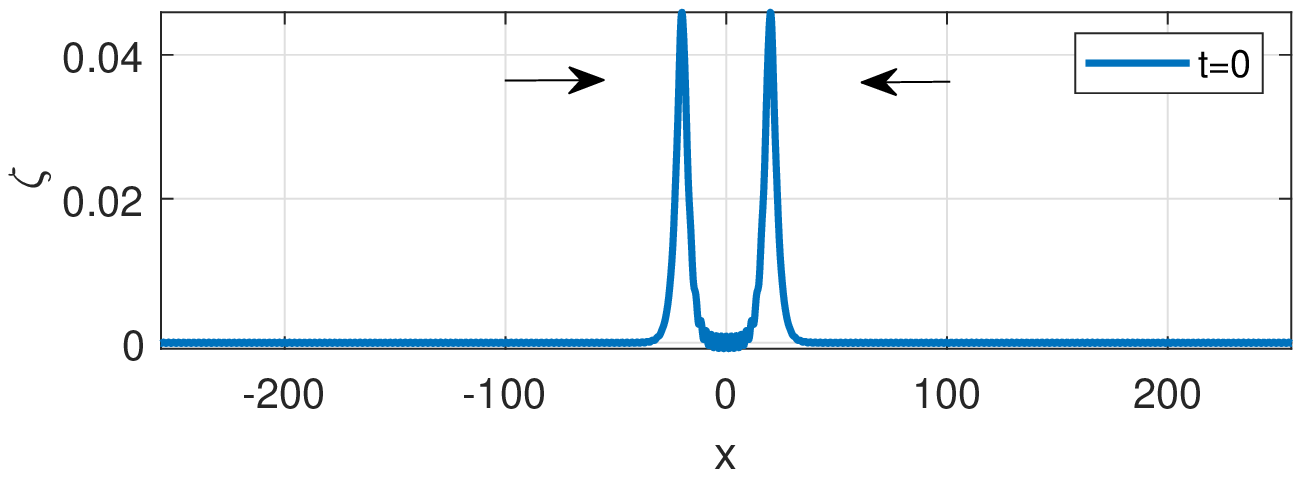}}
\subfigure[]
{\includegraphics[width=\columnwidth]{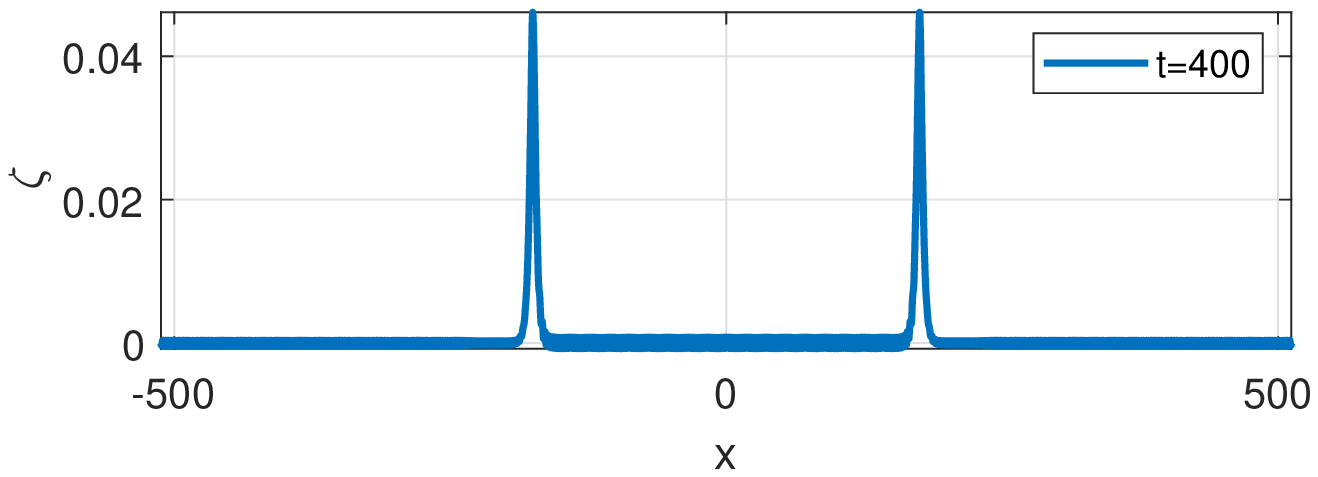}}
\subfigure[]
{\includegraphics[width=\columnwidth]{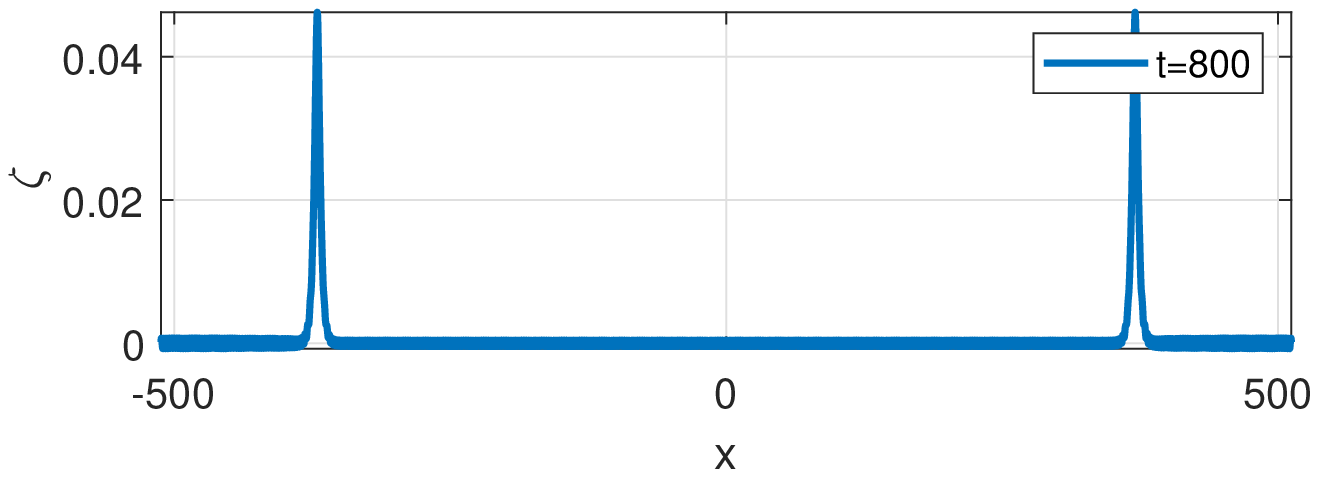}}
\caption{Symmetric head-on collision of GSW's. Case (A2) with (\ref{53a}).  $\zeta$ component of the numerical solution.}
\label{fdds5_19}
\end{figure}

The temporal interval of collision lasts approximately from $t=30$ to $t=50$; thereafter two symmetric GSW's emerge. The evolution of the amplitude of the $\zeta$ component of the numerical solution is shown in Figure \ref{fdds5_19b}(a). This seems to stabilize in a value around $4.62\times 10^{-2}$, so the emerging waves are taller (and hence faster), and the relative difference in amplitude is about $1.2\times 10^{-2}$. By comparing Figure \ref{fdds5_19c}(a) with Figures \ref{fdds5_19c}(b) and (c), we note that after the collision the amplitude of the ripples is larger. 

\begin{figure}[htbp]
\centering
\subfigure[]
{\includegraphics[width=\columnwidth]{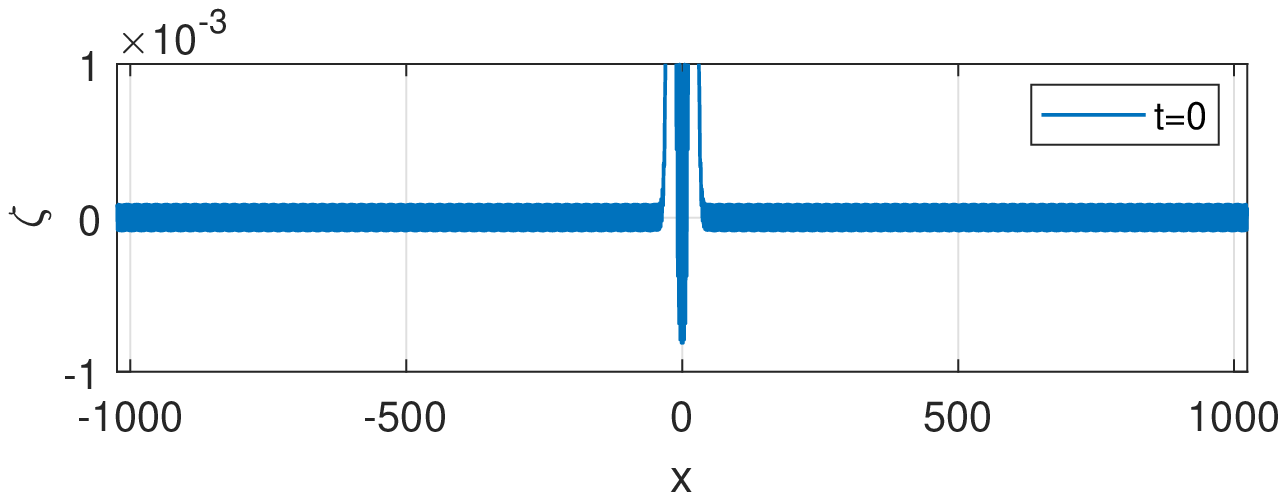}}
\subfigure[]
{\includegraphics[width=\columnwidth]{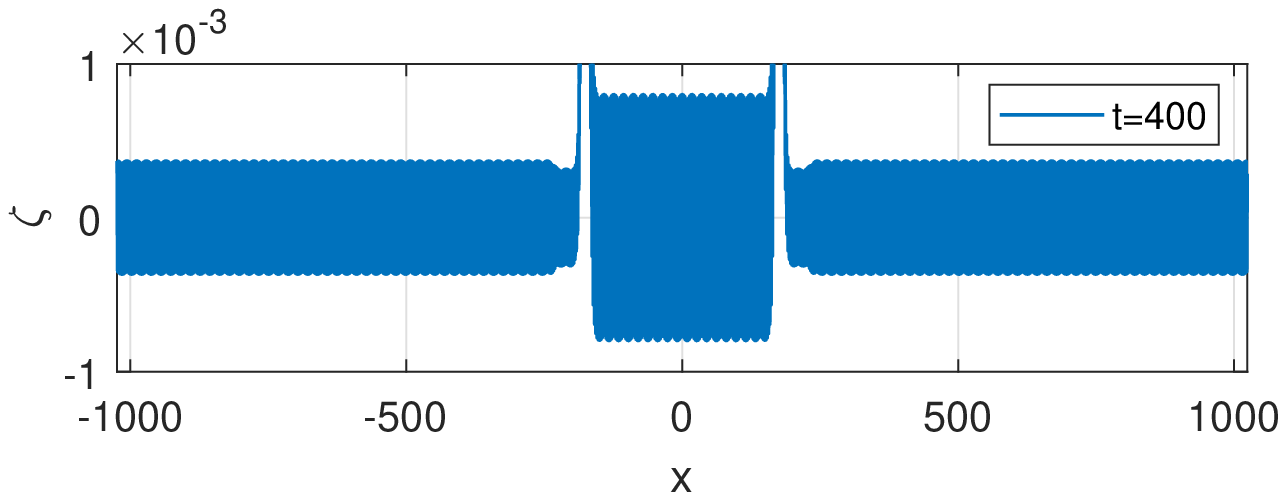}}
\subfigure[]
{\includegraphics[width=\columnwidth]{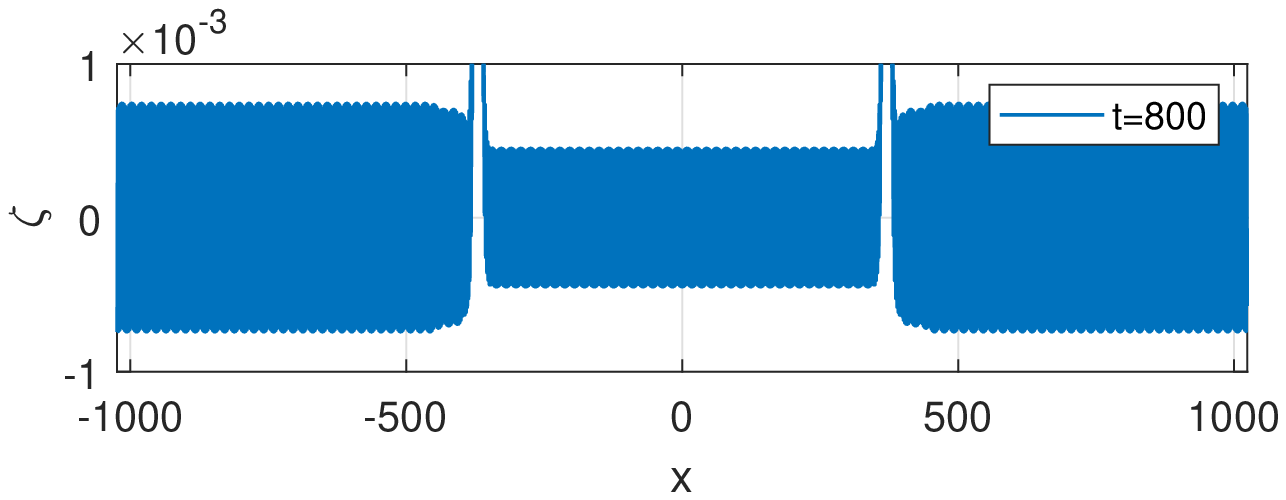}}
\caption{Symmetric head-on collision of GSW's. Case (A2) with (\ref{53a}).  $\zeta$ component of the numerical solution. Magnifications of Figure \ref{fdds5_19}.}
\label{fdds5_19c}
\end{figure}

In order to illustrate the dynamics of non-symmetric head-on collisions, we consider, for the system with parameters given by (\ref{53a}), the evolution of an initial profile consisting of a superposition of two approximate GSW's with absolute values of speeds $c_{s}^{(1)}=c_{\gamma,\delta}+0.01$, centered at $x_{0}^{(1)}=-20$ and traveling to the right, and $c_{s}^{(2)}=c_{\gamma,\delta}+0.005$, centered at $x_{0}^{(1)}=20$ and traveling to the left. The results are shown in Figures \ref{fdds5_20} and \ref{fdds5_20c}.

\begin{figure}[htbp]
\centering
\centering
\subfigure[]
{\includegraphics[width=\columnwidth]{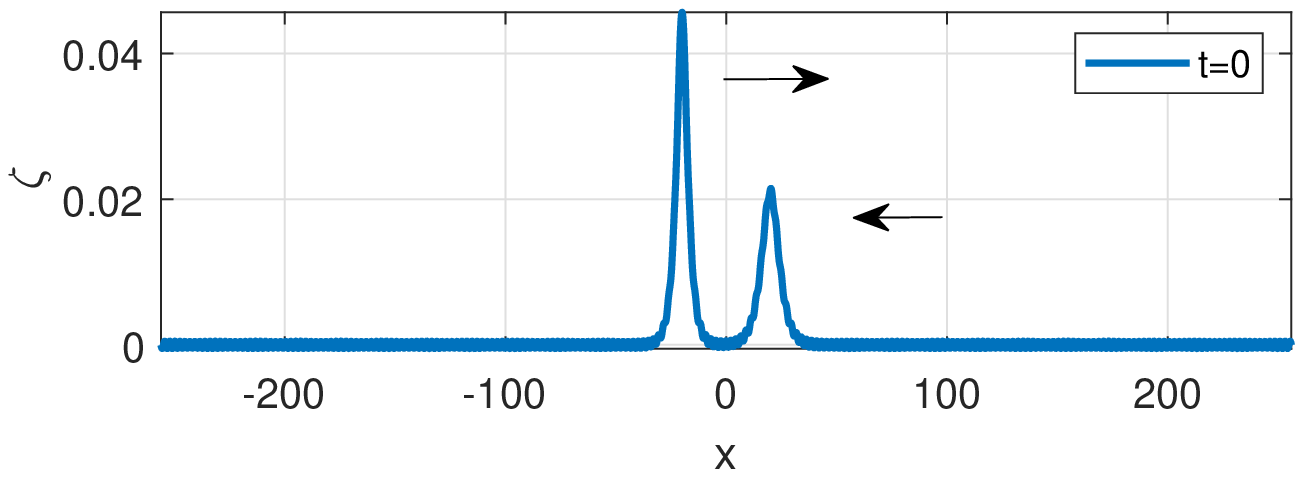}}
\subfigure[]
{\includegraphics[width=\columnwidth]{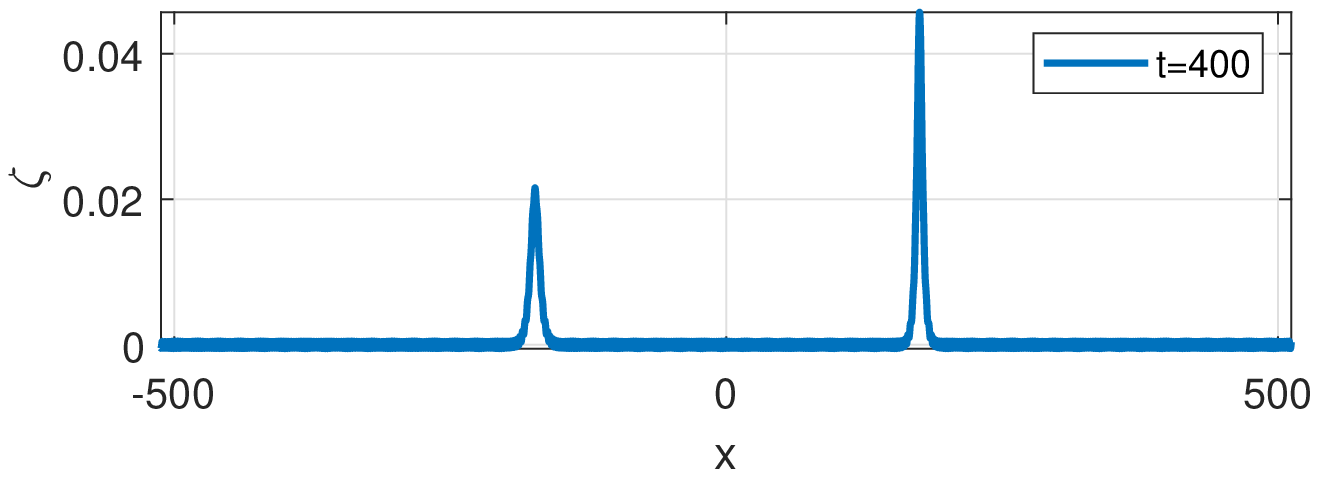}}
\subfigure[]
{\includegraphics[width=\columnwidth]{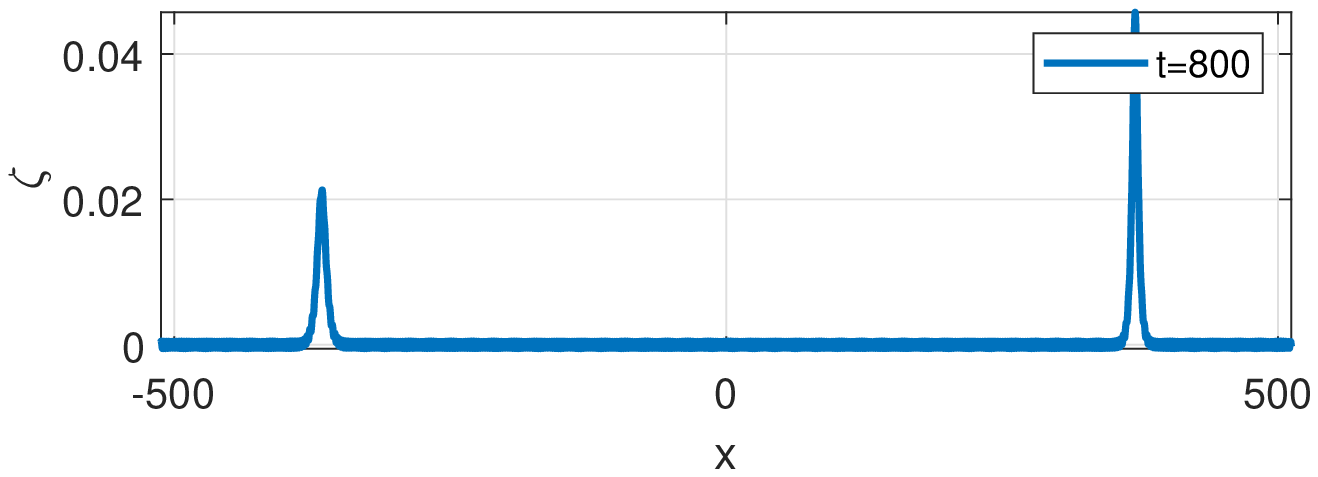}}
%\subfigure[]
%{\includegraphics[width=10cm]{gsw12_1.eps}}
%\subfigure[]
%{\includegraphics[width=10cm]{gsw12_2.eps}}
%\subfigure[]
%{\includegraphics[width=10cm]{gsw12_3.eps}}
%\subfigure[]
%{\includegraphics[width=10cm]{gsw12_4.eps}}
%\subfigure[]
%{\includegraphics[width=10cm]{gsw12_4m.eps}}
\caption{Non-symmetric head-on collision of GSW's. Case (A2) with (\ref{53a}). $\zeta$ component of the numerical solution.}
\label{fdds5_20}
\end{figure}

Similar effects to those of the symmetric collision case are observed. Note from Figure \ref{fdds5_20c} that in this case the ripples generated after the collision do not increase in  amplitude, with respect to those before the interaction, in a significant way, at least in comparison with the symmetric case, cf. Figure \ref{fdds5_19c}. 

\begin{figure}[htbp]
\centering
\subfigure[]
{\includegraphics[width=\columnwidth]{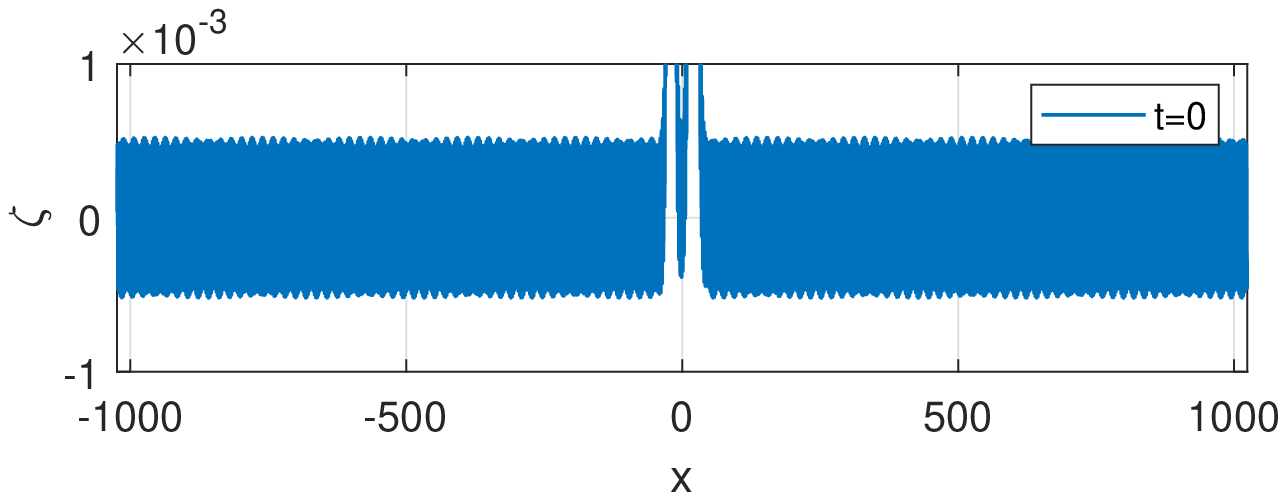}}
\subfigure[]
{\includegraphics[width=\columnwidth]{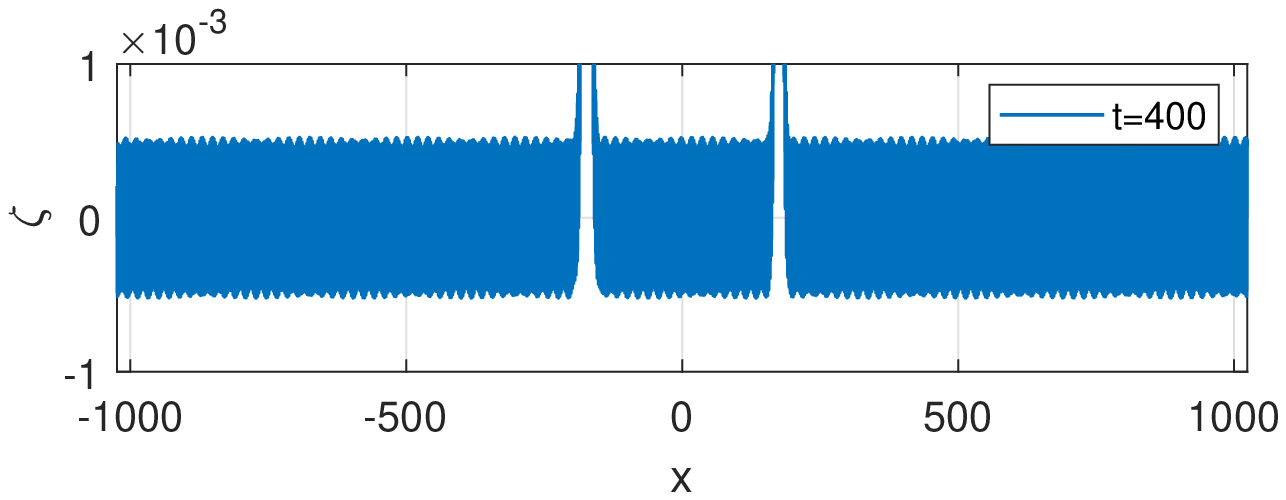}}
\subfigure[]
{\includegraphics[width=\columnwidth]{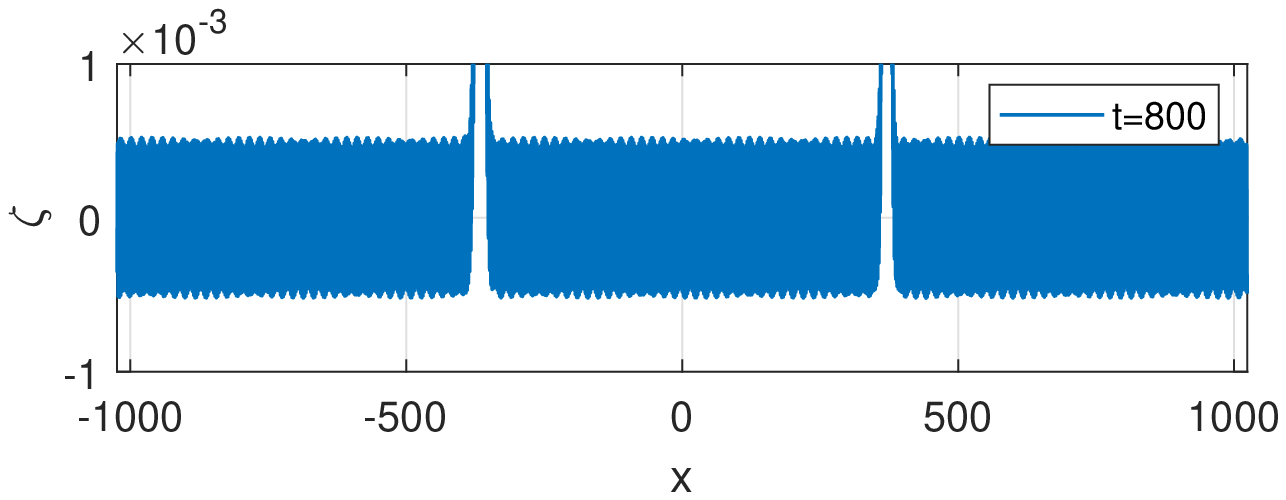}}
\caption{Non-symmetric head-on collision of GSW's. Case (A2) with (\ref{53a}).  $\zeta$ component of the numerical solution. Magnifications of Figure \ref{fdds5_20}.}
\label{fdds5_20c}
\end{figure}

The corresponding evolution of the amplitudes in both experiments of head-on collisions is shown in Figure \ref{fdds5_19b}. In the nonsymmetric case, the emerging taller wave increases in amplitude with a relative increase of $2.1\times 10^{-2}$.
\subsection{Dynamics of classical solitary wave solutions with non monotone decay}
\label{sec55}

%\begin{figure}[htbp]
%\centering
%\subfigure[]
%{\includegraphics[width=\columnwidth]{bb_gsw_symho_amp.eps}}
%\subfigure[]
%{\includegraphics[width=\columnwidth]{bb_gsw_nsymho_amp.eps}}
%\caption{Head-on collisions of GSW's. Case (A2) with (\ref{53a}). Evolution of amplitude: (a) Symmetric case; (b) Non-symmetric case (tallest emerging wave).}
%\label{fdds5_19b}
%\end{figure}

%\subsection{Dynamics of classical solitary wave solutions with nonmonotone decay}
%\label{sec55}
In this section we present some experiments concerning the dynamics of CSW solutions of (\ref{BB2}) with non monotone decay, whose existence was justified by an application of the Normal Form Theory in section \ref{sec41}, and for speeds smaller than the corresponding speed of sound. For simplicity we focus on the behaviour under small and large perturbations of the type (\ref{53b}) as those used in sections \ref{sec53} and \ref{sec54}.

\begin{figure}[htbp]
\centering
\subfigure[]
{\includegraphics[width=\columnwidth]{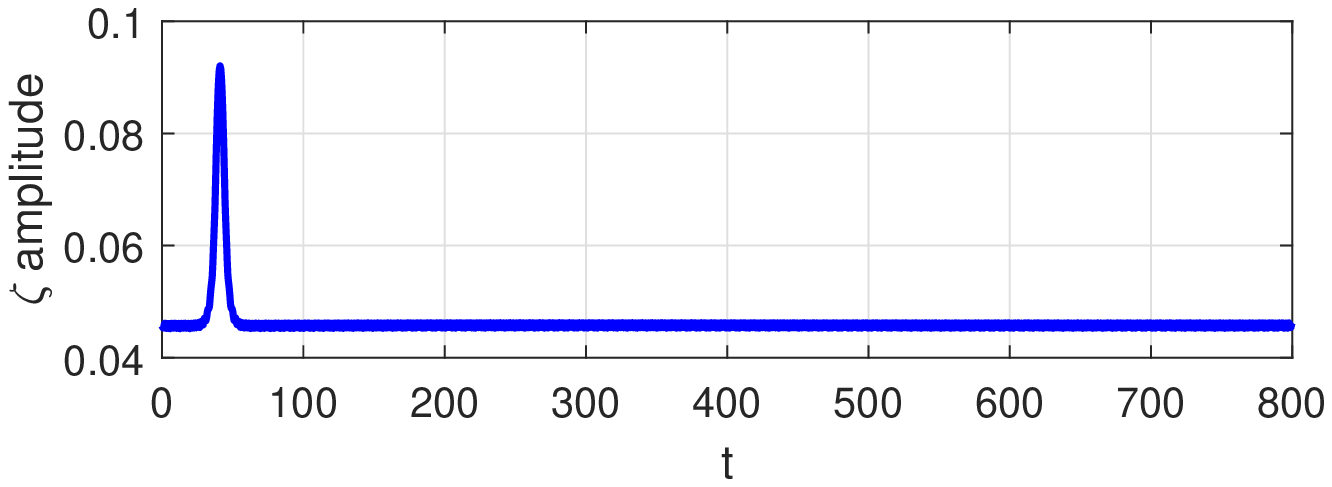}}
\subfigure[]
{\includegraphics[width=\columnwidth]{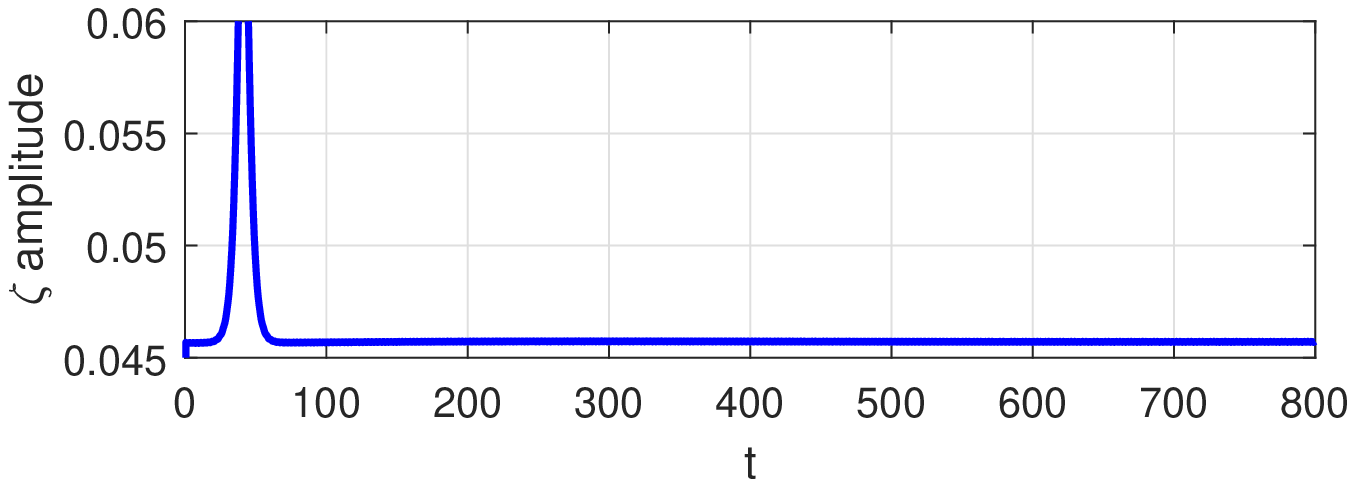}}
\caption{Head-on collisions of GSW's. Case (A2) with (\ref{53a}). Evolution of amplitude: (a) Symmetric case; (b) Non-symmetric case (tallest emerging wave).}
\label{fdds5_19b}
\end{figure}

%generated numerically at the end of section \ref{sec42} and related to the linearization of (\ref{NFT1}) at $U=0$ and the NFT, cf. section \ref{sec41}.

We consider the numerical profile obtained in section \ref{sec42} from the parameters 
\begin{eqnarray}
&&\gamma=0.5, \delta=0.9,\nonumber\\
&& a=-1/9, b=0, c=-1/6, d=S(\gamma,\delta)-\frac{a}{\kappa_{1}}-b-c\approx 0.7058.\label{53aa}
\end{eqnarray}
The $\zeta$ component has a maximum negative excursion of about $-2.8740$ and the speed is $c_{s}=c_{\gamma,\delta}-0.2\approx 0.3976$. The $\zeta_{h}$ and $v_{h}$ components are perturbed in amplitude with a perturbation factor $A=1.1$. The perturbed wave $(A\zeta_{h},Av_{h})$ is taken as initial condition of the numerical method to approximate (\ref{BB2}) with $L=1024$, and the evolution of the resulting numerical approximation is monitored up to $t=800$ and shown in Figure \ref{fdds5_21}.
\begin{figure}[htbp]
\centering
\centering
\subfigure[]
{\includegraphics[width=\columnwidth]{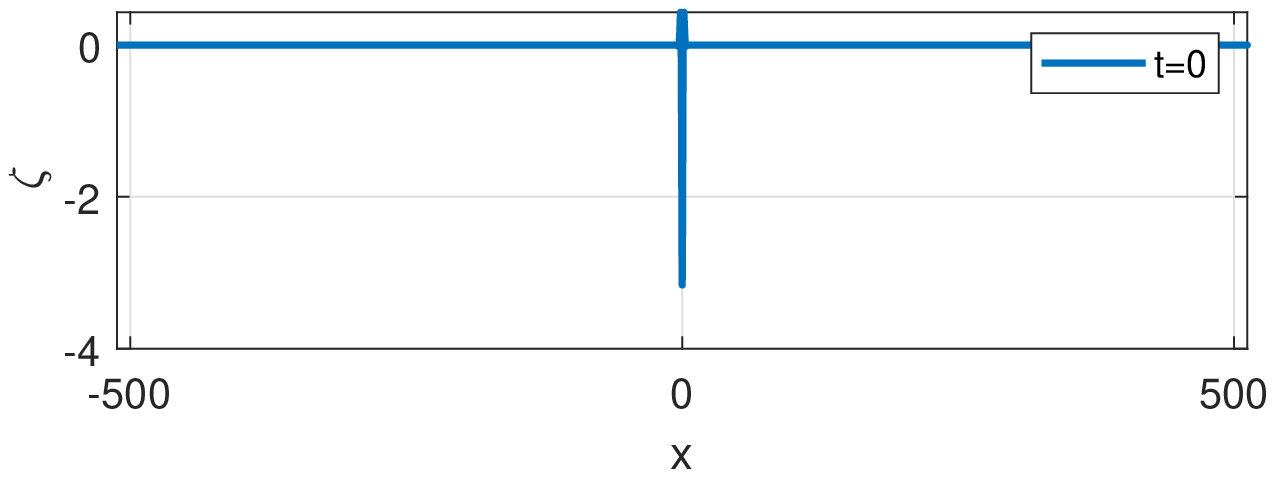}}
\subfigure[]
{\includegraphics[width=\columnwidth]{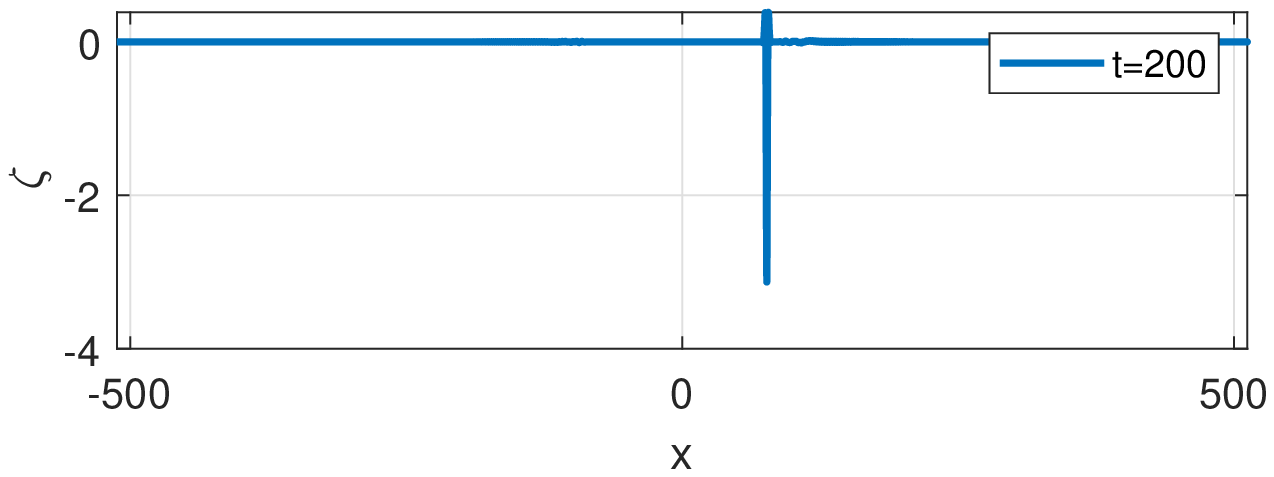}}
\subfigure[]
{\includegraphics[width=\columnwidth]{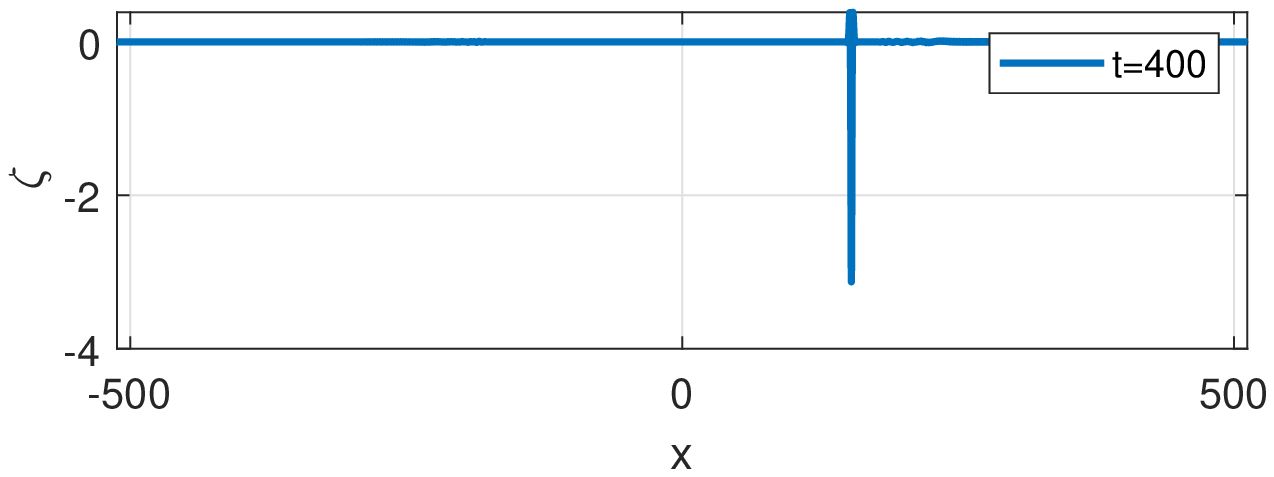}}
\subfigure[]
{\includegraphics[width=\columnwidth]{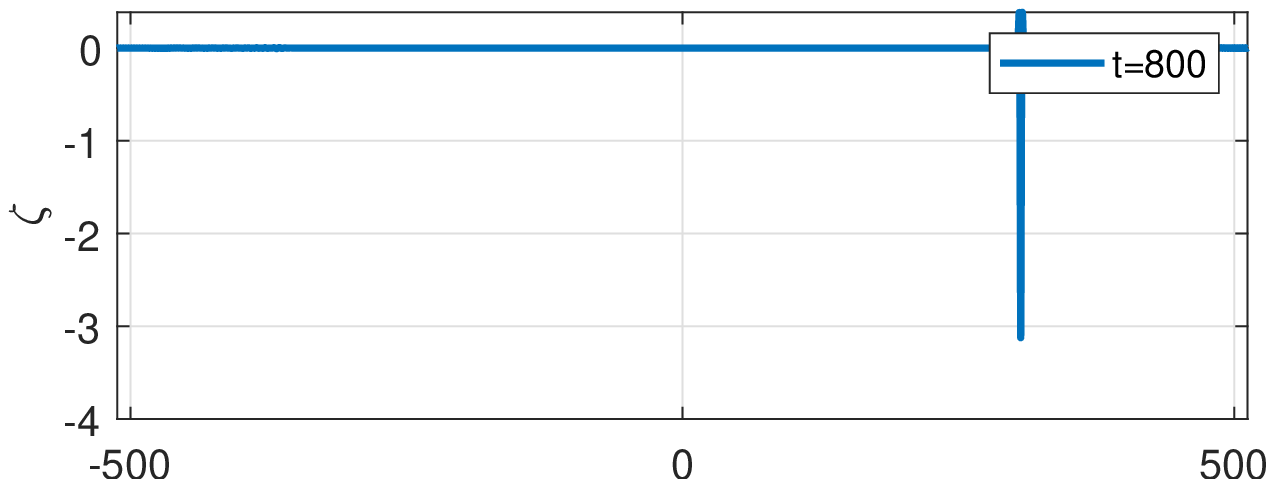}}
\caption{Small perturbation of a CSW with nonmonotone decay. $A=1.1$, $\zeta$ component of the numerical solution.}
\label{fdds5_21}
\end{figure}

During the evolution, a new solitary wave of the same type is formed. Compared to the initial perturbed wave, whose maximum negative excursion was approximately $-3.1614$, the emerging wave dips to about $-3.1376$. The wave is slower, with a speed around $0.3832$, see Figure \ref{fdds5_21c}. 

\begin{figure}[htbp]
\centering
\centering
\subfigure[]
{\includegraphics[width=\columnwidth]{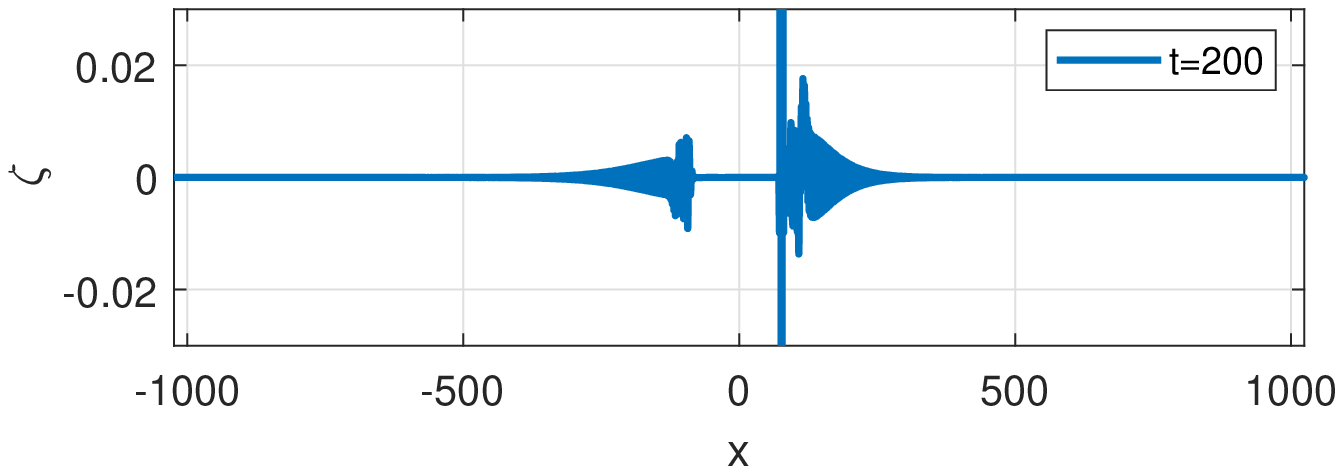}}
\subfigure[]
{\includegraphics[width=\columnwidth]{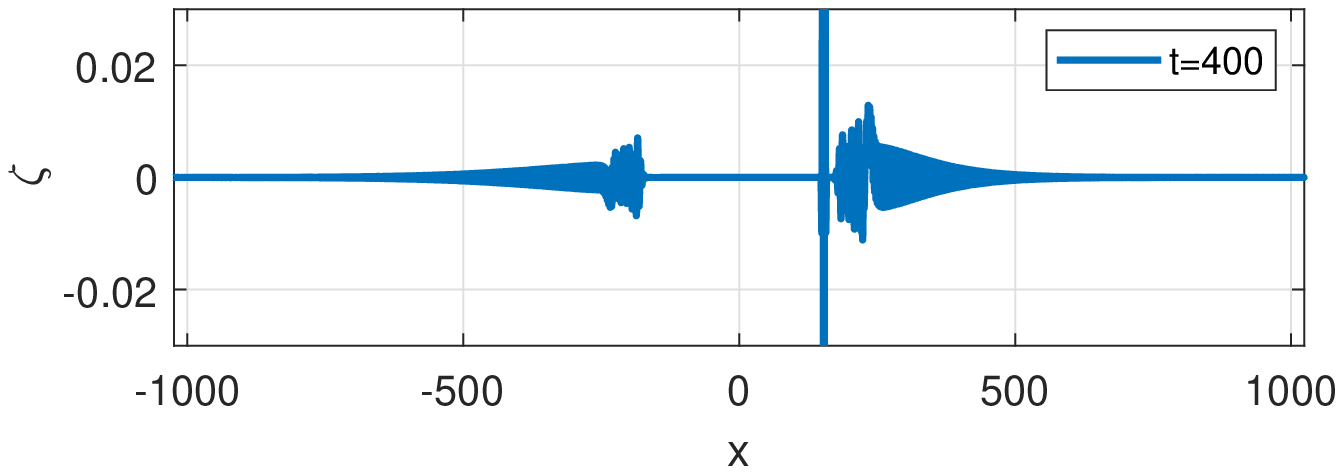}}
\subfigure[]
{\includegraphics[width=\columnwidth]{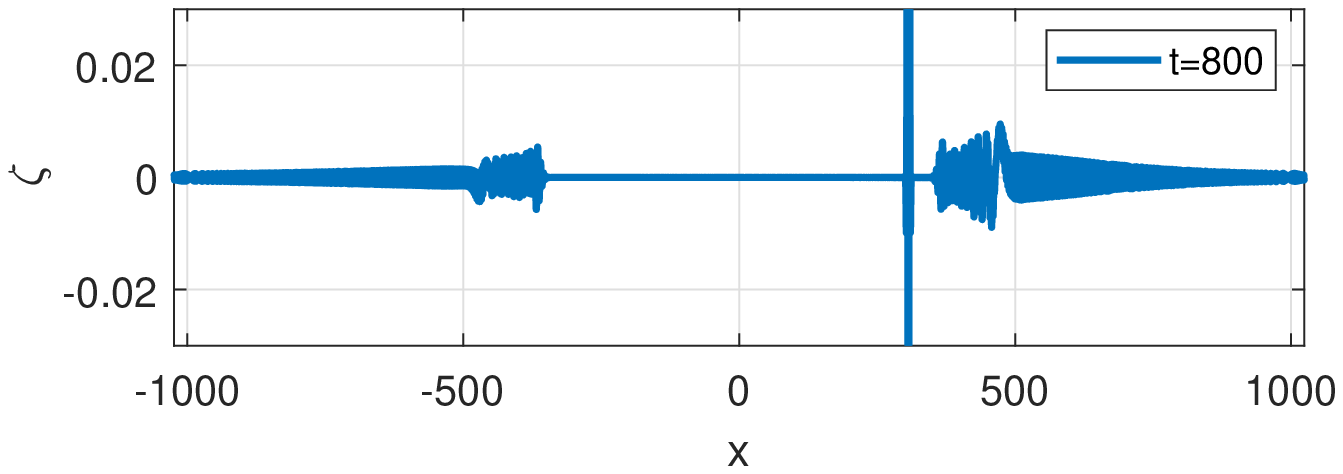}}
\caption{Small perturbation of a CSW with nonmonotone decay. $A=1.1$. $\zeta$ component of the numerical solution. Magnifications of Figure \ref{fdds5_21}.}
\label{fdds5_21b}
\end{figure}

Behind and in front of the emerging solitary wave, some other small structures form. They are observed in the magnification of Figure \ref{fdds5_21} given by Figure \ref{fdds5_21b}, and displayed in more detail in Figure \ref{fdds5_21d}.
%
%
%,  see Figure \ref{fdds5_21b}. The main component of these waves seem to be dispersive, although nonlinear forms are not discarded, see Figure \ref{fdds5_21d}.

%\begin{figure}[htbp]
%\centering
%\centering
%\subfigure[]
%{\includegraphics[width=\columnwidth]{ocsw_amp.eps}}
%\subfigure[]
%{\includegraphics[width=\columnwidth]{ocsw_speed.eps}}
%\caption{Small perturbation of CSW. Evolution of  maximum negative excursion (a) and speed (b) of the emerging  solitary wave.}
%\label{fdds5_21c}
%\end{figure}
%\begin{figure}[htbp]
%\centering
%\centering
%\subfigure[]
%{\includegraphics[width=6.27cm,height=5.05cm]{ocsw_2m1.eps}}
%\subfigure[]
%{\includegraphics[width=6.27cm,height=5.05cm]{ocsw_2m2.eps}}
%\subfigure[]
%{\includegraphics[width=6.24cm,height=5.2cm]{ocsw_3m1.eps}}
%\subfigure[]
%{\includegraphics[width=6.27cm,height=5cm]{ocsw_3m2.eps}}
%\subfigure[]
%{\includegraphics[width=6.27cm]{ocsw_4m1.eps}}
%\subfigure[]
%{\includegraphics[width=6.27cm,height=5.1cm]{ocsw_4m2.eps}}
%\caption{Small perturbation of a CSW with nonmonotone decay. $A=1.1$. $\zeta$ component of the numerical solution. Magnifications of Figure \ref{fdds5_21b}.}
%\label{fdds5_21d}
%\end{figure}

The main character of these small-amplitude waves seems to be dispersive. The generation of dispersive oscillations in front of the emerging soltary wave can be justified from the study of small-amplitude solutions of the linearized system (\ref{53}), (\ref{54}). In this case we have a linear dispersion relation of the form
\begin{eqnarray*}
\omega(k)=\omega_{\pm}(k)=-kc_{s}\pm c_{\gamma,\delta}k\phi(k^{2}),
\end{eqnarray*}
where $\phi:[0,\infty)\rightarrow\mathbb{R}$ is the function
\begin{eqnarray*}
\phi(x)=\sqrt{\frac{(1-\widetilde{a}x)(1-cx)}{1+dx}},\; \widetilde{a}=\frac{a}{\kappa_{1}},
\end{eqnarray*}
with $a, c, d$ given by (\ref{53aa}). 
The local phase speed (relative to the speed of the CSW) is therefore
\begin{eqnarray*}
v_{\pm}(k)=-c_{s}\pm c_{\gamma,\delta}\phi(k^{2}).
\end{eqnarray*}
An analysis of $\phi$ similar to that made in section \ref{sec54} for the case (A1) shows that the function $\phi$ has the form displayed in Figure \ref{Z1}(a): It is decreasing up to some $x^{*}$ and increasing for $x\geq x^{*}$ with $\phi(x)\rightarrow +\infty$ as $x\rightarrow \infty$. As in section \ref{sec54}, this means that for large $|k|$
$$v_{+}(k)>-c_{s}+c_{\gamma,\delta},$$ and most of the components of the dispersive tail travel to the right and in front of the solitary wave. Similarly, for the group velocities
\begin{eqnarray*}
\omega^{\prime}_{\pm}(k)=-c_{s}\pm c_{\gamma,\delta}\psi(k^{2}),
\end{eqnarray*}
where  $\psi:[0,\infty)\rightarrow\mathbb{R}$, given by (\ref{514e}), has the form shown in Figure \ref{Z1}(b), and therefore, for $|k|$ large enough
 \begin{eqnarray*}
\omega^{\prime}_{+}(k)\geq -c_{s}+ c_{\gamma,\delta}>0.
\end{eqnarray*}
Hence one dispersive group travels to the right and in front of the solitary wave.
\begin{figure}[htbp]
\centering
\centering
\subfigure[]
{\includegraphics[width=6.27cm,height=5.05cm]{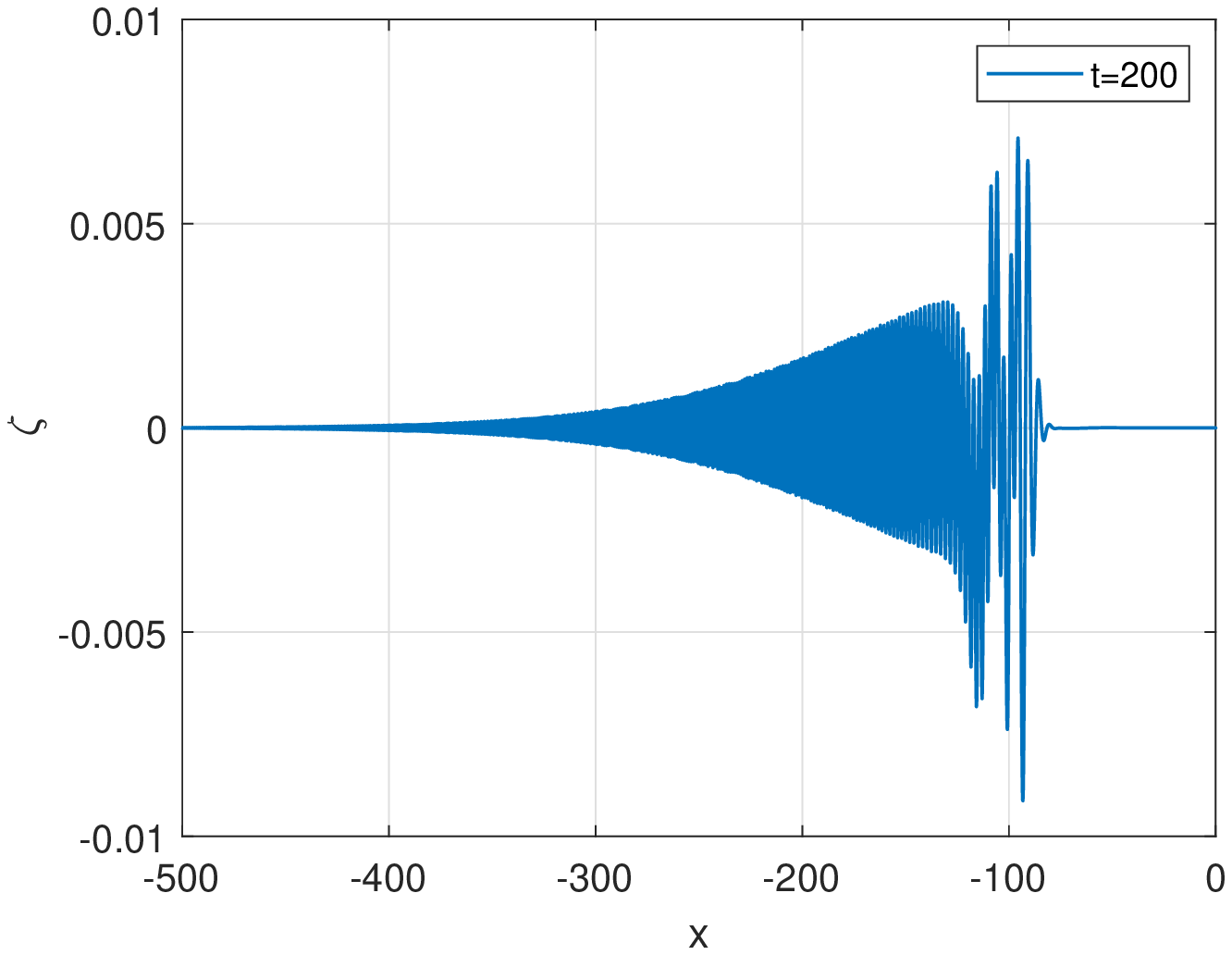}}
\subfigure[]
{\includegraphics[width=6.27cm,height=5.05cm]{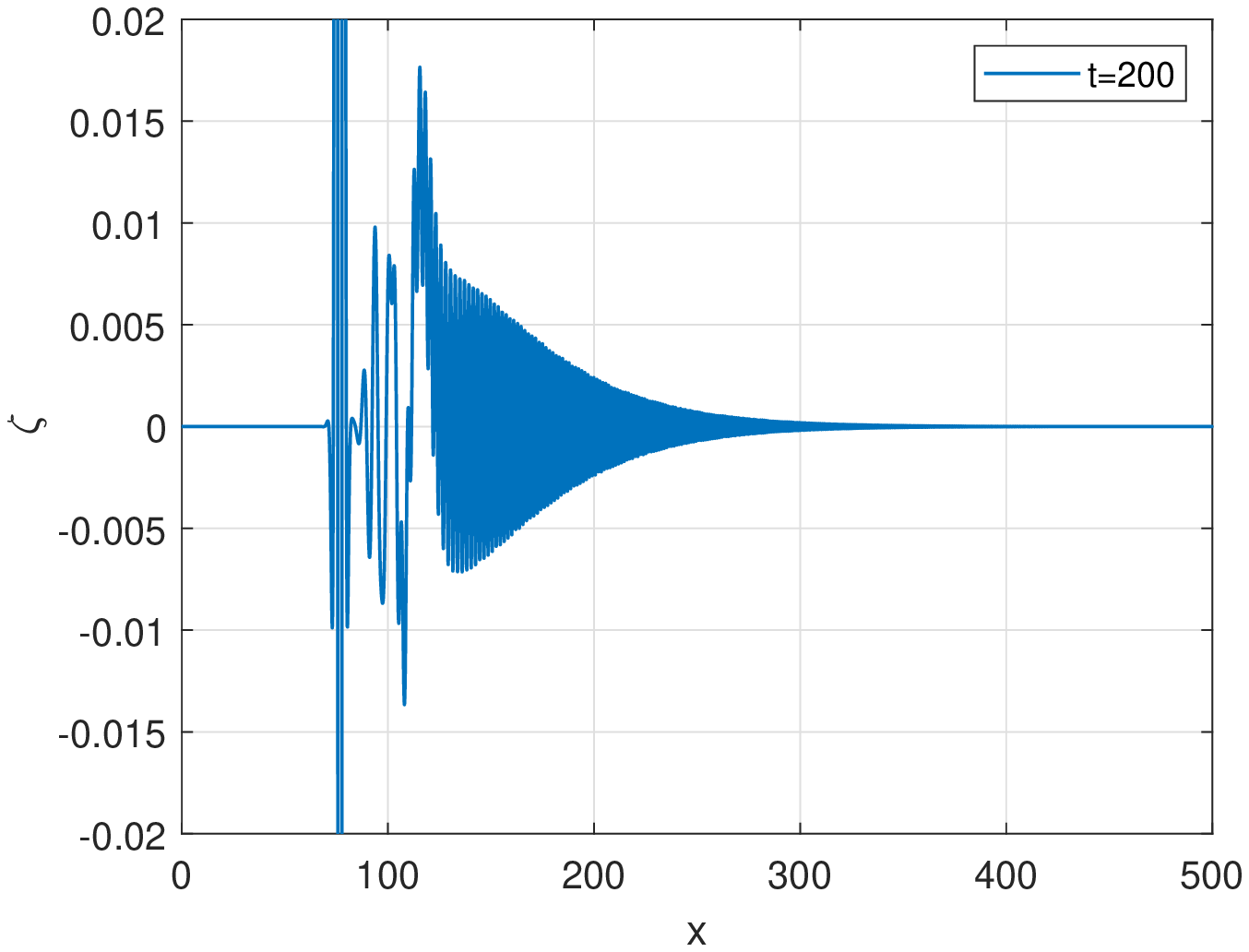}}
\subfigure[]
{\includegraphics[width=6.24cm,height=5.2cm]{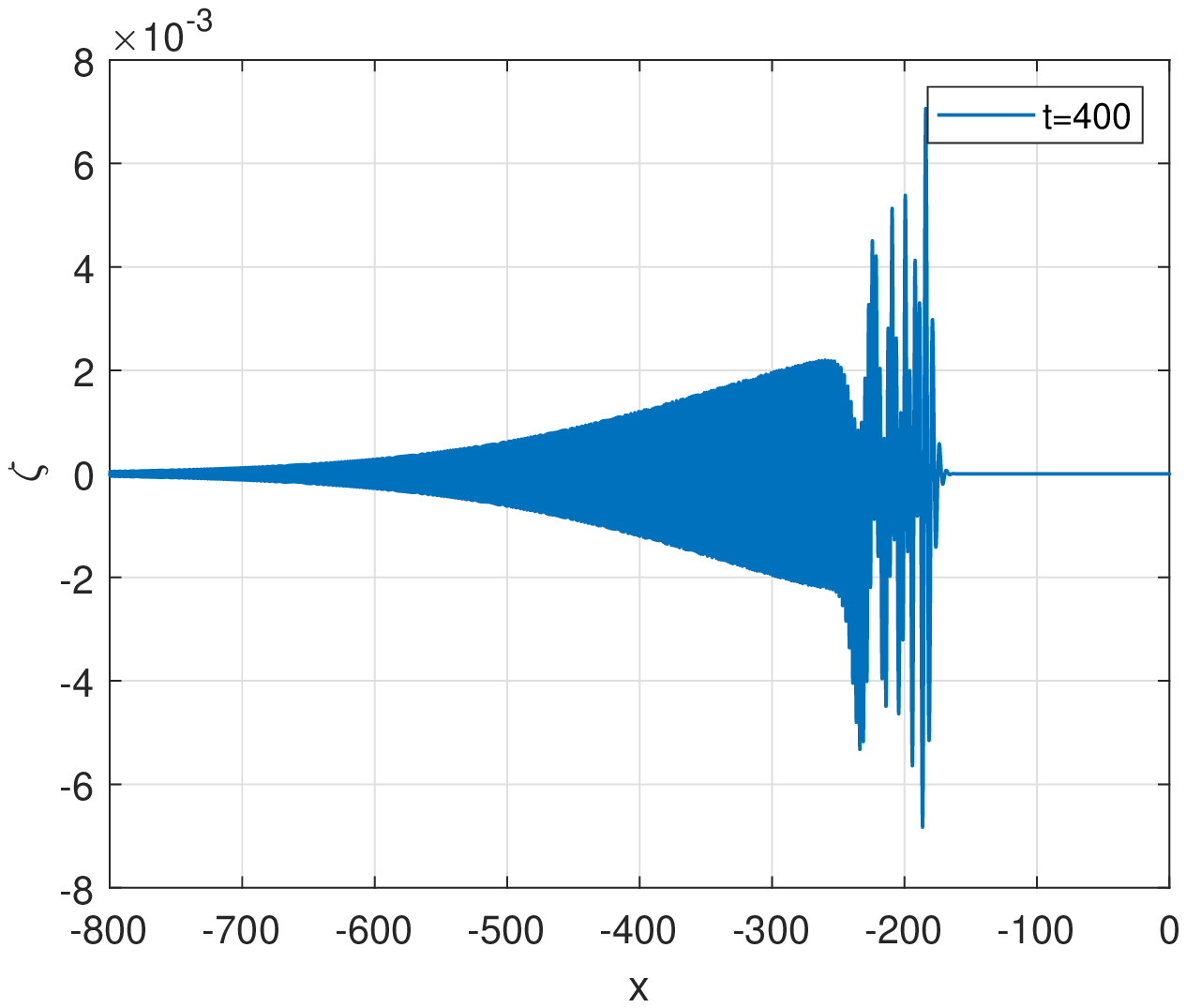}}
\subfigure[]
{\includegraphics[width=6.27cm,height=5cm]{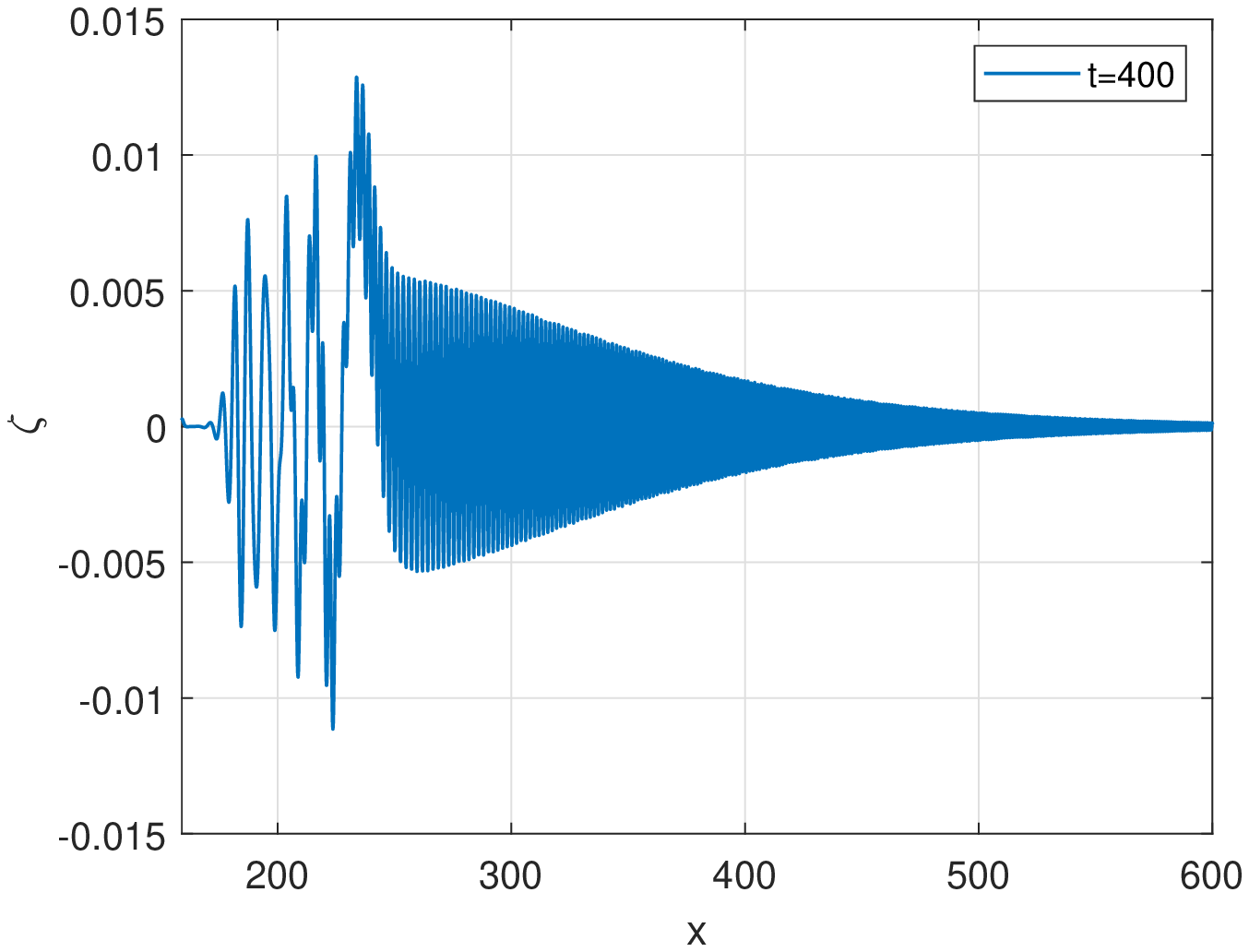}}
\subfigure[]
{\includegraphics[width=6.27cm]{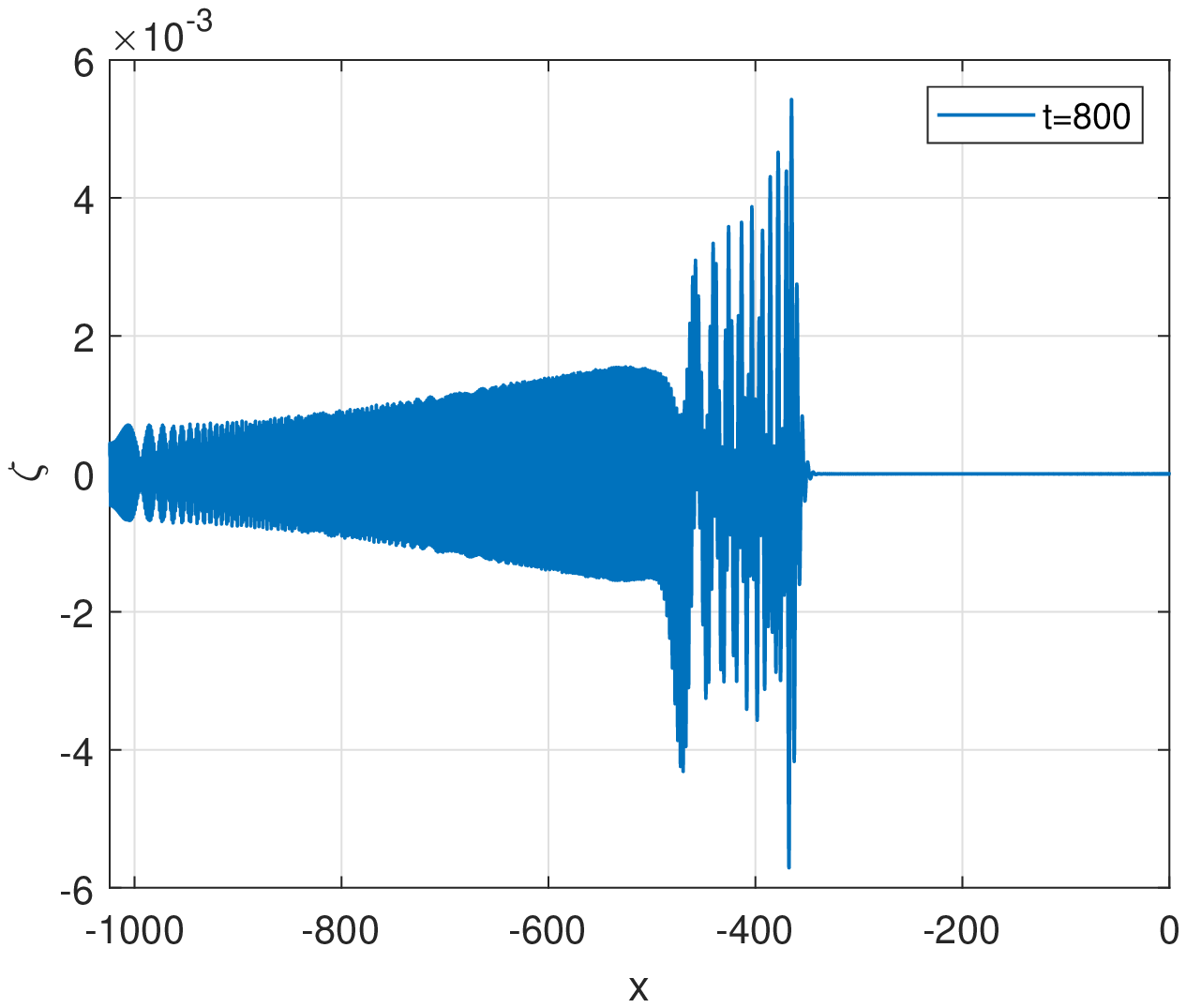}}
\subfigure[]
{\includegraphics[width=6.27cm,height=5.1cm]{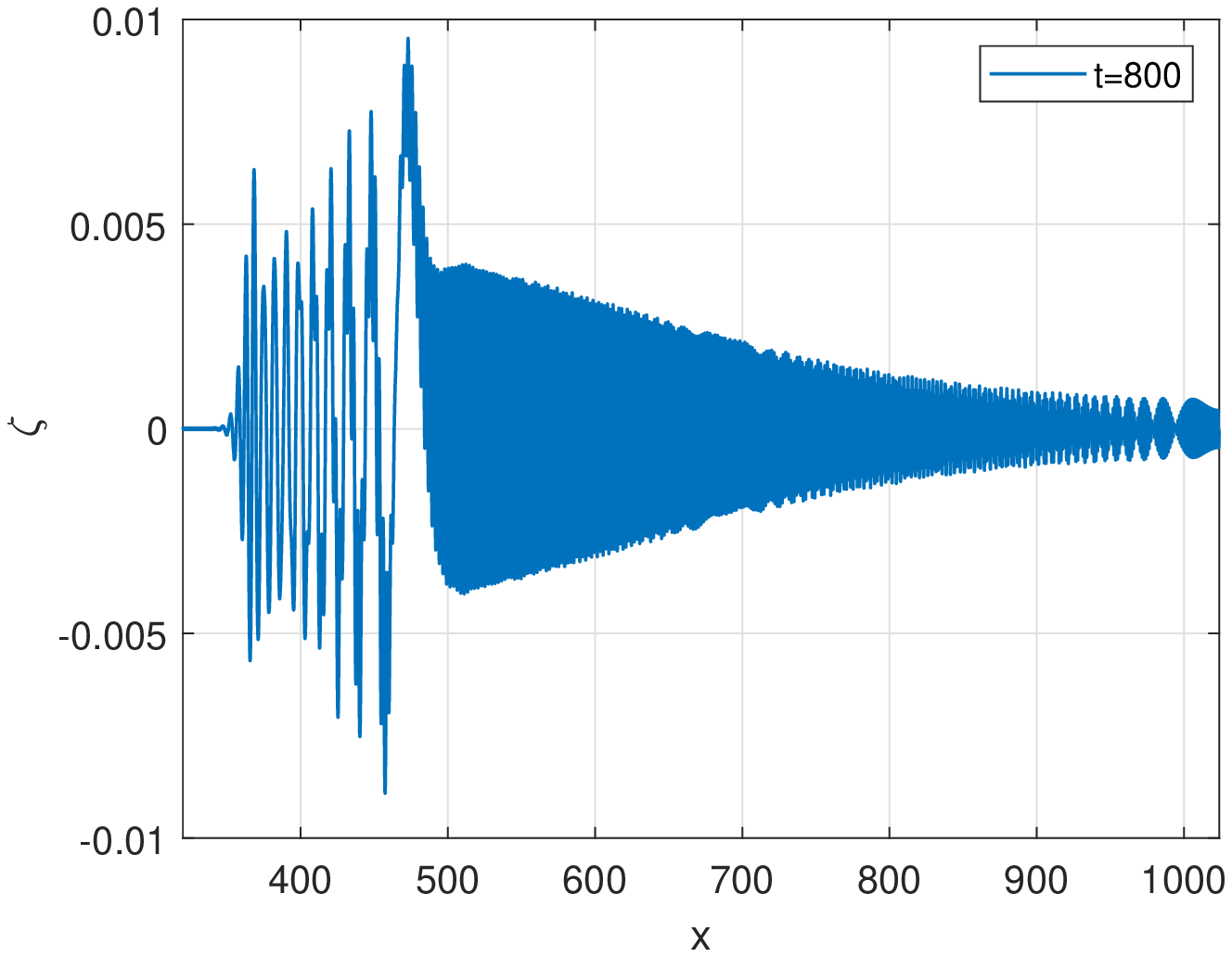}}
\caption{Small perturbation of a CSW with nonmonotone decay. $A=1.1$. $\zeta$ component of the numerical solution. Magnifications of Figure \ref{fdds5_21b}.}
\label{fdds5_21d}
\end{figure}

\begin{figure}[htbp]
\centering
\centering
\subfigure[]
{\includegraphics[width=\columnwidth]{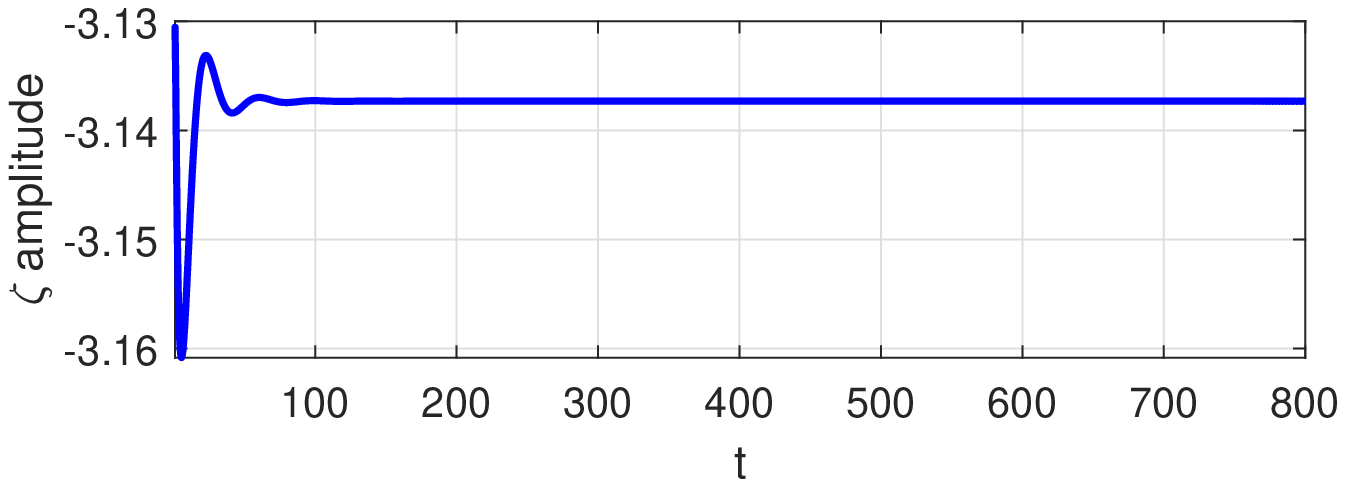}}
\subfigure[]
{\includegraphics[width=\columnwidth]{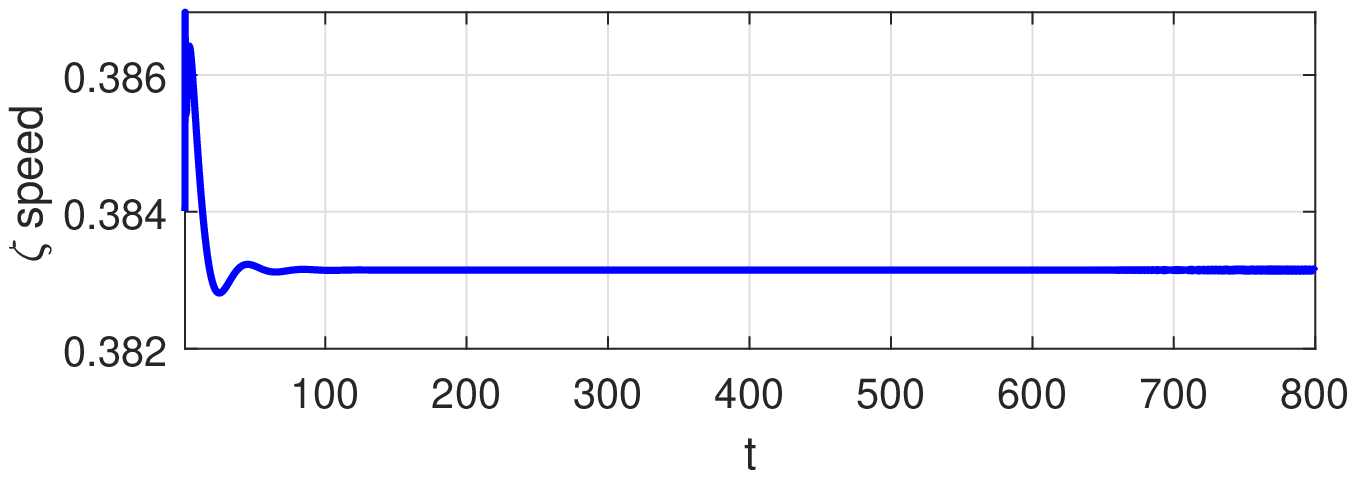}}
\caption{Small perturbation of a CSW with nonmonotone decay. $A=1.1$. Evolution of  maximum negative excursion (a) and speed (b) of the emerging  solitary wave.}
\label{fdds5_21c}
\end{figure}

We reproduce now the analogous experiment with $A=2.1$ in (\ref{53b}), and the results are shown in Figures \ref{fdds5_22}-\ref{fdds5_22c}. As the perturbation factor $A$ grows, the size of both tails also grows. In addition, Figures \ref{fdds5_22b} and \ref{fdds5_22d} suggest the formation of nonlinear structures, in the form of wavelets and perhaps some very small CSW's with non monotone decay. (These two structure were conjectured in Figure \ref{fdds5_21d}.)  
The emerging solitary wave is shorter (its maximum negative excursion is approximately $-5.233$, compared to $-6.035$ for the initial pertubed wave) and slower (the speed is now about $0.237$).

\begin{figure}[htbp]
\centering
\centering
\subfigure[]
{\includegraphics[width=\columnwidth]{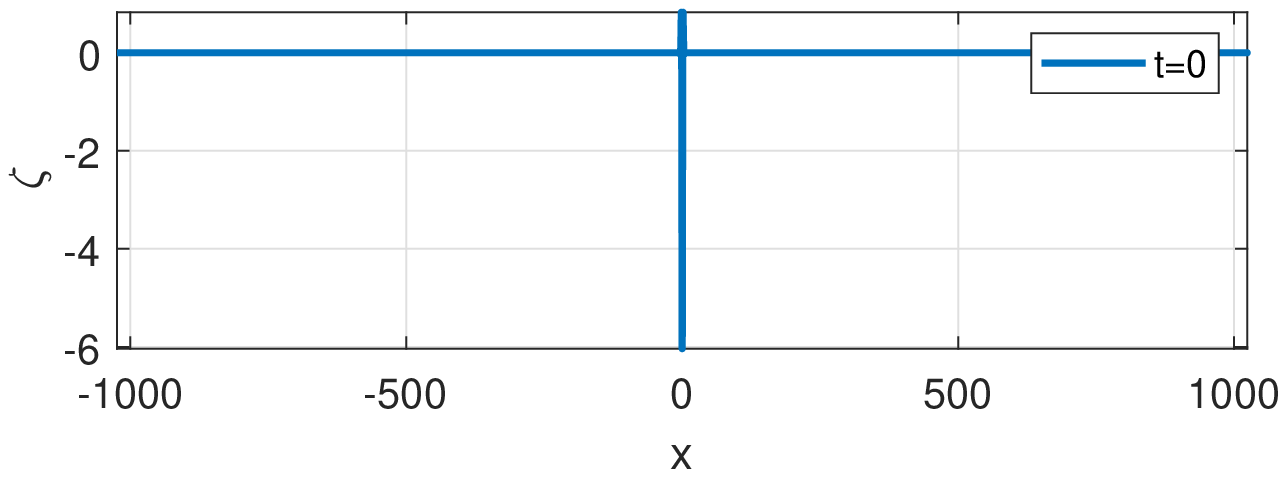}}
\subfigure[]
{\includegraphics[width=\columnwidth]{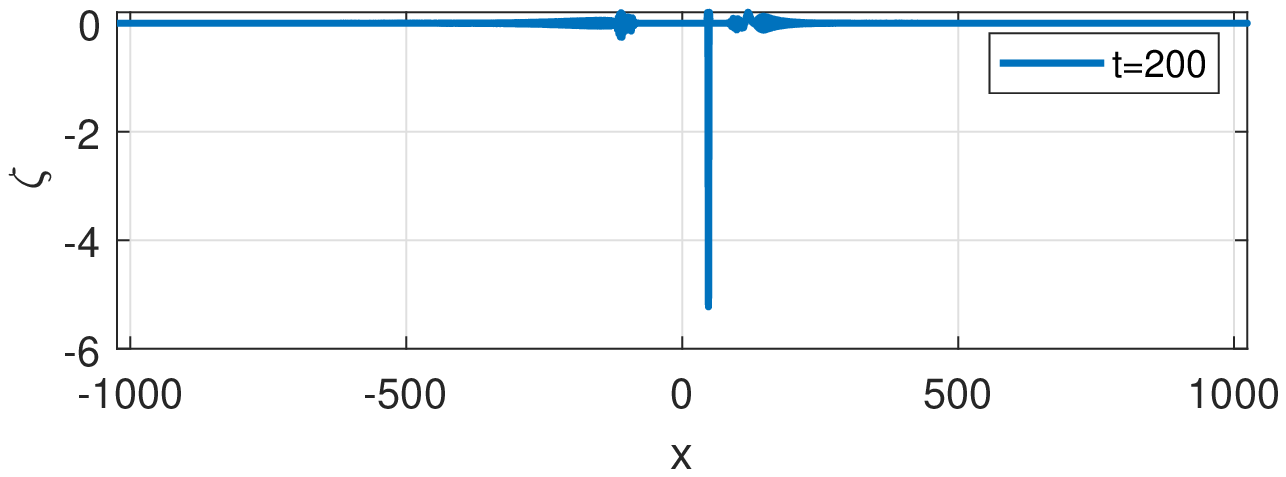}}
\subfigure[]
{\includegraphics[width=\columnwidth]{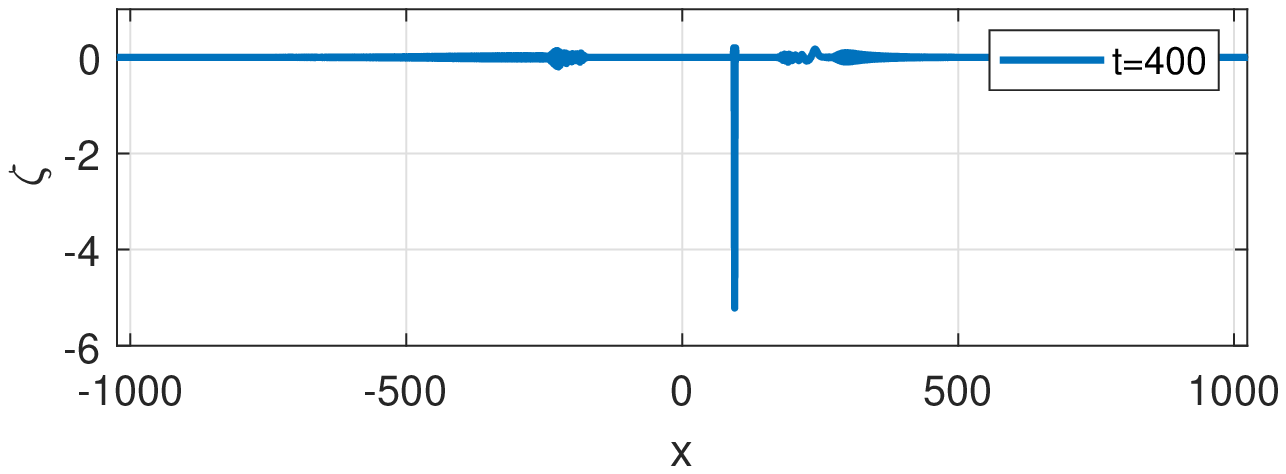}}
\subfigure[]
{\includegraphics[width=\columnwidth]{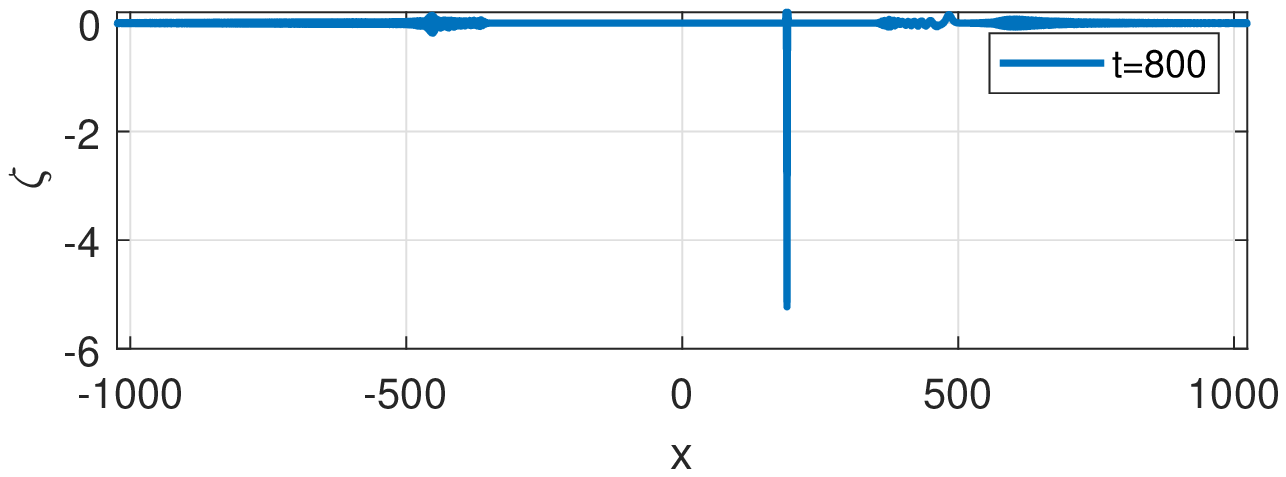}}
\caption{Perturbation of a CSW with nonmonotone decay. $A=2.1$. $\zeta$ component of the numerical solution.}
\label{fdds5_22}
\end{figure}
\begin{figure}[htbp]
\centering
\centering
\subfigure[]
{\includegraphics[width=\columnwidth]{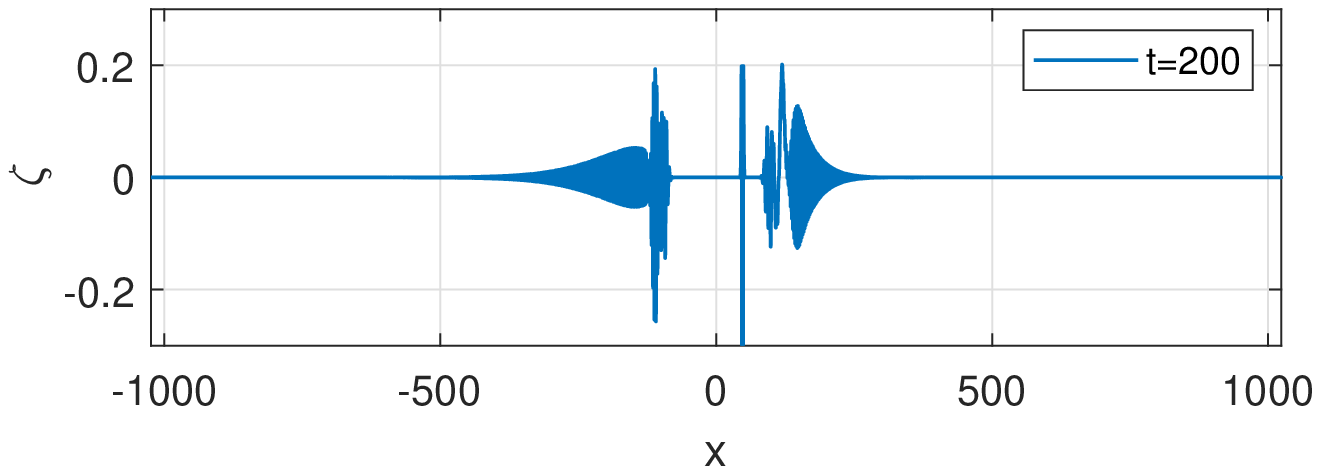}}
\subfigure[]
{\includegraphics[width=\columnwidth]{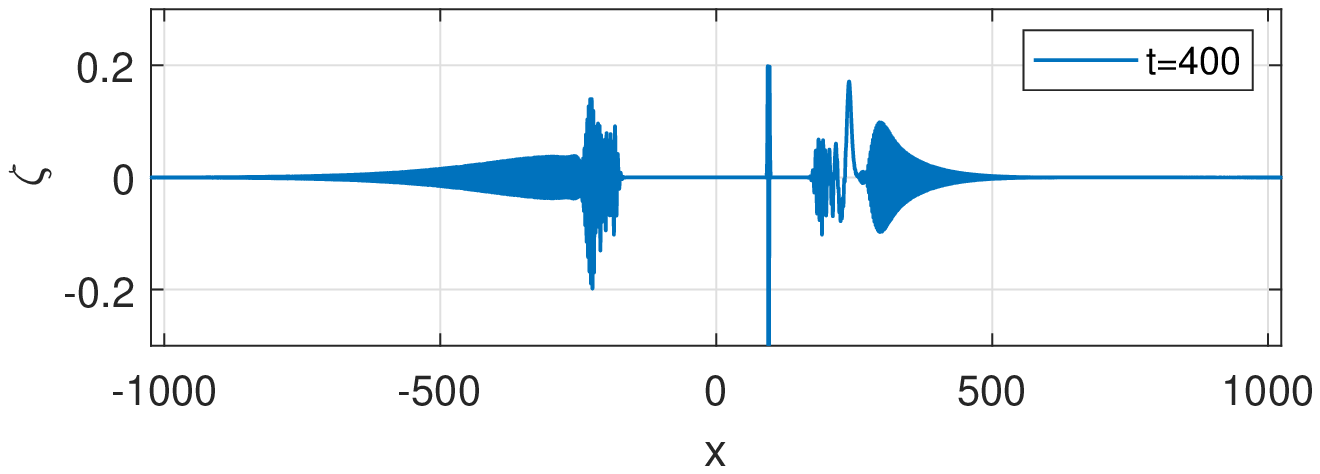}}
\subfigure[]
{\includegraphics[width=\columnwidth]{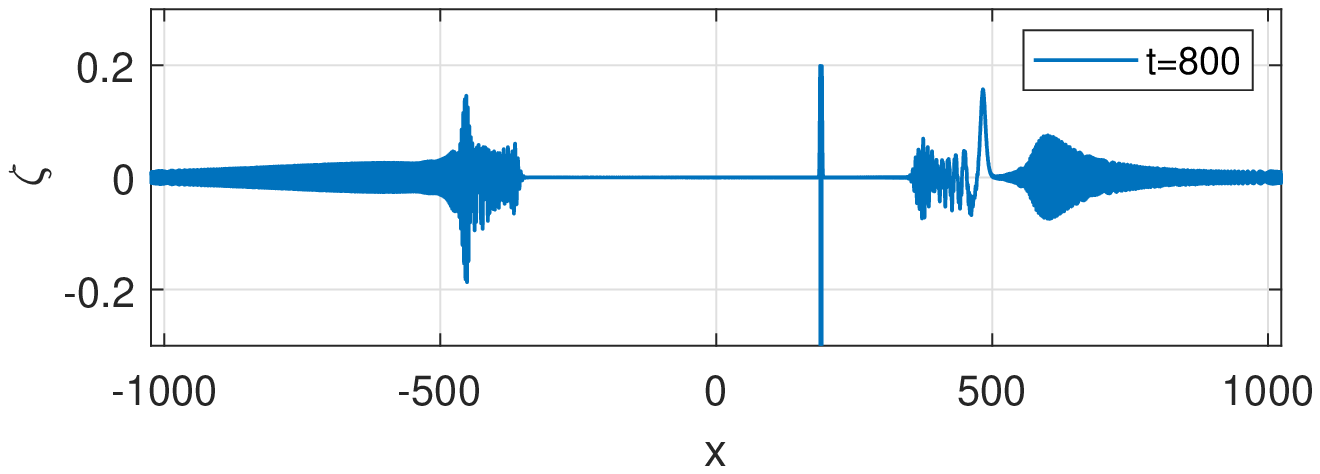}}
\caption{Perturbation of a CSW with nonmonotone decay. $A=2.1$. $\zeta$ component of the numerical solution. Magnifications of Figure \ref{fdds5_22}.}
\label{fdds5_22b}
\end{figure}

\section{Concluding remarks}
\label{sec6}
The present paper is concerned with the three-parameter family of internal-wave Boussinesq/Boussinesq (B/B) systems (\ref{BB2}). They model the bi-directional propagation of internal waves along the interface of a two-layer system of fluids under a rigid-lid assumption for the upper layer and over a rigid bottom bounding the lower layer below. The systems were derived in \cite{BonaLS2008} under the hypothesis that the flow is in the Boussinesq regime in both layers and are described by four parameters, $a, b, c, d$, three of them independent, like those corresponding to surface wave propagation, \cite{BonaChS2002,BonaChS2004}. 

\begin{figure}[htbp]
\centering
\centering
\subfigure[]
{\includegraphics[width=6.27cm,height=5.05cm]{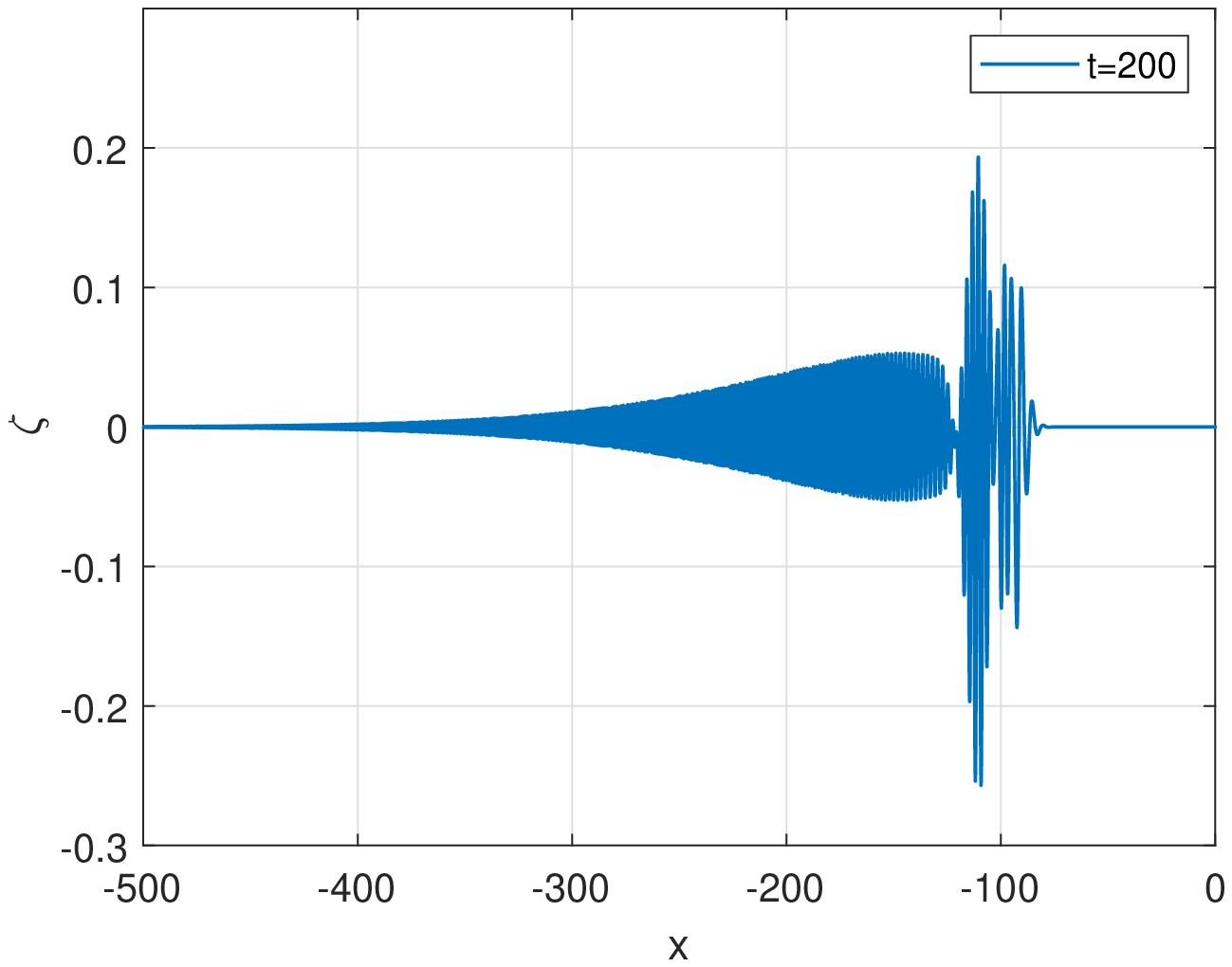}}
\subfigure[]
{\includegraphics[width=6.27cm,height=5.05cm]{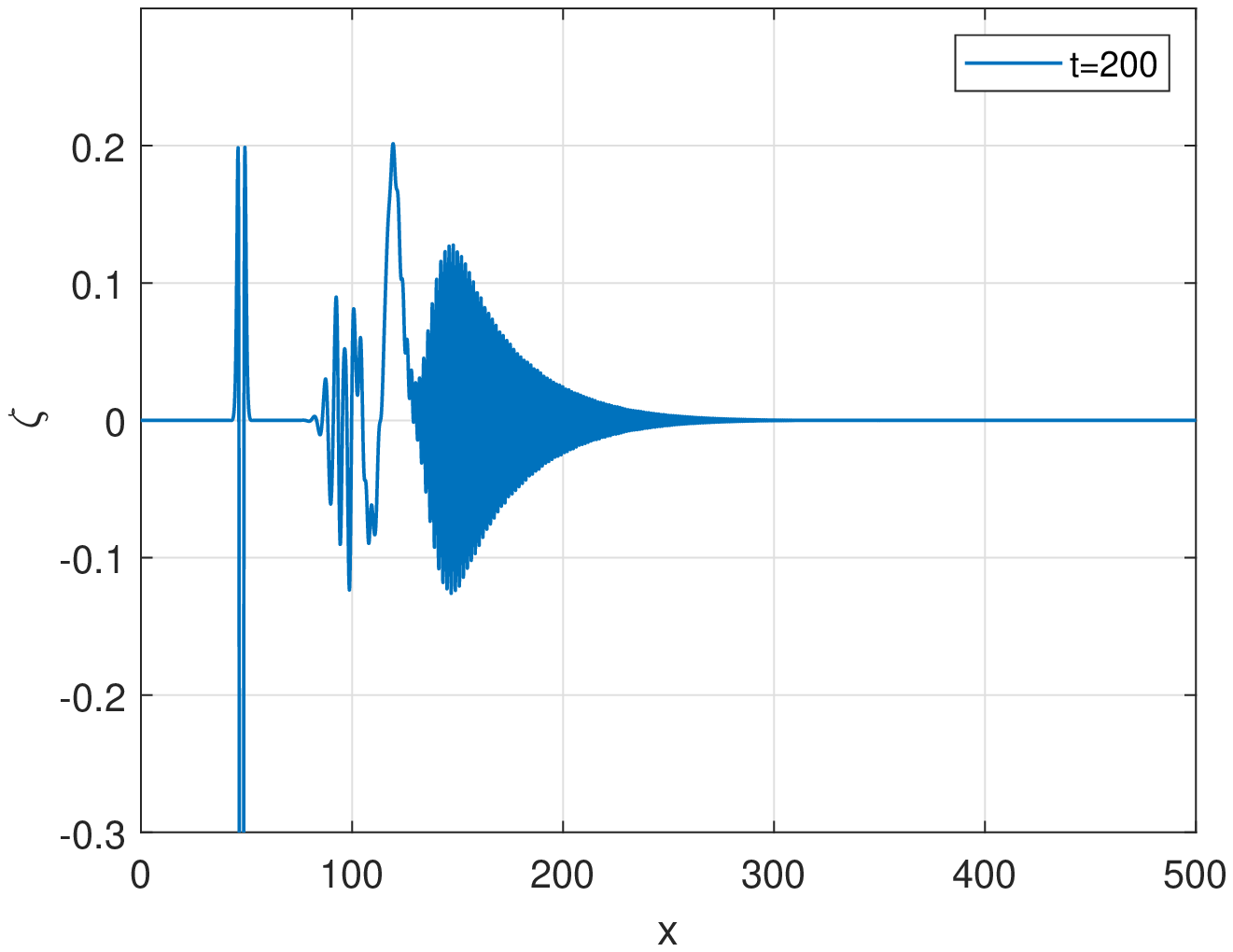}}
\subfigure[]
{\includegraphics[width=6.24cm,height=5.2cm]{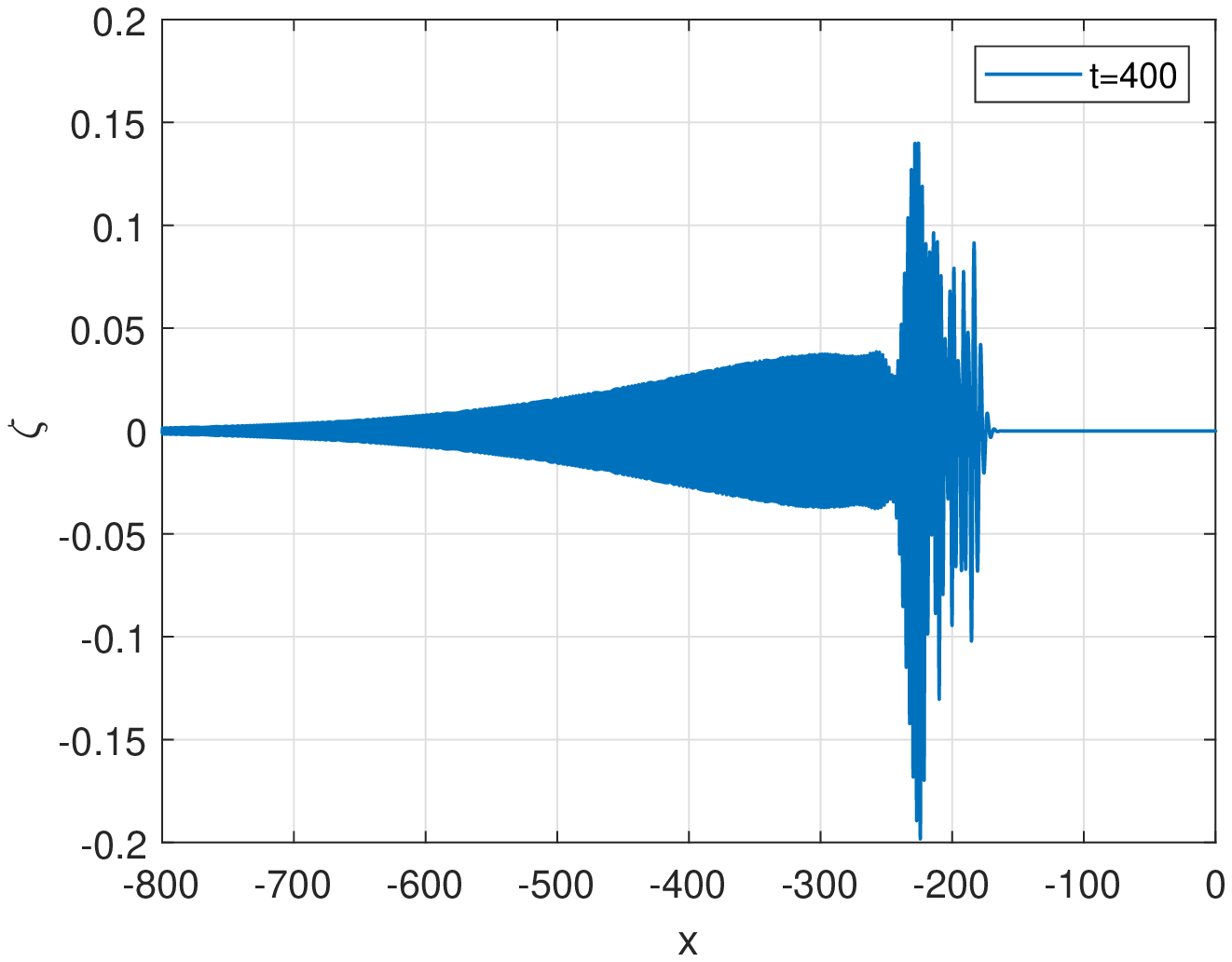}}
\subfigure[]
{\includegraphics[width=6.27cm,height=5cm]{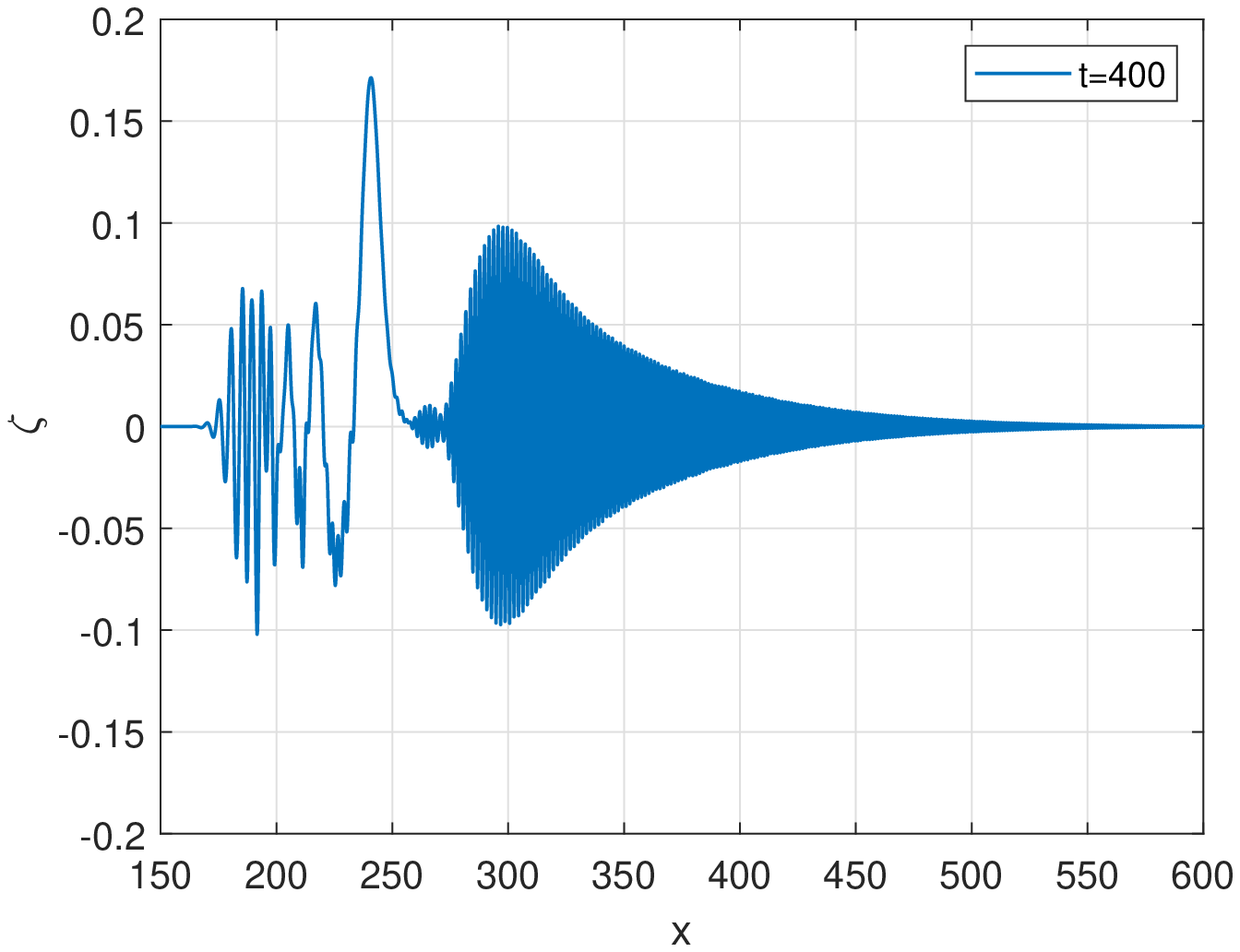}}
\subfigure[]
{\includegraphics[width=6.27cm]{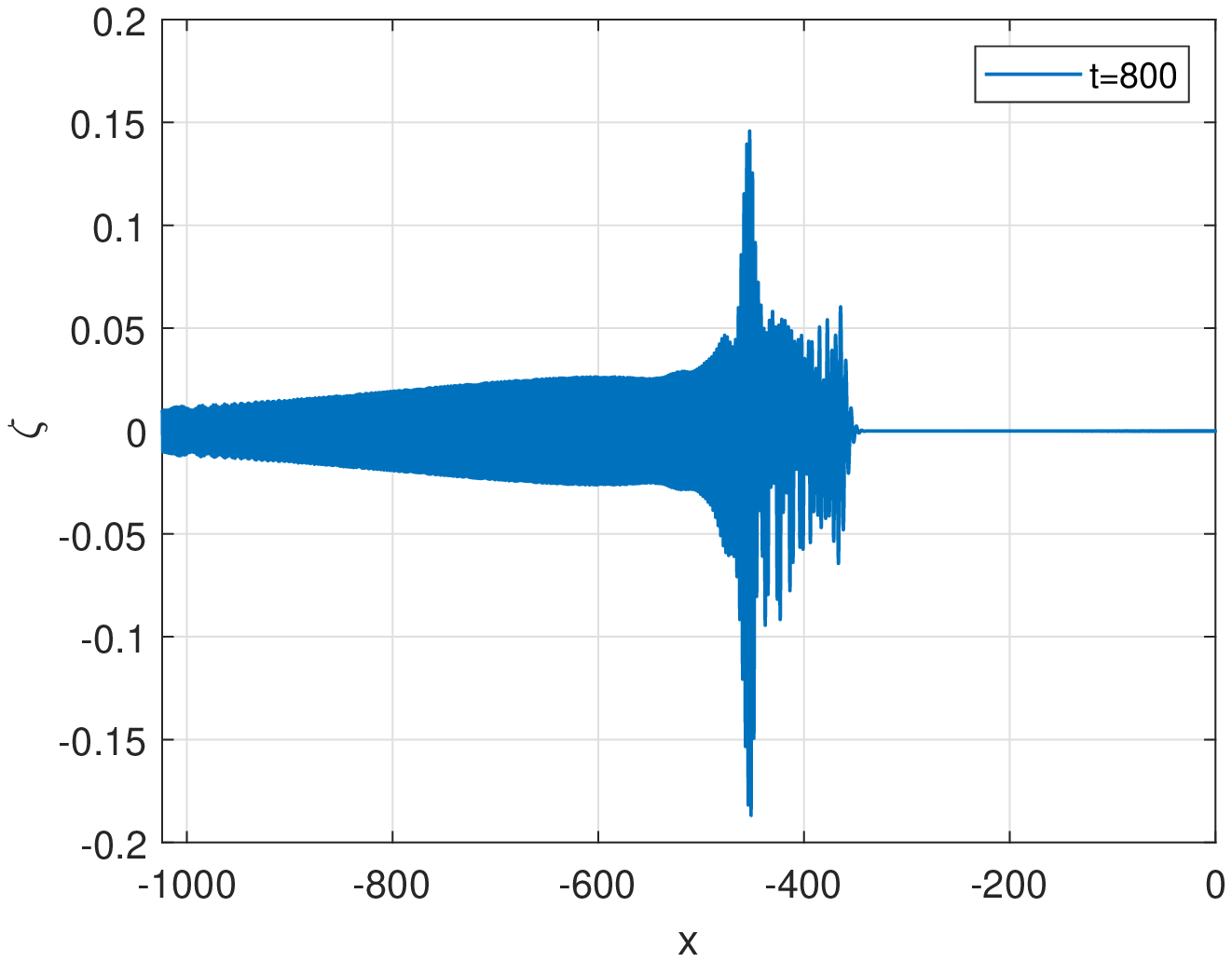}}
\subfigure[]
{\includegraphics[width=6.27cm,height=5.1cm]{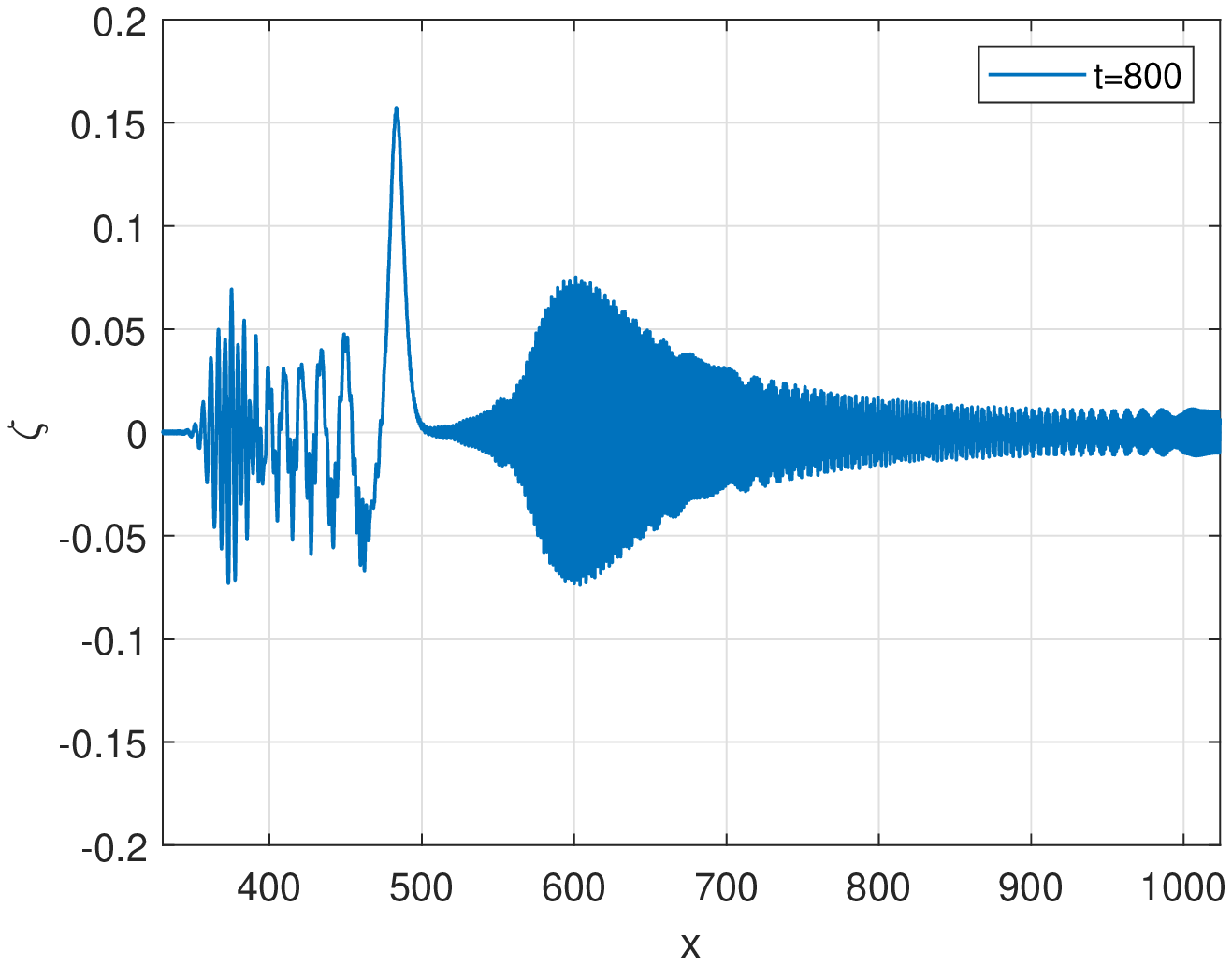}}
\caption{Perturbation of a CSW with nonmonotone decay. $A=2.1$. $\zeta$ component of the numerical solution. Magnifications of Figure \ref{fdds5_22b}.}
\label{fdds5_22d}
\end{figure}

\begin{figure}[htbp]
\centering
\centering
\subfigure[]
{\includegraphics[width=\columnwidth]{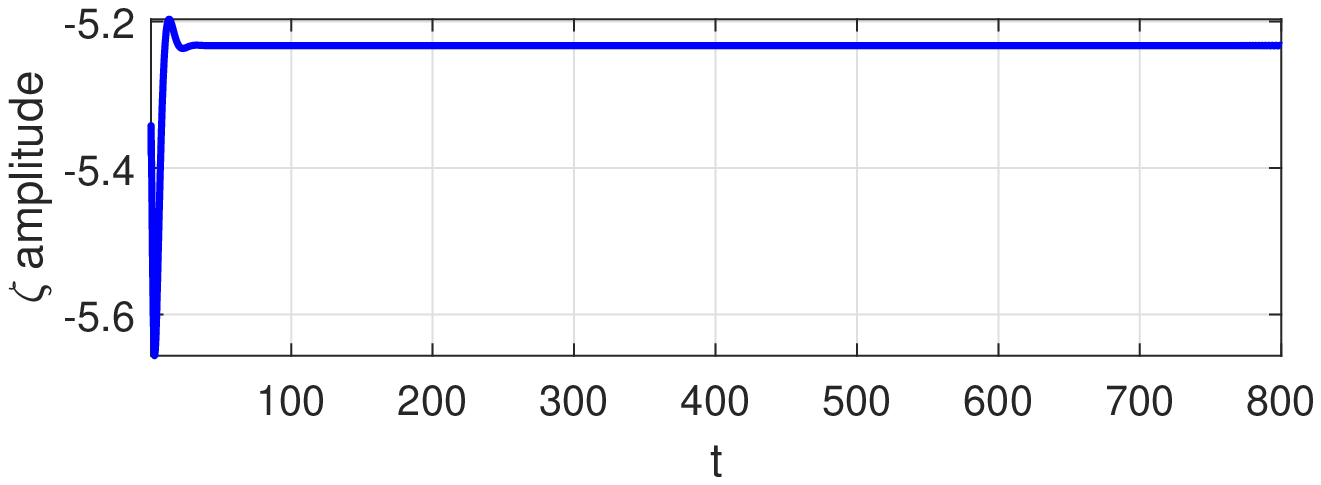}}
\subfigure[]
{\includegraphics[width=\columnwidth]{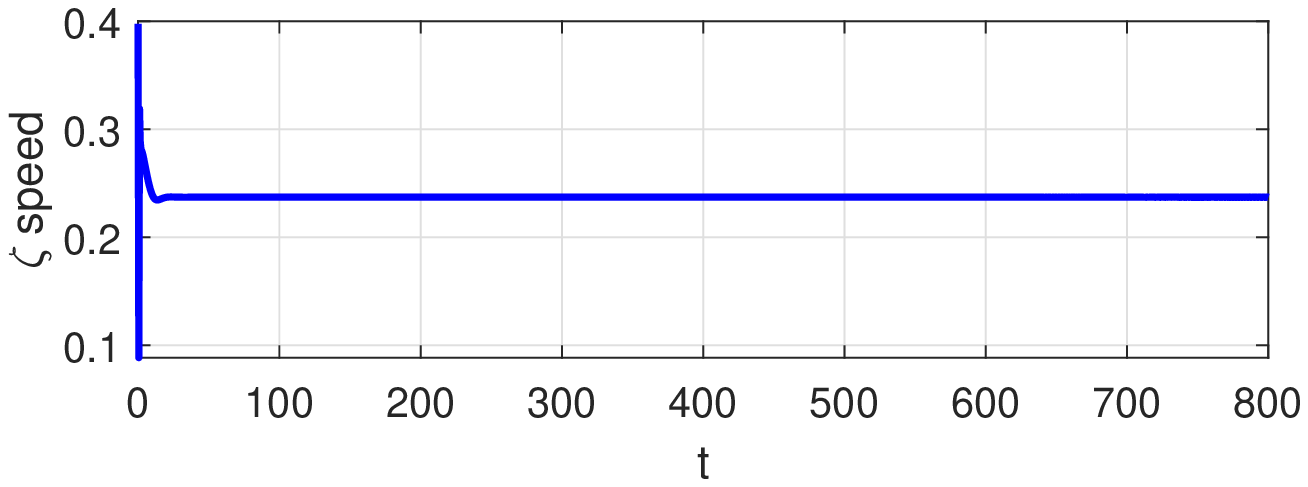}}
\caption{Perturbation of a CSW with nonmonotone decay. $A=2.1$. Evolution of  maximum negative excursion (a) and speed (b) of the emerging  solitary wave.}
\label{fdds5_22c}
\end{figure}
%The tendency seems to indicate that for larger values of $A$, the stability of the wave    will break somehow.
%\section{Concluding remarks}
%\label{sec6}
%The present paper is concerned with the three-parameter family of internal wave Boussinesq/Boussinesq (B/B) systems (\ref{BB2}). They model the bi-directional propagation of internal waves along the interface of a two-layer system of fluids under a rigid-lid assumption for the upper layer and over a rigid bottom bounding the lower layer below. The systems were derived in \cite{BonaLS2008} under the hypothesis that the flow is in the Boussinesq regime in both layers and are described by four parameters, $a, b, c, d$, three of them independent, like those corresponding to surface wave propagation, \cite{BonaChS2002,BonaChS2004}. 

In Section \ref{sec2} several theoretical aspects of this family of systems are discussed. We first present an alternative derivation, cf. \cite{BonaChS2002,Duran2019},
based on asymptotic expansions of the velocity potential associated to the fluid layers without using nonlocal operators as in \cite{BonaLS2008}. Then the theory developed in \cite{BonaChS2002} for the case of surface waves is used to review linear and nonlinear well-posedness of the internal-wave systems. Specifically, the B/B systems are linearly well-posed when $a,c\leq 0, b,d\geq 0$. As for local (and in some cases conditionally global) nonlinear well-posedness, an analysis similar to the one in \cite{BonaChS2004}
establishes seven types of systems, depending on the parameters $a, b, c, d$, in corresponding Sobolev spaces where existence, uniqueness and regularity locally in time of solutions hold. They correspond to the cases (i)-(vii) in section \ref{sec22}.
%
%are:
%\begin{itemize}
%\item Case (i): $b,d>0, a=c=0$ ($H^{s}\times H^{s}, s\geq 0$).
%\item Case (ii): $b,d>0, a,c<0$ ($H^{s}\times H^{s}, s\geq 0$).
%\item Case (iii): $b=0, d>0, a,c<0$ ($H^{s}\times H^{s+3}, s\geq 1$). 
%\item Case (iv): $b=0, d>0, a=c=0$, or
%$b>0, d=0, a=c=0$ ($H^{s}\times H^{s+1}, s\geq 1$; conditional global existence).
%\item Case (v): $b,d>0, a=0, c<0$ or $b,d>0, a<0, c=0$ ($H^{s+1}\times H^{s}, s\geq 0$; conditional global existence).
%\item Case (vi): $b=0, d>0, a<0, c=0$ ($H^{s}\times H^{s+2}, s\geq 1$).
%\item Case (vii): $b>0, d=0, a<0, c=0$ or $b=0, d>0, a=0, c<0$ ($H^{s}\times H^{s}, s\geq 2$).
%\end{itemize}
When $b=d$ these systems admit a Hamiltonian structure, and the Hamiltonian and other conserved quantities are derived in section \ref{sec23}.

Section \ref{sec3} is devoted to the error analysis of the spectral semidiscretization for approximating the periodic ivp for the B/B systems. Error estimates for the semidiscrete schemes are derived for each case (i)-(vii) of nonlinearly well posed systems obtained in section \ref{sec2}. 
%Specifically, if $N$ denotes the order of the trigonometric approximation , $T>0$, $(\zeta,u)\in C^{1}(0,T,H^{\mu}), \mu\geq 1$ is the solution of the periodic ivp and $(\zeta_{N},u_{N})$ denotes the semidiscrete solution, then, the following error estimates hold:
%\begin{itemize}
%\item Case (i): $\max_{0\leq t\leq T}\left(||\zeta_{N}-\zeta||+||u_{N}-u||\right)=O(N^{-\mu}),\; \mu\geq 1$.
%\item Case (ii): $\max_{0\leq t\leq T}\left(||\zeta_{N}-\zeta||+||u_{N}-u||\right)=O(N^{-\mu}),\; \mu\geq 1$.
%\item Case (iii): $\max_{0\leq t\leq T}\left(||\zeta_{N}-\zeta||+||u_{N}-u||_{1}\right)=O(N^{1-\mu}),\; \mu>3/2$. 
%\item Case (iv): $\max_{0\leq t\leq T}\left(||\zeta_{N}-\zeta||+||u_{N}-u||_{1}\right)=O(N^{1-\mu}),\; \mu>3/2$.
%\item Case (v): $\max_{0\leq t\leq T}\left(||\zeta_{N}-\zeta||+||u_{N}-u||\right)=O(N^{-\mu}),\; \mu\geq 1$.
%\item Case (vi): $\max_{0\leq t\leq T}\left(||\zeta_{N}-\zeta||+||u_{N}-u||_{1}\right)=O(N^{1-\mu}),\; \mu>3/2$. 
%\item Case (vii): $\max_{0\leq t\leq T}\left(||\zeta_{N}-\zeta||+||u_{N}-u||\right)=O(N^{1-\mu}),\; \mu>3/2$. 
%\end{itemize}
%In section \ref{sec4} we consider the existence and numerical approximation of solitary-wave solutions of the systems (\ref{BB2}). We apply several techniques for proving existence of such waves, namely Normal Form Theory, \cite{I,HaragusI}, valid for solitary-wave speeds close to the limiting value $c_{\gamma,\delta}=\sqrt{(1-\gamma)/(\delta+\gamma)}$, and Toland's theory, \cite{Toland1986}, Concentration-Compactness theory, \cite{Lions}, and Positive Operator theory, \cite{BenjaminBB1990}, 

In sections \ref{sec4} and \ref{sec5} we study solitary-wave solutions of the B/B systems. Section \ref{sec4} is concerned with the existence and numerical generation of this type of solutions. In the first part, section \ref{sec41}, we apply standard theories such as Normal Form Theory (NFT), \cite{I,HaragusI}, Toland's Theory, \cite{Toland1986}, Concentration-Compactness Theory (CCT), \cite{Lions},  and Positive Operator Theory (POT), \cite{BenjaminBB1990}, in order to derive existence results. If we make use of the linearization of the system for the solitary waves (\ref{BB6}), written as a first-order system (\ref{NFT1}), at the origin, NFT allows us to establish the existence of classical solitary waves (CSW's) and generalized solitary waves (GSW's), in two \lq generic\rq\ cases (cf. Table \ref{tavle0}), respectively:
\begin{itemize}
\item $a,c\leq 0, b,d\geq 0, bd-ac/\kappa_{1}>0$ (CSW),
\item $a,c\leq 0, b,d\geq 0, bd-ac/\kappa_{1}<0$ (GSW),
\end{itemize}
where $\kappa_{1}=1/(\delta+\gamma)$, with $\delta$ and $\gamma$ denoting, respectively, the depth and density ratios of the two-layer system of fluids. Existence of such solitary waves is ensured by the NFT when the magnitude of the speed $c_{s}$ is greater than but close to the limiting value $c_{\gamma,\delta}=\sqrt{(1-\gamma)/(\delta+\gamma)}$ (speed of sound). In addition, NFT also predicts periodic solutions close to the region of generation of CSW's, as well as classical solitary waves with non monotone decay for speeds $c_{s}$ satisfying $|c_{s}|<c_{\gamma,\delta}$.

The rest of the theories contribute more results for the existence of CSW's of speeds not necessarily close to $c_{\gamma,\delta}$. Thus:
\begin{itemize}
\item When $a,c\leq 0, b=d>0$ (Hamiltonian case), Toland's theory ensures the existence of classical solitary waves. A specific speed-amplitude relation of the form  (\ref{48b}) holds.
\item When $a,c< 0, b=d>0$, CCT establishes the existence of CSW's  for speeds $c_{s}$ satisfying a bound of the form (\ref{49e}) and $|c_{s}|<c_{\gamma,\delta}$. 
\item The application of POT proves the existence of CSW's when $b,d>0, a,c\leq 0$, and $bd-ac/\kappa_{1}>0$, with speeds $c_{s}$ satisfying $|c_{s}|>c_{\gamma,\delta}$.
\end{itemize}
For particular values of $c_{s}$, some exact formulas of CSW's of ${\rm sech}^{2}$ type may be obtained by using similar arguments to those of \cite{MChen1998} valid in the case of surface waves.

% \cite{Petv1976,pelinovskys}, accelerated by the Minimal Polynomial extrapolation (MPE) algorithm, \cite{sidi,sidifs,smithfs,AlvarezD2015}.

These existence results are illustrated in section \ref{sec42}, where CSW and GSW solutions are numerically generated. The numerical procedure consists of discretizing the system (\ref{BB6}) of ode's satisfied by the solitary wave profiles, on a long enough interval with periodic boundary conditions, by the Fourier collocation method. In the Fourier space, the differential equation systems for the profiles become algebraic systems, which are iteratively solved by the Petviashvili's scheme, \cite{Petv1976}, accelerated with vector extrapolation techniques, \cite{sidi}. Each case is illustrated by exhibiting the approximate solitary wave profiles ($\zeta$ and $v_{\beta}$ components), the corresponding phase portraits, and the decrease of the residual error with the number of iterations in order to check the convergence of the iteration. It is worth mentioning that, in order to illustrate the results of Toland's Theory, the iterative method is modified by projection techniques so that the speed-amplitude relation (\ref{48b}) is satisfied up to machine accuracy. In addition, numerical experiments suggest that in all cases, and for both CSW's and GSW's, the amplitude is an increasing function of the speed difference $c_{s}-c_{\gamma,\delta}$.

In section \ref{sec5} we make a computational study on some aspects of the dynamics of classical and generalized solitary-wave solutions of the B/B systems by solving numerically the periodic ivp on a long enough interval. The numerical method used for these evolution simulations is a spectral Fourier discretization in space coupled with an implicit fourth-order RK-composition method based on the implicit midpoint rule as time integrator. The choice of the spatial discretization is justified by the error estimates in section \ref{sec3}, while the time integrator has been analyzed in \cite{DD} in the case of spectral discretizations of the periodic ivp for the KdV equation and shown to be computationally efficient when used in other nonlinear dispersive equations as well, \cite{FrutosS1992,DDM}. The resulting full discretization is computationally validated in section \ref{sec52}, where several numerical experiments are performed to check its accuracy, as well as the accuracy of the numerical solitary waves as traveling waves.

%that of the scheme of generation for solitary waves implemented in section \ref{sec42}.

Section \ref{sec53} is devoted to a computational study of the dynamics of CSW solutions which are strictly positive or negative. The experiments are made for the generic case $a,c<0, b,d>0, bd-ac/\kappa_{1}>0$ and are concerned with the ensuing evolution from small and large perturbations of CSW's, from superpositions of CSW's (in order to study overtaking and head-on collisions), and from Gaussian pulses, in order to study resolution into solitary waves. Some of the main conclusions of this study are reported here:
\begin{itemize}
\item Under small initial perturbations, the solution evolves into a modified CSW with small dispersive tails following the main wave. The generation and structure of these tails are justified by an analysis of small-amplitude solutions of the associated linearized system in a reference frame moving with the speed of the solitary wave. In the case of these strictly positive or negative classical solitary waves, the results predict the formation of two types of dispersive oscillation groups, trailing the solitary wave and traveling in opposite directions.
\item Increasing the size of the perturbation of the initial solitary wave may lead to the generation of additional stable, nonlinear structures. They may consist of smaller CSW's, CSW's of non monotone decay or others of wavelet type.
\item The collisions are, as expected, inelastic. In both cases (overtaking and head-on collisions), two CSW's emerge; the effects of the inelastic interactions include the generation of tails of dispersive nature and nonlinear structures of the same type as those already mentioned.
\item The evolution ensuing from initial Gaussian pulses shows resolution into a train of CSW's, leaving a small structure behind which seems to be dispersive.
\end{itemize}
In section \ref{sec54} we perform a corresponding computational study of the dynamics of GSW's. To this end, the experiments are in the generic case $a,c<0, b,d>0, bd-ac/\kappa_{1}<0$ and are of the same type as those in the study of section \ref{sec53}. The main conclusions are:
\begin{itemize}
\item The experiments with small perturbations of GSW's suggest that the question of stability of these waves is more intricate, in the sense that the formation of emerging, stable GSW's seems to require smaller initial perturbations and takes longer time than in the case of CSW's. In some examples the dispersive tails are hard to observe, as they are hidden in the ripples of the structure being of much smaller size.
\item The experiments with larger perturbations of an initial GSW show the formation of nonlinear structures of similar type to those observed in the case of CSW's, but overimposed on the ripples. In addition, larger ripples along with dispersive tails form in front of an emerging GSW, a fact that surely affects its stability.
\item The experiments show the evolution, from initial Gaussian pulses, of a train of solitary-wave type pulses traveling to the right. Ripples are formed between each pair of consecutive pulses, mixed somehow with dispersive tails. These dispersive groups are also observed behind the train. Finally, a second train of solitary-wave pulses is formed, traveling to the left, consisting of nonmonotonically decaying solitary waves; thus each wave of this train should have speed less than the limiting value $c_{\gamma,\delta}$.
\item The main effect observed in the experiments of overtaking and head-on collisions of GSW's is the formation of two emerging GSW's with ripples of different size, larger, in general, than those of the initial GSW's. Dispersive structures seem now to be smaller and are superimposed on the emerging ripples.
\end{itemize}

The existence of nonmonotonically decaying classical solitary waves, established in section \ref{sec41} and numerically generated in section \ref{sec42}, as well as their role in the dynamics of other solitary waves, observed in the experiments of sections \ref{sec53} and \ref{sec54}, motivate the numerical experiments of section \ref{sec55}, devoted to a computational study of the behaviour of these waves under small and large perturbations. The experiments suggest persistence and stability of these waves. From small perturbations, a new solitary wave of the same type is formed. The analysis of small-amplitude plane wave solutions of the linearized system in a reference moving with the speed of the solitary wave, shows now that the main part of the dispersive oscillations travels to the right in front of the solitary wave, with the rest traveling to the left, behind it. When the perturbation factor grows, the formation of a new solitary wave with nonmonotone decay is observed, and now the dispersive tails seems to be accompanied by the generation of small nonlinear structures, in the form of wavelets or nonmonotone CSW's.
%
%
%, with tails emanating in front of and behind it. 
%
%
%These waves appear to have different components, some with dispersive nature but some others which do not disperse and whose size grows with the perturbation factor of the initial profile. Thus it is expected  than large perturbations will generate some type of instability).
%

\section*{Acknowledgements}
Vassilios Dougalis and Angel Duran would like to acknowledge travel support, that made possible this collaboration, from the Institute of Mathematics (IMUVA) of the University of Valladolid, and the Institute of Applied and Computational Mathematics of FORTH.
Leetha Saridaki was supported by the grant \lq\lq Innovative Actions in Environmental Research and Development (PErAn)\rq\rq (MIS5002358), implemented under the \lq\lq Action for the strategic development of the Research and Technological sector'' funded by the Operational Program \lq\lq Competitiveness, and Innovation'' (NSRF 2014-2020) and co- financed by Greece and the EU (European Regional Development Fund). The grant was issued to the Institute of Applied and Computational Mathematics of FORTH.

%\tableofcontents
\end{document}